\documentclass[10pt]{article}   % The 12pt increases the font size.
\usepackage{amssymb,amsmath, amsthm, amsfonts}
\usepackage{graphicx}
\usepackage{listings}
\usepackage{lstautogobble}
\usepackage{enumerate}
\usepackage[shortlabels]{enumitem}
\usepackage{thmtools}
\usepackage{thm-restate}

\usepackage{authblk}

\usepackage{mathtools}
\usepackage{physics}
\usepackage{color}
\usepackage{hyperref}
\usepackage[margin=.5in]{geometry}
\usepackage{mathrsfs}
\usepackage[final]{showlabels}
\setlist{  
  listparindent=\parindent,
  parsep=0pt,
}

\theoremstyle{plain}
\newtheorem{thm}{Theorem}[section]
\newtheorem{prop}[thm]{Proposition}
\newtheorem{lemma}[thm]{Lemma}
\newtheorem{cor}[thm]{Corollary}

\theoremstyle{definition}
\newtheorem{mydef}[thm]{Definition}

\newtheorem{remark}[thm]{Remark}
\newtheorem{prob}{Problem}[section]

\numberwithin{equation}{section} %Equation numbering

\DeclarePairedDelimiter\ipp{\langle}{\rangle}
\DeclarePairedDelimiter{\brak}{\lbrack}{\rbrack}
\DeclarePairedDelimiter{\paren}{\lparen}{\rparen}
\DeclarePairedDelimiter{\brac}{\lbrace}{\rbrace}
\DeclarePairedDelimiter{\jp}{\langle}{\rangle}

\DeclareMathOperator{\Rxx}{\frac{\partial_{1}^{2}}{\Delta}}

\DeclareMathOperator{\supp}{supp}
\DeclareMathOperator{\diam}{diam}

\DeclareMathOperator{\dist}{dist}
\DeclareMathOperator{\sgn}{sgn}

\newcommand{\p}{{\partial}}

\newcommand{\R}{{\mathbb{R}}}
\newcommand{\C}{{\mathbb{C}}}
\newcommand{\N}{{\mathbb{N}}}

\newcommand{\Z}{{\mathbb{Z}}}
\newcommand{\K}{{\mathcal{K}}}
\renewcommand{\S}{{\mathbb{S}}}

\newcommand{\ol}{\overline}
\newcommand{\ul}{\underline}

\newcommand{\E}{{\mathcal{E}}}
\renewcommand{\L}{{\mathcal{L}}}
\renewcommand{\P}{{\mathcal{P}}}

\newcommand{\ueta}{\underline{\eta}}
\newcommand{\uD}{\underline{D}}

\def\XXint#1#2#3{{\setbox0=\hbox{$#1{#2#3}{\int}$ }
\vcenter{\hbox{$#2#3$ }}\kern-.6\wd0}}

\begin{document}
\title{Global Well-Posedness and Scattering for the Elliptic-Elliptic Davey-Stewartson System at $L^{2}$-Critical Regularity}
\author[1]{Matthew Rosenzweig}
\affil[1]{University of Texas at Austin}

\maketitle{}

\begin{abstract}
In this paper, we prove global well-posedness and scattering of the Cauchy problem for the elliptic-elliptic Davey-Stewartson system (eeDS) for initial data $u_{0}\in L^{2}(\R^{2})$ in the defocusing case and for $u_{0}\in L^{2}(\R^{2})$ with mass below that of the ground state in the focusing case. This result resolves the large data problem at the scaling-critical regularity left open by Ghidaglia and Saut in their work \cite{Ghidaglia1990}, which initiated the mathematical study of the Cauchy problem for the system. Our proof uses the concentration compactness/rigidity road map of Kenig and Merle together with the long-time Strichartz estimate approach of Dodson. Due to the failure of the endpoint $L_{t}^{2}L_{x}^{\infty}$ Strichartz estimate, we rely heavily on bilinear Strichartz estimates. We overcome the obstruction to applying such estimates caused by the lack of permutation invariance of the eeDS nonlinearity under frequency decomposition by introducing a new frequency cube decomposition of the nonlinearity and proving bilinear estimates suited to this decomposition. In both the defocusing and focusing cases, we overcome the lack of an a priori interaction Morawetz estimate by exploiting the spatial and frequency localization of the minimal counterexamples which we reduce to considering.
\end{abstract}

\tableofcontents

\pagebreak %Pagebreak for table of contents

\section{Introduction}\label{sec:in}
\subsection{The Davey-Stewartson System}
In this paper, we consider the Cauchy problem for the elliptic-elliptic Davey-Stewartson system (eeDS)
\begin{equation} %Davey-Stewartson system
\begin{cases}
(i\partial_{t}+\Delta)u = \mu |u|^{2}u+(\partial_{x_{1}}\varphi)u\\
\alpha\partial_{x_{1}}^{2}\varphi+\partial_{x_{2}}^{2}\varphi = -\gamma\partial_{x_{1}}(|u|^{2})
\end{cases}, \qquad (t,x)\in \R\times\R^{2},\label{eq:intro_eeds}
\end{equation}
where $\mu\in\{\pm 1\}$ and $\alpha,\gamma$ are positive real parameters. The system \eqref{eq:intro_eeds} is a special case of the full Davey-Stewartson system
\begin{equation}
\begin{cases}
i\partial_{t}u+\sigma_{1}\partial_{x_{1}}^{2}u+\partial_{x_{2}}^{2}u=\sigma_{2}|u|^{2}u+(\partial_{x_{1}}\varphi) u\\
\alpha\partial_{x_{1}}^{2}\varphi+\partial_{x_{2}}^{2}\varphi=-\gamma\partial_{x_{1}}(|u|^{2}),
\end{cases}
\qquad (t,x)\in\R\times\R^{2}, \label{eq:full_DS}
\end{equation}
where $(\sigma_{1},\sigma_{2},\alpha,\gamma)\in \{\pm 1\} \times\{\pm 1\} \times\R\times\R_{+}$. Following \cite{Ghidaglia1990}, we classify the system based on the values of $(\sgn(\sigma_{1}),\sgn(\alpha))$ in the table below.
\begin{small}
\begin{center}
\begin{tabular}{|c|c|c|c|}
\hline
Elliptic-Elliptic & Hyperbolic-Elliptic & Elliptic-Hyperbolic & Hyperbolic-Hyperbolic \\
\hline
$(1,1)$ & $(-1,1)$ & $(1,-1$) & $(-1,-1)$\\
\hline
\end{tabular}
\end{center}
\end{small}
Furthermore, when $\sigma_{2}=1$, we say that the system \eqref{eq:full_DS} is \emph{defocusing}, and when $\sigma_{2}=-1$, we say that the system is \emph{focusing}.

The system was first introduced by Davey and Stewartson (\cite{Davey1974}) as a formal multiple scales approximation for the evolution of surface wave packets on an incompressible, irrotational, and inviscid fluid subject to only gravity, which are moreover slowly modulated in both space and time. Subsequently, the system was formally derived in \cite{Djordjevic1977} and \cite{ablowitz_segur_1979} as multiple scales approximation for slowly modulated wave packets subject to both gravity and surface tension (see also the work \cite{Craig1997} for a more rigorous mathematical treatment). The elliptic-elliptic case corresponds to the physical setting of both gravity and surface tension, where surface tension is not too strong relative to gravity. From this perspective of the water waves problem, the Davey-Stewartson system is the natural 2D generalization of the 1D cubic NLS. Forthcoming work by the author will consider the problem of full justification of the Davey-Stewartson system as a modulation approximation to the underlying 3D water waves problem, following earlier work by Totz (\cite{totz2015justification}) on the infinite-depth limiting case.

Following the nomenclature of \cite{sulem1999nonlinear}, for parameters $(\sigma_{1},\sigma_{2},\alpha,\gamma)=(1,\pm 1,-1,\mp 2)$, we refer to the sytem as DSI. For parameters $(\sigma_{1},\sigma_{2},\alpha,\gamma)=(-1,1,1,2)$ and $(\sigma_{1},\sigma_{2},\alpha,\gamma)=(-1,-1,1,-2)$, the system is called defocusing DSII and focusing DSII, respectively. The DSI and DSII systems are known to be integrable by the inverse scattering transform (IST), and in fact these are the only two integrable cases (\cite{Shulman1983}). For more on the integrable nature of the system, we refer the reader to \cite{ablowitz1991solitons}.

For $\lambda>0$, the solution class to the system \eqref{eq:full_DS} is invariant under the scaling transformation
\begin{align}
u(t,x)\mapsto u_{\lambda}(t,x) &\coloneqq \lambda u(\lambda^{2}t,\lambda x)\\
\varphi(t,x) \mapsto \varphi_{\lambda}(t,x) &\coloneqq \lambda\varphi(\lambda^{2}t,\lambda x).
\end{align}
For $\alpha>0$, classical solutions to \eqref{eq:full_DS} conserve \emph{mass}
\begin{equation} %Mass 
M(u(t)) \coloneqq \int_{\R^{2}} |u(t,x)|^{2}dx,
\end{equation}
\emph{momentum}
\begin{equation} %Momentum 
P(u(t)) \coloneqq 2\int_{\R^{2}}\Im{\bar{u}\nabla u}(t,x)dx,
\end{equation}
and \emph{energy}
\begin{equation} %Energy
E(u(t),\varphi(t)) \coloneqq \int_{\R^{2}}\left(\sigma_{1}|\partial_{x_{1}}u(t,x)|^{2}+|\partial_{x_{2}}u(t,x)|^{2}+\frac{1}{2}\left(\sigma_{2}|u(t,x)|^{4}-\frac{\alpha}{\gamma}|\partial_{x_{1}}\varphi(t,x)|^{2}-\frac{1}{\gamma}|\partial_{x_{2}}\varphi(t,x)|^{2}\right)\right)dx.
\end{equation}
The DS scaling leaves the $L_{x}^{2}$ norm of $u(t)$ invariant, which is at the regularity of the mass functional. For this reason, we refer to the Davey-Stewartson system as being $L^{2}$- or mass-critical. In this work, we focus on solutions at this critical regularity.

We now specialize to the simplified elliptic-elliptic Davey-Stewartson system where $(\alpha,\gamma)=(1,1)$:
\begin{equation}
\begin{cases}
(i\partial_{t}+\Delta)u = (\mu|u|^{2}+\partial_{x_{1}}\varphi)u \eqqcolon F(u)\\
\Delta\varphi=-\partial_{x_{1}}(|u|^{2}),
\end{cases}
\qquad (t,x)\in\R\times\R^{2}, \enspace \mu\in\{\pm 1\}. \label{eq:simp_DS}
\end{equation}
Hereafter to, we refer to the system \eqref{eq:simp_DS} simply as eeDS. By solving the Poisson equation for $\p_{x_{1}}\varphi$, we can write the Cauchy problem for \eqref{eq:simp_DS} in terms of a single nonlocal, nonlinear Schr\"{o}dinger equation
\begin{equation}
\begin{cases}
(i\partial_{t}+\Delta) u = \left(\mu|u|^{2}-\frac{\p_{1}^{2}}{\Delta}(|u|^{2})\right)u, & (t,x)\in\R\times\R^{2} \\
u(0)=u_{0}
\end{cases}
,\qquad \mu\in\{\pm 1\},\label{eq:DS}
\end{equation}
where $\frac{\partial_{1}^{2}}{\Delta}$ is the homogeneous of degree zero Fourier multiplier with symbol $\frac{\xi_{1}^{2}}{|\xi|^{2}}$. We introduce the notation $\E=\frac{\p_{1}^{2}}{\Delta}$, so that the $F(u)=-\mathcal{L}_{\mu}(|u|^{2})u$, where $\mathcal{L}_{\mu}=-\mu Id+\E$. In the sequel, we will drop the subscript $\mu$ as its value will be clear from context. Formally, \eqref{eq:DS} resembles the 2D cubic nonlinear Schr\"{o}dinger equation (NLS)
\begin{equation}
(i\p_{t}+\Delta)u = \mu |u|^{2}u, \label{eq:NLS}
\end{equation}
but differs by an additional nonlocal, nonlinear term. Also similarly to the cubic NLS, the eeDS enjoys a number of symmetries, such as spacetime translation, time reversal, Galilean transformation, phase rotation, and pseudoconformal transformation. However, a key difference between the eeDS equation \eqref{eq:DS} and the cubic NLS \eqref{eq:NLS} is that the solutions to the former with spherically symmetric initial data are not necessarily spherically symmetric. This failure to propagate radial symmetry is a consequence of the lack of radial symmetry in the symbol of $\E$. Thus, one cannot hope to study the large-data global theory for \eqref{eq:DS} by applying techniques which exploit radial symmetry, as was first done for the 2D cubic NLS (e.g. \cite{Killip2009}).

To discuss the local Cauchy theory for \eqref{eq:DS}, we must first clarify our notion of a solution. In this paper, we work exclusively with the class of strong solutions, which are defined below. Given this exclusivity, we sometimes abbreviate ``strong solution" by ``solution".

\begin{mydef}[Solution] %Definition of a solution
We say a function $u:I\times\R^{2}\rightarrow\mathbb{C}$ on a time interval $I\subset\R$ is a \emph{strong solution} to \eqref{eq:DS} if it belongs to the class $C_{t,loc}^{0}L_{x}^{2}(I\times\R^{2})\cap L_{t,loc}^{4}L_{x}^{4}(I\times\R^{2})$, and we have the Duhamel formula
\begin{equation}
	u(t_{1})=e^{i(t_{1}-t_{0})\Delta}u(t_{0})-i\int_{t_{0}}^{t_{1}}e^{i(t_{1}-t)\Delta}F(u(t))dt, \qquad \forall t_{0},t_{1}\in I.
\end{equation}
Here, the notation $e^{it\Delta}$ denotes the free Schr\"{o}dinger propagator, which is the Fourier multiplier with symbol $e^{-it|\xi|^{2}}$. We refer to the interval $I$ as the \emph{lifespan} of the solution. We say that a solution has \emph{maximal lifespan} if there does not exist a time interval $J$ such that $I\subsetneq J$. We say that a solution is \emph{global} if $I=\R$.
\end{mydef}

Solutions which do not satisfy finite spacetime bounds on their lifespans are said to blow up, which we make precise with the next definition. The specific spacetime norm $L_{t,x}^{4}$ comes from the Strichartz estimate used to prove the local well-posedness of \eqref{eq:DS}, which we review below.
\begin{mydef}[Blowup] %Definition of Blowup
We say that a solution $u :I\times\R^{2}\rightarrow \mathbb{C}$ to \eqref{eq:DS} \emph{blows up forward in time} if there exists $t_{0}\in I$ such that
\begin{equation}
\int_{t_{0}}^{\sup I}\int_{\R^{2}} |u(t,x)|^{4}dxdt = \infty.
\end{equation}
Analogously, we say that $u$ \emph{blows up backward in time} if there exists $t_{0}\in I$ such that
\begin{equation}
\int_{\inf I}^{t_{0}}\int_{\R^{2}}|u(t,x)|^{4}dxdt.
\end{equation}
\end{mydef}

Global solutions to \eqref{eq:DS} which asymptotically evolve like the free solution are said to scatter.
\begin{mydef}[Scattering] %Definition of Scattering
Let $u : I \times \R^{2}\rightarrow\mathbb{C}$ be a solution to \eqref{eq:DS}. For an asymptotic state $u_{+}\in L^{2}(\R^{2})$, we say that $u$ \emph{scatters forward in time} to $e^{it\Delta}u_{+}$, if $\sup I = \infty$ and $\lim_{t\rightarrow\infty} \|u(t)-e^{it\Delta}u_{+}\|_{L_{x}^{2}(\R^{2})}=0$. If $u_{-}\in L^{2}(\R^{2})$, we say that $u$ \emph{scatters backward in time} to $e^{it\Delta}u_{-}$ if $\inf I=-\infty$ and $\lim_{t\rightarrow-\infty} \|u(t)-e^{it\Delta}u_{-}\|_{L_{x}^{2}(\R^{2})}=0$.
\end{mydef}

The mathematical study of the Cauchy problem for \eqref{eq:DS} and more generally, \eqref{eq:full_DS}, was initiated by Ghidaglia and Saut in \cite{Ghidaglia1990}. Below we excerpt the local well-posedness and small data global well-posedness results from that work which are specific to initial data in $L^{2}(\R^{2})$. The proof of these results follows from a fixed-point argument in the same spirit as that by Cazenave and Weissler (\cite{Cazenave1990},\cite{cazenave2003semilinear}) for the mass-critical NLS
\begin{equation}
(i\p_{t}+\Delta) u = \mu |u|^{1+\frac{4}{d}}u,
\end{equation}
with the new ingredient that the operator $\E$ is bounded on $L^{p}(\R^{2})$ for $1<p<\infty$.

\begin{restatable}[LWP, \cite{Ghidaglia1990}]{thm}{LWP}
\label{thm:LWP} %Ghidaglia and Saut Cauchy Theory
Fix $\mu\in\{\pm 1\}$. For $u_{0}\in L^{2}(\R^{2})$ and $t_{0}\in\R$, there exists a unique maximal-lifespan solution $u:I\times\R^{2}\rightarrow\mathbb{C}$ to \eqref{eq:DS}. Moreover,
\begin{enumerate}
		\item\label{item:LWP_smg}
		There exists $m_{0}>0$ such that if $\|u_{0}\|_{L^{2}(\R^{2})}<m_{0}$, then the solution $u$ is global, $\|u\|_{L_{t,x}^{4}(\R\times\R^{2})}\leq C(\|u_{0}\|_{L^{2}(\R^{2})})$, and $u$ scatters both forward and backward in time. Here, $C(\cdot)$ is a nondecreasing function such that $C(0)=0$.
		\item\label{item:LWP_open}
		$I$ is an open interval containing $t_{0}$.
		\item\label{item:LWP_bup}
		If $\sup(I)<\infty$ (resp. $\inf(I)>-\infty$), then $u$ blows up forward (resp. backward) in time.
		\item\label{item:LWP_cd}
		The solution map $L^{2}(\R^{2})\rightarrow C_{t,loc}^{0}L_{x}^{2}(I\times\R^{2})\cap L_{t,x}^{4}(I\times\R^{2})$, $u_{0}\mapsto u$ is uniformly continuous on compact time intervals for bounded subsets of initial data.
		\item\label{item:LWP_scat}
		If $\sup(I)=\infty$ (resp. $\inf(I)=-\infty$) and $u$ does not blow up forward (resp. backward) in time, then $u$ scatters forward (resp. backward) in time to some asymptotic state $u_{+}\in L^{2}(\R^{2})$ (resp. $u_{-}\in L^{2}(\R^{2})$).
		\item\label{item:LWP_mc}
		$u$ conserves the mass functional: $M(u(t_{1})) = M(u(t_{2}))$ for all $t_{1},t_{2}\in I$.
\end{enumerate}
\end{restatable}

We now proceed to discuss the possible obstructions to global well-posedness and scattering for solutions to \eqref{eq:DS} with initial data in $L^{2}(\R^{2})$ having arbitrarily large mass. So far in the local and small data global theory, we have not distinguished between whether the equation \eqref{eq:DS} is defocusing ($\mu=1$) or focusing ($\mu=-1$); however, we do so now.

From assertion \ref{item:LWP_smg} of theorem \ref{thm:LWP}, we know that for $\mu=\pm 1$, maximal-lifespan solutions with mass below the threshold $m_{0}^{2}$ are global, do not blow up either forward or backward in time, and scatter. However, in the focusing case $\mu=-1$, there is a mass threshold for which finite-time blowup may occur for solutions with mass above that threshold. More precisely, Ghidaglia and Saut showed the \emph{virial identity}
\begin{equation} %Virial identity
\frac{d^{2}}{dt^{2}}V(t) = 8E(u_{0}) = 8\int_{\R^{2}}\paren*{\frac{1}{2}|\nabla u(t,x)|^{2}+\frac{1}{4}\left(\mu |u(t,x)|^{2}-\E(|u|^{2})(t,x)\right)|u(t,x)|^{2}}dx,
\end{equation}
where $V$ denotes the variance of the solution $u$ defined by
\begin{equation} %Variance of solution
V(t) \coloneqq \int_{\R^{2}} |x|^{2}|u(t,x)|^{2}dx, \qquad \jp{x}u\in L^{2}(\R^{2}).
\end{equation}
When $\mu=1$, the energy functional $E$ is positive definite by Plancherel's theorem, and this property can be used to obtain a global $C_{t}^{0}H_{x}^{1}(\R\times\R^{2})$ solution for initial data in $H^{1}(\R^{2})$. However, when $\mu=-1$, the energy functional is no longer necessarily positive definite, and if the energy is negative, then one necessarily has finite-time blowup by the standard Glassey argument (\cite{Glassey1977}). Moreover, by taking as one's initial data the function
\begin{equation}
u_{0,\lambda}(x_{1},x_{2}) := \lambda\exp\left(-\left(\frac{x_{1}^{2}}{\beta_{1}^{2}}+\frac{x_{2}^{2}}{\beta_{2}^{2}}\right)\right),
\end{equation}
for an appropriate choice of $\lambda,\beta_{1},\beta_{2}>0$, one obtains a solution $u$ with negative energy.

Looking for standing wave solutions of \eqref{eq:DS}, where $\mu=-1$, of the form
\begin{equation}
u(t,x) = e^{i\tau t}Q(x),
\end{equation}
for some temporal frequency $\tau$ and real-valued function $Q$, one finds that $Q$ solves the \emph{ground state equation}
\begin{equation}
\Delta Q -\tau Q+ \mathcal{L}(Q^{2})Q=0, \qquad \mathcal{L}(Q^{2}) \coloneqq Q^{2}+\frac{\partial_{1}^{2}}{\Delta}(Q^{2}).\label{eq:GSE}
\end{equation}
By extending the analysis for standing wave solutions of the NLS, Cipolatti (\cite{Cipolatti1992}, \cite{Cipolatti1993}) proved the existence of solutions to \eqref{eq:GSE}.

\begin{thm}[Existence of ground states, \cite{Cipolatti1992}, \cite{Cipolatti1993}] %Existence of GSE solutions
Let $\mathcal{X}$ denote the subset of $H^{1}(\R^{2})$	satisfying \eqref{eq:GSE}, and let $\mathcal{G}$ denote the set of minimizers of the functional
\begin{equation}
S(f) \coloneqq \int_{\R^{2}}\frac{1}{2}|\nabla f|^{2}-\frac{1}{4}|f|^{2}\mathcal{L}(|f|^{2})+\frac{\tau}{2}|f|^{2}dx.
\end{equation}
\begin{enumerate}[(i)]
\item
$\mathcal{G}$ contains a real positive function.
\item
$Q\in \mathcal{G}$ if and only if
\begin{equation}
I(Q) = \min_{f\in H^{1}(\R^{2}); V(f)=0} \int_{\R^{2}} |\nabla f|^{2}dx,
\end{equation}
where
\begin{equation}
V(f) \coloneqq \int_{\R^{2}}\paren*{\frac{1}{2}|f|^{4}-\frac{1}{2}|f|^{2}\E(|f|^{2})-\tau |f|^{2}}dx.
\end{equation}
\item
If $Q$ is a solution to \eqref{eq:GSE}, then $Q\in C^{2}(\R^{2})$ and is rapidly decaying in the sense
\begin{equation}
|Q(x)| + |\nabla Q(x)| \lesssim e^{-\nu x}, \qquad \forall x\in\R^{2},
\end{equation}
for some $\nu>0$.
\end{enumerate}
\end{thm}

Following the approach developed by Weinstein in \cite{weinstein1982} for characterizing the ground state associated to the NLS energy functional in terms of the sharp constant constant for the Gagliardo-Nirenberg inequality, Papanicolau, Sulem, Sulem, and Wang \cite{papanicolaou1994focusing} characterized solutions of \eqref{eq:GSE} (for $\tau=1$) in terms of a Gagliardo-Nirenberg-type inequality for the operator $\mathcal{L}$.

\begin{thm}[Gagliardo-Nirenberg inequality for $\L$, \cite{papanicolaou1994focusing}]\label{thm:GN} %GN ineq for DS
The optimal constant $C_{opt,\L}$ for the inequality
\begin{equation}
|\ipp{\mathcal{L}(|u|^{2}), |u|^{2}}_{L^{2}(\R^{2})}| \leq C_{opt,\L} \|u\|_{L^{2}(\R^{2})}^{2} \|\nabla u\|_{L^{2}(\R^{2})}^{2},
\end{equation}
is $C_{opt,\L} = \frac{2}{\|Q\|_{L^{2}(\R^{2})}^{2}}$, where $Q$ is a positive solution of the equation
\begin{equation}
\Delta Q - Q +\L(Q^{2})Q=0.
\end{equation}
\end{thm}

By defining the spacetime function $u(t,x) \coloneqq e^{it}Q(x)$, where $Q$ solves the ground state equation \eqref{eq:GSE}, one obtains a solution to \eqref{eq:DS} in the focusing case, which blows up both forward and backward in time. Furthermore, by applying a pseudoconformal transformation to $u$, one obtains a solution which blows up in finite time. For initial data in $H^{1}(\R^{2})$ with mass strictly below that of the ground state, one sees from theorem \ref{thm:GN} that the energy functional is positive definite and controls the $\dot{H}^{1}(\R^{2})$ norm. Therefore, by the same argument as in the defocusing case, the solution to \eqref{eq:DS} is global. However, in neither the defocusing nor focusing case does scattering follow from this $H^{1}$ global result.

\begin{remark} %Open problem of uniqueness of groundstate
To the author's knowledge, uniqueness of solutions to equation \eqref{eq:GSE} is not known; however, the mass of solutions to \eqref{eq:GSE} is unique and therefore our notation $M(Q)$ is unambiguous.
\end{remark}

\subsection{Main result}

In the defocusing case $\mu=1$, there are no obvious obstructions to global well-posedness and scattering, and we therefore expect solutions to \eqref{eq:DS} to be global and scatter for arbitrary initial data in $L^{2}(\R^{2})$. However, in the focusing case $\mu=-1$, the mass of the ground state presents a clear threshold for global existence of solutions with initial data in $L^{2}(\R^{2})$. Since there are no clear obstructions to global well-posedness and scattering strictly below the ground state threshold, we expect solutions to be global and scatter for initial data in $L^{2}(\R^{2})$ with mass strictly below that of the ground state. Thus, we are led to formulate what we call the \emph{$L^{2}$ large data problem (LDP)} not addressed by theorem \ref{thm:LWP}.

\begin{prob}[$L^{2}$ Large Data Problem (LDP)]\label{prob:LDP} %Formulation of L2 LDP
Show the following:
\begin{itemize}
\item
If $\mu=1$, then solutions to the Cauchy problem \eqref{eq:DS} with initial data $u_{0}\in L^{2}(\R^{2})$ are global and scatter.
\item
If $\mu=-1$, then solutions to the Cauchy problem \eqref{eq:DS} with initial data $u_{0}\in L^{2}(\R^{2})$ satisfying $M(u_{0}) < M(Q)$ are global and scatter.
\end{itemize}
\end{prob}

That Ghidaglia and Saut were unable to address the $L^{2}$ LDP in \cite{Ghidaglia1990} is not so surprising, as little was known in the early 1990s about the global theory for the related cubic NLS with large $L^{2}(\R^{2})$ initial data, presumably for which the analysis should be more straightforward. It is only until the last fifteen to twenty years that great advances have been made on large data critical problems in dispersive PDE, in particular on the model pNLS equation
\begin{equation}
(i\p_{t}+\Delta)u = \mu|u|^{p-1}u, \qquad (t,x)\in\R\times\R^{d}, \enspace \mu\in\{\pm 1\}.
\end{equation}
To our knowledge, the best result for GWP of solutions to the Cauchy problem \eqref{eq:DS} is for initial data $u_{0}\in H^{4/7^{+}}(\R^{2})$, which is due to Shen and Guo in \cite{Shen2006}. This work relies on first-generation almost conservation law techniques in the spirit of the I-team's work \cite{colliander2002almost}. Almost conservation law techniques have benefited from numerous refinements over the years (e.g. resonant decompositions, interaction Morawetz estimates, etc.), yielding improved GWP results, in the context of NLS equations, in particular the 2D cubic NLS. However, it has not been clear how to implement these refinements in the context of the eeDS system for reasons which will be discussed later.

If we instead consider the special hyperbolic-elliptic case of the Davey-Stewartson system which is the defocusing DSII equation
\begin{equation}\label{eq:d_DSII}
i\partial_{t}u+(\partial_{2}^{2}-\partial_{1}^{2})u=\frac{\partial_{2}^{2}-\partial_{1}^{2}}{\Delta}(|u|^{2})u,
\end{equation}
then the $L^{2}$ LDP has been recently solved by Nachman, Regev, and Tataru in \cite{Nachman2017}.
\begin{thm}[Defocusing DSII GWPS, \cite{Nachman2017}]
\label{thm:d_DSII}
Solutions to \eqref{eq:d_DSII} are global and satisfy the uniform the spacetime estimate
\begin{equation}
\|u\|_{L_{t,x}^{4}(\R\times\R^{2})} \leq C\left(\|u_{0}\|_{L^{2}(\R^{2})}\right),
\end{equation}
where $C(\cdot)$ is a nondecreasing function such that $C(0)=0$. Moreover, solutions scatter both forward and backward in time.
\end{thm}
We remind the reader that the DSII equation is one of two cases for which the Davey-Stewartson system is completely integrable by the IST and that neither of these two cases includes the elliptic-elliptic regime.

Having discussed the $L^{2}$ large data problem for the elliptic-elliptic Davey-Stewartson system, we now state our main result, which is the following theorem.
\begin{thm}[Main result]\label{thm:main}
If $\mu=1$, then solutions of the defocusing eeDS equation \eqref{eq:DS} are global and satisfy the uniform spacetime estimate
\begin{equation}
\|u\|_{L_{t,x}^{4}(\R\times\R^{2})} \leq C\left(\|u_{0}\|_{L^{2}(\R^{2})}\right),
\end{equation}
where $C(\cdot):[0,\infty)\rightarrow [0,\infty)$ is a nondecreasing function such that $C(0)=0$. Moreover, solutions scatter both forward and backward in time.	

If $\mu=-1$ and $M(u_{0})<M(Q)$, then solutions of the focusing eeDS equation \eqref{eq:DS} are global and satisfy the uniform spacetime estimate
\begin{equation}
\|u\|_{L_{t,x}^{4}(\R\times\R^{2})} \leq C\left(\|u_{0}\|_{L^{2}(\R^{2})}\right).
\end{equation}
Moreover, solutions scatter both forward and backward in time.
\end{thm}

\begin{remark} %Generalized eeDS remark
Although our theorem \ref{thm:main} is formulated for solutions to the simplified eeDS equation \eqref{eq:DS} and in our proof we work with \eqref{eq:DS}, our argument extends to the full elliptic-elliptic system \emph{mutatis mutandis}. Our motivation for restricting attention to \eqref{eq:DS} is to simplify the notation.
\end{remark}

\begin{remark} %H^s scattering remark
By a standard argument (see subsection 4.4.5 of \cite{Sohinger2011}), one can show that scattering in $H^{s}$, for $s>0$, is a consequence of our theorem \ref{thm:main}. It is worth remarking that to address the global behavior of solutions to \eqref{eq:DS} at high regularities it was necessary for us to the study the equation at the critical regularity.
\end{remark}

\begin{remark} %Hartree remark
Forthcoming work by the author (\cite{RosenzweigHartree}) solves the $L^{2}$ LDP for the 3D mass-critical Hartree equation
\begin{equation}
(i\p_{t}+\Delta)u=\mu (|x|^{-2}\ast |u|^{2})u, \qquad (t,x)\in \R\times\R^{3}, \enspace \mu\in\{\pm 1\},
\end{equation} 
for which one faces the new challenge of an asymmetric natural scattering norm $L_{t}^{6}L_{x}^{18/7}(I\times\R^{3})$.
\end{remark}

Before proceeding to discuss the outline of the proof of our main result, we briefly compare our theorem \ref{thm:main} for the elliptic-elliptic DS to the theorem \ref{thm:d_DSII} of Nachman, Regev, and Tataru for the defocusing DSII, which we remind the reader is of hyperbolic-elliptic type. First, our work covers both the defocusing and focusing cases, whereas \cite{Nachman2017} only covers the defocusing case. Second, the proofs are very different. The main result of \cite{Nachman2017} are new estimates for pseudodifferential operators with rough symbols. From these new estimates together with the integrable structure of the DSII, the authors are able to solve the $L^{2}$ LDP for the defocusing DSII equation and in addition obtain a precise description of the asymptotic behavior of the solution. In our case, we make no use of any integrable structure, as the eeDS is known to \emph{not} be integrable by the IST (\cite{Shulman1983}). Our proof is instead inspired by the work of Dodson on the 2D cubic NLS (\cite{Dodson2016}) and follows very much in the tradition of the concentration compactness/rigidty roadmap first laid out by Kenig and Merle (\cite{Kenig2006}), which we review in the next subsection. We view our main result, theorem \ref{thm:main}, not only as solving the large data problem for a specific equation, but also as making the existing theory for tackling critical large data problems more robust with respect to the local/nonlocal distinction.

\subsection{Outline of the proof}\label{ssec:in_out} %Outline of the proof
\subsubsection{Review of global theory for critical pNLS}
To give context to our main result and introduce tools that we use in its proof, we briefly review some of the advances that have been made over the last two decades in the study of the pNLS (hereafter referred to as NLS) equation at critical regularity. The first major advance was by Bourgain, who proved in \cite{Bourgain1999} global spacetime bounds for the defocusing energy-critical NLS in dimensions $d=3,4$ with spherically symmetric initial data. In this work, he introduced the tool of ``induction on energy", which can be seen as the hallmark tool of the old paradigm for studying the critical NLS LDP. Moreover, Bourgain oberved that to prove global spacetime bounds, it suffices to consider solutions which are concentrated in both physical and frequency spaces. The work \cite{Grillakis2000} gave a different proof for the 3D radial defocusing energy-critical NLS, and the work \cite{Tao2005} treated the higher dimensional cases, also under the assumption of radial initial data. Colliander, Keel, Staffilani, Takaoka, and Tao (\cite{Colliander2008}) removed the radial symmetry assumption by significantly refining the induction on energy technique. The authors also introduced the new tool of a frequency-localized interaction Morawetz inequality, which has proved useful for studying the critical NLS LDP outside the energy-critical setting. \cite{ryckman2007global}, \cite{visan2006thesis}, and \cite{visan2007} treated the higher dimensional cases, resolving the conjectured global spacetime bounds for the defocusing energy-critical NLS.

The next advance, which we highlight, comes from Kenig and Merle in \cite{Kenig2006} in the context of the focusing energy-critical NLS dimensions $d=3,4,5$ with spherically symmetric initial data. In that work, the authors used earlier work by Keraani (\cite{Keraani2001}, \cite{Keraani2006}) to prove via a concentration compactness argument that the failure of global spacetime bounds implies that the existence of minimal energy blowup solutions, which enjoy even better spatial and frequency localization properties than the almost minimal blowup solutions previously considered in the induction on energy approach. By means of a rigidity argument, they then proved that such solutions ultimately cannot exist. Kenig and Merle's new approach provided a general framework for studying critical large data dispersive problems called the concentration compactness/rigidity road map, which, due to its efficiency and modularity, has largely supplanted the prior induction on energy approach in subsequent work. \cite{killip2007unpub} extended the result of Kenig and Merle to all dimensions, and \cite{Killip2010} removed the assumption of spherically symmetric initial data for dimensions $d\geq 5$.

Turning to the mass-critical NLS, Tao, Visan, and Zhang made in \cite{Tao2006} the first advance on the critical LDP for the defocusing case with spherically symmetric initial data in dimensions $d\geq 3$. Killip, Tao, and Visan in \cite{killip2009cubic} then solved the 2D defocusing and focusing cases also for spherically symmetric initial data. In particular, that work showed that by further appealing to the concentration compactness, one can prove the existence of minimal mass blowup solutions with even better properties than had been previously observed. \cite{Killip2008} adapted the argument from \cite{killip2009cubic} to cover the focusing case in dimensions $d\geq 3$, again only for spherically symmetric initial data.

The previously cited works on the mass-critical NLS heavily rely on the assumption of spherically symmetric initial data. In the radial case, one has access to improved Strichartz estimates, the in/out decomposition, and other tools, and one also does not have to worry about Galilean invariance, which is an additional noncompact symmetry compared to the energy-critical NLS. The breakthrough in removing the radial assumption came from Dodson in the seminal works \cite{Dodson2012}, \cite{dodson2016global}, \cite{Dodson2016} on the defocusing mass-critical NLS in all dimensions. In these works, Dodson introduced the new tool of long-time Strichartz estimates for special minimal blowup solutions. This tool has since found application outside the mass-critical setting by simplifying proofs of old results for the energy-critical NLS (see \cite{Visan2012}, \cite{killip2012global}). In the work \cite{Dodson2015}, he treated the focusing mass-critical NLS in all dimensions, by introducing a substitute for the frequency-localized interaction Morawetz inequality previously used in the rigidity step of the defocusing case. Dodson's 2D result \cite{Dodson2016} is, in particular, a very deep theorem, as one has the failure of the double endpoint $L_{t}^{2}L_{x}^{\infty}$ Strichartz estimate (\cite{montgomery-smith1998}) and therefore Dodson had to compensate by relying heavily on bilinear estimates and special function spaces. This work is the most relevant to our present setting and is, in part, the inspiration for our interest in the $L^{2}$ LDP for the elliptic-elliptic Davey-Stewartson system.

To prove our main result \ref{thm:main}, we use the concentration compactness/rigidity road map first introduced by Kenig and Merle together with the long-time Strichartz technique introduced by Dodson. The guiding principle throughout the proof is to use as much of the Kenig-Merle/Dodson road map as possible. However, the existing theory does not accommodate the eeDS equation \eqref{eq:DS}. Therefore, we work to extend the theory to handle semilinear nonlocal nonlinearities, such as those of eeDS type.

\subsubsection{Concentration compactness}
For the purposes of solving the $L^{2}$ large data problem, it is convenient to reformulate it in a more quantitative fashion (cf. \cite{Kenig2006}, \cite{Tao2006}, \cite{Tao2008}, and \cite{Killip2009}). To do so, we consider the precise relationship between the mass $M(u)$ of a solution $u$ to \eqref{eq:DS} and the scattering size
\begin{equation}
S_{I}(u) \coloneqq \|u\|_{L_{t,x}^{4}(I\times\R^{2})}^{4}.
\end{equation}
In the defocusing case, $\mu=1$, we define the function $L^{+}:[0,\infty)\rightarrow [0,\infty]$ by
\begin{equation}\label{eq:L+}
L^{+}(M) \coloneqq \sup\{S_{I}(u) : u: I\times\R^{2}\rightarrow\mathbb{C}, \enspace M(u)\leq M\},
\end{equation}
where the supremum is taken over all solutions $u:I\times\R^{2}\rightarrow\mathbb{C}$ to the defocusing equation \eqref{eq:DS}. In the focusing case, $\mu=-1$, we define the function $L^{-}:[0,M(Q)] \rightarrow [0,\infty]$ by
\begin{equation}\label{eq:L-}
L^{-}(M) \coloneqq \sup\{S_{I}(u) : u: I\times\R^{2}\rightarrow\mathbb{C}, \enspace M(u)\leq M\},
\end{equation}
where the supremum is taken over all solutions $u:I\times\R^{2}\rightarrow\mathbb{C}$ to the focusing equation \eqref{eq:DS}. It is tautological that $L^{\pm}$ is nondecreasing. Moreover, from assertions \ref{item:LWP_smg} and \ref{item:LWP_cd} of theorem \ref{thm:LWP}, $L^{\pm}$ is continuous and there exists some $M_{0}\in (0,\infty)$ such that $L^{\pm}(M)<\infty$ for $M\leq M_{0}$. If there exists a maximal-lifespan solution $u:I\times\R^{2}\rightarrow\C$ to the defocusing eeDS such that $S_{I}(u)=\infty$, then there exists a \emph{critical mass} $M_{c}\in (0,\infty)$ such that
\begin{equation}
L^{+}(M) < \infty, \quad \forall M<M_{c} \enspace \text{and} \enspace L^{+}(M)=\infty, \quad \forall M>M_{c}.
\end{equation}
Similarly, if there exists a maximal-lifespan solution $u:I\times\R^{2}\rightarrow\C$ to the focusing eeDS such that $M(u)<M(Q)$ and $S_{I}(u)=\infty$, then there exists a critical mass $M_{c}\in (0,M(Q))$ such that
\begin{equation}
L^{-}(M)<\infty, \quad \forall M<M_{c} \enspace \text{and} \enspace L^{+}(M)=\infty, \quad \forall M>M_{c}.
\end{equation}

\begin{remark}
Hereafter, we omit the superscript $\pm$ in $L$, as the defocusing/focusing nature will be clear from context.
\end{remark}

With the preceding observations, we obtain the following quantitative version of the $L^{2}$ LDP.
\begin{prob}[Quantitative $L^{2}$ LDP]\label{prob:LDP_q}
Show the following:
\begin{itemize}
\item
If $\mu=1$, then $M_{c}=\infty$;
\item
If $\mu=-1$, then $M_{c}=M(Q)$.
\end{itemize}
\end{prob}

To solve problem \ref{prob:LDP_q}, we argue by contradiction. We assume that if $\mu=1$, then the critical mass $M_{c}<\infty$ and if $\mu=-1$, then $M_{c}<M(Q)$. Following the approach of \cite{Keraani2006}, \cite{Tao2006}, \cite{Tao2008}, and \cite{Killip2009}, we prove a stability lemma for the eeDS and use it together the linear profile decomposition for the free propagator $e^{it\Delta}$ due to Merle and Vega (\cite{Merle1998}) to prove the existence of a maximal-lifespan solution $u:I\times\R^{2}\rightarrow\mathbb{C}$ to \eqref{eq:DS} which blows up both forward and backward in time and has mass exactly equal to the critical mass $M_{c}$. We call such a solution to equation \eqref{eq:DS} a \emph{minimal mass blowup solution}. Moreover, we show that such solutions have the special property of being almost periodic modulo the group $G$ generated by the symmetries of phase rotation, spatial translation, $L^{2}$ scaling, and Galilean transformation, which means their orbits are precompact in the quotient of the action $L^{2}(\R^{2}){/} G$. We use the following equivalent definition of almost periodicity (see remark 3 succeeding Definition 5.1 in \cite{KillipClay}), which is more quantitative.

\begin{mydef}[Symmetry group $G$]\label{def:G_grp} %Definition of the symmetry group
For a phase $\theta\in \R/2\pi\mathbb{Z}$, position $x_{0}\in\R^{2}$, frequency $\xi_{0}\in\R^{2}$, and scaling parameter $\lambda>0$, we define the unitary transformation $g_{\theta,\xi_{0},x_{0},\lambda} : L^{2}(\R^{2}) \rightarrow L^{2}(\R^{2})$ by
\begin{equation}
g_{\theta,\xi_{0},x_{0},\lambda}f(x) \coloneqq \lambda^{-1}e^{i\theta}e^{ix\cdot\xi_{0}}f\paren*{\frac{x-x_{0}}{\lambda}}.
\end{equation}
We let $G$ denote the collection of such transformations. The reader may check that $G$ is a group with identity $g_{0,0,0,1}$, inverse $g_{\theta,\xi_{0},x_{0},\lambda}^{-1}=g_{-\theta-x_{0}\cdot\xi_{0},-\lambda\xi_{0},-x_{0}/\lambda, \lambda^{-1}}$, and group law
\begin{equation}
	g_{\theta,\xi_{0},x_{0},\lambda}g_{\theta',\xi_{0}',x_{0}',\lambda'}=g_{\theta+\theta'-x_{0}\cdot\xi_{0}'/\lambda, \xi_{0}+\xi_{0}'/\lambda, x_{0}+\lambda x_{0}', \lambda\lambda'}.
\end{equation}
We denote the quotient of the action of $G$ on $L^{2}(\R^{2})$ (i.e. the space of $G$-orbits $Gf:=\{gf : g\in G\}$ for $f\in L^{2}(\R^{2})$) by $L^{2}(\R^{2}){/}G$, which we endow with the quotient (complete) metric topology. For $g_{\theta,\xi_{0},x_{0},\lambda}\in G$, we define the action $T_{g_{\theta,\xi_{0},x_{0},\lambda}}$ on spacetime functions $u:I\times\R^{2}\rightarrow\mathbb{C}$ by
\begin{align}
&T_{g_{\theta,\xi_{0},x_{0},\lambda}}u:\lambda^{2}I\times\R^{2}\rightarrow\mathbb{C}, \qquad \lambda^{2}I \coloneqq \{\lambda^{2}t: t\in I\} \\
&(T_{g_{\theta,\xi_{0},x_{0},\lambda}}u)(t,x) \coloneqq \lambda^{-1}e^{i\theta}e^{ix\cdot\xi_{0}}e^{-it|\xi_{0}|^{2}}u\left(\frac{t}{\lambda^{2}}, \frac{x-x_{0}-2\xi_{0}t}{\lambda}\right),
\end{align}
which may be expressed more succinctly as
\begin{equation}
(T_{g_{\theta,\xi_{0},x_{0},\lambda}}u)(t) = g_{\theta-t|\xi_{0}|^{2},\xi_{0},x_{0}+2\xi_{0}t,\lambda}\left(u\left(\frac{t}{\lambda^{2}}\right)\right).
\end{equation}
\end{mydef}

\begin{mydef}[Almost periodic modulo symmetries (APMS)]
We say that $u\in C_{t,loc}^{0}L_{x}^{2}(I\times\R^{2})$ is \emph{almost periodic modulo $G$} if there exists a spatial center function $x:I\rightarrow\R^{2}$, a frequency center function $\xi:I\rightarrow\R^{2}$, a frequency scale function $N:I\rightarrow [0,\infty)$, and a compactness modulus function $C:(0,\infty)\rightarrow [0,\infty)$, such that for every $\eta>0$, we have the spatial and frequency localization estimates
\begin{equation}
\int_{|x-x(t)| \geq C(\eta)/N(t)} |u(t,x)|^{2}dx + \int_{|\xi-\xi(t)| \geq C(\eta)N(t)} |\hat{u}(t,\xi)|^{2}d\xi \leq \eta, \qquad\forall t\in I.
\end{equation}
\end{mydef}

\begin{restatable}[Reduction to almost periodic solutions]{thm}{APreduc}
\label{thm:AP_reduc}
If $\mu=1$, assume that $M_{c}<\infty$; if $\mu=-1$, assume that $M_{c}<M(Q)$. Then there exists a minimal mass blowup solution $u:I\times\R^{2}\rightarrow\mathbb{C}$ to \eqref{eq:DS} which is almost periodic modulo symmetries.
\end{restatable}

Next, using the further refinements of the concentration compactness obtained in \cite{Killip2009} (see also \cite{KillipClay}) for the $2D$ cubic NLS, we show the existence of a special class of minimal blowup solutions which are almost periodic modulo symmetries. We refer to these as \emph{admissible blowup solutions} and consider them exclusively in this work.

\begin{mydef}[Admissible blowup solution]\label{def:adm_sol}
We say that a maximal-lifespan solution $u:I\times\R^{2}\rightarrow\mathbb{C}$ to \eqref{eq:DS} which has mass $M_{c}$, blows up both forward and backward in time, and is almost periodic modulo symmetries with parameters $x(t),\xi(t),N(t)$ and compactness modulus function $C(\cdot)$ is \emph{admissible} if the following properties are satisfied:
\begin{enumerate}[(i)]
\item
$[0,\infty)\subset I$;
\item
$N(t)\leq 1$ for all $t\in [0,\infty)$ and $x(0)=\xi(0)=0$;
\item
$N,\xi \in C_{loc}^{1}([0,\infty))$ and satisfy the pointwise derivative bounds
\begin{equation}
|N'(t)| + |\xi'(t)| \lesssim_{u} N(t)^{3}, \qquad \forall t\in [0,\infty).
\end{equation}
\end{enumerate}
\end{mydef}

\begin{restatable}[Reduction to admissible blowup solutions]{cor}{admsol}
\label{cor:adm_sol}
If $\mu=1$, assume that $M_{c}<\infty$; if $\mu=-1$, assume that $M_{c}<M(Q)$. Then there exists an admissible blowup solution $u:I\times\R^{2}\rightarrow\C$ to \eqref{eq:DS}.
\end{restatable}

To obtain a contradiction, it suffices to show that $u\equiv 0$. Using the dichotomy introduced in \cite{Dodson2012}, \cite{dodson2016global}, and \cite{Dodson2016}, we consider two scenarios for blowup, which we ultimately preclude. The first is the \emph{rapid frequency cascade} scenario
\begin{equation}
\int_{0}^{\infty}N(t)^{3}dt<\infty,
\end{equation}
and the second is the \emph{quasi-soliton} scenario
\begin{equation}
\int_{0}^{\infty}N(t)^{3}dt=\infty.
\end{equation}
The motivation for considering these two scenarios comes from the scaling of the interaction Morawetz estimate for the NLS.

\begin{remark}
The concentration compactness step may also be performed for the hyperbolic-elliptic DS system (heDS). The work \cite{dodson2017profile} established a profile decomposition for the free propagator $e^{it(\p_{2}^{2}-\p_{1}^{2})}$, following the earlier work \cite{Rogers2006} which proved a refined Strichartz estimate. Since the linear Strichartz estimates for $e^{it(\p_{2}^{2}-\p_{1}^{2})}$ are the same as the estimates for $e^{it\Delta}$, one can also prove a stability result for the heDS Cauchy problem, which then provides one with all the necessary ingredients to prove the existence of minimal blowup solutions. 
\end{remark}

\subsubsection{Long-time Strichartz estimate}
After the concentration compactness step, we next prove a long-time Strichartz estimate for admissible blowup solutions to equation \eqref{eq:DS}, which is theorem \ref{thm:LTSE}. This estimate is the workhorse of the overall proof, and we use it to preclude both the rapid frequency cascade and the quasi-soliton scenarios. We will not state the theorem here, as formulating it requires a bit of machinery; but we comment that we use the $U^{p}$ and $V^{p}$ spaces introduced by Koch and Tataru in \cite{Koch2004} and the norms introduced by Dodson in \cite{Dodson2016} to compensate for the failure of the double endpoint $L_{t}^{2}L_{x}^{\infty}$ Strichartz estimate. The proof of theorem \ref{thm:LTSE} relies heavily on Littlewood-Paley theory adapted to the frequency center $\xi(t)$ and scale $N(t)$ together with bilinear Strichartz estimates which give an improvement over the classical estimate obtained by Bourgain (\cite{Bourgain1998}; see proposition \ref{prop:bs_B} below). We prove these bilinear estimates using the interaction Morawetz technique of Planchon and Vega in \cite{Planchon2009}, which was heavily exploited by Dodson in \cite{Dodson2016}. We emphasize that we do not use the long-time Strichartz estimate of \cite{Dodson2016}, as that estimate is specific to the 2D cubic NLS. Moreover, its proof exploits local structure of the cubic NLS not shared by the eeDS.

\subsubsection{Rigidity}
With the long-time Strichartz estimate in hand, we proceed to the rigidity step of precluding the two scenarios for blowup. To preclude the rapid frequency cascade scenario, we follow the work \cite{Dodson2016} by combining our long-time Strichartz estimate together with the additional regularity argument from the earlier works \cite{Tao2008} and \cite{Killip2009} on the mass-critical NLS. More precisely, we use theorem \ref{thm:LTSE} to show that an admissible blowup solution $u:I\times\R^{2}\rightarrow\mathbb{C}$ must possess additional regularity.

\begin{restatable}[$H_{x}^{3}$ regularity]{lemma}{areg}
\label{lem:a_reg}
If $u$ is an admissible blowup solution such that $\int_{0}^{\infty}N(t)^{3}=K<\infty$, then $u\in C_{t}^{0}H_{x}^{3}([0,\infty)\times\R^{2})$ and satisfies the estimate
\begin{equation}
\sup_{0\leq t<\infty} \|u(t)\|_{\dot{H}_{x}^{3}(\R^{2})} \lesssim K^{3}.
\end{equation}
\end{restatable}

This additional regularity implies that the energy of a Galilean transformation of the solution must tend to zero as time tends to $\infty$. Conservation of energy then implies that the solution is identically zero, which is a contradiction.

\begin{restatable}[No rapid frequency cascade]{thm}{norfc}
\label{thm:no_rfc}
There does not exist an admissible blowup solution such that $\int_{0}^{\infty}N(t)^{3}dt=K<\infty$.
\end{restatable}

To preclude the quasi-soliton scenario, we prove a frequency-localized interaction Morawetz ``type" estimate for admissible solutions to equation \eqref{eq:DS} under the specific assumption that $\int_{0}^{\infty}N(t)^{3}dt=\infty$. On an interval $[0,T]$, we show that our Morawetz functional $M(t)$ is bounded from below $\int_{0}^{T}N(t)^{3}dt=K$ and bounded from above by a quantity which is $o(K)$. Since $K$ may be taken arbitrarily large (by taking $T$ arbitrarily large) in the quasi-soliton scenario, we obtain a contradiction. 

\begin{restatable}[No quasi-soliton]{thm}{noqs}
\label{thm:no_qs}
There does not exist an admissible blowup solution such that $\int_{0}^{\infty}N(t)^{3}dt=\infty$.
\end{restatable}

Recall that frequency-localized interaction Morawetz estimates were first introduced in \cite{Colliander2008} in the context of the 3D defocusing energy-critical NLS. There, the solution is truncated to high frequencies, whereas in the mass-critical case, the solution is truncated to low frequencies. We emphasize that this estimate is \emph{not} an a priori estimate for sufficiently regular solutions to the eeDS but only an estimate for a special class of solutions which we ultimately show do not exist. Frequency truncation of the solution introduces error terms which we can control with our long-time Strichartz estimate using an argument of \cite{Dodson2016}, similar in spirit to the ``almost Morawetz" estimates often used in conjunction with the I-method (\cite{Colliander2007}, \cite{dodson2009almost}, \cite{dodson2011improved}). 

\subsection{New difficulties in the DS setting}
Let us now comment on some of the difficulties of the proof and how our work differs from the existing literature, in particular Dodson's work \cite{Dodson2016} on the 2D cubic NLS.

\subsubsection{Difficulty 1: Asymptotically orthogonal group actions}
The concentration compactness step proceeds fairly similarly to the work \cite{Tao2008}; however, dealing with asymptotically orthogonal (see definition \ref{def:as_o}) sequences of symmetry group actions $g_{n}\in G'$ is a bit more delicate compared to the case for the 2D cubic NLS. Since we cannot simply appeal to associativity as with an algebraic nonlinearity, we have to consider cases of the ordering of asymptotically orthogonal symmetry actions in the eeDS nonlinearity: whether the asymptotic orthogonal occurs inside or outside the argument of the nonlocal operator $\E$ (e.g. equation \eqref{eq:aso_cases}). To prove lemma \ref{lem:as_sol} on the asymptotic solvability of equation \eqref{eq:DS} by the approximate solutions constructed from the nonlinear profiles, we must perform a careful case analysis of the definition of asymptotic orthogonality together with using the boundedness of the operator $\E$ on $L^{p}$ and $\dot{C}^{\alpha}$ spaces.

\subsubsection{Difficulty 2: Bilinear estimates}
The proof of our long-time Strichartz estimate (theorem \ref{thm:LTSE}) acquires a new level of difficulty compared to the 2D cubic NLS when we need to apply \emph{bilinear estimates} in order to prove a bootstrap lemma used to close the proof of the inductive step. Here, the issue is that the eeDS nonlinearity is \emph{not} permutation invariant under frequency decompositions, modulo complex conjugates (c.c.). Since the eeDS nonlinearity is not algebraic, we cannot appeal to associativity to group high and low frequency factors as we please. Not only do we need to know the total number of high and low frequency factors present in the decomposed nonlinearity (near and far, if the center of Littlewood-Paley projectors is not the origin), we also have to exercise care about whether they fall inside or outside the argument of the nonlocal operator $\E$.

Let us illustrate this difficulty with a toy example. Let $N_{lo}$ and $N_{hi}$ be two dyadic frequencies with $N_{lo}\ll N_{hi}$, and suppose that we wish to estimate the quantities
\begin{align}
&\|P_{N_{hi}}\brak*{\E\paren*{(P_{N_{lo}}u) (\ol{P_{N_{hi}}u})} (P_{N_{hi}}u)}\|_{L_{t,x}^{4/3}(J\times\R^{2})}\\
&\|P_{N_{hi}}\brak*{\E\paren*{|P_{N_{lo}}u|^{2}} (P_{N_{hi}}u)}\|_{L_{t,x}^{4/3}(J\times\R^{2})}
\end{align}
using bilinear Strichartz estimates. We dualize the problem to consider the quantities
\begin{align}
\mathrm{Case}_{1} &= \|(P_{N_{hi}}v)P_{N_{hi}}\brak*{\E\paren*{(P_{N_{lo}}u)(\ol{P_{N_{hi}}u})} (P_{N_{hi}}u)}\|_{L_{t,x}^{1}(J\times\R^{2})}\\
\mathrm{Case}_{2} &= \|(P_{N_{hi}}v)P_{N_{hi}}\brak*{\E\paren*{|P_{N_{lo}}u|^{2}} (P_{N_{hi}}u)}\|_{L_{t,x}^{1}(J\times\R^{2})}.
\end{align}
To estimate $\mathrm{Case}_{1}$, we can use Cauchy-Schwarz and Plancherel's theorem to obtain that
\begin{equation}
\mathrm{Case}_{1} \leq \|(P_{N_{hi}}v) (P_{N_{lo}}u)\|_{L_{t,x}^{1}(J\times\R^{2})} \| (P_{N_{hi}}u) (P_{N_{lo}}u)\|_{L_{t,x}^{1}(J\times\R^{2})}
\end{equation}
and then close with two applications of the classical bilinear Strichartz estimate of Bourgain. To estimate $\mathrm{Case}_{2}$, we somehow need to pair $P_{N_{lo}}u$ inside the nonlocal operator $\E$ with $P_{N_{hi}}u$ outside $\E$ and $\ol{P_{N_{lo}}u}$ inside $\E$ with $P_{N_{hi}}v$ outside $\E$ without destroying the cancellation in the Schwartz kernel $\mathcal{K}$ of $\E$. Taking absolute values of everything and using Minkowski's inequality is the worst thing to do, as the modulus of the  Schwartz kernel of $\E$ decays like $|x|^{-2}$ at infinity, which is not in $L^{1}$.

To overcome this obstruction, we use an idea of Chae, Cho, and Lee from \cite{Chae2010} and introduce additional frequency decompositions of the nonlinearity $\E\left(|P_{N_{lo}}u|^{2}\right) (P_{N_{hi}}u)$. We call this tool the \emph{double frequency decomposition}. We illustrate the main steps of this procedure for the toy example below.

First, we perform a homogeneous Littlewood-Paley decomposition of the symbol of $\E$
\begin{equation}
\E=\sum_{k\in\mathbb{Z}} \dot{P}_{k}\E =: \sum_{k\in\mathbb{Z}} \E_{k}.
\end{equation}
Since the symbol of $\E$ is $C^{\infty}$ outside the origin and homogeneous of degree zero, the kernel $\mathcal{K}_{k}$ of $\E_{k}$ is Schwartz class and satisfies the $k$-uniform rapid decay estimates
\begin{equation}
|\mathcal{K}_{k}(x)| \lesssim_{M} 2^{2k}\langle{2^{k}x}\rangle^{-M}, \qquad \forall M >0.
\end{equation}
By the triangle inequality, we have that
\begin{equation}
\mathrm{Case}_{2} \leq \sum_{k: 2^{k}\leq 4N_{lo}} \|(P_{N_{hi}}v) \E_{k}\left( |P_{N_{lo}}u|^{2}\right) (P_{N_{hi}}u)\|_{L_{t,x}^{1}(J\times\R^{2})}.
\end{equation}
Provided that we can gain some ``smallness" in $k$ (i.e. a factor of $2^{\delta k}$ for some $\delta>0$) and that our final estimates are $k$-uniform, it suffices to consider each term in the low-frequency summation.

To gain the needed smallness, we perform a second level of frequency decomposition, this time into frequency cubes. More precisely, let $\{Q_{a}^{k}\}_{a\in\mathbb{Z}^{2}}$ denote the collection of dyadic cubes of side length $2^{k}$ which tile Fourier space $\hat{\R}^{2}$, and let $P_{Q_{a}^{k}}$ denote Fourier projection onto the set $Q_{a}^{k}$. For each $k$, we decompose
\begin{equation}
|P_{N_{lo}}u|^{2} = \sum_{a,a'\in\mathbb{Z}^{2}} (P_{N_{lo}}P_{Q_{a}^{k}}u) \ol{(P_{N_{lo}}P_{Q_{a'}^{k}}u)}.
\end{equation}
A priori, this last step seems ill-advised because now we have to consider the interaction of lots of cubes. However, when we consider the expression
\begin{equation}
\sum_{a,a'\in\Z^{2}}\E_{k}\left( (P_{N_{lo}}P_{Q_{a}^{k}}u) \ol{(P_{N_{lo}}P_{Q_{a'}^{k}}u)}\right),
\end{equation}
we observe that the cubes are \emph{almost orthogonal}. Indeed,
\begin{equation}
\E_{k}\left( (P_{N_{lo}}P_{Q_{a}^{k}}u) (\ol{P_{N_{lo}}P_{Q_{a'}^{k}}u})\right)\neq 0 \Longrightarrow \dist\left(Q_{a}^{k}, Q_{a'}^{k}\right) \leq 2^{k+10}.
\end{equation}
Now that we have extracted some cancellation, we do the crudest thing possible by using Minkowski's inequality and Cauchy-Schwarz to estimate
\begin{equation}
\begin{split}
&\sum_{|a+a'|\leq 2^{10}} \|(P_{N_{hi}}v)\E_{k}\left( (P_{N_{lo}}P_{Q_{a}^{k}}u) \ol{(P_{N_{lo}}P_{Q_{a'}^{k}}u)}\right)\|_{L_{t,x}^{1}(J\times\R^{2})}\\
&\phantom{=} \lesssim \int_{\R^{2}}dy 2^{2k}\langle{2^{k}y}\rangle^{-10}\left(\sum_{a\in\Z^{2}} \|(P_{N_{hi}}v) (P_{N_{lo}}P_{Q_{a}^{k}}\tau_{y}u)\|_{L_{t,x}^{2}(J\times\R^{2})}^{2}\right)^{1/2} \left(\sum_{a\in\Z^{2}} \| (P_{N_{hi}}u) (P_{N_{lo}}P_{Q_{a}^{k}}\tau_{y}u)\|_{L_{t,x}^{2}(J\times\R^{2})}^{2}\right)^{1/2}.
\end{split}
\end{equation}
We now use Galilean invariance to apply Bourgain's bilinear Strichartz estimate at the level of each cube to gain two factors of $(2^{k}/N_{hi})^{1/2}$. Thus, the RHS of the preceding inequality is
\begin{equation}
\lesssim \left(\frac{2^{k}}{N_{hi}}\right) \|u\|_{U_{\Delta}^{2}(J\times\R^{2})}\|v\|_{U_{\Delta}^{2}(J\times\R^{2})} \|P_{N_{lo}}u\|_{X_{\Delta}^{2}(J\times\R^{2})}^{2}.
\end{equation}
We can sum over the integers $k$ such that $2^{k}\lesssim 4N_{lo}$ to complete the argument.

The preceding toy example illustrates fairly well the argument needed to apply classical bilinear estimates (e.g. \cite{Bourgain1998}, \cite{Tao2003}). However, we do not know how to close the proof of theorem \ref{thm:LTSE} using the existing bilinear estimates. Instead, we use an idea of Dodson from \cite{Dodson2016} which is to prove new bilinear estimates for admissible blowup solutions to \eqref{eq:DS}, which give a logarithmic improvement over the classical bilinear estimates. We emphasize that we \emph{do not use} Dodson's bilinear estimates, as his estimates are specific to the cubic NLS. Moreover, because we need to use the double frequency decomposition to have any hope of applying whatever bilinear estimates we prove, we need bilinear estimates where the low frequency scale is that of side length of each cube $Q_{a}^{k}$. This was \emph{not} the case in \cite{Dodson2016}, where the estimates were proved at the low frequency scale of the Littlewood-Paley projector $P_{N_{lo}}$.

We prove our bilinear Strichartz estimates using the interaction Morawetz technique of \cite{Planchon2009}, which makes no use of the spacetime Fourier transform, but instead relies on integration by parts and the local conservation laws of an equation. A priori, such integration by parts arguments seem ill-suited for the eeDS equation because the nonlinearity is nonlocal. Moreover, it is not obvious that the eeDS conservation laws can be written in the usual divergence form, as is the case for the NLS. However, the eeDS nonlinearity is just a linear combination of products of Riesz transforms, which have good differential structure. Therefore, we can write the local mass and momentum conservation laws of the equation in the usual divergence form. There are some additional issues of certain remainder terms in the integration by parts arguments not vanishing as they would in the case of an algebraic nonlinearity, but we can handle these issues with careful commutator estimates.

\subsubsection{Difficulty 3: Interaction Morawetz estimate}
The last difficulty which we highlight in the introduction is the lack of an a priori interaction Morawetz estimate for \eqref{eq:DS}. We remind the reader that a frequency-localized interaction Morawetz estimate was used in \cite{Dodson2016} to preclude the quasi-soliton. We recall the 2D interaction Morawetz estimate from \cite{Colliander2009} and \cite{Planchon2009}: for a solution $u$ to the defocusing cubic NLS, we have that
\begin{equation}
\||\nabla|^{1/2}(|u|^{2})\|_{L_{t,x}^{2}(J\times\R^{2})}^{2} \lesssim \sup_{t\in J} \|u(t)\|_{L_{x}^{2}(\R^{2})}^{2}\|u(t)\|_{\dot{H}_{x}^{1/2}(\R^{2})}^{2}.
\end{equation}
To illustrate the difference between the defocusing cubic NLS and defocusing eeDS, let us recall part of the proof from \cite{Planchon2009}. Define a Morawetz action
\begin{equation}
M(t) \coloneqq 2\int_{\mathbb{S}^{1}}\int_{\R^{4}}\frac{(x-y)_{\omega}}{|(x-y)_{\omega}|}|u(t,y)|^{2}\Im{\bar{u}\partial_{\omega} u}(t,x)dxdy d\omega.
\end{equation}
Differentiating in time, using the local mass and momentum conservations
\begin{align}
\partial_{t}T_{00}&=\partial_{j}T_{0j}\\
\partial_{t}T_{jk}&=-\partial_{k}L_{jk}-\partial_{k}T_{jk},
\end{align}
and performing a number of integration by parts, we obtain that
\begin{equation}
\begin{split}
\frac{dM}{dt}(t) &= \int_{\mathbb{S}^{1}}\int_{\{x_{\omega}=y_{\omega}\}} |\partial_{\omega}\left(\Tr_{x_{\omega}=y_{\omega}}(\bar{u}\otimes u)\right)(t,x,y)|^{2}d\mathcal{H}^{3}(x,y)d\omega\\
&\phantom{=}+\int_{\mathbb{S}^{1}}\int_{\{x_{\omega}=y_{\omega}\}} |u(t,y)|^{2} T_{jk}(t,x)\omega_{j}\omega_{k}d\mathcal{H}^{3}(x,y)d\omega,
\end{split}
\end{equation}
where $\Tr_{x_{\omega}=y_{\omega}}$ denotes the restriction to the hypersurface $\{x_{\omega}=y_{\omega}\}\subset\R^{4}$. Now in the cubic NLS setting, $T_{jk}=\delta_{jk}|u|^{4}$ and therefore
\begin{equation}
\frac{dM}{dt}(t)\geq \int_{\mathbb{S}^{1}}\int_{\{x_{\omega}=y_{\omega}\}}|\partial_{\omega}\left(\Tr_{x_{\omega}=y_{\omega}}(\bar{u}\otimes u)\right)(t,x,y)|^{2}d\mathcal{H}^{3}(x,y)d\omega.
\end{equation}
But in the eeDS setting,
\begin{equation}
T_{jk} = \delta_{kj}|u|^{4}-\delta_{jk}|\frac{\nabla\partial_{1}}{\Delta}(|u|^{2})|^{2}-2\delta_{1k}|u|^{2}\frac{\partial_{j}\partial_{k}}{\Delta}(|u|^{2})+2\frac{\partial_{k}\partial_{1}}{\Delta}(|u|^{2})\frac{\partial_{j}\partial_{1}}{\Delta}(|u|^{2}),
\end{equation}
and it is not clear to us that
\begin{equation}
\int_{\mathbb{S}^{1}}\int_{\{x_{\omega}=y_{\omega}\}} |u(t,y)|^{2}T_{jk}(t,x)\omega_{j}\omega_{k}d\mathcal{H}^{3}(x,y) \geq 0,
\end{equation}
which would then allow us to discard this term. Since the strategy of the proof of frequency-localized interaction Morawetz estimates is to mimic the proof of the a priori estimate using the equation satisfied by the frequency-truncated solution $w$, then handle the error terms separately (in the mass-critical case with the long-time Strichartz estimate), it is unclear to us how to implement this strategy in the eeDS case.

\begin{remark}
If our defocusing eeDS equation were instead of the form
\begin{equation}
(i\partial_{t}+\Delta) u = |u|^{2}u-\lambda\frac{\partial_{1}^{2}}{\Delta}(|u|^{2})u,
\end{equation}
then by using weighted singular integral estimates, one can obtain an interaction Morawetz estimate, provided that $\lambda>0$ is sufficiently small (\cite{Tzirakis2017}). However, we do not wish to impose any such restriction, and moreover, this argument would not work in the focusing case, since the sign in front of $|u|^{2}u$ is negative.
\end{remark}

Instead, our approach in both the defocusing and focusing cases is inspired by the work \cite{Dodson2015} on the focusing mass-critical NLS. We do not ask for an a priori interaction Morawetz estimate for solutions to \eqref{eq:DS}, but instead more modestly ask for an estimate satisfied by admissible blowup solutions. In caricature, we construct a time-dependent potential adapted to the spatial and frequency localization properties of admissible solutions which essentially localizes the solution in physical space to a ball of centered at $x(t)$ of radius $\sim \frac{R^{2}}{N(t)}$, where $R\gg 1$. Using this potential, we construct a Morawetz action $M_{R}(t)$. We then proceed via integration by parts arguments exploiting the local conservation laws of equation \eqref{eq:DS}. We can overcome the obstruction to an a priori interaction Morawetz estimate by showing that up to a negligible error, the desired nonnegativity condition is satisfied, where in the focusing case, we have to use the additional condition that $M(u)<M(Q)$ together with theorem \ref{thm:GN}.

\subsection{Organization of the paper} %Outline of the paper
Having outlined the proof of our main result theorem, \ref{thm:main}, we now comment on the organization of the paper.

In section \ref{sec:prelim}, we introduce the basic notation used in this work and record some preliminary facts from Harmonic Analysis in addition to the classical linear and bilinear Strichartz estimates for the Schr\"{o}dinger equation. We also introduce the definitions of the $U_{\Delta}^{p}$ and $V_{\Delta}^{p}$ adapted function spaces and record some of their basic properties which will be used extensively in the sequel. Since most of the results stated in section \ref{sec:prelim} are by now standard in the literature, proofs are generally omitted.

In section \ref{sec:MMB}, we complete the concentration compactness step of the proof. The main results are theorem \ref{thm:AP_reduc} and corollary \ref{cor:adm_sol}, which give the existence of admissible blowup solutions to equation \eqref{eq:DS}. Readers familiar with the literature on minimal counterexamples for the critical NLS may wish to focus on the proof of lemma \ref{lem:as_sol}, as there we encounter new difficulties in the eeDS setting.

In section \ref{sec:LTSE_norm}, we introduce the norms from \cite{Dodson2016} which we use to formulate our long-time Strichartz estimate for equation \eqref{eq:DS}. We also prove some basic embeddings into standard Strichartz admissible spacetime Lebesgue spaces satisfied by these norms. As section \ref{sec:LTSE_norm} is quite technical and not very well motivated on its own, the reader may wish to skip over it on first reading and instead consult it as needed during the course of the proof of theorem \ref{thm:LTSE} in section \ref{sec:LTSE}.

In section \ref{sec:LTSE}, we state and prove our long-time Strichartz estimate for admissible blowup solutions to \ref{eq:DS}, which is theorem \ref{thm:LTSE}. We begin the section with an extended overview outlining the main steps of the proof of theorem \ref{thm:LTSE}, and we have organized the remainder of section \ref{sec:LTSE} into subsections corresponding to each step of the proof. We state in section \ref{sec:LTSE} the three improved bilinear Strichartz estimates (propositions \ref{prop:ibs_1}, \ref{prop:ibs_2}, and \ref{prop:ibs_3}), which are used to close the proof, but their proofs are postponed until the succeeding section. 

In section \ref{sec:BSE}, we give the proofs of the three improved bilinear Strichartz estimates (propositions \ref{prop:ibs_1}, \ref{prop:ibs_2}, and \ref{prop:ibs_3}) stated and used in section \ref{sec:LTSE}. Subsection \ref{ssec:BSE_cl} reviews the tensorial formulation of the local mass and momentum conservation laws for \eqref{eq:DS} and their analogues for the frequency-localized solution $w=P_{\leq j}u$. Subsections \ref{ssec:BSE_1} - \ref{ssec:BSE_3} contain the proofs of the respective bilinear Strichartz estimates.

In section \ref{sec:RFC}, we complete the first part of the rigidity step, which is to preclude the rapid frequency cascade scenario. The main result is theorem \ref{thm:no_rfc}.

In section \ref{sec:QS}, we complete the second part of the rigidity step, which is to preclude the quasi-soliton scenario. The main result is theorem \ref{thm:no_qs}. Subsection \ref{ssec:QS_pre} contains some preliminary lemmas which we frequently use to prove lower bounds for the main terms and upper bounds for the errors in the ensuing subsections. In subsection \ref{ssec:QS_con}, we construct our frequency-localized interaction Morawetz type functional, which we use in both the defocusing and focusing cases. In subsection \ref{ssec:QS_fl_err}, we prove the estimate for the error terms arising from truncating the solution to low frequencies. In subsection \ref{ssec:QS_dfoc}, we perform the lower and upper bound analysis in the defocusing case and balance all parameters floating around to obtain a contradiction. Finally, in subsection \ref{ssec:QS_foc}, we prove the lower and upper bound analysis in the focusing case and again balance all the parameters to obtain a contradiction, completing the proof of \ref{thm:no_qs}. Since we have then exhausted the rapid frequency cascade/quasi-soliton dichotomy, the proof of theorem \ref{thm:main} is complete.

\subsection{Acknowledgments}
The author would like to thank his advisor, Nata\v{s}a Pavlovi\'c, for a number of helpful conversations during the course of this project and during the preparation of the manuscript, in addition to her constant encouragement. The author thanks Andrea Nahmod for sharing with him helpful references during the course of the project. The author also thanks Nikos Tzirakis for sharing with him unpublished work on the use of weighted singular integral inequalities to obtain interaction Morawetz estimates at the French-American Conference on Nonlinear Dispersive PDE hosted by CIRM in June 2017. The author gratefully acknowledges financial support from The University of Texas at Austin through a Provost Graduate Excellence Fellowship.

\section{Preliminaries}\label{sec:prelim}
\subsection{Notation}\label{ssec:prelim_not}
We use the Einstein summation convention where a repeated index denotes summation over that index (e.g. $a_{k}b_{k}$.)

We denote the open ball centered at $x\in\R^{2}$ of radius $r>0$ by $B(x,r)$. When the radius is understood to be a dyadic integer, $r=2^{i}$ for some $i\in\mathbb{Z}$, we will just write $B(x,i)$ and refer to the ball as dyadic. Similarly, we denote the annulus centered at $x$ of outer radius $R$ and inner radius $r$ by $A(x,r,R)$. When the $R=2^{i}$ and $r=2^{j}$ are understood to be dyadic integers, we just write $A(x,i,j)$ and refer to the annulus as dyadic.

We use the vector notation $\ul{x}_{n}\coloneqq (x_{1},\ldots,x_{n})\in (\R^{d})^{n}\cong \R^{nd}$. We sometimes will denote the Lebesgue measure on $(\R^{d})^{n}$ by $d\ul{x}_{n}$. Similarly, we use the notation $D \coloneqq -i\nabla$, so that $\ul{D}_{n}\coloneqq (D_{1},\ldots,D_{n})$ is an $n$-tuple of $\R^{d}$-valued Fourier multipliers.

For $k\in\mathbb{Z}$, we denote the collection of dyadic cubes of side length $2^{k}$ which tile $\R^{2}$ by $\{Q_{a}^{k}\}_{a\in\mathbb{Z}^{2}}$, where $Q_{a}^{k} \coloneqq 2^{k}a+[-2^{k-1},2^{k-1}]^{2}$.

We use the notation $X\lesssim Y$ and $Y\gtrsim X$ when there exists some constant $C>0$ such that $X\leq CY$. When $X\lesssim Y$ and $Y\lesssim X$, we write $X\sim Y$. To denote the dependence of the implicit constant on some parameter $p$, we use the subscript notation $X\lesssim_{p}Y$, $X \gtrsim_{p} Y$, and $X \sim_{p} Y$.

We use the Japanese bracket notation $\langle{x}\rangle := (1+|x|^{2})^{1/2}$.

Given a measurable space function $f:\R^{2}\rightarrow\mathbb{C}$ or spacetime function $u:\R^{2}\rightarrow\mathbb{C}$ and a vector $y\in\R^{2}$, we use the notation $\tau_{y}f$ or $\tau_{y}u$ to denote the spatial translates $(\tau_{y}f)(x)\coloneqq f(x-y)$ or $(\tau_{y}u)(t,x) \coloneqq u(t,x-y)$, respectively.

We define the mixed norm Lebesgue space $L_{t}^{p}L_{x}^{q}(I\times\R^{d})$ to be the Banach space equipped with the norm
\begin{equation}
\|f\|_{L_{t}^{p}L_{x}^{q}(I\times\R^{d})} \coloneqq \left(\int_{I}\left(\int_{\R^{d}} |f(t,x)|^{q}dx\right)^{p/q}\right)^{1/q},
\end{equation}
with the usual modifications if $p=\infty$ or $q=\infty$. When $I=\R$, we will usually just write $L_{t}^{p}L_{x}^{q}$ instead of $L_{t}^{p}L_{x}^{q}(\R\times\R^{d})$. Similarly, we will just write $L^{q}$ in lieu of $L^{q}(\R^{d})$.

We denote the $n$-dimensional Hausdorff measure on $\R^{d}$ by $d\mathcal{H}^{n}$.

We denote the Schwartz space on $\R^{d}$ by $\mathcal{S}(\R^{d})$. We denote the space of tempered distributions (i.e. the dual of the Schwartz space) by $\mathcal{S}'(\R^{d})$.

We denote the Fourier transform of a tempered distribution $f$ by $\hat{f}$ or $\mathcal{F}(f)$. For $L^{1}$ functions, we define the Fourier transform via the convention
\begin{equation}
\hat{f}(\xi) \coloneqq \int_{\R^{d}}f(x)e^{-ix\cdot\xi}dx, \qquad \forall\xi\in\R^{d}
\end{equation}
and we define the inverse Fourier transform $\mathcal{F}^{-1}$ by
\begin{equation}
f^{\vee}(x) \coloneqq \frac{1}{(2\pi)^{d}}\int_{\R^{d}}\hat{f}(\xi)e^{i\xi\cdot x}d\xi, \qquad \forall x\in\R^{d}.
\end{equation}
Given a measurable symbol $m:(\R^{d})^{n} \rightarrow\mathbb{C}$, we define the $n$-linear Fourier multiplier $m(\uD_{n})$ by
\begin{equation}
m(\ul{D}_{n})(f_{1},\ldots,f_{n})(x) \coloneqq \frac{1}{(2\pi)^{nd}}\int_{\R^{nd}} m(\ul{\xi}_{n})\hat{f}_{1}(\xi_{1})\cdots \hat{f}_{n}(\xi_{n}) e^{i\sum_{j=1}^{n}\xi_{j}\cdot x}d\ul{\xi}_{n}, \qquad \forall x\in\R^{d}.
\end{equation}

Given two operators $T$ and $S$, we denote their commutator by $\comm{T}{S} \coloneqq TS-ST$. If $S$ is a pointwise multiplication operator (i.e. $(Sf)(x) \coloneqq a(x)f(x)$ for some measurable function $a$), then we write the commutator of $T$ and $S$ as $\comm{T}{a}$; analogously, if $T$ is a pointwise multiplication operator.

\subsection{Basic Harmonic Analysis}\label{ssec:prelim_ha}
In this section, we present some classical results from Harmonic Analysis which will be used throughout the course of the paper. The reader may find proofs of these results in any standard reference on the subject, such as \cite{grafakos2014c}, \cite{grafakos2014m}, \cite{Stein1993}, and \cite{Taylor2007}, in addition to the appendix of \cite{tao2006nonlinear}.

We first recall some facts about homogeneous tempered distributions.

Let $\Omega\in L^{1}(\mathbb{S}^{d-1})$ with integral zero. Define a tempered distribution $W_{\Omega}$ by
\begin{equation}
\ipp{W_{\Omega},f} \coloneqq \lim_{\varepsilon\rightarrow 0^{+}}\int_{|x|\geq\varepsilon}\frac{\Omega(x/|x|)}{|x|^{d}}f(x)dx, \qquad \forall f\in \mathcal{S}(\R^{d}).
\end{equation}
Observe that $W_{\Omega}\in\mathcal{S}'(\R^{d})$ is a homogeneous distribution of degree $-d$.

\begin{prop}\label{prop:dist_class}
Let $m\in C^{\infty}(\R^{d}\setminus\{0\})$ be a tempered distribution which is homogeneous of degree zero. Then there exists $b\in\mathbb{C}$ and $\Omega\in C^{\infty}(\mathbb{S}^{d-1})$ with integral zero, such that
\begin{equation}
m^{\vee}=b\delta_{0}+W_{\Omega}
\end{equation}
in the sense of tempered distributions.
\end{prop}

\begin{prop}[Hardy's inequality]\label{prop:H}
If $0\leq s<\frac{d}{2}$, then
\begin{equation}
\| |x|^{-s}f\|_{L^{2}(\R^{d})} \lesssim_{s,d} \|f\|_{\dot{H}^{s}(\R^{d})}.
\end{equation}
\end{prop}

\begin{prop}[Hardy-Littlewood-Sobolev lemma]\label{prop:HLS}
Suppose $1<p<\infty$, $1<q<\infty$, $0<r< d$, and $\frac{r}{d}=1-(\frac{1}{p}-\frac{1}{q})$. Then
\begin{equation}
\|f\ast |\cdot|^{-r}\|_{L^{q}(\R^{d})} \lesssim_{p,q,d}\|f\|_{L^{p}(\R^{d})}.
\end{equation}
\end{prop}

We next recall some facts from the Calder\'{o}n-Zygmund theory for singular integral operators. Since our equation \eqref{eq:DS} is translation-invariant, we limit our attention to the translation-invariant case of the theory (i.e. Calder\'{o}n-Zygmund operators of convolution type).
\begin{mydef}[Calder\'{o}n-Zygmund kernels] %Calderon-Zygmund operators
Let $K:\R^{d}\setminus\{0\}\rightarrow\mathbb{C}$ be a measurable function which satisfies
\begin{enumerate}[(a)]
	\item %Size condition
	the \emph{size condition}
		\begin{equation}\label{eq:si_con}
			|K(x)|\leq \frac{A_{1}}{|x|^{d}},\qquad \forall x\in\R^{d}\setminus\{0\}
		\end{equation}
		and
	\item %Smoothness condition
	the \emph{smoothness condition}
		\begin{equation}\label{eq:sm_con}
			|\nabla K(x)|\leq \frac{A_{2}}{|x|^{d+1}}, \qquad \forall x\in\R^{d}\setminus\{0\}.
		\end{equation}
\end{enumerate}
We say that $K$ is a \emph{Calder\'{o}n-Zygmund kernel (CZK)}.
\end{mydef}

Let $K$ be a CZK, and suppose that there exists a tempered distribution $W$ associated to $K$ by
\begin{equation}\label{eq:CZK_con}
\ipp{W,f} = \lim_{\delta_{j}\rightarrow 0^{+}}\int_{|x|\geq \delta_{j}} K(x)f(x)dx, \qquad \forall f\in\mathcal{S}(\R^{d}),
\end{equation}
for some sequence $\delta_{j}\rightarrow 0^{+}$ as $j\rightarrow\infty$.
\begin{remark}
Note that if the kernel $K$ satisfies $\int_{R_{1}<|x|<R_{2}}K(x)dx=0$ for any $0<R_{1}<R_{2}<\infty$, then the choice of sequence $\delta_{j}$ is immaterial.
\end{remark}

\begin{prop}[Calder\'{o}n-Zygmund theorem]\label{prop:CZ}
Let $K$ be a Calder\'{o}n-Zygmund kernel with constants $A_{1},A_{2}$ in \eqref{eq:si_con} and \eqref{eq:sm_con}, respectively. Let $W\in\mathcal{S}'(\R^{d})$ be a distribution associated to $K$ in the form of \eqref{eq:CZK_con}. If the operator $T$ defined by convolution with $W$ has a bounded extension on $L^{r}(\R^{d})$ with norm $B$ for some $1<r\leq\infty$, then $T$ has an extension mapping $L^{1}(\R^{2})$ to $L^{1,\infty}(\R^{2})$ with norm
\begin{equation}
	\|T\|_{L^{1}\rightarrow L^{1,\infty}}\lesssim_{d} A_{2}+B,
\end{equation}
and $T$ extends to a bounded operator on $L^{p}(\R^{d})$ for $1<p<\infty$ with norm
\begin{equation}
	\|T\|_{L^{p}\rightarrow L^{p}}\lesssim_{p,d} A_{2}+B.
\end{equation}
We call $T$ a \emph{Calder\'{o}n-Zygmund operator (CZO)}.
\end{prop}

For our purposes, an important example of CZOs are the Riesz transforms. For $j=1,\ldots,d$, we define tempered distributions $W_{j}$ on $\R^{d}$ by
\begin{equation}
\langle{W_{j},f}\rangle \coloneqq \dfrac{\Gamma\left(\frac{d+1}{2}\right)}{\pi^{\frac{d+1}{2}}}\lim_{\varepsilon\rightarrow 0^{+}}\int_{|y|\geq\varepsilon}\frac{y_{j}}{|y|^{d+1}}f(y)dy, \qquad \forall f\in\mathcal{S}(\R^{d}).
\end{equation}
We define the $j^{th}$ Riesz transform $R_{j}$ by convolution with the distribution $W_{j}$, and we denote the $d$-dimensional vector of Riesz transforms by $\vec{R}$. The operator $R_{j}$ is a Fourier multiplier with symbol $-i\frac{\xi_{j}}{|\xi|}$, from which we see that $\E=-R_{1}^{2}$. By Plancherel's theorem, $R_{j}$ is bounded on $L^{2}$, and so by the Calder\'{o}n-Zygmund theorem, $R_{j}$ is indeed a CZO. In particular, the operator $\E$ is bounded on $L^{p}(\R^{2})$ for every $1<p<\infty$.

We also need some results on the $L^{p}$ boundedness of (multilinear) Fourier multipliers.
\begin{prop}[H\"{o}rmander-Mikhlin multiplier theorem]\label{prop:HM}
Let $m:\R^{d}\setminus \{0\}\rightarrow\C$ belong to $C^{\infty}(\R^{d}\setminus\{0\})$ and satisfy the derivative estimates
\begin{equation}
|(\partial^{\alpha}m)(\xi)| \lesssim_{|\alpha|} |\xi|^{-|\alpha|}, \qquad \forall \xi\in\R^{d}\setminus\{0\}, \enspace \forall \alpha \in\N_{0}^{d}.
\end{equation}
Then for any exponent $1<p<\infty$, the operator $m(D):\mathcal{S}(\R^{d})\rightarrow \mathcal{S}'(\R^{d})$ extends to a bounded operator $m(D): L^{p}(\R^{d})\rightarrow L^{p}(\R^{d})$. 
\end{prop}

\begin{prop}[Coifman-Meyer multiplier theorem, \cite{coifman1978ast}, \cite{coifman1978commutateurs}, \cite{kenig1999multilinear}, \cite{Grafakos2002}]\label{prop:CM}
Let $N\in\N$, and let $m:\R^{Nd}\setminus\{0\} \rightarrow\C$ belong to $C^{\infty}(\R^{Nd}\setminus\{0\})$ and satisfy the derivative estimates
\begin{equation}
|\partial_{\xi_{1}}^{\alpha_{1}}\cdots\partial_{\xi_{N}}^{\alpha_{N}}m(\ul{\xi}_{N})| \lesssim_{\alpha_{1},\ldots,\alpha_{N}} \paren*{|\xi_{1}|+\cdots+|\xi_{N}|}^{-(|\alpha_{1}|+\cdots+|\alpha_{N}|)}, \qquad \forall \ul{\xi}_{N}\in\R^{Nd}, \enspace \forall \ul{\alpha}_{N} \in \N_{0}^{Nd}.
\end{equation}
Then for any exponents $1<p_{1},\ldots,p_{N}\leq\infty$ and $\frac{1}{N}<p<\infty$ satisfying $\frac{1}{p_{1}}+\cdots+\frac{1}{p_{N}}=\frac{1}{p}$, the operator $m(\ul{D}_{N}):\mathcal{S}(\R^{d})^{N}\rightarrow \mathcal{S}'(\R^{d})$ extends to a bounded operator $L^{p_{1}}(\R^{d})\times\cdots\times L^{p_{N}}(\R^{d})\rightarrow L^{p}(\R^{d})$.
\end{prop}

We next recall some basic facts from the Littlewood-Paley theory.
\begin{mydef}[Littlewood-Paley decomposition]
Let $\phi\in C_{c}^{\infty}(\R^{d})$ be a radial, nonincreasing function, such that $0\leq \phi\leq 1$ and
\begin{equation}
\phi(x)=
\begin{cases}
1, & {|x|\leq 1}\\
0, & {|x|>2}
\end{cases}.
\end{equation}
Define the dyadic partitions of unity
\begin{align}
1&=\phi(x)+\sum_{j=1}^{\infty}[\phi(2^{-j}x)-\phi(2^{-j+1}x)] \eqqcolon \psi_{0}(x)+\sum_{j=1}^{\infty}\psi_{j}(x), \qquad \forall x\in\R^{d}\\
1&=\sum_{j\in\mathbb{Z}}[\phi(2^{-j}x)-\phi(2^{-j+1}x)] \eqqcolon \sum_{j\in\mathbb{Z}}\dot{\psi}_{j}(x), \qquad \forall x\in\R^{d}\setminus\{0\}.
\end{align}
For any nonnegative integer $j$, define the inhomogeneous Littlewood-Paley projector $P_{j}$ by
\begin{equation}
P_{j}f \coloneqq \psi_{j}(D)f = \int_{\R^{d}}K_{j}(x-y)f(y)dy,
\end{equation}
and for any integer $j$, define the homogeneous Littlewood-Paley projector $\dot{P}_{j}$ by
\begin{equation}
\dot{P}_{j}f \coloneqq	\psi_{j}(D)f = \int_{\R^{d}}\dot{K}_{j}(x-y)f(y)dy.
\end{equation}
Observe that $K_{j},\dot{K}_{j'}\in\mathcal{S}(\R^{d})$ and for any integers $j\in\N_{0}$ and $j'\in \Z$,
\begin{equation}
|K_{j}(x)| \lesssim_{N,d} 2^{dj}\jp{2^{j}x}^{-N}, \enspace |\dot{K}_{j'}(x)| \lesssim_{N,d} 2^{dj'}\jp{2^{j'}x}^{-N}, \qquad \forall x\in\R^{d}, \enspace N\in\N_{0}.
\end{equation}
In particular, $K_{j},\dot{K}_{j'}\in L^{1}(\R^{d})$ with $\sup_{j,j'}\{\|K_{j}\|_{L^{1}(\R^{d})}+\|\dot{K}_{j'}\|_{L^{1}(\R^{d})}\} \lesssim 1$. For an integer $j<0$, we define $P_{j}$ to be the zero operator. Lastly, we define the operators
\begin{equation}
P_{j_{1}\leq\cdot\leq j_{2}} \coloneqq \sum_{j_{1}\leq j\leq j_{2}}P_{j}, \qquad \dot{P}_{j_{1}\leq\cdot\leq j_{2}} \coloneqq \sum_{j_{1}\leq j\leq j_{2}}\dot{P}_{j}.
\end{equation}
For a real number $N>0$ or a set $E\subset\R^{d}$, we define the frequency localization operators
\begin{equation}
P_{\leq N} \coloneqq \phi(N^{-1}D), \quad P_{>N} \coloneqq 1-\phi(N^{-1}D), \quad P_{N} \coloneqq \phi(N^{-1}D)-\phi(2N^{-1}D), \quad P_{E} \coloneqq \mathrm{1}_{E}(D).
\end{equation}
\end{mydef}

\begin{prop}[Bernstein's lemma]
For $s\geq 0$ and $1\leq p\leq q\leq\infty$, we have the inequalities
\begin{align}
\|P_{>N}f\|_{L^{p}(\R^{d})} &\lesssim_{p,s,d} N^{-s}\||\nabla|^{s}P_{> N}f\|_{L^{p}(\R^{d})} \\
\|P_{\leq N}|\nabla|^{s}f\|_{L^{p}(\R^{d})} &\lesssim_{p,s,d} N^{s}\|P_{\leq N}f\|_{L^{p}(\R^{d})}\\
\|P_{N}|\nabla|^{\pm s}f\|_{L^{p}(\R^{d})} &\lesssim_{p,s,d} N^{\pm s}\|P_{N}f\|_{L^{p}(\R^{d})} \\
\|P_{\leq N}f\|_{L^{q}(\R^{d})} &\lesssim_{p,q,d} N^{d(\frac{1}{p}-\frac{1}{q})} \|P_{\leq N}f\|_{L^{p}(\R^{d})} \\
\|P_{N}f\|_{L^{q}(\R^{d})} &\lesssim_{p,q,d} N^{d(\frac{1}{p}-\frac{1}{q})} \|P_{N}f\|_{L^{p}(\R^{d})}.
\end{align}
\end{prop}

\begin{prop}[Littlewood-Paley theorem]\label{prop:LP}
For $1<p<\infty$ and $f\in L^{p}(\R^{d})$,
\begin{equation}
\|f\|_{L^{p}(\R^{d})}\sim_{p,d} \|\left(\sum_{j=0}^{\infty}|P_{j}f|^{2}\right)^{1/2}\|_{L^{p}(\R^{d})} \sim_{p,d} \|\left(\sum_{j\in\mathbb{Z}}|\dot{P}_{j}f|^{2}\right)^{1/2}\|_{L^{p}(\R^{d})}.
\end{equation}
\end{prop}

Since we will need to consider solutions to the eeDS equation which are localized in frequency around some time-dependent center $\xi(t)\in\R^{2}$, we generalize the Littlewood-Paley projectors $P_{j},\dot{P}_{j}$ by defining
\begin{equation}
P_{\xi_{0},j}u \coloneqq e^{ix\cdot\xi_{0}}P_{j}(e^{-i(\cdot)\cdot\xi_{0}}u),\quad \dot{P}_{\xi_{0},j}u \coloneqq e^{ix\cdot\xi_{0}}\dot{P}_{j}(e^{-i(\cdot)\cdot\xi_{0}}u), \qquad \xi_{0}\in\R^{d}.
\end{equation}
Observe that
\begin{equation}
(P_{\xi_{0},j}f)(x) =\int_{\R^{2}}K_{j}(x-y)e^{i(x-y)\cdot\xi_{0}}f(y)dy, \qquad \forall x\in \R^{d}
\end{equation}
and $K_{j}(\cdot)e^{i(\cdot)\cdot\xi_{0}}\in L^{1}(\R^{2})$--analogously for $\dot{P}_{\xi_{0},j}$. It is a straightforward exercise to check that Bernstein's lemma and the Littlewood-Paley theorem also hold (with the same implicit constant) when $P_{j}$ and $\dot{P}_{j}$ are replaced by $P_{\xi_{0},j}$ and $\dot{P}_{\xi_{0},j}$, respectively.

We also need a Littlewood-Paley theory adapted to dyadic cube frequency localization. We specialize to the 2D case. We first construct a suitable $C^{\infty}$ partition of unity adapted to dyadic cubes. Let $\chi$ be a $C^{\infty}$ function satisfying $0\leq \chi\leq 1$ and
\begin{equation}
\chi(\xi) =
\begin{cases}
1, & {\forall \xi\in[-1,1]^{2}}\\ 
0, & {\forall \xi\notin (-2,2)^{2}}
\end{cases}.
\end{equation}
For $a\in\mathbb{Z}^{2}$ and $k\in\mathbb{Z}$, define the localizing function
\begin{equation}
\chi_{a,k}(\xi) \coloneqq \dfrac{\chi(\frac{\xi-2^{k}a}{2^{k}})}{\sum_{b\in\mathbb{Z}^{2}} \chi(\frac{\xi-2^{k}b}{2^{k}})}, \qquad \forall \xi\in\R^{2}.
\end{equation}
It is evident that $\supp(\chi_{a,k})\subset Q_{a}^{k+2}$. Moreover, $\chi_{0,0}\in C_{c}^{\infty}(\R^{2})$, and from scaling and translation invariance, it follows that $\chi_{a,k}\in C_{c}^{\infty}(\R^{2})$ with derivative bounds
\begin{equation}
\sup_{a\in\Z^{2}}\|\nabla^{N}\chi_{a,k}\|_{\infty} \lesssim_{N} 2^{-Nk}, \qquad \forall N\in\mathbb{N}_{0}.
\end{equation}
Next, observe that by dilation and translation/modulation invariance
\begin{equation}
\chi_{a,k}^{\vee}(x) = 2^{2k}e^{i2^{k}a\cdot x}\chi_{0,0}^{\vee}(2^{k}x), \qquad \forall x\in\R^{2}.
\end{equation}
Since $\chi_{0,0} \in C_{c}^{\infty}(\R^{2})$, it follows that $\chi_{0,0}^{\vee}$ is a Schwartz function. Hence, for all $k\in\mathbb{Z}$ and $a\in\mathbb{Z}^{2}$,
\begin{equation}
|\chi_{a,k}^{\vee}(x)| \lesssim_{N} 2^{2k}\langle{2^{k}x}\rangle^{-N}, \qquad \forall x\in\R^{2}, \enspace \forall N\in\N_{0}.
\end{equation}
We use the notation $Q_{a}^{k}$ to denote the cube
\begin{equation}
Q_{a}^{k} \coloneqq 2^{k}a + [-2^{k-1},2^{k-1}]^{2},
\end{equation}
and we define the smooth frequency projector onto $Q_{a}^{k}$, denoted by $\mathcal{P}_{Q_{a}^{k}}$, by
\begin{equation}
\widehat{\mathcal{P}_{Q_{a}^{k}}f}(\xi)=\chi_{a,k}(\xi)\hat{f}(\xi), \qquad \forall \xi\in\R^{2}.
\end{equation}

In the proof of the long-time Strichartz estimate, we will use a sparsification trick for frequency decompositions in order to apply bilinear estimates. To state our lemma, we first need to introduce the notion of a \emph{sparse decomposition}.

\begin{mydef}\label{def:sparse} %Definition of sparase collection
For any $k\in\mathbb{Z}$, we say that a subcollection $\mathcal{C}$ of dyadic cubes $Q_{a}^{k}$ of side length $2^{k}$ is \emph{sparse} if any distinct cubes $Q_{a}^{k}, Q_{a'}^{k}\in\mathcal{C}$ satisfy the distance condition $\dist(Q_{a}^{k}, Q_{a'}^{k})\geq 2^{k+4}$.
\end{mydef}

The advantage of working with a sparse cube decomposition is that (up to a harmless spatial translation) we can shift the frequency projector between factors of a product of two functions.

\begin{lemma}(Shifting trick)\label{lem:sparse} %Shifting trick
Let $\mathcal{C}$ be a collection of dyadic cubes of side length $2^{k}$. Then there exists a constant $C=C(\mathcal{C})>0$, such that for all $k\in\mathbb{Z}$ and $v,w \in L_{t,x}^{4}(I\times\R^{2})$,
\begin{equation}
\paren*{\sum_{a\in\mathcal{C}} \|v(\P_{Q_{a}^{k}}w)\|_{L_{t,x}^{2}(I\times\R^{2})}^{2}}^{1/2} \leq C\sup_{y\in\R^{2}} \paren*{\sum_{a\in\Z^{2}} \|(\P_{Q_{a}^{k}}v)(\tau_{y}w)\|_{L_{t,x}^{2}(I\times\R^{2})}^{2}}^{1/2}.
\end{equation}
\end{lemma}
\begin{proof}
We first partition $\mathcal{C}$ into $O(1)$ sparse subcollections $\mathcal{C}_{1},\ldots,\mathcal{C}_{N}$, where $N$ is independent of $k\in\mathbb{Z}$. To do this, we argue as follows. Partition (up to measure zero overlap) $\R^{2}$ into dyadic cubes $\{Q_{a}^{k+10}\}_{a\in\mathbb{Z}}$ of side length $2^{k+10}$. Now each cube $Q_{a}^{k+10}$ contains exactly $2^{20}$ dyadic children $Q_{a'}^{k}$ of side length $2^{k}$. Moreover,
\begin{equation}
Q_{a'}^{k}\subset Q_{a}^{k+10} \Rightarrow 2^{k}a'= 2^{k+10}a+2^{k}(a'-2^{10}a), \qquad (a'-2^{10}a)\in \{0,\ldots,2^{10}-1\}^{2}.
\end{equation}
We now partition $\mathcal{C}$ into $2^{20}$ subgroups $\mathcal{C}_{i,j}$ for $(i,j)\in\{0,\ldots,2^{10}-1\}^{2}$ by the assignment
\begin{equation}
Q_{a'}^{k}\in \mathcal{C}_{i,j} \Leftrightarrow (a'-2^{10}a)=(i,j).
\end{equation}
We claim that each subcollection $\mathcal{C}_{i,j}$ is sparse. Indeed, if $Q_{a}^{k}, Q_{a'}^{k}\in\mathcal{C}_{i,j}$, where $a\neq a'$, then by construction there exist unique dyadic parents $Q_{b}^{k+10}\supset Q_{a}^{k}$ and $Q_{b'}^{k+10}\supset Q_{a'}^{k}$, where $b\neq b'$ and
\begin{equation}
(a-2^{10}b)=(a'-2^{10}b') \Leftrightarrow a-a'=2^{10}(b-b'),
\end{equation}
which implies that $|a-a'| \geq 2^{10}$.

Now without loss of generality, we may suppose that $\mathcal{C}$ is sparse. We now perform two further frequency decompositions into cubes $\{Q_{a'}^{k}\}_{a'\in\Z^{2}}$ and $\{Q_{a''}^{k}\}_{a''\in\Z^{2}}$ of side length $2^{k}$ as follows:
\begin{align}
\paren*{\sum_{a\in\mathcal{C}} \|v(\P_{Q_{a}^{k}}u)\|_{L_{t,x}(I\times\R^{2})}^{2}}^{1/2} &= \paren*{\sum_{a\in\mathcal{C}}\sum_{a''\in\Z^{2}}\sum_{a'\in\Z^{2}} \|\P_{Q_{a''}^{k}}[(\P_{Q_{a'}^{k}}v)(\P_{Q_{a}^{k}}u)]\|_{L_{t,x}^{2}(I\times\R^{2})}^{2}}^{1/2} \nonumber\\
&= \paren*{\sum_{a''\in\Z^{2}}\sum_{a'\in\Z^{2}}\sum_{a\in\mathcal{C}}\|\P_{Q_{a''}^{k}}[(\P_{Q_{a'}^{k}}v)(\P_{Q_{a}^{k}}u)]\|_{L_{t,x}^{2}(I\times\R^{2})}^{2}}^{1/2}.\label{eq:sparse_sum} 
\end{align}
We claim that for each pair $(a',a'')\in\Z^{2}\times\Z^{2}$ fixed, there is at most one $a=a(a',a'')\in\mathcal{C}$ in the summation such that
\begin{equation}
\P_{Q_{a''}^{k}}[(\P_{Q_{a'}^{k}}v)(\P_{Q_{a}^{k}}u)] \neq 0.
\end{equation}
Indeed, for any $\eta_{1},\eta_{2}\in Q_{a'}^{k}$ and $\xi_{1}\in Q_{a_{1}}^{k}$, $\xi_{2}\in Q_{a_{2}}^{k}$, where $a_{1}\neq a_{2}$, the reverse triangle inequality, together with the fact that $\mathcal{C}$ is sparse, implies that
\begin{equation}
|(\xi_{1}+\eta_{1})-(\xi_{2}+\eta_{2})| \geq |\xi_{1}-\xi_{2}| - |\eta_{1}-\eta_{2}| \geq 2^{k+4}-2^{k+\frac{1}{2}} > 2^{k+\frac{1}{2}} = \diam(Q_{a''}^{k}).
\end{equation}
Now by the Fubini-Tonelli theorem and Minkowski's inequality,
\begin{align}
\|\P_{Q_{a''}^{k}}[(\P_{Q_{a'}^{k}}v)(\P_{Q_{a}^{k}}u)]\|_{L_{t,x}^{2}(I\times\R^{2})}^{2} &= \int_{I}\int_{\R^{2}} \left|\int_{\R^{2}} K_{Q_{a''}^{k}}(z)(\P_{Q_{a'}^{k}}v)(t,x-z)(\P_{Q_{a}^{k}}u)(t,x-z)dz\right|^{2}dxdt \nonumber\\
&=\int_{I}\int_{\R^{2}}\left|\int_{\R^{2}} K_{Q_{a''}^{k}}(z)(\P_{Q_{a'}^{k}}v)(t,x-z)\paren*{\int_{\R^{2}}K_{Q_{a}^{k}}(y)u(t,x-y-z)dy}dz\right|^{2}dxdt \nonumber\\
&\leq\paren*{\int_{\R^{2}} |K_{Q_{a}^{k}}(y)| \paren*{\int_{I}\int_{\R^{2}} \left| \int_{\R^{2}} K_{Q_{a''}^{k}}(z) (\P_{Q_{a'}^{k}}v)(t,x-z)u(t,x-y-z)dz\right|^{2}dxdt}^{1/2}dy}^{2} \nonumber\\
&\lesssim \paren*{\int_{\R^{2}}2^{2k}\jp{2^{k}y}^{-10} \|\P_{Q_{a''}^{k}}[(\P_{Q_{a'}^{k}}v)(\tau_{y}u)]\|_{L_{t,x}^{2}(I\times\R^{2})} dy}^{2}.
\end{align}
Substituting this last estimate into \eqref{eq:sparse_sum}, we obtain that
\begin{align}
\eqref{eq:sparse_sum} &\lesssim \paren*{\sum_{a''\in\Z^{2}}\sum_{a'\in\Z^{2}} \paren*{\int_{\R^{2}} 2^{2k}\jp{2^{k}y}^{-10} \|\P_{Q_{a''}^{k}}[(\P_{Q_{a'}^{k}}v)(\tau_{y}u)]\|_{L_{t,x}^{2}(I\times\R^{2})}dy}^{2}}^{1/2} \nonumber\\
&\leq \int_{\R^{2}}2^{2k}\jp{2^{k}y}^{-10} \paren*{\sum_{a'\in\Z^{2}}\sum_{a''\in\Z^{2}} \|\P_{Q_{a''}^{k}}[(\P_{Q_{a'}^{k}}v)(\tau_{y}u)]\|_{L_{t,x}^{2}(I\times\R^{2})}^{2}}^{1/2}dy,
\end{align}
where we use the embedding $L_{y}^{1}\ell_{a',a''}^{2}\subset \ell_{a',a''}^{2}L_{y}^{1}$ to obtain the ultimate inequality. By Plancherel's theorem and almost orthogonality of the projectors $\P_{Q_{a}^{k}}$, we see that
\begin{align}
 \int_{\R^{2}}2^{2k}\jp{2^{k}y}^{-10} \paren*{\sum_{a'\in\Z^{2}}\sum_{a''\in\Z^{2}} \|\P_{Q_{a''}^{k}}[(\P_{Q_{a'}^{k}}v)(\tau_{y}u)]\|_{L_{t,x}^{2}(I\times\R^{2})}^{2}}^{1/2}dy &\lesssim \int_{\R^{2}}2^{2k}\jp{2^{k}y}^{-10} \paren*{\sum_{a'\in\Z^{2}} \|(\P_{Q_{a'}^{k}}v)(\tau_{y}u)\|_{L_{t,x}^{2}(I\times\R^{2})}^{2}}^{1/2}dy \nonumber\\
&\lesssim \sup_{y\in\R^{2}} \|(\P_{Q_{a}^{k}}v)(\tau_{y}u)\|_{\ell_{a}^{2}L_{t,x}^{2}(\mathbb{Z}^{2}\times I\times\R^{2})}.
\end{align}
\end{proof}

\subsection{Linear and bilinear estimates}\label{ssec:prelim_lbe}
In this subsection, we record some linear and bilinear estimates all of which are by now standard in the dispersive literature. We have restricted to the $d=2$ case, although we remark that the estimates below hold in arbitrary dimension \emph{mutatis mutandis}.

\begin{prop}[Strichartz estimates, \cite{strichartz1977}, \cite{Yajima1987}, \cite{Ginibre1992}, \cite{Keel1998}]\label{prop:ls}
A pair $(p,q)$ of real numbers is \emph{admissible} (for $d=2$) if $2<p\leq\infty$ and $\frac{2}{p}=2(\frac{1}{2}-\frac{1}{q})$. If $u: I\times\R^{2}\rightarrow\mathbb{C}$ is a solution to the Cauchy problem
\begin{equation}
\begin{cases}
(i\partial_{t}+\Delta)u=F(t), & (t,x)\in I\times\R^{2} \\
u(0)=u_{0}, & x\in \R^{2} 
\end{cases},
\end{equation}
then for all admissible pairs $(p,q)$ and $(\tilde{p},\tilde{q})$, we have that
\begin{equation}
	\|u\|_{L_{t}^{p}L_{x}^{q}(I\times\R^{2})} \lesssim_{p,q,\tilde{p},\tilde{q}} \|u_{0}\|_{L^{2}(\R^{2})}+\|F\|_{L_{t}^{\tilde{p}'}L_{x}^{\tilde{q}'}(I\times\R^{2})},
\end{equation}
where $\tilde{p}'$ and $\tilde{q}'$ denote the H\"{o}lder conjugates of $\tilde{p}$ and $\tilde{q}$, respectively.
\end{prop}

\begin{prop}[Bilinear Strichartz estimate, \cite{Bourgain1998}]\label{prop:bs_B} %Bourgain's bilinear Strichartz estimate
Let $u_{0},v_{0}\in L^{2}(\R^{2})$. If $u_{0}$ and $v_{0}$ have Fourier supports in the annuli $|\xi|\sim N$ and $|\xi|\sim M$, where $M\ll N$, respectively, then for any $p\in [2,\infty]$,
\begin{equation}
\|(e^{it\Delta}u_{0})(e^{it\Delta}v_{0})\|_{L_{t}^{p}L_{x}^{p'}(\R\times\R^{2})} \lesssim_{p} \left(\frac{M}{N}\right)^{1/p}\|u_{0}\|_{L^{2}(\R^{2})}\|v_{0}\|_{L^{2}(\R^{2})}
\end{equation}
\end{prop}

By Galilean invariance of the free Schr\"{o}dinger equation, we obtain the following cube-localized version of the preceding bilinear estimate. See also the proof of proposition \ref{prop:ibs_1} for a way of proving the corollary directly.
\begin{cor}\label{prop:cube_BS} %Cube-localized version of Bourgain's bilinear Strichartz estimate
Let $0<c_{1}<c_{2}$, and let $\delta \in (0,1)$. Let $\xi_{0}\in\R^{2}$, and let $Q_{\xi_{0}}$ be a cube centered at $\xi_{0}$ of side length $l(Q_{\xi_{0}})=M$. Let $u_{0},v_{0}\in L^{2}(\R^{2})$. 
\begin{enumerate}[(i)]
\item\label{item:cube_BS_1}
If $u_{0}$ has Fourier support in the annulus $A(0,c_{1}N,c_{2}N)$ and $v_{0}$ has Fourier support in the cube $Q_{\xi_{0}} \subset B(0,(1-\delta)c_{1}N)$, where $M\ll N$, then for any $p\in [2,\infty]$,
\begin{equation}
\|(e^{it\Delta}u_{0})(e^{it\Delta}v_{0})\|_{L_{t}^{p}L_{x}^{p'}(\R\times\R^{2})}\lesssim_{p,c_{1},c_{2},\delta} \left(\frac{M}{N}\right)^{1/p}\|u_{0}\|_{L^{2}(\R^{2})}\|v_{0}\|_{L^{2}(\R^{2})}
\end{equation}
\item\label{item:cube_BS_2}
If $u_{0}$ has Fourier support in the cube $Q_{\xi_{0}} \subset A(0,c_{1}N,c_{2}N)$ and $v_{0}$ has Fourier support in the ball $B(0,(1-\delta)c_{1}N)$, where $M\ll N$, then for any $p\in [2,\infty]$,
\begin{equation}
\|(e^{it\Delta}u_{0})(e^{it\Delta}v_{0})\|_{L_{t}^{p}L_{x}^{p'}(\R\times\R^{2})} \lesssim_{p,c_{1},c_{2},\delta} \paren*{\frac{M}{N}}^{1/p} \|u_{0}\|_{L^{2}(\R^{2})} \|v_{0}\|_{L^{2}(\R^{2})}.
\end{equation}
\end{enumerate}
\end{cor}
\begin{proof}
We first prove \ref{item:cube_BS_1}. By complex interpolation with the trivial $L_{t}^{\infty}L_{x}^{1}$ bilinear estimate, it suffices to prove the desired estimate for $p=2$. Observe that by translation invariance of the Lebesgue measure and Galilean invariance of the free Schr\"{o}dinger equation, we can write
\begin{align}
\|(e^{it\Delta}u_{0})(e^{it\Delta}v_{0})\|_{L_{t,x}^{2}(\R\times\R^{2})} &= \|\paren*{T_{g_{0,-\xi_{0},0,1}}(e^{it\Delta}u_{0})}\paren*{T_{g_{0,-\xi_{0},0,1}}(e^{it\Delta}v_{0})}\|_{L_{t,x}^{2}(\R\times\R^{2})} \nonumber\\
&= \|(e^{it\Delta}u_{0,\xi_{0}})(e^{it\Delta}v_{0,\xi_{0}})\|_{L_{t,x}^{2}(\R\times\R^{2})},
\end{align}
where $u_{0,\xi_{0}} \coloneqq g_{0,-\xi_{0},0,1}u_{0}$ and $v_{0,\xi_{0}} \coloneqq g_{0,-\xi_{0},0,1}v_{0}$. Observe that $v_{0,\xi_{0}}$ has Fourier support in the cube $Q_{0} \coloneqq [-\frac{M}{2},\frac{M}{2}]^{2}$. Now if $\xi\in \supp(\hat{u}_{0,\xi_{0}})$, then $\xi+\xi_{0}\in A(0,c_{1}N,c_{2}N)$. Since $Q_{\xi_{0}} \subset B(0,(1-\delta)c_{1}N)$, in particular, $|\xi_{0}| \leq (1-\delta)c_{1}N$. Hence by the (reverse) triangle inequality,
\begin{equation}
\delta c_{1}N = \paren*{c_{1}-(1-\delta)c_{1}}N \leq |\xi| \leq \paren*{c_{2}+(1-\delta)c_{1}}N,
\end{equation}
which implies that $u_{0,\xi_{0}}$ has Fourier support in the annulus $A(0,\delta c_{1}N, (c_{2}+(1-\delta)c_{1})N)$. Applying proposition \ref{prop:bs_B} with $u_{0,\xi_{0}}$ and $v_{0,\xi_{0}}$, we conclude that
\begin{equation}
\|(e^{it\Delta}u_{0,\xi_{0}})(e^{it\Delta}v_{0,\xi_{0}})\|_{L_{t,x}^{2}(\R\times\R^{2})} \lesssim_{\delta,c_{1},c_{2}} \paren*{\frac{M}{N}}^{1/2} \|u_{0,\xi_{0}}\|_{L^{2}(\R^{2})} \|v_{0,\xi_{0}}\|_{L^{2}(\R^{2})} = \paren*{\frac{M}{N}}^{1/2} \|u_{0}\|_{L^{2}(\R^{2})} \|v_{0}\|_{L^{2}(\R^{2})}.
\end{equation}

We now prove \ref{item:cube_BS_2}. As before, it suffices to prove the desired estimate for $p=2$. Let $u_{0,\xi_{0}}$ and $v_{0,\xi_{0}}$ be defined as before. This time, $u_{0,\xi_{0}}$ has Fourier support in the cube $Q_{0}$. Since $Q_{\xi_{0}}\subset A(0,c_{1}N,c_{2}N)$, in particular, $\xi_{0}\in A(0,c_{1}N,c_{2}N)$. If $\xi \in \supp(\hat{v}_{0,\xi_{0}})$, then $\xi+\xi_{0}\in B(0,(1-\delta)c_{1}N)$. Hence by the (reverse) triangle inequality,
\begin{equation}
\delta c_{1} N = \paren*{-(1-\delta)c_{1}+c_{1}}N \leq |\xi| \leq \paren*{(1-\delta)c_{1}+c_{2}}N,
\end{equation}
which implies that $v_{0,\xi_{0}}$ has Fourier support in the annulus $A(0,\delta c_{1}N, ((1-\delta)c_{1}+c_{2})N)$. Applying proposition \ref{prop:bs_B} with $u_{0,\xi_{0}}$ and $v_{0,\xi_{0}}$, we obtain the desired conclusion.
\end{proof}

While the bilinear estimate of proposition \ref{prop:bs_B} is invariant under the DS scaling, the following bilinear restriction estimate  for the paraboloid due to Tao is subcritical with respect to the scaling symmetry.
\begin{prop}[Bilinear restriction estimate, \cite{Tao2003}]\label{prop:br_T} %Tao's bilinear restriction estimate for paraboloids
Let $q>\frac{5}{2}$, and let $c>0$. Then there exists a constant $C(c,q)>0$ such that for all $u_{0},v_{0}\in L^{2}(\R^{2})$ satisfying
\begin{equation}
	\supp(\hat{u}_{0}),\supp(\hat{v}_{0})\subset B(0,N), \qquad \dist(\supp(\hat{u}_{0}),\supp(\hat{u}_{0}))\geq cN,
\end{equation}
we have that
\begin{equation}\label{eq:br_T}
	\|(e^{it\Delta}u_{0})(e^{it\Delta}v_{0})\|_{L_{t,x}^{q}(\R\times\R^{2})}\leq C(c,q)N^{2-\frac{4}{q}}\|u_{0}\|_{L^{2}(\R^{2})}\|v_{0}\|_{L^{2}(\R^{2})}.
\end{equation}
\end{prop}

By interpolating between the estimate \eqref{eq:br_T} with $q=\frac{5}{2}+\varepsilon$ and the trivial $L_{t}^{\infty}L_{x}^{1}$ bilinear estimate, for any $\varepsilon>0$, we obtain the bilinear estimate
\begin{equation}
\|(e^{it\Delta}u_{0})(e^{it\Delta}v_{0})\|_{L_{t}^{2}L_{x}^{\frac{3}{2}+}(\R\times\R^{2})}\lesssim N^{-\frac{1}{3}+}\|u_{0}\|_{L^{2}(\R^{2})}\|v_{0}\|_{L^{2}(\R^{2})}.
\end{equation}

\subsection{$U_{\Delta}^{p}$ and $V_{\Delta}^{p}$ spaces}\label{ssec:prelim_fs}
In this subsection, we introduce the definitions of the $U^{p}$ and $V^{p}$ spaces and their adapted counterparts $U_{\Delta}^{p}$ and $V_{\Delta}^{p}$ and record some basic properties of these spaces. The $U^{p}$ and $V^{p}$ spaces were first applied in the dispersive setting in \cite{Koch2004} and have since proved useful as an alternative to Bourgain's Fourier restriction $X^{s,b}$ spaces when studying critical Cauchy problems. For the omitted proofs of the results in this subsection and further properties of these spaces, we refer the reader to Section 2 of \cite{Hadac2009} (see also the erratum \cite{hadac2010erratum}) and Chapter 4 of \cite{koch2014dispersive}.

Let $\mathcal{Z}$ denote the set of finite partitions $-\infty<t_{1}<\cdots<t_{N}<\infty$. In this subsection, $(H,\ipp{,})$ denotes a generic Hilbert space, unless specified otherwise.

We say that a function $f$ is \emph{ruled} if at every point (including $\pm\infty$) left and right limits exist. The space of ruled functions equipped with the uniform norm is a Banach space, which we denote by $\mathcal{R}$. We let $\mathcal{R}_{rc}$ denote the closed subspace of $\mathcal{R}$ consisting of right-continuous functions with $\lim_{t\rightarrow -\infty} f(t)=0$.

A step function $f:\R\rightarrow H$ is one for which there exists a partition $\{t_{k}\}_{k=1}^{N}\in\mathcal{Z}$ such that $f$ is constant on the intervals $(-\infty,t_{1})$, $(t_{k},t_{k+1})$, and $(t_{N},\infty)$. We denote the set of step functions by $\mathscr{S}$. We denote the subset of right-continuous step functions by $\mathscr{S}_{rc}$.

\begin{mydef}[$U^{p}$ spaces] %Definition of $U_{\Delta}^{p}$ spaces
Let $1\leq p<\infty$. We say that $u:\R \rightarrow H$ is a $p$-atom if there exists a partition $\{t_{k}\}_{k=1}^{N}\in\mathcal{Z}$, such that
\begin{equation}
u = \sum_{k=1}^{N}1_{[t_{k},t_{k+1})}(t)u_{k}, \quad \sum_{k=1}^{N}\|u_{k}\|_{H}^{p}\leq 1,
\end{equation}
where $t_{N+1} \coloneqq \infty$. We define the $U^{p}$ norm by
\begin{equation}
\|u\|_{U^{p}(\R\times\R^{2})} \coloneqq \inf \brac*{\sum_{j}|c_{j}| : u=\sum_{j}c_{j}u_{j}, \quad \text{$u_{j}$ are $U^{p}$ atoms}}.
\end{equation}
Given an interval $[a,b]$ and a function $u:[a,b]\rightarrow H$, we define the norm $\|u\|_{U^{p}([a,b])} \coloneqq \|u1_{[a,b)}\|_{U^{p}(\R)}$. We define $U^{p}([a,b])$ to be the normed space of functions $u:[a,b]\rightarrow H$ for which this norm is finite.
\end{mydef}

\begin{mydef}[$V^{p}$ spaces]
Let $1\leq p<\infty$. We define $V^{p}(\R)$ to be the normed space of functions $v:\R\rightarrow H$ such that $\lim_{t\rightarrow \pm\infty}v(t)$ exists, $v(\infty)=0$, and
\begin{equation}
\|v\|_{V^{p}(\R)} \coloneqq \sup_{\{t_{k}\}_{k=1}^{N}\in\mathcal{Z}} \paren*{\sum_{k=1}^{N} \|v(t_{k+1})-v(t_{k})\|_{H}^{p}}^{1/p} <\infty,
\end{equation}
where $t_{N+1}\coloneqq \infty$. We define $V_{-}^{p}(\R)$ to be the closed subspace consisting of functions $v:\R\rightarrow H$ such that $v(-\infty)\coloneqq \lim_{t\rightarrow -\infty}v(t)=0$. Given an interval $[a,b]$ and a function $v:[a,b]\rightarrow H$, we define the norm $\|v\|_{V^{p}([a,b])} \coloneqq \|v1_{[a,b)}\|_{V^{p}(\R)}$. We define $V^{p}([a,b])$ to be the normed space of functions $v:[a,b]\rightarrow H$ for which this norm is finite. We analogously define $V_{-}^{p}([a,b])$.

We denote the subspaces of $V^{p}(\R)$ (resp. $V_{-}^{p}(\R)$) and $V^{p}([a,b])$ (resp. $V_{-}^{p}([a,b])$) consisting of right-continuous functions by $V_{rc}^{p}(\R)$ (resp. $V_{-,rc}^{p}(\R)$) and $V_{rc}^{p}([a,b])$ (resp. $V_{-,rc}^{p}([a,b])$), respectively.
\end{mydef}

When the domain of the $U^{p}$ or $V^{p}$ space under consideration is not important to the formulation of a result, we will generally omit it.

\begin{prop}[$V^{p}$ basic properties]
Let $1\leq p<q<\infty$. Then the following hold.
\begin{enumerate}[(i)]
\item
$\|\cdot\|_{V^{p}}$ defines a norm, and $V^{p}$ (resp. $V_{-}^{p}$) is a Banach subspace of $\mathcal{R}$. $V_{rc}^{p}$ (resp. $V_{-,rc}^{p}$) is a closed subspace of $V^{p}$.
\item
$V^{\infty}\coloneqq \mathcal{R}$ with norm $\|\cdot\|_{V^{\infty}} \coloneqq \|\cdot\|_{\sup}$.
\item
If $1\leq p\leq q\leq \infty$, then $V^{p}\subset V^{q}$ with $\|\cdot\|_{V^{q}} \leq \|\cdot\|_{V^{p}}$.
\end{enumerate}
\end{prop}

\begin{prop}[$U^{p}$ basic properties]
Let $1\leq p<q<\infty$. Then the following hold.
\begin{enumerate}[(i)]
\item
Functions in $U^{p}$ are ruled and right-continuous. If $u\in U^{p}$, then $\lim_{t\rightarrow -\infty}u(t)=0$ in $H$.
\item
$\|\cdot\|_{U^{p}}$ defines a norm, and $U^{p}$ equipped with this norm is a Banach space. Moreover, $\|\cdot\|_{\sup} \leq \|\cdot\|_{U^{p}}$.
\item
If $p<q$, then $U^{p}\subset U^{q}$ with $\|\cdot\|_{U^{q}} \leq \|\cdot\|_{U^{p}}$.
\item
If $1\leq p<\infty$, then $U^{p} \subset V_{-,rc}^{p}$ and $\|\cdot\|_{V^{p}} \leq 2^{1/p} \|\cdot\|_{U^{p}}$.
\end{enumerate}
\end{prop}

\begin{prop}[$V_{-,rc}^{p}\subset U^{q}$]
If $1\leq p<q<\infty$, then the embedding $V_{-,rc}^{p} \subset U^{q}$ is continuous.
\end{prop}

\begin{prop}\label{prop:Up_lp}
Let $a\leq b\leq c$, and let $1\leq p<\infty$. Then
\begin{equation}
\|u\|_{U^{p}([a,c])}^{p} \leq \|u\|_{U^{p}([a,b])}^{p} + \|u\|_{U^{p}([b,c])}^{p}
\end{equation}
\end{prop}
\begin{proof}
See Lemma B.5 in \cite{koch2016conserved}.
\end{proof}

\begin{prop}[Duality]
Let $1<p,q<\infty$ satisfy $\frac{1}{p}+\frac{1}{q}=1$. Then there is a unique continuous bilinear map $B: U^{q}\times V^{p}\rightarrow \R$ which satisfies (with $t_{0}=a$ and $u(t_{0})=0$)
\begin{equation}
B(u,v) = \sum_{i=1}^{N}\ipp{u(t_{i})-u(t_{i-1}),v(t_{i})}
\end{equation}
for $v\in V^{p}$ and $u\in\mathscr{S}_{rc}$ with associated partition $\{t_{k}\}_{k=1}^{N}\in\mathcal{Z}$. The map $B$ satisfies
\begin{equation}
|B(u,v)| \leq \|u\|_{U^{q}}\|v\|_{V^{p}}.
\end{equation}
The map
\begin{equation}
V^{p}\rightarrow (U^{q})^{*}, \qquad v\mapsto B(\cdot,v)
\end{equation}
is a surjective isometry if $1\leq q<\infty$.

Define the subspace $V_{c}^{q}(\R)$ by
\begin{equation}
V_{c}^{q}(\R) \coloneqq \{v\in V^{q}(\R)\cap C(\R) : \lim_{t\rightarrow\pm\infty} v(t)=0\}.
\end{equation}
Then the map
\begin{equation}
U^{p} \rightarrow (V_{c}^{q})^{*}, \qquad u\mapsto B(u,\ol{\cdot})
\end{equation}
is a surjective isometry.
\end{prop}

\begin{remark}
Test functions are weak-* dense in $V^{p}$, which is a useful fact for our applications of $U^{p}$-$V^{p}$ duality in section \ref{sec:LTSE}.
\end{remark}

\begin{mydef}[$DU^{p}$ space]
We define the normed space $DU^{p}(\R)$ (resp. $DU^{p}([a,b])$) to be the space of distributional derivatives of $U^{p}(\R)$ (resp. $U^{p}([a,b])$) functions equipped with the induced norm. 
\end{mydef}

We now specialize to the case where $H=L^{2}(\R^{2})$.
\begin{mydef}[$U_{\Delta}^{p}$, $V_{\Delta}^{p}$, and $DU_{\Delta}^{p}$ spaces]
For $1\leq p<\infty$, we define $U_{\Delta}^{p}(\R\times\R^{2}) \coloneqq e^{it\Delta}\paren*{U^{p}(\R; L^{2}(\R^{2}))}$ with norm
\begin{equation}
\|u\|_{U_{\Delta}^{p}(\R\times\R^{2})} \coloneqq \|e^{-it\Delta}u\|_{U^{p}(\R)}.
\end{equation}

We define $V_{\Delta}^{p}(\R\times\R^{2}) \coloneqq e^{it\Delta}\paren*{V^{p}(\R; L^{2}(\R^{2}))}$ with norm
\begin{equation}
\|v\|_{V_{\Delta}^{p}(\R\times\R^{2})} \coloneqq \|e^{-it\Delta}v\|_{V^{p}(\R)}.
\end{equation}
We define $V_{-,\Delta}^{p}(\R\times\R^{2})$ and $V_{-,rc,\Delta}^{p}(\R\times\R^{2})$ with respective norms in completely analogous fashion.

We define $DU_{\Delta}^{p}(\R\times\R^{2}) \coloneqq e^{it\Delta}\paren*{DU^{p}(\R; L^{2}(\R^{2}))}$ with norm
\begin{equation}
\|f\|_{DU_{\Delta}^{p}(\R\times\R^{2})} \coloneqq \|e^{-it\Delta}f\|_{DU^{p}(\R)}.
\end{equation}

We define the spaces $U_{\Delta}^{p}([a,b]\times\R^{2})$, $V_{\pm,\Delta}^{p}([a,b]\times\R^{2})$, $V_{\pm, rc,\Delta}^{p}([a,b]\times\R^{2})$, and $DU_{\Delta}^{p}([a,b]\times\R^{2})$ with respect norms in completely analogous fashion.
\end{mydef}
When the choice of time domain for our adapted function (i.e. $[a,b]$ or $\R$) is not important to the formulation of a result, we will often omit it or use the generic notation $I$ or $J$.

We now define function spaces built on the $U_{\Delta}^{p}$ and $V_{\Delta}^{p}$ spaces, which are adapted to frequency cube decompositions.
\begin{mydef}[$X_{\Delta}^{p,k}$ and $Y_{\Delta}^{p,k}$ spaces]
Let $1<p<\infty$ and $k\in\Z$, and let $J\in\{[a,b],\R\}$. We define $X_{\Delta}^{p,k}(J\times\R^{2})$ to be the normed space of functions such that for each $a\in\mathbb{Z}^{2}$, the function $P_{Q_{a}^{k}}u\in U_{\Delta}^{p}(J\times\R^{2})$, and
\begin{equation}
\|u\|_{X_{\Delta}^{p,k}(J\times\R^{2})}^{2} \coloneqq \sum_{a\in\mathbb{Z}^{2}} \|P_{Q_{a}^{k}}u\|_{U_{\Delta}^{p}(J\times\R^{2})}^{2} < \infty.
\end{equation}
Similarly, we define $Y_{\Delta}^{p,k}(J\times\R^{2})$ to be the normed space of functions such that for each $a\in\mathbb{Z}^{2}$, the function $P_{Q_{a}^{k}}v\in V_{\Delta}^{p}(J\times\R^{2})$, and
\begin{equation}
\|v\|_{Y_{\Delta}^{p,k}(J\times\R^{2})}^{2} \coloneqq \sum_{a\in\mathbb{Z}^{2}} \|P_{Q_{a}^{k}}v\|_{V_{\Delta}^{p}(J\times\R^{2})}^{2} <\infty.
\end{equation}
\end{mydef}

\begin{lemma} %Embedding into $\ell^{2}(U_{\Delta}^{2})$
Let $J\in\{[a,b],\R\}$. Let $\{Q\}$ be a collection of cubes in $\R^{2}$ such that $\|\sum_{Q}1_{Q}\|_{L^{\infty}}<\infty$ (i.e. boundedly overlapping). Then
\begin{equation}
\left(\sum_{Q}\|P_{Q}u\|_{U_{\Delta}^{2}(J\times\R^{2})}^{2}\right)^{1/2} \lesssim \|u\|_{U_{\Delta}^{2}(J\times\R^{2})},
\end{equation}
where the implicit constant only depends on $\|\sum_{Q}1_{Q}\|_{L^{\infty}}$.
\end{lemma}
\begin{proof}
Suppose $u=\sum_{j}c_{j}u_{j}$ is a $U_{\Delta}^{2}$ atomic decomposition of $u$, where
\begin{equation}
u_{j} = \sum_{k}1_{[t_{j,k},t_{j,k+1})} e^{it\Delta}u_{j,k}.
\end{equation}
For each cube $Q$, define the function
\begin{equation}
u_{j, Q} \coloneqq \sum_{k}1_{[t_{j,k},t_{j,k+1})}e^{it\Delta}\frac{P_{Q}u_{j,k}}{(\sum_{k}\|P_{Q}u_{j,k}\|_{L^{2}}^{2})^{1/2}},
\end{equation}
which is a $U_{\Delta}^{2}$ atom. Observe that 
\begin{equation}
u_{Q} \coloneqq P_{Q}u = \sum_{j}c_{j,Q}u_{j, Q}, \qquad c_{j, Q} \coloneqq c_{j}\paren*{\sum_{k}\|P_{Q}u_{j, k}\|_{L^{2}}^{2}}^{1/2}
\end{equation}
is an atomic decomposition for $u_{Q}$. Therefore by Minkowski's inequality,
\begin{align}
\left(\sum_{Q}\|u_{Q}\|_{U_{\Delta}^{2}(J\times\R^{2})}^{2}\right)^{1/2} \leq \left(\sum_{Q} \left(\sum_{j}|c_{j}|\left(\sum_{k}\|P_{Q}u_{j, k}\|_{L^{2}}^{2}\right)^{1/2}\right)^{2} \right)^{1/2}	&\leq \sum_{j}|c_{j}| \left(\sum_{Q}\sum_{k}\|P_{Q}u_{j,k}\|_{L^{2}}^{2}\right)^{1/2} \nonumber\\
&= \sum_{j}|c_{j}| \left(\sum_{k}\sum_{Q}\|P_{Q}u_{j, k}\|_{L^{2}}^{2}\right)^{1/2} \nonumber\\
&\lesssim \sum_{j}|c_{j}|\left(\sum_{k}\|u_{j, k}\|_{L^{2}}^{2}\right)^{1/2}\nonumber\\
&= \sum_{j}|c_{j}|,
\end{align}
where we use Plancherel's theorem together with the bounded overlap of the cubes $Q$ to obtain the penultimate inequality. Taking the infimmum of the RHS of the ultimate equality over all $U_{\Delta}^{2}$ atomic decompositions of $u$ completes the proof of the lemma.
\end{proof}

\begin{prop}[$X_{\Delta}^{2,k}$ and $Y_{\Delta}^{2,k}$ embeddings]
Uniformly in $k\in\mathbb{Z}$, we have the continuous embeddings
\begin{equation}
U_{\Delta}^{2} \subset X_{\Delta}^{2,k} \subset Y_{\Delta}^{2,k} \subset V_{\Delta}^{2}.
\end{equation}
\end{prop}
\begin{proof}
It suffices to prove the embedding $U_{\Delta}^{2}\subset X_{\Delta}^{2,k}$, as the remaining embeddings follow from $U_{\Delta}^{2}\subset V_{\Delta}^{2}$ and duality. But this embedding is an immediate consequence of the preceding lemma.
\end{proof}

\begin{prop}\label{prop:V_dual}
Let $1<p<\infty$. Let $J\subset \R$ be a subinterval, and let $F\in L_{t}^{1}L_{x}^{2}(J\times\R^{2})$. Then
\begin{equation}
\left\|\int_{0}^{t}e^{i(t-\tau)\Delta}F(\tau)d\tau\right\|_{U_{\Delta}^{p}(J\times\R^{2})}\lesssim \sup_{\|v\|_{V_{-,rc,\Delta}^{p'}(J\times\R^{2})=1}} \left|\int_{J}\langle{F(t),v(t)}\rangle dt\right|.
\end{equation}
\end{prop}
\begin{proof}
The hypothesis that $F\in L_{t}^{1}L_{x}^{2}(J\times\R^{2})$ implies that $\int_{0}^{t}e^{i(t-\tau)\Delta}F(\tau)d\tau \in V_{-,c}^{1}(J\times\R^{2})\subset U^{p}(J\times\R^{2})$. The desired conclusion then follows from Lemma 4.34 in \cite{koch2014dispersive}.
\end{proof}

The following proposition was proved in \cite{Dodson2016} and is extremely useful in the estimation of $U_{\Delta}^{2}$ norms.
\begin{prop}[Lemma 3.4 of \cite{Dodson2016}]\label{prop:U2_duh} %Dodson's inhomogeneous estimate
Let $J=\bigcup_{m=1}^{k}J^{m}$ be a union of consecutive intervals $J^{m}=[a_{m},b_{m}]$, where $b_{m}=a_{m+1}$ for $m=1,\ldots,k-1$. Let $F\in L_{t}^{1}L_{x}^{2}(J\times\R^{2})$. Then for any $t_{0}\in J$,
	\begin{equation}
		\begin{split}
			\left\|\int_{t_{0}}^{t}e^{i(t-\tau)\Delta}F(\tau)d\tau\right\|_{U_{\Delta}^{2}(J\times\R^{2})} &\lesssim \sum_{m=1}^{k}\left\|\int_{J^{m}}e^{-i\tau}F(\tau)d\tau\right\|_{L^{2}(\R^{2})}\\
			&\phantom{=} +\left(\sum_{m=1}^{k}\left(\sup_{\|v_{m}\|_{V_{-,rc,\Delta}^{2}(J^{m}\times\R^{2})}=1}\left|\int_{J^{m}}\langle{F(\tau),v_{m}(\tau)}\rangle d\tau\right|\right)^{2}\right)^{1/2}
		\end{split}
	\end{equation}
\end{prop}

\begin{remark}
The assumption $F\in L_{t}^{1}L_{x}^{2}(J\times\R^{2})$ is purely qualitative: the implied constants in the statements of propositons \ref{prop:V_dual} and \ref{prop:U2_duh} do not depend on the quantity $\|F\|_{L_{t}^{1}L_{x}^{2}(J\times\R^{2})}$. Moreover, this condition will always be satisfied by Strichartz estimates.
\end{remark}

Linear and bilinear Strichartz estimates transfer to $U_{\Delta}^{p}$ and $V_{\Delta}^{p}$ spaces, a consequence of a more general transference principle underlying these spaces.
\begin{prop}[Transference principle, Proposition 2.19 of \cite{Hadac2009}]\label{prop:trans} %Transference principle
Let $I\subset \R$ be an interval, and let
\begin{equation}
T_{0}:L^{2}(\R^{2})\times\cdots\times L^{2}(\R^{2})\rightarrow L_{loc}^{1}(\R^{2})
\end{equation}
be an $n$-linear operator. If for some $1\leq p,q\leq \infty$, we have that
\begin{equation}
		\left\|T_{0}\left(e^{it\Delta}\phi_{1},\ldots,e^{it\Delta}\phi_{n}\right)\right\|_{L_{t}^{p}L_{x}^{q}(I\times\R^{2})}\leq C_{p,q}\prod_{i=1}^{n}\|\phi_{i}\|_{L^{2}(\R^{2})},
\end{equation}
then there exists an extension $T:U_{\Delta}^{p}(I\times\R^{2})\times\cdots\times U_{\Delta}^{p}(I\times\R^{2})\rightarrow L_{t}^{p}L_{x}^{q}(I\times\R^{2})$ such that
\begin{equation}
		\|T(u_{1},\ldots,u_{n})\|_{L_{t}^{p}L_{x}^{q}(I\times\R^{2})}\leq C_{p,q}\prod_{i=1}^{n}\|u_{i}\|_{U_{\Delta}^{p}(I\times\R^{2})}
\end{equation}
and $T(u_{1},\ldots,u_{n})(t)(x)=T_{0}(u_{1}(t),\ldots,u_{n}(t))(x)$ a.e.
\end{prop}

In the proof of our long-time Strichartz estimate (theorem \ref{thm:LTSE}), we make use of the following interpolation result for bounded linear operators from $U_{\Delta}^{p}$ into a Banach space. Although this result is not strictly necessary, it is convenient for avoiding the fact that the $V_{\Delta}^{2+}\subset U_{\Delta}^{2}$ but that $U_{\Delta}^{2}\subsetneq V_{\Delta}^{2}$ when applying bilinear estimates.

\begin{prop}[Linear interpolation, Propositon 2.20 of \cite{Hadac2009}]\label{prop:lin_interp}
Let $q>1$, $X$ be a Banach space, and $T:U_{\Delta}^{q}\rightarrow X$ be a bounded linear operator with norm $\|T\|_{U_{\Delta}^{q}\rightarrow X}\leq C_{q}$. Additionally, suppose that there is $C_{p}\in (0,C_{q}]$, for some $1\leq p<q$, such that $T$ is well-defined on $U_{\Delta}^{p}$ and satisfies the estimate $\|Tu\|_{X}\leq C_{p}\|u\|_{U_{\Delta}^{p}}$ for all $u\in U_{\Delta}^{p}$. Then $T$ satisfies the estimate
\begin{equation}
		\|Tu\|_{X} \leq \frac{4C_{p}}{\alpha_{p,q}}\left(\ln\frac{C_{q}}{C_{p}}+2\alpha_{p,q}+1\right)\|u\|_{V_{\Delta}^{p}}, \qquad \forall u\in V_{-,rc,\Delta}^{p},
\end{equation}
where $\alpha_{p,q}\coloneqq (1-\frac{p}{q})\ln(2)$.
\end{prop}

Let us now give an application of propositions \ref{prop:trans} and \ref{prop:lin_interp} to proving linear and bilinear estimates at the level of $U_{\Delta}^{p}, V_{\Delta}^{p}$ spaces (cf. Corollary 2.21 of \cite{Hadac2009}).
\begin{prop}\label{prop:bs_i}
Let $I\subset \R$ be an interval.
\begin{enumerate}[(i)]
\item\label{item:bs_i_S}
Let $(p,q)$ be an admissible pair. Then 
\begin{equation}
\|u\|_{L_{t}^{p}L_{x}^{q}(I\times\R^{2})} \lesssim_{p} \|u\|_{U_{\Delta}^{p}(I\times\R^{2})}.
\end{equation}
\item\label{item:bs_i_B}
If $u$ and $v$ have spatial Fourier supports in the annuli $|\xi| \sim N$ and $|\xi|\sim M$, respectively, where $M\ll N$, then for any $p\in [2,\infty]$,
\begin{equation}
\|uv\|_{L_{t}^{p}L_{x}^{p'}(I\times\R^{2})} \lesssim_{p} \left(\frac{M}{N}\right)^{1/p} \|u\|_{U_{\Delta}^{p}(I\times\R^{2})} \|v\|_{U_{\Delta}^{p}(I\times\R^{2})}
\end{equation}
and
\begin{equation}
\|uv\|_{L_{t}^{p}L_{x}^{p'}(I\times\R^{2})} \lesssim_{p} \left(\frac{M}{N}\right)^{1/p}\left(\ln\left(\frac{N}{M}\right)+1 \right)^{2}\|u\|_{V_{\Delta}^{p}(I\times\R^{2})} \|v\|_{V_{\Delta}^{p}(I\times\R^{2})}.
\end{equation}
\item\label{item:bs_i_cube}
Let $0<c_{1}<c_{2}$, and let $\delta\in (0,1)$. Let $\xi_{0}\in\R^{2}$, and let $Q_{\xi_{0}}$ be a cube centered at $\xi_{0}$ of side length $l(Q_{\xi_{0}})=M$.
\begin{enumerate}[a.]
\item
If $u$ has spatial Fourier support in the annulus $A(0,c_{1}N,c_{2}N)$ and $v$ has spatial Fourier support in the cube $Q_{\xi_{0}}\subset B(0,(1-\delta)c_{1}N)$, where $M\ll N$, then for any $p\in [2,\infty]$,
\begin{equation}
\|u v\|_{L_{t}^{p}L_{x}^{p'}(I\times\R^{2})} \lesssim_{p,c_{1},c_{2},\delta}\paren*{\frac{M}{N}}^{1/p} \|u\|_{U_{\Delta}^{p}(I\times\R^{2})} \|v\|_{U_{\Delta}^{p}(I\times\R^{2})}.
\end{equation}
and
\begin{equation}
\|uv\|_{L_{t}^{p}L_{x}^{p'}(I\times\R^{2})} \lesssim_{p,c_{1},c_{2},\delta} \paren*{\frac{M}{N}}^{1/p}\paren*{\ln\paren*{\frac{N}{M}}+1}^{2}\|u\|_{V_{\Delta}^{p}(I\times\R^{2})} \|v\|_{V_{\Delta}^{p}(I\times\R^{2})}.
\end{equation}
\item
If $u$ has spatial Fourier support in the $Q_{\xi_{0}}\subset A(0,c_{1}N,c_{2}N)$ and $v$ has spatial Fourier support in the ball $B(0,(1-\delta)c_{1}N)$, where $M\ll N$, then for any $p\in [2,\infty]$,
\begin{equation}
\|u v\|_{L_{t}^{p}L_{x}^{p'}(I\times\R^{2})} \lesssim_{p,c_{1},c_{2},\delta}\paren*{\frac{M}{N}}^{1/p} \|u\|_{U_{\Delta}^{p}(I\times\R^{2})} \|v\|_{U_{\Delta}^{p}(I\times\R^{2})}
\end{equation}
and
\begin{equation}
\|uv\|_{L_{t}^{p}L_{x}^{p'}(I\times\R^{2})} \lesssim_{p,c_{1},c_{2},\delta} \paren*{\frac{M}{N}}^{1/p}\paren*{\ln\paren*{\frac{N}{M}}+1}^{2}\|u\|_{V_{\Delta}^{p}(I\times\R^{2})} \|v\|_{V_{\Delta}^{p}(I\times\R^{2})}.
\end{equation}
\end{enumerate}
\item\label{item:bs_i_T}
Let $q>5/2$ and $c>0$. Then there exists a constant $C(c,q)>0$, such that for $u$ and $v$ satisfying $\supp(\hat{u}),\supp(\hat{v})\subset B(0,N)$ and $\dist\left(\supp(\hat{u}),\supp(\hat{v})\right) \geq cN$, we have the inequality
\begin{equation}
\|uv\|_{L_{t,x}^{q}(I\times\R^{2})} \leq C(c,q)N^{2-\frac{4}{q}} \|u\|_{U_{\Delta}^{q}(I\times\R^{2})} \|v\|_{U_{\Delta}^{q}(I\times\R^{2})}.
\end{equation}
\end{enumerate}
\end{prop}
\begin{proof}
We omit the proof of assertion \ref{item:bs_i_S}, as it is a straightforward consequence of proposition \ref{prop:trans}, and the proofs of assertions \ref{item:bs_i_cube} and \ref{item:bs_i_T}, as they follow from the same argument used to prove \ref{item:bs_i_B} together with corollary \ref{prop:cube_BS} and proposition \ref{prop:br_T}, respectively.

By modifying the definition of the Littlewood-Paley projectors if necessary, we may prove the assertion with $u$ and $v$ replaced by $P_{N}u$ and $P_{M}v$, respectively. Let $\tilde{P}_{N}$ and $\tilde{P}_{M}$ denote ``fattened" Littlewood-Paley projectors satisfying $\tilde{P}_{N}P_{N}=P_{N}$ and $\tilde{P}_{M}P_{M}=P_{M}$, respectively. The first inequality in \ref{item:bs_i_B} follows from propositions \ref{prop:bs_B} and \ref{prop:trans} with $T_{0}$ defined by
\begin{equation}
T_{0}(e^{it\Delta}\phi_{1},e^{it\Delta}\phi_{2}) \coloneqq \paren*{\tilde{P}_{N}e^{it\Delta}\phi_{1}}\paren*{\tilde{P}_{M}e^{it\Delta}\phi_{2}}
\end{equation}
and $C_{p,p'}=C_{p}(M/N)^{1/p}$.

For the second inequality in \ref{item:bs_i_B}, first note that for the case $p=\infty$, we have the stronger estimate
\begin{equation}
\|uv\|_{L_{t}^{\infty}L_{x}^{1}(I\times\R^{2})} \leq \|u\|_{V_{\Delta}^{\infty}(I\times\R^{2})} \|v\|_{V_{\Delta}^{\infty}(I\times\R^{2})}
\end{equation}
by H\"{o}lder's inequality and the fact that $V_{\Delta}^{\infty}\subset L_{t}^{\infty}L_{x}^{2}$. Now fix $p\in [2,\infty)$ and let $q\coloneqq 2p$. Let $u\in U_{\Delta}^{p}(I\times\R^{2})$. Let $X \coloneqq L_{t}^{p}L_{x}^{p'}(I\times\R^{2})$, and define the operator $T:U_{\Delta}^{p}(I\times\R^{2})\rightarrow X$ by
\begin{equation}
v\mapsto \paren*{\tilde{P}_{M}v}\paren*{\tilde{P}_{N}u}.
\end{equation}
Observe that $T$ has operator norm satisfying the inequality
\begin{equation}
\|T\|_{U_{\Delta}^{p}(I\times\R^{2})\rightarrow X} \leq \tilde{C}_{p}\left(\frac{M}{N}\right)^{1/p} \|\tilde{P}_{N}u\|_{U_{\Delta}^{p}(I\times\R^{2})} \eqqcolon C_{p}.
\end{equation}
where $\tilde{C}_{p}$ is a multiple (by an absolute constant) of the implicit constant in proposition \ref{prop:bs_B}. Also, observe that $T$ is well-defined on $U_{\Delta}^{q}(I\times\R^{2})$ and satisfies the estimate
\begin{align}
\|Tv\|_{X}=\|(\tilde{P}_{N}u)(\tilde{P}_{M}v)\|_{L_{t}^{p}L_{x}^{p'}(I\times\R^{2})} &\leq \|\tilde{P}_{N}u\|_{L_{t}^{2p}L_{x}^{2p'}(I\times\R^{2})} \|\tilde{P}_{M}v\|_{L_{t}^{2p}L_{x}^{2p'}(I\times\R^{2})} \nonumber\\
&\leq \tilde{C}_{q} \|\tilde{P}_{N}u\|_{U_{\Delta}^{p}(I\times\R^{2})}\|v\|_{U_{\Delta}^{q}(I\times\R^{2})} \nonumber\\
&\eqqcolon C_{q} \|v\|_{U_{\Delta}^{q}(I\times\R^{2})},
\end{align}
where we use H\"{o}lder's inequality to obtain the first inequality and assertion \ref{item:bs_i_S} together with the admissibility of $(2p,2p')$ to obtain the second inequality. Above, $\tilde{C}_{q}$ is a multiple (by an absolute constant) of the implicit constant in the homogeneous Strichartz estimate for $L_{t}^{2p}L_{x}^{2p'}$. Taking $C_{q}$ larger by a fixed absolute factor, we may assume that $C_{q}\geq C_{p}$. Therefore, by proposition \ref{prop:lin_interp}, we obtain that
\begin{align}
\|(\tilde{P}_{N}u)(\tilde{P}_{M}v)\|_{L_{t}^{p}L_{x}^{p'}(I\times\R^{2})} &= \|Tv\|_{X} \nonumber\\
&\lesssim_{p} \left(\frac{M}{N}\right)^{1/p}\left(\ln\left(\frac{N}{M}\right)+1\right) \|u\|_{U_{\Delta}^{p}(I\times\R^{2})}\|v\|_{V_{\Delta}^{p}(I\times\R^{2})} .
\end{align}
Replacing $u$ by $P_{N}u$ and $v$ by $P_{M}v$ in the preceding estimate shows that
\begin{equation}
\|(P_{N}u)(P_{M}v)\|_{L_{t}^{p}L_{x}^{p'}(I\times\R^{2})} \lesssim_{p} \left(\frac{M}{N}\right)^{1/p}\left(\ln\left(\frac{N}{M}\right)+1\right) \|P_{N}u\|_{U_{\Delta}^{p}(I\times\R^{2})}\|P_{M}v\|_{V_{\Delta}^{p}(I\times\R^{2})}.
\end{equation}
Repeating the argument above on $u$, we obtain the desired conclusion.
\end{proof}

The last result which we prove in this subsection concerns the invariance of $U_{\Delta}^{p}$ under the action of the symmetry group $G'$ (see definition \ref{def:G'}).

\begin{prop}[$G'$ invariance of $U_{\Delta}^{p}$]
Let $1\leq p<\infty$, and let $I\subset \R$. Then for $u\in U_{\Delta}^{p}(I\times\R^{2})$, and $g_{\theta,\xi_{0},x_{0},\lambda,t_{0}}\in G'$, we have that
\begin{equation}
\|u\|_{U_{\Delta}^{p}(I\times\R^{2})} = \|T_{g_{\theta,\xi_{0},x_{0},\lambda,t_{0}}}u\|_{U_{\Delta}^{p}(\lambda^{2}I\times\R^{2})}.
\end{equation}
\end{prop}
\begin{proof}
Given $\varepsilon>0$, let
\begin{equation}
u=\sum_{j}c_{j}u_{j}, \qquad u_{j} = \sum_{k} 1_{[t_{j,k},t_{j,k+1})}(t)e^{it\Delta}u_{j,k}
\end{equation}
be an atomic decomposition for $u$ such that $\sum_{j} |c_{j}| \leq \|u\|_{U_{\Delta}^{p}(I\times\R^{2})} +\varepsilon$. For each $j$, define the function $v_{j} \coloneqq T_{g_{\theta,\xi_{0},x_{0},\lambda,t_{0}}}u_{j}$, and define the function $v\coloneqq T_{g_{\theta,\xi_{0},x_{0},\lambda,t_{0}}}u$. We claim that each $v_{j}$ is a $U_{\Delta}^{p}(\lambda^{2}I\times\R^{2})$ atom. Indeed, observe that
\begin{equation}
v_{j}(t) = \sum_{k} \paren*{T_{g_{\theta,\xi_{0},x_{0},\lambda}}e^{it_{n}\Delta}\paren*{1_{[t_{j,k},t_{j,k+1})}(\cdot)e^{i(\cdot)\Delta}u_{j,k}}}(t) = \sum_{k}1_{[\lambda^{2}t_{j,k},\lambda^{2}t_{j,k+1})}(t)\paren*{T_{g_{\theta,\xi_{0},x_{0},\lambda}}e^{it_{n}\Delta}\paren*{e^{i(\cdot)\Delta}u_{j,k}}}(t).
\end{equation}
Since the free Schr\"{o}dinger equation is invariant under the action of $G'$, we have that
\begin{equation}
\paren*{T_{g_{\theta,\xi_{0},x_{0},\lambda}}e^{it_{n}\Delta}\paren*{e^{i(\cdot)\Delta}u_{j,k}}}(t) = e^{it\Delta}\paren*{g_{\theta,\xi_{0},x_{0},\lambda}e^{it_{n}\Delta}u_{j,k}}.
\end{equation}
Since $g_{\theta,\xi_{0},x_{0},\lambda}e^{it_{n}\Delta}$ is a unitary operator on $L_{x}^{2}$, the claim follows immediately. Therefore, $v=\sum_{j} c_{j}v_{j}$ is an atomic decomposition for $v$, and
\begin{equation}
\|v\|_{U_{\Delta}^{p}(\lambda^{2}I\times\R^{2})}\leq \sum_{j} |c_{j}| \leq \|u\|_{U_{\Delta}^{p}(I\times\R^{2})}+\varepsilon. 
\end{equation}
Sinve $\varepsilon>0$ was arbitrary, we conclude that $\|v\|_{U_{\Delta}^{p}(\lambda^{2}I\times\R^{2})} \leq \|u\|_{U_{\Delta}^{p}(I\times\R^{2})}$. By writing $u=T_{g_{\theta,x_{0},\xi_{0},\lambda,t_{0}}^{-1}}v$ and repeating the argument above, we obtain the reverse inequality, which completes the proof of the proposition.
\end{proof}

\section{Minimal mass blowup solutions}\label{sec:MMB}

\subsection{Overview}\label{ssec:MMB_ov}
In this section, we prove theorem \ref{thm:AP_reduc} and corollary \ref{cor:adm_sol}, which reduce our consideration to admissible blowup solutions to \eqref{eq:DS}. Compared to works on the mass-critical NLS, the primary new ingredient here is the $L^{p}$ and $C^{\alpha}$ boundedness properties of the operator $\E$. As for the NLS, the key convergence result used to prove theorem \ref{thm:AP_reduc} is the following Palais-Smale condition modulo the mass symmetry group $G$ (definition \ref{def:G_grp}).

\begin{prop}[Palais-Smale condition]\label{prop:PS}
Fix $\mu\in\{\pm 1\}$. If $\mu=1$, suppose that the critical mass $M_{c}\in (0,\infty)$; if $\mu=-1$, suppose that the critical mass $M_{c}\in (0,M(Q))$. Given a sequence of solutions $u_{n}:I_{n}\times\R^{2}\rightarrow\mathbb{C}$ to \eqref{eq:DS} and a sequence of times $t_{n}\in I_{n}$ such that $\limsup_{n\rightarrow\infty} M(u_{n})=M_{c}$ and
\begin{equation}
\lim_{n\rightarrow\infty} S_{\geq t_{n}}(u_{n}) = \lim_{n\rightarrow\infty}S_{\leq t_{n}}(u_{n})=\infty,
\end{equation}
the sequence $Gu_{n}(t_{n})$ has a convergent subsequence in $L^{2}(\R^{2}){/}G$.
\end{prop}
The primary ingredients for proving proposition \ref{prop:PS} are a quantitative stability lemma for equation \eqref{eq:DS}, which we prove in subsection \ref{ssec:MMB_stab} and the profile decomposition for the Schr\"{o}dinger propagator $e^{it\Delta}$ obtained by Merle and Vega in \cite{Merle1998}, which we recall in subsection \ref{ssec:MMB_cc}. Assuming proposition \ref{prop:PS}, the proof of which we defer to subsection \ref{ssec:MMB_PS}, theorem \ref{thm:AP_reduc} follows by the argument in \cite{Tao2008}, which we quickly sketch.

\begin{proof}
By definition of $M_{c}$, there exists a sequence of maximal-lifespan solutions $u_{n}:I_{n}\times\R^{2}\rightarrow\mathbb{C}$ to \eqref{eq:DS}, such that $M(u_{n})\leq M_{c}$ and $\lim_{n\rightarrow\infty}S_{I_{n}}(u_{n})=\infty$. For each $n\in\N$, the absolute continuity of the Lebesgue integral implies that there exists $t_{n}\in I_{n}$ such that
\begin{equation}
S_{\geq t_{n}}(u_{n}) = \min\left\{\frac{1}{2} S_{I_{n}}(u_{n}), 2^{n}\right\},
\end{equation}
and by time translation symmetry, we may assume that $t_{n}\equiv 0$. Applying proposition \ref{prop:PS}, up to a subsequence, there exists $u_{0}\in L^{2}(\R^{2})$ such that $Gu_{n}(0)\rightarrow Gu_{0}$ in the quotient space $L^{2}(\R^{2}){/}G$. Equivalently, there there exists a sequence $g_{n}\in G$ and $g_{0}\in G$ such that $g_{n}u_{n}(0) \rightarrow g_{0}u_{0}$ in $L^{2}$. So after relabeling, we may assume that $g_{n}\equiv 1$. Hence, $u_{n}(0)\rightarrow u_{0}$ in $L^{2}(\R^{2})$ as $n\rightarrow\infty$, which implies that $M(u_{0}) \leq M_{c}$.

Let $u:I\times\R^{2}\rightarrow\mathbb{C}$ be the maximal-lifespan solution to \eqref{eq:DS} with initial data $u_{0}$. We claim that $u$ blows up both forward and backward in time. Otherwise, by assertion \ref{item:LWP_bup} of theorem \ref{thm:LWP}, we have without loss of generality that $S_{\geq 0}(u)<\infty$. By assertion \ref{item:LWP_cd} of theorem \ref{thm:LWP}, $[0,\infty)\subset I_{n}$ for all $n\gg 1$ and
\begin{equation}
\limsup_{n\rightarrow\infty} S_{\geq 0}(u_{n}) < \infty,
\end{equation}
which contradicts our choice of the sequence $u_{n}$. Hence, $M(u_{0}) \geq M_{c}$, from which we conclude that $M(u_{0})=M_{c}$.

To see that blow up both forward and backward in time implies almost periodicity modulo $G$, let $Gu(t_{n}')$ be any sequence in the orbit $\{Gu(t): t\in I\}$. Since $S_{\geq t_{n}'}(u)=S_{\leq t_{n}'}(u)=\infty$, we can apply proposition \ref{prop:PS} to obtain a convergent subsequence in $L^{2}(\R^{2}){/}G$. One can then show that precompactness in $L^{2}(\R^{2}){/}G$ implies (actually is equivalent to) the existence of parameters $x(t),\xi(t),N(t)$ and compactness modulus function $C(\cdot)$.
\end{proof}

\subsection{Stability}\label{ssec:MMB_stab}
In this subsection, we extend the local Cauchy theory for \eqref{eq:DS} by proving a stability result for solutions (cf. Lemma 3.1 in \cite{Tao2008} and Theorem 3.7 in \cite{KillipClay}).

\begin{lemma}[Stability]\label{lem:stab} %Stability lemma
Fix $\mu\in\{\pm 1\}$. Let $I$ be a compact interval, and let $\tilde{u}$ be an approximate (strong) solution to \eqref{eq:DS} in the sense that
\begin{equation}
(i\p_{t}+\Delta)\tilde{u}=F(\tilde{u})+e, \qquad (t,x)\in I\times\R^{2},
\end{equation}
for some spacetime function $e$. Assume that
\begin{align}
\|\tilde{u}\|_{L_{t}^{\infty}L_{x}^{2}(I\times\R^{2})} &\leq M\\
\|\tilde{u}\|_{L_{t,x}^{4}(I\times\R^{2})} &\leq L,
\end{align}
for some constants $M,L>0$. Let $t_{0}\in I$, and let $u(t_{0})$ satisfy
\begin{equation}
\|u(t_{0})-\tilde{u}(t_{0})\|_{L^{2}(\R^{2})} \leq M'
\end{equation}
for some constant $M'>0$. Assume the smallness conditions
\begin{align}
\|e^{i(t-t_{0})\Delta}\paren*{u(t_{0})-\tilde{u}(t_{0})}\|_{L_{t,x}^{4}(I\times\R^{2})} &\leq \varepsilon \\
\|e\|_{L_{t,x}^{4/3}(I\times\R^{2})} &\leq \varepsilon,
\end{align}
for some $0<\varepsilon\leq \varepsilon_{1}$, where $\varepsilon_{1}(M,M',L)$ is a small constant. Then there exists a (unique) solution to \eqref{eq:DS} on $I\times\R^{2}$ with initial data $u(t_{0})$ at time $t=t_{0}$ satisfying the conditions
\begin{align}
\|u-\tilde{u}\|_{L_{t,x}^{4}(I\times\R^{2})} &\leq C(M,M',L)\varepsilon, \\
\|u-\tilde{u}\|_{L_{t}^{\infty}L_{x}^{2}(I\times\R^{2})} &\leq C(M,M',L)M',\\
\|u\|_{L_{t,x}^{4}(I\times\R^{2})} &\leq C(M,M',L).
\end{align}
\end{lemma}

\begin{proof} %Proof of stability lemma
We first prove the lemma under assumption that $L\leq \varepsilon_{0}$, where $\varepsilon_{0}=\varepsilon_{0}(M,M')>0$ is sufficiently small. By symmetry, we may assume without loss of generality that $t_{0}=\inf I$. Let $u:I'\times\R^{2}\rightarrow\mathbb{C}$ be the maximal-lifespan solution to \eqref{eq:DS} with initial data $u(t_{0})$, and write $u=\tilde{u}+w$ on the time interval $I'' \coloneqq I\cap I'$. Then $w$ is a strong solution to the Cauchy problem
\begin{equation}
\begin{cases}
(i\p_{t}+\Delta)w = F(\tilde{u}+w)-F(\tilde{u})-e, & (t,x)\in I''\times\R^{2} \\
w(t_{0}) = u(t_{0})-\tilde{u}(t_{0}), & x\in\R^{2}
\end{cases}.
\end{equation}
By Duhamel's formula,
\begin{equation}
w(t)=e^{i(t-t_{0})\Delta}w(t_{0})-i\int_{t_{0}}^{t}e^{i(t-\tau)\Delta}\left[F\paren*{\tilde{u}(\tau)+w(\tau)}-F(\tilde{u}(\tau))-e(\tau)\right]d\tau, \qquad t\in I''.
\end{equation}
For $T\in I''$, define the quantities $A(T) \coloneqq \|F(\tilde{u}+w)-F(\tilde{u})\|_{L_{t,x}^{4}([t_{0},T]\times\R^{2})}$ and $X(T) \coloneqq \|w\|_{L_{t,x}^{4}([t_{0},T]\times\R^{2})}$. Using the identity
\begin{equation}
F\paren*{\tilde{u}+w}-F(\tilde{u}) = -\L(|\tilde{u}|^{2})w - \L\paren*{2\Re{\tilde{u}\bar{w}}+|w|^{2}}(\tilde{u}+w),
\end{equation}
we see from H\"{o}lder's and triangle inequalities and Plancherel's theorem that
\begin{align}
A(T) &\lesssim \|w\|_{L_{t,x}^{4}([t_{0},T]\times\R^{2})}\|\tilde{u}\|_{L_{t,x}^{4}([t_{0},T]\times\R^{2})}^{2} \nonumber\\
&\phantom{=}+\paren*{\|\tilde{u}\|_{L_{t,x}^{4}([t_{0},T]\times\R^{2})}\|w\|_{L_{t,x}^{4}([t_{0},T]\times\R^{2})}+\|w\|_{L_{t,x}^{4}([t_{0},T]\times\R^{2})}^{2}} \paren*{\|\tilde{u}\|_{L_{t,x}^{4}([t_{0},T]\times\R^{2})}+\|w\|_{L_{t,x}^{4}([t_{0},T]\times\R^{2})}} \nonumber \\
&\lesssim L^{2}X(T)+LX(T)^{2}+X(T)^{3}.
\end{align}
By Strichartz estimates, the triangle inequality, and the assumptions of the lemma, it follows that
\begin{align}
X(T) &\lesssim \|e^{i(t-t_{0})\Delta}w(t_{0})\|_{L_{t,x}^{4}([t_{0},T]\times \R^{2})}+\paren*{\|(F(\tilde{u}+w)-F(\tilde{u}))\|_{L_{t,x}^{4/3}([t_{0},T]\times \R^{2})}+\|e\|_{L_{t,x}^{4/3}([t_{0},T]\times\R^{2})}} \nonumber\\
&\leq \varepsilon+\paren*{\|F(\tilde{u}+w)-F(\tilde{u})\|_{L_{t,x}^{4/3}([t_{0},T]\times\R^{2})}+\varepsilon} \nonumber\\
&\leq 2\varepsilon+A(T).
\end{align}
Combining the inequalities for $A(T)$ and $X(T)$, we obtain that
\begin{equation}
A(T) \lesssim L^{2}(\varepsilon+A(T)) + (\varepsilon+A(T))^{3} \leq \varepsilon_{0}^{2}(\varepsilon+A(T)) + (\varepsilon+A(T))^{3}.
\end{equation}
Hence, for $\varepsilon_{0}>0$ sufficiently small, we conclude from a standard continuity argument that $A(T)\lesssim \varepsilon$ for all $T\in I''$, which implies that $X(T)\lesssim \varepsilon$ for all $T\in I''$. By the triangle inequality,
\begin{equation}
\|u\|_{L_{t,x}^{4}([t_{0},T]\times\R^{2})} \leq \|\tilde{u}\|_{L_{t,x}^{4}([t_{0},T]\times\R^{2})} + \|w\|_{L_{t,x}^{4}([t_{0},T]\times\R^{2})} \lesssim L + \varepsilon \leq \varepsilon_{0}+\varepsilon, \qquad \forall T\in I'',
\end{equation}
and by Duhamel's formula, triangle inequality, and Strichartz estimates,
\begin{equation}
\|u-\tilde{u}\|_{L_{t}^{\infty}L_{x}^{2}([t_{0},T]\times\R^{2})} \lesssim \|u(t_{0})-\tilde{u}(t_{0})\|_{L^{2}(\R^{2})} + \varepsilon \leq M'+\varepsilon, \qquad \forall T\in I''.
\end{equation}
To see that $I=I''$, we argue as follows. If $\sup I' < \sup I$, so in particular $\sup I'$ is finite, then by assertion \ref{item:LWP_bup} of theorem \ref{thm:LWP},
\begin{equation}
\lim_{T\rightarrow \sup I'^{-}} \|u\|_{L_{t,x}^{4}([t_{0},T]\times\R^{2})} = \infty,
\end{equation}
which contradicts that $\|u\|_{L_{t,x}^{4}([t_{0},T]\times\R^{2})} \lesssim \varepsilon_{0}+\varepsilon$ for all $T\in I''$.

For the general case of arbitrarily large $L$, we use the absolute continuity of the Lebesgue integral to subdivide $I$ into $J\sim (1+\frac{L}{\varepsilon})^{4}$ compact subintervals $I_{j} \coloneqq [t_{j},t_{j+1}]$, for $0\leq j < J$, such that
\begin{equation}
\|\tilde{u}\|_{L_{t,x}^{4}(I_{j}\times\R^{2})} \leq \varepsilon_{0},
\end{equation}
where $\varepsilon_{0}=\varepsilon_{0}(M,2M')$ is the small constant determined above with $M'$ replaced by $2M'$. We then use the result of the preceding paragraph on each subinterval $I_{j}$ and sum the estimates to obtain the desired conclusion. We omit the details.
\end{proof}

\subsection{Concentration compactness}\label{ssec:MMB_cc} %Concentration compactness
In this subsection, we collect some standard results from the mass-critical theory of concentration compactness. The material here is not specific to the eeDS, as it pertains to the linear part of the equation, which is of course the same as for nonlinear Schr\"{o}dinger equations.

Since the free Schr\"{o}dinger equation is invariant under the action of $e^{it_{0}\Delta}$ for fixed time $t_{0}$, we first enlarge the symmetry group $G$ by adding the action of the free propagator $e^{it\Delta}$.

\begin{mydef}[Enlarged group $G'$]\label{def:G'}
For a phase $\theta \in\R/2\pi\mathbb{Z}$, position $x_{0}\in\R^{2}$, frequency $\xi_{0}\in\R^{2}$, scaling parameter $\lambda>0$, and time $t_{0}\in\R$, we define the unitary operator $g_{\theta_{0},x_{0},\xi_{0},\lambda,t_{0}}:L^{2}(\R^{2})\rightarrow L^{2}(\R^{2})$ by
\begin{equation}
g_{\theta,\xi_{0},x_{0},\lambda,t_{0}} \coloneqq g_{\theta,\xi_{0},x_{0},\lambda}e^{it_{0}\Delta}.
\end{equation}
We denote the collection of all such transformations by $G'$. We define the $G'$ action on spacetime functions $u:\R\times\R^{2}\rightarrow\mathbb{C}$ by
\begin{equation}
(T_{g_{\theta,\xi_{0},x_{0},\lambda,t_{0}}}u)(t,x) \coloneqq \frac{1}{\lambda}e^{i\theta}e^{ix\cdot\xi_{0}}e^{-it|\xi_{0}|^{2}}(e^{it_{0}\Delta}u)\left(\frac{t}{\lambda^{2}}, \frac{x-x_{0}-2\xi_{0}t}{\lambda}\right),
\end{equation}
which may equivalently be expressed as
\begin{equation}
(T_{g_{\theta,\xi_{0},x_{0},\lambda,t_{0}}}u)(t) = g_{\theta-t|\xi_{0}|^{2},\xi_{0},x_{0}+2\xi_{0}t,\lambda,t_{0}}\paren*{u\paren*{\frac{t}{\lambda^{2}}}}.
\end{equation}
\end{mydef}

One can show that $G'$ is a group and that the action of $G'$ on global spacetime functions preserves both the class and scattering size of solutions of the free Schr\"{o}dinger equation. Furthermore, we may topologize $G'$ by endowing it with the strong operator topology.

To state the linear profile decomposition, we need to define what it means for sequences in $G'$ to be asymptotically orthogonal.

\begin{mydef}[Asymptotic orthogonality]\label{def:as_o}
We say that two sequences $\{g_{n}\}_{n=1}^{\infty}, \{g_{n}'\}_{n=1}^{\infty}\in G'$ are \emph{asymptotically orthgonal} if the sequence $g_{n}^{-1}g_{n}'$ diverges to $\infty$, by which we mean that given any compact subset $K\subset G'$, there exists an $n(K)\in\mathbb{N}$ such that $g_{n}^{-1}g_{n}'\notin K$ for all $n\geq n_{0}$.
\end{mydef}
It is classical that asymptotic orthogonality is equivalent to
\begin{equation}\label{eq:as_o}
	\lim_{n\rightarrow\infty}\frac{\lambda_{n}^{(j)}}{\lambda_{n}^{(k)}}+\frac{\lambda_{n}^{(k)}}{\lambda_{n}^{(j)}}+\lambda_{n}^{(j)}\lambda_{n}^{(k)}|\xi_{n}^{(j)}-\xi_{n}^{(k)}|^{2}+\frac{|t_{n}^{(j)}(\lambda_{n}^{(j)})^{2}-t_{n}^{(k)}(\lambda_{n}^{(k)})^{2}|}{\lambda_{n}^{(j)}\lambda_{n}^{(k)}}+\frac{|x_{n}^{(j)}-x_{n}^{(k)}-2t_{n}^{(j)}(\lambda_{n}^{(j)})^{2}(\xi_{n}^{(j)}-\xi_{n}^{(k)})|^{2}}{\lambda_{n}^{(j)}\lambda_{n}^{(k)}}=\infty,
\end{equation}
for $j\neq k$.

We record some well-known consequences of asymptotic orthogonality (see section 4 of \cite{Tao2008}).
\begin{lemma}\label{lem:as_o}
Let $g_{n},g_{n}'$ be asymptotically orthogonal sequences in $G'$. Then
\begin{equation}
\lim_{n\rightarrow\infty} \ipp{g_{n}f,g_{n}'f'}_{L^{2}(\R^{2})}=0, \qquad \forall f,f'\in L^{2}(\R^{2}).
\end{equation}
If $v,v'\in L_{t,x}^{4}(\R\times\R^{2})$ and $\theta\in (0,1)$, then
\begin{equation}
	\lim_{n\rightarrow\infty} \| |T_{g_{n}}v|^{1-\theta} |T_{g_{n}'}v'|^{\theta}\|_{L_{t,x}^{4}(\R\times\R^{2})}=0.
\end{equation}
\end{lemma}

Using the preceding lemma, one obtains the following asymptotic ``decoupling'' of the mass and scattering size functionals evaluated on sums of functions.
\begin{prop}\label{prop:as_dcup}
Let $J\in\N$. If $g_{n}^{(1)},\ldots,g_{n}^{(J)}$ are sequences in $G'$ which are asymptotically orthogonal and $f^{(1)},\ldots,f^{(J)}\in L_{x}^{2}(\R^{2})$ and $v^{(1)},\ldots,v^{(J)}\in L_{t,x}^{4}(\R\times\R^{2})$, then
\begin{equation}
\lim_{n\rightarrow\infty} \brac*{M\paren*{\sum_{j=1}^{J}g_{n}^{(j)}f^{(j)}}-\sum_{j=1}^{J}M(f^{(j)})}=0
\end{equation}
and
\begin{equation}
\lim_{n\rightarrow\infty} \brac*{S\paren*{\sum_{j=1}^{J}g_{n}^{(j)}v^{(j)}} - \sum_{j=1}^{J}S(g_{n}^{(j)}v^{(j)})}=0.
\end{equation}
\end{prop}

Lastly, we recall the profile decomposition for the 2D linear Schr\"{o}dinger equation due to Merle and Vega (\cite{Merle1998}). We present the formulation from \cite{KillipClay} specialized to the 2D case.

\begin{thm}[Linear profile decompositon, \cite{Merle1998},\cite{KillipClay}]\label{prop:lpd} %Linear profile decomposition
Let $u_{n}$ be a bounded sequence in $L^{2}(\R^{2})$. Passing to a subsequence if necessary, there exists $J^{*}\in\N_{0}\cup\{\infty\}$, functions $\{\phi^{(j)}\}_{j=1}^{J^{*}} \in L^{2}(\R^{2})$, and pairwise asymptotically orthogonal sequences of group elements $\{g_{n}^{(j)}\}_{j=1}^{J^{*}}\in G'$ so that implicitly defining $w_{n}^{(J)}$ by
\begin{equation}
u_{n}=\sum_{j=1}^{J}g_{n}^{(j)}\phi^{(j)}+w_{n}^{(J)}, \qquad \forall J\in \{1,\ldots,J^{*}\},
\end{equation}
the following properties hold:

\begin{equation}\label{eq:as_ss}
\lim_{J\rightarrow J^{*}}\limsup_{n\rightarrow\infty}S\paren*{e^{it\Delta}w_{n}^{(J)}}=0,
\end{equation}

\begin{equation}
(g_{n}^{(j)})^{-1}w_{n}^{(J)} \xrightharpoonup[L^{2}(\R^{2})]{n\rightarrow\infty} 0, \qquad \forall j\in\{1,\ldots,J\}, \label{eq:w_L2}
\end{equation}
and
\begin{equation}\label{eq:m_dcup}
\sup_{J\in\{1,\ldots,J^{*}\}}\lim_{n\rightarrow\infty}\left|M(u_{n})-\sum_{j=1}^{l}M(\phi^{(j)})-M(w_{n}^{(J)})\right|=0.
\end{equation}
We refer to \eqref{eq:as_ss} as asymptotically vanishing scattering size and \eqref{eq:m_dcup} as mass decoupling.
\end{thm}

\subsection{Proof of Palais-Smale condition}\label{ssec:MMB_PS}
We now turn to proving proposition \ref{prop:PS}. By translating the $u_{n}$ in time and relabeling, we may assume that $t_{n}\equiv 0$, and therefore the assumption of the proposition becomes
\begin{equation}
\lim_{n\rightarrow\infty}S_{\geq 0}(u_{n})=\lim_{n\rightarrow\infty}S_{\leq 0}(u_{n})=\infty
\end{equation}
Applying the linear profile decomposition (proposition \ref{prop:lpd}) to the bounded sequence $u_{n}(0)$ and passing to a subsequence if necessary, we obtain the decomposition
\begin{equation}
u_{n}(0)=\sum_{j=1}^{J}g_{n}^{(j)}\phi^{(j)}+w_{n}^{(J)}, \qquad \forall J\in\{1,\ldots,J^{*}\}.
\end{equation}
By definition of $G'$, we can factorize the group element $g_{n}^{(j)}$ as $g_{n}^{(j)}=h_{n}^{(j)}e^{it_{n}^{(j)}\Delta}$, where $t_{n}^{(j)}\in\R$ and $h_{n}^{(j)}\in G$. By a diagonalization argument, we may assume that for all $j\in\{1,\ldots,J^{*}\}$, the sequence $t_{n}^{(j)}$ converges to some time $t^{(j)}\in [-\infty,+\infty]$. Moreover, if the limit $t^{(j)}\in (-\infty,+\infty)$, then we may write
\begin{equation}
g_{n}^{(j)}\phi^{(j)}=h_{n}^{(j)}e^{i(t_{n}^{(j)}-t^{(j)})\Delta}e^{it^{(j)}\Delta}\phi^{(j)},
\end{equation}
so that after modifying the $\phi^{(j)}$ and relabeling, we may assume that $t_{n}^{(j)}\rightarrow t^{(j)}\in\{0,\pm\infty\}$. Now if $\lim_{n\rightarrow\infty}t_{n}^{(j)}=0$, then by absorbing the error term $h_{n}^{(j)}\paren*{e^{it_{n}^{(j)}\Delta}\phi^{(j)}-\phi^{(j)}}$ into $w_{n}^{(J)}$, we may assume that $t_{n}^{(j)}\equiv 0$.

From the mass decoupling property \eqref{eq:m_dcup}, we obtain that
\begin{align}
\lim_{J\rightarrow J^{*}}\sum_{j=1}^{J}M(\phi^{(j)}) &= \lim_{J\rightarrow J^{*}}\paren*{\sum_{j=1}^{J}M(\phi^{(j)})+\lim_{n\rightarrow\infty}\paren*{M(u_{n})-\sum_{j=1}^{J}M(\phi^{(j)})-M(w_{n}^{(J)})}} \nonumber\\
&\leq \lim_{J\rightarrow J^{*}}\limsup_{n\rightarrow\infty} \paren*{M(u_{n})-M(w_{n}^{(J)})} \nonumber\\
&\leq \limsup_{n\rightarrow\infty} M(u_{n}) \nonumber\\
&\leq M_{c},
\end{align}
where the ultimate inequality follows from our choice of the sequence $u_{n}$. In particular, $\sup_{j\in\{1,\ldots,J^{*}\}}M(\phi^{(j)}) \leq M_{c}$. We next show that the supremum is actually equal to $M_{c}$.

\begin{lemma}\label{lem:M_sup}
$\sup_{j\in\{1,\ldots,J^{*}\}} M(\phi^{(j)})=M_{c}$.
\end{lemma}

\begin{proof} %Proof of lemma M_sup
We prove the lemma by contradiction. Suppose that there exists $\varepsilon>0$ such that $\sup_{j\in\{1,\ldots,J^{*}\}}M(\phi^{(j)})\leq M_{c}-\varepsilon$. Since the function $L(\cdot)$ defined in \eqref{eq:L+} and \eqref{eq:L-} is monotonically nondecreasing, finite, and continuous on the interval $[0,M_{c}-\varepsilon]$, by revisiting theorem \ref{thm:LWP}, one can show that there exists a real number $B=B(\varepsilon,M_{c})>0$ such that $L(M)\leq BM$ for all $M\in [0,M_{c}-\varepsilon]$.

Next, for each $j\in \{1,\ldots,J^{*}\}$, we define the \emph{nonlinear profile} $v^{(j)}:\R\times\R^{2}\rightarrow\mathbb{C}$ associated to $\phi^{(j)}$ and depending on the limiting value of $t_{n}^{(j)}$ as follows.
\begin{itemize}[leftmargin=*]
\item
If $t_{n}^{(j)}\equiv 0$, then define $v^{(j)}$ to be the maximal-lifespan solution to \eqref{eq:DS} with initial datum $v^{(j)}(0)=\phi^{(j)}$.
\item
If $\lim_{n\rightarrow\infty}t_{n}^{(j)}=+\infty$, then define $v^{(j)}$ to be the maximal-lifespan solution to \eqref{eq:DS} which scatters forward in time to $e^{it\Delta}\phi^{(j)}$.
\item
If $\lim_{n\rightarrow\infty}t_{n}^{(j)}=-\infty$, then define $v^{(j)}$ to be the maximal-lifespan solution to \eqref{eq:DS} which scatters backwards in time to $e^{it\Delta}\phi^{(j)}$.
\end{itemize}
By mass conservation,
\begin{equation}
M(v^{(j)})=M(\phi^{(j)})\leq M_{c}-\varepsilon,
\end{equation}
so by definition of the critical mass, we have that $L(M_{c}-\varepsilon)<\infty$, from which it follows that $v^{(j)}$ is indeed defined globally in time, as implicitly claimed above. Moreover, by definition of $L(\cdot)$,
\begin{equation}\label{eq:nlp_bd}
S(v^{(j)})\leq L(M(\phi^{(j)}))\leq BM(\phi^{(j)}).
\end{equation}

Now for each $J\leq \{1,\ldots,J^{*}\}$, we define the sequence of approximants $u_{n}^{(J)}\in C_{t}^{0}L_{x}^{2}(\R\times\R^{2})\cap L_{t,x}^{4}(\R\times\R^{2})$ to $u_{n}$ by
\begin{equation}
u_{n}^{(J)}(t) \coloneqq \sum_{j=1}^{J} v_{n}^{(j)}(t) + e^{it\Delta}w_{n}^{(J)},
\end{equation}
where
\begin{equation}
v_{n}^{(J)}(t) \coloneqq T_{h_{n}^{(j)}}[v^{(j)}(\cdot+t_{n}^{(j)})](t).
\end{equation}
By proposition \ref{prop:as_dcup} (using that $T_{h_{n}^{(j)}}$ preserves the scattering size), we have that
\begin{equation}
\lim_{n\rightarrow\infty} S\paren*{\sum_{j=1}^{J}v_{n}^{(j)}} = \sum_{j=1}^{J}S(v^{(j)}),
\end{equation}
and so we obtain the estimate
\begin{equation}
\lim_{n\rightarrow\infty}S(u_{n}^{(J)}) \lesssim \limsup_{n\rightarrow\infty} S\paren*{\sum_{j=1}^{J}v_{n}^{(j)}}+\limsup_{n\rightarrow\infty}S\paren*{e^{it\Delta}w_{n}^{(J)}} =\sum_{j=1}^{J}S(v^{(j)})+\limsup_{n\rightarrow\infty}S\paren*{e^{it\Delta}w_{n}^{(J)}}.
\end{equation}
Taking the $\lim_{J\rightarrow J^{*}}$ of both sides and using \eqref{eq:as_ss}, we obtain that
\begin{equation}\label{eq:ss_bd}
\lim_{J\rightarrow J^{*}}\limsup_{n\rightarrow\infty} S(u_{n}^{(J)}) \lesssim \lim_{J\rightarrow J^{*}}\sum_{j=1}^{J}S(v^{(j)})\leq \lim_{J\rightarrow J^{*}}\sum_{j=1}^{J}BM(\phi^{(j)})=B\sum_{j=1}^{J^{*}}M(\phi^{(j)})\leq BM_{c}.
\end{equation}

\begin{lemma}[Asymptotic agreement of initial data]\label{lem:as_id} %Asymptotic ageement of the initial data
For each finite $J\in \{1,\ldots,J^{*}\}$,
\begin{equation}
\lim_{n\rightarrow\infty} M\paren*{u_{n}^{(J)}(0)-u_{n}(0)}=0.
\end{equation}
\end{lemma}
\begin{proof}
We first claim that it suffices to show that
\begin{equation}\label{eq:as_id_red}
\lim_{n\rightarrow\infty}M\paren*{v_{n}^{(j)}(0)-g_{n}^{(j)}\phi^{(j)}}=0, \qquad \forall j\in\{1,\ldots,J\}.
\end{equation}
Indeed, observe that
\begin{equation}
u_{n}^{(J)}(0)-u_{n}(0) = \paren*{\sum_{j=1}^{J}v_{n}^{(j)}(0)+e^{i(0)\Delta}w_{n}^{(J)}}-\paren*{\sum_{j=1}^{J}g_{n}^{(j)}\phi^{(j)}+w_{n}^{(J)}} = \sum_{j=1}^{J}\paren*{v_{n}^{(j)}(0)-g_{n}^{(j)}\phi^{(j)}},
\end{equation}
which by Cauchy-Schwarz implies that
\begin{equation}
M\paren*{u_{n}^{(J)}(0)-u_{n}(0)} \leq J\sum_{j=1}^{J} M\paren*{v_{n}^{(j)}(0)-g_{n}^{(j)}\phi^{(j)}}.
\end{equation}
Since $J\in\mathbb{N}$ is fixed,
\begin{equation}
\limsup_{n\rightarrow\infty} M\paren*{u_{n}^{(J)}(0)-u_{n}(0)} \leq J\sum_{j=1}^{J}\limsup_{n\rightarrow\infty}M\paren*{v_{n}^{(j)}(0)-g_{n}^{(j)}\phi^{(j)}} =0,
\end{equation}
comleting the proof of the claim.

We turn to showing \eqref{eq:as_id_red}. Observe that by $G$-invariance of the mass functional,
\begin{equation}
M\paren*{v_{n}^{(j)}(0)-g_{n}^{(j)}\phi^{(j)}}=M\paren*{h_{n}^{(j)}(v^{(j)}(t_{n}^{(j)}))-h_{n}^{(j)}(e^{it_{n}^{(j)}\Delta}\phi^{(j)})}=M\paren*{v^{(j)}(t_{n}^{(j)})-e^{it_{n}^{(j)}\Delta}\phi^{(j)}}.
\end{equation}
To conclude the proof of \eqref{eq:as_id_red}, we consider the cases for the construction of $v^{(j)}$.
\begin{itemize}[leftmargin=*]
\item
If $t_{n}^{(j)}\equiv 0$, then since $v^{(j)}(0) = \phi^{(j)}$ by definition, so that
\begin{equation}
\lim_{n\rightarrow\infty} M\left(v^{(j)}(t_{n}^{(j)})-e^{it_{n}^{(j)}\Delta}\phi^{(j)}\right)=M\left(v^{(j)}(0)-\phi^{(j)}\right)=0.
\end{equation}
\item
If $\lim_{n\rightarrow\infty}t_{n}^{(j)}=\pm\infty$, then the desired result follows from the definitions of forward scattering and backward scattering, respectively.
\end{itemize}
\end{proof}

With the exception of the proof of the stability lemma, we have so far not had to consider the fact that the eeDS nonlinearity differs from that of NLS by a nonlocal term. However, this difference will cause a nontrivial technical complication in the proof of the subsequent lemma.

\begin{lemma}[Asymptotic solvability of eqn.]\label{lem:as_sol}
\begin{equation}
\lim_{J\rightarrow J^{*}}\limsup_{n\rightarrow\infty} \|(i\partial_{t}+\Delta)u_{n}^{(J)}-F(u_{n}^{(J)})\|_{L_{t,x}^{4}(\R\times\R^{2})}=0.
\end{equation}
\end{lemma}
\begin{proof}
Since $v^{(j)}$ is a solution to the eeDS equation and $e^{it\Delta}w_{n}^{(l)}$ is a solution to the free Schr\"{o}dinger equation, we have that
\begin{equation}
(i\partial_{t}+\Delta)u_{n}^{(J)}=\sum_{j=1}^{J}(i\partial_{t}+\Delta)v_{n}^{(j)}+\underbrace{(i\partial_{t}+\Delta)(e^{it\Delta}w_{n}^{(l)})}_{=0}=\sum_{j=1}^{J}F(v_{n}^{(j)}).
\end{equation}
By the triangle inequality,
\begin{align}
\|(i\partial_{t}+\Delta)u_{n}^{(J)}-F(u_{n}^{(J)})\|_{L_{t,x}^{4/3}(\R\times\R^{2})}&\leq \|(i\partial_{t}+\Delta)u_{n}^{(J)}-F(u_{n}^{(J)}-e^{it\Delta}w_{n}^{(J)})\|_{L_{t,x}^{4/3}(\R\times\R^{2})} \nonumber\\
&\phantom{=}+\|F(u_{n}^{(J)}-e^{it\Delta}w_{n}^{(J)})-F(u_{n}^{(J)})\|_{L_{t,x}^{4/3}(\R\times\R^{2})} \nonumber\\
&=\|\paren*{\sum_{j=1}^{J}F(v_{n}^{(j)})}-F\left(\sum_{j=1}^{J}v_{n}^{(j)}\right)\|_{L_{t,x}^{4/3}(\R\times\R^{2})} \nonumber\\
&\phantom{=}+\|F(u_{n}^{(J)}-e^{it\Delta}w_{n}^{(J)})-F(u_{n}^{(J)})\|_{L_{t,x}^{4/3}(\R\times\R^{2})}.
\end{align}
Therefore, it suffices  to prove
\begin{enumerate}[(i)]
\item\label{item:as_sol_1}
\begin{equation}
\lim_{J\rightarrow J^{*}}\limsup_{n\rightarrow\infty} \|\sum_{j=1}^{J}F(v_{n}^{(j)})-F\left(\sum_{j=1}^{J}v_{n}^{(j)}\right)\|_{L_{t,x}^{4/3}(\R\times\R^{2})} = 0
\end{equation}
and
\item\label{item:as_sol_2}
\begin{equation}
\lim_{J\rightarrow J^{*}} \limsup_{n\rightarrow\infty} \|F(u_{n}^{(J)}-e^{it\Delta}w_{n}^{(J)})-F(u_{n}^{(J)})\|_{L_{t,x}^{4/3}(\R\times\R^{2})} = 0.
\end{equation}
\end{enumerate}

To show \ref{item:as_sol_2}, we recall that $F(u)=-\mathcal{L}(|u|^{2})u$ and use the elementary identity
\begin{equation}
\mathcal{L}(|u-v|^{2})(u-v)-\mathcal{L}(|u|^{2})u=-v\mathcal{L}(|u|^{2})+\mathcal{L}\paren*{-2\Re{u\bar{v}}+|v|^{2}}(u-v),
\end{equation}
H\"{o}lder's inequality, and the $L^{2}$ boundedness of the operator $\mathcal{L}$ to obtain that
\begin{equation}
\|F(u_{n}^{(J)} - e^{it\Delta}w_{n}^{(J)})-F(u_{n}^{(J)})\|_{L_{t,x}^{4/3}(\R\times\R^{2})}\lesssim \|e^{it\Delta}w_{n}^{(J)}\|_{L_{t,x}^{4}(\R\times\R^{2})} \paren*{\|u_{n}^{(J)}\|_{L_{t,x}^{4}(\R\times\R^{2})}^{2}+\|e^{it\Delta}w_{n}^{(J)}\|_{L_{t,x}^{4}(\R\times\R^{2})}^{2}}.
\end{equation}
Taking $\lim_{J\rightarrow J^{*}}\limsup_{n\rightarrow\infty}$ of both sides and using \eqref{eq:as_ss} and \eqref{eq:ss_bd} completes the proof of \ref{item:as_sol_2}.

To show \ref{item:as_sol_1}, we decompose the nonlinearity $F(u)$ into a ``local piece" $F_{1}(u)$ and a ``nonlocal piece" by $F(u)=\mu|u|^{2}u-\E(|u|^{2})u=:F_{1}(u)+F_{2}(u)$. For $F_{1}(u)$,  we use the elementary inequality
\begin{equation}
	\left| F_{1}\paren*{\sum_{j=1}^{J}z_{j}} -\sum_{j=1}^{J}F_{1}(z_{j})\right| \lesssim_{J}\sum_{j\neq j'}|z_{j}||z_{j'}|^{2}, \qquad \forall z_{1},\ldots,z_{l}\in\mathbb{C}
\end{equation}
to obtain
\begin{equation}
\|F_{1}\paren*{\sum_{j=1}^{J}v_{n}^{(j)}}-\sum_{j=1}^{J}F_{1}(v_{n}^{(j)})\|_{L_{t,x}^{4/3}(\R\times\R^{2})} \lesssim_{J}\sum_{j\neq j'}\|v_{n}^{(j)}v_{n}^{(j')}\|_{L_{t,x}^{2}(\R\times\R^{2})}\|v_{n}^{(j')}\|_{L_{t,x}^{4}(\R\times\R^{2})}.
\end{equation}
Now, writing out the definitions of $v_{n}^{(j)}, v_{n}^{(j')}$, using the invariance of the $L_{t,x}^{4}$ norm under the $G'$ spacetime action together with with lemma \ref{lem:as_o}, we obtain that for each $J$ fixed,
\begin{equation}
\limsup_{n\rightarrow\infty} \sum_{j\neq j'}\|v_{n}^{(j)}v_{n}^{(j')}\|_{L_{t,x}^{2}(\R\times\R^{2})}\|v_{n}^{(j')}\|_{L_{t,x}^{4}(\R\times\R^{2})}=0,
\end{equation}
which completes the proof of \ref{item:as_sol_1} for the case of $F_{1}$.

For $F_{2}$, we need to proceed more delicately, as we must not destroy the cancellation in the kernel of the operator $\E$. First, observe that for each $J$ fixed,
\begin{align}
\left|F_{2}\paren*{\sum_{j=1}^{J}v_{n}^{(j)}} - \sum_{j=1}^{J}F_{2}(v_{n}^{(j)})\right| &\leq \left|\E\paren*{\sum_{j_{2}\neq j_{3}}^{J}v_{n}^{(j_{2})}\ol{v_{n}^{(j_{3})}}} \paren*{\sum_{j_{1}=1}^{J}v_{n}^{(j_{1})}} \right|+\left|\sum_{j_{2}=1}^{J}\E\paren*{|v_{n}^{(j_{2})}|^{2}}\paren*{\sum_{j_{1}\neq j_{2}}^{J} v_{n}^{(j_{1})}}\right| \nonumber\\
&=:\mathrm{Term}_{1}+\mathrm{Term}_{2}. \label{eq:aso_cases}
\end{align}

We first consider $\mathrm{Term}_{1}$. Since $T_{h_{n}^{(j)}}$ preserves the scattering size, it follows that $S(v_{n}^{(j)})\leq BM(\phi^{(j)})$. Therefore, we obtain from H\"{o}lder's inequality, Plancherel's theorem, the triangle inequality, and lemma \ref{lem:as_o} that
\begin{align}
\limsup_{n\rightarrow\infty}\|\mathrm{Term}_{1}\|_{L_{t,x}^{4/3}(\R\times\R^{2})} &\leq \limsup_{n\rightarrow\infty}\sum_{j_{2}\neq j_{3}}^{J} \|v_{n}^{(j_{2})}v_{n}^{(j_{3})}\|_{L_{t,x}^{2}(\R\times\R^{2})}\sum_{j_{1}=1}^{J} \|v^{(j_{1})}\|_{L_{t,x}^{4}(\R\times\R^{2})} \nonumber\\
&=0.
\end{align}

We next consider $\mathrm{Term}_{2}$. We first use the triangle inequality to obtain that
\begin{equation}
\|\mathrm{Term}_{2}\|_{L_{t,x}^{4/3}(\R\times\R^{2})}\leq \sum_{j_{1}\neq j_{2}}^{J} \|\E(|v_{n}^{(j_{2})}|^{2})v_{n}^{(j_{1})}\|_{L_{t,x}^{4/3}(\R\times\R^{2})}.
\end{equation}
So it suffices to show that for each $(j_{1},j_{2})\in\{1,\ldots,J\}^{2}$ with $j_{1}\neq j_{2}$,
\begin{equation}
\limsup_{n\rightarrow\infty}\|\E(|v_{n}^{(j_{2})}|^{2})v_{n}^{(j_{1})}\|_{L_{t,x}^{4/3}(\R\times\R^{2})}=0,
\end{equation}
which we do by contradiction. Passing to a subsequence if necessary, we may suppose that there exists a pair $(j_{1},j_{2})$ with $j_{1}\neq j_{2}$ and such that
\begin{equation}
\lim_{n\rightarrow\infty}\|\E(|v_{n}^{(j_{2})}|^{2})v_{n}^{(j_{1})}\|_{L_{t,x}^{4/3}(\R\times\R^{2})}=\delta>0.
\end{equation}
Given $\epsilon>0$, by density, we may find functions $v_{\epsilon}^{(j_{1})}, v_{\epsilon}^{(j_{2})}\in C_{c}^{\infty}(\R\times\R^{2})$ such that $\|v_{\epsilon}^{(j_{i})}-v^{(j_{i})}\|_{L_{t,x}^{4}(\R\times\R^{2})}<\epsilon$ for $i=1,2$. Hence by \eqref{eq:nlp_bd}, H\"{o}lder's and triangle inequalities, and Plancherel's theorem, we see that
\begin{align}
\limsup_{n\rightarrow\infty} \|\E(|v_{n}^{(j_{2})}|^{2})v_{n}^{(j_{1})}\|_{L_{t,x}^{4/3}(\R\times\R^{2})} &= \limsup_{n\rightarrow\infty}\|\E\paren*{|(v_{n}^{(j_{2})}-v_{\epsilon,n}^{(j_{2})}+v_{\epsilon,n}^{(j_{2})}|^{2}}\paren*{(v_{n}^{(j_{1})}-v_{\epsilon,n}^{(j_{1})})+v_{\epsilon,n}^{(j_{1})}}\|_{L_{t,x}^{4/3}(\R\times\R^{2})} \nonumber\\
&\leq \limsup_{n\rightarrow\infty}\brac*{\|\E(|v_{\varepsilon,n}^{(j_{2})}|^{2})v_{\varepsilon,n}^{(j_{1})}\|_{L_{t,x}^{4/3}(\R\times\R^{2})}} + C\epsilon\paren*{\epsilon+(BM_{c})^{1/4}}^{2}, \label{eq:as_sol_c2_red}
\end{align}
for some absolute constant $C>0$. Choosing $\epsilon>0$ sufficiently small so that $C\epsilon(\epsilon+(BM_{c})^{1/4})^{2}<\delta$, we see that it suffices to show that the first term in \eqref{eq:as_sol_c2_red} tends to zero as $n\rightarrow\infty$. So without loss of generality, we may assume that $v^{(j_{1})}, v^{(j_{2})}\in C_{c}^{\infty}(\R\times\R^{2})$.

To obtain a contradiction, we examine the asymptotic orthogonality condition for $j_{1},j_{2}$ and divide into cases.
\begin{itemize}[leftmargin=*]
\item %Case 1
If $\lim_{n\rightarrow\infty}\frac{\lambda_{n}^{(j_{1})}}{\lambda_{n}^{(j_{2})}}=0$, then we make the spacetime change of variables
\begin{equation}
s = \frac{t}{(\lambda_{n}^{(j_{1})})^{2}} +t_{n}^{(j_{1})}, \enspace y=\frac{x-x_{n}^{(j_{1})}-2\xi_{n}^{(j_{1})}t}{\lambda_{n}^{(j_{1})}}
\end{equation}
and use the translation and dilation symmetry of $\E$ to write
\begin{equation}
\|\E(|v_{n}^{(j_{1})}|^{2})v_{n}^{(j_{2})}\|_{L_{t,x}^{4/3}} = \paren*{\frac{\lambda_{n}^{(j_{1})}}{\lambda_{n}^{(j_{2})}}} \paren*{\int_{\R}\int_{\R^{2}}ds dy |\E(|v^{(j_{1})}(s)|^{2})(y)|^{4/3} |v^{(j_{2})}(\Phi_{n}(s,y))|^{4/3}}^{3/4},
\end{equation}
where
\begin{equation}
\Phi_{n}(s,y) \coloneqq \paren*{\frac{(\lambda_{n}^{(j_{1})})^{2}(s-t_{n}^{(j_{1})})}{(\lambda_{n}^{(j_{2})})^{2}} + t_{n}^{(j_{2})}, \frac{\lambda_{n}^{(j_{1})}y+x_{n}^{(j_{1})}-x_{n}^{(j_{2})} +2(\lambda_{n}^{(j_{1})})^{2}(t_{n}^{(j_{1})}-s)(\xi_{n}^{(j_{2})}-\xi_{n}^{(j_{1})})}{\lambda_{n}^{(j_{2})}}}.
\end{equation}
Then by H\"{o}lder's inequality and the Calder\'{o}n-Zygmund theorem,
\begin{equation}
\|\E(|v_{n}^{(j_{1})}|^{2})v_{n}^{(j_{2})}\|_{L_{t,x}^{4/3}} \leq \paren*{\frac{\lambda_{n}^{(j_{1})}}{\lambda_{n}^{(j_{2})}}} \||v^{(j_{1})}|^{2}\|_{L_{t,x}^{4/3}} \|v^{(j_{2})}\|_{L_{t,x}^{\infty}},
\end{equation}
which tends to zero as $n\rightarrow\infty$.

\item %Case 2
If $\lim_{n\rightarrow\infty}\frac{\lambda_{n}^{(j_{2})}}{\lambda_{n}^{(j_{1})}}=0$, then we make the spacetime change of variables,
\begin{equation}
s=\frac{t}{(\lambda_{n}^{(j_{2})})^{2}} +t_{n}^{(j_{2})}, \enspace y=\frac{x-x_{n}^{(j_{2})}-2\xi_{n}^{(j_{2})}t}{\lambda_{n}^{(j_{2})}}
\end{equation}
and use the translation and dilation symmetry of $\E$ to write
\begin{equation}
\|\E(|v_{n}^{(j_{1})}|^{2})v_{n}^{(j_{2})}\|_{L_{t,x}^{4/3}} = \paren*{\frac{\lambda_{n}^{(j_{2})}}{\lambda_{n}^{(j_{1})}}}^{2}  \paren*{\int_{\R}\int_{\R^{2}}dsdy |\E(|v^{(j_{1})}\circ\Phi_{n}(s,\cdot)|^{2})(y)|^{4/3} |v^{(j_{2})}(s,y)|^{4/3}}^{3/4},
\end{equation}
where
\begin{equation}
\Phi_{n}(s,y) \coloneqq \paren*{\frac{(\lambda_{n}^{(j_{2})})^{2}(s-t_{n}^{(j_{2})})}{(\lambda_{n}^{(j_{1})})^{2}} + t_{n}^{(j_{1})}, \frac{\lambda_{n}^{(j_{2})}y + x_{n}^{(j_{2})} - x_{n}^{(j_{1})} + 2(\lambda_{n}^{(j_{2})})^{2}(t_{n}^{(j_{2})}-s)(\xi_{n}^{(j_{1})}-\xi_{n}^{(j_{2})})}{\lambda_{n}^{(j_{1})}}}.
\end{equation}
Then by H\"{o}lder's inequality,
\begin{align}
\|\E(|v_{n}^{(j_{1})}|^{2})v_{n}^{(j_{2})}\|_{L_{t,x}^{4/3}} &\leq \paren*{\frac{\lambda_{n}^{(j_{2})}}{\lambda_{n}^{(j_{1})}}}^{2} \|\E(|v^{(j_{1})}\circ\Phi_{n}|^{2})\|_{L_{t,x}^{\infty}} \|v^{(j_{2})}\|_{L_{t,x}^{4/3}} \nonumber\\
&\lesssim \paren*{\frac{\lambda_{n}^{(j_{2})}}{\lambda_{n}^{(j_{1})}}}^{2-}\paren*{\|v^{(j_{1})}\|_{L_{t}^{\infty}L_{x}^{\infty-}}^{2}+\|v^{(j_{1})}\|_{L_{t}^{\infty}{C}_{x}^{0+}}^{2}} \|v^{(j_{2})}\|_{L_{t,x}^{4/3}},
\end{align}

\item %Case 3
Suppose that we are not in the previous two cases. Then passing to a subsequence if necessary, we may assume that $\lim_{n\rightarrow\infty}\frac{\lambda_{n}^{(j_{2})}}{\lambda_{n}^{(j_{1})}}$ exists and is positive. Let $R=R(j_{1},j_{2})>0$ be sufficiently large so that $\supp_{t,x}(v^{(j_{1})}), \supp_{t,x}(v^{(j_{2})}) \subset B_{t,x}(0,R)$. If $\E(|v_{n}^{(j_{1})}|^{2})(t,x) v_{n}^{(j_{2})}(t,x) \neq 0$, then
\begin{equation}
\max\brac*{\left|\frac{t}{(\lambda_{n}^{(j_{1})})^{2}}+t_{n}^{(j_{1})}\right|, \left|\frac{t}{(\lambda_{n}^{(j_{2})})^{2}}+t_{n}^{(j_{2})}\right|} \leq R.
\end{equation}
By the triangle inequality, the preceding condition implies that
\begin{equation}\label{eq:as_sol_C3}
\frac{|(\lambda_{n}^{(j_{1})})^{2}t_{n}^{(j_{1})}-(\lambda_{n}^{(j_{2})})^{2}t_{n}^{(j_{2})}|}{\lambda_{n}^{(j_{1})}\lambda_{n}^{(j_{2})}} \leq CR,
\end{equation}
where $C=C(j_{1},j_{2})>0$ is some constant depending on the value of $\lim_{n\rightarrow\infty}\frac{\lambda_{n}^{(j_{1})}}{\lambda_{n}^{(j_{2})}}$. If
\begin{equation}
\lim_{n\rightarrow\infty}\frac{|t_{n}^{(j_{1})}(\lambda_{n}^{(j_{1})})^{2}-t_{n}^{(j_{2})}(\lambda_{n}^{(j_{2})})^{2}|}{\lambda_{n}^{(j_{1})}\lambda_{n}^{(j_{2})}}=\infty,
\end{equation}
then the LHS of the inequality \eqref{eq:as_sol_C3} tends to $\infty$ as $n\rightarrow\infty$. Hence, there exists an integer $n_{0}=n_{0}(R)\in\N$ such that for all $n\geq n_{0}$, $\|\E(|v_{n}^{(j_{1})}|^{2})v_{n}^{(j_{2})}\|_{L_{t,x}^{4/3}}=0$.

\item %Case 4
Passing to a subsequence if necessary, suppose in addition that
\begin{equation}
\lim_{n\rightarrow\infty}\frac{|t_{n}^{(j_{1})}(\lambda_{n}^{(j_{1})})^{2}-t_{n}^{(j_{2})}(\lambda_{n}^{(j_{2})})^{2}|}{\lambda_{n}^{(j_{1})}\lambda_{n}^{(j_{2})}}
\end{equation}
exists and is finite. If
\begin{equation}
\lim_{n\rightarrow\infty}\lambda_{n}^{(j_{1})}\lambda_{n}^{(j_{2})}|\xi_{n}^{(j_{1})}-\xi_{n}^{(j_{2})}|^{2}=\infty,
\end{equation}
then by using that $\lim_{n\rightarrow\infty} \frac{\lambda_{n}^{(j_{1})}}{\lambda_{n}^{(j_{2})}}$ exists and is finite and positive, we have that
\begin{equation}
\lim_{n\rightarrow\infty}\frac{(\lambda_{n}^{(j_{1})})^{2}|\xi_{n}^{(j_{2})}-\xi_{n}^{(j_{1})}|}{\lambda_{n}^{(j_{2})}}=\infty.
\end{equation}
Since the symbol of $\E$ belongs to $L^{\infty}$ and $\|\mathcal{F}_{t,x}(|v^{(j_{1})}|^{2})\|_{L_{t,x}^{1}}<\infty$ by the qualitative assumption that $v^{(j_{1})}\in C_{c}^{\infty}(\R\times\R^{2})$, the Riemann-Lebesgue lemma yields that, given $\epsilon>0$, there exists $R'=R'(\epsilon)>0$ such that $\sup_{|x|\geq R'}|\E(|v^{(j_{1})}(t)|^{2})(x)|<\epsilon$ for all $t\in \supp_{t}(v^{(j_{1})})$. Therefore by the triangle inequality,
\begin{align}
\|\E(|v_{n}^{(j_{1})}|^{2})v_{n}^{(j_{2})}\|_{L_{t,x}^{4/3}} &\lesssim \epsilon\sup_{n} \paren*{\frac{\lambda_{n}^{(j_{1})}}{\lambda_{n}^{(j_{2})}}} \|v^{(j_{2})}\circ\Phi_{n}\|_{L_{t,x}^{4/3}} \\
&\phantom{=}+\paren*{\int_{|s|\leq R}\int_{|y|\leq R'}dsdy |\E(|v^{(j_{1})}(s)|^{2})(x)|^{4/3} |v^{(j_{2})}\circ\Phi_{n}(s,y)|^{4/3}}^{3/4} \nonumber\\
&\eqqcolon \mathrm{Term}_{1,n}+\mathrm{Term}_{2,n},
\end{align}
where $\Phi_{n}$ is defined as in the first case.

For the sequence $\mathrm{Term}_{1,n}$, we use the assumption that $\frac{\lambda_{n}^{(j_{1})}}{\lambda_{n}^{(j_{2})}} \sim 1$ to obtain the estimate
\begin{equation}
\limsup_{n\rightarrow\infty} \mathrm{Term}_{1,n} \lesssim_{j_{1},j_{2}} \epsilon \|v^{(j_{2})}\|_{L_{t,x}^{4/3}}.
\end{equation}

For the sequence $\mathrm{Term}_{2,n}$, the compact spacetime support of $v^{(j_{2})}$ implies that the integrand vanishes identically unless
\begin{equation}
\frac{|x_{n}^{(j_{1})}-x_{n}^{(j_{2})} + 2(\lambda_{n}^{(j_{1})})^{2}(t_{n}^{(j_{1})}-s)(\xi_{n}^{(j_{2})}-\xi_{n}^{(j_{1})})|}{\lambda_{n}^{(j_{2})}} \leq C(R+R'),
\end{equation}
where $C>0$ is some constant depending on the diameter of $\supp_{t,x}(v^{(j_{2})})$ and the value of $\lim_{n\rightarrow\infty} \frac{\lambda_{n}^{(j_{1})}}{\lambda_{n}^{(j_{2})}}$. Hence, we may write
\begin{equation}
\frac{2s(\xi_{n}^{(j_{1})}-\xi_{n}^{(j_{2})})(\lambda_{n}^{(j_{1})})^{2}}{\lambda_{n}^{(j_{2})}} = \underbrace{-\frac{x_{n}^{(j_{1})}-x_{n}^{(j_{2})}-2(\lambda_{n}^{(j_{1})})^{2}t_{n}^{(j_{1})}(\xi_{n}^{(j_{1})}-\xi_{n}^{(j_{2})})}{\lambda_{n}^{(j_{2})}}}_{\coloneqq y_{n}} + \sigma_{n}, \qquad \sigma_{n} \in \ol{B_{x}}(0, C(R+R')),
\end{equation}
which implies that $s$ in the support of the integrand $1_{|s|\leq R}1_{|y|\leq R'}\E(|v^{(j_{1})}|^{2}) (v^{(j_{2})}\circ\Phi_{n})$ must satisfy
\begin{equation}
\left| |s| - \frac{|y_{n}|}{\frac{2|\xi_{n}^{(j_{1})}-\xi_{n}^{(j_{2})}|(\lambda_{n}^{(j_{1})})^{2}}{\lambda_{n}^{(j_{2})}}}\right| \leq \frac{|\sigma_{n}|}{\frac{2|\xi_{n}^{(j_{1})}-\xi_{n}^{(j_{2})}|(\lambda_{n}^{(j_{1})})^{2}}{\lambda_{n}^{(j_{2})}}} \lesssim \frac{C(R+R')}{\frac{2|\xi_{n}^{(j_{1})}-\xi_{n}^{(j_{2})}|(\lambda_{n}^{(j_{1})})^{2}}{\lambda_{n}^{(j_{2})}}}.
\end{equation}
The RHS of the ultimate inequality tends to zero as $n\rightarrow\infty$, hence the 1D Lebesgue measure of the set $J_{n}$ of times $s\in\R$ such that $1_{|y|\leq R'}\E(|v^{(j_{1})}(s)|^{2}) (v^{(j_{2})}\circ\Phi_{n}(s,\cdot)) \neq 0$ tends to zero as $n\rightarrow\infty$. By the absolute continuity of the Lebesgue integral, we conclude that $\lim_{n\rightarrow\infty}\mathrm{Term}_{2,n}=0$.

Combining our analysis for both sequences, we have shown that
\begin{equation}
\limsup_{n\rightarrow\infty} \|\E(|v_{n}^{(j_{1})}|^{2})v_{n}^{(j_{2})}\|_{L_{t,x}^{4/3}} \lesssim_{j_{1},j_{2}} \epsilon \|v^{(j_{2})}\|_{L_{t,x}^{4/3}}.
\end{equation}
Since $\epsilon>0$ was arbitrary, we conclude that the LHS of the inequality is in fact zero.

\item %Case 5
Passing to a subsequence if necessary, suppose in addition that
\begin{equation}
\lim_{n\rightarrow\infty}\lambda_{n}^{(j_{1})}\lambda_{n}^{(j_{2})}|\xi_{n}^{(j_{1})}-\xi_{n}^{(j_{2})}|^{2}
\end{equation}
exists and is finite, possibly zero. Exhausting the possible cases in the definition of asymptotic orthogonality, we must have that
\begin{equation}
\lim_{n\rightarrow\infty} \frac{|x_{n}^{(j_{1})}-x_{n}^{(j_{2})}-2t_{n}^{(j_{1})}(\lambda_{n}^{(j_{1})})^{2}(\xi_{n}^{(j_{1})}-\xi_{n}^{(j_{2})})|^{2}}{\lambda_{n}^{(j_{1})}\lambda_{n}^{(j_{2})}} = \infty,
\end{equation}
which, using the assumption that $\frac{\lambda_{n}^{(j_{1})}}{\lambda_{n}^{(j_{2})}}\sim 1$, implies that
\begin{equation}
\lim_{n\rightarrow\infty} \frac{|x_{n}^{(j_{1})}-x_{n}^{(j_{2})}-2t_{n}^{(j_{1})}(\lambda_{n}^{(j_{1})})^{2}(\xi_{n}^{(j_{1})}-\xi_{n}^{(j_{2})})|}{\lambda_{n}^{(j_{2})}}=\infty.
\end{equation}
We make the same spacetime change of variable as in the previous case to obtain
\begin{equation}
\|\E(|v_{n}^{(j_{1})}|^{2})v_{n}^{(j_{2})}\|_{L_{t,x}^{4/3}} \leq \mathrm{Term}_{1,n}+\mathrm{Term}_{2,n},
\end{equation}
with $\mathrm{Term}_{1,n},\mathrm{Term}_{2,n}$ defined exactly as before. The analysis for the sequence $\mathrm{Term}_{1,n}$ is the same as in the previous case, but now for the sequence $\mathrm{Term}_{2,n}$, we use that $(s,y)$ must satisfy the condition
\begin{equation}
\left|\frac{\lambda_{n}^{(j_{1})}}{\lambda_{n}^{(j_{2})}}y-\frac{2s(\lambda_{n}^{(j_{1})})^{2}(\xi_{n}^{(j_{2})}-\xi_{n}^{(j_{1})})}{\lambda_{n}^{(j_{2})}}+\frac{x_{n}^{(j_{1})}-x_{n}^{(j_{2})} +2(\lambda_{n}^{(j_{1})})^{2}t_{n}^{(j_{1})}(\xi_{n}^{(j_{2})}-\xi_{n}^{(j_{1})})}{\lambda_{n}^{(j_{2})}}\right| \leq R
\end{equation}
in order for $\Phi_{n}(s,y)$ to belong to the spacetime support of $v^{(j_{2})}$. By the reverse triangle inequality,
\begin{equation}
\begin{split}
&\left|\frac{\lambda_{n}^{(j_{1})}}{\lambda_{n}^{(j_{2})}}y-\frac{2s(\lambda_{n}^{(j_{1})})^{2}(\xi_{n}^{(j_{2})}-\xi_{n}^{(j_{1})})}{\lambda_{n}^{(j_{2})}}+\frac{x_{n}^{(j_{1})}-x_{n}^{(j_{2})} +2(\lambda_{n}^{(j_{1})})^{2}t_{n}^{(j_{1})}(\xi_{n}^{(j_{2})}-\xi_{n}^{(j_{1})})}{\lambda_{n}^{(j_{2})}}\right| \\
&\phantom{=} \geq \left|\frac{x_{n}^{(j_{1})}-x_{n}^{(j_{2})} +2(\lambda_{n}^{(j_{1})})^{2}t_{n}^{(j_{1})}(\xi_{n}^{(j_{2})}-\xi_{n}^{(j_{1})})}{\lambda_{n}^{(j_{2})}}\right| -\frac{\lambda_{n}^{(j_{1})}}{\lambda_{n}^{(j_{2})}}R - \frac{2(\lambda_{n}^{(j_{1})})^{2}|\xi_{n}^{(j_{2})}-\xi_{n}^{(j_{1})}|}{\lambda_{n}^{(j_{2})}}R.
\end{split}
\end{equation}
Since the $\lim_{n\rightarrow\infty}$ of the lower bound is $\infty$, we conclude that there exists an $n_{0}=n_{0}(j_{1},j_{2})\in\N$ such that for all $n\geq n_{0}$, $\mathrm{Term}_{2,n}= 0$. Hence, we have shown that for every $\epsilon>0$,
\begin{equation}
\limsup_{n\rightarrow\infty}\|\E(|v_{n}^{(j_{1})}|^{2})v_{n}^{(j_{2})}\|_{L_{t,x}^{4/3}} \lesssim_{j_{1},j_{2}} \epsilon\|v^{(j_{2})}\|_{L_{t,x}^{4/3}},
\end{equation}
which implies that the LHS is in fact zero.
\end{itemize}

By exhaustion of all possible cases, we have obtained the necessary contradiction to complete the proof of \ref{item:as_sol_1} and hence the proof of lemma \ref{lem:as_sol}.
\end{proof}

%Remainder of proof; same as for NLS
The remainder of the proof of lemma \ref{lem:M_sup} proceeds as in the NLS case. We know that $\lim_{J\rightarrow J^{*}}\limsup_{n\rightarrow\infty} S(u_{n}^{(J)}) \leq BM_{c}$. Let $J_{0}\in\N$ be sufficiently large so that for every $J\geq J_{0}$, $\limsup_{n\rightarrow\infty} S(u_{n}^{(J)}) \leq 2BM_{c}$. By definition of $\limsup$ and lemmas \ref{lem:as_id} and \ref{lem:as_sol}, taking $J_{0}$ larger if necessary, for each $J\geq J_{0}$, there exists $n_{0}=n_{0}(J)\in\N$ such that for all $n\geq n_{0}$,
\begin{equation}
S(u_{n}^{(J)}) \leq 3BM_{c}; \quad \|(i\p_{t}+\Delta)u_{n}^{(J)}-F(u_{n}^{(J)})\|_{L_{t,x}^{4/3}(\R\times\R^{2})}\leq \delta'; \quad M\paren*{u_{n}^{(J)}(0)-u_{n}(0)} \leq \delta,
\end{equation}
where $\delta',\delta>0$ are parameters to values of which to be subsequently fixed. By hypothesis that $\lim_{n\rightarrow\infty} S(u_{n}) =\infty$, there exists an $n_{0}'\in\N$ such that for all $n\geq n_{0}'$,
\begin{equation}
S(u_{n})^{1/4} \geq (100BM_{c})^{1/4}+1.
\end{equation}
Since $u_{n}^{(J)}$ is global, choosing $\delta',\delta>0$ sufficiently small and by taking the compact interval $I$ arbitrarily large in the statement of lemma \ref{lem:stab}, we see that for fixed $J\geq J_{0}$, there is an integer $n\geq n_{0}$ and a global solution $\tilde{u}_{n}$ to \eqref{eq:DS} with initial data $u_{n}(0)$ such that
\begin{equation}
S(u_{n}^{(J)}-\tilde{u}_{n}) \leq 1.
\end{equation}
By uniqueness of solutions to \eqref{eq:DS}, $\tilde{u}_{n}=u_{n}$. By the triangle inequality,
\begin{equation}
S(u_{n})^{1/4} \leq S(u_{n}^{(J)}-u_{n})^{1/4} + S(u_{n}^{(J)})^{1/4} \leq 1 + (3BM_{c})^{1/4} < S(u_{n})^{1/4}
\end{equation}
which is a contradiction. This last step completes the proof of lemma \ref{lem:M_sup}.
\end{proof}

We next show that there is only one linear profile $\phi^{(j)}$ (i.e. $J^{*}=1$).

\begin{lemma}\label{lem:1_bub}
$J^{*}=1$.
\end{lemma}
\begin{proof}
Given any $\varepsilon \in (0,M_{c}/2)$, by definition of supremum, there exists a $j_{1}\in\{1,\ldots,J^{*}\}$ such that $M(\phi^{(j_{1})})\geq M_{c}-\varepsilon$.

If $M(\phi^{(j_{1})})=M_{c}$, then since
\begin{equation}
M_{c} = M(\phi^{(j_{1})}) \leq \sum_{j=1}^{J^{*}}M(\phi^{(j)}) \leq M_{c},
\end{equation}
we see that $\sum_{j\in\{1,\ldots,J^{*}\}\setminus\{j_{1}\}} M(\phi^{(j)}) = 0$.

If $M(\phi^{(j_{1})}) < M_{c}-\varepsilon'$ for some $\varepsilon'>0$, then by definition of supremum, there exists $j_{2}\in \{1,\ldots,J^{*}\}\setminus\{j_{1}\}$ such that $M(\phi^{(j_{1})}) \geq M_{c}-\varepsilon'$. But then
\begin{equation}
M_{c} < 2M_{c}-\varepsilon-\varepsilon' \leq M(\phi^{(j_{1})}) + M(\phi^{(j_{2})}) \leq \sum_{j=1}^{J^{*}} M(\phi^{(j)}) \leq M_{c},
\end{equation}
which is a contradiction.
\end{proof}

With lemma \ref{lem:1_bub}, we may omit the superscripts $(j), (J)$ so that the linear profile decomposition simplifies to
\begin{equation}
u_{n}(0)=h_{n}e^{it_{n}\Delta}\phi+w_{n},
\end{equation}
for some sequences $t_{n}\in\R$ with $t_{n}\equiv 0$ or $\lim_{n\rightarrow\infty}t_{n}\in\{\pm\infty\}$, $h_{n}\in G$, $\phi\in L^{2}(\R^{2})$ with $M(\phi)=M_{c}$, and $w_{n}\in L^{2}(\R^{2})$ with $\lim_{n\rightarrow\infty}M(w_{n})=0$. Note that by Strichartz estimates, the last property implies $\lim_{n\rightarrow\infty}S(e^{it\Delta}w_{n})=0$.

To conclude the proof of proposition \ref{prop:PS}, we consider three cases based on the limiting value of the sequence $t_{n}$.
\begin{itemize}[leftmargin=*]
\item %Case 1
If $t_{n}\equiv 0$, then the desired conclusion is immediate from the $G$-invariance of the mass functional and $\lim_{n\rightarrow\infty} M(w_{n})=0$.
\item %Case 2
Suppose that $\lim_{n\rightarrow\infty}t_{n}=\infty$. By Strichartz estimates, $S(e^{it\Delta}\phi)\lesssim M_{c}^{2}$, from which dominated convergence and time translation invariance imply that
\begin{equation}
\lim_{n\rightarrow\infty}S_{\geq t_{n}}\left(e^{it\Delta}\phi\right)=\lim_{n\rightarrow\infty}S_{\geq 0}\left(e^{it\Delta}e^{it_{n}\Delta}\phi\right)=0.
\end{equation}
Since the free Schr\"{o}dinger equation is $G$-invariant, we have the operator identity $T_{h_{n}}e^{it\Delta}=e^{it\Delta}h_{n}$. Since $T_{h_{n}}$ preserves the partial scattering size $S_{\geq 0}$, we obtain that
\begin{equation}
\lim_{n\rightarrow\infty}S_{\geq 0}\left(e^{it\Delta}h_{n}e^{it_{n}\Delta}\phi\right)=0,
\end{equation}
which together with the linear profile decomposition for $u_{n}(0)$ and the triangle inequality implies that
\begin{equation}
\lim_{n\rightarrow\infty}S_{\geq 0}\left(e^{it\Delta}u_{n}(0)\right)=0.
\end{equation}
Applying lemma \ref{lem:stab} with $\tilde{u}\equiv 0$, we conclude that
\begin{equation}
\lim_{n\rightarrow\infty}S_{\geq 0}(u_{n})=0,
\end{equation}
which contradicts our blowup assumption for for the sequence $u_{n}$. Therefore, we have precluded this case.
\item %Case 3
Suppose that $\lim_{n\rightarrow\infty}t_{n}=-\infty$. The argument to preclude this case is completely analogously to that of the previous case, except now replacing $S_{\geq 0}(\cdot)$ by $S_{\leq 0}(\cdot)$.
\end{itemize}

\subsection{Further refinements}\label{ssec:MMB_ref}
As shown in the work \cite{Killip2009} for the mass-critical NLS, by further appealing to the concentration compactness theory, one can exhibit almost periodic modulo symmetries blowup solutions with additional properties. The arguments to prove these additional properties are quite general and largely do not depend on the exact form of the nonlinearity but only on the concentration compactness principle satisfied by the equation. As the eeDS possesses an analogous concentration compactness theory, these arguments apply mostly with little modification. Therefore we will only sketch below some of the proofs, highlighting any modifications necessary for the eeDS case. We closely follow \cite{KillipClay} in our exposition here. In the sequel, solution will always refer to a solution to \eqref{eq:DS} unless specified otherwise.

\begin{remark}\label{rem:const_dep}
If $u:I\times\R^{2}\rightarrow\mathbb{C}$ is a maximal-lifespan solution which is almost periodic modulo symmetries with parameters $x(t),\xi(t), N(t)$ and compactness modulus function $C$, then $u_{\lambda}:\lambda^{-2}I\times\R^{2}\rightarrow\mathbb{C}$ defined by
\begin{equation}
u_{\lambda}(t,x) \coloneqq \lambda u(\lambda^{2}t,\lambda x), \qquad (t,x) \in \lambda^{-2}I\times\R^{2}
\end{equation}
is a maximal-lifespan solution which is almost periodic modulo symmetries with parameters
\begin{align}
x_{\lambda}(t) &\coloneqq \lambda^{-1}x(\lambda^{2}t),\\
\xi_{\lambda}(t) &\coloneqq \lambda \xi(\lambda^{2}t),\\
N_{\lambda}(t) &\coloneqq \lambda N(\lambda^{2}t)
\end{align}
and compactness modulus function $C$. Since the mass and scattering size are invariant under the DS scaling, the implicit constants in this subsection which depend on a solution $u$ are uniform in all such rescalings of $u$.
\end{remark}

\begin{lemma}[Local constancy of parameters]\label{lem:lc}
Let $u:I\times\R^{2}\rightarrow\C$ be a nonzero maximal lifespan solution which is almost periodic modulo symmetries with parameters $x(t),\xi(t),N(t)$. Then there exists a $0<\delta=\delta(u)\leq 1$, such that for every $t_{0}\in I$,
\begin{equation}
\left[t_{0}-\frac{\delta}{N(t_{0})^{2}}, t_{0}+\frac{\delta}{N(t_{0})^{2}}\right] \subset I
\end{equation}
and
\begin{gather}
N(t)\sim_{u}N(t_{0})\\
|\xi(t)-\xi(t_{0})| \lesssim_{u} N(t_{0})\\
|x(t)-x(t_{0}) - 2(t-t_{0})\xi(t_{0})| \lesssim_{u} N(t_{0})^{-1}
\end{gather}
for all $|t-t_{0}| \leq \delta N(t_{0})^{-2}$.
\end{lemma}
\begin{proof}
See lemma 5.18 in \cite{KillipClay}.
\end{proof}.

\begin{cor}[$N(t)$ at blowup]
Let $u:I\times\R^{2}\rightarrow\mathbb{C}$ be a nonzero maximal lifespan solution which is almost periodic modulo symmetries with frequency scale function $N(t)$. If one of the endpoints $T$ of $I$ is finite, then
\begin{equation}
N(t) \gtrsim_{u} |T-t|^{-1/2}.
\end{equation}
If $I$ is (semi-)infinite, then for any $t_{0}\in I$,
\begin{equation}
N(t) \gtrsim_{u} \min\{N(t_{0}), |t-t_{0}|^{-1/2}\}.
\end{equation}
\end{cor}
\begin{proof}
See corollary 5.19 in \cite{KillipClay}.
\end{proof}

\begin{lemma}[Local quasi-boundness]\label{lem:lqb}
Let $u:I\times\R^{2}\rightarrow\mathbb{C}$ be a nonzero solution which is almost periodic modulo symmetries with frequency scale function $N(t)$. If $K\subset I$ is compact, then
\begin{equation}
0 < \inf_{t\in K} N(t) \leq \sup_{t\in K} N(t) <\infty.
\end{equation}
\end{lemma}
\begin{proof}
See lemma 5.20 in \cite{KillipClay}.
\end{proof}

\begin{lemma}[Strichartz norms via N(t)]\label{lem:SN}
Let $u:I\times\R^{2}\rightarrow\mathbb{C}$ be a nonzero maximal-lifespan solution that is almost periodic modulo symmetries with parameters $x(t),\xi(t),N(t)$. Then for any compact subinterval $J\subset I$,
\begin{equation}
\int_{J}N(t)^{2}dt \lesssim_{u} \|u\|_{L_{t,x}^{4}(J\times\R^{2})}^{4} \lesssim_{u} 1+\int_{J}N(t)^{2}dt.
\end{equation}
\end{lemma}
\begin{proof}
The lemma follows from the argument in the proof of lemma 5.21 in \cite{KillipClay} with two additional steps of H\"{o}lder's inequality and Plancherel's theorem to estimate the nonlinear term in Duhamel's formula.
\end{proof}

\begin{lemma}\label{lem:lc_int}
For any nonzero maximal-lifespan solution $u:I\times\R^{2}\rightarrow\C$ with frequency scale function $N(t)$,
\begin{equation}
\|u\|_{L_{t,x}^{4}([t_{0},t_{0}+\delta N(t_{0})^{-2}]\times\R^{2})} \sim_{u} \|u\|_{L_{t,x}^{4}([t_{0}-\delta N(t_{0})^{-2},t_{0}]\times\R^{2})}\sim_{u} 1, \qquad \forall t_{0}\in I,
\end{equation}
where $\delta$ is as is in the statement of lemma \ref{lem:lc}.
\end{lemma}
\begin{proof}
For $t_{0}\in I$, define $J \coloneqq [t_{0}-\frac{\delta}{N(t_{0})^{2}},t_{0}+\frac{\delta}{N(t_{0})^{2}}]$. By lemma \ref{lem:SN}, we have that
\begin{equation}
\int_{J}N(t)^{2}dt\lesssim_{u} \|u\|_{L_{t,x}^{4}(J\times\R^{2})}^{4} \lesssim_{u} 1+\int_{J}N(t)^{2}dt.
\end{equation}
Since $N(t)\sim_{u} N(J)$ for all $t\in J$, we have that
\begin{equation}
\int_{J}N(t)^{2}dt \sim_{u} N(J)^{2} |J| = \frac{2\delta N(J)^{2}}{N(t_{0})^{2}} \sim_{u} 1,
\end{equation}
from which the statement of the lemma follows readily.
\end{proof}

\begin{cor}\label{cor:lc_int}
Let $u:I\times\R^{2}\rightarrow\mathbb{C}$ be a nonzero maximal-lifespan solution which is almost periodic modulo symmetries with parameters $x(t),\xi(t),N(t)$. For every compact subinterval $J\subset I$ such that $\|u\|_{L_{t,x}^{4}(J\times\R^{2})} \leq 1$, we have that
\begin{gather}
N(t_{1}) \sim_{u} N(t_{2})\\
|\xi(t_{1})-\xi(t_{2})| \lesssim_{u} N(J_{k})\\
|x(t_{1})-x(t_{2})-2(t_{1}-t_{2})\xi(t_{2})| \lesssim_{u} N(J_{k})^{-1}
\end{gather}
for all $t_{1},t_{2}\in J$. Additionally,
\begin{equation}
|J| \sim_{u} \frac{1}{N(J)^{2}}
\end{equation}
and
\begin{equation}
\int_{J}N(t)^{3}dt \sim N(J).
\end{equation}
\end{cor}
\begin{proof}
It suffices to consider the case where $J=[a,b]\subset I$ satisfies $\|u\|_{L_{t,x}^{4}(J\times\R^{2})}=1$. Let $t_{0}=a$, and define
\begin{equation}
t_{1} \coloneqq \min\brac*{b,t_{0}+\frac{\delta}{N(t_{0})^{2}}}.
\end{equation}
If $t_{0},\ldots,t_{j}$ have been defined, then we define $t_{j+1}$ by the formula
\begin{equation}
t_{j+1} \coloneqq \min\brac*{b,t_{j}+\frac{\delta}{N(t_{j})^{2}}}.
\end{equation}
We claim that there exists some $n\in\mathbb{N}$ such that $t_{n}=b$. Indeed, this is immediate from the fact that $\inf_{t\in J} N(t) > 0$. We next claim that it suffices to show that
\begin{equation}
N(t_{j}) \sim_{u} N(t_{k}), \qquad \forall j,k\in\{0,\ldots,n\}.
\end{equation}
Indeed, for any point $t\in J$, there exists a subinterval $[t_{j},t_{j+1}]$, for some $j\in\{0,\ldots,n-1\}$, which contains $t$ and by construction, $N(t') \sim_{u} N(t_{j})$ for all $t'\in [t_{j},t_{j+1}]$. Let $c_{1}=c_{1}(u)$ and $c_{2}=c_{2}(u)$ denote the implicit constants in lemma \ref{lem:lc_int}. Taking $c_{1}$ smaller if necessary, we may assume that $c_{1}\leq \frac{1}{2}$. Then
\begin{equation}
nc_{1}^{4}\leq \sum_{j=0}^{n-1} \|u\|_{L_{t,x}^{4}([t_{j},t_{j+1}]\times\R^{2})}^{4} = \|u\|_{L_{t,x}^{4}(J\times\R^{2})}^{4} = 1,
\end{equation}
which implies that $n \leq c_{1}^{-4}$.  Now, $N(t_{j})\sim_{u} N(t_{k})$ follows from $N(t_{j})\sim_{u} N(t_{j+1})$ and induction on $j$. The remaining assertions for $\xi(t)$ and $x(t)$ follow by similar arguments; we omit the details.

We prove the remaining two assertions of the lemma. Observe that
\begin{equation}
\frac{c_{2}^{-4}\delta}{N(J)^{2}}\leq \frac{n\delta}{N(J)^{2}} \lesssim_{u} |J| \lesssim_{u} \frac{n\delta}{N(J)^{2}} \leq \frac{c_{1}^{-4}\delta}{N(J)^{2}},
\end{equation}
which shows that $|J|\sim_{u} N(J)^{-2}$. Lastly, using the first and penultimate assertions of the lemma, we have that
\begin{equation}
\int_{J}N(t)^{3}dt \sim_{u} N(J)^{3} |J| \sim_{u} N(J).
\end{equation}
\end{proof}

\begin{remark}
By partitioning any compact interval $K\subset I$ into consecutive intervals $J_{k}$ such that $\|u\|_{L_{t,x}^{4}(J_{k}\times\R^{2})}=1$, we see from corollary \ref{cor:lc_int} that
\begin{equation}
\int_{K}N(t)^{3}dt = \sum_{k} \int_{J_{k}}N(t)^{3} \sim_{u} \sum_{k}N(J_{k}).
\end{equation}
\end{remark}

\begin{lemma}\label{lem:APMS_class}
There exists a nonzero maximal-lifespan solution $u:I\times\R^{2}\rightarrow\mathbb{C}$ which is almost periodic modulo symmetries, blows up both forward and backward in time, and in the focusing case satisfies $M(u)<M(Q)$. Furthermore, the lifespan $I$ and the frequency scale function $N(t)$ fall into one of the following three scenarios:
\begin{enumerate}[I.]
\item\label{item:APMS_class_1} %Soliton-like
$I=\R$ and
\begin{equation}
N(t)=1, \qquad  \forall t\in\R;
\end{equation}
\item\label{item:APMS_class_2} %Double high low frequency cascade
$I=\R$ and
\begin{equation}
\liminf_{t\rightarrow-\infty} N(t)=\limsup_{t\rightarrow \infty} N(t)=0, \qquad \sup_{t\in\R}N(t)<\infty;
\end{equation}
\item\label{item:APMS_class_3} %Self-similar
$I=(0,\infty)$ and
\begin{equation}
N(t)=t^{-1/2}, \qquad \forall t\in I.
\end{equation}
\end{enumerate}
\end{lemma}
\begin{proof}
See theorem 5.24 in \cite{KillipClay}.
\end{proof}

We next refine lemma \ref{lem:APMS_class} to prove corollary \ref{cor:adm_sol}.
\begin{proof}
Let $u:I\times\R^{2}\rightarrow\mathbb{C}$ be a maximal-lifespan solution, which is almost periodic modulo symmetries with parameters $x_{1}(t),\xi_{1}(t),N_{1}(t)$ and compactness modulus function $C_{1}$ and which blows up both forward and backward in time, whose existence is guaranteed by lemma \ref{lem:APMS_class}. By time translating $u$ and relabeling in the event of case \ref{item:APMS_class_3}, we may assume without loss of generality that $[0,\infty)\subset I$. Moreover, by modifying $u$ further and relabeling, we may assume that $N_{1}(t)\leq 1$ and $x(0)=\xi(0)=0$.

We first show that there exists a frequency scale function $N_{2}(t)$ such that
\begin{equation}
N_{2}\in C_{loc}^{1}([0,\infty)) \enspace \text{and} \enspace |N_{2}'(t)| \lesssim N_{2}(t)^{3}, \qquad  \forall t\in[0,\infty).
\end{equation}
Since these two conditions are already both satisfied in the event of either case \ref{item:APMS_class_1} or case \ref{item:APMS_class_3}, it suffices to consider case \ref{item:APMS_class_2}. Partition $[0,\infty)$ into consecutive subintervals $J_{k} \coloneqq [t_{k-1},t_{k})$ such that $\|u\|_{L_{t,x}^{4}(J_{k}\times\R^{2})}^{4}=1$, for $k\in\mathbb{N}$. For each $k\in\mathbb{N}$, define the points
\begin{align}
t_{k,l} &\coloneqq t_{k-1} + \frac{t_{k}-t_{k-1}}{4},\\
t_{k,m} &\coloneqq \frac{t_{k-1}+t_{k}}{2},\\
t_{k,h} &\coloneqq t_{k-1} + \frac{3(t_{k}-t_{k-1})}{4}.
\end{align}
Now define the frequency scale function $N_{2}:I\rightarrow[0,\infty)$ as follows. First, we define
\begin{equation}
N_{2}(t) \coloneqq \frac{N_{1}(t)}{\sup_{t'\in I} N_{1}(t')}, \qquad t\in I\cap (-\infty,0].
\end{equation}
Next, for $k\in\mathbb{N}$, we define
\begin{equation}
N_{2}(t) \coloneqq \frac{N_{1}(J_{k})}{\sup_{t'\in I} N_{1}(t')}, \qquad t_{k,l}\leq t\leq t_{k,h}.
\end{equation}
To extend the definition of $N_{2}$ to the entire lifespan of $u$, we smoothly interpolate between $t_{0}$ and $t_{1,l}$ and between $t_{k,h}$ and $t_{k+1,l}$ for $k\in\mathbb{N}$. We omit the details for the former case and only consider the latter. We define
\begin{equation}
N_{2}(t) \coloneqq \frac{1}{\sup_{t'\in I} N_{1}(t')}\paren*{\dfrac{N_{1}(J_{k+1})e^{-\frac{1}{t-t_{k,h}}}}{e^{-\frac{1}{t_{k+1,l}-t}} + e^{-\frac{1}{t-t_{k,h}}}}+\dfrac{N_{1}(J_{k})e^{-\frac{1}{t_{k+1,l}-t}}}{e^{-\frac{1}{t_{k+1,l}-t}} + e^{-\frac{1}{t-t_{k,h}}}} }, \qquad t_{k,h}<t<t_{k+1,l}.
\end{equation}
It is evident that $N_{2}\in C_{loc}^{\infty}([0,\infty))$ and $N_{2}(t)\leq 1$ for all $t\in [0,\infty)$. Since $|t_{k,h}-t_{k+1,l}| \sim_{u} N_{1}(J_{k})^{-2}\sim_{u} N_{1}(J_{k+1})^{-2}$, it follows that
\begin{equation}
|N_{2}'(t)| \lesssim_{u} N_{2}(J_{k})^{3}\sim_{u}  N_{2}(t)^{3}, \qquad \forall t \in J_{k}
\end{equation}
for all $k\in\N$.

To construct $\xi_{2}$, we argue similarly. First, define
\begin{equation}
\xi_{2}(t) \coloneqq \xi_{1}(t), \qquad t\in I\cap (-\infty,0].
\end{equation}
Next, for $k\in\mathbb{N}$, define
\begin{equation}
\xi_{2}(t) \coloneqq \xi_{1}(J_{k}), \qquad t_{k,l}\leq t\leq t_{k,h}.
\end{equation}
To extend the definition of $\xi_{2}$ to the entire lifespan of $u$, we now define
\begin{equation}
\xi_{2}(t) \coloneqq \paren*{\dfrac{e^{-\frac{1}{t-t_{k,h}}}}{e^{-\frac{1}{t_{k+1,l}-t}} + e^{-\frac{1}{t-t_{k,h}}}}\xi_{1}(J_{k+1}) + \dfrac{e^{-\frac{1}{t_{k+1,l}-t}}}{e^{-\frac{1}{t_{k+1,l}-t}} + e^{-\frac{1}{t-t_{k,h}}}}\xi_{1}(J_{k})}, \qquad t_{k,h}<t<t_{k+1,l}.
\end{equation}
It is evident that $\xi_{2}\in C_{loc}^{\infty}([0,\infty))$. Since $|t_{k,h}-t_{k+1,l}| \sim_{u} N_{1}(J_{k+1})^{-2}\sim_{u} N_{1}(t)^{-2}\sim_{u} N_{1}(J_{k})^{-2}$ for $t\in [t_{k,h},t_{k+1,l}]$ and $|\xi_{1}(t)-\xi_{1}(t_{k,h})| \lesssim_{u} N_{1}(J_{k}) \sim_{u} N_{1}(J_{k+1})$ for $t\in [t_{k,h},t_{k+1,l}]$, it follows that
\begin{equation}
|\xi_{2}'(t)| \lesssim_{u} N_{2}(J_{k})^{3} \sim_{u} N_{2}(t)^{3}, \qquad \forall t\in J_{k}
\end{equation}
for all $k\in\N$.

We define $x_{2}(t) \coloneqq x_{1}(t)$ for all $t\in I$.

Lastly, we check that there is a compactness modulus function $C_{2}$ such that $u$ is almost periodic modulo symmetries with parameters $x_{2}(t),\xi_{2}(t), N_{2}(t)$. It is evident from the definition of $N_{2}$ that there exist constants $c_{1}(u), c_{2}(u)>0$ such that $N_{1}(t)\leq c_{1}(u)N_{2}(t)$ and $N_{2}(t)\leq c_{2}(u)N_{1}(t)$ for all $t\in [0,\infty)$. Additionally, taking $c_{1}(u)$ larger if necessary, we also have $|\xi_{1}(t)-\xi_{2}(t)| \leq c_{1}(u)N_{2}(t)$ for all $t\in [0,\infty)$. Hence, for any $\eta>0$, the condition
\begin{equation}
|\xi-\xi_{2}(t)| \geq c_{1}(u)(1+C_{1}(\eta))N_{2}(t) 
\end{equation}
implies by the reverse triangle inequality that
\begin{equation}
|\xi-\xi_{1}(t)| \geq |\xi-\xi_{2}(t)| -|\xi_{1}(t)-\xi_{2}(t)| \geq |\xi-\xi_{2}(t)| - c_{1}(u)N_{2}(t) \geq C_{1}(\eta)N_{1}(t).
\end{equation}
Since $x_{2}=x_{1}$ by definition, for any $\eta>0$,
\begin{equation}
|x-x_{2}(t)| \geq \frac{c_{2}(u)C_{1}(\eta)}{N_{2}(t)} \Longrightarrow |x-x_{1}(t)| \geq \frac{C_{1}(\eta)}{N_{1}(t)}.
\end{equation}
Thus, we define $C_{2}(\eta) \coloneqq c_{1}(u)+\paren*{c_{1}(u)+c_{2}(u)}C_{1}(\eta)$.
\end{proof}

\section{Long-time Strichartz norms}\label{sec:LTSE_norm}

In this section, we introduce the machinery necessary to state our long-time Strichartz estimate for admissible blowup solutions to the eeDS equation. Specifically, we use the $X, Y$ norms and the associated maximal $\tilde{X}_{k}, \tilde{Y}_{k}$ norms introduced by Dodson in his work \cite{Dodson2016} on the 2D cubic NLS to construct a (stronger) substitute for the long-time Strichartz estimate in the work \cite{Dodson2012} on the mass-critical NLS in dimensions $d\geq 3$. We remind the reader that in \cite{Dodson2012}, the long-time Strichartz estimate was in terms of the endpoint Strichartz estimate norm $L_{t}^{2}L_{x}^{\frac{2d}{d-2}}$, which is well-known to fail in dimension $d=2$ (see \cite{montgomery-smith1998}).

We have tried to stay consistent with the notation and terminology in \cite{Dodson2016}; however, one change we have introduced is the definition of admissible tuple below. This definition is just a repackaging of the various quantities and parameters appearing in the construction of the long-time Strichartz estimate in \cite{Dodson2016}, and we have adopted it so as to make very clear the independence of implicit constants on certain parameters. Such independence is crucial to the overall argument.

\subsection{$\tilde{X}_{k_{0}}$ and $\tilde{Y}_{k_{0}}$ norms} %Dodson norms

\begin{mydef}[Admissible tuple]\label{def:adm_tup} %Definition of admissible tuple
Let $u:I\times\R^{2}\rightarrow\mathbb{C}$ be an admissible blowup solution to \eqref{eq:DS} with parameters $x(t),\xi(t),N(t)$ and compactness modulus function $C(\cdot)$. Let $[a,b]\subset [0,\infty)$. Let $0<\epsilon_{3}<\epsilon_{2}<\epsilon_{1}\leq 2^{-100}$ be three small parameters. We say that the tuple $(\epsilon_{1},\epsilon_{2},\epsilon_{3})$ is \emph{admissible} for $u$ if the following conditions are satisfied:
\begin{enumerate}[(i)]
\item\label{item:adm_tup_1}
\begin{equation}
|\xi'(t)| + |N'(t)|\leq 2^{-20}\dfrac{N(t)^{3}}{\epsilon_{1}^{1/2}}, \qquad \forall t\geq 0;
\end{equation}
\item\label{item:adm_tup_2}
\begin{equation}\label{eq:sp_frq_loc}
	\int_{|x-x(t)|\geq \frac{2^{-20}\epsilon_{3}^{-1/4}}{N(t)}}|u(t,x)|^{2}dx + \int_{|\xi-\xi(t)|\geq 2^{-20}\epsilon_{3}^{-1/4}N(t)}|\hat{u}(t,\xi)|^{2}d\xi \leq \epsilon_{2}^{2}, \qquad \forall t\geq 0;
\end{equation}
\item\label{item:adm_tup_3}
$\epsilon_{3}<\epsilon_{2}^{10}<\epsilon_{1}^{10}$.
\end{enumerate}
\end{mydef}

While the quantity $\int_{a}^{b}\int_{\R^{2}}|u(t,x)|^{4}dxdt$ is DS scale-invariant, the quantity $\int_{a}^{b}N(t)^{3}dt$ is not. Hence, given an admissible blowup solution $u:I\times\R^{2}\rightarrow\C$, an interval $[a,b]\subset [0,\infty)$, a nonnegative integer $k_{0}$, and an admissible tuple $(\epsilon_{1},\epsilon_{2},\epsilon_{3})$ for $u$, we can always rescale $u$ by $u_{\lambda}\coloneqq \lambda u(\lambda^{2}\cdot,\lambda\cdot)$, where $\lambda\coloneqq \frac{\epsilon_{3}2^{k_{0}}}{\int_{a}^{b}N(t)^{3}dt}$, to obtain another admissible blowup solution so that properties \ref{item:adm_tup_1}, \ref{item:adm_tup_2}, \ref{item:adm_tup_3} are satisfied with $\xi_{\lambda}, x_{\lambda}, N_{\lambda}$ in addition to
\begin{align}
\int_{\lambda^{-2}a}^{\lambda^{-2}b}\int_{\R^{2}} |u_{\lambda}(t,x)|^{4}dxdt &= 2^{k_{0}}\label{eq:adm_rs_1}\\
\int_{\lambda^{-2}a}^{\lambda^{-2}b}N_{\lambda}(t)^{3}dt &= \epsilon_{3}2^{k_{0}}. \label{eq:adm_rs_2}
\end{align}
Below, we drop the subscript $\lambda$ in $u_{\lambda}$ for notational convenience and assume that $u$ satisfies \eqref{eq:adm_rs_1} and \eqref{eq:adm_rs_2}.

We partition the time interval $[a,b]$ in two different ways: $J_{l}$ intervals, which are called small intervals, and $J^{\alpha}$ intervals. The small intervals arise naturally from the concentration compactness theory, as seen in subsection \ref{ssec:MMB_ref}. The $J^{\alpha}$ intervals give us finer control on the variation of the frequency center $\xi(t)$ and frequency scale $N(t)$.

\begin{mydef}[Small intervals]\label{def:small_int} %Definition of small intervals
We say that $J_{l}\subset [a,b]$ with $\|u\|_{L_{t,x}^{4}(J_{l}\times\R^{2})}=1$ is a \emph{small interval}.
\end{mydef}

\begin{mydef}[$J^{\alpha}$ intervals]\label{def:Ja_int} %Definition of $J^{\alpha}$ intervals
We define the consecutive subintervals $J^{\alpha}\subset [a,b]$ by
\begin{equation}
\int_{J^{\alpha}}\paren*{N(t)^{3}+\epsilon_{3}\|u(t)\|_{L_{x}^{4}(\R^{2})}^{4}}dt=2\epsilon_{3}.
\end{equation}
\end{mydef}

\begin{remark} %Small and J^{\alpha} embeddings
Observe that Duhamel's formula, Strichartz estimates, duality, and mass conservation imply that for any admissible pair $(p,q)$,
\begin{align}
&\|u\|_{L_{t}^{p}L_{x}^{q}(J_{l}\times\R^{2})} \lesssim_{p,q} \|u\|_{U_{\Delta}^{2}(J_{l}\times\R^{2})}  \lesssim_{u} 1\\
&\|u\|_{L_{t}^{p}L_{x}^{q}(J^{\alpha}\times\R^{2})} \lesssim_{p,q} \|u\|_{U_{\Delta}^{2}(J^{\alpha}\times\R^{2})} \lesssim_{u} 1,
\end{align}
where the implicit constant in the ultimate inequalities only depends on the mass $M(u)$.
\end{remark}

We now construct dyadic groupings of $J^{\alpha}$ intervals. These groupings are the $G_{\kappa}^{j}$ intervals, on which we have simultaneous control on the variation of the frequency center and scale functions $\xi(t)$ and $N(t)$ in addition to the scattering size.
\begin{mydef}
For integers $0\leq j \leq k_{0}$ and $0 \leq \kappa \leq 2^{k_{0}-j}$, set
\begin{equation}
G_{\kappa}^{j} \coloneqq \bigcup_{\alpha=\kappa 2^{j}}^{(\kappa+1)2^{j}-1} J^{\alpha}
\end{equation}
If $[t_{0},t_{1}]=G_{\kappa}^{j}$, $[t_{0}',t_{1}']=J^{\alpha}$ and $[t_{0}'',t_{1}'']=J_{l}$, then we define $\xi(G_{\kappa}^{j}) := \xi(t_{0})$, $\xi(J^{\alpha}) := \xi(t_{0}')$, and $\xi(J_{l}) := \xi(t_{0}'')$.
\end{mydef}
Note that given a $u$-admissible tuple $(\epsilon_{1},\epsilon_{2},\epsilon_{3})$, we have that
\begin{equation}
\int_{0}^{T}\left(N(t)^{3}+\epsilon_{3}\|u(t)\|_{L_{x}^{4}(\R^{2})}^{4}\right)dt = 2^{k_{0}+1}\epsilon_{3},
\end{equation}
which implies that $G_{0}^{k_{0}}=[0,T]$.

Observe that by the fundamental theorem of calculus,
\begin{equation}
	|\xi(t)-\xi(G_{\kappa}^{j})| \leq \int_{G_{\kappa}^{j}}2^{-20}\epsilon_{1}^{-1/2}N(t)^{3}dt = 2^{j-19}\epsilon_{3}\epsilon_{1}^{-1/2}
\end{equation}
Hence, for all $t \in G_{\kappa}^{j}$, we have the inclusions
\begin{equation}
	\{\xi\in\R^{2} : 2^{j-1} \leq |\xi-\xi(t)| \leq 2^{j+1}\} \subset \{\xi\in\R^{2} : 2^{j-2} \leq |\xi-\xi(G_{\kappa}^{j})| \leq 2^{j+2}\} \subset \{\xi \in \R^{2} : 2^{j-3} \leq |\xi-\xi(t)| \leq 2^{j+3}\}
\end{equation}
and
\begin{equation}
	\{\xi\in\R^{2} : |\xi-\xi(t)|\leq 2^{j+1}\} \subset \{\xi \in\R^{2} :|\xi-\xi(G_{\kappa}^{j})|\leq 2^{j+2}\}\subset\{\xi \in\R^{2} : |\xi-\xi(t)|\leq 2^{j+3}\}
\end{equation}
The preceding two observations will be useful to our Littlewood-Paley analysis adapted to the frequency center $\xi(t)$ in section \ref{sec:LTSE}.

We now define the $X,Y,\tilde{X}_{k},\tilde{Y}_{k}$ norms introduced in \cite{Dodson2016}.
\begin{mydef}[$X,\tilde{X}_{k}$ norm] %Definition of $\tilde{X}_{k_{0}}$ spaces
For any $G_{\kappa}^{j}\subset [a,b]$, define
\begin{equation}
	\|u\|_{X(G_{\kappa}^{j}\times\R^{2})}^{2} \coloneqq \sum_{0\leq i<j}2^{i-j}\sum_{G_{\alpha}^{i}\subset G_{\kappa}^{j}}\|P_{\xi(G_{\alpha}^{i}), i-2\leq\cdot\leq i+2}u\|_{U_{\Delta}^{2}(G_{\alpha}^{i}\times\R^{2})}^{2}+\sum_{i\geq j}\|P_{\xi(G_{\kappa}^{j}), i-2\leq i+2}u\|_{U_{\Delta}^{2}(G_{\kappa}^{i}\times\R^{2})}^{2}.
\end{equation}
For $k\in\N_{0}$, define the maximal $\tilde{X}_{k}$ norm
\begin{equation}
	\|u\|_{\tilde{X}_{k}([a,b]\times\R^{2})}^{2} \coloneqq \sup_{0\leq j\leq \min\{k,k_{0}\}}\sup_{G_{\kappa}^{j}\subset [a,b]}\|u\|_{X(G_{\kappa}^{j}\times\R^{2})}^{2}.
\end{equation}
\end{mydef}

\begin{mydef}[$Y,\tilde{Y}_{k}$ norms] %Definition of $\tilde{Y}_{k_{0}}$ spaces
For $G_{\kappa}^{j}\subset [a,b]$, define
\begin{equation}
\begin{split}
	\|u\|_{Y(G_{\kappa}^{j}\times\R^{2})}^{2} &\coloneqq \sum_{0\leq i <j}2^{i-j}\sum_{G_{\alpha}^{i}\subset G_{\kappa}^{j} : N(G_{\alpha}^{i})\leq \epsilon_{3}^{1/2}2^{i-5}}\|P_{\xi(G_{\alpha}^{i}), i-2\leq\cdot\leq i+2}u\|_{U_{\Delta}^{2}(G_{\alpha}\times\R^{2})}^{2}\\
	&\phantom{=}+ \sum_{i \geq j : N(G_{\kappa}^{j})\leq\epsilon_{3}^{1/2}2^{i-5}}\|P_{\xi(G_{\kappa}^{j}), i-2\leq\cdot\leq i+2}u\|_{U_{\Delta}^{2}(G_{\kappa}^{j}\times\R^{2})}^{2}.
\end{split}
\end{equation}
For $k\in\N_{0}$, define the maximal $\tilde{Y}_{k}([a,b]\times\R^{2})$ norm analogously to $\tilde{X}_{k}([a,b]\times\R^{2})$.
\end{mydef}

\begin{lemma}[Cheap inductive estimate]\label{lem:chp_ie}
We have that
\begin{equation}
\|u\|_{\tilde{X}_{j+1}([a,b]\times\R^{2})}^{2}\leq 2\|u\|_{\tilde{X}_{j}([a,b]\times\R^{2})}^{2}
\end{equation}
\begin{equation}
\|u\|_{\tilde{Y}_{j+1}([a,b]\times\R^{2})}^{2}\leq 2\|u\|_{\tilde{Y}_{j}([a,b]\times\R^{2})}^{2}.
\end{equation}
\end{lemma}
\begin{proof}
The stated estimates follow readily from the observation that any interval $G_{\kappa}^{j+1}$ is the union of two $G_{\alpha}^{j}$ intervals and proposition \ref{prop:Up_lp}.
\end{proof}
The estimates of lemma \ref{lem:chp_ie}, although quite basic, will be useful in establishing the base case of and closing the inductive step in the proof of our long-time Strichartz estimate (theorem \ref{thm:LTSE}). 

\subsection{Embeddings}
%Basic estimates
We now record some basic embeddings of $\tilde{X}_{j}$ spaces into Lebesgue spaces $L_{t}^{p}L_{x}^{q}$, which will be used throughout the proof of theorem \ref{thm:LTSE}. Lemma \ref{lem:lohi_embed} is from \cite{Dodson2016}, while lemma \ref{lem:BT_embed} is used implicitly in several instances in that work.

\begin{lemma}\label{lem:lohi_embed} %Low frequency embedding lemma
Let $(p,q)$ be an admissible pair. Then for all integers $0\leq i<j$ and intervals $G_{\kappa}^{j}\subset [0,T]$, we have the estimates
\begin{equation}
\|P_{\xi(t),i}u\|_{L_{t}^{p}L_{x}^{q}(G_{\kappa}^{j}\times\R^{2})} \lesssim_{p,q} 2^{\frac{j-i}{p}} \|u\|_{\tilde{X}_{j}(G_{\kappa}^{j}\times\R^{2})}
\end{equation}
and
\begin{equation}
\|P_{\xi(t),\geq j}u\|_{L_{t}^{p}L_{x}^{q}(G_{\kappa}^{j}\times\R^{2})} \lesssim_{p,q} \|u\|_{X(G_{\kappa}^{j}\times\R^{2})}.
\end{equation}
\end{lemma}
\begin{proof}
We first prove the first estimate. By H\"{o}lder's inequality, Strichartz estimates, and the definition of the $X,\tilde{X}_{j}$ norms, we have that
\begin{align} 
\|P_{\xi(t),i}u\|_{L_{t}^{p}L_{x}^{q}(G_{\kappa}^{j}\times\R^{2})} &=\left(\sum_{G_{\alpha}^{i}\subset G_{\kappa}^{j}}\|P_{\xi(t),i}\|_{L_{t}^{p}L_{x}^{q}(G_{\kappa}^{j}\times\R^{2})}^{p}\right)^{1/p} \nonumber\\
&\leq \left(\sum_{G_{\alpha}^{i}\subset G_{\kappa}^{j}}\|P_{\xi(t),i}u\|_{L_{t}^{p}L_{x}^{q}(G_{\alpha}^{i}\times\R^{2})}^{2}\right)^{1/p}\left(\sup_{G_{\alpha}^{i}\subset G_{\kappa}^{j}}\|P_{\xi(t),i}u\|_{L_{t}^{p}L_{x}^{q}(G_{\alpha}^{i}\times\R^{2})}\right)^{1-\frac{2}{p}} \nonumber\\
&\lesssim_{p,q}\left(\sum_{G_{\alpha}^{i}\subset G_{\kappa}^{j}}\|P_{\xi(G_{\alpha}^{i}),i-2\leq\cdot\leq i+2}u\|_{U_{\Delta}^{2}(G_{\alpha}^{i}\times\R^{2})}^{2}\right)^{1/p}\left(\sup_{G_{\alpha}^{i}\subset G_{\kappa}^{j}}\|P_{\xi(G_{\alpha}^{i}),i-2\leq\cdot\leq i+2}u\|_{U_{\Delta}^{2}(G_{\alpha}^{i}\times\R^{2})}\right)^{1-\frac{2}{p}} \nonumber\\
&\leq 2^{\frac{j-i}{p}}\|u\|_{X(G_{\kappa}^{j}\times\R^{2})}^{2/p}\|u\|_{\tilde{X}_{j}(G_{\kappa}^{j}\times\R^{2})}^{1-\frac{2}{p}} \nonumber\\
&\leq 2^{\frac{j-i}{p}}\|u\|_{\tilde{X}_{j}(G_{\kappa}^{j}\times\R^{2})}.
\end{align}

We next prove the second estimate. By the triangle inequality, Littlewood-Paley inequality, Minkowski's inequality, Strichartz estimates, and the definition of the $X(G_{\kappa}^{j}\times\R^{2})$ norm,
\begin{align}
\|P_{\xi(t),\geq j}u\|_{L_{t}^{p}L_{x}^{q}(G_{\kappa}^{j}\times\R^{2})} &\lesssim_{q} \|P_{\xi(t),j}u\|_{L_{t}^{p}L_{x}^{q}(G_{\kappa}^{j}\times\R^{2})} + \|P_{\xi(t),j+1}u\|_{L_{t}^{p}L_{x}^{q}(G_{\kappa}^{j}\times\R^{2})}+\|\paren*{\sum_{l\geq j} |P_{\xi(t),l}P_{\xi(t),\geq j+2}u|^{2}}^{1/2}\|_{L_{t}^{p}L_{x}^{q}(G_{\kappa}^{j}\times\R^{2})} \nonumber\\
&\lesssim_{p,q} \|P_{\xi(G_{\kappa}^{j}),j-2\leq\cdot\leq j+2}u\|_{U_{\Delta}^{2}(G_{\kappa}^{j}\times\R^{2})} + \|P_{\xi(G_{\kappa}^{j}),j-1\leq\cdot\leq j+3}u\|_{U_{\Delta}^{2}(G_{\kappa}^{j}\times\R^{2})} \nonumber\\
&\phantom{=} + \paren*{\sum_{l\geq j} \|P_{\xi(G_{\kappa}^{j}),l-2\leq\cdot\leq l+2}u\|_{U_{\Delta}^{2}(G_{\kappa}^{j}\times\R^{2})}^{2}}^{1/2} \nonumber\\
&\lesssim \|u\|_{X(G_{\kappa}^{j}\times\R^{2})}.
\end{align}
\end{proof}

\begin{lemma}[Bernstein-type lemma]\label{lem:BT_embed} %Bernstein-type embedding lemma
Let $(p,q)$ be an admissible pair and let $q<r\leq\infty$ satisfy $\frac{2}{q}-\frac{2}{r}>\frac{1}{p}$. Then for all integers $0\leq i\leq j$ and intervals $G_{\kappa}^{j}\subset [0,T]$, we have the estimate
\begin{equation}
\|P_{\xi(t),\leq i}u\|_{L_{t}^{p}L_{x}^{r}(G_{\kappa}^{j}\times\R^{2})} \lesssim_{p,q,r} 2^{2j(\frac{1}{q}-\frac{1}{r})} \|u\|_{\tilde{X}_{j}(G_{\kappa}^{j}\times\R^{2})}.
\end{equation}
Additionally, for any $s>\frac{1}{p}$, we have the estimate
\begin{equation}
\||\nabla-i\xi(t)|^{s} P_{\xi(t), \leq i}u\|_{L_{t}^{p}L_{x}^{q}(G_{\kappa}^{j}\times\R^{2})} \lesssim_{p,q,s} 2^{js} \|u\|_{\tilde{X}_{j}(G_{\kappa}^{j}\times\R^{2})}.
\end{equation}
In particular, by the boundedness of the Riesz transforms, we have the estimate
\begin{equation}
\| (\nabla-i\xi(t)) P_{\xi(t),\leq i}u\|_{L_{t}^{p}L_{x}^{q}(G_{\kappa}^{j}\times\R^{2})} \lesssim_{p,q} 2^{j}  \|u\|_{\tilde{X}_{j}(G_{\kappa}^{j}\times\R^{2})}.
\end{equation}
\end{lemma}

\begin{proof}
We prove the first assertion of the lemma. By the triangle inequality, Bernstein's lemma, followed by lemma \ref{lem:lohi_embed}, we have that
\begin{align}
\|P_{\xi(t),\leq i}u\|_{L_{t}^{p}L_{x}^{r}(G_{k}^{j}\times\R^{2})} &\lesssim\sum_{0\leq l\leq i}2^{l(\frac{2}{q}-\frac{2}{r})}\|P_{\xi(t),l}u\|_{L_{t}^{p}L_{x}^{q}(G_{k}^{j}\times\R^{2})} \nonumber\\
&\lesssim\sum_{0\leq l\leq i}2^{l(\frac{2}{q}-\frac{2}{r})}2^{\frac{j-l}{p}} \|u\|_{\tilde{X}_{j}(G_{\kappa}^{j}\times\R^{2})} \nonumber\\
&\lesssim 2^{\frac{j}{p} + (\frac{2}{q}-\frac{2}{r}-\frac{1}{p})i} \|u\|_{\tilde{X}_{j}(G_{\kappa}^{j}\times\R^{2})} \nonumber\\
&\lesssim 2^{2j(\frac{1}{q}-\frac{1}{r})} \|u\|_{\tilde{X}_{j}(G_{\kappa}^{j}\times\R^{2})}.
\end{align}

We now prove the second assertion of the lemma. By the triangle inequality, Bernstein's lemma, followed by lemma \ref{lem:lohi_embed}, we have that 
\begin{align}
\| |\nabla-i\xi(t)|^{s} P_{\xi(t),\leq i}u\|_{L_{t}^{p}L_{x}^{q}(G_{\kappa}^{j}\times\R^{2})} &\lesssim \sum_{0\leq l\leq i}2^{ls}\|P_{\xi(t),l}u\|_{L_{t}^{p}L_{x}^{q}(G_{\kappa}^{j}\times\R^{2})} \nonumber\\
&\lesssim\sum_{0\leq l\leq i}2^{ls}2^{\frac{j-l}{p}}\|u\|_{\tilde{X}_{j}(G_{\kappa}^{j}\times\R^{2})} \nonumber\\
&\lesssim 2^{js}\|u\|_{\tilde{X}_{j}(G_{\kappa}^{j}\times\R^{2})}.
\end{align}
\end{proof}

\section{Long-time Strichartz estimate}\label{sec:LTSE}

\subsection{Overview}\label{ssec:LTSE_ov}
We now have the necessary machinery to state our long-time Strichartz estimate for admissible blowup solutions to \eqref{eq:DS} (cf. Theorem 4.1 in \cite{Dodson2016}).

\begin{thm}[Long-time Strichartz estimate]\label{thm:LTSE}
Let $u$ be an admissible blowup solution to \eqref{eq:DS}. Then there exists constants $C(u)>0$ and $\epsilon_{1}(u)\gg \epsilon_{2}(u)\gg\epsilon_{3}(u)>0$, such that the following holds: for all admissible tuples $(\epsilon_{1},\epsilon_{2},\epsilon_{3})$ with $\epsilon_{j}\leq \epsilon_{j}(u)$ for $j=1,2,3$, integers $k_{0}\geq 20$, intervals $[0,T]$ with $\int_{0}^{T}\int_{\R^{2}}|u(t,x)|^{4}dxdt=2^{k_{0}}$ and $\int_{0}^{T}N(t)^{3}dt=K$, we have the inequalities
\begin{align}
\|u_{\lambda}\|_{\tilde{X}_{k_{0}}([0,\lambda^{-2}T]\times\R^{2})} &\leq C(u)\\
\|u_{\lambda}\|_{\tilde{Y}_{k_{0}}([0,\lambda^{-2}T]\times\R^{2})} &\leq \epsilon_{2}^{1/2},
\end{align}
where $\lambda = \frac{\epsilon_{3}2^{k_{0}}}{K}$ and $u_{\lambda} \coloneqq \lambda u(\lambda^{2}\cdot,\lambda\cdot)$.
\end{thm}

\begin{remark}
In the remainder of this section, we drop the subscript $\lambda$ in the notation $u_{\lambda}$ for the rescaled solution and instead assume that $u$ satisfies $\int_{0}^{T}\int_{\R^{2}}|u(t,x)|^{4}dxdt=2^{k_{0}}$ and $\int_{0}^{T}N(t)^{3}dt=\epsilon_{3}2^{k_{0}}$ (i.e. $\lambda=1$). We use the notation $\sim_{u}, \lesssim_{u}, \gtrsim_{u}$ to denote implicit constants which depend on $u$ through its mass and almost periodicity parameters but not on the data $(\epsilon_{1},\epsilon_{2},\epsilon_{3},k_{0},[0,T],K)$.
\end{remark}

We now comment on the strategy of the proof. Inspired by the breakthrough ideas introduced in \cite{Dodson2016}, the proof is an induction on scales argument, which is set up in terms of the index $k_{0}$ in the $\tilde{X}_{k_{0}}, \tilde{Y}_{k_{0}}$ norms. Unpackaging the definitions of the $\tilde{X}_{k_{0}}, \tilde{Y}_{k_{0}}$ norms, our goal is to show that
\begin{equation}
\sum_{0\leq i< j}2^{i-j}\sum_{G_{\alpha}^{i}\subset G_{\kappa}^{j}} \|P_{\xi(G_{\alpha}^{i}),i-2\leq\cdot\leq i+2}u\|_{U_{\Delta}^{2}(G_{\alpha}^{i}\times\R^{2})}^{2}+\sum_{i\geq j} \|P_{\xi(G_{\kappa}^{j}), i-2\leq\cdot\leq i+2}u\|_{U_{\Delta}^{2}(G_{\kappa}^{j}\times\R^{2})}^{2} \leq C(u)^{2}
\end{equation}
and
\begin{equation}
\begin{split}
&\sum_{0\leq i <j}2^{i-j}\sum_{G_{\alpha}^{i}\subset G_{\kappa}^{j} : N(G_{\alpha}^{i})\leq \epsilon_{3}^{1/2}2^{i-5}}\|P_{\xi(G_{\alpha}^{i}), i-2\leq\cdot\leq i+2}u\|_{U_{\Delta}^{2}(G_{\alpha}\times\R^{2})}^{2}\\
&\phantom{=}\qquad + \sum_{i \geq j : N(G_{\kappa}^{j})\leq\epsilon_{3}^{1/2}2^{i-5}}\|P_{\xi(G_{\kappa}^{j}), i-2\leq\cdot\leq i+2}u\|_{U_{\Delta}^{2}(G_{\kappa}^{j}\times\R^{2})}^{2} \\
&\phantom{=}\leq \epsilon_{2}
\end{split}
\end{equation}
for all $0\leq j\leq k_{0}$ and $G_{\kappa}^{j}\subset [0,T]$. To achieve this goal, we proceed in several steps.

In Step 1, we prove the base case of the induction. More precisely, we show that there exists some large constant $C_{1}(u)>0$, depending only on the mass $M(u)$, such that
\begin{align}
\|u\|_{\tilde{X}_{20}([0,T]\times\R^{2})}^{2} &\leq C_{1}(u)\\
\|u\|_{\tilde{Y}_{20}([0,T]\times\R^{2})}^{2} &\leq C_{1}(u)\epsilon_{2}^{3/2}.
\end{align}
The precise choice of the subscript $20$ is immaterial; it is just technically convenient to have some large integer. In this step, we proceed similarly to \cite{Dodson2016} we prove these estimates using \emph{linear} Strichartz estimates, some basic Littlewood-Paley theory, lemma \ref{lem:chp_ie}, and in the case of the $\tilde{Y}_{20}$ estimate, the frequency localization property \eqref{eq:sp_frq_loc}. The primary new ingredient in this step is extensive use of the $L^{p}$ boundedness of the operator $\E$ for $1<p<\infty$, which is a consequence of the Calder\'{o}n-Zygmund theorem.

In Step 2, we proceed to the indices $i>20$, which we consider throughout the remainder of the proof. We use Duhamel's formula to write the solution $u$ in terms of its linear and nonlinear parts,
\begin{equation}
P_{\xi(G_{\alpha}^{i}), i-2\leq\cdot\leq i+2}u(t) = e^{i(t-t_{\alpha}^{i})\Delta}P_{\xi(G_{\alpha}^{i}),i-2\leq\cdot\leq i+2}u(t_{\alpha}^{i})-i\int_{t_{\alpha}^{i}}^{t}e^{i(t-\tau)\Delta}P_{\xi(G_{\alpha}^{i}),i-2\leq\cdot\leq i+2}F(u(\tau))d\tau, \qquad \forall t\in G_{\alpha}^{i},
\end{equation}
for a judicious choice of $t_{\alpha}^{i}\in G_{\alpha}^{i}$; similarly for $G_{\alpha}^{i}$ replaced by $G_{\kappa}^{j}$. We can control the $U_{\Delta}^{2}(G_{\alpha}^{i}\times\R^{2})$ and $U_{\Delta}^{2}(G_{\kappa}^{j}\times\R^{2})$ norms of the linear part by
\begin{equation}
\|P_{\xi(G_{\alpha}^{i}), i-2\leq\cdot\leq i+2}u(t_{\alpha}^{i})\|_{L^{2}(\R^{2})} \enspace \text{and} \enspace \|P_{\xi(G_{\kappa}^{j}),i-2\leq\cdot\leq i+2}u(t_{\kappa}^{j})\|_{L^{2}(\R^{2})},
\end{equation}
respectively. Therefore we expect and show that the mass of the solution controls the total linear contribution
\begin{equation}
\sum_{0\leq i<j}2^{i-j}\sum_{G_{\alpha}^{i}\subset G_{\kappa}^{j}} \|P_{\xi(G_{\alpha}^{i}), i-2\leq\cdot\leq i+2}u(t_{\alpha}^{i})\|_{L^{2}(\R^{2})}^{2} + \sum_{i\geq j} \|P_{\xi(G_{\kappa}^{j}),i-2\leq\cdot\leq i+2}u(t_{\kappa}^{j})\|_{L^{2}(\R^{2})}^{2}.
\end{equation}
Indeed, we prove an $O(1)$ bound. When we are localizing $u(t_{\alpha}^{i})$ or $u(t_{\kappa}^{j})$ to frequencies which are far from the center $\xi(t)$ relative to the scale $N(t)$ on an interval $G_{\alpha}^{i}$ or $G_{\kappa}^{j}$, we expect from \eqref{eq:sp_frq_loc} for this contribution
\begin{equation}
\begin{split}
&\sum_{0\leq i<j} 2^{i-j}\sum_{G_{\alpha}^{i}\subset G_{\kappa}^{j}; N(G_{\alpha}^{i}) \leq \epsilon_{3}^{1/2}2^{i-5}} \|P_{\xi(G_{\alpha}^{i}), i-2\leq\cdot\leq i+2}u(t_{\alpha}^{i})\|_{L^{2}(\R^{2})}^{2} \\
&\phantom{=} + \sum_{i\geq j; N(G_{\kappa}^{j}) \leq \epsilon_{3}^{1/2}2^{i-5}} \|P_{\xi(G_{\kappa}^{j}),i-2\leq\cdot\leq i+2}u(t_{\kappa}^{j})\|_{L^{2}(\R^{2})}^{2}
\end{split}
\end{equation}
to be small. Indeed, we prove an $O(\epsilon_{2}^{3/2})$ bound.

In Step 3, we begin to estimate the contribution of the nonlinear part of the solution in Duhamel's formula. We consider the indices $20\leq i<j$ for which there are intervals $G_{\alpha}^{i}\subset G_{\kappa}^{j}$ on which $N(G_{\alpha}^{i}) > \epsilon_{3}^{1/2}2^{i-10}$ and the indices $i\geq j$ such that $N(G_{\kappa}^{j}) > \epsilon_{3}^{1/2}2^{i-10}$. In the quantities
\begin{equation}
\|\int_{t_{\alpha}^{i}}^{t} P_{\xi(G_{\alpha}^{i}),i-2\leq\cdot\leq i+2}F(u(\tau))d\tau\|_{U_{\Delta}^{2}(G_{\alpha}^{i}\times\R^{2})}^{2} \enspace \text{and} \enspace \|\int_{t_{\kappa}^{j}}^{t} P_{\xi(G_{\kappa}^{j}),i-2\leq\cdot\leq i+2}F(u(\tau))d\tau\|_{U_{\Delta}^{2}(G_{\kappa}^{j}\times\R^{2})}^{2},
\end{equation}
we are measuring the nonlinear part at frequencies which are \emph{near}, relatively speaking, the frequency center $\xi(t)$ on $G_{\alpha}^{i}$ and $G_{\kappa}^{j}$, respectively. By the frequency localization property \eqref{eq:sp_frq_loc}, the solution is mass concentrated on this region. Hence, to prove an $O(1)$ estimate for the total contribution
\begin{equation}
\begin{split}
&\sum_{20\leq i<j} 2^{i-j} \sum_{G_{\alpha}^{i}\subset G_{\kappa}^{j}; N(G_{\alpha}^{i})>\epsilon_{3}^{1/2}2^{i-10}} \|\int_{t_{\alpha}^{i}}^{t} P_{\xi(G_{\alpha}^{i}),i-2\leq\cdot\leq i+2}F(u(\tau))d\tau\|_{U_{\Delta}^{2}(G_{\alpha}^{i}\times\R^{2})}^{2} \\
&\phantom{=} + \sum_{i\geq j; N(G_{\kappa}^{j})>\epsilon_{3}^{1/2}2^{i-10}} \|\int_{t_{\kappa}^{j}}^{t} P_{\xi(G_{\kappa}^{j}),i-2\leq\cdot\leq i+2}F(u(\tau))d\tau\|_{U_{\Delta}^{2}(G_{\kappa}^{j}\times\R^{2})}^{2},
\end{split}
\end{equation}
we show that there not too many of these intervals $G_{\alpha}^{i}$. Similarly, we consider those indices $20\leq i<j$ for which there are intervals $G_{\alpha}^{i}\subset G_{\kappa}^{j}$ satisfying 
\begin{equation}
\epsilon_{3}^{1/2}2^{i-10}< N(G_{\alpha}^{i}) \leq \epsilon_{3}^{1/2}2^{i-5}
\end{equation}
and those indices $i\geq j$ satisfying $\epsilon_{3}^{1/2}2^{i-10} <N(G_{\kappa}^{j}) \leq \epsilon_{3}^{1/2}2^{i-5}$. Since we are now localizing the nonlinear part to frequencies with distance from $\xi(t)$ comparable to $\epsilon_{3}^{-1/2}N(G_{\alpha}^{i})$ and $\epsilon_{3}^{-1/2}N(G_{\kappa}^{j})$, respectively, we still expect from \eqref{eq:sp_frq_loc} and  indeed show an $O(\epsilon_{2}^{3/2})$ bound for the total contribution
\begin{equation}
\begin{split}
&\sum_{20\leq i<j} 2^{i-j} \sum_{G_{\alpha}^{i}\subset G_{\kappa}^{j}; \epsilon_{3}^{1/2}2^{i-10} < N(G_{\alpha}^{i}) \leq \epsilon_{3}^{1/2}2^{i-5}} \|\int_{t_{\alpha}^{i}}^{t} P_{\xi(G_{\alpha}^{i}),i-2\leq\cdot\leq i+2}F(u(\tau))d\tau\|_{U_{\Delta}^{2}(G_{\alpha}^{i}\times\R^{2})}^{2} \\
&\phantom{=} + \sum_{i\geq j; \epsilon_{3}^{1/2}2^{i-10} < N(G_{\kappa}^{j}) \leq \epsilon_{3}^{1/2}2^{i-5}} \|\int_{t_{\kappa}^{j}}^{t} P_{\xi(G_{\kappa}^{j}),i-2\leq\cdot\leq i+2}F(u(\tau))d\tau\|_{U_{\Delta}^{2}(G_{\kappa}^{j}\times\R^{2})}^{2}.
\end{split}
\end{equation}

In Step 4, we estimate the remaining contribution of the nonlinear part of the solution at frequencies which are \emph{far} away from the frequency center function $\xi(t)$ relative to the scale $N(t)$, or more precisely, the quantity
\begin{equation}
\begin{split}
&\sum_{20\leq i<j}2^{i-j}\sum_{G_{\alpha}^{i}\subset G_{\kappa}^{j}; N(G_{\alpha}^{i})\leq \epsilon_{3}^{1/2}2^{i-10}}\|\int_{t_{\alpha}^{i}}^{t}P_{\xi(G_{\alpha}^{i}), i-2\leq\cdot\leq i+2}F(u(\tau))d\tau\|_{U_{\Delta}^{2}(G_{\alpha}^{i}\times\R^{2})}^{2}\\
&\phantom{=} + \sum_{i\geq j; N(G_{\kappa}^{j})\leq\epsilon_{3}^{1/2}2^{i-10}}\|\int_{t_{\kappa}^{j}}^{t}e^{i(t-\tau)\Delta}P_{\xi(G_{\kappa}^{j}), i-2\leq\cdot\leq i+2}F(u(\tau))d\tau\|_{U_{\Delta}^{2}(G_{\kappa}^{j}\times\R^{2})}^{2}.
\end{split}
\end{equation}
As was the case in \cite{Dodson2016}, this is the most difficult step of the proof. Given the frequency localization property \eqref{eq:sp_frq_loc}, morally we expect this quantity too be small (i.e. $O(\epsilon_{2}^{\delta})$ for some $\delta>0$), but proving this far from obvious. The strategy is to use a bootstrap argument which exploits both the strong inductive nature of the $\tilde{X}_{j}$ and $\tilde{Y}_{j}$ norms and the ``smallness" provided by the $\tilde{Y}_{j}$ norm. Indeed, this latter point is precisely the reason why we have been estimating the $\tilde{Y}_{j}$ norm so far. To prove such a bootstrap proposition, we use Littlewood-Paley theory to split estimating
\begin{equation}
P_{\xi(G_{\alpha}^{i}), i-2\leq\cdot\leq i+2}F(u)
\end{equation}
into two model cases:
\begin{enumerate}[(i)]
\item\label{item:LTSE_mc_1}
\begin{equation}
P_{\xi(G_{\alpha}^{i}), i-2\leq\cdot\leq i+2}[\E\paren*{(P_{\xi(t), <i-10}u)(\ol{P_{\xi(t), \geq i-10}u})}(P_{\xi(G_{\alpha}^{i}), \geq i-5}u)],
\end{equation}
\item\label{item:LTSE_mc_2}
\begin{equation}
P_{\xi(G_{\alpha}^{i}), i-2\leq\cdot\leq i+2}[\E\paren*{|P_{\xi(t), <i-10}u|^{2}}(P_{\xi(G_{\alpha}^{i}), \geq i-5}u)].
\end{equation}
\end{enumerate}
Case \ref{item:LTSE_mc_1} is the easy one as we have two far frequency factors in the nonlinearity. Classical linear and bilinear estimates suffice, and we do not yet need the double frequency decomposition. Case \ref{item:LTSE_mc_2} is the hard one as we only have one far frequency factor in the nonlinearity. Moreover, both near frequency factors are inside the argument of the nonlocal operator $\E$. The classical linear and bilinear estimates do not suffice, and we cannot apply bilinear estimates to the nonlinearity without using the double frequency decomposition. It is as this step that we rely on three new bilinear Strichartz estimates adapted to the double frequency decomposition, the proofs of which we defer to section \ref{sec:BSE}.

Before proceeding to the details of the proof, we lastly remark that the proof of theorem \ref{thm:LTSE} is agnostic to the signs and magnitudes of the parameters $\mu_{1}$ and $\mu_{2}$ in the eeDS equation: the signs only become relevant in the rigidity step to preclude the quasi-soliton scenario. Therefore, we will simplify the notation by setting $\mu_{1}=\mu_{2}=1$.

\subsection{Step 1: Base case}\label{ssec:LTSE_S1} %Base case of the induction argument
In this subsection, we prove the base case of the induction argument.

\begin{lemma} [Base case] \label{lem:LTSE_base}
There exists a constant $C(u)>0$ such that
\begin{gather}
	\|u\|_{\tilde{X}_{20}([0,T]\times\R^{2})}^{2} \leq 2^{20}C(u)\\
	\|u\|_{\tilde{Y}_{20}([0,T]\times\R^{2})}^{2} \leq 2^{20}C(u)\epsilon_{2}^{3/2}
\end{gather}
\end{lemma}
\begin{proof}
We first claim that
\begin{align}
\|u\|_{\tilde{X}_{0}([0,T]\times\R^{2})}^{2}&\leq C(u)\\
\|u\|_{\tilde{Y}_{0}([0,T]\times\R^{2})}^{2}&\leq \epsilon_{2}^{3/2}.
\end{align}
Indeed, the first inequality follows from the fact that $\|u\|_{U_{\Delta}^{2}(J^{\alpha}\times\R^{2})}\lesssim_{u} 1$ for every $J^{\alpha}\subset G_{\kappa}^{j}$. For the second inequality, we observe from Duhamel's principle and duality that
\begin{equation}
\begin{split}
\paren*{\sum_{i\geq 0 : N(J^{\alpha}) \leq\epsilon_{3}^{1/2}2^{i-5}} \|P_{\xi(J^{\alpha}), i-2\leq \cdot\leq i+2}u\|_{U_{\Delta}^{2}(J^{\alpha}\times\R^{2})}^{2}}^{1/2} &\lesssim\|P_{\xi(t),\geq 2\epsilon_{3}^{-1/2}N(t)}u\|_{L_{t}^{\infty}L_{x}^{2}(J^{\alpha}\times\R^{2})}\\
&\phantom{=}+\|P_{\xi(t), \geq 2\epsilon_{3}^{-1/2}N(t)}F(u)\|_{L_{t}^{1}L_{x}^{2}(J^{\alpha}\times\R^{2})}.
\end{split}
\end{equation}
By the frequency localization property \eqref{eq:sp_frq_loc},
\begin{equation}
\|P_{\xi(t),\geq 2\epsilon_{3}^{-1/2}N(t)}u\|_{L_{t}^{\infty}L_{x}^{2}(J^{\alpha}\times\R^{2})} \leq \epsilon_{2}.
\end{equation}
Performing a near-far frequency decomposition
\begin{equation}
u=P_{\xi(t),<\epsilon_{3}^{-1/2}2^{-2}N(t)}u + P_{\xi(t),\geq \epsilon_{3}^{-1/2}2^{-2}N(t)}u
\end{equation}
in the expression $P_{\xi(t),\geq 2\epsilon_{3}^{-1/2}N(t)}F(u)$, we obtain a sum of terms where each term consists of three factors, at least one of which is supported on frequencies $\xi$ satisfying $|\xi-\xi(t)|\geq \epsilon_{3}^{-1/2}2^{-2}N(t)$. Then by H\"{o}lder's inequality, Calder\'{o}n-Zygmund theorem, Strichartz estimates, and the interpolation estimate
\begin{equation}
\|f\|_{L_{t}^{9}L_{x}^{18/7}}\leq\|f\|_{L_{t}^{\infty}L_{x}^{2}}^{3/4} \|f\|_{L_{t}^{9/4}L_{x}^{18}}^{1/4},
\end{equation}
it follows that
\begin{equation}
\|P_{\xi(t), \geq 2\epsilon_{3}^{-1/2}N(t)}F(u)\|_{L_{t}^{1}L_{x}^{2}(J^{\alpha}\times\R^{2})} \lesssim \|P_{\xi(t), \geq\epsilon_{3}^{-1/2}2^{-2}N(t)}u\|_{L_{t}^{\infty}L_{x}^{2}(J^{\alpha}\times\R^{2})}^{3/4} \|u\|_{L_{t}^{9/4}L_{x}^{18}(J^{\alpha}\times\R^{2})}^{9/4} \lesssim_{u} \epsilon_{2}^{3/4},
\end{equation}
where we use \eqref{eq:sp_frq_loc} to obtain the ultimate line.

To conclude the proof, we now use 20 applications of lemma \ref{lem:chp_ie}.
\end{proof}

\subsection{Step 2: Linear contribution estimate}\label{ssec:LTSE_S2}
Unpackaging the definitions of the $\tilde{X}_{k_{0}}$ and $\tilde{Y}_{k_{0}}$ norms, it remains for us to show that there exists a constant $C'(u)>0$, possibly larger by a fixed absolute factor than the constant $C(u)$ obtained in lemma \ref{lem:LTSE_base}, such that
\begin{equation}
\begin{split}
&\sum_{20\leq i<j} 2^{i-j}\sum_{G_{\alpha}^{i}\subset G_{\kappa}^{j}} \|P_{\xi(G_{\alpha}^{i}), i-2\leq \cdot\leq i+2}u\|_{U_{\Delta}^{2}(G_{\alpha}^{i}\times\R^{2})}^{2} +\sum_{i\geq j} \|P_{\xi(G_{\kappa}^{j}),i-2\leq\cdot\leq i+2}u\|_{U_{\Delta}^{2}(G_{\kappa}^{j}\times\R^{2})}^{2} \leq C'(u)
\end{split}
\end{equation}
and
\begin{equation}
\begin{split}
&\sum_{20\leq i<j} 2^{i-j}\sum_{G_{\alpha}^{i}\subset G_{\kappa}^{j}; N(G_{\alpha}^{i}) \leq \epsilon_{3}^{1/2}2^{i-5}} \|P_{\xi(G_{\alpha}^{i}), i-2\leq\cdot i+2}u\|_{U_{\Delta}^{2}(G_{\alpha}^{i}\times\R^{2})}^{2} \\
&\phantom{=}\qquad +\sum_{i\geq j; N(G_{\kappa}^{j})\leq \epsilon_{3}^{1/2}2^{i-5}} \|P_{\xi(G_{\kappa}^{j}),i-2\leq\cdot\leq i+2}u\|_{U_{\Delta}^{2}(G_{\kappa}^{j}\times\R^{2})}^{2}\\
&\phantom{=} \leq C'(u)\epsilon_{2}^{3/2},
\end{split}
\end{equation}
for all indices $20\leq j\leq k_{0}$ and intervals $G_{\kappa}^{j}\subset [0,T]$. Let $G_{\kappa}^{j}\subset [0,T]$ with $20\leq j\leq k_{0}$. By Duhamel's formula, for any $20\leq i<j$ and $t_{\alpha}^{i}\in G_{\alpha}^{i}\subset G_{\kappa}^{j}$,
\begin{equation}
\begin{split}
\|P_{\xi(G_{\alpha}^{i}),i-2\leq\cdot\leq i+2}u\|_{U_{\Delta}^{2}(G_{\alpha}^{i}\times\R^{2})}^{2} &\lesssim \|P_{\xi(G_{\alpha}^{i}),i-2\leq\cdot\leq i+2}u(t_{\alpha}^{i})\|_{L^{2}(\R^{2})}^{2}\\
&\phantom{=}+\|\int_{t_{\alpha}^{i}}^{t}e^{i(t-\tau)\Delta}P_{\xi(G_{\alpha}^{i}),i-2\leq\cdot\leq i+2}F(u(\tau))d\tau\|_{U_{\Delta}^{2}(G_{\alpha}^{i}\times\R^{2})}^{2}
\end{split}
\end{equation}
and for $t_{\kappa}^{j}\in G_{\kappa}^{j}$ and $i\geq j$,
\begin{equation}
\begin{split}
\|P_{\xi(G_{\kappa}^{j}),i-2\leq\cdot\leq i+2}u\|_{U_{\Delta}^{2}(G_{\kappa}^{j}\times\R^{2})}^{2} &\lesssim \|P_{\xi(G_{\kappa}^{j}),i-2\leq\cdot\leq i+2}u(t_{\kappa}^{j})\|_{L^{2}(\R^{2})}^{2} \\
&\phantom{=}+\|\int_{t_{\kappa}^{j}}^{t}e^{i(t-\tau)\Delta}P_{\xi(G_{\kappa}^{j}),i-2\leq\cdot\leq i+2}F(u(\tau))d\tau\|_{U_{\Delta}^{2}(G_{\kappa}^{j}\times\R^{2})}^{2}.
\end{split}
\end{equation}
For $i<j$, since $u$ belongs to $C_{t}^{0}L_{x}^{2}$ and $G_{\alpha}^{i}$ is compact, we may choose $t_{\alpha}^{i}$ to satisfy
\begin{equation}
\|P_{\xi(G_{\alpha}^{i}), i-2\leq\cdot\leq i+2}u(t_{\alpha}^{i})\|_{L^{2}(\R^{2})} = \inf_{t\in G_{\alpha}^{i}} \|P_{\xi(G_{\alpha}^{i}),i-2\leq\cdot\leq i+2}u(t)\|_{L_{x}^{2}(\R^{2})}.
\end{equation}
We choose $t_{\kappa}^{j}=t(G_{\kappa}^{j})$.

The goal now is to prove the following lemma.

\begin{lemma}[Linear part estimate]\label{lem:lin_bnd} %Bound for the linear contribution to Duhamel's formula
The following estimates hold uniformly in $20\leq j\leq k_{0}$ and $G_{\kappa}^{j}\subset [0,T]$:
\begin{equation}
\begin{split}
\sum_{20\leq i<j} 2^{i-j}\sum_{G_{\alpha}^{i}\subset G_{\kappa}^{j}} \|P_{\xi(G_{\alpha}^{i}),i-2\leq\cdot\leq i+2}u(t_{\alpha}^{i})\|_{L^{2}(\R^{2})}^{2}+\sum_{i\geq j} \|P_{\xi(G_{\kappa}^{j}),i-2\leq\cdot\leq i+2}u(t_{\kappa}^{j})\|_{L^{2}(\R^{2})} \lesssim_{u} 1
\end{split}
\end{equation}
and
\begin{equation}
\begin{split}
&\sum_{20\leq i<j} 2^{i-j}\sum_{G_{\alpha}^{i}\subset G_{\kappa}^{j}; N(G_{\alpha}^{i})\leq \epsilon_{3}^{1/2}2^{i-5}} \|P_{\xi(G_{\alpha}^{i}), i-2\leq\cdot\leq i+2}u(t_{\alpha}^{i})\|_{L^{2}(\R^{2})}^{2} + \sum_{i\geq j; N(G_{\kappa}^{j})\leq\epsilon_{3}^{1/2}2^{i-5}} \|P_{\xi(G_{\kappa}^{j}), i-2\leq\cdot\leq i+2}u(t_{\kappa}^{j})\|_{L^{2}(\R^{2})}^{2} \\
&\phantom{=} \lesssim_{u} \epsilon_{2}^{2}.
\end{split}
\end{equation}
\end{lemma}
\begin{proof}
We first show the first assertion of the lemma. By the definition of $t_{\alpha}^{i}$ and almost orthogonality of the Littlewood-Paley projectors, we have that
\begin{align}
&\sum_{20\leq i<j}2^{i-j}\sum_{G_{\alpha}^{i}\subset G_{\kappa}^{j}}\|P_{\xi(G_{\alpha}^{i}), i-2\leq \cdot\leq i+2}u(t_{\alpha}^{i})\|_{L^{2}(\R^{2})}^{2} \nonumber\\
&\phantom{=}=\sum_{20\leq i<j}2^{-j-1}\epsilon_{3}^{-1}\sum_{G_{\alpha}^{i}\subset G_{\kappa}^{j}}\int_{G_{\alpha}^{i}}\paren*{N(t)^{3}+\epsilon_{3}\|u(t)\|_{L_{x}^{4}(\R^{2})}^{4}}\|P_{\xi(G_{\alpha}^{i}), i-2\leq \cdot\leq i+2}u(t_{\alpha}^{i})\|_{L^{2}(\R^{2})}^{2}dt \nonumber\\
&\phantom{=}\lesssim\sum_{20\leq i<j}2^{-j}\epsilon_{3}^{-1}\sum_{G_{\alpha}^{i}\subset G_{\kappa}^{j}}\int_{G_{\alpha}^{i}}\paren*{N(t)^{3}+\epsilon_{3}\|u(t)\|_{L_{x}^{4}(\R^{2})}^{4}}\|P_{\xi(t), i-4\leq \cdot\leq i+4}u(t)\|_{L_{x}^{2}(\R^{2})}^{2}dt \nonumber\\
&\phantom{=}=\sum_{20\leq i<j}2^{-j}\epsilon_{3}^{-1}\int_{G_{\kappa}^{j}}\paren*{N(t)^{3}+\epsilon_{3}\|u(t)\|_{L_{x}^{4}(\R^{2})}^{4}}\|P_{\xi(t), i-4\leq \cdot\leq i+4}u(t)\|_{L_{x}^{2}(\R^{2})}^{2}dt \nonumber\\
&\phantom{=}\lesssim 2^{-j}\epsilon_{3}^{-1}\int_{G_{\kappa}^{j}}\paren*{N(t)^{3}+\epsilon_{3}\|u(t)\|_{L_{x}^{4}(\R^{2})}^{4}}\|u\|_{L_{t}^{\infty}L_{x}^{2}}^{2}dt \nonumber\\
&\phantom{=}\lesssim_{u} 1.
\end{align}
Now by Plancherel's theorem and almost orthogonality,
\begin{equation}
\sum_{i\geq j} \|P_{\xi(G_{\kappa}^{j}), i-2\leq\cdot\leq i+2}u(t_{\kappa}^{j})\|_{L^{2}(\R^{2})}^{2} \lesssim \|u(t_{\kappa}^{j})\|_{L^{2}(\R^{2})}^{2} \lesssim_{u} 1,
\end{equation}
which completes the proof of the first assertion of the lemma.

We now show the second assertion of the lemma. We use Plancherel's theorem to obtain
\begin{align}
&\sum_{20\leq i<j}2^{i-j}\sum_{G_{\alpha}^{i} \subset G_{\kappa}^{j}; N(G_{\alpha}^{i}) \leq \epsilon_{3}^{1/2}2^{i-5}} \|P_{\xi(G_{\alpha}^{i}),i-2\leq\cdot\leq i+2}u(t_{\alpha}^{i})\|_{L^{2}(\R^{2})}^{2} \nonumber\\
&\phantom{=}\lesssim\sum_{20\leq i<j}\epsilon_{3}^{-1}2^{-j}\sum_{G_{\alpha}^{i} \subset G_{\kappa}^{j}; N(G_{\alpha}^{i}) \leq \epsilon_{3}^{1/2}2^{i-5}} \int_{G_{\alpha}^{i}}\paren*{N(t)^{3}+\epsilon_{3}\|u(t)\|_{L_{x}^{4}(\R^{2})}^{4}} \|P_{\xi(t), i-4\leq\cdot\leq i+4}u(t)\|_{L_{x}^{2}(\R^{2})}^{2}dt. \label{eq:lin_bnd_f}
\end{align}
For each $20\leq i<j$ fixed and $G_{\alpha}^{i}\subset G_{\kappa}^{j}$, the condition $N(G_{\alpha}^{i})\leq \epsilon_{3}^{1/2}2^{i-5}$ implies that $N(t)\leq\epsilon_{3}^{1/2}2^{i-5}$ for all $t\in G_{\alpha}^{i}$. Since for fixed $i$, the sets $G_{\alpha}^{i}$ are disjoint (ignoring a measure zero overlap), we have the inclusion
\begin{equation}
\bigcup_{G_{\alpha}^{i}\subset G_{\kappa}^{j}: N(G_{\alpha}^{i})\leq \epsilon^{1/2}2^{i-5}} G_{\alpha}^{i} \subset \{t\in G_{\kappa}^{j}:N(t)\leq \epsilon_{3}^{1/2}2^{i-5}\}.
\end{equation}
Hence,
\begin{align}
\eqref{eq:lin_bnd_f} &= \sum_{20\leq i<j} \epsilon_{3}^{-1}2^{-j}\int_{\{t'\in G_{\kappa}^{j}: N(t')\leq \epsilon_{3}^{1/2}2^{i-5}\}} \paren*{N(t)^{3}+\epsilon_{3}\|u(t)\|_{L_{x}^{4}(\R^{2})}^{4}} \|P_{\xi(t),i-4\leq \cdot\leq i+4}u(t)\|_{L_{x}^{2}(\R^{2})}^{2}dt \nonumber\\
&\lesssim \epsilon_{3}^{-1}2^{-j}\int_{G_{\kappa}^{j}}\paren*{N(t)^{3}+\epsilon_{3}\|u(t)\|_{L_{x}^{4}(\R^{2})}^{4}} \|P_{\xi(t), \geq \epsilon_{3}^{-1/2}N(t)}u(t)\|_{L_{x}^{2}(\R^{2})}^{2}dt,
\end{align}
where the ultimate inequality follows from the Fubini-Tonelli theorem, followed by Plancherel's theorem. Using the frequency localization property \eqref{eq:sp_frq_loc}, we conclude that
\begin{equation}
\epsilon_{3}^{-1}2^{-j}\int_{G_{\kappa}^{j}}\paren*{N(t)^{3}+\epsilon_{3}\|u(t)\|_{L_{x}^{4}(\R^{2})}^{4}} \|P_{\xi(t), \geq \epsilon_{3}^{-1/2}N(t)}u(t)\|_{L_{x}^{2}(\R^{2})}^{2}dt \leq \epsilon_{2}^{2}.
\end{equation}
Lastly, by almost orthogonality and \eqref{eq:sp_frq_loc},
\begin{equation}
\sum_{i\geq j; N(G_{\kappa}^{j})\leq \epsilon_{3}^{1/2}2^{i-5}} \|P_{\xi(G_{\kappa}^{j}), i-2\leq\cdot\leq i+2}u(t_{\kappa}^{j})\|_{L^{2}(\R^{2})}^{2} \lesssim \sup_{t\geq 0}\|P_{\xi(t), \geq\epsilon_{3}^{-1/2}N(t)}u(t)\|_{L_{x}^{2}(\R^{2})}^{2} \leq \epsilon_{2}^{2},
\end{equation}
completing the proof of the second assertion of the lemma.
\end{proof}

\subsection{Step 3: Near frequency nonlinear estimate}\label{ssec:LTSE_S3}
Applying lemma \ref{lem:lin_bnd}, we have shown that
\begin{equation}\label{eq:Xlin_red}
\begin{split}
\|u\|_{X(G_{\kappa}^{j}\times\R^{2})}^{2} &\lesssim_{u} 1 + \sum_{20\leq i <j}2^{i-j}\sum_{G_{\alpha}^{i}\subset G_{\kappa}^{j}}\|\int_{t_{\alpha}^{i}}^{t}e^{i(t-\tau)\Delta}P_{\xi(G_{\alpha}^{i}), i-2\leq\cdot\leq i+2}F(u(\tau))d\tau\|_{U_{\Delta}^{2}(G_{\alpha}^{i}\times\R^{2})}^{2}\\
&\phantom{=} + \sum_{i\geq j}\|\int_{t_{\kappa}^{j}}^{t}e^{i(t-\tau)\Delta}P_{\xi(G_{\kappa}^{j}), i-2\leq\cdot\leq i+2}F(u(\tau))d\tau\|_{U_{\Delta}^{2}(G_{\kappa}^{j}\times\R^{2})}^{2}
\end{split}
\end{equation}
and
\begin{equation}\label{eq:Ylin_red}
\begin{split}
\|u\|_{Y(G_{\kappa}^{j}\times\R^{2})}^{2} &\lesssim_{u} \epsilon_{2}^{3/2} + \sum_{20\leq i <j}\sum_{G_{\alpha}^{i}\subset G_{\kappa}^{j}; N(G_{\alpha}^{i})\leq\epsilon_{3}^{1/2}2^{i-5}}\|\int_{t_{\alpha}^{i}}^{t}e^{i(t-\tau)\Delta}P_{\xi(G_{\alpha}^{i}), i-2\leq\cdot\leq i+2}F(u(\tau))d\tau\|_{U_{\Delta}^{2}(G_{\alpha}^{i}\times\R^{2})}^{2}\\
&\phantom{=}+\sum_{i\geq j ; N(G_{\kappa}^{j})\leq \epsilon_{3}^{1/2}2^{i-5}} \|\int_{t_{\kappa}^{j}}^{t}e^{i(t-\tau)\Delta}P_{\xi(G_{\kappa}^{j}), i-2\leq\cdot\leq i+2}F(u(\tau))d\tau\|_{U_{\Delta}^{2}(G_{\kappa}^{j}\times\R^{2})}^{2},
\end{split}
\end{equation}
so it remains for us to the estimate the nonlinear contributions in the RHSs of both estimates. We first observe that the range of summation over the intervals $G_{\alpha}^{i}\subset G_{\kappa}^{j}$ in \eqref{eq:Xlin_red} does not match the range of summation over the intervals $G_{\alpha}^{i}\subset G_{\kappa}^{j}$ in \eqref{eq:Ylin_red}. The former includes intervals $G_{\alpha}^{i}$ on which the frequency scale function $N(t)$ is much larger than $\epsilon_{3}^{1/2}2^{i}$. Since we ultimately want to use the $\tilde{Y}_{k}$ norms to close the inductive estimate for $\|u\|_{\tilde{X}_{k+1}([0,T]\times\R^{2})}$, we want to match the $G_{\alpha}^{i}$ range of summation between \eqref{eq:Xlin_red} and \eqref{eq:Ylin_red}. Doing so requires us to estimate the nonlinear part of the solution at frequency scales smaller than $\epsilon_{3}^{-1/2 }N(G_{\alpha}^{i})$ and $\epsilon_{3}^{-1/2}N(G_{\kappa}^{j})$. Presumably by \eqref{eq:sp_frq_loc}, $\hat{u}$ is mass concentrated on this region, and therefore each term
\begin{align}
&\|\int_{t_{\alpha}^{i}}^{t}e^{i(t-\tau)\Delta}P_{\xi(G_{\alpha}^{i}), i-2\leq\cdot\leq i+2}F(u(\tau))d\tau\|_{U_{\Delta}^{2}(G_{\alpha}^{i}\times\R^{2})}^{2}\\
&\|\int_{t_{\kappa}^{j}}^{t}e^{i(t-\tau)\Delta}P_{\xi(G_{\kappa}^{j}),i-2\leq\cdot\leq i+2}F(u(\tau))d\tau\|_{U_{\Delta}^{2}(G_{\kappa}^{j}\times\R^{2})}^{2}
\end{align}
should be ``large" over these intervals. Therefore if we have any hope of proving theorem \ref{thm:LTSE}, we should show that there cannot be too many such terms, so that the total contribution $\lesssim 1$. Thus, the goal of this subsection is prove the following lemma.

\begin{lemma}[Near/Intermediate frequency estimate]\label{lem:LTSE_NI_est} %frequency scale control lemma
The following inequalities hold uniformly in $20\leq j\leq k_{0}$ and $G_{\kappa}^{j}\subset [0,T]$:
\begin{equation}\label{eq:LTSE_NI_1}
\begin{split}
&\sum_{20\leq i< j} 2^{i-j}\sum_{G_{\alpha}^{i}\subset G_{\kappa}^{j}; N(G_{\alpha}^{i}) >\epsilon_{3}^{1/2} 2^{i-10}} \|\int_{t_{\alpha}^{i}}^{t}e^{i(t-\tau)\Delta} P_{\xi(G_{\alpha}^{i}),i-2\leq\cdot\leq i+2}F(u(\tau))d\tau \|_{U_{\Delta}^{2}(G_{\alpha}^{i}\times\R^{2})}^{2} \\
&\phantom{=}\qquad +\sum_{i\geq j; N(G_{\kappa}^{j})> \epsilon_{3}^{1/2}2^{i-10}} \|\int_{t_{\kappa}^{j}}^{t}e^{i(t-\tau)\Delta}P_{\xi(G_{\kappa}^{j}),i-2\leq\cdot\leq i+2}F(u(\tau))d\tau\|_{U_{\Delta}^{2}(G_{\kappa}^{j}\times\R^{2})}^{2} \\
&\phantom{=} \lesssim_{u} 1
\end{split}
\end{equation}
and
\begin{equation}\label{eq:LTSE_NI_2}
\begin{split}
&\sum_{20\leq i <j} 2^{i-j}\sum_{G_{\alpha}^{i}\subset G_{\kappa}^{j}; \epsilon_{3}^{1/2}2^{i-10}<N(G_{\alpha}^{i})\leq \epsilon_{3}^{1/2}2^{i-5}} \|\int_{t_{\alpha}^{i}}^{t}e^{i(t-\tau)\Delta}P_{\xi(G_{\alpha}^{i}),i-2\leq\cdot\leq i+2}F(u(\tau))d\tau\|_{U_{\Delta}^{2}(G_{\alpha}^{i}\times\R^{2})}^{2} \\
&\phantom{=}\qquad +\sum_{i\geq j; \epsilon_{3}^{1/2}2^{i-10}<N(G_{\kappa}^{j}) \leq\epsilon_{3}^{1/2}2^{i-5}} \|\int_{t_{\kappa}^{j}}^{t}e^{i(t-\tau)\Delta}P_{\xi(G_{\kappa}^{j}),i-2\leq\cdot\leq i+2}F(u(\tau))d\tau\|_{U_{\Delta}^{2}(G_{\kappa}^{j}\times\R^{2})}^{2} \\
&\lesssim_{u} \epsilon_{2}^{3/2}.
\end{split}
\end{equation}
\end{lemma}
\begin{proof}
We first prove the estimate \eqref{eq:LTSE_NI_1}. To make precise that there are not too many intervals $G_{\alpha}^{i}\subset G_{\kappa}^{j}$ such that $N(G_{\alpha}^{i}) \gtrsim \epsilon_{3}^{1/2}2^{i}$, it is convenient first to reduce to the case where $G_{\kappa}^{j}$ is a union of small $J_{l}$ intervals (cf. pg. 3462-3463 in \cite{Dodson2016}). This technical simplification allows us to use the discretization
\begin{equation}
\int_{\bigcup J_{l}}N(t)^{3}dt \sim_{u} \sum_{l}N(J_{l}).
\end{equation}
To make the reduction, we first observe that we can write
\begin{equation}
G_{\kappa}^{j} = ((J_{1}\cup J_{2})\cap G_{\kappa}^{j}) \cup \tilde{G}_{\kappa}^{j},
\end{equation}
where $J_{1}, J_{2}$ are the two (possibly empty) small intervals which intersect $G_{\kappa}^{j}$ but are not contained in $G_{\kappa}^{j}$. Hence by proposition \ref{prop:Up_lp},
\begin{equation}
\begin{split}
&\sum_{20\leq i<j} 2^{i-j}\sum_{G_{\alpha}^{i}\subset G_{\kappa}^{j}; N(G_{\alpha}^{i})>\epsilon_{3}^{1/2}2^{i-10}} \|\int_{t_{\alpha}^{i}}^{t}e^{i(t-\tau)\Delta}P_{\xi(G_{\alpha}^{i}), i-2\leq\cdot\leq i+2}F(u(\tau))d\tau\|_{U_{\Delta}^{2}(G_{\alpha}^{i}\times\R^{2})}^{2}\\
&\phantom{=}\lesssim \sum_{20\leq i<j} 2^{i-j}\sum_{G_{\alpha}^{i}\subset G_{\kappa}^{j}; N(G_{\alpha}^{i})>\epsilon_{3}^{1/2}2^{i-10}} \|\int_{t_{\alpha}^{i}}^{t}e^{i(t-\tau)\Delta}P_{\xi(G_{\alpha}^{i}), i-2\leq\cdot\leq i+2}F(u(\tau))d\tau\|_{U_{\Delta}^{2}((J_{1}\cup J_{2})\cap G_{\alpha}^{i}\times\R^{2})}^{2} \\
&\phantom{=}\qquad +\sum_{20\leq i<j} 2^{i-j}\sum_{G_{\alpha}^{i}\subset G_{\kappa}^{j}; N(G_{\alpha}^{i})>\epsilon_{3}^{1/2}2^{i-10}} \|\int_{t_{\alpha}^{i}}^{t}e^{i(t-\tau)\Delta}P_{\xi(G_{\alpha}^{i}), i-2\leq\cdot\leq i+2}F(u(\tau))d\tau\|_{U_{\Delta}^{2}(\tilde{G}_{\alpha}^{i}\times\R^{2})}^{2}.
\end{split}
\end{equation}
By duality, the embedding $\ell_{\alpha}^{1}\subset \ell_{\alpha}^{2}$, H\"{o}lder's inequality, and the Calder\'{o}n-Zygmund theorem,
\begin{align}
&\sum_{20\leq i<j} 2^{i-j}\sum_{G_{\alpha}^{i}\subset G_{\kappa}^{j}} \|\int_{t_{\alpha}^{i}}^{t}e^{i(t-\tau)\Delta} P_{\xi(G_{\alpha}^{i}),i-2\leq\cdot\leq i+2}F(u(\tau))d\tau \|_{U_{\Delta}^{2}(G_{\alpha}^{i}\cap (J_{1}\cup J_{2})\times\R^{2})}^{2} \nonumber\\
&\phantom{=}\lesssim \sum_{20\leq i<j} 2^{i-j}\sum_{G_{\alpha}^{i}\subset G_{\kappa}^{j}} \|F(u)\|_{L_{t}^{1}L_{x}^{2}(G_{\alpha}^{i}\cap (J_{1}\cup J_{2})\times \R^{2})}^{2} \nonumber\\
&\phantom{=}\lesssim \|u\|_{L_{t}^{3}L_{x}^{6}(J_{1}\cup J_{2}\times\R^{2})}^{6} \nonumber\\
&\phantom{=}\lesssim 1.
\end{align}
Similarly,
\begin{equation}
\begin{split}
&\sum_{i\geq j; N(G_{\kappa}^{j})>\epsilon_{3}^{1/2}2^{i-10}} \|\int_{t_{\kappa}^{j}}^{t} e^{i(t-\tau)\Delta} P_{\xi(G_{\kappa}^{j}), i-2\leq\cdot\leq i+2}F(u(\tau))d\tau\|_{U_{\Delta}^{2}(G_{\kappa}^{j}\times\R^{2})}^{2}\\
&\phantom{=}\lesssim \sum_{i\geq j; N(G_{\kappa}^{j})>\epsilon_{3}^{1/2}2^{i-10}} \|\int_{t_{\kappa}^{j}}^{t} e^{i(t-\tau)\Delta}P_{\xi(G_{\kappa}^{j}), i-2\leq\cdot\leq i+2}F(u(\tau))d\tau\|_{U_{\Delta}^{2}((J_{1}\cup J_{2})\cap G_{\kappa}^{j}\times\R^{2})}^{2} \\
&\phantom{=}\qquad +\sum_{i\geq j; N(G_{\kappa}^{j})>\epsilon_{3}^{1/2}2^{i-10}} \|\int_{t_{\kappa}^{j}}^{t} e^{i(t-\tau)\Delta}P_{\xi(G_{\kappa}^{j}), i-2\leq\cdot\leq i+2}F(u(\tau))d\tau\|_{U_{\Delta}^{2}(\tilde{G}_{\kappa}^{j}\times\R^{2})}^{2}.
\end{split}
\end{equation}
Arguing similarly as before together with almost orthogonality, we have that
\begin{align}
\sum_{i\geq j} \|\int_{t_{\kappa}^{j}}^{t}e^{i(t-\tau)\Delta}P_{\xi(G_{\kappa}^{j}),i-2\leq\cdot\leq i+2}F(u(\tau))d\tau\|_{U_{\Delta}^{2}((J_{1}\cup J_{2})\cap G_{\kappa}^{j}\times\R^{2})}^{2} &\lesssim \sum_{i\geq j} \|P_{\xi(G_{\kappa}^{j}), i-2\leq\cdot\leq i+2}F(u)\|_{L_{t}^{1}L_{x}^{2}(J_{1}\cup J_{2}\times\R^{2})}^{2} \nonumber\\
&\lesssim 1.
\end{align}

Now to prove \eqref{eq:LTSE_NI_1}, it suffices to consider, for $i$ fixed, the intervals $G_{\alpha}^{i}\subset G_{\kappa}^{j}$ satisfying $N(G_{\alpha}^{i}) \geq \epsilon_{3}^{1/2}2^{i-5}$ because the contribution of the intervals satisfying $\epsilon_{3}^{1/2}2^{i-10}\leq N(G_{\alpha}^{i}) < \epsilon_{3}^{1/2}2^{i-5}$ is estimated by \eqref{eq:LTSE_NI_2}. By proposition \ref{prop:Up_lp},
\begin{equation}
\|\int_{t_{\alpha}^{i}}^{t}e^{i(t-\tau)\Delta}P_{\xi(G_{\alpha}^{i}),i-2\leq\cdot\leq i+2}F(u(\tau))d\tau\|_{U_{\Delta}^{2}(\tilde{G}_{\alpha}^{i}\times\R^{2})}^{2} \leq \sum_{J_{l}\subset G_{\kappa}^{j}} \|\int_{t_{\alpha}^{i}}^{t}e^{i(t-\tau)\Delta}P_{\xi(G_{\alpha}^{i}),i-2\leq\cdot\leq i+2}F(u(\tau))d\tau\|_{U_{\Delta}^{2}(\tilde{G}_{\alpha}^{i}\cap J_{l}\times\R^{2})}^{2}.
\end{equation}
Now observe that if $N(G_{\alpha}^{i})\geq \epsilon_{3}^{1/2}2^{i-5}$, then by the fundamental theorem of calculus and the estimate $|N'(t)| \leq 2^{-20}\epsilon_{1}^{-1/2}N(t)^{3}$, we have that $N(t)\geq \epsilon_{3}^{1/2}2^{i-6}$ for all $t\in G_{\alpha}^{i}$. Also observe that if $J_{l}\subset G_{\kappa}^{j}$ is a small interval such that $J_{l}\cap G_{\alpha}^{i}\neq \emptyset$, where $G_{\alpha}^{i} \subset G_{\kappa}^{j}$ satisfying $N(G_{\alpha}^{i}) \geq \epsilon_{3}^{1/2}2^{i-5}$, then $N(J_{l})\geq \epsilon_{3}^{1/2}2^{i-6}$. Therefore,
\begin{equation}
\begin{split}
&\sum_{J_{l}\subset G_{\kappa}^{j}} \|\int_{t_{\alpha}^{i}}^{t}e^{i(t-\tau)\Delta}P_{\xi(G_{\alpha}^{i}),i-2\leq\cdot\leq i+2}F(u(\tau))d\tau\|_{U_{\Delta}^{2}(\tilde{G}_{\alpha}^{i}\cap J_{l}\times\R^{2})}^{2} \\
&\phantom{=}= \sum_{J_{l}\subset G_{\kappa}^{j}; N(J_{l})\geq \epsilon_{3}^{1/2}2^{i-6}}\|\int_{t_{\alpha}^{i}}^{t}e^{i(t-\tau)\Delta}P_{\xi(G_{\alpha}^{i}),i-2\leq\cdot\leq i+2}F(u(\tau))d\tau\|_{U_{\Delta}^{2}(\tilde{G}_{\alpha}^{i}\cap J_{l}\times\R^{2})}^{2}.
\end{split}
\end{equation}
By duality and Plancherel's theorem,
\begin{equation}
\|\int_{t_{\alpha}^{i}}^{t}e^{i(t-\tau)\Delta}P_{\xi(G_{\alpha}^{i}),i-2\leq\cdot\leq i+2}F(u(\tau))d\tau\|_{U_{\Delta}^{2}(\tilde{G}_{\alpha}^{i}\cap J_{l}\times\R^{2})}^{2} \leq \|F(u)\|_{L_{t}^{1}L_{x}^{2}(\tilde{G}_{\alpha}^{i}\cap J_{l}\times\R^{2})}^{2}.
\end{equation}
Therefore by the embedding $\ell_{\alpha}^{1}\subset \ell_{\alpha}^{2}$, H\"{o}lder's inequality, followed by Calder\'{o}n-Zygmund theorem together with the fact that $J_{l}$ is small,
\begin{align}
&\sum_{20\leq i<j}2^{i-j}\sum_{G_{\alpha}^{i}\subset G_{\kappa}^{j}; N(G_{\alpha}^{i)}\geq\epsilon_{3}^{1/2}2^{i-5}}\sum_{J_{l}\subset G_{\kappa}^{j}; N(J_{l})\geq\epsilon_{3}^{1/2}2^{i-6}} \|F(u)\|_{L_{t}^{1}L_{x}^{2}(J_{l}\cap\tilde{G}_{\alpha}^{i}\times\R^{2})}^{2} \nonumber\\
&\phantom{=}\leq \sum_{20\leq i<j} 2^{i-j}\sum_{J_{l}\subset G_{\kappa}^{j}; N(J_{l})\geq\epsilon_{3}^{1/2}2^{i-6}}\|F(u)\|_{L_{t}^{1}L_{x}^{2}(J_{l}\times\R^{2})}^{2} \nonumber\\
&\phantom{=}\lesssim \sum_{20\leq i<j} \sum_{J_{l}\subset G_{\kappa}^{j}; N(J_{l})\geq\epsilon_{3}^{1/2}2^{i-6}} 2^{i-j}\nonumber\\
&\phantom{=} \lesssim 2^{-j}\sum_{J_{l}\subset G_{\kappa}^{j}} \epsilon_{3}^{-1/2}N(J_{l})
\end{align}
where the ultimate inequality follows from interchanging the order of summation. Since $N(J)\sim_{u}\int_{J}N(t)^{3}dt$ for all small intervals $J$, we conclude that
\begin{equation}
2^{-j}\sum_{J_{l}\subset G_{\kappa}^{j}} \epsilon_{3}^{-1/2}N(J_{l}) \lesssim_{u} 2^{-j}\epsilon_{3}^{-1/2}\int_{G_{\kappa}^{j}}N(t)^{3}dt \lesssim \epsilon_{3}^{1/2}.
\end{equation}

Again by duality,
\begin{equation}
\|\int_{t_{\kappa}^{j}}^{t}e^{i(t-\tau)\Delta}P_{\xi(G_{\kappa}^{j}),i-2\leq\cdot\leq i+2}F(u(\tau))d\tau\|_{U_{\Delta}^{2}(\tilde{G}_{\kappa}^{j}\times\R^{2})}^{2} \lesssim \|P_{\xi(G_{\kappa}^{j}),i-2\leq\cdot\leq i+2}F(u)\|_{L_{t}^{1}L_{x}^{2}(\tilde{G}_{\kappa}^{j}\times\R^{2})}^{2}.
\end{equation}
Next, observe that if $N(G_{\kappa}^{j}) \geq \epsilon_{3}^{1/2}2^{i-5}$, where $i\geq j$, then $N(t)\geq \epsilon_{3}^{1/2}2^{j-6}$ for all $t\in G_{\kappa}^{j}$. Also observe that if $J_{l}$ is a small interval such that $J_{l}\cap G_{\kappa}^{j}\neq\emptyset$, where $N(G_{\kappa}^{j}) \geq\epsilon_{3}^{1/2}2^{j-5}$, then $N(J_{l})\geq \epsilon_{3}^{1/2}2^{j-6}$. Therefore by Minkowski's and H\"{o}lder's inequalities, Plancherel's and Calder\'{o}n-Zygmund theorems, and the fact that $J_{l}$ is small,
\begin{align}
&\sum_{i\geq j; N(G_{\kappa}^{j})\geq \epsilon_{3}^{1/2}2^{i-5}} \|P_{\xi(G_{\kappa}^{j}), i-2\leq\cdot\leq i+2}F(u)\|_{L_{t}^{1}L_{x}^{2}(\tilde{G}_{\kappa}^{j}\times\R^{2})}^{2} \nonumber\\
&\phantom{=}\lesssim \paren*{\sum_{J_{l}\subset G_{\kappa}^{j}; N(J_{l})\geq \epsilon_{3}^{1/2}2^{j-6}} \int_{J_{l}}\paren*{\sum_{i\geq j; N(G_{\kappa}^{j}) > \epsilon_{3}^{1/2}2^{i-5}}\|P_{\xi(G_{\kappa}^{j}), i-2\leq\cdot\leq i+2}F(u(t))\|_{L_{x}^{2}(\R^{2})}^{2}}^{1/2}dt}^{2} \nonumber\\
&\phantom{=}\lesssim |\{J_{l} : J_{l}\subset G_{\kappa}^{j}; N(J_{l})\geq\epsilon_{3}^{1/2}2^{j-6}\}|^{2}.
\end{align}
Since
\begin{align}
|\{J_{l} : J_{l}\subset G_{\kappa}^{j}; N(J_{l})\geq\epsilon_{3}^{1/2}2^{j-6}\}| &= \sum_{J_{l}\subset G_{\kappa}^{j}; N(J_{l})\geq \epsilon_{3}^{1/2}2^{j-6}} 1 \nonumber\\
&\leq \sum_{J_{l}\subset G_{\kappa}^{j}; N(J_{l})\geq\epsilon_{3}^{1/2}2^{j-6}}\epsilon_{3}^{-1/2}2^{6-j} N(J_{l}) \nonumber\\
&\sim_{u} \epsilon_{3}^{-1/2}2^{-j}\sum_{J_{l}\subset G_{\kappa}^{j}} \int_{J_{l}}N(t)^{3}dt \nonumber\\
&\lesssim \epsilon_{3}^{1/2},
\end{align}
it follows that
\begin{equation}
\sum_{i\geq j; N(G_{\kappa}^{j}) \geq \epsilon_{3}^{1/2}2^{i-5}} \|P_{\xi(G_{\kappa}^{j}), i-2\leq\cdot\leq i+2}F(u)\|_{L_{t}^{1}L_{x}^{2}(\tilde{G}_{\kappa}^{j}\times\R^{2})}^{2} \lesssim \epsilon_{3},
\end{equation}
which completes the proof of the estimate \eqref{eq:LTSE_NI_1}.

We now prove the estimate \eqref{eq:LTSE_NI_2}. As before, we first reduce to the case where $G_{\kappa}^{j}$ is a union of small intervals. With $J_{1}, J_{2}$ defined as above, we have that
\begin{equation}
\begin{split}
&\sum_{20\leq i<j} 2^{i-j}\sum_{G_{\alpha}^{i}\subset G_{\kappa}^{j}; \epsilon_{3}^{1/2}2^{i-10}< N(G_{\alpha}^{i})\leq \epsilon_{3}^{1/2}2^{i-5}} \|\int_{t_{\alpha}^{i}}^{t}e^{i(t-\tau)\Delta} P_{\xi(G_{\alpha}^{i}), i-2\leq\cdot\leq i+2}F(u(\tau))d\tau\|_{U_{\Delta}^{2}(G_{\alpha}^{i}\times\R^{2})}^{2}\\
&\phantom{=}\lesssim \sum_{20\leq i<j} 2^{i-j}\sum_{G_{\alpha}^{i}\subset G_{\kappa}^{j}; \epsilon_{3}^{1/2}2^{i-10}< N(G_{\alpha}^{i})\leq \epsilon_{3}^{1/2}2^{i-5}} \|\int_{t_{\alpha}^{i}}^{t} e^{i(t-\tau)\Delta}P_{\xi(G_{\alpha}^{i}), i-2\leq\cdot\leq i+2}F(u(\tau))d\tau\|_{U_{\Delta}^{2}((J_{1}\cup J_{2})\cap G_{\alpha}^{i}\times\R^{2})}^{2} \\
&\phantom{=}\qquad + \sum_{20\leq i<j} 2^{i-j}\sum_{G_{\alpha}^{i}\subset G_{\kappa}^{j}; \epsilon_{3}^{1/2}2^{i-10}< N(G_{\alpha}^{i})\leq \epsilon_{3}^{1/2}2^{i-5}} \|\int_{t_{\alpha}^{i}}^{t}e^{i(t-\tau)\Delta} P_{\xi(G_{\alpha}^{i}), i-2\leq\cdot\leq i+2}F(u(\tau))d\tau\|_{U_{\Delta}^{2}(\tilde{G}_{\alpha}^{i}\times\R^{2})}^{2}.
\end{split}
\end{equation}
By duality and Strichartz estimates,
\begin{equation}
\|\int_{t_{\alpha}^{i}}^{t}e^{i(t-\tau)\Delta} P_{\xi(G_{\alpha}^{i}), i-2\leq\cdot\leq i+2}F(u(\tau))d\tau\|_{U_{\Delta}^{2}((J_{1}\cup J_{2})\cap G_{\alpha}^{i}\times\R^{2})} \lesssim \|P_{\xi(G_{\alpha}^{i}), i-2\leq\cdot\leq i+2}F(u)\|_{L_{t,x}^{4/3}((J_{1}\cup J_{2})\cap G_{\alpha}^{i}\times\R^{2})}.
\end{equation}
Hence, using the embedding $\ell_{\alpha}^{4/3}\subset \ell_{\alpha}^{2}$, we have that
\begin{align}
&\sum_{20\leq i<j} 2^{i-j}\sum_{G_{\alpha}^{i}\subset G_{\kappa}^{j};N(G_{\alpha}^{i})\leq \epsilon_{3}^{1/2}2^{i-5}} \|P_{\xi(G_{\alpha}^{i}), i-2\leq\cdot\leq i+2}F(u)\|_{L_{t,x}^{4/3}(G_{\alpha}^{i}\cap (J_{1}\cup J_{2})\times\R^{2})}^{2} \nonumber\\
&\phantom{=}\lesssim \sum_{20\leq i<j} 2^{i-j} \paren*{\sum_{G_{\alpha}^{i}\subset G_{\kappa}^{j}; N(G_{\alpha}^{i}) \leq \epsilon_{3}^{1/2}2^{i-5}} \|P_{\xi(G_{\alpha}^{i}), i-2\leq \cdot\leq i+2}F(u)\|_{L_{t,x}^{4/3}(G_{\alpha}^{i}\cap (J_{1}\cup J_{2})\times\R^{2})}^{4/3}}^{3/2} \nonumber\\
&\phantom{=}\lesssim \sum_{20 \leq i<j} 2^{i-j}\|P_{\xi(t), \geq \epsilon_{3}^{-1/2}N(t)}F(u)\|_{L_{t,x}^{4/3}(J_{1}\cup J_{2}\times\R^{2})}^{2} \nonumber\\
&\phantom{=}\lesssim \|P_{\xi(t), \geq \epsilon_{3}^{-1/2}N(t)}F(u)\|_{L_{t,x}^{4/3}(J_{1}\cup J_{2}\times\R^{2})}^{2}.
\end{align}
Now performing a near-far frequency decomposition of $u$ in $P_{\xi(t), \geq \epsilon_{3}^{-1/2}N(t}F(u(t))$, we obtain a sum of terms, each term containing at least one factor with Fourier support in the region $\{|\xi-\xi(t)|\geq 2^{-3}\epsilon_{3}^{-1/2}N(t)\}$. Therefore by H\"{o}lder's inequality and the Calder\'{o}n-Zygmund theorem
\begin{align}
\|P_{\xi(t), \geq \epsilon_{3}^{-1/2}N(t)}F(u)\|_{L_{t,x}^{4/3}(J_{1}\cup J_{2}\times\R^{2})}^{2} &\lesssim \|P_{\xi(t),\geq 2^{-3}\epsilon_{3}^{-1/2}N(t)}u\|_{L_{t}^{\infty}L_{x}^{2}(J_{1}\cup J_{2}\times\R^{2})}^{2} \|v_{1}v_{2}\|_{L_{t}^{4/3}L_{x}^{4}(J_{1}\cup J_{2}\times\R^{2})}^{2} \nonumber\\
&\lesssim \epsilon_{2}^{2}\|u\|_{L_{t}^{8/3}L_{x}^{8}(J_{1}\cup J_{2}\times\R^{2})}^{4} \nonumber\\
&\lesssim \epsilon_{2}^{2},
\end{align}
where we use that $(8/3,8)$ is an admissible pair together with the fact that $J_{1},J_{2}$ are small to obtain the ultimate inequality. Above, we have used the notation $v_{j}$ to denote a Littlewood-Paley projection of $u$ or $\bar{u}$.

Similarly,
\begin{equation}
\begin{split}
&\sum_{i\geq j; \epsilon_{3}^{1/2}2^{i-10}<N(G_{\kappa}^{j})\leq \epsilon_{3}^{1/2}2^{i-5}} \|\int_{t_{\kappa}^{j}}^{t}e^{i(t-\tau)\Delta}P_{\xi(G_{\kappa}^{j}), i-2\leq\cdot\leq i+2}F(u(\tau))d\tau\|_{U_{\Delta}^{2}(G_{\kappa}^{j}\times\R^{2})}^{2} \\
&\phantom{=}\leq \sum_{i\geq j; \epsilon_{3}^{1/2}2^{i-10}< N(G_{\kappa}^{j})\leq \epsilon_{3}^{1/2}2^{i-5}} \|\int_{t_{\kappa}^{j}}^{t}e^{i(t-\tau)\Delta}P_{\xi(G_{\kappa}^{j}), i-2\leq\cdot\leq i+2}F(u(\tau))d\tau\|_{U_{\Delta}^{2}((J_{1}\cup J_{2})\cap G_{\kappa}^{j}\times\R^{2})}^{2} \\
&\phantom{=}\qquad +\sum_{i\geq j; \epsilon_{3}^{1/2}2^{i-10} < N(G_{\kappa}^{j})\leq \epsilon_{3}^{1/2}2^{i-5}} \|\int_{t_{\kappa}^{j}}^{t}e^{i(t-\tau)\Delta}P_{\xi(G_{\kappa}^{j}), i-2\leq\cdot\leq i+2}F(u(\tau))d\tau\|_{U_{\Delta}^{2}(\tilde{G}_{\kappa}^{j}\times\R^{2})}^{2}.
\end{split}
\end{equation}
By duality,
\begin{equation}
\|\int_{t_{\kappa}^{j}}^{t}e^{i(t-\tau)\Delta}P_{\xi(G_{\kappa}^{j}), i-2\leq\cdot\leq i+2}F(u(\tau))d\tau\|_{U_{\Delta}^{2}((J_{1}\cup J_{2})\cap G_{\kappa}^{j}\times\R^{2})} \leq \|P_{\xi(G_{\kappa}^{j}),i-2\leq\cdot\leq i+2}F(u)\|_{L_{t}^{1}L_{x}^{2}((J_{1}\cup J_{2})\cap G_{\kappa}^{j}\times\R^{2})}.
\end{equation}
By Minkowski's inequality, Plancherel's theorem, and almost orthogonality, we have that
\begin{align}
&\sum_{i\geq j; \epsilon_{3}^{1/2}2^{i-10}<N(G_{\kappa}^{j})\leq \epsilon_{3}^{1/2}2^{i-5}} \|P_{\xi(G_{\kappa}^{j}),i-2\leq\cdot\leq i+2}F(u)\|_{L_{t}^{1}L_{x}^{2}((J_{1}\cup J_{2})\cap G_{\kappa}^{j}\times\R^{2})}^{2} \nonumber\\
&\phantom{=}\lesssim \paren*{\int_{(J_{1}\cup J_{2})\cap G_{\kappa}^{j}} \paren*{\sum_{i\geq j; \epsilon_{3}^{1/2}2^{i-10}<N(G_{\kappa}^{j})\leq \epsilon_{3}^{1/2}2^{i-5}} \|P_{\xi(t),i-4\leq\cdot\leq i+4}F(u(t))\|_{L_{x}^{2}(\R^{2})}^{2}}^{1/2}dt}^{2} \nonumber\\
&\phantom{=}\lesssim \|P_{\xi(t), \geq \epsilon_{3}^{-1/2}N(t)}F(u)\|_{L_{t}^{1}L_{x}^{2}((J_{1}\cup J_{2})\cap G_{\kappa}^{j}\times\R^{2}}^{2}.
\end{align}
Now by performing a near-far decomposition of $u$ in the expression $P_{\xi(t), \geq \epsilon_{3}^{-1/2}N(t)}F(u(t))$ and arguing similarly to as above, we obtain the estimate
\begin{align}
\|P_{\xi(t), \geq \epsilon_{3}^{-1/2}N(t)}F(u)\|_{L_{t}^{1}L_{x}^{2}((J_{1}\cup J_{2})\cap G_{\kappa}^{j}\times\R^{2}}^{2} &\lesssim \|P_{\xi(t), \geq 2^{-3}\epsilon_{3}^{-1/2}N(t)}u\|_{L_{t}^{9}L_{x}^{18/7}(J_{1}\cup J_{2}\times\R^{2})}^{2} \|v_{1}\|_{L_{t}^{9/4}L_{x}^{18}(J_{1}\cup J_{2}\times\R^{2})}^{2} \|v_{2}\|_{L_{t}^{9/4}L_{x}^{18}(J_{1}\cup J_{2}\times\R^{2})}^{2} \nonumber\\
&\lesssim \|P_{\xi(t), \geq 2^{-3}\epsilon_{3}^{-1/2}N(t)}u\|_{L_{t}^{9}L_{x}^{18/7}(J_{1}\cup J_{2}\times\R^{2})}^{2}.
\end{align}
So by interpolating between the admissible pairs $(\infty,2)$ and $(9/4,18)$ to get $(9,18/7)$, we obtain that
\begin{align}
\|P_{\xi(t), \geq 2^{-3}\epsilon_{3}^{-1/2}N(t)}u\|_{L_{t}^{9}L_{x}^{18/7}(J_{1}\cup J_{2}\times\R^{2})}^{2} &\leq \|P_{\xi(t),\geq 2^{-3}\epsilon_{3}^{-1/2}N(t)}u\|_{L_{t}^{\infty}L_{x}^{2}(J_{1}\cup J_{2}\times\R^{2})}^{3/2} \|P_{\xi(t), \geq 2^{-3}\epsilon_{3}^{-1/2}N(t)}u\|_{L_{t}^{9/4}L_{x}^{18}(J_{1}\cup J_{2}\times\R^{2})}^{1/2} \nonumber\\
&\lesssim \epsilon_{2}^{3/2},
\end{align}
where the ultimate inequality follows from the frequency localization property \eqref{eq:sp_frq_loc} applied to the first factor and Bernstein's lemma, together with the fact that $J_{1}, J_{2}$ are small, applied to the second factor. Hence, we have shown that
\begin{equation}
\sum_{i\geq j; \epsilon_{3}^{1/2}2^{i-10} < N(G_{\kappa}^{j})\leq \epsilon_{3}^{1/2}2^{i-5}} \|\int_{t_{\kappa}^{j}}^{t}e^{i(t-\tau)\Delta}P_{\xi(G_{\kappa}^{j}), i-2\leq\cdot\leq i+2}F(u(\tau))d\tau\|_{U_{\Delta}^{2}((J_{1}\cup J_{2})\cap G_{\kappa}^{j}\times\R^{2})}^{2} \lesssim \epsilon_{2}^{3/2}.
\end{equation}

Next, we observe that $N(G_{\alpha}^{i}) >\epsilon_{3}^{1/2}2^{i-10}$ implies that $N(t)\geq\epsilon_{3}^{1/2}2^{i-11}$ for all $t\in G_{\alpha}^{i}$, and if $J_{l}\cap G_{\alpha}^{i}\neq\emptyset$, where $G_{\alpha}^{i}\subset G_{\kappa}^{j}$ with $N(G_{\alpha}^{i})>\epsilon_{3}^{1/2}2^{i-10}$, then $N(J_{l})\geq \epsilon_{3}^{1/2}2^{i-11}$. Hence by repeating the argument  above used to obtain the estimate
\begin{equation}
\sum_{20\leq i<j} 2^{i-j}\sum_{G_{\alpha}^{i}\subset G_{\kappa}^{j}; N(G_{\alpha}^{i})\geq \epsilon_{3}^{1/2}2^{i-5}} \|\int_{t_{\alpha}^{i}}^{t}e^{i(t-\tau)\Delta}P_{\xi(G_{\alpha}^{i}), i-2\leq\cdot\leq i+2}F(u(\tau))d\tau\|_{U_{\Delta}^{2}(\tilde{G}_{\alpha}^{i}\times\R^{2})}^{2} \lesssim \epsilon_{3}^{1/2},
\end{equation}
we obtain the estimate
\begin{equation}
\sum_{20\leq i<j} 2^{i-j}\sum_{G_{\alpha}^{i}\subset G_{\kappa}^{j}; \epsilon_{3}^{1/2}2^{i-10}<N(G_{\alpha}^{i})\leq \epsilon_{3}^{1/2}2^{i-5}} \|\int_{t_{\alpha}^{i}}^{t}e^{i(t-\tau)\Delta}P_{\xi(G_{\alpha}^{i}), i-2\leq\cdot\leq i+2}F(u(\tau))d\tau\|_{U_{\Delta}^{2}(\tilde{G}_{\alpha}^{i}\times\R^{2})}^{2} \lesssim \epsilon_{3}^{1/2}.
\end{equation}

Lastly, we observe that $N(G_{\kappa}^{j}) \geq \epsilon_{3}^{1/2}2^{j-10}$ implies that $N(t) \geq \epsilon_{3}^{1/2}2^{j-11}$ for all $t\in G_{\kappa}^{j}$, and if $J_{l}$ is a small interval such that $J_{l}\cap G_{\kappa}^{j}\neq\emptyset$, where $N(G_{\kappa}^{j})\geq \epsilon_{3}^{1/2}2^{j-10}$, then $N(J_{l})\geq \epsilon_{3}^{1/2}2^{j-11}$. Hence by repeating the argument above used to obtain the estimate
\begin{equation}
\sum_{i\geq j; N(G_{\kappa}^{j})\geq \epsilon_{3}^{1/2}2^{i-5}} \|\int_{t_{\kappa}^{j}}^{t}e^{i(t-\tau)\Delta}P_{\xi(G_{\kappa}^{j}),i-2\leq\cdot\leq i+2}F(u(\tau))d\tau\|_{U_{\Delta}^{2}(\tilde{G}_{\kappa}^{j}\times\R^{2})}^{2} \lesssim \epsilon_{3},
\end{equation}
we obtain the estimate
\begin{equation}
\sum_{i\geq j; N(G_{\kappa}^{j})\geq \epsilon_{3}^{1/2}2^{i-10}} \|\int_{t_{\kappa}^{j}}^{t}e^{i(t-\tau)\Delta}P_{\xi(G_{\kappa}^{j}),i-2\leq\cdot\leq i+2}F(u(\tau))d\tau\|_{U_{\Delta}^{2}(\tilde{G}_{\kappa}^{j}\times\R^{2})}^{2} \lesssim \epsilon_{3},
\end{equation}
which completes the proof of the estimate \ref{eq:LTSE_NI_2} and therefore the proof of the lemma.
\end{proof}

\subsection{Step 4: Far frequency nonlinear estimate}\label{ssec:LTSE_S4}
Applying lemma \ref{lem:LTSE_NI_est}, we have shown that
\begin{equation}
\begin{split}
\|u\|_{X(G_{\kappa}^{j}\times\R^{2})}^{2} &\lesssim_{u} 1+\sum_{20\leq i<j}2^{i-j}\sum_{G_{\alpha}^{i}\subset G_{\kappa}^{j}; N(G_{\alpha}^{i})\leq \epsilon_{3}^{1/2}2^{i-10}}\|\int_{t_{\alpha}^{i}}^{t}P_{\xi(G_{\alpha}^{i}), i-2\leq\cdot\leq i+2}F(u(\tau))d\tau\|_{U_{\Delta}^{2}(G_{\alpha}^{i}\times\R^{2})}^{2} \\
&\phantom{=}\qquad + \sum_{i\geq j; N(G_{\kappa}^{j})\leq\epsilon_{3}^{1/2}2^{i-10}}\|\int_{t_{\kappa}^{j}}^{t}e^{i(t-\tau)\Delta}P_{\xi(G_{\kappa}^{j}), i-2\leq\cdot\leq i+2}F(u(\tau))d\tau\|_{U_{\Delta}^{2}(G_{\kappa}^{j}\times\R^{2})}^{2}
\end{split}
\end{equation}
and
\begin{equation}
\begin{split}
\|u\|_{Y(G_{\kappa}^{j}\times\R^{2})}^{2} &\lesssim_{u} \epsilon_{2}^{3/2} + \sum_{20\leq i\leq j}2^{i-j}\sum_{G_{\alpha}^{i}\subset G_{\kappa}^{j}; N(G_{\alpha}^{i})\leq \epsilon_{3}^{1/2}2^{i-10}} \|\int_{t_{\alpha}^{i}}^{t}e^{i(t-\tau)\Delta}P_{\xi(G_{\alpha}^{i}), i-2\leq\cdot\leq i+2}F(u(\tau))d\tau\|_{U_{\Delta}^{2}(G_{\alpha}^{i}\times\R^{2})}^{2} \\
&\phantom{=}\qquad+\sum_{i\geq j; N(G_{\kappa}^{j})\leq \epsilon_{3}^{1/2}2^{i-10}} \|\int_{t_{\kappa}^{j}}^{t}e^{i(t-\tau)\Delta}P_{\xi(G_{\kappa}^{j}), i-2\leq\cdot\leq i+2}F(u(\tau))d\tau\|_{U_{\Delta}^{2}(G_{\kappa}^{j}\times\R^{2})}^{2}
\end{split}
\end{equation}
for all integers $20\leq j\leq k_{0}$ and intervals $G_{\kappa}^{j} \subset [0,T]$. Our aim now is to show using the induction hypothesis that the remaining nonlinear contributions in the RHSs of the inequalities are ``small" as measured by the parameters $\epsilon_{1},\epsilon_{2},\epsilon_{3}$. We should expect this smallness because we have extracted and estimated the piece of the solution localized to the frequency ball $|\xi-\xi(t)|\lesssim N(t)$ and are now considering the pieces of the solution at frequencies much larger than $N(t)$. According to frequency localization property \eqref{eq:sp_frq_loc}, the contributions of these pieces should be small. To make this heuristic rigorous, we use an idea of \cite{Dodson2016}, which is to proceed by a bootstrap argument to close the proof of the inductive step. Thus, the goal of this subsection is to prove the following lemma.

\begin{lemma} [Bootstrap]\label{lem:LTSE_b}
The following estimate holds uniformly in $20\leq j\leq k_{0}$ and $G_{\kappa}^{j}\subset [0,T]$:
\begin{equation}
\begin{split}
&\sum_{0\leq i<j} 2^{i-j}\sum_{G_{\alpha}^{i}\subset G_{\kappa}^{j} ; N(G_{\alpha}^{i})\leq \epsilon_{3}^{1/2}2^{i-10}} \|\int_{t_{\alpha}^{i}}^{t}e^{i(t-\tau)\Delta}P_{\xi(G_{\alpha}^{i}), i-2\leq\cdot\leq i+2}F(u(\tau))d\tau\|_{U_{\Delta}^{2}(G_{\alpha}^{i}\times\R^{2})}^{2}\\
&\phantom{=}\qquad +\sum_{i \geq j; N(G_{\kappa}^{j}) \leq\epsilon_{3}^{1/2}2^{i-10}} \|\int_{t_{\kappa}^{j}}^{t}e^{i(t-\tau)\Delta}P_{\xi(G_{\kappa}^{j}), i-2\leq\cdot\leq i+2}F(u(\tau))d\tau\|_{U_{\Delta}^{2}(G_{\kappa}^{j}\times\R^{2})}^{2} \\
&\phantom{=}\lesssim _{u} \|u\|_{\tilde{Y}_{j}([0,T]\times\R^{2})}^{2}\paren*{\epsilon_{2}\|u\|_{\tilde{X}_{j}([0,T]]\times\R^{2})}^{3} + \epsilon_{2}^{1/3}\|u\|_{\tilde{X}_{j}([0,T]\times\R^{2})}^{5/3} + \paren*{\epsilon_{2}+\|u\|_{\tilde{Y}_{j}([0,T]\times\R^{2})}\paren*{1+\|u\|_{\tilde{X}_{j}([0,T]\times\R^{2})}^{4}}}^{2}}.
\end{split}
\end{equation}
\end{lemma}

We remark that it is at the stage now of proving lemma \ref{lem:LTSE_b} where our proof acquires a substantial new level of difficulty compared to Dodson's work \cite{Dodson2016}. The reason is the reliance on bilinear Strichartz estimates, which are a priori ill-suited to nonlocal nonlinearities such as the eeDS nonlinearity, as they are not permutation-invariant under frequency decomposition.

We now to proceed to the details of proving lemma \ref{lem:LTSE_b}. To avoid distinguishing between the local and nonlocal cases, and since the arguments used to treat the part of the nonlinearity which is $\E(|u|^{2})u$ are strictly more difficult than those needed to treat the part of the nonlinearity which is $\mu |u|^{2}u$, we will assume that $F(u)= \E(|u|^{2})u$ for the remainder of this subsection. We begin by performing a near-far frequency decomposition of $u$ to obtain, for indices $20\leq i<j$,
\begin{equation}\label{eq:LTSE_b_fd}
\begin{split}
P_{\xi(G_{\alpha}^{i}), i-2\leq\cdot\leq i+2}F(u(\tau)) &= P_{\xi(G_{\alpha}^{i}), i-2\leq\cdot\leq i+2}[\E(|P_{\xi(\tau), <i-10}u(\tau)|^{2})(P_{\xi(G_{\alpha}^{i}), < i-5}u(\tau))]\\
&\phantom{=}+ 2P_{\xi(G_{\alpha}^{i}), i-2\leq\cdot\leq i+2}[\E\paren*{\Re{(P_{\xi(\tau), \geq i-10}u(\tau))(\ol{P_{\xi(\tau), <i-10}u(\tau)})}}(P_{\xi(G_{\alpha}^{i}), < i-5}u(\tau))]\\
&\phantom{=}+ P_{\xi(G_{\alpha}^{i}), i-2\leq\cdot\leq i+2}[\E(|P_{\xi(\tau), \geq i-10}u(\tau)|^{2}))(P_{\xi(G_{\alpha}^{i}), <i-5}u(\tau))]\\
&\phantom{=}+ P_{\xi(G_{\alpha}^{i}), i-2\leq\cdot\leq i+2}[\E(|P_{\xi(\tau), < i-10}u(\tau)|^{2})(P_{\xi(G_{\alpha}^{i}), \geq i-5}u(\tau))]\\
&\phantom{=}+ 2P_{\xi(G_{\alpha}^{i}), i-2\leq\cdot\leq i+2}[\E\paren*{\Re{(P_{\xi(\tau), <i-10}u(\tau))(\ol{P_{\xi(\tau), \geq i-10}u})}}(P_{\xi(G_{\alpha}^{i}), \geq i-5}u(\tau))]\\
&\phantom{=}+ P_{\xi(G_{\alpha}^{i}), i-2\leq\cdot\leq i+2}[\E(|P_{\xi(\tau),\geq i-10}u(\tau)|^{2})(P_{\xi(G_{\alpha}^{i}), \geq i-5} u(\tau))]
\end{split}
\end{equation}
with an analogous decomposition for $P_{\xi(G_{\kappa}^{j}), i-2\leq\cdot\leq i+2}F(u(\tau))$, for $i\geq j$. We need to be more careful in our grouping of terms based on the number and ordering of near and far frequency factors than in \cite{Dodson2016}, as the eeDS nonlinearity is not permutation-invariant modulo complex conjugates. The most difficult case occurs when there are two near frequency factors and one far factor in the nonlinear expression, with both near frequency factors occurring inside the argument of the operator $\E$. We split the RHS of \eqref{eq:LTSE_b_fd} into two groups of terms: $\mathrm{Easy}$ and $\mathrm{Hard}$. Terms in $\mathrm{Easy}$ contain two far frequency factors, and we can estimate them with H\"{o}lder's inequality, linear Strichartz estimates, and the frequency localization property \eqref{eq:sp_frq_loc}. Terms in $\mathrm{Hard}$ contain one far frequency factor and two near frequency factors, and we will have to invest significantly more effort, in particular, relying on three improved bilinear Strichartz estimates specific to our setting, in order to estimate them.

For integers $20\leq i\leq j$ and intervals $G_{\alpha}^{i}\subset G_{\kappa}^{j}\subset [0,T]$, define
\begin{align} %Definition of easy case
	\mathrm{Easy}_{G_{\alpha}^{i}} &\coloneqq P_{\xi(G_{\alpha}^{i}), i-2\leq\cdot\leq i+2}[\E(|P_{\xi(\tau), >i-10}u|^{2})(P_{\xi(G_{\alpha}^{i}), <i-5}u)] \nonumber\\
	&\phantom{=}+ 2P_{\xi(G_{\alpha}^{i}), i-2\leq\cdot\leq i+2}[\E\paren*{\Re{(P_{\xi(\tau), <i-10}u)(\ol{P_{\xi(\tau), \geq i-10}u})}}(P_{\xi(G_{\alpha}^{i}), \geq i-5}u)] \nonumber\\
	&\phantom{=}+ P_{\xi(G_{\alpha}^{i}), i-2\leq\cdot\leq i+2}[\E(|P_{\xi(\tau),\geq i-10}u|^{2})P_{\xi(G_{\alpha}^{i}), \geq i-5} u] \nonumber\\
	&\eqqcolon \mathrm{Easy}_{G_{\alpha}^{i},1} + \mathrm{Easy}_{G_{\alpha}^{i},2}+\mathrm{Easy}_{G_{\alpha}^{i},3}
\end{align}
and
\begin{align} %Definition of hard case
	\mathrm{Hard}_{G_{\alpha}^{i}} &\coloneqq P_{\xi(G_{\alpha}^{i}), i-2\leq\cdot\leq i+2}[\E(|P_{\xi(\tau), <i-10}u|^{2})(P_{\xi(G_{\alpha}^{i}), \geq i-5}u)] \nonumber\\
	&\phantom{=}+ 2P_{\xi(G_{\alpha}^{i}), i-2\leq\cdot\leq i+2}[\E\paren*{\Re{(P_{\xi(\tau), \geq i-10}u)(\ol{P_{\xi(\tau), <i-10}u})}}(P_{\xi(G_{\alpha}^{i}), < i-5}u)]\nonumber\\
	&\eqqcolon \mathrm{Hard}_{G_{\alpha}^{i},1}+\mathrm{Hard}_{G_{\alpha}^{i},2}.
\end{align}
For $i\geq j$, we define $\mathrm{Easy}_{i}$ and $\mathrm{Hard}_{i}$ analogously, replacing $G_{\alpha}^{i}$ with $G_{\kappa}^{j}$ above. We first dispense with the easy case.

\begin{lemma}[Easy estimate]\label{lem:LTSE_b_easy}
The following estimate holds uniformly in $20\leq j\leq k_{0}$ and $G_{\kappa}^{j}\subset [0,T]$:
\begin{equation}
\begin{split}
&\sum_{20\leq i < j}2^{i-j}\sum_{G_{\alpha}^{i}\subset G_{\kappa}^{j}; N(G_{\alpha}^{i})\leq \epsilon_{3}^{1/2}2^{i-5}} \|\int_{t_{\alpha}^{i}}^{t}e^{i(t-\tau)\Delta}\mathrm{Easy}_{G_{\alpha}^{i}}(\tau)d\tau\|_{U_{\Delta}^{2}(G_{\alpha}^{i}\times\R^{2})}^{2} \\
&\phantom{=}\qquad + \sum_{i\geq j; N(G_{\kappa}^{j})\leq \epsilon_{3}^{1/2}2^{i-10}} \|\int_{t_{\kappa}^{j}}^{t}e^{i(t-\tau)\Delta}\mathrm{Easy}_{i}(\tau)d\tau\|_{U_{\Delta}^{2}(G_{\kappa}^{j}\times\R^{2})}^{2}\\
&\phantom{=}\lesssim \paren*{\epsilon_{2}\|u\|_{\tilde{X}_{j}([0,T]\times\R^{2})}^{3} + \epsilon_{2}^{1/3} \|u\|_{\tilde{X}_{j}([0,T]\times\R^{2})}^{5/3} }\|u\|_{\tilde{Y}_{j}([0,T]\times\R^{2})}^{2}.
\end{split}
\end{equation}
\end{lemma}
\begin{proof}
For $20\leq i<j$, we first decompose $\mathrm{Easy}_{G_{\alpha}^{i}}$ by
\begin{equation}
\mathrm{Easy}_{G_{\alpha}^{i}} = \mathrm{Easy}_{G_{\alpha}^{i},1} + \mathrm{Easy}_{G_{\alpha}^{i},2} + \mathrm{Easy}_{G_{\alpha}^{i},3},
\end{equation}
and for $i\geq j$, we decompose $\mathrm{Easy}_{i}$ by
\begin{equation}
\mathrm{Easy}_{i} = \mathrm{Easy}_{i,1} + \mathrm{Easy}_{i,2} + \mathrm{Easy}_{i,3}.
\end{equation}
We only present the details for the contributions of $\mathrm{Easy}_{G_{\alpha}^{i},1}, \mathrm{Easy}_{i,1}, \mathrm{Easy}_{G_{\alpha}^{i},3}, \mathrm{Easy}_{i,3}$, as one can treat the contributions of the remaining terms by similar arguments.

\begin{description}[leftmargin=*]
\item[Estimate for $\mathrm{Easy}_{G_{\alpha}^{i},3}$:]
For $G_{\alpha}^{i}\subset G_{\kappa}^{j}$ fixed, we have by duality that
\begin{align}
\|\int_{t_{\alpha}^{i}}^{t}e^{i(t-\tau)\Delta}\mathrm{Easy}_{G_{\alpha}^{i},3}(\tau)d\tau\|_{U_{\Delta}^{2}(G_{\alpha}^{i}\times\R^{2})} &= \sup_{\|v\|_{V_{\Delta}^{2}(G_{\alpha}^{i}\times\R^{2})}=1} \int_{G_{\alpha}^{i}} \ipp{v(t),P_{\xi(G_{\alpha}^{i}), i-2\leq\cdot\leq i+2}[\E(|P_{\xi(t),>i-10}u(t)|^{2}) (P_{\xi(G_{\alpha}^{i}), \geq i-5}u(t))]} dt \nonumber\\
&= \sup_{\|v\|_{V_{\Delta}^{2}(G_{\alpha}^{i}\times\R^{2})}=1} \int_{G_{\alpha}^{i}} \ipp{P_{\xi(G_{\alpha}^{i}), i-2\leq\cdot\leq i+2}v(t),\E(|P_{\xi(t),>i-10}u(t)|^{2}) (P_{\xi(G_{\alpha}^{i}), \geq i-5}u(t))} dt
\end{align}
Since $\|v\|_{V_{\Delta}^{2}(G_{\alpha}^{i}\times\R^{2})}=1$ implies that $\|P_{\xi(G_{\alpha}^{i}),i-2\leq\cdot\leq i+2}v\|_{V_{\Delta}^{2}(G_{\alpha}^{i}\times\R^{2})} \leq 1$ by Plancherel's theorem, we may assume without loss of generality that $v$ has Fourier support in the dyadic annulus $A(\xi(G_{\alpha}^{i}), i-3, i+3)$ and satisfies $\|v\|_{V_{\Delta}^{2}(G_{\alpha}^{i}\times\R^{2})}\leq 1$. By Cauchy-Schwarz, triangle inequality, Plancherel's theorem, followed by the bilinear Strichartz estimate of proposition \ref{prop:bs_i}, we have that
\begin{align}
\int_{G_{\alpha}^{i}} \langle{v(t),\E(|P_{\xi(t),>i-10}u(t)|^{2})P_{\xi(G_{\alpha}^{i}), \geq i-5}u(t)}\rangle dt &\lesssim \|vP_{\xi(G_{\alpha}^{i}), \geq i-5}u\|_{L_{t,x}^{2}(G_{\alpha}^{i}\times\R^{2})} \|\E(|P_{\xi(t),>i-10}u|^{2})\|_{L_{t,x}^{2}(G_{\alpha}^{i}\times\R^{2})} \nonumber\\
&\leq \sum_{l\geq i-5}\|vP_{\xi(G_{\alpha}^{i}), l}u\|_{L_{t,x}^{2}(G_{\alpha}^{i}\times\R^{2})} \|P_{\xi(t),>i-10}u\|_{L_{t,x}^{4}(G_{\alpha}^{i}\times\R^{2})}^{2} \nonumber\\
&\lesssim \sum_{l\geq i-5}2^{(i-l){\frac{1}{2}}^{-}}\|P_{\xi(G_{\alpha}^{i}), l}u\|_{U_{\Delta}^{2}(G_{\alpha}^{i}\times\R^{2})}  \|P_{\xi(t),>i-10}u\|_{L_{t,x}^{4}(G_{\alpha}^{i}\times\R^{2})}^{2}.
\end{align}
Now interpolating between the admissible pairs $(\infty,2)$ and $(3,6)$ to get $(4,4)$, we have the estimate
\begin{align}
\|P_{\xi(t),>i-10}u\|_{L_{t,x}^{4}(G_{\alpha}^{i}\times\R^{2})}^{2} \leq \|P_{\xi(t),>i-10}u\|_{L_{t}^{\infty}L_{x}^{2}(G_{\alpha}^{i}\times\R^{2})}^{1/2} \|P_{\xi(t),>i-10}u\|_{L_{t}^{3}L_{x}^{6}(G_{\alpha}^{i}\times\R^{2})}^{3/2} \lesssim \epsilon_{2}^{1/2} \|u\|_{X(G_{\alpha}^{i}\times\R^{2})}^{3/2},
\end{align}
where we use the condition $N(G_{\alpha}^{i}) \leq \epsilon_{3}^{1/2}2^{i-5}$ together with the frequency localization property \eqref{eq:sp_frq_loc} and lemma \ref{lem:lohi_embed} to obtain the ultimate inequality. So by H\"{o}lder's inequality in $l$, we have that
\begin{equation}
\sum_{l\geq i-5}2^{(i-l){\frac{1}{2}}^{-}}\|P_{\xi(G_{\alpha}^{i}), l}u\|_{U_{\Delta}^{2}(G_{\alpha}^{i}\times\R^{2})}  \|P_{\xi(t),>i-10}u\|_{L_{t,x}^{4}(G_{\alpha}^{i}\times\R^{2})}^{2} \lesssim \epsilon_{2}^{1/2} \|u\|_{X(G_{\alpha}^{i}\times\R^{2})}^{3/2}\paren*{\sum_{l\geq i-5}2^{(i-l)\frac{1}{2}^{-}} \|P_{\xi(G_{\alpha}^{i}),l}u\|_{U_{\Delta}^{2}(G_{\alpha}^{i}\times\R^{2})}^{2}}^{1/2}.
\end{equation}

We proceed to sum over the intervals $G_{\alpha}^{i}$ and integers $20\leq i<j$. Observe from the dyadic structure that for each $0\leq l\leq j$, the interval $G_{\kappa}^{j}$ has $2^{j-l}$ children $G_{\beta}^{l}$. Similarly, for each $0\leq i\leq l$, each interval $G_{\beta}^{l}$ has $2^{l-i}$ children $G_{\alpha}^{i}$, and each interval $G_{\alpha}^{i}$ has a unique parent $G_{\beta(\alpha)}^{l}$. Additionally, if an interval $G_{\alpha}^{i}$ with $N(G_{\alpha}^{i})\leq \epsilon_{3}^{1/2}2^{i-5}$ intersects an interval $G_{\beta}^{l}$, then by the fundamental theorem of calculus and the estimate $|N'(t)|\leq 2^{-20}\epsilon_{1}^{-1/2}N(t)^{3}$, we have that $N(G_{\beta}^{l})\leq \epsilon_{3}^{1/2}2^{l-5}$. Similarly, if $G_{\alpha}^{i} \subset G_{\kappa}^{j}$ where $N(G_{\alpha}^{i})\leq \epsilon_{3}^{1/2}2^{i-5}$, then $N(G_{\kappa}^{j})\leq \epsilon_{3}^{1/2}2^{j-5}$. Hence, we have that
\begin{align}
&\sum_{20\leq i<j}2^{i-j}\sum_{G_{\alpha}^{i}\subset G_{\kappa}^{j}; N(G_{\alpha}^{i})\leq \epsilon_{3}^{1/2}2^{i-5}} \sum_{l\geq i-5}2^{(i-l){\frac{1}{2}}^{-}} \|P_{\xi(G_{\alpha}^{i}), l}u\|_{U_{\Delta}^{2}(G_{\alpha}^{i}\times\R^{2})}^{2}\nonumber\\
&\phantom{=}\lesssim \sum_{20\leq i<j}2^{i-j}\sum_{G_{\alpha}^{i}\subset G_{\kappa}^{j}; N(G_{\alpha}^{i})\leq \epsilon_{3}^{1/2}2^{i-5}} \sum_{i-5\leq l \leq i} \|P_{\xi(G_{\alpha}^{i}), l}u\|_{U_{\Delta}^{2}(G_{\alpha}^{i}\times\R^{2})}^{2} + \sum_{l>i}2^{(i-l){\frac{1}{2}}^{-}}\|P_{\xi(G_{\alpha}^{i}), l}u\|_{U_{\Delta}^{2}(G_{\alpha}^{i}\times\R^{2})}^{2} \nonumber\\
&\phantom{=}\lesssim \sum_{20\leq i<j}2^{i-j} \sum_{i-5\leq l \leq i} \sum_{G_{\beta}^{l}\subset G_{\kappa}^{j}; N(G_{\beta}^{l}) \leq \epsilon_{3}^{1/2}2 ^{l-5}} \|P_{\xi(G_{\beta}^{l}), l-2\leq\cdot\leq l+2}u\|_{U_{\Delta}^{2}(G_{\beta}^{l}\times\R^{2})}^{2} \nonumber\\
&\phantom{=}\qquad +\sum_{20\leq i<j}2^{i-j}\sum_{i<l<j}\sum_{G_{\beta}^{l}\subset G_{\kappa}^{j}; N(G_{\beta}^{l}) \leq \epsilon_{3}^{1/2}2^{l-5}}2^{(i-l){\frac{1}{2}}^{-}+(l-i)}  \|P_{\xi(G_{\beta}^{l}), l-2\leq\cdot\leq l+2}u\|_{U_{\Delta}^{2}(G_{\beta}^{l}\times\R^{2})}^{2} \nonumber\\
&\phantom{=}\qquad +\sum_{20\leq i<j}\sum_{l\geq j; N(G_{\kappa}^{j})\leq\epsilon_{3}^{1/2}2^{l-5}} 2^{(i-l){\frac{1}{2}}^{-}}  \|P_{\xi(G_{\kappa}^{j}), l-2\leq\cdot\leq l+2}u\|_{U_{\Delta}^{2}(G_{\kappa}^{j}\times\R^{2})}^{2} \nonumber\\
&\phantom{=}\eqqcolon \mathrm{Term}_{1} + \mathrm{Term}_{2} + \mathrm{Term}_{3}.
\end{align}
We consider each of the $\mathrm{Term}_{j}$. For $\mathrm{Term}_{1}$, we interchange the order of the $i$ and $l$ summations to obtain
\begin{align}
\mathrm{Term}_{1} &= \sum_{15\leq l<j}\sum_{l\leq i\leq \min\{l+5,j-1\}}2^{i-j}\sum_{G_{\beta}^{l}\subset G_{\kappa}^{j}; N(G_{\beta}^{l})\leq \epsilon_{3}^{1/2}2^{l-5}}\|P_{\xi(G_{\beta}^{l}), l-2\leq\cdot\leq l+2}u\|_{U_{\Delta}^{2}(G_{\beta}^{l}\times\R^{2})}^{2} \nonumber\\
&\lesssim \sum_{0\leq l<j} 2^{l-j}\sum_{G_{\beta}^{l}\subset G_{\kappa}^{j}; N(G_{\beta}^{l})\leq \epsilon_{3}^{1/2}2^{l-5}}\|P_{\xi(G_{\beta}^{l}), l-2\leq\cdot\leq l+2}u\|_{U_{\Delta}^{2}(G_{\beta}^{l}\times\R^{2})}^{2} \nonumber\\
&\leq \|u\|_{\tilde{Y}_{j}([0,T]\times\R^{2})}^{2},
\end{align}
where the ultimate inequality follows from the definition of the $\tilde{Y}_{j}$ norm. Similarly, for $\mathrm{Term}_{2}$, we have that
\begin{align}
\mathrm{Term}_{2} &= \sum_{20<l<j}2^{l-j} \sum_{20\leq i<\min\{j,l\}} 2^{(i-l){\frac{1}{2}}^{-}} \sum_{G_{\beta}^{l}\subset G_{\kappa}^{j}; N(G_{\beta}^{l}) \leq \epsilon_{3}^{1/2}2^{l-5}} \|P_{\xi(G_{\beta}^{l}), l-2\leq\cdot\leq l+2}u\|_{U_{\Delta}^{2}(G_{\beta}^{l}\times\R^{2})}^{2} \nonumber\\
&\lesssim \sum_{20<l<j}2^{l-j} \sum_{G_{\beta}^{l}\subset G_{\kappa}^{j}; N(G_{\beta}^{l}) \leq \epsilon_{3}^{1/2}2^{l-5}} \|P_{\xi(G_{\beta}^{l}), l-2\leq\cdot\leq l+2}u\|_{U_{\Delta}^{2}(G_{\beta}^{l}\times\R^{2})}^{2} \nonumber\\
&\leq \|u\|_{\tilde{Y}_{j}([0,T]\times\R^{2})}^{2},
\end{align}
and for $\mathrm{Term}_{3}$,
\begin{align}
\mathrm{Term}_{3} =\sum_{l\geq j; N(G_{\kappa}^{j})\leq \epsilon_{3}^{1/2}2^{l-5}} \sum_{20\leq i<j} 2^{(i-l){\frac{1}{2}}^{-}}  \|P_{\xi(G_{\kappa}^{j}), l-2\leq\cdot\leq l+2}u\|_{U_{\Delta}^{2}(G_{\kappa}^{j}\times\R^{2})}^{2} \lesssim \|u\|_{\tilde{Y}_{j}([0,T]\times\R^{2})}^{2}.
\end{align}
This last estimate completes the analysis for the contribution of $\mathrm{Easy}_{G_{\alpha}^{i},3}$, and we have shown that
\begin{equation}
\sum_{20\leq i<j}\sum_{G_{\alpha}^{i}\subset G_{\kappa}^{j}; N(G_{\alpha}^{i})\leq\epsilon_{3}^{1/2}2^{i-5}}\|\int_{t_{\alpha}^{i}}^{t}e^{i(t-\tau)\Delta}\mathrm{Easy}_{G_{\alpha}^{i},3}(\tau)d\tau\|_{U_{\Delta}^{2}(G_{\alpha}^{i}\times\R^{2})}^{2} \lesssim \epsilon_{2}\|u\|_{\tilde{X}_{j}([0,T]\times\R^{2})}^{3} \|u\|_{\tilde{Y}_{j}([0,T]\times\R^{2})}^{2}.
\end{equation}

\item[Estimate for $\mathrm{Easy}_{i,3}$:]
By repeating the arguments above, we have that for $i\geq j$,
\begin{equation}
 \|\int_{t_{\kappa}^{j}}^{t}e^{i(t-\tau)\Delta}\mathrm{Easy}_{i,3}(\tau)d\tau\|_{U_{\Delta}^{2}(G_{\kappa}^{j}\times\R^{2})}^{2} \lesssim \epsilon_{2} \|u\|_{\tilde{X}_{j}([0,T]\times\R^{2})}^{3}\sum_{l\geq i-5} 2^{(i-l){\frac{1}{2}}^{-}} \|P_{\xi(G_{\kappa}^{j}), l}u\|_{U_{\Delta}^{2}(G_{\kappa}^{j}\times\R^{2})}^{2}.
\end{equation}
Proceeding as before,
\begin{align}
\sum_{i\geq j; N(G_{\kappa}^{j}) \leq \epsilon_{3}^{1/2}2^{i-10}} \sum_{l\geq i-5}2^{(i-l){\frac{1}{2}}^{-}} \|P_{\xi(G_{\kappa}^{j}), l}u\|_{U_{\Delta}^{2}(G_{\kappa}^{j}\times\R^{2})}^{2} &= \sum_{i\geq j; N(G_{\kappa}^{j}) \leq \epsilon_{3}^{1/2}2^{i-10}}\sum_{i-5\leq l<j} 2^{(i-l){\frac{1}{2}}^{-}} \|P_{\xi(G_{\kappa}^{j}), l}u\|_{U_{\Delta}^{2}(G_{\kappa}^{j}\times\R^{2})}^{2} \nonumber\\
&\phantom{=} +\sum_{i\geq j; N(G_{\kappa}^{j}) \leq \epsilon_{3}^{1/2}2^{i-10}}\sum_{l\geq \max\{j,i-5\}} 2^{(i-l){\frac{1}{2}}^{-}} \|P_{\xi(G_{\kappa}^{j}), l}u\|_{U_{\Delta}^{2}(G_{\kappa}^{j}\times\R^{2})}^{2} \nonumber\\
&\eqqcolon \mathrm{Term}_{1}+\mathrm{Term}_{2}.
\end{align}
Interchanging the order of the $i$ and $l$ summations, we obtain that
\begin{align}
\mathrm{Term}_{1} &\lesssim \sum_{j-5\leq l<j}\sum_{j<i\leq l+5; N(G_{\kappa}^{j}) \leq \epsilon_{3}^{1/2}2^{i-10}}2^{(i-l){\frac{1}{2}}^{-}} \|P_{\xi(G_{\kappa}^{j}), l}u\|_{U_{\Delta}^{2}(G_{\kappa}^{j}\times\R^{2})}^{2} \nonumber\\
&\lesssim \sum_{l\geq j-5} 2^{l-j} \sum_{G_{\beta}^{l}\subset G_{\kappa}^{j}; N(G_{\beta}^{l}) \leq \epsilon_{3}^{1/2}2^{l-5}} \|P_{\xi(G_{\beta}^{l}), l-2\leq\cdot\leq l+2}u\|_{U_{\Delta}^{2}(G_{\beta}^{l}\times\R^{2})}^{2} \nonumber\\
&\lesssim \|u\|_{\tilde{Y}_{j}([0,T]\times\R^{2})}^{2}
\end{align}
and
\begin{align}
\mathrm{Term}_{2} &\lesssim \sum_{l\geq j}\sum_{j\leq i\leq l+5; N(G_{\kappa}^{j}) \leq \epsilon_{3}^{1/2}2^{i-10}} 2^{(i-l){\frac{1}{2}}^{-}}\|P_{\xi(G_{\kappa}^{j}), l-2\leq\cdot\leq l+2}u\|_{U_{\Delta}^{2}(G_{\kappa}^{j}\times\R^{2})}^{2} \nonumber\\
&\lesssim \sum_{l\geq j; N(G_{\kappa}^{j})\leq \epsilon_{3}^{1/2}2^{l-5}} \|P_{\xi(G_{\kappa}^{j}), l-2\leq\cdot\leq l+2}u\|_{U_{\Delta}^{2}(G_{\kappa}^{j}\times\R^{2})}^{2} \nonumber\\
&\leq \|u\|_{\tilde{Y}_{j}([0,T]\times\R^{2})}^{2}.
\end{align}
This last estimate completes the proof of the estimate for the contribution of $\mathrm{Easy}_{i,3}$ for $i\geq j$, and we have shown that
\begin{equation}
\sum_{i\geq j; N(G_{\kappa}^{j})\leq \epsilon_{3}^{1/2}2^{i-10}} \|\int_{t_{\kappa}^{j}}^{t}e^{i(t-\tau)\Delta}\mathrm{Easy}_{i}(\tau)d\tau\|_{U_{\Delta}^{2}(G_{\kappa}^{j}\times\R^{2})}^{2} \lesssim \epsilon_{2}\|u\|_{\tilde{X}_{j}([0,T]\times\R^{2})}^{3} \|u\|_{\tilde{Y}_{j}([0,T]\times\R^{2})}^{2}.
\end{equation}

\item[Estimate for $\mathrm{Easy}_{G_{\alpha}^{i},1}$:]
The argument here is similar to as before, but we need to be careful about pairing near and far frequency factors since both far frequency factors are inside the argument of $\E$. Fortunately though, Plancherel's theorem shows that for $v\in V_{\Delta}^{2}(G_{\alpha}^{i}\times\R^{2})$ with spatial Fourier support in the dyadic annulus $A(\xi(G_{\alpha}^{i}), i-3, i+3)$, we have that
\begin{equation}
\ipp{v, \E(|P_{\xi(t), >i-10}u|^{2})(P_{\xi(G_{\alpha}^{i}), < i-5}u)} = \ipp{v(\ol{P_{\xi(G_{\alpha}^{i}), < i-5}u}),  P_{i-4\leq\cdot\leq i+4}\E(|P_{\xi(t),>i-10}u|^{2})}.
\end{equation}
We frequency decompose the expression $P_{i-4\leq\cdot\leq i+4}\E(|P_{\xi(t),>i-10}u|^{2})$ by
\begin{equation}
\begin{split}
P_{i-4\leq\cdot\leq i+4}\E(|P_{\xi(t),>i-10}u|^{2}) &= P_{i-4\leq\cdot\leq i+4}\E\paren*{|P_{\xi(t),i-10<\cdot\leq i+10}u|^{2}} \\
&\phantom{=}+2P_{i-4\leq\cdot\leq i+4}\E\paren*{\Re{(P_{\xi(t),i-10<\cdot\leq i+10)}u)(\ol{P_{\xi(t),>i+10}u)}}} \\
&\phantom{=}+P_{i-4\leq\cdot\leq i+4}\E\paren*{|P_{\xi(t),>i+10}u|^{2}}.
\end{split}
\end{equation}
We only consider the last term in the RHS, as it is the most difficult case. Fourier support analysis shows that the two factors inside of $\E$ must be supported at comparable frequencies. Hence, we may write
\begin{equation}
P_{i-4\leq\cdot\leq i+4}\E\paren*{|P_{\xi(t),>i+10}u|^{2}} = \sum_{l>i+10}\sum_{|l-l'| \leq 2} P_{i-4\leq\cdot\leq i+4}\E\paren*{(P_{\xi(t), l}u)(\ol{P_{\xi(t),l'}u})}.
\end{equation}
Since the symbol of $\E$ belongs to $C^{\infty}(\R^{2}\setminus\{0\})$ and is homogeneous of degree zero, the kernel $\mathcal{K}_{i-4\leq \cdot\leq i+4}$ of the operator $P_{i-4\leq\cdot\leq i+4}\E$ is Schwartz and $\sup_{i\geq 20}\|K_{i-4\leq\cdot\leq i+4}\|_{L^{1}(\R^{2})} \lesssim 1$. Therefore, we can use Minkowski's inequality to write
\begin{equation}
\begin{split}
&\int_{G_{\alpha}^{i}} |\ipp{v(t)(\ol{P_{\xi(G_{\alpha}^{i}), < i-5}u(t)}),  P_{i-4\leq\cdot\leq i+4}\E\paren*{(P_{\xi(t),l}u(t))(\ol{P_{\xi(t),l'}u(t)})}}|dt \\
&\phantom{=}\leq \int_{\R^{2}} |\mathcal{K}_{i-4\leq\cdot\leq i+4}(z)|\|v(\ol{P_{\xi(G_{\alpha}^{i}), < l-5}u}) (P_{\xi(t),l}\tau_{z}u)(\ol{P_{\xi(t),l'}\tau_{z}u})\|_{L_{t,x}^{1}(G_{\alpha}^{i}\times\R^{2})}dz.
\end{split}
\end{equation}
By H\"{o}lder's inequality and mass conservation,
\begin{align}
\|v(\ol{P_{\xi(G_{\alpha}^{i}), < l-5}u}) (P_{\xi(t),l}\tau_{z}u)(\ol{P_{\xi(t),l'}\tau_{z}u})\|_{L_{t,x}^{1}(G_{\alpha}^{i}\times\R^{2})} &\lesssim \|v(P_{\xi(t),l}\tau_{z}u)(\ol{P_{\xi(t),l'}\tau_{z}u})\|_{L_{t}^{1}L_{x}^{2}(G_{\alpha}^{i}\times\R^{2})} \nonumber\\
&\leq \|vP_{\xi(t),l}\tau_{z}u\|_{L_{t}^{3/2}L_{x}^{3}(G_{\alpha}^{i}\times\R^{2})} \|P_{\xi(t),l'}\tau_{z}u\|_{L_{t}^{3}L_{x}^{6}(G_{\alpha}^{i}\times\R^{2})}.
\end{align}
By interpolating between the admissible pairs $(\infty,2)$ and $(5/2,10)$ to get $(3,6)$, then using the frequency localization property \eqref{eq:sp_frq_loc}, we see that
\begin{equation}
\|P_{\xi(t),l'}\tau_{z}u\|_{L_{t}^{3}L_{x}^{6}(G_{\alpha}^{i}\times\R^{2})} \leq \epsilon_{2}^{1/6} \|P_{\xi(t),l'}u\|_{L_{t}^{5/2}L_{x}^{10}(G_{\alpha}^{i}\times\R^{2})}^{5/6} \lesssim \epsilon_{2}^{1/6} \|u\|_{\tilde{X}_{j}([0,T]\times\R^{2})}^{5/6},
\end{equation}
where the ultimate inequality follows from lemma \ref{lem:lohi_embed}. Next, interpolating between $(2,2)$ and $(5/4,5)$ to get $(3/2,3)$, followed by applying the bilinear Strichartz estimate of proposition \ref{prop:bs_i}, we see that
\begin{align}
\|vP_{\xi(t),l}\tau_{z}u\|_{L_{t}^{3/2}L_{x}^{3}(G_{\alpha}^{i}\times\R^{2})} &\leq \|vP_{\xi(t),l}\tau_{z}u\|_{L_{t,x}^{2}(G_{\alpha}^{i}\times\R^{2})}^{4/9} \|vP_{\xi(t),l}\tau_{z}u\|_{L_{t}^{5/4}L_{x}^{5}(G_{\alpha}^{i}\times\R^{2})}^{5/9} \nonumber\\
&\lesssim 2^{(i-l)\frac{2}{9}^{-}} \|P_{\xi(G_{\alpha}^{i}),l-2\leq\cdot l+2}u\|_{U_{\Delta}^{2}(G_{\alpha}^{i}\times\R^{2})},
\end{align}
where we also use H\"{o}lder's and Minkowski's inequalities, Strichartz estimates, and the embedding $V_{\Delta}^{2}\subset U_{\Delta}^{5/2}$ to obtain the ultimate inequality. Since our final estimates are uniform in the translation parameter $z$, we see that
\begin{equation}
\begin{split}
&\int_{\R^{2}} |\mathcal{K}_{i-4\leq\cdot\leq i+4}(z)|\|v(\ol{P_{\xi(G_{\alpha}^{i}), < l-5}u}) (P_{\xi(t),l}\tau_{z}u)(\ol{P_{\xi(t),l'}\tau_{z}u})\|_{L_{t,x}^{1}(G_{\alpha}^{i}\times\R^{2})}dz \\
&\phantom{=}\lesssim \epsilon_{2}^{1/6}\|u\|_{\tilde{X}_{j}([0,T]\times\R^{2})}^{5/6}2^{(i-l)\frac{2}{9}-}\|P_{\xi(G_{\alpha}^{i}),l-2\leq\cdot\leq l+2}u\|_{U_{\Delta}^{2}(G_{\alpha}^{i}\times\R^{2})}.
\end{split}
\end{equation}
Now summing over $l,l'$, we obtain that
\begin{align}
&\sum_{l>i+10}\sum_{|l-l'|\leq 2} \epsilon_{2}^{1/6}\|u\|_{\tilde{X}_{j}([0,T]\times\R^{2})}^{5/6}2^{(i-l)\frac{2}{9}-}\|P_{\xi(G_{\alpha}^{i}),l-2\leq\cdot\leq l+2}u\|_{U_{\Delta}^{2}(G_{\alpha}^{i}\times\R^{2})} \nonumber\\
&\phantom{=}\lesssim \sum_{l\geq i+10} 2^{(i-l)\frac{2}{9}^{-}} \epsilon_{2}^{1/6}\|u\|_{\tilde{X}_{j}([0,T]\times\R^{2})}^{5/6}\|P_{\xi(G_{\alpha}^{i}),l-2\leq\cdot\leq l+2}u\|_{U_{\Delta}^{2}(G_{\alpha}^{i}\times\R^{2})} \nonumber\\
&\phantom{=} \lesssim \epsilon_{2}^{1/6}\|u\|_{\tilde{X}_{j}([0,T]\times\R^{2})}^{5/6}\paren*{\sum_{l\geq i+10} 2^{\frac{(i-l)}{3}} \|P_{\xi(G_{\alpha}^{i}),l-2\leq\cdot l+2}u\|_{U_{\Delta}^{2}(G_{\alpha}^{i}\times\R^{2})}^{2}}^{1/2},
\end{align}
where the ultimate inequality follows from Cauchy-Schwarz in $l$. Taking the sum $\sum_{20\leq i<j}2^{i-j}\sum_{G_{\alpha}^{i}\subset G_{\kappa}^{j}; N(G_{\alpha}^{i})\leq \epsilon_{3}^{1/2}2^{i-5}}(\cdot)^{2}$ of the last expression and proceeding as before in the case of $\mathrm{Easy}_{G_{\alpha}^{i}, 3}$, we obtain the final estimate
\begin{equation}
\sum_{20\leq i<j}2^{i-j} \sum_{G_{\alpha}^{i}\subset G_{\kappa}^{j}; N(G_{\alpha}^{i})\leq\epsilon_{3}^{1/2}2^{i-5}} \|\int_{t_{\alpha}^{i}}^{t}e^{i(t-\tau)\Delta}\mathrm{Easy}_{G_{\alpha}^{i},1}(\tau)d\tau\|_{U_{\Delta}^{2}(G_{\alpha}^{i}\times\R^{2})}^{2} \lesssim \epsilon_{2}^{1/3} \|u\|_{\tilde{X}_{j}([0,T]\times\R^{2})}^{5/3} \|u\|_{\tilde{Y}_{j}([0,T]\times\R^{2})}^{2}.
\end{equation}

\item[Estimate for $\mathrm{Easy}_{i,1}$:]
By repeating the argument in the preceding case, then proceeding as in the case of the contribution of $\mathrm{Easy}_{i,3}$, we obtain the final estimate
\begin{equation}
\sum_{i\geq j; N(G_{\kappa}^{j})\leq\epsilon_{3}^{1/2}2^{i-10}} \|\int_{t_{\kappa}^{j}}^{t}\mathrm{Easy}_{i,1}(\tau)d\tau\|_{U_{\Delta}^{2}(G_{\kappa}^{j}\times\R^{2})}^{2} \lesssim \epsilon_{2}^{1/3} \|u\|_{\tilde{X}_{j}([0,T]\times\R^{2})}^{5/3} \|u\|_{\tilde{Y}_{j}([0,T]\times\R^{2})}^{2}.
\end{equation}
\end{description}
\end{proof}

We next establish the hard estimate, which will occupy our attention for the remainder of this subsection.
\begin{lemma}[Preliminary hard estimate]\label{lem:LTSE_b_hard}
The following estimate holds uniformly in $20\leq j\leq k_{0}$, $20\leq i<j$, and $G_{\alpha}^{i}\subset G_{\kappa}^{j}\subset [0,T]$ satisfying $N(G_{\alpha}^{i})\leq \epsilon_{3}^{1/2}2^{i-5}$:
\begin{equation}
\begin{split}
&\|\int_{t_{\alpha}^{i}}^{t}e^{i(t-\tau)\Delta}\mathrm{Hard}_{G_{\alpha}^{i}}(\tau) d\tau\|_{U_{\Delta}^{2}(G_{\alpha}^{i}\times\R^{2})}\\
&\phantom{=}\lesssim_{u} \|P_{\xi(G_{\alpha}^{i}), i-5\leq\cdot\leq i+5}u\|_{U_{\Delta}^{2}(G_{\alpha}^{i}\times\R^{2})}\paren*{\epsilon_{2} + \|u\|_{\tilde{Y}_{i}(G_{\alpha}^{i}\times\R^{2})}\paren*{1+\|u\|_{\tilde{X}_{i}(G_{\alpha}^{i}\times\R^{2})}^{4}}}.
\end{split}
\end{equation}
The following estimate holds uniformly in $20\leq j\leq k_{0}$, $i \geq j$, and $G_{\kappa}^{j}\subset [0,T]$ satisfying $N(G_{\kappa}^{j}) \leq \epsilon_{3}^{1/2}2^{i-10}$:
\begin{equation}
\begin{split}
&\|\int_{t_{\kappa}^{j}}^{t}e^{i(t-\tau)\Delta}\mathrm{Hard}_{i}(\tau)d\tau\|_{U_{\Delta}^{2}(G_{\kappa}^{j}\times\R^{2})}\\
&\phantom{=}\lesssim_{u} \|P_{\xi(G_{\kappa}^{j}), i-5 \leq\cdot\leq i+5}u\|_{U_{\Delta}^{2}(G_{\kappa}^{j}\times\R^{2})} \paren*{\epsilon_{2} + \|u\|_{\tilde{Y}_{j}(G_{\kappa}^{j}\times\R^{2})} \paren*{1+\|u\|_{\tilde{X}_{j}(G_{\kappa}^{j}\times\R^{2})}^{4}}}.
\end{split}
\end{equation}
\end{lemma}

By some straightforward Littlewood-Paley analysis and proposition \ref{prop:Up_lp}, we obtain the following estimate for the total contribution of the $\mathrm{Hard}_{G_{\alpha},i}$ and $\mathrm{Hard}_{i}$. We omit the details.

\begin{cor}[Hard estimate]\label{cor:LTSE_b_hard} 
The following estimate holds uniformly in $20\leq j\leq k_{0}$ and $G_{\kappa}^{j}\subset [0,T]$:
\begin{equation}
\begin{split}
&\sum_{20\leq i<j}2^{i-j} \sum_{G_{\alpha}^{i}\subset G_{\kappa}^{j}; N(G_{\alpha}^{i})\leq \epsilon_{3}^{1/2}2^{i-5}} \|\int_{t_{\alpha}^{i}}^{t}e^{i(t-\tau)\Delta}\mathrm{Hard}_{G_{\alpha}^{i}}(\tau) d\tau\|_{U_{\Delta}^{2}(G_{\alpha}^{i}\times\R^{2})}^{2} \\
&\phantom{=}\qquad + \sum_{i\geq j; N(G_{\kappa}^{j})\leq \epsilon_{3}^{1/2}2^{i-10}}\|\int_{t_{\kappa}^{j}}^{t}e^{i(t-\tau)\Delta}\mathrm{Hard}_{i}(\tau)d\tau\|_{U_{\Delta}^{2}(G_{\kappa}^{j}\times\R^{2})}^{2}\\
&\phantom{=}\lesssim_{u} \|u\|_{\tilde{Y}_{j}(G_{\kappa}^{j}\times\R^{2})}^{2}\paren*{\epsilon_{2}+\|u\|_{\tilde{Y}_{j}(G_{\kappa}^{j}\times\R^{2})}\paren*{1+\|u\|_{\tilde{X}_{j}(G_{\kappa}^{j}\times\R^{2})}^{4}}}^{2}.
\end{split}
\end{equation}
\end{cor}

We now prove lemma \ref{lem:LTSE_b_hard}.
\begin{proof}
We only present the details for the contribution of
\begin{equation}
	\mathrm{Hard}_{G_{\alpha}^{i},1} = P_{\xi(G_{\alpha}^{i}), i-2\leq\cdot\leq i+2}[\E(|P_{\xi(\tau), <i-10}u|^{2})P_{\xi(G_{\alpha}^{i}), \geq i-5}u]
\end{equation}
for $20\leq i<j$. The contribution of $\mathrm{Hard}_{i,1}$ for $i\geq j$ follows by analogous arguments \emph{mutatis mutandis}, while the contributions of $\mathrm{Hard}_{G_{\alpha}^{i},2}$ and $\mathrm{Hard}_{i,2}$ are strictly easier to estimate as one near and one far frequency factor fall inside the argument of the operator $\E$.

As in the proof of lemma \ref{lem:LTSE_NI_est}, we first reduce to the case where $G_{\alpha}^{i}$ is the union of small intervals $J_{l}$. Let $J_{1},J_{2}$ be the two (possibly empty) small intervals which intersect $G_{\alpha}^{i}$ but are not contained in $G_{\alpha}^{i}$, and define the interval $\tilde{G}_{\alpha}^{i}:= G_{\alpha}^{i}\setminus (J_{1}\cup J_{2})$. By duality, triangle inequality, and the embedding $V_{\Delta}^{2} \subset L_{t,x}^{4}$, we have that
\begin{equation}
\|\int_{t_{\alpha}^{i}}^{t}e^{i(t-\tau)\Delta}\mathrm{Hard}_{G_{\alpha}^{i},1}(\tau)d\tau\|_{U_{\Delta}^{2}(G_{\alpha}^{i}\times\R^{2})} \lesssim \|\int_{t_{\alpha}^{i}}^{t}\mathrm{Hard}_{G_{\alpha}^{i},1}(\tau)d\tau\|_{U_{\Delta}^{2}(\tilde{G}_{\alpha}^{i}\times\R^{2})} + \|\mathrm{Hard}_{G_{\alpha}^{i},1}\|_{L_{t,x}^{4/3}(G_{\alpha}^{i}\cap (J_{1}\cup J_{2}) \times\R^{2})}.
\end{equation}
We claim that up to an acceptable error (i.e. one which we can absorb into the RHS of lemma \ref{lem:LTSE_b_hard}), we may assume that $t_{\alpha}^{i}\in \tilde{G}_{\alpha}^{i}$. Indeed,  otherwise suppose that $t_{\alpha}^{i}\in J_{1}\cup J_{2}$, and without loss of generality assume that $t_{\alpha}^{i}\in J_{1}$. Let $\tilde{t}_{\alpha}^{i}$ denote the left endpoint of of the interval $\tilde{G}_{\alpha}^{i}$. Then
\begin{align}
\|\int_{t_{\alpha}^{i}}^{t}e^{i(t-\tau)\Delta}\mathrm{Hard}_{G_{\alpha}^{i},1}(\tau)d\tau-\int_{\tilde{t}_{\alpha}^{i}}^{t}e^{i(t-\tau)\Delta}\mathrm{Hard}_{G_{\alpha}^{i},1}(\tau)d\tau\|_{U_{\Delta}^{2}(G_{\alpha}^{i}\times\R^{2})} &= \|\int_{\tilde{t}_{\alpha}^{i}}^{t_{\alpha}^{i}}e^{i(t-\tau)\Delta}\mathrm{Hard}_{G_{\alpha}^{i},1}(\tau)d\tau\|_{U_{\Delta}^{2}(G_{\alpha}^{i}\times\R^{2})} \nonumber\\
&\lesssim \|\mathrm{Hard}_{G_{\alpha}^{i},1}\|_{L_{t,x}^{4/3}(J_{1}\cap G_{\alpha}^{i}\times\R^{2})},
\end{align}
where the ultimate line follows from the definition of the $U_{\Delta}^{2}$ norm and the dual homogeneous Strichartz estimate. To show that the error is acceptable, it remains for us to estimate the quantity
\begin{equation}
\sum_{l=1,2}\|P_{\xi(G_{\alpha}^{i}), i-2\leq\cdot\leq i+2}[\E(|P_{\xi(t), <i-10}u|^{2})(P_{\xi(G_{\alpha}^{i}), \geq i-5}u)]\|_{L_{t,x}^{4/3}(J_{l}\cap G_{\alpha}^{i}\times\R^{2})},
\end{equation}
which we do now.

By symmetry of argument, it suffices to consider the case of $J_{1}$ (i.e. $l=1$ in the preceding equation). We first perform another near-far frequency decomposition and use triangle inequality to obtain
\begin{align}
&\|(P_{\xi(G_{\alpha}^{i}), i-5\leq\cdot\leq i+5}u)\E\paren*{|P_{\xi(t),<i-10}u|^{2}}\|_{L_{t,x}^{4/3}(G_{\alpha}^{i}\cap J_{1}\times\R^{2})} \nonumber\\
&\phantom{=}\lesssim\|(P_{\xi(G_{\alpha}^{i}), i-5\leq\cdot\leq i+5}u)\E\paren*{|P_{\xi(t),<i-10}P_{\xi(J_{1}\cap G_{\alpha}^{i}), \leq 2\epsilon_{3}^{-1/4}N(J_{1}\cap G_{\alpha}^{i})}u|^{2}}\|_{L_{t,x}^{4/3}(G_{\alpha}^{i}\cap J_{1}\times\R^{2})} \nonumber\\
&\phantom{=}\quad+ \|(P_{\xi(G_{\alpha}^{i}),i-5\leq\cdot\leq i+5}u)\E\paren*{(P_{\xi(t),<i-10}P_{\xi(J_{l}\cap G_{\alpha}^{i}), \leq 2\epsilon_{3}^{-1/4}N(J_{1}\cap G_{\alpha}^{i})}u) (\ol{P_{\xi(t),<i-10}P_{\xi(J_{1}\cap G_{\alpha}^{i}),>2\epsilon_{3}^{-1/4}N(J_{1}\cap G_{\alpha}^{i})}u})}\|_{L_{t,x}^{4/3}(G_{\alpha}^{i}\cap J_{1}\times\R^{2})} \nonumber\\
&\phantom{=}\quad+\|(P_{\xi(G_{\alpha}^{i}), i-5\leq\cdot\leq i+5}u)\E\paren*{|P_{\xi(t),<i-10}P_{\xi(J_{1}\cap G_{\alpha}^{i}), >2\epsilon_{3}^{-1/4}N(J_{1}\cap G_{\alpha}^{i})}u|^{2}}\|_{L_{t,x}^{4/3}(G_{\alpha}^{i}\cap J_{1}\times\R^{2})} \nonumber\\
&\phantom{=}\eqqcolon \mathrm{Term}_{1}+\mathrm{Term}_{2} + \mathrm{Term}_{3}.
\end{align}
To estimate $\mathrm{Term}_{2}+\mathrm{Term}_{3}$, we use H\"{o}lder's inequality, Bernstein's lemma, Strichartz estimates, and the frequency localization property \eqref{eq:sp_frq_loc} to obtain that
\begin{align}
\mathrm{Term}_{2} +\mathrm{Term}_{3} &\lesssim  \|P_{\xi(J_{1}\cap G_{\alpha}^{i}),>2\epsilon_{3}^{-1/4}N(J_{1}\cap G_{\alpha}^{i})}u\|_{L_{t}^{\infty}L_{x}^{2}(G_{\alpha}^{i}\cap J_{1}\times\R^{2})} \|P_{\xi(G_{\alpha}^{i}), i-5\leq\cdot\leq i+5}u\|_{L_{t}^{8/3}L_{x}^{8}(G_{\alpha}^{i}\cap J_{1}\times\R^{2})}\times \nonumber\\
&\phantom{=}\qquad\|P_{\xi(t),>i-10}u\|_{L_{t}^{8/3}L_{x}^{8}(G_{\alpha}^{i}\cap J_{1}\times\R^{2})} \nonumber\\
&\lesssim  \epsilon_{2}\|P_{\xi(G_{\alpha}^{i}),i-5\leq\cdot\leq i+5}u\|_{U_{\Delta}^{2}(G_{\alpha}^{i}\cap J_{1}\times\R^{2})}.
\end{align}

To estimate $\mathrm{Term}_{1}$, we want to use a bilinear Strichartz estimate to exploit the frequency separation between the far frequency factor and one of the near frequency factors. However, the reader will notice that $\mathrm{Term}_{1}$ is of the worst case since both of the near frequency factors fall inside the argument of $\E$; therefore, it is not a priori clear how to apply bilinear estimates without destroying the cancellation in the kernel of $\E$. We address this difficulty here and throughout this work with a tool which we call the \emph{double frequency decomposition}, which appeared in a simpler form in \cite{Chae2010} in which the authors considered global well-posedness for the 3D mass-critical Hartree equation at subcritical regularities below $H^{1}$. Although the present case is the easiest of the applications of the double frequency decomposition which we will make in this paper, it is nevertheless illustrative of the idea of the argument without being overly technical. Moreover, we will not be so detailed in the routine steps of the decomposition in the sequel.

We first perform a homogeneous Littlewood-Paley decomposition of the symbol of $\E$,
\begin{equation}
\E = \sum_{k\in\mathbb{Z}} \E\dot{P}_{k} \eqqcolon \sum_{k\in\mathbb{Z}}\E_{k},
\end{equation}
and observe that
\begin{equation}
\begin{split}
&\E\paren*{|P_{\xi(t),<i-10}P_{\xi(J_{1}\cap G_{\alpha}^{i}), \leq 2\epsilon_{3}^{-1/4}N(J_{1}\cap G_{\alpha}^{i})}u|^{2}} \\
&\phantom{=} = \sum_{k\leq \min\{\log_{2}(8\epsilon_{3}^{-1/4}N(J_{1})\cap G_{\alpha}^{i}),i-8\}}\E_{k}\paren*{|P_{\xi(t),<i-10}P_{\xi(J_{1}\cap G_{\alpha}^{i}), \leq 2\epsilon_{3}^{-1/4}N(J_{1}\cap G_{\alpha}^{i})}u|^{2}}.
\end{split}
\end{equation}
Since the symbol of $\E$ belongs to $C^{\infty}(\R^{2}\setminus\{0\})$ and is homogeneous of degree zero, it follows from the usual scaling and integration by parts argument that the kernel $\mathcal{K}_{k}$ of $\E_{k}$ is Schwartz class and satisfies the $k$-uniform rapid decay estimate
\begin{equation}
|\mathcal{K}_{k}(x)| \lesssim_{N} 2^{2k}\jp{2^{k}x}^{-N}, \qquad \forall x\in\R^{2}, \enspace \forall N\geq 0.
\end{equation}
Now for each $k\leq i-8$, we decompose Fourier space $\hat{\R}^{2}$ into dyadic cubes $Q_{a}^{k}$ of side length $2^{k}$ so that
\begin{equation}
\begin{split}
&\E_{k}\paren*{|P_{\xi(t),<i-10}P_{\xi(J_{1}\cap G_{\alpha}^{i}), \leq 2\epsilon_{3}^{-1/4}N(J_{1}\cap G_{\alpha}^{i})}u|^{2}} \\
&\phantom{=}= \sum_{a,a'\in\mathbb{Z}^{2}} \E_{k}\paren*{(P_{\xi(t),<i-10}P_{\xi(J_{1}\cap G_{\alpha}^{i}), \leq 2\epsilon_{3}^{-1/4}N(J_{1}\cap G_{\alpha}^{i})}P_{Q_{a}^{k}}u)(\ol{P_{\xi(t),<i-10}P_{\xi(J_{1}\cap G_{\alpha}^{i}), \leq 2\epsilon_{3}^{-1/4}N(J_{1}\cap G_{\alpha}^{i})}P_{Q_{a'}^{k}}u})}.
\end{split}
\end{equation}
We claim that for each $a\in\mathbb{Z}^{2}$ fixed, there are $O(1)$ cubes $Q_{a'}^{k}$ such that the summand in the RHS of the preceding equality is nonzero or equivalently, $(Q_{a}^{k}-Q_{a'}^{k})\cap \{2^{k-2}\leq |\xi| \leq 2^{k+2}\}\neq \emptyset$. Indeed, fix $a\in \mathbb{Z}^{2}$ and suppose that $|a-a'| > 2^{5}$. Then for any $\xi\in Q_{a}^{k}$ and $\xi' \in -Q_{a'}^{k}$, we have by the reverse triangle inequality that
\begin{equation}
|\xi+\xi'| = |(\xi-2^{k}a) + 2^{k}(a-a') + (\xi'+2^{k}a')| > 2^{5+k} - 2^{k+\frac{1}{2}} - 2^{k+\frac{1}{2}} \geq 2^{4+k},
\end{equation}
which shows that $\xi+\xi'$ does not belong to the support of the symbol of $\E_{k}$. Since there are at most $2^{20}$ pairs $a,a'\in\Z^{2}$ satisfying the condition $|a-a'|\leq 2^{5}$, we obtain the claim.

Now by Minkowski's inequality,
\begin{align}
&\|(P_{\xi(G_{\alpha}^{i}), i-5\leq\cdot\leq i+5}u)\E_{k}\paren*{|P_{\xi(t),<i-10}P_{\xi(J_{1}\cap G_{\alpha}^{i}), \leq 2\epsilon_{3}^{-1/4}N(J_{1}\cap G_{\alpha}^{i})}u|^{2}}\|_{L_{t,x}^{4/3}(J_{1}\cap G_{\alpha}^{i}\times\R^{2})} \nonumber\\
&\phantom{=}\lesssim \int_{\R^{2}}dy |\mathcal{K}_{k}(y)|\sum_{|a-a'|\leq 2^{5}} \|(P_{\xi(G_{\alpha}^{i}), i-5\leq\cdot\leq i+5}u) (P_{\xi(t), <i-10}P_{\xi(J_{1}\cap G_{\alpha}^{i}), \leq 2\epsilon_{3}^{-1/4}N(J_{1}\cap G_{\alpha}^{i})}P_{Q_{a}^{k}}\tau_{y}u) \nonumber\\
&\phantom{=}\qquad (\ol{P_{\xi(t),<i-10}P_{\xi(J_{1}\cap G_{\alpha}^{i}), \leq 2\epsilon_{3}^{-1/4}N(J_{1}\cap G_{\alpha}^{i})}P_{Q_{a'}^{k}}\tau_{y}u})\|_{L_{t,x}^{4/3}(J_{1}\cap G_{\alpha}^{i}\times\R^{2})} \nonumber\\
&\phantom{=}\lesssim \sup_{y\in\R^{2}}\sum_{|a-a'|\leq 2^{5}} \|(P_{\xi(G_{\alpha}^{i}), i-5\leq\cdot\leq i+5}u) (P_{\xi(t), <i-10}P_{\xi(J_{1}\cap G_{\alpha}^{i}), \leq 2\epsilon_{3}^{-1/4}N(J_{1}\cap G_{\alpha}^{i})}P_{Q_{a}^{k}}\tau_{y}u) \nonumber\\
&\phantom{=}\qquad\ol{(P_{\xi(t),<i-10}P_{\xi(J_{1}\cap G_{\alpha}^{i}), \leq 2\epsilon_{3}^{-1/4}N(J_{1}\cap G_{\alpha}^{i})}P_{Q_{a'}^{k}}\tau_{y}u)}\|_{L_{t,x}^{4/3}(J_{1}\cap G_{\alpha}^{i}\times\R^{2})} \nonumber\\
&\phantom{=}\lesssim \sup_{y\in\R^{2}}\sum_{|a-a'|\leq 2^{5}} \|(P_{\xi(G_{\alpha}^{i}), i-5\leq\cdot\leq i+5}u) (P_{\xi(J_{1}\cap G_{\alpha}^{i}), < 2\epsilon_{3}^{-1/4}N(J_{1}\cap G_{\alpha}^{i})}	P_{Q_{a}^{k}}\tau_{y}u)\|_{L_{t,x}^{2}(G_{\alpha}^{i}\cap J_{1}\times\R^{2})} \|P_{Q_{a'}^{k}}u\|_{L_{t,x}^{4}(J_{1}\times\R^{2})} \nonumber\\
&\phantom{=}\lesssim \sup_{y\in\R^{2}} \paren*{\sum_{a\in\Z^{2}} \|(P_{\xi(G_{\alpha}^{i}), i-5\leq\cdot\leq i+5}u) (P_{\xi(J_{1}\cap G_{\alpha}^{i}), <2\epsilon_{3}^{-1/4}N(J_{1}\cap G_{\alpha}^{i})}P_{Q_{a}^{k}}\tau_{y}u)\|_{L_{t,x}^{2}(G_{\alpha}^{i}\cap J_{1}\times\R^{2})}^{2}}^{1/2} \|u\|_{X_{\Delta}^{2,k}(J_{1}\times\R^{2})} \nonumber \\
&\phantom{=}\lesssim \sup_{y\in\R^{2}}  \paren*{\sum_{a\in\Z^{2}} \|(P_{\xi(G_{\alpha}^{i}), i-5\leq\cdot\leq i+5}u) (P_{\xi(J_{1}\cap G_{\alpha}^{i}), < 2\epsilon_{3}^{-1/4}N(J_{1}\cap G_{\alpha}^{i})}P_{Q_{a}^{k}}\tau_{y}u)\|_{L_{t,x}^{2}(G_{\alpha}^{i}\cap J_{1}\times\R^{2})}^{2}}^{1/2}, \label{eq:LTSE_SD_ex_bl}
\end{align}
where we use H\"{o}lder's inequality to obtain the antepenultimate inequality; Cauchy-Schwarz, almost orthogonality, and Strichartz estimates to obtain the penultimate inequality; and the embedding $U_{\Delta}^{2}\subset X_{\Delta}^{2,k}$ together with $J_{1}$ is small to obtain the ultimate inequality. To estimate \eqref{eq:LTSE_SD_ex_bl}, we the assumption that $N(G_{\alpha}^{i})\leq \epsilon_{3}^{1/2}2^{i-5}$ in order to use Galilean invariance to apply the bilinear Strichartz estimate of proposition \ref{prop:bs_i} at the level of each cube $Q_{a}^{k}$, obtaining that
\begin{align}
\eqref{eq:LTSE_SD_ex_bl} &\lesssim \paren*{\sum_{a; Q_{a}^{k}\cap B(\xi(J_{1}\cap G_{\alpha}^{i}), 4\epsilon_{3}^{-1/4}N(J_{1}\cap G_{\alpha}^{i}))\neq \emptyset} \paren*{2^{\frac{k-i}{2}}\|P_{\xi(G_{\alpha}^{i}), i-5\leq\cdot\leq i+5}u\|_{U_{\Delta}^{2}(G_{\alpha}^{i}\times\R^{2})} \|P_{Q_{a}^{k}}u\|_{U_{\Delta}^{2}(J_{1}\times\R^{2})}}^{2}}^{1/2} \nonumber\\
&\lesssim 2^{\frac{k-i}{2}}\|P_{\xi(G_{\alpha}^{i}), i-5\leq\cdot\leq i+5}u\|_{U_{\Delta}^{2}(G_{\alpha}^{i}\times\R^{2})},
\end{align}
where we use the spatial translation invariance of the $U_{\Delta}^{2}$ norm, the embedding $U_{\Delta}^{2}\subset X_{\Delta}^{2,k}$, and that $J_{1}$ is small to obtain the ultimate inequality. Now summing over $k\leq \log_{2}(8\epsilon_{3}^{-1/4}N(J_{1}\cap G_{\alpha}^{i}))$, we obtain that
\begin{align}
&\sum_{k \leq \log_{2}(8\epsilon_{3}^{-1/4}N(J_{1}\cap G_{\alpha}^{i}))} \|\E_{k}\paren*{|P_{\xi(t),<i-10}P_{\xi(J_{1}\cap G_{\alpha}^{i}), \leq 2\epsilon_{3}^{-1/4}N(J_{1}\cap G_{\alpha}^{i})}u|^{2}}\|_{L_{t,x}^{4/3}(G_{\alpha}^{i}\cap J_{1}\times\R^{2})} \nonumber\\
&\phantom{=}\lesssim \paren*{2^{-i}\epsilon_{3}^{-1/4}N(J_{1}\cap G_{\alpha}^{i})}^{1/2} \|P_{\xi(G_{\alpha}^{i}), i-5\leq \cdot\leq i+5}u\|_{U_{\Delta}^{2}(G_{\alpha}^{i}\times\R^{2})} \nonumber\\
&\phantom{=}\lesssim \epsilon_{2} \|P_{\xi(G_{\alpha}^{i}), i-5\leq\cdot\leq i+5}u\|_{U_{\Delta}^{2}(G_{\alpha}^{i}\times\R^{2})},
\end{align}
where we use that $N(J_{1}\cap G_{\alpha}^{i})\leq \epsilon_{3}^{1/2}2^{i-5}$ and that $\epsilon_{3}\leq \epsilon_{2}^{10}$ to obtain the ultimate inequality. Thus, we have shown that
\begin{equation}
\mathrm{Term}_{1} \lesssim \epsilon_{2}\|P_{\xi(G_{\alpha}^{i}), i-5\leq\cdot\leq i+5}u\|_{U_{\Delta}^{2}(G_{\alpha}^{i}\times\R^{2})},
\end{equation}
completing the proof of the claim that
\begin{equation}
\|\mathrm{Hard}_{G_{\alpha}^{i},1}\|_{L_{t,x}^{4/3}(G_{\alpha}^{i}\cap (J_{1}\cup J_{2})\times\R^{2})} \lesssim \epsilon_{2} \|P_{\xi(G_{\alpha}^{i}),i-5\leq\cdot\leq i+5}u\|_{U_{\Delta}^{2}(G_{\alpha}^{i}\times\R^{2})}.
\end{equation}

We now estimate the quantity
\begin{equation}
\|\int_{t_{\alpha}^{i}}^{t}\mathrm{Hard}_{G_{\alpha}^{i},1}(\tau)d\tau\|_{U_{\Delta}^{2}(G_{\alpha}^{i}\times\R^{2})}
\end{equation}
under the assumptions that $G_{\alpha}^{i}$ is a union of small intervals $J_{l}$ and that $t_{\alpha}^{i} \in G_{\alpha}^{i}$. We decompose $\E(|P_{\xi(t), \leq i-10}u|^{2})$ by
\begin{align}
\E\paren*{|P_{\xi(t), \leq i-10}u|^{2}} &= \sum_{0\leq l_{2} \leq i-10} \E\paren*{(P_{\xi(t), l_{2}}u)(\ol{P_{\xi(t), \leq l_{2}}u})}+\sum_{0<l_{2}\leq i-10} \E\paren*{(P_{\xi(t),<l_{2}}u)(\ol{P_{\xi(t),l_{2}}u})}.
\end{align}
By symmetry and triangle inequality, it suffices to estimate
\begin{equation}
\sum_{0\leq l_{2} \leq i-10} \|\int_{t_{\alpha}^{i}}^{t}e^{i(t-\tau)\Delta}P_{\xi(G_{\alpha}^{i}), i-2\leq\cdot\leq i+2}\brak*{(P_{\xi(G_{\alpha}^{i}), i-5\leq\cdot\leq i+5}u(\tau)))\E\paren*{(P_{\xi(t), l_{2}}u(\tau)) (\ol{P_{\xi(t), \leq l_{2}}u(\tau)}}}d\tau\|_{U_{\Delta}^{2}(G_{\alpha}^{i}\times\R^{2})}.
\end{equation}
For each integer $0\leq l_{2}\leq i-10$, we split the small intervals constituting the time interval $G_{\alpha}^{i}$ into two different groupings based on the size of their maximum frequency scale relative to $2^{l_{2}}$:
\begin{equation}
G_{\alpha}^{i} = \paren*{\bigcup_{J_{l}\subset G_{\alpha}^{i}; N(J_{l}) \geq \epsilon_{3}^{1/2}2^{l_{2}-6}} J_{l}} \cup \paren*{\bigcup_{J_{l}\subset G_{\alpha}^{i}; N(J_{l})< \epsilon_{3}^{1/2}2^{l_{2}-6}} J_{l}} \eqqcolon G_{\alpha,l_{2},\geq}^{i} \cup G_{\alpha,l_{2},<}^{i}.
\end{equation}
Using proposition \ref{prop:U2_duh}, we have that
\begin{equation}
\begin{split}
&\sum_{0\leq l_{2} \leq i-10} \|\int_{t_{\alpha}^{i}}^{t}e^{i(t-\tau)\Delta}P_{\xi(G_{\alpha}^{i}), i-2\leq\cdot\leq i+2}\brak*{(P_{\xi(G_{\alpha}^{i}), i-5\leq\cdot\leq i+5}u(\tau))\E\paren*{(P_{\xi(t), l_{2}}u(\tau))(\ol{P_{\xi(t), \leq l_{2}}u(\tau)})}}d\tau\|_{U_{\Delta}^{2}(G_{\alpha,l_{2},\geq}^{i}\times\R^{2})} \\
&\phantom{=}\lesssim \sum_{0\leq l_{2}\leq i-10} \paren*{\sum_{J_{l}\subset G_{\alpha}^{i}; N(J_{l})\geq\epsilon_{3}^{1/2}2^{l_{2}-6}}\|P_{\xi(G_{\alpha}^{i}), i-2\leq\cdot\leq i+2}\brak*{(P_{\xi(G_{\alpha}^{i}), i-5\leq\cdot\leq i+5}u)\E\paren*{(P_{\xi(\tau), l_{2}u})(\ol{P_{\xi(\tau), \leq l_{2}}u})}}\|_{V_{\Delta}^{2}(J_{l}\times\R^{2})^{*}}^{2}}^{1/2} \\
&\phantom{=}\quad +\sum_{0\leq l_{2}\leq i-10}\sum_{J_{l}\subset G_{\alpha}^{i} ; N(J_{l})\geq\epsilon_{3}^{1/2}2^{l_{2}-6}}\|\int_{J_{l}}e^{-it\Delta}P_{\xi(G_{\alpha}^{i}), i-2\leq\cdot\leq i+2}\brak*{(P_{\xi(G_{\alpha}^{i}), i-5\leq\cdot\leq i+5}u(t))\E\paren*{(P_{\xi(t), l_{2}}u(t))(\ol{P_{\xi(t), \leq l_{2}}u(t)})}}dt\|_{L_{x}^{2}(\R^{2})}.
\end{split}
\end{equation}
Rather repeat the step as before for $G_{\alpha,l_{2},<}^{i}$, it is more convenient to re-express this set in terms of subintervals $G_{\beta}^{l_{2}}\subset G_{\alpha}^{i}$. To see this convenience, observe that if $G_{\beta}^{l_{2}}\cap G_{\alpha,l_{2},<}^{i}\neq\emptyset$, then there exists some $l$ such that $N(J_{l})<\epsilon_{3}^{1/2}2^{l_{2}-6}$ and $G_{\beta}^{l_{2}}\cap J_{l}\neq\emptyset$. If $t_{J_{l}}\in G_{\beta}^{l_{2}}\cap J_{l}$, then
\begin{equation}
N(t) \leq |N(t)-N(t_{J_{l}})| + N(t_{J_{l}}) < \int_{G_{\beta}^{l_{2}}} 2^{-20}\epsilon_{1}^{-1/2} N(\tau)^{3}d\tau + \epsilon_{3}^{1/2}2^{l_{2}-6} \leq \epsilon_{3}^{1/2}2^{l_{2}-5}, \qquad \forall t\in G_{\beta}^{l_{2}},
\end{equation}
which implies that $N(G_{\beta}^{l_{2}}) \leq \epsilon_{3}^{1/2}2^{l_{2}-5}$. Therefore we are estimating the near frequency piece of $\mathrm{Hard}_{G_{\alpha}^{i}}$ at frequencies which are at distance from $\xi(t)$ much larger than $N(t)$, allowing us to put certain factors in a $\tilde{Y}_{j}$ norm. Now applying proposition \ref{prop:U2_duh} as before, we obtain that
\begin{equation}
\begin{split}
&\sum_{0\leq l_{2} \leq i-10} \|\int_{t_{\alpha}^{i}}^{t}e^{i(t-\tau)\Delta}P_{\xi(G_{\alpha}^{i}), i-2\leq\cdot\leq i+2}\brak*{(P_{\xi(G_{\alpha}^{i}), i-5\leq\cdot\leq i+5}u(\tau))\E\paren*{(P_{\xi(t), l_{2}}u(\tau))(\ol{P_{\xi(t), \leq l_{2}}u(\tau)})}}d\tau\|_{U_{\Delta}^{2}(G_{\alpha,l_{2},<}^{i}\times\R^{2})} \\
&\phantom{=}\lesssim \sum_{0\leq l_{2}\leq i-10} \paren*{\sum_{G_{\beta}^{l_{2}}\subset G_{\alpha}^{i} ; N(G_{\beta}^{l_{2}})\leq\epsilon_{3}^{1/2}2^{l_{2}-5}}\|P_{\xi(G_{\alpha}^{i}), i-2\leq\cdot\leq i+2}\brak*{(P_{\xi(G_{\alpha}^{i}), i-5\leq\cdot\leq i+5}u)\E\paren*{(P_{\xi(\tau), l_{2}u})(\ol{P_{\xi(\tau), \leq l_{2}}u})}}\|_{V_{\Delta}^{2}(G_{\beta}^{l_{2}}\times\R^{2})^{*}}^{2}}^{1/2}\\
&\phantom{=}\enspace +\sum_{0\leq l_{2}\leq i-10}\sum_{G_{\beta}^{l_{2}}\subset G_{\alpha}^{i} ; N(G_{\beta}^{l_{2}})\leq\epsilon_{3}^{1/2}2^{l_{2}-5}}\|\int_{G_{\beta}^{l_{2}}}e^{-it\Delta}P_{\xi(G_{\alpha}^{i}), i-2\leq\cdot\leq i+2}\brak*{(P_{\xi(G_{\alpha}^{i}), i-5\leq\cdot\leq i+5}u(t))\E\paren*{(P_{\xi(t), l_{2}}u(t))(\ol{P_{\xi(t), \leq l_{2}}u(t)})}}dt\|_{L_{x}^{2}(\R^{2})}. 
\end{split}
\end{equation}
Combining our estimates, we have that
\begin{align}
&\sum_{0\leq l_{2} \leq i-10} \|\int_{t_{\alpha}^{i}}^{t}e^{i(t-\tau)\Delta}P_{\xi(G_{\alpha}^{i}), i-2\leq\cdot\leq i+2}\brak*{(P_{\xi(G_{\alpha}^{i}), i-5\leq\cdot\leq i+5}u(\tau))\E\paren*{(P_{\xi(t), l_{2}}u(\tau))(\ol{P_{\xi(t), \leq l_{2}}u(\tau)})}}d\tau\|_{U_{\Delta}^{2}(G_{\alpha}^{i}\times\R^{2})} \nonumber\\
&\phantom{=}\lesssim \sum_{0\leq l_{2}\leq i-10} \paren*{\sum_{J_{l}\subset G_{\alpha}^{i} ; N(J_{l})\geq\epsilon_{3}^{1/2}2^{l_{2}-6}} \|P_{\xi(G_{\alpha}^{i}), i-2\leq\cdot\leq i+2}\brak*{(P_{\xi(G_{\alpha}^{i}), i-5\leq\cdot\leq i+5}u)\E\paren*{(P_{\xi(\tau), l_{2}u})(\ol{P_{\xi(\tau), \leq l_{2}}u})}}\|_{V_{\Delta}^{2}(J_{l}\times\R^{2})^{*}}^{2}}^{1/2} \label{eq:LTSE_b_h1}\\
&\phantom{=}\enspace +\sum_{0\leq l_{2}\leq i-10}\sum_{J_{l}\subset G_{\alpha}^{i} ; N(J_{l})\geq\epsilon_{3}^{1/2}2^{l_{2}-6}}\|\int_{J_{l}}e^{-it\Delta}P_{\xi(G_{\alpha}^{i}), i-2\leq\cdot\leq i+2}\brak*{(P_{\xi(G_{\alpha}^{i}), i-5\leq\cdot\leq i+5}u(t))\E\paren*{(P_{\xi(t), l_{2}}u(t))(\ol{P_{\xi(t), \leq l_{2}}u(t)})}}dt\|_{L_{x}^{2}(\R^{2})} \label{eq:LTSE_b_h2}\\
&\phantom{=}\enspace +\sum_{0\leq l_{2}\leq i-10} \paren*{\sum_{G_{\beta}^{l_{2}}\subset G_{\alpha}^{i} ; N(G_{\beta}^{l_{2}})\leq\epsilon_{3}^{1/2}2^{l_{2}-5}}\|P_{\xi(G_{\alpha}^{i}), i-2\leq\cdot\leq i+2}\brak*{(P_{\xi(G_{\alpha}^{i}), i-5\leq\cdot\leq i+5}u)\E\paren*{(P_{\xi(\tau), l_{2}u})(\ol{P_{\xi(\tau), \leq l_{2}}u})}}\|_{V_{\Delta}^{2}(G_{\beta}^{l_{2}}\times\R^{2})^{*}}^{2}}^{1/2} \label{eq:LTSE_b_h3}\\
&\phantom{=}\enspace +\sum_{0\leq l_{2}\leq i-10}\sum_{G_{\beta}^{l_{2}}\subset G_{\alpha}^{i} ; N(G_{\beta}^{l_{2}})\leq\epsilon_{3}^{1/2}2^{l_{2}-5}}\|\int_{G_{\beta}^{l_{2}}}e^{-it\Delta}P_{\xi(G_{\alpha}^{i}), i-2\leq\cdot\leq i+2}\brak*{(P_{\xi(G_{\alpha}^{i}), i-5\leq\cdot\leq i+5}u(t))\E\paren*{(P_{\xi(t), l_{2}}u(t))(\ol{P_{\xi(t), \leq l_{2}}u(t)})}}dt\|_{L_{x}^{2}(\R^{2})}. \label{eq:LTSE_b_h4}
\end{align}
We proceed to estimate each of the terms \eqref{eq:LTSE_b_h1} - \eqref{eq:LTSE_b_h4}.

\begin{description}[leftmargin=*]
\item[Estimate for \eqref{eq:LTSE_b_h1}:]
Our basic strategy is to use duality and bilinear estimates which exploit the frequency separation between the various factors that are respectively near and far from the frequency center $\xi(t)$. Since we hope to use bilinear estimates, we need to use the double frequency decomposition to handle the nonlocal operator $\E$. We now turn to the details.

Let $v\in V_{\Delta}^{2}$ have spatial Fourier support in the dyadic annulus $A(\xi(G_{\alpha}^{i}), i-3,i+3)$ and satisfy $\|v\|_{V_{\Delta}^{2}(J_{l}\times\R^{2})}\leq 1$. Then by the double frequency decomposition, Minkowski's inequality, and Cauchy-Schwarz, we obtain that
\begin{align}
&\left|\int_{J_{l}} \ipp{v(t),(P_{\xi(G_{\alpha}^{i}), i-5\leq\cdot\leq i+5}u(t))\E\paren*{(P_{\xi(t),l_{2}}u(t))(\ol{P_{\xi(t),\leq l_{2}}u(t)})}}dt\right| \nonumber\\
&\phantom{=}\lesssim \sum_{k\leq l_{2}+2}\sup_{y\in\R^{2}}\left\{\paren*{\sum_{a; Q_{a}^{k}\cap B(\xi(J_{l}), i-7)\neq\emptyset} \|(P_{\xi(G_{\alpha}^{i}),i-5\leq\cdot\leq i+5}u)(P_{\xi(t), l_{2}}P_{Q_{a}^{k}}\tau_{y}u)\|_{L_{t,x}^{2}(J_{l}\times\R^{2})}^{2}}^{1/2} \times \right. \nonumber\\
&\phantom{=}\qquad \left.\paren*{\sum_{a'; Q_{a'}^{k}\cap B(\xi(J_{l}), i-7)\neq\emptyset} \|v(P_{\xi(t),\leq l_{2}}P_{Q_{a'}^{k}}\tau_{y}u)\|_{L_{t,x}^{2}(J_{l}\times\R^{2})}^{2}}^{1/2}\right\}.
\end{align}
Using Galilean invariance to apply the bilinear Strichartz estimate of proposition \ref{prop:bs_i}, we have that
\begin{align}
&\sup_{y\in\R^{2}} \paren*{\sum_{a; Q_{a}^{k} \cap B(\xi(J_{l}), i-7)\neq\emptyset} \|(P_{\xi(G_{\alpha}^{i}), i-5\leq\cdot\leq i+5}u)(P_{\xi(t),l_{2}}P_{Q_{a}^{k}}\tau_{y}u)\|_{L_{t,x}^{2}(J_{l}\times\R^{2})}^{2}}^{1/2} \nonumber\\
&\phantom{=}\lesssim 2^{\frac{k-i}{2}}\|P_{\xi(G_{\alpha}^{i}), i-5\leq\cdot\leq i+5}u\|_{U_{\Delta}^{2}(J_{l}\times\R^{2})} \|u\|_{X_{\Delta}^{2,k}(J_{l}\times\R^{2})}
\end{align}
and
\begin{align}
\sup_{y\in\R^{2}} \paren*{\sum_{a'; Q_{a'}^{k}\cap B(\xi(J_{l}),i-7)\neq\emptyset}\|v(P_{\xi(t),\leq l_{2}}P_{Q_{a}^{k}}\tau_{y}u)\|_{L_{t,x}^{2}(J_{l}\times\R^{2})}^{2}}^{1/2} &\lesssim 2^{(k-i)\frac{1}{2}^{-}} \|v\|_{V_{\Delta}^{2}(J_{l}\times\R^{2})} \|u\|_{X_{\Delta}^{2,k}(J_{l}\times\R^{2})} \nonumber \\
&\leq 2^{(k-i)\frac{1}{2}^{-}}\|u\|_{X_{\Delta}^{2,k}(J_{l}\times\R^{2})}.
\end{align}
Using that $\sup_{k} \|u\|_{X_{\Delta}^{2,k}(J_{l}\times\R^{2})}\lesssim 1$ and $N(J_{l})\geq \epsilon_{3}^{1/2}2^{l_{2}-5}$, we see that
\begin{equation}
\begin{split}
&\sum_{k\leq l_{2}+2} 2^{(k-i)1^{-}}\|P_{\xi(G_{\alpha}^{i}), i-5\leq\cdot\leq i+5}u\|_{U_{\Delta}^{2}(J_{l}\times\R^{2})} \|u\|_{X_{\Delta}^{2,k}(J_{l}\times\R^{2})}^{2} \\
&\phantom{=} \lesssim_{\delta,\delta'} 2^{(l_{2}-i)(1-\delta)} \paren*{\epsilon_{3}^{-1/2}2^{-i}N(J_{l})}^{\delta'} \|P_{\xi(G_{\alpha}^{i}), i-5\leq\cdot\leq i+5}u\|_{U_{\Delta}^{2}(J_{l}\times\R^{2})}
\end{split}
\end{equation}
for any $0<\delta',\delta<1$ satisfying $\delta'+\delta<1$. We choose $\delta'=\frac{1}{2}$.

Now summing over $J_{l}$ and then over $l_{2}$, we obtain that
\begin{align}
&\sum_{0\leq l_{2}\leq i-10} \paren*{\sum_{J_{l}\subset G_{\alpha}^{i}; N(J_{l})\geq\epsilon_{3}^{1/2}2^{l_{2}-6}} \paren*{2^{(l_{2}-i)(1-\delta)} \paren*{\epsilon_{3}^{-1/2}2^{-i}N(J_{l})}^{1/2} \|P_{\xi(G_{\alpha}^{i}), i-5\leq\cdot\leq i+5}u\|_{U_{\Delta}^{2}(J_{l}\times\R^{2})}}^{2}}^{1/2} \nonumber\\
&\phantom{=}\leq \|P_{\xi(G_{\alpha}^{i}),i-5\leq\cdot\leq i+5}u\|_{U_{\Delta}^{2}(G_{\alpha}^{i}\times\R^{2})}\sum_{0\leq l_{2}\leq i-10}2^{(l_{2}-i)(1-\delta)}\paren*{\sum_{J_{l}\subset G_{\alpha}^{i}; N(J_{l})\geq\epsilon_{3}^{1/2}2^{l_{2}-6}} \epsilon_{3}^{-1/2}2^{-i}N(J_{l})}^{1/2} \nonumber\\
&\phantom{=}\lesssim_{u} \|P_{\xi(G_{\alpha}^{i}),i-5\leq\cdot\leq i+5}u\|_{U_{\Delta}^{2}(G_{\alpha}^{i}\times\R^{2})}\sum_{0\leq l_{2}\leq i-10}2^{(l_{2}-i)(1-\delta)}\paren*{\epsilon_{3}^{-1/2}2^{-i}\int_{G_{\alpha}^{i}}N(t)^{3}dt}^{1/2} \nonumber\\
&\phantom{=}\lesssim \epsilon_{3}^{1/4}\|P_{\xi(G_{\alpha}^{i}),i-5\leq\cdot\leq i+5}u\|_{U_{\Delta}^{2}(G_{\alpha}^{i}\times\R^{2})} \nonumber\\
&\phantom{=}\leq \epsilon_{2}^{2}\|P_{\xi(G_{\alpha}^{i}),i-5\leq\cdot\leq i+5}u\|_{U_{\Delta}^{2}(G_{\alpha}^{i}\times\R^{2})},
\end{align}
since $\epsilon_{3} \leq\epsilon_{2}^{10}$ by assumption.

Thus, we have shown that
\begin{equation}
\eqref{eq:LTSE_b_h1} \lesssim \epsilon_{2}^{2}\|P_{\xi(G_{\alpha}^{i}),i-5\leq\cdot\leq i+5}u\|_{U_{\Delta}^{2}(G_{\alpha}^{i}\times\R^{2})}.
\end{equation}

\item[Estimate for \eqref{eq:LTSE_b_h2}:]
To estimate this term, we use an idea of \cite{Dodson2016}, which is to establish an improved bilinear Strichartz estimate, specialized to our setting and which takes into account the variation of $\xi(t)$, using the interaction Morawetz technique of \cite{Planchon2009}. 

\begin{restatable}[Improved BSE 1]{prop}{ibsI}
\label{prop:ibs_1}
There exists a constant $C(u)>0$ such that the following holds: for any integers $0\leq l_{2}\leq i-10$, $20\leq i<j$, $20\leq j\leq k_{0}$; any intervals $J_{l}\subset G_{\alpha}^{i}\subset G_{\kappa}^{j}\subset [0,T]$ with $J_{l}$ small; any $v_{0}\in\mathcal{S}(\R^{2})$ with Fourier support in a cube of side length $2^{k}$, where $k\leq l_{2}+2$, intersecting the dyadic annulus $A(\xi(G_{\alpha}^{i}), i-6,i+6)$, we have the estimate
\begin{equation}
\begin{split}
\|(e^{it\Delta}v_{0}) (P_{\xi(t), \leq l_{2}}u)\|_{L_{t,x}^{2}(J_{l}\times\R^{2})}^{2} &\leq C(u) 2^{k-i} \|v_{0}\|_{L^{2}(\R^{2})}^{2} \\
&\phantom{=}+C(u)2^{k-l_{2}-i} \|v_{0}\|_{L^{2}(\R^{2})}^{2} \paren*{\int_{J_{l}} |\xi'(t)| \sum_{l_{1}\leq l_{2}} 2^{\frac{l_{1}-l_{2}}{2}} \|P_{\xi(t),l_{1}}u(t)\|_{L_{x}^{2}(\R^{2})} \|P_{\xi(t),l_{2}}u(t)\|_{L_{x}^{2}(\R^{2})}dt}.
\end{split}
\end{equation}
The same estimate holds with $P_{\xi(t), \leq l_{2}}$ replaced by $P_{\xi(t),l_{2}}$.
\end{restatable}

\begin{remark}
The constant $C(u)$ in the statement of proposition \ref{prop:ibs_1}, and also for propositions \ref{prop:ibs_2} and \ref{prop:ibs_3} in the sequel, only depends on a given admissible blowup solution $u$ through its mass (it is polynomial in $M(u)$), not the particular rescaling $u_{\lambda}$ in the statement of theorem \ref{thm:LTSE}. Recall that we made a decision to drop the subscript $\lambda$ in $u_{\lambda}$ at the beginning of the proof of theorem \ref{thm:LTSE} for notational convenience.
\end{remark}

We postpone the proof of proposition \ref{prop:ibs_1} until subsection \ref{ssec:BSE_1} to instead show how to use \ref{prop:ibs_1} to obtain an acceptable estimate for \eqref{eq:LTSE_b_h2}.

By duality, Plancherel's theorem, and an approximation argument, to estimate
\begin{equation}
\|\int_{J_{l}}e^{-it\Delta}P_{\xi(G_{\alpha}^{i}), i-2\leq\cdot\leq i+2}\brak*{(P_{\xi(G_{\alpha}^{i}), i-5\leq\cdot\leq i+5}u(t))\E\paren*{(P_{\xi(t), l_{2}}u(t))(\ol{P_{\xi(t), \leq l_{2}}u(t)})}}dt\|_{L_{x}^{2}(\R^{2})},
\end{equation}
it suffices to estimate
\begin{equation}
\left|\ipp{e^{it\Delta}v_{0},(P_{\xi(G_{\alpha}^{i}), i-5\leq\cdot\leq i+5}u)\E\paren*{(P_{\xi(t), l_{2}}u)(\ol{P_{\xi(t), \leq l_{2}}u})}}_{J_{l}}\right|
\end{equation}
for $v_{0}\in\mathcal{S}(\R^{2})$ with Fourier support in the dyadic annulus $A(\xi(G_{\alpha}^{i}), i-3, i+3)$ and $\|v_{0}\|_{L^{2}(\R^{2})}\leq 1$. Now by the double frequency decomposition (with the smooth projectors $\P_{Q_{a}^{k}}$), Minkowski's inequality, and Cauchy-Schwarz,
\begin{align}
&\|(e^{it\Delta}v_{0}) (P_{\xi(G_{\alpha}^{i}),i-5\leq\cdot\leq i+5}u)\E\paren*{(P_{\xi(t),l_{2}}u)(\ol{P_{\xi(t),\leq l_{2}}u})}\|_{L_{t,x}^{1}(J_{l}\times\R^{2})} \nonumber\\
&\phantom{=} \lesssim \sum_{k\leq l_{2}+2} \sup_{y\in\R^{2}} \|(e^{it\Delta}v_{0})(P_{\xi(t),l_{2}}\P_{Q_{a}^{k}}\tau_{y}u)\|_{\ell_{a}^{2}L_{t,x}^{2}(\Z^{2}\times J_{l}\times\R^{2})} \|(P_{\xi(G_{\alpha}^{i}),i-5\leq\cdot\leq i+5}u)(\ol{P_{\xi(t),\leq l_{2}}\P_{Q_{a}^{k}}\tau_{y}u})\|_{\ell_{a}^{2}L_{t,x}^{2}(\Z^{2}\times J_{l}\times\R^{2})}.
\end{align}
Take the first factor. We want to apply proposition \ref{prop:ibs_1} to it, but the reader will observe that the cube localization is with respect to $P_{\xi(t),l_{2}}u$, while in the statement of proposition \ref{prop:ibs_1}, the localization is with respect to $v_{0}$. Therefore we need a lemma to shift the cube localization between the factors in the integrand. Applying lemma \ref{lem:sparse} with $v=e^{it\Delta}v_{0}$ and $w=P_{\xi(t), l_{2}}u$ and then applying proposition \ref{prop:ibs_1}, we obtain that
\begin{align}
&\sup_{y\in\R^{2}} \|(e^{it\Delta}v_{0})(\P_{Q_{a}^{k}}P_{\xi(t),l_{2}}\tau_{y}u)\|_{\ell_{a}^{2}L_{t,x}^{2}(\Z^{2}\times J_{l}\times\R^{2})} \nonumber\\
&\phantom{=}\lesssim \sup_{y\in\R^{2}} \|(e^{it\Delta}P_{Q_{a}^{k}}v_{0}) (P_{\xi(t),l_{2}}\tau_{y}u)\|_{\ell_{a}^{2}L_{t,x}^{2}(\Z^{2}\times J_{l}\times\R^{2})} \nonumber\\
&\phantom{=}\lesssim_{u} 2^{\frac{k-i}{2}}\|P_{Q_{a}^{k}}v_{0}\|_{\ell_{a}^{2}L_{x}^{2}(\Z^{2}\times\R^{2})} \nonumber\\
&\phantom{=}\qquad + 2^{\frac{k-l_{2}-i}{2}} \|P_{Q_{a}^{k}}v_{0}\|_{\ell_{a}^{2}L_{x}^{2}(\Z^{2}\times\R^{2})} \paren*{\int_{J_{l}} |\xi'(t)| \sum_{l_{1}\leq l_{2}} 2^{\frac{l_{1}-l_{2}}{2}} \|P_{\xi(t),l_{1}}u(t)\|_{L_{x}^{2}(\R^{2})} \|P_{\xi(t),l_{2}}u(t)\|_{L_{x}^{2}(\R^{2})}dt}^{1/2}  \nonumber\\
&\phantom{=}\lesssim 2^{\frac{k-i}{2}} + 2^{\frac{k-l_{2}-i}{2}}\paren*{\int_{J_{l}}|\xi'(t)| \sum_{l_{1}\leq l_{2}} 2^{\frac{l_{1}-l_{2}}{2}} \|P_{\xi(t),l_{1}}u(t)\|_{L_{x}^{2}(\R^{2})} \|P_{\xi(t),l_{2}}u(t)\|_{L_{x}^{2}(\R^{2})}dt}^{1/2},
\end{align}
where we use that Plancherel's theorem and $\|v_{0}\|_{L^{2}(\R^{2})}\leq 1$ to obtain the ultimate inequality. Now take the second factor
\begin{equation}
\|(P_{\xi(G_{\alpha}^{i}),i-5\leq\cdot\leq i+5}u)(\ol{P_{\xi(t),\leq l_{2}}\P_{Q_{a}^{k}}\tau_{y}u})\|_{\ell_{a}^{2}L_{t,x}^{2}(\Z^{2}\times J_{l}\times\R^{2})}.
\end{equation}
We use the atomic decomposition of $P_{\xi(G_{\alpha}^{i}),i-5\leq\cdot\leq i+5}u$ and apply lemma \ref{lem:sparse} followed by proposition \ref{prop:ibs_1} to each atom of $P_{\xi(G_{\alpha}^{i}),i-5\leq\cdot\leq i+5}u$ to obtain
\begin{equation}
\begin{split}
&\sup_{y\in\R^{2}} \|(P_{\xi(G_{\alpha}^{i}),i-5\leq\cdot\leq i+5}u)(\ol{P_{\xi(t),\leq l_{2}}\P_{Q_{a}^{k}}\tau_{y}u})\|_{\ell_{a}^{2}L_{t,x}^{2}(\Z^{2}\times J_{l}\times\R^{2})} \\
&\phantom{=}\lesssim_{u} 2^{\frac{k-i}{2}} \|P_{\xi(G_{\alpha}^{i}),i-5\leq\cdot\leq i+5}u\|_{U_{\Delta}^{2}(J_{l}\times\R^{2})} \\
&\phantom{=} \qquad + 2^{\frac{k-l_{2}-i}{2}} \|P_{\xi(G_{\alpha}^{i}),i-5\leq\cdot\leq i+5}u\|_{U_{\Delta}^{2}(J_{l}\times\R^{2})} \paren*{\int_{J_{l}} |\xi'(t)|  \sum_{l_{1}\leq l_{2}} 2^{\frac{l_{1}-l_{2}}{2}} \|P_{\xi(t),l_{1}}u(t)\|_{L_{x}^{2}(\R^{2})} \|P_{\xi(t),l_{2}}u(t)\|_{L_{x}^{2}(\R^{2})}dt}^{1/2}.
\end{split}
\end{equation}
Combining these estimates for the first and second factors and then using Cauchy-Schwarz, we obtain that
\begin{equation}
\begin{split}
&\sum_{k\leq l_{2}+2}  \sup_{y\in\R^{2}} \|(e^{it\Delta}v_{0})(P_{\xi(t),l_{2}}\P_{Q_{a}^{k}}\tau_{y}u)\|_{\ell_{a}^{2}L_{t,x}^{2}(\Z^{2}\times J_{l}\times\R^{2})} \|(P_{\xi(G_{\alpha}^{i}),i-5\leq\cdot\leq i+5}u)(\ol{P_{\xi(t),\leq l_{2}}\P_{Q_{a}^{k}}\tau_{y}u})\|_{\ell_{a}^{2}L_{t,x}^{2}(\Z^{2}\times J_{l}\times\R^{2})} \\
&\phantom{=}\lesssim_{u} 2^{l_{2}-i} \|P_{\xi(G_{\alpha}^{i}),i-5\leq\cdot\leq i+5}u\|_{U_{\Delta}^{2}(J_{l}\times\R^{2})} \\
&\phantom{=}\qquad + 2^{-i} \|P_{\xi(G_{\alpha}^{i}),i-5\leq\cdot\leq i+5}u\|_{U_{\Delta}^{2}(J_{l}\times\R^{2})} \int_{J_{l}} |\xi'(t)|  \sum_{l_{1}\leq l_{2}} 2^{\frac{l_{1}-l_{2}}{2}} \|P_{\xi(t),l_{1}}u(t)\|_{L_{x}^{2}(\R^{2})} \|P_{\xi(t),l_{2}}u(t)\|_{L_{x}^{2}(\R^{2})}dt.
\end{split}
\end{equation}
Now summing over $J_{l}$, then over $l_{2}$, we obtain that
\begin{align}
&\sum_{0\leq l_{2}\leq i-10}\sum_{J_{l}\subset G_{\alpha}^{i}; N(J_{l})\geq \epsilon_{3}^{1/2}2^{l_{2}-6}} 2^{l_{2}-i} \|P_{\xi(G_{\alpha}^{i}),i-5\leq\cdot\leq i+5}u\|_{U_{\Delta}^{2}(J_{l}\times\R^{2})} \nonumber\\
&\phantom{=}\lesssim \|P_{\xi(G_{\alpha}^{i}),i-5\leq\cdot\leq i+5}u\|_{U_{\Delta}^{2}(G_{\alpha}^{i}\times\R^{2})} \sum_{J_{l}\subset G_{\alpha}^{i}} \sum_{0\leq l_{2}\leq \min\{i-10,\log_{2}(2^{6}\epsilon_{3}^{-1/2}N(J_{l}))\}} 2^{l_{2}-i} \nonumber\\
&\phantom{=}\lesssim 2^{-i} \|P_{\xi(G_{\alpha}^{i}),i-5\leq\cdot\leq i+5}u\|_{U_{\Delta}^{2}(G_{\alpha}^{i}\times\R^{2})}\sum_{J_{l}\subset G_{\alpha}^{i}} \epsilon_{3}^{-1/2} N(J_{l}) \nonumber\\
&\phantom{=}\lesssim_{u} 2^{-i}  \|P_{\xi(G_{\alpha}^{i}),i-5\leq\cdot\leq i+5}u\|_{U_{\Delta}^{2}(G_{\alpha}^{i}\times\R^{2})} \epsilon_{3}^{-1/2} \int_{G_{\alpha}^{i}}N(t)^{3}dt \nonumber\\
&\phantom{=}\lesssim \epsilon_{3}^{1/2}  \|P_{\xi(G_{\alpha}^{i}),i-5\leq\cdot\leq i+5}u\|_{U_{\Delta}^{2}(G_{\alpha}^{i}\times\R^{2})}.
\end{align}
Also,
\begin{align}
&\sum_{0\leq l_{2}\leq i-10}\sum_{J_{l}\subset G_{\alpha}^{i}; N(J_{l})\geq \epsilon_{3}^{1/2}2^{l_{2}-6}} 2^{-i/2} \|P_{\xi(G_{\alpha}^{i}),i-5\leq\cdot\leq i+5}u\|_{U_{\Delta}^{2}(J_{l}\times\R^{2})} \int_{J_{l}} |\xi'(t)|  \sum_{l_{1}\leq l_{2}} 2^{\frac{l_{1}-l_{2}}{2}} \|P_{\xi(t),l_{1}}u(t)\|_{L_{x}^{2}(\R^{2})} \|P_{\xi(t),l_{2}}u(t)\|_{L_{x}^{2}(\R^{2})}dt  \nonumber\\
&\phantom{=}\lesssim 2^{-i}\|P_{\xi(G_{\alpha}^{i}),i-5\leq\cdot\leq i+5}u\|_{U_{\Delta}^{2}(J_{l}\times\R^{2})}\sum_{J_{l}\subset G_{\alpha}^{i}} \sum_{0\leq l_{2}\leq i-10}\int_{J_{l}} |\xi'(t)|\sum_{l_{1}\leq l_{2}} 2^{\frac{l_{1}-l_{2}}{2}} \|P_{\xi(t),l_{1}}u(t)\|_{L_{x}^{2}(\R^{2})} \|P_{\xi(t), l_{2}}u(t)\|_{L_{x}^{2}(\R^{2})}dt& \nonumber \\
&\phantom{=}\lesssim 2^{-i}\|P_{\xi(G_{\alpha}^{i}),i-5\leq\cdot\leq i+5}u\|_{U_{\Delta}^{2}(G_{\alpha}^{i}\times\R^{2})} \epsilon_{1}^{-1/2}\int_{G_{\alpha}^{i}} N(t)^{3}dt \nonumber\\
&\phantom{=}\lesssim \epsilon_{3}^{1/2} \|P_{\xi(G_{\alpha}^{i}),i-5\leq\cdot\leq i+5}u\|_{U_{\Delta}^{2}(G_{\alpha}^{i}\times\R^{2})},
\end{align}
where we use Schur's test, Plancherel's theorem and mass conservation to obtain the penultimate inequality. Bookkeeping our estimates, we have shown that
\begin{equation}
\eqref{eq:LTSE_b_h2} \lesssim_{u} \epsilon_{3}^{1/2}\|P_{\xi(G_{\alpha}^{i}), i-5\leq\cdot\leq i+5}u\|_{U_{\Delta}^{2}(G_{\alpha}^{i}\times\R^{2})} \leq \epsilon_{2}^{5}\|P_{\xi(G_{\alpha}^{i}), i-5\leq\cdot\leq i+5}u\|_{U_{\Delta}^{2}(G_{\alpha}^{i}\times\R^{2})}.
\end{equation}

\begin{remark}
The reader might wonder why we did not just prove proposition \ref{prop:ibs_1} with the cube localization applied to $P_{\xi(t),\leq l_{2}}u$. The reason is that we would then have to consider commutators of the form $P_{\xi(t),Q_{a}^{k}}F(u)-F(P_{\xi(t),Q_{a}^{k}}u)$ as opposed to commutators of the form $P_{\xi(t), \leq l_{2}}F(u)-F(P_{\xi(t),\leq l_{2}}u)$. The latter are easier to estimate.
\end{remark}

\item[Estimate for \eqref{eq:LTSE_b_h3}:]
As before, we want to use bilinear estimates exploiting the frequency separation between the various factors. However, since we are considering frequencies which are larger than $N(t)$, we will also need to use the $\tilde{Y}_{j}$ norm to obtain a source of smallness. As before, we need an improved bilinear Strichartz estimate, which is contained in the next proposition, the proof of which we defer to subsection \ref{ssec:BSE_2}.

\begin{restatable}[Improved BSE 2]{prop}{ibsII}\label{prop:ibs_2} %Second improved bilinear Strichartz estimate
There exists a constant $C(u)>0$ such that the following holds: for any integers $20\leq i\leq j\leq k_{0}$, $0\leq l_{2}\leq i-10$, and $k \leq l_{2}+2$; any intervals $G_{\beta}^{l_{2}}\subset G_{\alpha}^{i}\subset G_{\kappa}^{j}\subset [0,T]$; any $v_{0}\in\mathcal{S}(\R^{2})$ with Fourier support in a cube of side length $2^{k}$ intersecting the dyadic annulus $A(\xi(G_{\alpha}^{i}), i-6, i+6)$, we have the estimate
\begin{equation}
\|(e^{it\Delta}v_{0})(P_{\xi(t),\leq l_{2}}u)\|_{L_{t,x}^{2}(G_{\beta}^{l_{2}}\times\R^{2})}^{2} \leq C(u)2^{k-l_{2}}\|v_{0}\|_{L^{2}(\R^{2})}^{2}\paren*{1+\|u\|_{\tilde{X}_{i}(G_{\alpha}^{i}\times\R^{2})}^{4}}.
\end{equation}
The same estimate holds with $P_{\xi(t),\leq l_{2}}$ replaced by $P_{\xi(t),l_{2}}$.
\end{restatable}

We now apply proposition \ref{prop:ibs_2} to obtain an acceptable estimate for \eqref{eq:LTSE_b_h3}. We have by duality that for all $0\leq l_{2}\leq i-10$ and $G_{\beta}^{l_{2}}\subset G_{\alpha}^{i}$,
\begin{equation}
\begin{split}
&\|P_{\xi(G_{\alpha}^{i}), i-2\leq\cdot\leq i+2}\brak*{(P_{\xi(G_{\alpha}^{i}),i-5\leq\cdot\leq i+5}u)\E\paren*{(P_{\xi(t),l_{2}}u)(\ol{P_{\xi(t),\leq l_{2}}u})}}\|_{V_{\Delta}^{2}(G_{\beta}^{l_{2}}\times\R^{2})^{*}} \\
&\phantom{=}\leq \sup_{\|v_{\beta}^{l_{2}}\|_{V_{\Delta}^{2}(G_{\beta}^{l_{2}}\times\R^{2})}=1} \|(P_{\xi(G_{\alpha}^{i}),i-5\leq\cdot\leq i+5}u)\E((P_{\xi(t),l_{2}}u)\ol{(P_{\xi(t),\leq l_{2}}u)}) (P_{\xi(G_{\alpha}^{i}),i-2\leq\cdot\leq i+2}v_{\beta}^{l_{2}})\|_{L_{t,x}^{1}(G_{\beta}^{l_{2}}\times\R^{2})}
\end{split}
\end{equation}
Renormalizing and relabeling if necessary, we may may replace $P_{\xi(G_{\alpha}^{i}),i-2\leq\cdot\leq i+2}v_{\beta}^{l_{2}}$ with $v_{\beta}^{l_{2}}$ and assume that $v_{\beta}^{l_{2}}$ has spatial Fourier support in the dyadic annulus $A(\xi(G_{\alpha}^{i}),i-3,i+3)$. Moreover, by an approximation argument, we may assume that $v_{\beta}^{l_{2}}$ is spatially Schwartz. Using the double frequency decomposition and Cauchy-Schwarz, we obtain that
\begin{equation}
\begin{split}
&\|(P_{\xi(G_{\alpha}^{i}),i-5\leq\cdot\leq i+5}u)\E\paren*{(P_{\xi(t),l_{2}}u) (\ol{P_{\xi(t),\leq l_{2}}u})} (P_{\xi(G_{\alpha}^{i}),i-2\leq\cdot\leq i+2}v_{\beta}^{l_{2}})\|_{L_{t,x}^{1}(G_{\beta}^{l_{2}}\times\R^{2})} \\
&\phantom{=}\lesssim \sum_{k\leq l_{2}+2}\int_{\R^{2}}dy 2^{2k}\jp{2^{k}y}^{-10} \|v_{\beta}^{l_{2}}(\P_{Q_{a}^{k}}P_{\xi(t),l_{2}}\tau_{y}u)\|_{\ell_{a}^{2}L_{t,x}^{2}(\Z^{2}\times G_{\beta}^{l_{2}}\times\R^{2})}  \|(P_{\xi(G_{\alpha}^{i}),i-5\leq\cdot\leq i+5}u) (\ol{\P_{Q_{a}^{k}}P_{\xi(t),\leq l_{2}}\tau_{y}u})\|_{\ell_{a}^{2}L_{t,x}^{2}(\Z^{2}\times G_{\beta}^{l_{2}}\times\R^{2})}.
\end{split}
\end{equation}
Now applying Minkowski's inequality to interchange the $G_{\beta}^{l_{2}}$ summation and the $l$ summation and then grouping the intervals $G_{\beta}^{l_{2}}$ into $2^{i-10}$ larger intervals $G_{\gamma}^{i-10}$, we see that
\begin{align}
&\Bigg(\sum_{G_{\beta}^{l_{2}}\subset G_{\alpha}^{i}; N(G_{\beta}^{l_{2}})\leq \epsilon_{3}^{1/2}2^{l_{2}-5}} \Bigg(\sum_{k\leq l_{2}+2}\int_{\R^{2}}2^{2k}\langle{2^{k}y}\rangle^{-10} \|v_{\beta}^{l_{2}}(\P_{Q_{a}^{k}}P_{\xi(t),l_{2}}\tau_{y}u)\|_{\ell_{a}^{2}L_{t,x}^{2}(\Z^{2}\times G_{\beta}^{l_{2}}\times\R^{2})} \times \nonumber\\
&\phantom{=}\qquad \|(P_{\xi(G_{\alpha}^{i}),i-5\leq\cdot\leq i+5}u) (\ol{\P_{Q_{a}^{k}}P_{\xi(t),\leq l_{2}}\tau_{y}u})\|_{\ell_{a}^{2}L_{t,x}^{2}(\Z^{2}\times G_{\beta}^{l_{2}}\times\R^{2})}\Bigg)^{2} \Bigg)^{1/2} \nonumber\\
&\phantom{=}\leq \sum_{k\leq l_{2}+2}\int_{\R^{2}}dy 2^{2k}\jp{2^{k}y}^{-10} \Bigg(\sum_{G_{\beta}^{l_{2}}\subset G_{\alpha}^{i}; N(G_{\beta}^{l_{2}})\leq \epsilon_{3}^{1/2}2^{l_{2}-5}} \|v_{\beta}^{l_{2}}(\P_{Q_{a}^{k}}P_{\xi(t),l_{2}}\tau_{y}u)\|_{\ell_{a}^{2}L_{t,x}^{2}(\Z^{2}\times G_{\beta}^{l_{2}}\times\R^{2})}^{2} \times \nonumber\\
&\phantom{=}\qquad \|(P_{\xi(G_{\alpha}^{i}),i-5\leq\cdot\leq i+5}u) (\ol{\P_{Q_{a'}^{k}}P_{\xi(t),\leq l_{2}}\tau_{y}u})\|_{\ell_{a}^{2}L_{t,x}^{2}(\Z^{2}\times G_{\beta}^{l_{2}}\times\R^{2})}^{2} \Bigg)^{1/2} \nonumber\\
&\phantom{=}\leq \sum_{k\leq l_{2}+2}\int_{\R^{2}}dy 2^{2k}\jp{2^{k}y}^{-10} \paren*{\sup_{G_{\beta}^{l_{2}}\subset G_{\alpha}^{i}; N(G_{\beta}^{l_{2}}) \leq \epsilon_{3}^{1/2}2^{l_{2}-5}} \|v_{\beta}^{l_{2}}(\P_{Q_{a}^{k}}P_{\xi(t),l_{2}}\tau_{y}u)\|_{\ell_{a}^{2}L_{t,x}^{2}(\Z^{2}\times G_{\beta}^{l_{2}}\times\R^{2})}} \times \nonumber\\
&\phantom{=}\qquad \paren*{\sum_{G_{\gamma}^{i-10}\subset G_{\alpha}^{i}} \|(P_{\xi(G_{\alpha}^{i}),i-5\leq\cdot\leq i+5}u)(\ol{\P_{Q_{a}^{k}}P_{\xi(t),\leq l_{2}}\tau_{y}u})\|_{\ell_{a}^{2}L_{t,x}^{2}(\Z^{2} \times G_{\gamma}^{i-10}\times\R^{2})}^{2}}^{1/2}. \label{eq:LTSE_b_h3_app}
\end{align}
Using lemma \ref{lem:sparse}, the bilinear Strichartz estimate of proposition \ref{prop:bs_i}, the embedding $U_{\Delta}^{2} \subset X_{\Delta}^{2,k}$, and spatial translation invariance, we have that
\begin{align}
\sup_{y\in\R^{2}} \|v_{\beta}^{l_{2}}(\P_{Q_{a}^{k}}P_{\xi(t),l_{2}}\tau_{y}u)\|_{\ell_{a}^{2}L_{t,x}^{2}(\Z^{2}\times G_{\beta}^{l_{2}}\times\R^{2})} &\lesssim 2^{(k-i){\frac{1}{2}}^{-}} \|v_{\beta}^{l_{2}}\|_{V_{\Delta}^{2}(G_{\beta}^{l_{2}}\times\R^{2})} \|P_{\xi(G_{\beta}^{l_{2}}),l_{2}-2\leq\cdot\leq l_{2}+2}u\|_{U_{\Delta}^{2}(G_{\beta}^{l_{2}}\times\R^{2})} \nonumber\\
&\lesssim 2^{(k-i){\frac{1}{2}}^{-}} \|u\|_{\tilde{Y}_{i}(G_{\alpha}^{i}\times\R^{2})}, \label{eq:LTSE_ibs3_comm}
\end{align}
for all $G_{\beta}^{l_{2}}\subset G_{\alpha}^{i}$ with $N(G_{\beta}^{l_{2}}) \leq \epsilon_{3}^{1/2}2^{l_{2}-5}$. Using lemma \ref{lem:sparse} and proposition \ref{prop:ibs_2} applied to each atom of $P_{\xi(G_{\alpha}^{i}), i-5\leq\cdot\leq i+5}u$, we have that for every $G_{\gamma}^{i-10}\subset G_{\alpha}^{i}$,
\begin{equation}\label{eq:LTSE_ibs2_app}
\sup_{y\in\R^{2}}\|(P_{\xi(G_{\alpha}^{i}), i-5\leq\cdot\leq i+5}u) (\ol{\P_{Q_{a}^{k}}P_{\xi(t),\leq l_{2}}\tau_{y}u})\|_{\ell_{a}^{2}L_{t,x}^{2}(\Z^{2}\times G_{\gamma}^{i-10}\times\R^{2})} \lesssim_{u} \|P_{\xi(G_{\alpha}^{i}), i-5\leq\cdot\leq i+5}u\|_{U_{\Delta}^{2}(G_{\alpha}^{i}\times\R^{2})} \paren*{1+\|u\|_{\tilde{X}_{i}(G_{\alpha}^{i}\times\R^{2})}^{2}}.
\end{equation}
Combining these two estimates, integrating with respect to $y\in\R^{2}$, and then summing over $k\leq l_{2}+2$, we conclude that
\begin{equation}
\eqref{eq:LTSE_b_h3_app} \lesssim_{u} 2^{(l_{2}-i)\frac{1}{2}^{-}} \|P_{\xi(G_{\alpha}^{i}),i-5\leq\cdot\leq i+5}u\|_{U_{\Delta}^{2}(G_{\alpha}^{i}\times\R^{2})} \|u\|_{\tilde{Y}_{i}(G_{\alpha}^{i}\times\R^{2})} \paren*{1+\|u\|_{\tilde{X}_{i}(G_{\alpha}^{i}\times\R^{2})}^{2}}.
\end{equation}
Summing over $0\leq l_{2}\leq i-10$, we conclude that
\begin{equation}
\eqref{eq:LTSE_b_h3} \lesssim_{u} \|P_{\xi(G_{\alpha}^{i}),i-5\leq\cdot\leq i+5}u\|_{U_{\Delta}^{2}(G_{\alpha}^{i}\times\R^{2})} \|u\|_{\tilde{Y}_{i}(G_{\alpha}^{i}\times\R^{2})} \paren*{1+\|u\|_{\tilde{X}_{i}(G_{\alpha}^{i}\times\R^{2})}^{2}}.
\end{equation}

\item[Estimate for \eqref{eq:LTSE_b_h4}:]
Our strategy is similar to the argument for estimating \eqref{eq:LTSE_b_h3}, but now we improve the preceding bilinear Strichartz proposition \ref{prop:ibs_2} from an $\ell^{\infty}$ estimate over $0\leq l_{2}\leq i-10$ to an $\ell^{2}$ estimate over $0\leq l_{2}\leq i-10$. The precise result is the following proposition, the proof of which we defer to subsection \ref{ssec:BSE_3}.

\begin{restatable}[Improved BSE 3]{prop}{ibsIII}
\label{prop:ibs_3} %Third bilinear Strichartz estimate
There exists a constant $C(u)>0$ such that the following holds: for any integers $20\leq i\leq j\leq k_{0}$, any intervals $G_{\alpha}^{i}\subset G_{\kappa}^{j}\subset [0,T]$, and any $v_{0}\in\mathcal{S}(\R^{2})$ with Fourier support in the dyadic annulus $A(\xi(G_{\alpha}^{i}),i-6,i+6)$, we have the estimate
\begin{equation}
\sum_{0\leq l_{2}\leq i-10} \paren*{\sup_{k\leq l_{2}+2} \|(e^{it\Delta}P_{Q_{a}^{k}}v_{0})(P_{\xi(t),\leq l_{2}}u)\|_{\ell_{a}^{2}L_{t,x}^{2}(\Z^{2}\times G_{\alpha}^{i}\times\R^{2})}}^{2} \leq C(u)\|v_{0}\|_{L^{2}(\R^{2})}^{2} \paren*{1+\|u\|_{\tilde{X}_{i}(G_{\alpha}^{i}\times\R^{2})}^{6}}.
\end{equation}
The same estimate holds with $P_{\xi(t),\leq l_{2}}$ replaced by $P_{\xi(t),l_{2}}$.
\end{restatable}

We now use proposition \ref{prop:ibs_3} to obtain an acceptable estimate for \eqref{eq:LTSE_b_h4}. For each interval $G_{\beta}^{l_{2}} \subset G_{\alpha}^{i}$, we have by duality that
\begin{equation}
\begin{split}
&\|\int_{G_{\beta}^{l_{2}}} e^{-it\Delta}P_{\xi(G_{\alpha}^{i}), i-2\leq\cdot\leq i+2}\brak*{(P_{\xi(G_{\alpha}^{i}),i-5\leq\cdot\leq i+5}u) \E\paren*{(P_{\xi(t),l_{2}}u)(\ol{P_{\xi(t),\leq l_{2}}u})}}dt\|_{L_{x}^{2}(\R^{2})} \\
&\phantom{=}= \sup_{\|v_{0}\|_{L^{2}(\R^{2})}=1} \int_{G_{\beta}^{l_{2}}} \ipp{e^{it\Delta}P_{\xi(G_{\alpha}^{i}), i-2\leq\cdot\leq i+2}v_{0}, (P_{\xi(G_{\alpha}^{i}), i-5\leq\cdot\leq i+5}u) \E\paren*{(P_{\xi(t), l_{2}}u)(\ol{P_{\xi(t), \leq l_{2}}u})}}dt. \label{eq:LTSE_b_h4_dual}
\end{split}
\end{equation}
For any $v_{0}\in\mathcal{S}(\R^{2})$ with $\|v_{0}\|_{L^{2}(\R^{2})}\leq 1$, we apply the double frequency decomposition, Cauchy-Schwarz, and the bilinear Strichartz estimate of proposition \ref{prop:bs_i} to obtain the estimate
\begin{equation}
\begin{split}
\eqref{eq:LTSE_b_h4_dual} &\lesssim \sum_{k\leq l_{2}+2} \int_{\R^{2}}dy 2^{2k}\jp{2^{k}y}^{-10} 2^{\frac{k-i}{2}}\|P_{\xi(G_{\beta}^{l_{2}}),l_{2}-2\leq\cdot\leq l_{2}+2}u\|_{U_{\Delta}^{2}(G_{\beta}^{l_{2}}\times\R^{2})} \times\\
&\phantom{=}\qquad \|(P_{\xi(G_{\alpha}^{i}), i-5\leq\cdot\leq i+5}u)(\ol{\P_{Q_{a}^{k}}P_{\xi(t),\leq l_{2}}\tau_{y}u})\|_{\ell_{a}^{2}L_{t,x}^{2}(\Z^{2}\times G_{\beta}^{l_{2}}\times\R^{2})}.
\end{split}
\end{equation}
For each $0\leq l_{2}\leq i-10$ fixed, we sum over the intervals $G_{\beta}^{l_{2}}$, then use Minkowski's inequality to interchange the order of the $y$ integration and $k$ summation with the $G_{\beta}^{l_{2}}$ summation, and then use Cauchy-Schwarz in $G_{\beta}^{l_{2}}$ to obtain
\begin{align}
\sum_{G_{\beta}^{l_{2}}\subset G_{\alpha}^{i}; N(G_{\beta}^{l_{2}}) \leq \epsilon_{3}^{1/2} 2^{l_{2}-5}} \eqref{eq:LTSE_b_h4_dual} &\lesssim \sum_{k\leq l_{2}+2}\int_{\R^{2}}dy 2^{2k}\jp{2^{k}y}^{-10} 2^{\frac{k-i}{2}} \paren*{\sum_{G_{\beta}^{l_{2}}\subset G_{\alpha}^{i}; N(G_{\beta}^{l_{2}}) \leq \epsilon_{3}^{1/2}2^{l_{2}-5}} \|P_{\xi(G_{\beta}^{l_{2}}),l_{2}-2\leq\cdot\leq l_{2}+2}u\|_{U_{\Delta}^{2}(G_{\beta}^{l_{2}}\times\R^{2})}^{2}}^{1/2} \nonumber\\
&\phantom{=}\qquad \times\paren*{\sum_{G_{\beta}^{l_{2}}\subset G_{\alpha}^{i}; N(G_{\beta}^{l_{2}})\leq \epsilon_{3}^{1/2}2^{l_{2}-5}} \|(P_{\xi(G_{\alpha}^{i}), i-5\leq\cdot\leq i+5}u)(\ol{\P_{Q_{a}^{k}}P_{\xi(t),\leq l_{2}}\tau_{y}u})\|_{\ell_{a}^{2}L_{t,x}^{2}(\Z^{2}\times G_{\beta}^{l_{2}}\times\R^{2})}^{2}}^{1/2} \nonumber\\
&\leq 2^{\frac{l_{2}-i}{2}}\paren*{\sum_{G_{\beta}^{l_{2}}\subset G_{\alpha}^{i}; N(G_{\beta}^{l_{2}}) \leq \epsilon_{3}^{1/2}2^{l_{2}-5}} \|P_{\xi(G_{\beta}^{l_{2}}), l_{2}-2\leq\cdot\leq l_{2}+2}u\|_{U_{\Delta}^{2}(G_{\beta}^{l_{2}}\times\R^{2})}^{2}}^{1/2}\times \nonumber\\
&\phantom{=}\qquad \sup_{y\in\R^{2}}\paren*{\sup_{k\leq l_{2}+2}\|(P_{\xi(G_{\alpha}^{i}), i-5\leq\cdot\leq i+5}u)(\ol{\P_{Q_{a}^{k}}P_{\xi(t),\leq l_{2}}\tau_{y}u})\|_{\ell_{a}^{2}L_{t,x}^{2}(\Z^{2}\times G_{\alpha}^{i}\times\R^{2})}}.\label{eq:LTSE_b_h4_l2}
\end{align}
Now summing over the range $0\leq l_{2}\leq i-10$, then using Cauchy-Schwarz in $l_{2}$, we obtain the estimate
\begin{equation}
\begin{split}
\sum_{0\leq l_{2}\leq i-10} \eqref{eq:LTSE_b_h4_l2} &\leq \paren*{\sum_{0\leq l_{2}\leq i-10}2^{l_{2}-i}\sum_{G_{\beta}^{l_{2}}\subset G_{\alpha}^{i}; N(G_{\beta}^{l_{2}}) \leq \epsilon_{3}^{1/2}2^{l_{2}-5}} \|P_{\xi(G_{\beta}^{l_{2}}), l_{2}-2\leq\cdot\leq l_{2}+2}u\|_{U_{\Delta}^{2}(G_{\beta}^{l_{2}}\times\R^{2})}^{2}}^{1/2}\times \nonumber\\
&\phantom{=}\quad \sup_{y\in\R^{2}}\paren*{\sum_{0\leq l_{2}\leq i-10} \paren*{\sup_{k\leq l_{2}+2} \|(P_{\xi(G_{\alpha}^{i}), i-5\leq\cdot\leq i+5}u)(\ol{\P_{Q_{a}^{k}}P_{\xi(t),\leq l_{2}}\tau_{y}u})\|_{\ell_{a}^{2}L_{t,x}^{2}(\Z^{2}\times G_{\alpha}^{i}\times\R^{2})}}^{2}}^{1/2} \label{eq:LTSE_b_h4_f}.
\end{split}
\end{equation}
The first factor in the product on the RHS of \eqref{eq:LTSE_b_h4_f} is $\lesssim \|u\|_{\tilde{Y}_{i}(G_{\alpha}^{i}\times\R^{2})}$ by definition of the $\tilde{Y}_{i}(G_{\alpha}^{i}\times\R^{2})$ norm. To estimate the second factor in the product, we use lemma \ref{lem:sparse} and then apply proposition \ref{prop:ibs_3} to the atoms of $\P_{Q_{a}^{k}}P_{\xi(G_{\alpha}^{i}), i-5\leq\cdot\leq i+5}u$ to conclude that
\begin{align}
\eqref{eq:LTSE_b_h4_f} &\lesssim_{u} \|u\|_{\tilde{Y}_{i}(G_{\alpha}^{i}\times\R^{2})} \|\P_{Q_{a}^{k}}P_{\xi(G_{\alpha}^{i}), i-5\leq\cdot\leq i+5}u\|_{\ell_{a}^{2}U_{\Delta}^{2}(\Z^{2}\times G_{\alpha}^{i}\times\R^{2})} \paren*{1+\|u\|_{\tilde{X}_{i}(G_{\alpha}^{i}\times\R^{2})}^{3}} \nonumber\\
&\lesssim \|u\|_{\tilde{Y}_{i}(G_{\alpha}^{i}\times\R^{2})} \|P_{\xi(G_{\alpha}^{i}), i-5\leq\cdot\leq i+5}u\|_{U_{\Delta}^{2}(G_{\alpha}^{i}\times\R^{2})} \paren*{1+\|u\|_{\tilde{X}_{i}(G_{\alpha}^{i}\times\R^{2})}^{3}}.
\end{align}
Hence, we have shown that
\begin{equation}
\eqref{eq:LTSE_b_h4} \lesssim_{u} \|u\|_{\tilde{Y}_{i}(G_{\alpha}^{i}\times\R^{2})} \|P_{\xi(G_{\alpha}^{i}), i-5\leq\cdot\leq i+5}u\|_{U_{\Delta}^{2}(G_{\alpha}^{i}\times\R^{2})} \paren*{1+\|u\|_{\tilde{X}_{i}(G_{\alpha}^{i}\times\R^{2})}^{3}}.
\end{equation}
\end{description}
After a bit of bookkeeping, we obtain the desired estimate in the statement of lemma \ref{lem:LTSE_b_hard}.
\end{proof}

Collecting the estimates of lemma \ref{lem:LTSE_b_easy} and corollary \ref{cor:LTSE_b_hard} completes the proof of lemma \ref{lem:LTSE_b}.

\subsection{Step 5: Conclusion of proof}\label{ssec:LTSE_S5} %Conclusion of proof of LTSE
We now show how the conclusion of the proof of theorem \ref{thm:LTSE} follows from Steps 1-4. We define $\tilde{C}(u)$ to be the maximum of the finitely many implicit constants obtained so far in the proof. Applying lemma \ref{lem:LTSE_b} and taking the maximum over integers $20\leq j\leq k_{0}$ and intervals $G_{\kappa}^{j}\subset [0,T]$, we have shown the following lemma.

\begin{lemma}\label{lem:LTSE_b_f}
For any $20\leq k\leq k_{0}$, we have that
\begin{align}
&\|u\|_{\tilde{X}_{k}([0,T]\times\R^{2})}^{2} \nonumber\\
&\phantom{=}\leq \tilde{C}(u)\paren*{1+\|u\|_{\tilde{Y}_{k}([0,T]\times\R^{2})}^{2}\paren*{\epsilon_{2}\|u\|_{\tilde{X}_{k}([0,T]\times\R^{2})}^{3}+\epsilon_{2}^{1/3}\|u\|_{\tilde{X}_{k}([0,T]\times\R^{2})}^{5/3}+\paren*{\epsilon_{2}+\|u\|_{\tilde{Y}_{k}([0,T]\times\R^{2})}\paren*{1+\|u\|_{\tilde{X}_{k}([0,T]\times\R^{2})}^{4}}}^{2}} } \\
&\|u\|_{\tilde{Y}_{k}([0,T]\times\R^{2})}^{2} \nonumber\\
&\phantom{=}\leq \tilde{C}(u)\paren*{\epsilon_{2}^{3/2}+\|u\|_{\tilde{Y}_{k}([0,T]\times\R^{2})}^{2}\paren*{\epsilon_{2}\|u\|_{\tilde{X}_{k}([0,T]\times\R^{2})}^{3}+\epsilon_{2}^{1/3}\|u\|_{\tilde{X}_{k}([0,T]\times\R^{2})}^{5/3}+\paren*{\epsilon_{2}+\|u\|_{\tilde{Y}_{k}([0,T]\times\R^{2})}\paren*{1+\|u\|_{\tilde{X}_{k}([0,T]\times\R^{2})}^{4}}}^{2}}}.
\end{align}
\end{lemma}

\begin{lemma}[Inductive step]\label{lem:LTSE_ind}
There exists $\epsilon_{2}(u)>0$ such that for all $u$-admissible tuples $(\epsilon_{1},\epsilon_{2},\epsilon_{3})$ with $0<\epsilon_{2}\leq \epsilon_{2}(u)$, the following holds: if for some integer $20\leq k_{*}\leq k_{0}$, we have that
\begin{align}
\|u\|_{\tilde{X}_{k_{*}}([0,T]\times\R^{2})}^{2} &\leq C_{0} \label{eq:LTSE_ind1}\\
\|u\|_{\tilde{Y}_{k_{*}}([0,T]\times\R^{2})}^{2} &\leq \epsilon_{2}, \label{eq:LTSE_ind2}
\end{align}
where $C_{0}\coloneqq 2^{-20}\tilde{C}(u)$, then
\begin{align}
\|u\|_{\tilde{X}_{k_{*}+1}([0,T]\times\R^{2})}^{2} &\leq C_{0}\\
\|u\|_{\tilde{Y}_{k_{*}+1}([0,T]\times\R^{2})}^{2} &\leq \epsilon_{2}.
\end{align}
\end{lemma}
\begin{proof}
Using lemma \ref{lem:LTSE_b_f} and the cheap inductive estimate (lemma \ref{lem:chp_ie}) to crudely estimate $\|u\|_{\tilde{X}_{k*+1}([0,T]\times\R^{2})}^{2}$ and $\|u\|_{\tilde{Y}_{k_{*}+1}([0,T]\times\R^{2})}^{2}$ on the RHS, we obtain that
\begin{align}
\|u\|_{\tilde{X}_{k_{*}+1}([0,T]\times\R^{2})}^{2} &\leq \tilde{C}(u)\Bigg(1+2\|u\|_{\tilde{Y}_{k_{*}}([0,T]\times\R^{2})}^{2}\Bigg(\epsilon_{2}(2\|u\|_{\tilde{X}_{k_{*}}([0,T]\times\R^{2})}^{2})^{3/2} + \epsilon_{2}^{1/3}(2\|u\|_{\tilde{X}_{k_{*}}([0,T]\times\R^{2})}^{2})^{5/6} \nonumber\\
&\phantom{=}\qquad+ \paren*{\epsilon_{2}+(2\|u\|_{\tilde{Y}_{k_{*}}([0,T]\times\R^{2})}^{2})^{1/2}\paren*{1+4\|u\|_{\tilde{X}_{k_{*}}([0,T]\times\R^{2})}^{4}}}^{2}\Bigg)\Bigg) \\
\|u\|_{\tilde{Y}_{k_{*}+1}([0,T]\times\R^{2})}^{2} &\leq \tilde{C}(u)\Bigg(\epsilon_{2}^{3/2} + 2\|u\|_{\tilde{Y}_{k_{*}}([0,T]\times\R^{2})}^{2}\Bigg(\epsilon_{2}(2\|u\|_{\tilde{X}_{k_{*}}([0,T]\times\R^{2})}^{2})^{3/2} + \epsilon_{2}^{1/3}(2\|u\|_{\tilde{X}_{k_{*}}([0,T]\times\R^{2})}^{2})^{5/6} \nonumber\\
&\phantom{=}\qquad+ \paren*{\epsilon_{2}+(2\|u\|_{\tilde{Y}_{k_{*}}([0,T]\times\R^{2})}^{2})^{1/2}\paren*{1+4\|u\|_{\tilde{X}_{k_{*}}([0,T]\times\R^{2})}^{4}}}^{2}\Bigg)\Bigg).
\end{align}
Using the induction hypothesis to estimate $\|u\|_{\tilde{X}_{k_{*}}([0,T]\times\R^{2})}^{2}, \|u\|_{\tilde{Y}_{k_{*}}([0,T]\times\R^{2})}^{2}$ on the RHS above, we see that
\begin{align}
\|u\|_{\tilde{X}_{k_{*}+1}([0,T]\times\R^{2})}^{2} &\leq \tilde{C}(u)\paren*{1+2\epsilon_{2}\paren*{\epsilon_{2}(2C_{0})^{3/2}+\epsilon_{2}^{1/3}(2C_{0})^{5/6}+\paren*{\epsilon_{2}+(2\epsilon_{2})^{1/2}\paren*{1+4C_{0}^{2}}}^{2}}} \\
\|u\|_{\tilde{Y}_{k_{*}+1}([0,T]\times\R^{2})}^{2} &\leq \tilde{C}(u)\paren*{\epsilon_{2}^{3/2}+2\epsilon_{2}\paren*{\epsilon_{2}(2C_{0})^{3/2}+\epsilon_{2}^{1/3}(2C_{0})^{5/6}+\paren*{\epsilon_{2}+(2\epsilon_{2})^{1/2}\paren*{1+4C_{0}^{2}}}^{2}}}.
\end{align}
Now choose $\epsilon_{2}(u)>0$ sufficiently small so that
\begin{align}
&2^{-19}\epsilon_{2}(u)\paren*{\epsilon_{2}(u)(2C_{0})^{3/2}+\epsilon_{2}(u)^{1/3}(2C_{0})^{5/6}+\paren*{\epsilon_{2}(u)+(2\epsilon_{2}(u))^{1/2}\paren*{1+4C_{0}^{2}}}^{2}} \leq 1-2^{-20} \\
&\tilde{C}(u)\epsilon_{2}(u)^{1/2}+2\tilde{C}(u)\paren*{\epsilon_{2}(u)(2C_{0})^{3/2}+\epsilon_{2}(u)^{1/3}(2C_{0})^{5/6}+\paren*{\epsilon_{2}(u)+(2\epsilon_{2}(u))^{1/2}\paren*{1+4C_{0}^{2}}}^{2}} \leq 1.
\end{align}
Then, after recalling that $C_{0}=2^{-20}\tilde{C}(u)$, it follows that for $0<\epsilon_{2}\leq \epsilon_{2}(u)$,
\begin{align}
\|u\|_{\tilde{X}_{k_{*}+1}([0,T]\times\R^{2})}^{2} &\leq C_{0}\\
\|u\|_{\tilde{Y}_{k_{*}+1}([0,T]\times\R^{2})}^{2} &\leq \epsilon_{2}.
\end{align}
\end{proof}

Since we defined $\tilde{C}(u)$ to be the maximum of all the implicit constants obtained in Steps 1-4, lemma \ref{lem:LTSE_base} implies that \eqref{eq:LTSE_ind1}, \eqref{eq:LTSE_ind2} hold with $k_{*}=20$, provided that $0<\epsilon_{2}\leq \epsilon_{2}(u)$, where $\epsilon_{2}(u)>0$ is sufficiently small so that $2^{20}\tilde{C}(u)\epsilon_{2}(u)^{1/2}\leq 1$. Hence, the proof of theorem \ref{thm:LTSE} is complete by induction on $k_{*}$.

\section{Proofs of bilinear Strichartz estimates}\label{sec:BSE}

\subsection{Overview}\label{ssec:BSE_ov}
In this section, we pay our debt to the reader by proving the three bilinear Strichartz estimates (propositions \ref{prop:ibs_1}, \ref{prop:ibs_2}, and \ref{prop:ibs_3}) stated and used in subsection \ref{ssec:LTSE_S4}. Propositions \ref{prop:ibs_1}, \ref{prop:ibs_2}, and \ref{prop:ibs_3} are specific to our setting in two ways: first, the proofs use the specific structure of the eeDS equation, in particular its local conservation laws; second, the bilinear estimates are adapted to the frequency cube localization coming from the double frequency decomposition. The first specificity is superficial as any ``good" structure of the eeDS equation is shared by other semilinar, local dispersive equations (e.g. cubic NLS). Indeed, the real difficulty is that the eeDS lacks certain structure possessed by the NLS. The second specificity is nontrivial. As remarked in the introduction, the work \cite{Dodson2016} also heavily relied on bilinear estimates. If we were to prove exact analogues of the bilinear estimates in that work (theorems 4.5-4.7), then we would have estimates which we do not know how to use. Moreover, it is not clear to us how to obtain our propositions \ref{prop:ibs_1}, \ref{prop:ibs_2}, and \ref{prop:ibs_3} a posteriori from the eeDS analogues of the bilinear estimates in \cite{Dodson2016}. Thus, we go back to the drawing board.

\cite{Dodson2016} does provide us with an idea for proving our propositions \ref{prop:ibs_1}, \ref{prop:ibs_2}, and \ref{prop:ibs_3}. In caricature, we use the observation of \cite{Planchon2009} that one can prove bilinear Strichartz estimates by defining an appropriate interaction Morawetz functional and then proceeding by the standard fundamental theorem of calculus argument used to prove monotonicity formulae. We rely heavily on the local conservation laws of equation \eqref{eq:DS} and integration by parts together with the interplay between the spatial Radon and Fourier transforms. Additionally, we use the observation of \cite{Planchon2009} that the interaction Morawetz functional is Galilean invariant, which is useful for exploiting the Littlewood-Paley theory adapted to the frequency center $\xi(t)$.

Initially, one might think that proceeding by monotonicity formulae in the eeDS setting is ill-advised. Indeed, we have remarked in the introduction that we do not know how to prove an a priori interaction Morawetz estimate for suitably regular solutions of \eqref{eq:DS}. The fact saving our strategy is that at this stage, we can directly estimate the contribution of the nonlinear term via the Calder\'{o}n-Zygmund theorem; we do not need to rely on any sort of nonnegativity condition. This indifference to the sign or positive definiteness of the nonlinear contribution reflects that our arguments in this section are agnostic to the defocusing/focusing distinction. Moreover, we expect our arguments to apply to any nonlinearity containing good differential structure (e.g. Riesz transforms) to facilitate integration by parts.

\subsection{Frequency-localized conservation laws}\label{ssec:BSE_cl}
 As we will make heavy use of the local conservation laws for the the Schr\"{o}dinger and Davey-Stewartson equations, we first recall them for the readers' benefit (cf. subsection 12.1.3 of \cite{sulem1999nonlinear}). Let $v=e^{it\Delta}v_{0}$ solve the free Schr\"{o}dinger equation, and let $u$ solve the eeDS equation
\begin{equation}
(i\partial_{t}+\Delta) u = \frac{\mu_{1}\partial_{1}^{2}+\mu_{2}\partial_{2}^{2}}{\Delta}(|u|^{2})u,
\end{equation}
where $\mu_{1},\mu_{2}\in\R$ and $v_{0},u$ are sufficiently smooth. Then the tensors $\{T_{jk}^{v}\}_{j,k=0}^{2}, \{T_{jk}^{u}\}_{j,k=0}^{2}$ associated to $v$ and $u$, respectively, satisfy the equations (written using Einstein summation)
\begin{gather}
	\partial_{t}T_{00}^{v}+\partial_{j}T_{0j}^{v}=0\\
	\partial_{t}T_{00}^{u}+\partial_{j}T_{0j}^{u}=0\\
	\partial_{t}T_{0j}^{v}+\partial_{k}L_{jk}^{v}=0\\
	\partial_{t}T_{0j}^{u}+\partial_{k}L_{jk}^{u}+\partial_{k}T_{jk}^{u}=0,
\end{gather}
where for $1\leq j,k\leq 2$,
\begin{gather}
T_{00}^{v} \coloneqq |v|^{2}, \qquad T_{00}^{u} \coloneqq |u|^{2}\\
T_{0j}^{v} = T_{j0}^{v} \coloneqq 2\Im{\bar{v}\p_{j}v}, \qquad T_{0j}^{u} = T_{j0}^{u} \coloneqq 2\Im{\bar{u}\p_{j} u}\\
L_{jk}^{v} \coloneqq -\partial_{j}\partial_{k}(|v|^{2})+4\Re{\ol{\partial_{k}v}\partial_{j}v}, \qquad L_{jk}^{u} \coloneqq -\partial_{j}\partial_{k}(|u|^{2})+4\Re{\ol{\partial_{k}u}\partial_{j}u} \\
T_{jk}^{u} \coloneqq \mu_{2}\delta_{kj}|u|^{4}-(\mu_{2}-\mu_{1})\delta_{kj}|\frac{\nabla\partial_{1}}{\Delta}(|u|^{2})|^{2}-2(\mu_{2}-\mu_{1})\delta_{k1}|u|^{2} \frac{\partial_{k}\partial_{j}}{\Delta}(|u|^{2})+2(\mu_{2}-\mu_{1})\frac{\partial_{k}\partial_{1}}{\Delta}(|u|^{2})\frac{\partial_{j}\partial_{1}}{\Delta}(|u|^{2}).
\end{gather}
In the sequel, we will work with the frequency-localized solution $w \coloneqq P_{\xi(t), \leq l_{2}}u$, rather than $u$. Therefore we first derive local conservation laws for $w$ analogous to those obtained above for $u$. Observe that $w$ satisfies the equation
\begin{gather}
	(i\p_{t}+\Delta)w=F(w)+\mathcal{N}_{1}+\mathcal{N}_{2}\eqqcolon F(w)+\mathcal{N}\\
	\mathcal{N}_{1} \coloneqq P_{\xi(t),\leq l_{2}}F(u)-F(w)\\ 
	\mathcal{N}_{2} \coloneqq \paren*{\frac{d}{dt}P_{\xi(t)\leq l_{2}}}u,
\end{gather}
where $\frac{d}{dt}P_{\xi(t),\leq l_{2}}$ is the time-dependent Fourier multiplier with symbol
\begin{equation}
	2^{-l_{2}}(\nabla\phi)\paren*{\frac{\xi-\xi(t)}{2^{l_{2}}}}\cdot -\xi'(t).
\end{equation}
Define $L_{jk}^{w}, T_{jk}^{w}$ analogously to $L_{jk}^{u}, T_{jk}^{u}$. Then by the calculus and the equation satisfied by $w$, we have the \emph{frequency-localized local mass conservation identity}
\begin{align}
\p_{t}T_{00}^{w} = w\ol{\partial_{t}w} + (\partial_{t}w)\ol{w} &= w\ol{i^{-1}\paren*{-\Delta w+F(w)+\mathcal{N}}}  + i^{-1}\paren*{-\Delta w+F(w)+\mathcal{N}}\ol{w} \nonumber\\
&=w\ol{i^{-1}\paren*{-\Delta w+F(w)}} + i^{-1}\paren*{-\Delta w+F(w)}\ol{w} + w\ol{i^{-1}\mathcal{N}} + i^{-1}\mathcal{N}\ol{w}  \nonumber\\
&=-\partial_{k}T_{0k}^{w}+2\Re{i^{-1}\bar{w}\mathcal{N}} \nonumber\\
&=-\partial_{k}T_{0k}^{w}+2\Im{\bar{w}\mathcal{N}}
\end{align}	
and the \emph{frequency-localized local momentum conservation identity}
\begin{align}
\p_{t} T_{0j}^{w} &= 2\Im{\ol{\partial_{t}w}\partial_{j}w}+ 2\Im{\bar{w}\partial_{j}\partial_{t}w} \nonumber\\
&= 2\Im{\ol{i^{-1}(-\Delta w+F(w)+\mathcal{N})}\partial_{j}w} + 2\Im{\bar{w}\partial_{j}i^{-1}(-\Delta w+F(w)+\mathcal{N})} \nonumber\\
&= 2\Im{\ol{i^{-1}\paren*{-\Delta w+F(w)}}\partial_{j}w} + 2\Im{\bar{w}i^{-1}\partial_{j}\paren*{-\Delta w+F(w)}}+ 2\Im{\ol{i^{-1}\mathcal{N}}\partial_{j}w} + 2\Im{i^{-1}\bar{w}\partial_{j}\mathcal{N}} \nonumber \\
&= -\partial_{k}L_{jk}^{w}-\partial_{k}T_{jk}^{w}+2\Re{\bar{\mathcal{N}}\partial_{j}w}-2\Re{\bar{w}\partial_{j}\mathcal{N}}.
\end{align}

\subsection{Bilinear Strichartz estimate I}\label{ssec:BSE_1}
We recall the statement of proposition \ref{prop:ibs_1}:
\ibsI*

\begin{proof}
Let $v_{0}$ be as in the statement of the proposition. Define functions $v:= e^{it\Delta}v_{0}$ and $w := P_{\xi(t),\leq l_{2}}u$ so that $v$ and $w$ as above. Following \cite{Planchon2009} and \cite{Dodson2016}, for every $\omega\in \mathbb{S}^{1}$, define the interaction Morawetz functional
\begin{equation}
M_{\omega}(t) \coloneqq 2\int_{\R^{4}}\frac{(x-y)_{\omega}}{|(x-y)_{\omega}|} |w(t,y)|^{2}\Im{\bar{v}\partial_{w}v}(t,x)dxdy+2\int_{\R^{4}} \frac{(x-y)_{\omega}}{|(x-y)_{\omega}|} |v(t,y)|^{2}\Im{\bar{w}\partial_{w}w}(t,x)dxdy.
\end{equation}
and define the spherically averaged interaction Morawetz functional
\begin{equation}
M(t) \coloneqq \int_{\mathbb{S}^{1}} M_{\omega}(t)d\omega,
\end{equation}
where $d\omega$ denotes the unit-normalized surface measure on $\S^{1}$.

Differentiating $M_{\omega}$ with respect to time and using the local conservation laws for $v$ and $w$, we obtain that
\begin{equation}
\begin{split}
\frac{d}{dt}M_{\omega}(t) &= 2\int_{\R^{4}}\frac{(x-y)_{\omega}}{|(x-y)_{\omega}|} \paren*{-\partial_{k}T_{0k}^{w}+2\Im{\bar{w}\mathcal{N}}}(t,y)\Im{\bar{v}\partial_{\omega}v}(t,x)dxdy \\
&\phantom{=}+\int_{\R^{4}}\frac{(x-y)_{\omega}}{|(x-y)_{\omega}|} |w(t,y)|^{2} \paren*{-\partial_{k}L_{jk}^{v}}(t,x)\omega_{j} dxdy \\
&\phantom{=}+2\int_{\R^{4}}\frac{(x-y)_{\omega}}{|(x-y)_{\omega}|} \paren*{-\partial_{k}T_{0k}^{v}}(t,y)\Im{\bar{w}\partial_{\omega}w}(t,x) dxdy \\
&\phantom{=}+\int_{\R^{4}} \frac{(x-y)_{\omega}}{|(x-y)_{\omega}|} |v(t,y)|^{2}\paren*{-\partial_{k}L_{jk}^{w}-\partial_{k}T_{jk}^{w}+2\Re{\bar{\mathcal{N}}\partial_{j}w}-2\Re{\bar{w}\partial_{j}\mathcal{N}}}(t,x)\omega_{j} dxdy.
\end{split}
\end{equation}
Integrating by parts in $y$, we have that
\begin{align}
-2\int_{\R^{4}} \frac{(x-y)_{\omega}}{|(x-y)_{\omega}|}(\partial_{k}T_{0k}^{w})(t,y)\Im{\bar{v}\partial_{\omega}v}(t,x) dxdy &= -4\int_{\{x_{\omega}=y_{\omega}\}} T_{0k}^{w}(t,y)\omega_{k}\Im{\bar{v}\partial_{\omega}v}(t,x) d\mathcal{H}^{3}(x,y)\nonumber\\
&=-8\int_{\{x_{\omega}=y_{\omega}\}} \Im{\bar{w}\partial_{\omega}w}(t,y)\Im{\bar{v}\partial_{\omega}v}(t,x) d\mathcal{H}^{3}(x,y).
\end{align}
Integrating by parts in $x$, we have that
\begin{align}
\int_{\R^{4}} \frac{(x-y)_{\omega}}{|(x-y)_{\omega}|} |w(t,y)|^{2} (-\partial_{k}L_{jk}^{v})(t,x)\omega_{j} dxdy &=2\int_{\{x_{\omega}=y_{\omega}\}} |w(t,y)|^{2}\paren*{4\Re{\ol{\partial_{k}v}\partial_{j}v}-\partial_{j}\partial_{k}(|v|^{2})}(t,x)\omega_{j}\omega_{k} d\mathcal{H}^{3}(x,y) \nonumber\\
&=8\int_{\{x_{\omega}=y_{\omega}\}} |w(t,y)|^{2}|\partial_{\omega}v(t,x)|^{2} d\mathcal{H}^{3}(x,y) \nonumber\\
&\phantom{=} + 2\int_{\{x_{\omega}=y_{\omega}\}} \partial_{\omega}(|w|^{2})(t,y)\partial_{\omega}(|v|^{2})(t,x) d\mathcal{H}^{3}(x,y),
\end{align}
where the second term in the RHS of the ultimate equality follows from another integration by parts in $x_{\omega}$. Integrating by parts in $y$, we have that
\begin{align}
-2\int_{\R^{4}} \frac{(x-y)_{\omega}}{|(x-y)_{\omega}|} (\partial_{k}T_{0k}^{v})(t,y) \Im{\bar{w}\partial_{\omega}w}(t,x) dxdy &= -4\int_{\{x_{\omega}=y_{\omega}\}} T_{0k}^{v}(t,y)\omega_{k}\Im{\bar{w}\partial_{\omega}w}(t,x) d\mathcal{H}^{3}(x,y) \nonumber\\
&=-8\int_{\{x_{\omega}=y_{\omega}\}} \Im{\bar{v}\partial_{\omega}v}(t,y)\Im{\bar{w}\partial_{\omega}w}(t,x) d\mathcal{H}^{3}(x,y).
\end{align}
Integrating by parts in $x$, we have that
\begin{align}
&\int_{\R^{4}} \frac{(x-y)_{\omega}}{|(x-y)_{\omega}|}|v(t,y)|^{2} \paren*{-\partial_{k}L_{jk}^{w}-\partial_{k}T_{jk}^{w}}(t,x)\omega_{j} dxdy\nonumber \\
&\phantom{=} =2\int_{\{x_{\omega}=y_{\omega}\}} |v(t,y)|^{2} \paren*{4\Re{\ol{\partial_{k}w}\partial_{j}w}-\partial_{j}\partial_{k}(|w|^{2}) + T_{jk}^{w}}(t,x)\omega_{j}\omega_{k} d\mathcal{H}^{3}(x,y) \nonumber\\
&\phantom{=}= 8\int_{\{x_{\omega}=y_{\omega}\}} |v(t,y)|^{2} |\partial_{\omega}w(t,x)|^{2} d\mathcal{H}^{3}(x,y) + 2\int_{\{x_{\omega}=y_{\omega}\}} \partial_{\omega}(|v|^{2})(t,y)\partial_{\omega}(|w|^{2})(t,x) d\mathcal{H}^{3}(x,y)\nonumber\\
&\phantom{=}\qquad + 2\int_{\{x_{\omega}=y_{\omega}\}} |v(t,y)|^{2}T_{jk}^{w}(t,x)\omega_{j}\omega_{k}d\mathcal{H}^{3}(x,y),
\end{align}
where the second term in the RHS of the ultimate equality follows by another integration by parts in $x_{\omega}$. Noting that
\begin{equation}
|\partial_{\omega}\Tr_{x_{\omega}=y_{\omega}}(\bar{w}\otimes v)|^{2} = \Tr_{x_{\omega}=y_{\omega}}\paren*{|(\partial_{\omega}w)\otimes v|^{2}+|w\otimes (\partial_{\omega}v)|^{2} + \frac{1}{2} \partial_{\omega}(|w|^{2})\otimes \partial_{\omega}(|v|^{2}) - 2\Im{\bar{w}\partial_{\omega}w}\otimes \Im{\bar{v}\partial_{\omega}v}},
\end{equation}
we may rewrite as $\frac{d}{dt}M_{\omega}(t)$ as
\begin{equation}
\begin{split}
\frac{d}{dt}M_{\omega}(t) &= 8\int_{\{x_{\omega}=y_{\omega}\}} |(\partial_{\omega}\Tr_{x_{\omega}=y_{\omega}}(\bar{w}\otimes v))(t,x,y)|^{2} d\mathcal{H}^{3}(x,y) \\
&\phantom{=}+ 2\int_{\{x_{\omega}=y_{\omega}\}} |v(t,y)|^{2} T_{jk}^{w}(t,x)\omega_{j}\omega_{k} d\mathcal{H}^{3}(x,y) \\
&\phantom{=}+ 4\int_{\R^{4}} \frac{(x-y)_{\omega}}{|(x-y)_{\omega}|}\Im{\bar{w}\mathcal{N}}(t,y)\Im{\bar{v}\partial_{\omega}v}(t,x) dxdy \\
&\phantom{=}+2\int_{\R^{4}}\frac{(x-y)_{\omega}}{|(x-y)_{\omega}|} \Re{\bar{\mathcal{N}}\partial_{\omega}w}(t,x) dxdy \\
&\phantom{=}-2\int_{\R^{4}} \frac{(x-y)_{\omega}}{|(x-y)_{\omega}|} \Re{\bar{w}\partial_{\omega}\mathcal{N}}(t,x) dxdy.
\end{split}
\end{equation}
Above, $\Tr_{x_{\omega}=y_{\omega}}$ denotes the restriction an $\R^{4}\rightarrow\C$ function to the hyperplane $\{x_{\omega}=y_{\omega}\}\subset \R^{4}$.

For each $\omega\in\S^{1}$, let $\mathcal{O}_{\omega}:\R^{2}\rightarrow\R^{2}$ denote the rotation mapping the positively oriented basis $\{\omega,\omega^{\perp}\}$ onto the standard basis $\{e_{1},e_{2}\}$. Define the functions $v_{\omega},w_{\omega}:\R^{2}\rightarrow \C$ by $v_{\omega}(x)\coloneqq v(\mathcal{O}_{\omega}^{*}x)$ and $w_{\omega}(x)\coloneqq w(\mathcal{O}_{\omega}^{*}x)$ respectively. Let $\iota_{x_{1}=y_{1}}:\R^{3}\rightarrow\R^{4}$ denote the map $\iota_{x_{1}=y_{1}}(z_{1},z_{2},z_{3})\coloneqq (z_{1},z_{2},z_{1},z_{3})$. Then using the Fubini-Tonelli theorem, we see that
\begin{equation}
\begin{split}
\frac{d}{dt}M_{\omega}(t) &= 8\int_{\R}\int_{\R}\int_{\R}|\partial_{1}\paren*{\Tr_{x_{1}=y_{1}}(\bar{w}_{\omega}\otimes v_{\omega})\circ \iota_{x_{1}=y_{1}}}(t,x_{1},x_{2},y_{2})|^{2} dx_{1}dx_{2}dy_{2}\\
&\phantom{=}+2\int_{\{x_{\omega}=y_{\omega}\}} |v(t,y)|^{2} T_{jk}^{w}(t,x)\omega_{j}\omega_{k} d\mathcal{H}^{3}(x,y) \\
&\phantom{=}+ 4\int_{\R^{4}} \frac{(x-y)_{\omega}}{|(x-y)_{\omega}|}\Im{\bar{w}\mathcal{N}}(t,y)\Im{\bar{v}\partial_{\omega}v}(t,x) dxdy\\
&\phantom{=}+2\int_{\R^{4}}\frac{(x-y)_{\omega}}{|(x-y)_{\omega}|} |v(t,y)|^{2} \Re{\bar{\mathcal{N}}\partial_{\omega}w}(t,x) dxdy\\
&\phantom{=}-2\int_{\R^{4}} \frac{(x-y)_{\omega}}{|(x-y)_{\omega}|} |v(t,y)|^{2} \Re{\bar{w}\partial_{\omega}\mathcal{N}}(t,x) dxdy.
\end{split}
\end{equation}
Let $\iota_{x=y}:\R^{2}\rightarrow\R^{4}$ denote the map $\iota_{x=y}(z_{1},z_{2})\coloneqq (z_{1},z_{2},z_{1},z_{2})$, and consider the function $\partial_{1}\paren*{(\bar{w}_{\omega}\otimes v_{\omega})\circ \iota_{x=y}}:\R^{2}\rightarrow \mathbb{C}$. Observe that
\begin{equation}
\partial_{1}\paren*{(\bar{w}_{\omega}\otimes v_{\omega})\circ\iota_{x=y}}=\partial_{1}(\bar{w}_{\omega}v_{\omega}) = \paren*{\partial_{1}\paren*{(\bar{w}_{\omega}\otimes v_{\omega})\circ\iota_{x_{1}=y_{1}}}\circ\iota_{x_{2}=y_{2}}}.
\end{equation}
Let $\xi,\eta$ denote the Fourier variables conjugate to $x$ and $y$, respectively. We claim that $\mathcal{F}_{x_{1},x_{2},y_{2}}\paren*{(\bar{w}_{\omega}\otimes v_{\omega})\circ\iota_{x_{1}=y_{1}}}$ has $\eta_{2}$ support in an interval of side length $\sim 2^{k}$. Indeed, by the Fubini-Tonelli theorem,
\begin{equation}
\mathcal{F}_{x,y}(\bar{w}_{\omega}\otimes v_{\omega})(\xi,\eta) = \mathcal{F}\paren*{\bar{w}_{\omega}}(\xi)\mathcal{F}\paren*{v_{\omega}}(\eta) = \mathcal{F}\paren*{\bar{w}_{\omega}}(\xi)\mathcal{F}_{y_{1}}\paren*{\mathcal{F}_{y_{2}}\paren*{v_{\omega}}}(\eta), \qquad \forall(\xi,\eta)\in\R^{2}\times\R^{2},
\end{equation}
and by Fourier inversion,
\begin{equation}
\mathcal{F}_{y_{2}}(v_{\omega})(y_{1},\eta_{2}) = (2\pi)^{-1/2}\int_{\R}\mathcal{F}(v_{\omega})(\eta_{1},\eta_{2})e^{iy_{1}\cdot\eta_{1}}d\eta_{1} = (2\pi)^{-1/2}\int_{\R}\mathcal{F}(v)\paren*{\mathcal{O}_{\omega}^{*}(\eta_{1},\eta_{2})}e^{iy_{1}\cdot\eta_{1}}d\eta_{1}, \qquad (y_{1},\eta_{2})\in\R\times\R^{2}.
\end{equation}
Since $\hat{v}$ has support in a cube $Q$ of side length length $2^{k}$ centered at $c_{Q}\in\R^{2}$, it follows that $\mathcal{F}(v_{\omega})$ has support in a cube of side length $\sqrt{2}2^{k}$ centered at $\mathcal{O}_{\omega}^{}(c_{Q})$. It follows now that
\begin{equation}
\eta_{2}\notin [\mathcal{O}_{\omega}(c_{Q})_{2}-\sqrt{2}2^{k-1}, \mathcal{O}_{\omega}(c_{Q})_{2}+\sqrt{2}2^{k-1}] \Longrightarrow \mathcal{F}_{y_{2}}(v_{\omega})(y_{1},\eta_{2})=0.
\end{equation}
Observing that
\begin{align}
\mathcal{F}_{y_{2}}\paren*{\paren*{\bar{w}_{\omega}\otimes v_{\omega}}\circ\iota_{x_{1}=y_{1}}}(x_{1},x_{2},\eta_{2})=\int_{\R}\bar{w}_{\omega}(x)v_{\omega}(x_{1},y_{2})e^{-iy_{2}\cdot\eta_{2}}dy_{2} &= \bar{w}_{\omega}(x)\mathcal{F}_{y_{2}}(v_{\omega})(x_{1},\eta_{2})
\end{align}
completes the proof of the claim.

Therefore by Bernstein's inequality in the $\eta_{2}$ coordinate, we have that
\begin{align}
\int_{\R^{2}}|\paren*{\partial_{1}\paren*{(\bar{w}_{\omega}\otimes v_{\omega})\circ\iota_{x=y}}}(t,x)|^{2}dx \leq \|\partial_{1}\paren*{(\bar{w}_{\omega}\otimes v_{\omega})\circ\iota_{x_{1}=y_{1}}}\|_{L_{z_{1},z_{2}}^{2}L_{z_{3}}^{\infty}(\R^{2}\times \R)}^{2} &\lesssim 2^{k}\|\partial_{1}\paren*{(\bar{w}_{\omega}\otimes v_{\omega})\circ \iota_{x_{1}=y_{1}}} \|_{L_{z_{1},z_{2},z_{3}}^{2}(\R^{3})}^{2}.
\end{align}
Integrating both sides of the preceding inequality over $\mathbb{S}^{1}$ with respect to the measure $d\omega$ and using the invariance of the measure under rotation, we obtain
\begin{equation}
\int_{\R^{2}}|\nabla(\bar{w}v)(t,x)|^{2}dx \lesssim 2^{k}\int_{\mathbb{S}^{1}} \int_{\R^{3}}\left|\paren*{\partial_{1}\paren*{(\bar{w}_{\omega}\otimes v_{\omega})\circ\iota_{x_{1}=y_{1}}}}(t,\ul{z}_{3})\right|^{2}d\ul{z}_{3} d\omega.
\end{equation}

We next claim that $\bar{w}v$ has Fourier support in the region $\{|\xi| \geq 2^{i-7}\}$. Indeed, since $|\xi(t)-\xi(G_{\alpha}^{i})| \leq \epsilon_{3}^{1/2}2^{i-19}$, it follows that $\bar{w}$ has Fourier support in the dyadic ball $B(-\xi(G_{\alpha}^{i}), i-8)$. The claim then follows from the assumption that $v$ has Fourier support in the dyadic annulus $A(\xi(G_{\alpha}^{i}),i-6,i+6)$ and the reverse triangle inequality. Therefore by Bernstein's lemma,
\begin{equation}
\int_{\R^{2}}|(\bar{w}v)(t,x)|^{2}dx \lesssim 2^{-2i}\int_{\R^{2}}|\nabla(\bar{w}v)(t,x)|^{2}dx \lesssim 2^{k-2i}\int_{\S^{1}}\int_{\R^{3}}|\paren*{\partial_{1}\paren*{(\bar{w}_{\omega}\otimes v_{\omega})\circ\iota_{x_{1}=y_{1}}}}(t,\ul{z}_{3})|^{2}d\ul{z}_{3}d\omega.
\end{equation}
Integrating with respect to time over the interval $J_{l}$ and using the fundamental theorem of calculus, we finally obtain
\begin{align}
2^{2i-k}\int_{J_{l}}\int_{\R^{2}} |(\bar{w}v)(t,x)|^{2}dxdt &\lesssim M(t)\label{eq:ibs1_main_p}\\
&\phantom{=}-2\int_{J_{l}}\int_{\mathbb{S}^{1}}\int_{\{x_{\omega}=y_{\omega}\}} |v(t,y)|^{2} T_{jk}^{w}(t,x)\omega_{j}\omega_{k} d\mathcal{H}^{3}(x,y)dt \label{eq:ibs1_nl_p}\\
&\phantom{=}- 4\int_{J_{l}}\int_{\mathbb{S}^{1}}\int_{\R^{4}} \frac{(x-y)_{\omega}}{|(x-y)_{\omega}|}\Im{\bar{w}\mathcal{N}}(t,y)\Im{\bar{v}\partial_{\omega}v}(t,x) dxdydt \label{eq:ibs1_err1_p}\\
&\phantom{=}-2\int_{J_{l}}\int_{\mathbb{S}^{1}}\int_{\R^{4}}\frac{(x-y)_{\omega}}{|(x-y)_{\omega}|} |v(t,y)|^{2}\Re{\bar{\mathcal{N}}\partial_{\omega}w}(t,x) dxdydt \label{eq:ibs1_err2_p}\\
&\phantom{=}+2\int_{J_{l}}\int_{\mathbb{S}^{1}}\int_{\R^{4}} \frac{(x-y)_{\omega}}{|(x-y)_{\omega}|} |v(t,y)|^{2}\Re{\bar{w}\partial_{\omega}\mathcal{N}}(t,x) dxdydt.\label{eq:ibs1_err3_p}
\end{align}

To handle the integrations over $\omega\in\mathbb{S}^{1}$ and over $x,y\in \{x_{\omega}=y_{\omega}\}$, we use the following lemma, the proof of which we omit.
\begin{lemma}\label{lem:itrick} %Integration lemma
There exist constants $C,C'$ such that for $f\in C_{loc}^{1}(\R^{2})$ and $g\in L^{1}(\R^{4})\cap L^{\infty}(\R^{4})$,
\begin{equation}
\int_{\mathbb{S}^{1}} \frac{x_{\omega}}{|x_{\omega}|} (\p_{\omega}f)(x)d\omega = C\frac{x}{|x|}\cdot(\nabla f)(x), \qquad \forall x\in \R^{2}\setminus\{0\}
\end{equation}
and
\begin{equation}
\int_{\mathbb{S}^{1}}\int_{\{x_{w}=y_{\omega}\}}g(x,y)d\mathcal{H}^{3}(x,y)d\omega = C'\int_{\R^{4}} \frac{g(x,y)}{|x-y|} dxdy.
\end{equation}
\end{lemma}

Using lemma \ref{lem:itrick}, we see that
\begin{align}
M(t) &= 2C\int_{\R^{4}}\frac{(x-y)}{|x-y|}\cdot |w(t,y)|^{2}\Im{\bar{v}\nabla v}(t,x)dxdy+2C\int_{\R^{4}}\frac{(x-y)}{|x-y|}\cdot |v(t,y)|^{2}\Im{\bar{w}\nabla w}(t,x) dxdy \nonumber\\
&=2C\int_{\R^{4}}\frac{(x-y)}{|x-y|}\cdot |w(t,y)|^{2}\Im{\bar{v}(\nabla-i\xi(t)) v}(t,x)dxdy + 2C\int_{\R^{4}}\frac{(x-y)}{|x-y|}\cdot |v(t,y)|^{2}\Im{\bar{w}(\nabla-i\xi(t)) w}(t,x) dxdy, \label{eq:ibs1_main}
\end{align}
where we use the trivial identity $\frac{x-y}{|x-y|}=-\frac{y-x}{|y-x|}$ to obtain the ultimate line. By the same argument,
\begin{align}
\eqref{eq:ibs1_err1_p} + \eqref{eq:ibs1_err2_p} + \eqref{eq:ibs1_err3_p} &= -4C\int_{J_{l}}\int_{\R^{4}}\frac{(x-y)}{|x-y|}\cdot \Im{\bar{w}\mathcal{N}}(t,y)\Im{\bar{v}(\nabla-i\xi(t))v}(t,x)dxdydt \label{eq:ibs1_err1}\\
&\phantom{=}-2C\int_{J_{l}}\int_{\R^{4}}\frac{(x-y)}{|x-y|}\cdot |v(t,y)|^{2}\Re{\bar{\mathcal{N}}(\nabla-i\xi(t))w}(t,x) dxdy dt \label{eq:ibs1_err2}\\
&\phantom{=}+2C\int_{J_{l}}\int_{\R^{4}}\frac{(x-y)}{|x-y|}\cdot |v(t,y)|^{2}\Re{\bar{w}(\nabla-i\xi(t))\mathcal{N}}(t,x) dxdy dt, \label{eq:ibs1_err3}
\end{align}
and lastly
\begin{equation}
\eqref{eq:ibs1_nl_p} = -2C'\int_{J_{l}}\int_{\R^{4}}\paren*{\frac{\delta_{jk}}{|x-y|}-\frac{(x-y)_{j}(x-y)_{k}}{|x-y|^{3}}} |v(t,y)|^{2} T_{jk}^{w}(t,x) dxdy dt. \label{eq:ibs1_nl}
\end{equation}
We now estimate each of the terms \eqref{eq:ibs1_main}-\eqref{eq:ibs1_nl} separately.

\begin{description}[leftmargin=*]
\item[Estimate for \eqref{eq:ibs1_main}:]
By Cauchy-Schwarz, Bernstein's lemma, and the Fourier support hypotheses on $v$ and $w$, we have that
\begin{equation}
\sup_{t\in J_{l}}|M(t)| \lesssim 2^{i}\|w\|_{L_{t}^{\infty}L_{x}^{2}(J_{l}\times\R^{2})}\|v\|_{L_{t}^{\infty}L_{x}^{2}(J_{l}\times\R^{2})}^{2} + 2^{l_{2}}\|w\|_{L_{t}^{\infty}L_{x}^{2}(J_{l}\times\R^{2})}\|v\|_{L_{t}^{\infty}L_{x}^{2}(J_{l}\times\R^{2})}^{2} \lesssim 2^{i}\|v_{0}\|_{L^{2}(\R^{2})}^{2},
\end{equation}
where we use mass conservation to obtain the ultimate inequality.

\item[Estimate for \eqref{eq:ibs1_nl}:]
Observe from the definition of $T_{jk}^{w}$ that
\begin{equation}
\begin{split}
\left|\int_{J_{l}}\int_{\R^{4}}\paren*{\frac{\delta_{jk}}{|x-y|}-\frac{(x-y)_{j}(x-y)_{k}}{|x-y|^{3}}} |v(t,y)|^{2} T_{jk}^{w}(t,x) dxdydt\right| \lesssim \int_{J_{l}}\int_{\R^{4}}\frac{1}{|x-y|} |v(t,y)|^{2} |\vec{R}^{2}(|w|^{2})(t,x)|^{2} dxdy dt,
\end{split}
\end{equation}
which by H\"{o}lder's inequality, Hardy-Littlewood-Sobolev lemma, Calder\'{o}n-Zygmund theorem, and Strichartz estimates is
\begin{align}
\lesssim \| |\nabla|^{-1}(|v|^{2})\|_{L_{t}^{3}L_{x}^{6}(J_{l}\times\R^{2})} \||\vec{R}^{2}(|w|^{2})|^{2}\|_{L_{t}^{3/2}L_{x}^{6/5}(J_{l}\times\R^{2})}  &\lesssim \|v\|_{L_{t}^{6}L_{x}^{3}(J_{l}\times\R^{2})}^{2} \|w\|_{L_{t}^{6}L_{x}^{24/5}(J_{l}\times\R^{2})}^{4}\nonumber\\
&\lesssim \|v_{0}\|_{L^{2}(\R^{2})}^{2} \|w\|_{L_{t}^{6}L_{x}^{24/5}(J_{l}\times\R^{2})}^{4} \nonumber\\
&\lesssim 2^{l_{2}}\|v_{0}\|_{L^{2}(\R^{2})}^{2} \|w\|_{L_{t}^{6}L_{x}^{3}(J_{l}\times\R^{2})}^{4} \nonumber\\
&\lesssim 2^{l_{2}}\|v_{0}\|_{L^{2}(\R^{2})}^{2},
\end{align}
where we use Bernstein's lemma to obtain the penultimate inequality and that $J_{l}$ is small to obtain the ultimate inequality.

\item[Estimate for \eqref{eq:ibs1_err2}:] %Estimate for second error
Observe that by splitting $\mathcal{N}=\mathcal{N}_{1}+\mathcal{N}_{2}$ and using triangle and H\"{o}lder's inequalities
\begin{align}
|\eqref{eq:ibs1_err2}| &\lesssim \|v\|_{L_{t}^{\infty}L_{x}^{2}(J_{l}\times\R^{2})}^{2}\|\mathcal{N}_{1}\|_{L_{t,x}^{4/3}(J_{l}\times\R^{2})} \|(\nabla-i\xi(t)w\|_{L_{t,x}^{4}(J_{l}\times\R^{2})}\nonumber\\
&\phantom{=} +2^{-l_{2}}\|v\|_{L_{t}^{\infty}L_{x}^{2}(J_{l}\times\R^{2})}^{2}\int_{J_{l}}|\xi'(t)| \|P_{\xi(t),l_{2}-2\leq\cdot\leq l_{2}+2}u(t)\|_{L_{x}^{2}(\R^{2})} \|(\nabla-i\xi(t))P_{\xi(t),\leq l_{2}}u(t)\|_{L_{x}^{2}(\R^{2})}dt, 
\end{align}
By triangle and H\"{o}lder's inequalities, Bernstein's lemma, followed by Plancherel's theorem, the preceding expression is
\begin{equation}
\lesssim 2^{l_{2}}\|v_{0}\|_{L^{2}(\R^{2})}^{2} +2^{-l_{2}}\|v_{0}\|_{L^{2}(\R^{2})}^{2}\int_{J_{l}}|\xi'(t)| \|P_{\xi(t),l_{2}-2\leq\cdot\leq l_{2}+2}u(t)\|_{L_{x}^{2}(\R^{2})} \|(\nabla-i\xi(t))P_{\xi(t),\leq l_{2}}u(t)\|_{L_{x}^{2}(\R^{2})}dt,
\end{equation}
where we also use that $J_{l}$ is small.

\item[Estimate for \eqref{eq:ibs1_err3}:] %Estimate for third error
Observe that after integrating by parts in $x$, we have that
\begin{equation}
\begin{split}
\int_{J_{l}}\int_{\R^{4}}\frac{x-y}{|x-y|}\cdot |v(t,y)|^{2}\Re{\bar{w}(\nabla-i\xi(t))\mathcal{N}}(t,x)dxdydt &= -\int_{J_{l}}\int_{\R^{4}}\frac{x-y}{|x-y|}\cdot |v(t,y)|^{2}\Re{\ol{(\nabla-i\xi(t))w}\mathcal{N}}(t,x) dxdydt \\
&\phantom{=} -\int_{J_{l}}\int_{\R^{4}}\frac{1}{|x-y|} |v(t,y)|^{2}\Re{\bar{w}\mathcal{N}}(t,x) dxdydt,
\end{split}
\end{equation}
which implies that
\begin{equation}
\eqref{eq:ibs1_err3} = \eqref{eq:ibs1_err2} -2C\int_{J_{l}}\int_{\R^{4}}\frac{1}{|x-y|} |v(t,y)|^{2}\Re{\bar{w}\mathcal{N}}(t,x) dxdydt
\end{equation}
To estimate the second term on the RHS of the preceding equation, we split $\mathcal{N}=\mathcal{N}_{1}+\mathcal{N}_{2}$.

%Contribution of $\mathcal{N}_{1}$
First, consider the contribution of $\mathcal{N}_{1}$. We disregard the commutator structure in $\mathcal{N}_{1}$ and use H\"{o}lder's inequality, Hardy-Littlewood-Sobolev and Bernstein's lemmas, Calder\'{o}n-Zygmund theorem, and Strichartz estimates to estimate
\begin{align}
\int_{J_{l}}\int_{\R^{4}} \frac{1}{|x-y|}|v(t,y)|^{2} |(w\mathcal{N}_{1})(t,x)|dxdydt &\lesssim \||\nabla|^{-1}(|v|^{2})\|_{L_{t}^{3}L_{x}^{6}(J_{l}\times\R^{2})} \|w\|_{L_{t,x}^{\infty}(J_{l}\times\R^{2})} \|\mathcal{N}_{1}\|_{L_{t}^{3/2}L_{x}^{6/5}(J_{l}\times\R^{2})} \nonumber\\
&\lesssim 2^{l_{2}} \|v\|_{L_{t}^{6}L_{x}^{3}(J_{l}\times\R^{2})}^{2} \|w\|_{L_{t}^{9/2}L_{x}^{18/5}(J_{l}\times\R^{2})}^{3}&\nonumber\\
&\lesssim 2^{l_{2}}\|v_{0}\|_{L^{2}(\R^{2})}^{2}.
\end{align}
where we use mass conservation and that $J_{l}$ is small to obtain the ultimate inequality.

Next, consider the contribution of $\mathcal{N}_{2}$. Here, the structure of the eeDS equation does not play a role. By the Fubini-Tonelli theorem and Cauchy-Schwarz in $x$,
\begin{align}
&\int_{J_{l}}\int_{\R^{4}} \frac{1}{|x-y|}|v(t,y)|^{2} |w(t,x)| |\mathcal{N}_{2}(t,x)|dxdydt \nonumber\\
&\phantom{=}\leq \|v\|_{L_{t}^{\infty}L_{x}^{2}(J_{l}\times\R^{2})} \int_{J_{l}}\||x-y|^{-1/2}w\|_{L_{y}^{\infty}L_{x}^{2}(\R^{2}\times\R^{2})} \| |x-y|^{-1/2}\mathcal{N}_{2}\|_{L_{y}^{\infty}L_{x}^{2}(\R^{2}\times\R^{2})}dt \nonumber\\
&\phantom{=}\lesssim \|v_{0}\|_{L^{2}(\R^{2})}^{2} \int_{J_{l}} \||\nabla-i\xi(t)|^{1/2}P_{\xi(t),\leq l_{2}}u(t)\|_{L_{x}^{2}(\R^{2})} 2^{-l_{2}} \| |\nabla-i\xi(t)|^{1/2}P_{\xi(t),l_{2}-2\leq\cdot\leq l_{2}+2}u(t)\|_{L_{x}^{2}(\R^{2})}dt \nonumber\\
&\phantom{=}\lesssim 2^{-\frac{l_{2}}{2}} \|v_{0}\|_{L^{2}(\R^{2})}^{2} \int_{J_{l}}\|P_{\xi(t),l_{2}-2\leq\cdot\leq l_{2}+2}u(t)\|_{L_{x}^{2}(\R^{2})} \| |\nabla-i\xi(t)|^{1/2} P_{\xi(t),\leq l_{2}}u(t)\|_{L_{x}^{2}(\R^{2})} dt,\label{eq:ibs1_hardy}
\end{align}
where we use Hardy's inequality to obtain the penultimate inequality and Plancherel's theorem to obtain the ultimate inequality.

Collecting our estimates, we have shown that
\begin{equation}
\begin{split}
|\eqref{eq:ibs1_err3}| &\lesssim 2^{l_{2}}\|v_{0}\|_{L^{2}(\R^{2})}^{2} + 2^{-l_{2}}\|v_{0}\|_{L^{2}(\R^{2})}^{2}\int_{J_{l}}|\xi'(t)| \|P_{\xi(t),l_{2}-2\leq\cdot\leq l_{2}+2}u(t)\|_{L_{x}^{2}(\R^{2})} \|(\nabla-i\xi(t))P_{\xi(t),\leq l_{2}}u(t)\|_{L_{x}^{2}(\R^{2})}dt \\
&\phantom{=}+2^{-\frac{l_{2}}{2}} \|v_{0}\|_{L^{2}(\R^{2})}^{2} \int_{J_{l}}\|P_{\xi(t),l_{2}-2\leq\cdot\leq l_{2}+2}u(t)\|_{L_{x}^{2}(\R^{2})} \| |\nabla-i\xi(t)|^{1/2} P_{\xi(t),\leq l_{2}}u(t)\|_{L_{x}^{2}(\R^{2})} dt.
\end{split}
\end{equation}

\item[Estimate for \eqref{eq:ibs1_err1}:] %Estimate for first error
As before, we split $\mathcal{N}=\mathcal{N}_{1}+\mathcal{N}_{2}$ and first consider $\mathcal{N}_{1}$. We disregard the commutator structure of $\mathcal{N}_{1}$ and crudely estimate with H\"{o}lder's inequality, Plancherel's theorem, Strichartz estimates, and mass conservation to obtain that
\begin{align}
\left|\int_{J_{l}}\int_{\R^{4}}\frac{x-y}{|x-y|}\cdot\Im{\bar{w}\mathcal{N}_{1}}(t,y)\Im{\bar{v}(\nabla-i\xi(t))v}(t,x)dxdy dt\right| \lesssim 2^{i}\|v_{0}\|_{L^{2}(\R^{2})}^{2}.
\end{align}

We next consider the contribution of $\mathcal{N}_{2}$. We split $w=P_{\xi(t), l_{2}-5\leq\cdot\leq l_{2}}u+ P_{\xi(t),<l_{2}-5}u$. Take $P_{\xi(t),l_{2}-5\leq\cdot\leq l_{2}}u$. By the Fubini-Tonelli theorem, Cauchy-Schwarz in $x$ and $y$, followed by Plancherel's theorem, we have that
\begin{equation}
\begin{split}
&\int_{J_{l}}\int_{\R^{4}}|v(t,y)| |(\nabla-i\xi(t))v(t,y)| |P_{\xi(t),l_{2}-5\leq\cdot\leq l_{2}}u(t,x)| |\mathcal{N}_{2}(t,x)| dxdydt \\
&\phantom{=}\lesssim 2^{i-2l_{2}}\|v_{0}\|_{L^{2}(\R^{2})}^{2}\int_{J_{l}}|\xi'(t)| \|P_{\xi(t),l_{2}-2\cdot\leq l_{2}+2}u(t)\|_{L_{x}^{2}(\R^{2})} \| (\nabla-i\xi(t))P_{\xi(t),\leq l_{2}}u(t)\|_{L_{x}^{2}(\R^{2})}dt.
\end{split}
\end{equation}
Finally, take $P_{\xi(t),<l_{2}-5}u$. Since $\mathcal{N}_{2}$ has Fourier support in the dyadic annulus $A(\xi(t), l_{2}-2, l_{2}+2)$, it follows that $(\ol{P_{\xi(t),<l_{2}-5}u})\mathcal{N}_{2}$ has Fourier support in the dyadic annulus $A(0, l_{2}-3, l_{2}+3)$. Since $\nabla_{x}|x-y|=\frac{x-y}{|x-y|}$, we have by the Fubini-Tonelli theorem, Plancherel's, and the characterization of homogeneous distributions $|x|^{z}$ (see theorem 2.4.6 in \cite{grafakos2014c}) that there is a constant $C$ such that
\begin{equation}\label{eq:ibs1_fs_obs}
\begin{split}
&\int_{J_{l}}\int_{\R^{4}}\frac{x-y}{|x-y|}\cdot\Im{(\ol{P_{\xi(t),<l_{2}-5}u})\mathcal{N}_{2}}(t,y) \Im{\bar{v}(\nabla-i\xi(t))v}(t,x)dxdydt \\
&\phantom{=} =\int_{J_{l}}\int_{\R^{2}}\frac{iC\xi}{|\xi|^{3}}\mathcal{F}\paren*{\Im{(\ol{P_{\xi(t),<l_{2}-5}u})\mathcal{N}_{2}}}(t,\xi)\cdot\ol{\mathcal{F}\paren*{\Im{\bar{v}(\nabla-i\xi(t))v}}(t,\xi)}d\xi dt.
\end{split}
\end{equation}
Therefore by Cauchy-Schwarz in $\xi$, followed by Plancherel's theorem, the modulus of the preceding line is
\begin{align}
&\lesssim 2^{-2l_{2}}\int_{J_{l}} \|(\ol{P_{\xi(t),< l_{2}-5}u(t)})\mathcal{N}_{2}(t)\|_{L_{x}^{2}(\R^{2})} \|\dot{P}_{l_{2}-3\leq\cdot\leq l_{2}+3}\paren*{\bar{v}(t)(\nabla-i\xi(t))v(t)}\|_{L_{x}^{2}(\R^{2})}.
\end{align}
By Bernstein's lemma,
\begin{equation}
\|\dot{P}_{l_{2}-3\leq\cdot\leq l_{2}+3}\paren*{\bar{v}(t)(\nabla-i\xi(t))v(t)}\|_{L_{x}^{2}(\R^{2})} \lesssim 2^{l_{2}} \|\bar{v}(t)(\nabla-i\xi(t))v(t)\|_{L_{x}^{1}(\R^{2})} \lesssim 2^{i+l_{2}}\|v_{0}\|_{L^{2}(\R^{2})}^{2},
\end{equation}
where we use Cauchy-Schwarz and Plancherel's theorem to obtain the ultimate inequality. By H\"{o}lder's inequality, followed by Sobolev embedding,
\begin{align}
\|(\ol{P_{\xi(t),<l_{2}-5}u(t)})\mathcal{N}_{2}(t)\|_{L_{x}^{2}(\R^{2})} &\leq \|P_{\xi(t),<l_{2}-5}u(t)\|_{L_{x}^{4}(\R^{2})} \|\mathcal{N}_{2}(t)\|_{L_{x}^{4}(\R^{2})} \nonumber\\
&\lesssim 2^{-l_{2}/2}|\xi'(t)| \||\nabla-i\xi(t)|^{1/2}P_{\xi(t),<l_{2}-5}u(t)\|_{L_{x}^{2}(\R^{2})} \|P_{\xi(t),l_{2}-2\leq\cdot\leq l_{2}+2}u(t)\|_{L_{x}^{2}(\R^{2})}.
\end{align}
Hence,
\begin{equation}
\begin{split}
&2^{-2l_{2}}\int_{J_{l}} \|(\ol{P_{\xi(t),< l_{2}-5}u(t)})\mathcal{N}_{2}(t)\|_{L_{x}^{2}(\R^{2})} \|\dot{P}_{l_{2}-3\leq\cdot\leq l_{2}+3}\paren*{\bar{v}(t)(\nabla-i\xi(t))v(t)}\|_{L_{x}^{2}(\R^{2})} \\
&\phantom{=} \lesssim 2^{i-\frac{3}{2}l_{2}}\|v_{0}\|_{L^{2}(\R^{2})}^{2}\int_{J_{l}}|\xi'(t)| \|P_{\xi(t),l_{2}-2\leq\cdot\leq l_{2}+2}u(t)\|_{L_{x}^{2}(\R^{2})} \||\nabla-i\xi(t)|^{1/2}P_{\xi(t),<l_{2}-5}u(t)\|_{L_{x}^{2}(\R^{2})}dt. \label{eq:ibs1_N2}
\end{split}
\end{equation}

With this last estimate, we have shown that
\begin{equation}
\begin{split}
|\eqref{eq:ibs1_err1}| &\lesssim 2^{i}\|v_{0}\|_{L^{2}(\R^{2})}^{2} + 2^{i-2l_{2}}\|v_{0}\|_{L^{2}(\R^{2})}^{2}\int_{J_{l}}|\xi'(t)| \|P_{\xi(t),l_{2}-2\leq\cdot\leq l_{2}+2}u(t)\|_{L_{x}^{2}(\R^{2})}^{2} \|(\nabla-i\xi(t))P_{\xi(t),\leq l_{2}}u(t)\|_{L_{x}^{2}(\R^{2})}dt \\
&\phantom{=} + 2^{i-\frac{3}{2}l_{2}}\|v_{0}\|_{L^{2}(\R^{2})}^{2}\int_{J_{l}}|\xi'(t)| \|P_{\xi(t),l_{2}-2\leq\cdot\leq l_{2}+2}u(t)\|_{L_{x}^{2}(\R^{2})} \||\nabla-i\xi(t)|^{1/2}P_{\xi(t),<l_{2}-5}u(t)\|_{L_{x}^{2}(\R^{2})}dt.
\end{split}
\end{equation}
\end{description}

%Bookeeping of estimates
After a bit of bookkeeping, we have shown that
\begin{align}
\|vw\|_{L_{t,x}^{2}(J_{l}\times\R^{2})} &\lesssim 2^{k-i}\|v_{0}\|_{L^{2}(\R^{2})}^{2} + 2^{k+l_{2}-2i}\|v_{0}\|_{L^{2}(\R^{2})}^{2} \nonumber\\
&\phantom{=} +2^{k-2i-l_{2}}\|v_{0}\|_{L^{2}(\R^{2})}^{2}\int_{J_{l}}|\xi'(t)| \|P_{\xi(t),l_{2}-2\leq\cdot\leq l_{2}+2}u(t)\|_{L_{x}^{2}(\R^{2})} \|(\nabla-i\xi(t))P_{\xi(t),\leq l_{2}}u(t)\|_{L_{x}^{2}(\R^{2})}dt \nonumber\\
&\phantom{=} +2^{k-2i-\frac{l_{2}}{2}}\|v_{0}\|_{L^{2}(\R^{2})}^{2}\int_{J_{l}}|\xi'(t)| \|P_{\xi(t),l_{2}-2\leq\cdot\leq l_{2}+2}u(t)\|_{L_{x}^{2}(\R^{2})} \||\nabla-i\xi(t)|^{1/2}P_{\xi(t),\leq l_{2}}u(t)\|_{L_{x}^{2}(\R^{2})}dt \nonumber\\
&\phantom{=} + 2^{k-i-2l_{2}}\|v_{0}\|_{L^{2}(\R^{2})}^{2}\int_{J_{l}}|\xi'(t)| \|P_{\xi(t),l_{2}-2\leq\cdot\leq l_{2}+2}u(t)\|_{L_{x}^{2}(\R^{2})} \|(\nabla-i\xi(t))P_{\xi(t),\leq l_{2}}u(t)\|_{L_{x}^{2}(\R^{2})}dt \nonumber\\
&\phantom{=} + 2^{k-i-\frac{3}{2}l_{2}}\|v_{0}\|_{L^{2}(\R^{2})}^{2}\int_{J_{l}}|\xi'(t)| \|P_{\xi(t),l_{2}-2\leq\cdot\leq l_{2}+2}u(t)\|_{L_{x}^{2}(\R^{2})} \||\nabla-i\xi(t)|^{1/2}P_{\xi(t),\leq l_{2}}u(t)\|_{L_{x}^{2}(\R^{2})}dt \nonumber\\
&\lesssim 2^{k-i}\|v_{0}\|_{L^{2}(\R^{2})}^{2} \nonumber\\
&\phantom{=} + 2^{k-l_{2}-i}\|v_{0}\|_{L^{2}(\R^{2})}^{2}\int_{J_{l}}|\xi'(t)| \|P_{\xi(t),l_{2}-2\leq\cdot\leq l_{2}+2}u(t)\|_{L_{x}^{2}(\R^{2})}\sum_{0\leq l_{1}\leq l_{2}}2^{\frac{l_{1}-l_{2}}{2}} \|P_{\xi(t),l_{1}}u(t)\|_{L_{x}^{2}(\R^{2})}dt,
\end{align}
which completes the proof of proposition \ref{prop:ibs_1}.
\end{proof}

\subsection{Bilinear Strichartz estimate II} \label{ssec:BSE_2}
We recall the statement of proposition \ref{prop:ibs_2}.
\ibsII*

\begin{proof}
Let $v_{0}$ satisfy the conditions in the statement of the proposition and define $v$ and $w$ as in the proof of proposition \ref{prop:ibs_1}. Defining $M_{\omega}(t)$ and $M(t)$ as in the proof of proposition \ref{prop:ibs_1} and following the same initial arguments, we have the estimate
\begin{align}
\int_{G_{\beta}^{l_{2}}}\int_{\R^{2}} |w(t,x)v(t,x)|^{2} &\lesssim 2^{k-2i}\sup_{t\in G_{\beta}^{l_{2}}} |M(t)|\label{eq:ibs2_main}\\ 
&\phantom{=}+2^{k-2i}\int_{G_{\beta}^{l_{2}}}\int_{\R^{4}}\frac{1}{|x-y|} |v(t,y)|^{2} |\vec{R}^{2}(|w|^{2})(t,x)|^{2}dxdydt\label{eq:ibs2_nl}\\
&\phantom{=}+2^{k-2i}\left|\int_{G_{\beta}^{l_{2}}}\int_{\R^{4}}\frac{x-y}{|x-y|}\cdot|v(t,y)|^{2}\Re{\bar{\mathcal{N}}(\nabla-i\xi(t))w}(t,x)dxdydt\right| \label{eq:ibs2_err1}\\
&\phantom{=}+2^{k-2i}\left|\int_{G_{\beta}^{l_{2}}}\int_{\R^{4}}\frac{x-y}{|x-y|}\cdot |v(t,y)|^{2}\Re{\bar{w}(\nabla-i\xi(t))\mathcal{N}}(t,x)dxdydt\right| \label{eq:ibs2_err2}\\
&\phantom{=}+2^{k-2i}\left|\int_{G_{\beta}^{l_{2}}}\int_{\R^{4}}\frac{x-y}{|x-y|}\cdot \Im{\bar{w}\mathcal{N}}(t,y)\Im{\bar{v}(\nabla-i\xi(t))v}(t,x)dxdydt\right|. \label{eq:ibs2_err3}
\end{align}
Hereafter, the mixed norm notation $L_{t}^{p}L_{x}^{q}$ is taken over the taken over the spacetime slab $G_{\beta}^{l_{2}}\times\R^{2}$.

We estimate each of the terms \eqref{eq:ibs2_main}-\eqref{eq:ibs2_err3} separetly.

\begin{description}[leftmargin=*]
\item[Estimate for \eqref{eq:ibs2_main}:]
As in the case of the estimate for \eqref{eq:ibs1_main} in the proof of proposition \ref{prop:ibs_1}, we see that $|M(t)|\lesssim 2^{i}\|v_{0}\|_{L^{2}(\R^{2})}^{2}$ for all $t\in G_{\beta}^{l_{2}}$, which implies that
\begin{equation}
\eqref{eq:ibs2_main} \lesssim 2^{k-i}\|v_{0}\|_{L^{2}(\R^{2})}^{2}.
\end{equation}

\item[Estimate for \eqref{eq:ibs2_nl}:]
By H\"{o}lder's inequality, Hardy-Littlewood-Sobolev lemma, Calder\'{o}n-Zygmund theorem, and Strichartz estimates, we have that
\begin{equation}
\eqref{eq:ibs2_nl} \lesssim 2^{k-2i} \|v\|_{L_{t}^{6}L_{x}^{3}}^{2}\|w\|_{L_{t}^{6}L_{x}^{24/5}}^{4} \lesssim 2^{k-2i}\|v_{0}\|_{L^{2}}^{2} \|w\|_{L_{t}^{6}L_{x}^{24/5}}^{4}.
\end{equation}
By lemma \ref{lem:BT_embed}, $\|w\|_{L_{t}^{6}L_{x}^{24/5}}\lesssim 2^{\frac{l_{2}}{4}}\|w\|_{\tilde{X}_{l_{2}}(G_{\beta}^{l_{2}}\times\R^{2})}$, which implies that
\begin{equation}
	\eqref{eq:ibs2_nl} \lesssim 2^{k+l_{2}-2i}\|v_{0}\|_{L^{2}}^{2}\|u\|_{\tilde{X}_{l_{2}}(G_{\beta}^{l_{2}}\times\R^{2})}^{4} \lesssim 2^{k+l_{2}-2i}\|v_{0}\|_{L^{2}}^{2} \|u\|_{\tilde{X}_{i}(G_{\alpha}^{i}\times\R^{2})}^{4}.
\end{equation}

\item[Estimate for \eqref{eq:ibs2_err1}:]
Proceeding as in the case of the estimate for \eqref{eq:ibs1_err2} in the proof of proposition \ref{prop:ibs_1}, we have that
\begin{align}
\eqref{eq:ibs2_err1} &\lesssim 2^{k-l_{2}-2i}\|v_{0}\|_{L^{2}}^{2}\int_{G_{\beta}^{l_{2}}} |\xi'(t)| \|P_{\xi(t),l_{2}-2\leq\cdot\leq l_{2}+2}u(t)\|_{L_{x}^{2}} \| (\nabla-i\xi(t))P_{\xi(t),\leq l_{2}}u(t)\|_{L_{x}^{2}}dt \nonumber\\
&\phantom{=}+2^{k-2i}\|v_{0}\|_{L^{2}}^{2}\|\mathcal{N}_{1}\|_{L_{t}^{3/2}L_{x}^{6/5}} \|(\nabla-i\xi(t))w\|_{L_{t}^{3}L_{x}^{6}} \nonumber\\
&\eqqcolon \mathrm{Term}_{1}+\mathrm{Term}_{2}.
\end{align}

To estimate $\mathrm{Term}_{1}$, we use Plancherel's theorem and mass conservation to obtain
\begin{align}
\mathrm{Term}_{1} \lesssim 2^{k-2i}\|v_{0}\|_{L^{2}}^{2} 2^{-20}\epsilon_{1}^{-1/2} \int_{G_{\beta}^{l_{2}}} N(t)^{3}dt \leq \epsilon_{3}^{1/2} 2^{k+l_{2}-2i}\|v_{0}\|_{L^{2}}^{2}.
\end{align}

To estimate $\mathrm{Term}_{2}$, we first note that by lemma \ref{lem:BT_embed}, we have the estimate
\begin{equation}
\|(\nabla-i\xi(t))w\|_{L_{t}^{3}L_{x}^{6}} \lesssim 2^{l_{2}}\|u\|_{\tilde{X}_{l_{2}}(G_{\beta}^{l_{2}}\times\R^{2})} \lesssim 2^{l_{2}} \|u\|_{\tilde{X}_{i}(G_{\alpha}^{i}\times\R^{2})}.
\end{equation}
Now unlike in the proof of proposition \ref{prop:ibs_1}, we no longer ignore the commutator structure in our estimation of $\mathcal{N}_{1}$. We perform a near-far decomposition of $u$, writing $u=u_{l}+u_{h}$ where $u_{l}:=P_{\xi(t), \leq l_{2}-5}u$. Substituting this decomposition into $\mathcal{N}_{1}$,we obtain
\begin{equation}
\mathcal{N}_{1} = \mathcal{N}_{1,0}+\mathcal{N}_{1,1}+\mathcal{N}_{1,2}+\mathcal{N}_{1,3},
\end{equation}
where $\mathcal{N}_{1,j}$ contains $j$ factors $u_{h}$ and $3-j$ factors $u_{l}$, for $j=0,\ldots,3$:
\begin{align}
\mathcal{N}_{1,0} &\coloneqq P_{\xi(t), \leq l_{2}}[\E(|u_{l}|^{2})u_{l}]-\E(|u_{l}|^{2})u_{l} \\
\mathcal{N}_{1,1} &\coloneqq \comm{P_{\xi(t), \leq l_{2}}}{\E(|u_{l}|^{2})}u_{h} + 2P_{\xi(t), \leq l_{2}}\left[\E\paren*{\Re{u_{l}\bar{u}_{h}}}u_{l}\right] - 2\E\paren*{\Re{u_{l}(\ol{P_{\xi(t), \leq l_{2}}u_{h}})}}u_{l} \\
\mathcal{N}_{1,2} &\coloneqq P_{\xi(t), \leq l_{2}}\brak*{\E(|u_{h}|^{2})u_{l}} +2P_{\xi(t),\leq l_{2}}\brak*{\E\paren*{\Re{u_{h}\bar{u_{l}}}}u_{h}} \nonumber\\
&\phantom{=}-\E(|P_{\xi(t),\leq l_{2}}u_{h}|^{2})u_{l} -2\E\paren*{Re{(P_{\xi(t), \leq l_{2}}u_{h})\bar{u_{l}}}}(P_{\xi(t), \leq l_{2}}u_{h})\\
\mathcal{N}_{1,3} &\coloneqq P_{\xi(t), \leq l_{2}}\brak*{\E(|u_{h}|^{2})u_{h}} - \E(|P_{\xi(t), \leq l_{2}}u_{h}|^{2})(P_{\xi(t),\leq l_{2}}u_{h}).
\end{align}

We can quickly dispense with $\mathcal{N}_{1,0}$, as Fourier support analysis shows that $P_{\xi(t),\leq l_{2}}[\E(|u_{l}|^{2})u_{l}]=\E(|u_{l}|^{2})u_{l}$, which implies that $\mathcal{N}_{1,0}=0$.

To estimate $\mathcal{N}_{1,2}$ and $\mathcal{N}_{1,3}$, we use triangle and H\"{o}lder's inequalities, Calder\'{o}n-Zygmund theorem, followed by lemma \ref{lem:lohi_embed} and mass conservation to obtain
\begin{equation}
\|\mathcal{N}_{1,2}\|_{L_{t}^{3/2}L_{x}^{6/5}} + \|\mathcal{N}_{1,3}\|_{L_{t}^{3/2}L_{x}^{6/5}} \lesssim \|u_{h}\|_{L_{t}^{3}L_{x}^{6}}^{2} \|u\|_{L_{t}^{\infty}L_{x}^{2}} \lesssim \|u\|_{\tilde{X}_{l_{2}}(G_{\beta}^{l_{2}}\times\R^{2})}^{2} \leq \|u\|_{\tilde{X}_{i}(G_{\alpha}^{i}\times\R^{2})}^{2}.
\end{equation}

To estimate $\mathcal{N}_{1,1}$, we divide into two cases: (1) the factor $u_{h}$ falls inside the argument of the operator $\E$ or (2) the factor $u_{h}$ falls outside the argument of the operator $\E$. In case (1), we observe from the fundamental theorem of calculus that
\begin{equation}
\left|\paren*{\comm{P_{\xi(t),\leq l_{2}}}{\E(|u_{l}|^{2})}u_{h}}(x)\right| \leq  \int_{0}^{1}\int_{\R^{2}} |2^{2l_{2}}\phi^{\vee}(2^{l_{2}}y)| |y| |u_{h}(x-y)| |(\nabla \E)(|u_{l}|^{2})(x-\theta y)|dyd\theta
\end{equation}
and from the product rule that
\begin{equation}
\nabla |u_{l}|^{2}=\nabla(e^{-ix\cdot\xi(t)}u_{l})\ol{(e^{-ix\cdot\xi(t)}u_{l})} + \ol{\nabla(e^{-ix\cdot\xi(t)}u_{l})}(e^{-ix\cdot\xi(t)}u_{l}) = 2\Re{\ol{(e^{-ix\cdot\xi(t)}u_{l})}\nabla(e^{-ix\cdot\xi(t)}u_{l})}.
\end{equation}
Therefore by Minkowski's and H\"{o}lder's inequalities, Calder\'{o}n-Zygmund theorem, and mass conservation followed by lemmas \ref{lem:BT_embed} and \ref{lem:lohi_embed}, we have that
\begin{align}
\|\comm{P_{\xi(t),\leq l_{2}}}{\E(|u_{l}|^{2})}u_{h}\|_{L_{t}^{3/2}L_{x}^{6/5}} \lesssim 2^{-l_{2}}\|u_{h}\|_{L_{t}^{3}L_{x}^{6}}\|u_{l}\|_{L_{t}^{\infty}L_{x}^{2}} \|\nabla(e^{-ix\cdot\xi(t)}u_{l})\|_{L_{t}^{3}L_{x}^{6}}\lesssim \|u\|_{\tilde{X}_{i}(G_{\alpha}^{i}\times\R^{2})}^{2}.
\end{align}
In case (2), we can add zero to write
\begin{equation}
P_{\xi(t), \leq l_{2}}[\E(u_{h}\bar{u_{l}})u_{l}] - \E((P_{\xi(t), \leq l_{2}}u_{h})\bar{u_{l}})u_{l} = \comm{P_{\xi(t),\leq l_{2}}}{u_{l}}\E(u_{h}\bar{u_{l}}) +\E\paren*{\comm{P_{\xi(t),\leq l_{2}}}{\bar{u_{l}}}u_{h}}u_{l}.
\end{equation}
We estimate the $L_{t}^{3/2}L_{x}^{6/5}$ norm of the first term on the RHS as in case (1). For the second term on the RHS, we use H\"{o}lder's inequality, Calder\'{o}n-Zygmund theorem, and mass conservation to obtain
\begin{equation}
\|\E\paren*{\comm{P_{\xi(t),\leq l_{2}}}{\bar{u_{l}}}u_{h}}u_{l}\|_{L_{t}^{3/2}L_{x}^{6/5}} \lesssim \|\comm{P_{\xi(t),\leq l_{2}}}{\bar{u_{l}}}u_{h}\|_{L_{t}^{3/2}L_{x}^{6}} \lesssim \|u\|_{\tilde{X}_{i}(G_{\beta}^{l_{2}})}^{2},
\end{equation}
where we proceed as in case (1) to obtain the ultimate inequality. Thus, $\|\mathcal{N}_{1}\|_{L_{t}^{3/2}L_{x}^{6/5}} \lesssim \|u\|_{\tilde{X}_{i}(G_{\alpha}^{i}\times\R^{2})}^{2}$.

Bookkeeping our estimates, we conclude that
\begin{equation}
|\eqref{eq:ibs2_err1}| \lesssim 2^{k+l_{2}-2i}\|v_{0}\|_{L^{2}}^{2}\left(1+\|u\|_{\tilde{X}_{i}(G_{\alpha}^{i}\times\R^{2})}^{3}\right).
\end{equation}

\item[Estimate for \eqref{eq:ibs2_err2}:]
Integrating by parts in $x$, we see that
\begin{equation}
\eqref{eq:ibs2_err2} \leq \eqref{eq:ibs2_err1} + 2^{k-2i}\left|\int_{G_{\beta}^{l_{2}}}\int_{\R^{4}}\frac{1}{|x-y|} |v(t,y)|^{2}\Re{\bar{w}\mathcal{N}}(t,x)dxdydt\right|.
\end{equation}
Using  triangle and H\"{o}lder's inequalities, Hardy-Littlewood-Sobolev lemma, Calder\'{o}n-Zygmund theorem, and lemma \ref{lem:BT_embed}, we see that
\begin{align}
&2^{k-2i}\left|\int_{G_{\beta}^{l_{2}}}\int_{\R^{4}}\frac{1}{|x-y|} |v(t,y)|^{2}\Re{\bar{w}\mathcal{N}}(t,x)dxdydt\right| \nonumber\\
&\lesssim 2^{k-2i}\|v\|_{L_{t}^{6}L_{x}^{3}}^{2}\|\mathcal{N}_{1}\|_{L_{t}^{3/2}L_{x}^{6/5}}\|w\|_{L_{t,x}^{\infty}}+ 2^{k-2i}\int_{G_{\beta}^{l_{2}}}\int_{\R^{4}}\frac{1}{|x-y|} |v(t,y)|^{2}|(\bar{w}\mathcal{N}_{2})(t,x)|dxdydt \nonumber\\
&\lesssim 2^{k+l_{2}-2i}\|v_{0}\|_{L^{2}}^{2}\left(1+\|u\|_{\tilde{X}_{i}(G_{\alpha}^{i}\times\R^{2})}^{3}\right) + 2^{k-2i}\int_{G_{\beta}^{l_{2}}}\int_{\R^{4}}\frac{1}{|x-y|} |v(t,y)|^{2}|(\bar{w}\mathcal{N}_{2})(t,x)dxdydt.
\end{align}
To estimate the second term on the RHS, we use Cauchy-Schwarz and Hardy's inequality as in the proof of estimate \eqref{eq:ibs1_hardy} to obtain it is $\lesssim 2^{k+l_{2}-2i}(1+\|u\|_{\tilde{X}_{i}(G_{\alpha}^{i}\times\R^{2})}^{3})$. Therefore
\begin{equation}
|\eqref{eq:ibs2_err2}| \lesssim 2^{k+l_{2}-2i}\|v_{0}\|_{L^{2}}^{2}\left(1+\|u\|_{\tilde{X}_{i}(G_{\alpha}^{i}\times\R^{2})}^{3}\right).
\end{equation}

\item[Estimate for \eqref{eq:ibs2_err3}:]
As the reader will see, this term is the most difficult to estimate. We split $\mathcal{N}=\mathcal{N}_{1}+\mathcal{N}_{2}$ and first estimate the contribution of $\mathcal{N}_{2}$. Using the observation \eqref{eq:ibs1_fs_obs} together with Cauchy-Schwarz, Plancherel's theorem, triangle inequality, and Bernstein's lemma, we see that
\begin{align}
&\left|2^{k-2i}\int_{G_{\beta}^{l_{2}}}\frac{x-y}{|x-y|}\cdot\Im{\bar{w}\mathcal{N}_{2}}(t,y)\Im{\bar{v}(\nabla-i\xi(t))v}(t,x)dxdydt\right| \nonumber\\
&\lesssim 2^{k-2i}2^{i-l_{2}}\|v_{0}\|_{L^{2}}^{2} \int_{G_{\beta}^{l_{2}}} |\xi'(t)| \|P_{\xi(t), l_{2}-3\leq\cdot\leq l_{2}+3}u(t)\|_{L_{x}^{2}} \sum_{l_{3}\leq l_{2}}2^{l_{3}-l_{2}}\|P_{\xi(t), l_{3}}u(t)\|_{L_{x}^{2}}dt \label{eq:ibs2_err3_N2'}\\
&\lesssim 2^{k-i}\|v_{0}\|_{L^{2}}^{2},\label{eq:ibs2_err3_N2}
\end{align}
where we use $|\xi'(t)| \leq \epsilon_{1}^{-1/2}2^{-20}N(t)^{3}$, Plancherel's theorem, and mass conservation to obtain the ultimate inequality.

We next estimate the contribution of $\mathcal{N}_{1}$. We perform a near-far decomposition $u=u_{l}+u_{h}$, where $u_{l}\coloneqq P_{\xi(t),\leq l_{2}-5}u$ and substitute this identity into $\Im{\bar{w}\mathcal{N}_{1}}$ to obtain the decomposition
\begin{equation}
\Im{\bar{w}\mathcal{N}_{1}} = F_{0}+F_{1}+F_{2}+F_{3}+F_{4},
\end{equation}
where $F_{j}$ contains $j$ factors $u_{h}$ and $4-j$ factors $u_{l}$, for $j=0,\ldots,3$:

\begin{align}
F_{0} &\coloneqq \Im{\bar{u}_{l}P_{\xi(t),\leq l_{2}}\left[\E(|u_{l}|^{2})u_{l}\right] - |u_{l}|^{2}\E(|u_{l}|^{2})} \label{eq:ibs2_F0}\\
F_{1} &\coloneqq \Im{\bar{u}_{l}P_{\xi(t), \leq l_{2}}\left[\E(|u_{l}|^{2})u_{h}\right] + 2\bar{u}_{l}P_{\xi(t), \leq l_{2}}\left[\E\paren*{\Re{u_{l}\bar{u}_{h}}}u_{l}\right] + (\ol{P_{\xi(t), \leq l_{2}}u_{h}})P_{\xi(t), \leq l_{2}}\left[\E(|u_{l}|^{2})u_{l}\right]} \nonumber\\
&\phantom{=}-\Im{\bar{u}_{l}\E(|u_{l}|^{2})(P_{\xi(t), \leq l_{2}}u_{h}) + 2\bar{u}_{l}\E\paren*{\Re{u_{l}(\ol{P_{\xi(t), \leq l_{2}}u_{h}})}}u_{l} + (\ol{P_{\xi(t), \leq l_{2}}u_{h}})\E(|u_{l}|^{2})u_{l}} \label{eq:ibs2_F1}\\
F_{2} &\coloneqq 2\Im{(\ol{P_{\xi(t), \leq l_{2}}u_{l}}) P_{\xi(t), \leq l_{2}}\left[\E\paren*{\Re{u_{l}\bar{u}_{h}}}u_{h}\right]} + 2\Im{(\ol{P_{\xi(t), \leq l_{2}}u_{h}}) P_{\xi(t), \leq l_{2}}\left[\E\paren*{\Re{u_{l}\bar{u}_{h}}}u_{l}\right]} \nonumber\\
&\phantom{=}+ \Im{(\ol{P_{\xi(t), \leq l_{2}}u_{l}}) P_{\xi(t), \leq l_{2}}\left[\E(|u_{h}|^{2})u_{l}\right] + (\ol{P_{\xi(t), \leq l_{2}}u_{h}}) P_{\xi(t), \leq l_{2}}\left[\E(|u_{l}|^{2})u_{h}\right]} \nonumber\\
&\phantom{=} -2\Im{\bar{u}_{l}\E\paren*{\Re{u_{l}(\ol{P_{\xi(t), \leq l_{2}}u_{h}})}}(P_{\xi(t), \leq l_{2}}u_{h})} -2\Im{(\ol{P_{\xi(t), \leq l_{2}}u_{h}})\E\paren*{\Re{u_{l}(\ol{P_{\xi(t), \leq l_{2}}u_{h}})}}u_{l}}\nonumber\\
&\phantom{=} -\Im{|u_{l}|^{2}\E(|P_{\xi(t), \leq l_{2}}u_{h}|^{2}) + |P_{\xi(t), \leq l_{2}}u_{h}|^{2}\E(|u_{l}|^{2})} \label{eq:ibs2_F2}\\
F_{3} &\coloneqq \Im{\bar{u}_{l}P_{\xi(t),\leq l_{2}}\left[\E(|u_{h}|^{2})u_{h}\right] + (\ol{P_{\xi(t), \leq l_{2}}u_{h}})P_{\xi(t),\leq l_{2}}\left[\E(|u_{h}|^{2})u_{l}\right]} \nonumber\\
&\phantom{=}+2\Im{(\ol{P_{\xi(t),\leq l_{2}}u_{h}})P_{\xi(t),\leq l_{2}}\brak*{\E\paren*{\Re{u_{h}\bar{u}_{l}}}u_{h}}} \label{eq:ibs2_F3}\\
F_{4} &\coloneqq \Im{(\ol{P_{\xi(t), \leq l_{2}}u_{h}}) P_{\xi(t), \leq l_{2}}\brak*{\E(|u_{h}|^{2})u_{h}} - |P_{\xi(t), \leq l_{2}}u_{h}|^{2}\E(|P_{\xi(t), \leq l_{2}}u_{h}|^{2})}. \label{eq:ibs2_F4}
\end{align}

We can quickly dispense with the contribution of $F_{0}$ to \eqref{eq:ibs2_err3}, as Fourier support analysis shows that $F_{0}=0$, hence there is no contribution.

To estimate the contributions of $F_{3}$ and $F_{4}$, we put one of the factors ($u_{l}$ in the case of $F_{3}$) in $L_{t}^{\infty}L_{x}^{2}$ and use H\"{o}lder's inequality, Calder\'{o}n-Zygmund theorem, Strichartz estimates, and mass conservation to obtain
\begin{align}
2^{k-2i}\left|\int_{G_{\beta}^{l_{2}}}\int_{\R^{4}} \frac{x-y}{|x-y|}\cdot (F_{3}+F_{4})(t,y)\Im{\bar{v}(\nabla-i\xi(t))v}(t,x)dxdydt\right| &\lesssim 2^{k-i}\paren*{\|F_{3}\|_{L_{t,x}^{1}}+\|F_{4}\|_{L_{t,x}^{1}}}\|v_{0}\|_{L^{2}}^{2} \nonumber\\
&\lesssim 2^{k-i}\paren*{\|u_{h}\|_{L_{t}^{\infty}L_{x}^{2}}+\|u_{l}\|_{L_{t}^{\infty}L_{x}^{2}}}\|u_{h}\|_{L_{t,x}^{4}}^{3} \|v_{0}\|_{L^{2}}^{2} \nonumber\\
&\lesssim 2^{k-i}\|v_{0}\|_{L^{2}}^{2}\|u\|_{\tilde{X}_{l_{2}}(G_{\beta}^{l_{2}}\times\R^{2})}^{3},
\end{align}
where we use lemma \ref{lem:lohi_embed} to obtain the ultimate inequality.

We next estimate the contribution of $F_{1}$ to \eqref{eq:ibs2_err3}. We claim that $F_{1}$ has Fourier support in the region $\{|\xi|\geq 2^{l_{2}-4}\}$. To see the claim, first observe that the second line in \eqref{eq:ibs2_F1} is zero since we are taking the imaginary part of a real expression. Since $\E(|u_{l}|^{2})u_{l}$ has Fourier support in the dyadic ball $B(\xi(t), l_{2}-2)$, we see that
\begin{equation}
F_{1} = \Im{\bar{u}_{l}P_{\xi(t), \leq l_{2}}\left[\E(|u_{l}|^{2})u_{h}\right] + 2\bar{u}_{l}P_{\xi(t), \leq l_{2}}\left[\E\paren*{\Re{u_{l}\bar{u}_{h}}}u_{l}\right] + (\ol{P_{\xi(t), \leq l_{2}}u_{h}})\E(|u_{l}|^{2})u_{l}}.
\end{equation}
Next, split $u_{h}=P_{\xi(t),l_{2}-5<\cdot\leq l_{2}-2}u+P_{\xi(t),>l_{2}-2}u$, substitute into $F_{1}$, and then expand. Fourier support analysis shows that
\begin{align}
2\Im{\bar{u_{l}}P_{\xi(t),\leq l_{2}}\brak*{\E\paren*{\Re{u_{l}(\ol{P_{\xi(t),l_{2}-5<\cdot\leq l_{2}-2}u})}}u_{l}}} &= 2\Im{|u_{l}|^{2}\E\paren*{\Re{\bar{u}_{l}(P_{\xi(t),l_{2}-5<\cdot\leq l_{2}-2}u)}}} =0 \\
\Im{\bar{u_{l}}P_{\xi(t),\leq l_{2}}\brak*{\E(|u_{l}|^{2})(P_{\xi(t),l_{2}-5<\cdot\leq l_{2}-2}u)}} &= \Im{\bar{u_{l}}(P_{\xi(t),l_{2}-5<\cdot\leq l_{2}-2}u)\E(|u_{l}|^{2})}, \\
\Im{(\ol{P_{\xi(t),\leq l_{2}}P_{\xi(t),l_{2}-5<\cdot\leq l_{2}-2}u})\E(|u_{l}|^{2})u_{l}} &= \Im{(\ol{P_{\xi(t),l_{2}-5<\cdot\leq l_{2}-2}u})u_{l}\E(|u_{l}|^{2})},
\end{align}
and therefore
\begin{equation}
\Im{\bar{u_{l}}P_{\xi(t),\leq l_{2}}\brak*{\E(|u_{l}|^{2})(P_{\xi(t),l_{2}-5<\cdot\leq l_{2}-2}u)}} + \Im{(\ol{P_{\xi(t),\leq l_{2}}P_{\xi(t),l_{2}-5<\cdot\leq l_{2}-2}u})\E(|u_{l}|^{2})u_{l}}=0.
\end{equation}
Hence,
\begin{equation}
\begin{split}
F_{1} &= \Im{\bar{u_{l}}P_{\xi(t),\leq l_{2}}\brak*{\E(|u_{l}|^{2})(P_{\xi(t),>l_{2}-2}u)} + 2\bar{u_{l}}P_{\xi(t),\leq l_{2}}\brak*{\E\paren*{\Re{u_{l}(\ol{P_{\xi(t),>l_{2}-2}u})}}u_{l}}} \\
&\phantom{=} - \Im{\bar{u_{l}}\E(|u_{l}|^{2})(P_{\xi(t),l_{2}-2<\cdot\leq l_{2}}u)},
\end{split}
\end{equation}
which has Fourier support in the region $\{|\xi|\geq 2^{l_{2}-4}\}$ as claimed. We can now make an observation similar to \eqref{eq:ibs1_fs_obs} and proceed by the accompanying argument to obtain that
\begin{align}
&2^{k-2i}\left|\int_{G_{\beta}^{l_{2}}}\int_{\R^{4}}\frac{x-y}{|x-y|}\cdot F_{1}(t,y)\Im[\bar{v}(\nabla-i\xi(t))v](t,x)dxdydt\right|\nonumber\\
&\lesssim 2^{k-2l_{2}-i}\|v_{0}\|_{L^{2}}^{2}\|F_{1}\|_{L_{t,x}^{2}} \nonumber\\
&\lesssim 2^{k-2l_{2}-i}\|v_{0}\|_{L^{2}}^{2} \left(\|\E(|u_{l}|^{2})\|_{L_{t}^{\infty}L_{x}^{2}}\|u_{l}P_{\xi(t), l_{2}-2<\cdot\leq l_{2}+2}u_{h}\|_{L_{t}^{2}L_{x}^{\infty}}+\|u_{l}\|_{L_{t}^{4}L_{x}^{\infty}}^{2}\|\E(\bar{u}_{l}P_{\xi(t),l_{2}-2<\cdot\leq l_{2}+2}u_{h})\|_{L_{t}^{\infty}L_{x}^{2}} \right)\nonumber\\
&\lesssim 2^{k-2l_{2}-i}\|v_{0}\|_{L^{2}}^{2} \left(2^{l_{2}} \|u_{l}\|_{L_{t}^{4}L_{x}^{\infty}}\|P_{\xi(t),\leq l_{2}+3}u_{h}\|_{L_{t}^{4}L_{x}^{\infty}} + 2^{2l_{2}}\|u\|_{\tilde{X}_{l_{2}}(G_{\beta}^{l_{2}}\times\R^{2})}^{2} \right)\nonumber\\
&\lesssim 2^{k-i}\|v_{0}\|_{L^{2}}^{2}\|u\|_{\tilde{X}_{l_{2}}(G_{\beta}^{l_{2}}\times\R^{2})}^{2}, \label{eq:ibs2_F1_est}
\end{align}
where we use Bernstein's lemma and lemmas \ref{lem:BT_embed} and \ref{lem:lohi_embed} to obtain the last two inequalities.

Lastly, we consider the contribution of $F_{2}$ to \eqref{eq:ibs2_err3}, which is the most involved step. We first observe that the terms with a minus sign in front in the definition of $F_{2}$ vanish identically (we end up having the imaginary part of a real expression), hence
\begin{equation} %Refined identity for F2
\begin{split}
F_{2} &= 2\Im{(\ol{P_{\xi(t), \leq l_{2}}u_{l}}) P_{\xi(t), \leq l_{2}}\left[\E\paren*{\Re{u_{l}\bar{u}_{h}}}u_{h}\right]} + 2\Im{(\ol{P_{\xi(t), \leq l_{2}}u_{h}}) P_{\xi(t), \leq l_{2}}\left[\E\paren*{\Re{u_{l}\bar{u}_{h}}}u_{l}\right]} \\
&\phantom{=}+ \Im{(\ol{P_{\xi(t), \leq l_{2}}u_{l}}) P_{\xi(t), \leq l_{2}}\left[\E(|u_{h}|^{2})u_{l}\right] + (\ol{P_{\xi(t), \leq l_{2}}u_{h}}) P_{\xi(t), \leq l_{2}}\left[\E(|u_{l}|^{2})u_{h}\right]}.
\end{split}
\end{equation}
Next, we reduce to the case where $F_{2}$ is Fourier supported on low frequencies $|\xi|\lesssim 2^{l_{2}}$. This additional frequency localization will later give that the two factors of $u_{h}$ with higher frequency (relative to $\xi(t)$) must be Fourier supported at frequencies of comparable distance from $\xi(t)$, which will be helpful in proving certain multilinear Fourier multipliers are of Coifman-Meyer type. Turning to the details, we repeat the argument in the estimate for the contribution for $F_{1}$ to obtain
\begin{align}
2^{k-2i}\left|\int_{G_{\beta}^{l_{2}}}\int_{\R^{4}}\frac{x-y}{|x-y|}\cdot (P_{>l_{2}-10}F_{2})(t,y) \Im{\bar{v}(\nabla-i\xi(t))v}(t,x)dxdydt\right| &\lesssim 2^{k-l_{2}-i} \|v_{0}\|_{L^{2}}^{2} \|F_{2}\|_{L_{t}^{3/2}L_{x}^{6/5}} \nonumber\\
&\lesssim 2^{k-l_{2}-i}\|v_{0}\|_{L^{2}}^{2} \|u_{h}\|_{L_{t}^{3}L_{x}^{6}}^{2} \|u_{l}\|_{L_{t}^{\infty}L_{x}^{4}}^{2} \nonumber\\
&\lesssim 2^{k-i} \|v_{0}\|_{L^{2}}^{2} \|u\|_{\tilde{X}_{l_{2}}(G_{\beta}^{l_{2}}\times\R^{2})}^{2}.
\end{align}

We now estimate the terms comprising $P_{\leq l_{2}-10}F_{2}$ in two groups, which, as the reader will see, correspond to when the Fourier support intersects the time resonance set $\mathcal{T}$ (see definition \eqref{eq:tr}). We write $P_{\leq l_{2}-10}F_{2} =\mathrm{Group}_{1}+\mathrm{Group}_{2}$, where
\begin{align}
\mathrm{Group}_{1} &\coloneqq P_{\leq l_{2}-10}\paren*{\Im{(\ol{P_{\xi(t),\leq l_{2}}u_{l}})P_{\xi(t),\leq l_{2}}\brak*{\E\paren*{u_{l}\bar{u_{h}}}u_{h}} + (\ol{P_{\xi(t),\leq l_{2}}u_{h}})P_{\xi(t),\leq l_{2}}\brak*{\E\paren*{\bar{u_{l}}u_{h}}u_{l}}}} \nonumber\\
&\phantom{=}+P_{\leq l_{2}-10}\paren*{\Im{(\ol{P_{\xi(t),\leq l_{2}}u_{l}}) P_{\xi(t),\leq l_{2}}\brak*{\E(|u_{h}|^{2})u_{l}} + (\ol{P_{\xi(t),\leq l_{2}}u_{h}}) P_{\xi(t),\leq l_{2}}\brak*{\E(|u_{l}|^{2})u_{h}}}} \nonumber\\
&\eqqcolon \mathrm{Group}_{1,1} + \mathrm{Group}_{1,2}\\
\mathrm{Group}_{2} &\coloneqq P_{\leq l_{2}-10}\paren*{\Im{(\ol{P_{\xi(t),\leq l_{2}}u_{l}}) P_{\xi(t),\leq l_{2}}\brak*{\E\paren*{\bar{u_{l}}u_{h}}u_{h}} + (\ol{P_{\xi(t),\leq l_{2}}u_{h}})P_{\xi(t),\leq l_{2}}\brak*{\E\paren*{u_{l}\bar{u_{h}}}u_{l}}}} \nonumber\\
&=P_{\leq l_{2}-10}\paren*{\Im{(\ol{P_{\xi(t),\leq l_{2}}u_{l}}) P_{\xi(t),\leq l_{2}}\brak*{\E\paren*{\bar{u_{l}}u_{h}}u_{h}} -(P_{\xi(t),\leq l_{2}}u_{h}) P_{-\xi(t),\leq l_{2}}\brak*{\E\paren*{\bar{u_{l}}u_{h})\bar{u}_{l}}}}}.
\end{align}

We first estimate the contribution of $\mathrm{Group}_{1}$ to \eqref{eq:ibs2_err3}. Adding zero, we can write
\begin{align}
\mathrm{Group}_{1,2} &= \Im{\bar{u}_{l}P_{\xi(t),\leq l_{2}}\brak*{\E(|u_{h}|^{2})u_{l}}+(\ol{P_{\xi(t),\leq l_{2}}u_{h}})P_{\xi(t),\leq l_{2}}\brak*{\E(|u_{l}|^{2})u_{h}}} \nonumber\\
&=\Im{\bar{u}_{l}\comm{P_{\xi(t),\leq l_{2}}}{u_{l}}\E(|u_{h}|^{2})}+\Im{(\ol{P_{\xi(t),\leq l_{2}}u_{h}}) \comm{P_{\xi(t),\leq l_{2}}}{\E(|u_{l}|^{2})}u_{h}}.
\end{align}
By the fundamental theorem of calculus and Minkowski's inequality,
\begin{align}
\|\bar{u}_{l}\comm{P_{\xi(t),\leq l_{2}}}{u_{l}}\E(|u_{h}|^{2})\|_{L_{t,x}^{1}} + \|(\ol{P_{\xi(t),\leq l_{2}}u_{h}}) \comm{P_{\xi(t),\leq l_{2}}}{\E(|u_{l}|^{2})}u_{h}\|_{L_{t,x}^{1}} &\lesssim 2^{-l_{2}}\|u\|_{L_{t}^{\infty}L_{x}^{2}} \|(\nabla-i\xi(t))u_{l}\|_{L_{t}^{3}L_{x}^{6}}\|u_{h}\|_{L_{t}^{3}L_{x}^{6}}^{2}\nonumber\\
&\lesssim\|u\|_{\tilde{X}_{l_{2}}(G_{\beta}^{l_{2}}\times\R^{2})}^{3},
\end{align}
where we also use H\"{o}lder's inequality, Calder\'{o}n-Zygmund theorem, Strichartz estimates to obtain the penultimate inequality and mass conservation and lemmas \ref{lem:lohi_embed} and \ref{lem:BT_embed} to obtain the ultimate inequality. Therefore
\begin{equation}
2^{k-2i}\left|\int_{G_{\beta}^{l_{2}}}\int_{\R^{4}}\frac{(x-y)}{|x-y|}\cdot \Im{\bar{v}(\nabla-i\xi(t))v}(t,x) \mathrm{Group}_{1,2}(t,y)dxdydt\right| \lesssim 2^{k-i}\|v_{0}\|_{L^{2}}^{2}\|u\|_{\tilde{X}_{l_{2}}(G_{\beta}^{l_{2}}\times\R^{2})}^{3}.
\end{equation}
 
If $\E$ were some constant multiple of the identity operator, then we could just repeat the same argument to estimate $\|\mathrm{Group}_{1,1}\|_{L_{t,x}^{1}}$, but the presence of a nonlocal operator $\E$ presents an obstacle to that line of reasoning. Instead, we argue as follows. For simplicity, we only consider the first term in $\mathrm{Group}_{1,1}$, as the second term may be similarly estimated. Using the estimates obtained for the contributions of $F_{3}$ and $F_{4}$ above, it suffices to estimate the contribution to \eqref{eq:ibs2_err3} of the quantity
\begin{equation}
\begin{split}
&P_{\leq l_{2}-10}\paren*{\Im{(\ol{P_{\xi(t), \leq l_{2}-10}u_{l}}) P_{\xi(t), \leq l_{2}}\brak*{\E\paren*{(P_{\xi(t), \leq l_{2}-10}u_{l})\bar{u}_{h}}u_{h}}}} = P_{\leq l_{2}-10}\paren*{\Im{(\ol{P_{\xi(t), \leq l_{2}-10}u_{l}})\E\paren*{(P_{\xi(t), \leq l_{2}-10}u_{l})\bar{u}_{h}}u_{h}}}.
\end{split}
\end{equation}
Note that $(P_{\xi(t),\leq l_{2}-10}u_{l})\bar{u}_{h}$ has Fourier support in the region $\{|\xi|\geq 2^{l_{2}-7}\}$. Now using the operator identity $Id=\frac{\p_{1}^{2}}{\Delta} + \frac{\p_{2}^{2}}{\Delta}$, we observe the cancellation
\begin{align}
P_{\leq l_{2}-10}\paren*{\Im{(\ol{P_{\xi(t),\leq l_{2}-10}u})\Rxx\paren*{(P_{\xi(t),\leq l_{2}-10}u)\bar{u}_{h}}u_{h}}} &= P_{\leq l_{2}-10}\paren*{\Im{\frac{\p_{1}^{2}}{\Delta}\paren*{(\ol{P_{\xi(t),\leq l_{2}-10}u})u_{h}}\frac{\p_{1}^{2}}{\Delta}\paren*{(P_{\xi(t),\leq l_{2}-10}u)\bar{u}_{h}}}} \nonumber\\
&\phantom{=}+P_{\leq l_{2}-10}\paren*{\Im{\frac{\p_{2}^{2}}{\Delta}\paren*{(\ol{P_{\xi(t),\leq l_{2}-10}u})u_{h}}\frac{\p_{1}^{2}}{\Delta}\paren*{(P_{\xi(t),\leq l_{2}-10}u)\bar{u}_{h}}}} \nonumber\\
&= P_{\leq l_{2}-10}\paren*{\Im{\frac{\p_{2}^{2}}{\Delta}\paren*{(\ol{P_{\xi(t),\leq l_{2}-10}u})u_{h}}\frac{\p_{1}^{2}}{\Delta}\paren*{(P_{\xi(t),\leq l_{2}-10}u)\bar{u}_{h}}}}.
\end{align}
Now integrating by parts in $y_{1}$ and then integrating by parts in $y_{2}$, we see that
\begin{align}
&\int_{G_{\beta}^{l_{2}}}\int_{\R^{4}}\frac{(x-y)}{|x-y|}\cdot P_{\leq l_{2}-10}\paren*{\Im{\frac{\p_{2}^{2}}{\Delta}\paren*{(\ol{P_{\xi(t),\leq l_{2}-10}u})u_{h}}\frac{\p_{1}^{2}}{\Delta}\paren*{(P_{\xi(t),\leq l_{2}-10}u)\bar{u}_{h}}}}(t,y) \Im{\bar{v}(\nabla-i\xi(t))v}(t,x) dxdydt \nonumber\\
&=-\int_{G_{\beta}^{l_{2}}}\int_{\R^{4}} \frac{(x-y)}{|x-y|}\cdot P_{\leq l_{2}-10}\paren*{\Im{\frac{\p_{1}\p_{2}^{2}}{\Delta}\paren*{(\ol{P_{\xi(t),\leq l_{2}-10}u})u_{h}}\frac{\p_{1}}{\Delta}\paren*{(P_{\xi(t),\leq l_{2}-10}u)\bar{u}_{h}}}}(t,y)\Im{\bar{v}(\nabla-i\xi(t))v}(t,x) dxdydt \nonumber\\
&\phantom{=}+\int_{G_{\beta}^{l_{2}}}\int_{\R^{4}} \left(\frac{1}{|x-y|}-\frac{(x-y)_{1}}{|x-y|^{3}}\right)\cdot P_{\leq l_{2}-10}\paren*{\Im{\frac{\p_{2}^{2}}{\Delta}\paren*{(\ol{P_{\xi(t),\leq l_{2}-10}u})u_{h}}\frac{\p_{1}}{\Delta}\paren*{(P_{\xi(t),\leq l_{2}-10}u)\bar{u}_{h}}}}(t,y) \nonumber\\
&\phantom{=}\qquad \Im{\bar{v}(\nabla-i\xi(t))v}(t,x) dxdydt\nonumber\\
&=-\int_{G_{\beta}^{l_{2}}}\int_{\R^{4}} \left(\frac{1}{|x-y|}-\frac{(x-y)_{2}}{|x-y|^{3}}\right)\cdot P_{\leq l_{2}-10}\paren*{\Im{\frac{\p_{1}\p_{2}}{\Delta}\paren*{(\ol{P_{\xi(t),\leq l_{2}-10}u})u_{h}}\frac{\p_{1}}{\Delta}\paren*{(P_{\xi(t),\leq l_{2}-10}u)\bar{u}_{h}}}}\nonumber\\
&\phantom{=} \qquad\Im{\bar{v}(\nabla-i\xi(t))v}(t,x)dxdydt\nonumber\\
&\phantom{=}+\int_{G_{\beta}^{l_{2}}}\int_{\R^{4}} \left(\frac{1}{|x-y|}-\frac{(x-y)_{1}}{|x-y|^{3}}\right)\cdot P_{\leq l_{2}-10}\paren*{\Im{\frac{\p_{2}^{2}}{\Delta}\paren*{(\ol{P_{\xi(t),\leq l_{2}-10}u})u_{h}}\frac{\p_{1}}{\Delta}\paren*{(P_{\xi(t),\leq l_{2}-10}u)\bar{u}_{h}}}}(t,y) \nonumber\\
&\phantom{=}\qquad \Im{\bar{v}(\nabla-i\xi(t))v}(t,x) dxdydt. \label{eq:ibs2_IBP_trick}
\end{align}
By H\"{o}lder's inequality, Hardy-Littlewood-Sobolev lemma, Strichartz estimates, followed by Bernstein's lemma, we have that
\begin{align}
2^{k-2i}|\eqref{eq:ibs2_IBP_trick}| &\sum_{\mu,\nu=1}^{2}\lesssim 2^{k-i}\| |\nabla|^{-1}(|(P_{\xi(t),\leq l_{2}-10}u)\bar{u}_{h}|) \frac{\partial_{\mu\nu}}{\Delta}(\ol{(P_{\xi(t),\leq l_{2}-10}u)}u_{h})\|_{L_{t}^{3/2}L_{x}^{6/5}} \|v_{0}\|_{L^{2}}^{2}\nonumber\\
&\lesssim 2^{k-l_{2}-i}\|v_{0}\|_{L^{2}}^{2} \|u_{h}\|_{L_{t}^{3}L_{x}^{6}}^{2}\|u_{l}\|_{L_{t}^{\infty}L_{x}^{4}}^{2} \nonumber\\
&\lesssim 2^{k-i} \|v_{0}\|_{L^{2}}^{2} \|u\|_{\tilde{X}_{l_{2}}(G_{\beta}^{l_{2}}\times\R^{2})}^{2},
\end{align}
where we also use the Calder\'{o}n-Zygmund theorem and Bernstein's lemma to obtain the penultimate inequality, and Bernstein's lemma, mass conservation, and lemmas \ref{lem:lohi_embed} and \ref{lem:BT_embed} to obtain the ultimate inequality. This estimate completes the analysis for the contribution of $\mathrm{Group}_{1}$.

We now turn to estimating the terms in $\mathrm{Group}_{2}$, which will occupy our attention for the remainder of the estimate for \eqref{eq:ibs2_err3}. By the standard commutator argument used before, we have that
\begin{equation}
\|P_{\xi(t),\leq l_{2}}\brak*{\E\paren*{u_{l}\bar{u}_{h}}u_{l}}-\E\paren*{u_{l}(\ol{P_{\xi(t),\leq l_{2}}u_{h}})}u_{l}\|_{L_{t}^{3/2}L_{x}^{6/5}} \lesssim \|u\|_{\tilde{X}_{l_{2}}(G_{\beta}^{l_{2}}\times\R^{2})}^{2},
\end{equation}
and therefore
\begin{align}
&2^{k-2i}\left|\int_{G_{\beta}^{l_{2}}}\int_{\R^{4}}\frac{(x-y)}{|x-y|}\cdot P_{\leq l_{2}-10}\paren*{\Im{(\ol{P_{\xi(t),\leq l_{2}}u_{h}})\paren*{P_{\xi(t),\leq l_{2}}\brak*{\E\paren*{u_{l}\bar{u}_{h}}u_{l}}-\E\paren*{u_{l}(\ol{P_{\xi(t),\leq l_{2}}u_{h}})}u_{l}}}}(t,y)\right. \nonumber\\
&\phantom{=}\qquad \left. \Im{\bar{v}(\nabla-i\xi(t))v}(t,x) dxdydt\right| \nonumber\\
&\lesssim 2^{k-i} \|v_{0}\|_{L^{2}}^{2} \|u\|_{\tilde{X}_{l_{2}}(G_{\beta}^{l_{2}}\times\R^{2})}^{2}.
\end{align}
Since 
\begin{equation}
P_{\leq l_{2}-10}\paren*{\Im{(\ol{P_{\xi(t),\leq l_{2}}u_{l}})P_{\xi(t),\leq l_{2}}\brak*{\E\paren*{\bar{u}_{l}u_{h}}u_{h}}}} = P_{\leq l_{2}-10}\paren*{\Im{\bar{u}_{l}u_{h}\E\paren*{\bar{u}_{l}u_{h}}}},
\end{equation}
we see that it suffices to estimate
\begin{equation}\label{eq:ibs2_term_NFT}
\begin{split}
&2^{k-2i}\left|\int_{G_{\beta}^{l_{2}}}\int_{\R^{4}}\frac{(x-y)}{|x-y|}\cdot P_{\leq l_{2}-10}\paren*{\Im{\E\paren*{\bar{u}_{l}u_{h}}\bar{u}_{l}u_{h}-\E\paren*{\bar{u}_{l}(P_{\xi(t),\leq l_{2}}u_{h})}\bar{u}_{l}(P_{\xi(t),\leq l_{2}}u_{h})}}(t,y) \right.\\
&\phantom{=} \left.\qquad\Im{\bar{v}(\nabla-i\xi(t))v}(t,x) dxdydt\right|.
\end{split}
\end{equation}
Furthermore, by using the estimates for the contributions of $F_{3}$ and $F_{4}$, we may assume that $\bar{u}_{l}=\ol{P_{\xi(t),\leq l_{2}-15}u}$. As a final simplification, we note that the preceding quantity is Galilean invariant and therefore we may assume without loss of generality that $\xi(G_{\beta}^{l_{2}})=0$. Thus,
\begin{equation}
\sup_{t\in G_{\beta}^{l_{2}}}|\xi(t)| \leq \epsilon_{1}^{-1/2}2^{-20}\int_{G_{\beta}^{l_{2}}}N(t)^{3}dt \leq 2^{l_{2}-19}\epsilon_{3}^{1/2} \ll 2^{l_{2}}.
\end{equation}

To estimate \eqref{eq:ibs2_term_NFT}, we use an idea of \cite{Dodson2016}, which is to use the method of space-time resonances, specifically the time resonances, introduced by Germain, Masmoudi, and Shatah (see \cite{germain2010space} for a brief survey). The idea is to write $u_{l}, u_{h}$ in terms of their profiles $w_{l} \coloneqq e^{-it\Delta}u_{l}$ and $w_{h} \coloneqq e^{-it\Delta}u_{h}$, respectively. Using Fourier inversion, we then write
\begin{align}
&P_{\leq l_{2}-10}\paren*{\brak*{\E\paren*{\bar{u}_{l}u_{h}}\bar{u}_{l}u_{h}-\E\paren*{\bar{u}_{l}(P_{\xi(t),\leq l_{2}}u_{h})}\bar{u}_{l}(P_{\xi(t),\leq l_{2}}u_{h})}}(t,y) \nonumber\\
&= \frac{1}{(2\pi)^{8}i}\int_{\R^{8}}e^{iy\cdot\eta}\frac{d}{dt}\paren*{e^{it\Phi(\ueta_{4})}}m_{1}(t,\ueta_{4})\hat{w}_{h}(t,\eta_{1})\hat{w}_{h}(t,\eta_{2})\hat{\bar{w}}_{l}(t,\eta_{3})\hat{\bar{w}}_{l}(t,\eta_{4}) d\ueta_{4},
\end{align}
where
\begin{equation}
m(t,\ueta_{4}) \coloneqq \frac{1}{\Phi(\ueta_{4})}\phi\paren*{\frac{\eta}{2^{l_{2}-10}}}\paren*{1-\phi\paren*{\frac{\eta_{1}-\xi(t)}{2^{l_{2}}}}\phi\paren*{\frac{\eta_{2}-\xi(t)}{2^{l_{2}}}}}\frac{(\eta_{2}+\eta_{3})_{1}^{2}}{|\eta_{2}+\eta_{3}|^{2}}, \qquad \ueta_{4}=(\eta_{1},\eta_{2},\eta_{3},\eta_{4})\in\R^{8},
\end{equation}
and we have used the notation $\eta\coloneqq\eta_{1}+\cdots+\eta_{4}$ and
\begin{equation}
\Phi(\ueta_{4}) \coloneqq |\eta_{1}|^{2}+|\eta_{2}|^{2}-|\eta_{3}|^{2}-|\eta_{4}|^{2}, \qquad \ueta_{4} \in\R^{8}
\end{equation}
and used the identity
\begin{equation}
e^{it\Phi(\ueta_{4})} = \frac{1}{i\Phi(\ueta_{4})} \frac{d}{dt}\left(e^{it\Phi(\ueta_{4})}\right), \qquad \Phi(\ueta_{4})\neq 0.
\end{equation}
Provided that the phase function $\Phi(\ueta_{4})$ is nondegenerate, we can make a normal form transformation (i.e. integrate by parts in time),  replacing the nonlinearity by a higher order expression. The normal form transformation is smoothing by two derivaties (i.e. we gain a factor of $2^{-2l_{2}}$), and we can put all of the factors in $L_{t}^{\infty}L_{x}^{2}(G_{\beta}^{l_{2}}\times\R^{2})$ or $\tilde{X}_{l_{2}}(G_{\beta}^{l_{2}}\times\R^{2})$ at a cost of two derivatives (i.e. we gain a factor of $2^{2l_{2}}$), which is precisely the amount we can afford to lose. The obstruction to the normal form transformation is that the Fourier support of the nonlinearity may intersect the set of \emph{time resonances}
\begin{equation}\label{eq:tr}
\mathcal{T} \coloneqq \{\ueta_{4}\in\R^{8} : \Phi(\ueta_{4})=0\}.
\end{equation}
However, as we shall see below, thanks to our simplifications in the preceding steps, our nonlinearity is localized away from $\mathcal{T}$.

Substituting the above inverse Fourier transform into \eqref{eq:ibs2_term_NFT}, we need to estimate the modulus of the expression
\begin{equation}
\begin{split}
&\frac{1}{(2\pi)^{8}i}\int_{G_{\beta}^{l_{2}}}dt\int_{\R^{4}}dxdy\int_{\R^{8}}d\ueta_{4} \frac{(x-y)}{|x-y|}\cdot\Im{\bar{v}(\nabla-i\xi(t))}(t,x) \frac{d}{dt}\paren*{e^{it\Phi(\ueta_{4})}} e^{iy\cdot\eta} m(t,\ueta_{4}) \hat{w}_{h}(t,\eta_{1})\hat{w}_{h}(t,\eta_{2})\hat{\bar{w}}_{l}(t,\eta_{3})\hat{\bar{w}}_{l}(t,\eta_{4}).
\end{split}
\end{equation}
Before proceeding to the temporal integration by parts computation, we further analyze the symbol $m(t,\ueta_{4})$. Observe that if $|\eta_{1}|\geq 8|\eta_{2}|$, then
\begin{equation}
1-\phi\left(\frac{\eta_{1}-\xi(t)}{2^{l_{2}}}\right)\phi\left(\frac{\eta_{2}-\xi(t)}{2^{l_{2}}}\right)\neq 0 \Longrightarrow |\eta_{1}-\xi(t)| \geq 2^{l_{2}}.
\end{equation}
Therefore by the reverse triangle inequality,
\begin{equation}
|\eta| = |\eta_{1}+\eta_{2}+\eta_{3}+\eta_{4}| \geq \frac{|\eta_{1}|}{2} \geq 2^{l_{2}-1} \Longrightarrow \phi\left(\frac{\eta}{2^{l_{2}-10}}\right)=0
\end{equation}
on the support of $m$. By symmetry, we have the same conclusion if $|\eta_{2}| \geq 8|\eta_{1}|$. Hence, $2^{-3} < \frac{|\eta_{1}|}{|\eta_{2}|} < 2^{3}$ on the support of the symbol $m$. Furthermore,
\begin{equation}
|\eta_{3}| \leq 2^{-5}|\eta_{1}| \enspace \text{and} \enspace |\eta_{4}| \leq 2^{-5}|\eta_{2}|, \qquad \forall \ueta_{4}\in \supp(m(t))\cap \supp\paren*{\hat{w}_{h}(t)\otimes \hat{w}_{h}(t)\otimes \hat{\bar{w}}_{l}(t)\otimes\hat{\bar{w}}_{l}(t)}, \enspace \forall t\in G_{\beta}^{l_{2}}.
\end{equation}
Thus, for $\ueta_{4}\in \supp(m(t))\cap \supp\paren*{\hat{w}_{h}(t)\otimes\hat{w}_{h}(t)\otimes\hat{\bar{w}}_{l}(t)\otimes\hat{\bar{w}}_{l}(t)}$ and $t\in G_{\beta}^{l_{2}}$, we may write
\begin{align}
m(t,\ueta_{4}) = m(t,\ueta_{4}) \phi\left(\frac{\eta_{3}}{2^{l_{2}-10}}\right)\phi\left(\frac{\eta_{4}}{2^{l_{2}-10}}\right)\vartheta\left(2^{5}\frac{|\eta_{3}|}{|\eta_{1}|}\right)\vartheta\left(2^{5}\frac{|\eta_{4}|}{|\eta_{2}|}\right)\varphi\left(\frac{|\eta_{1}|}{|\eta_{2}|}\right) \eqqcolon m_{1}(t,\ueta_{4}),
\end{align}
where $0\leq \vartheta, \varphi\leq 1$ are even bump functions respectively satisfying
\begin{equation}
\vartheta(r)=
\begin{cases}
1, & {|r|\leq 1} \\
0, & {|r|>\frac{3}{2}}
\end{cases}
\end{equation}
and
\begin{equation}
\varphi(r)=
\begin{cases}
1,& {2^{-3}\leq |r| \leq 2^{3}}\\ 
0, & {|r| > 2^{4}, |r| < 2^{-4}} 
\end{cases}.
\end{equation}
We now claim that $m_{1}(t)$ is a Coifman-Meyer multiplier with operator norm $\lesssim 2^{-2l_{2}}$ uniformly in $t\in G_{\beta}^{l_{2}}\subset [0,T]$. Indeed, the proof of this claim is the content of the next lemma.

\begin{lemma}[First C-M estimate]\label{lem:CM_1} %First Coifman-Meyer multiplier lemma
For any real numbers $1<p_{1},\ldots,p_{4}\leq \infty$ and $0<p<\infty$ satisfying $\frac{1}{p_{1}}+\cdots+\frac{1}{p_{4}}=\frac{1}{p}$, there exists a constant $C_{p_{1},\ldots,p_{4}}$ such that the multilinear operator $m_{1}(t,\uD_{4}): \prod_{l=1}^{4}\mathcal{S}(\R^{2})\rightarrow\mathcal{S}'(\R^{2})$ with symbol $m_{1}(t)$ extends to a bounded operator
\begin{equation}
m_{1}(t,\uD_{4}):\prod_{l=1}^{4} L^{p_{l}}(\R^{2}) \rightarrow L^{p}(\R^{2}), \qquad \sup_{t\in G_{\beta}^{l_{2}}} \|m_{1}(t,\uD_{4})\|_{L^{p_{1}}\times\cdots\times L^{p_{4}}\rightarrow L^{p}} \leq C_{p_{1}\cdots p_{4}}2^{-2l_{2}}
\end{equation}
for all $20\leq i\leq j\leq k_{0}$, $0\leq l_{2}\leq i-10$, and $G_{\beta}^{l_{2}}\subset [0,T]$.
\end{lemma}

\begin{proof}
To prove the lemma, we use the Coifman-Meyer theorem. Throughout the proof, we suppress the dependence on time $t$ as all of our estimates will be uniform in $t\in G_{\beta}^{l_{2}}$. First, it is evident that the symbol $m_{1}$ is $C^{\infty}$ outside the origin. We next rescale $m_{1}$ in the spatial variable by defining
\begin{equation}
m_{1,l_{2}}(\ueta_{4}) \coloneqq 2^{2l_{2}}m_{1}(2^{l_{2}}\ueta_{4}).
\end{equation}
By scaling invariance, it suffices to prove that $m_{1,l_{2}}$ is a Coifman-Meyer multiplier with operator norm $\lesssim_{p_{1},\ldots,p_{4}} 1$. To accomplish this task, we need to verify the C-M condition derivative estimates
\begin{equation}
|(\nabla_{\eta_{1}}^{k_{1}}\nabla_{\eta_{2}}^{k_{2}}\nabla_{\eta_{3}}^{k_{3}}\nabla_{\eta_{4}}^{k_{4}}m_{1,l_{2}})(\ueta_{4})| \lesssim_{(k_{1},\ldots,k_{4})} (|\eta_{1}|+\cdots+|\eta_{4}|)^{-(k_{1}+\cdots+k_{4})}, \qquad \forall (k_{1},\ldots,k_{4})\in\mathbb{N}_{0}^{4}.
\end{equation}

By the chain and Leibnitz rules,
\begin{itemize}[leftmargin=*]
\item
\begin{align}
|\nabla_{\eta_{1}}^{k_{1}}\cdots\nabla_{\eta_{4}}^{k_{4}}\paren*{\frac{1}{\Phi(\eta)}}| &\lesssim_{k_{1},\ldots,k_{4}} \frac{1}{(|\eta_{1}|+\cdots+|\eta_{4}|)^{k_{1}+\cdots+k_{4}}}, \qquad \forall \ueta_{4}\in \supp(m_{1,l_{2}});
\end{align}
\item
\begin{align}
|\nabla_{\eta_{1}}^{k_{1}}\nabla_{\eta_{3}}^{k_{3}}\paren*{\frac{|\eta_{3}|}{|\eta_{1}|}}| \lesssim_{k_{1},k_{3}} \frac{|\eta_{3}|}{|\eta_{3}|^{k_{3}}|\eta_{1}|^{k_{1}+1}} \lesssim_{k_{1},k_{3}} \frac{1}{(|\eta_{1}|+\cdots+|\eta_{4}|)^{k_{1}+k_{3}}}, \qquad \forall \ueta_{4}\in \supp(m_{1,l_{2}}) \cap \supp\paren*{\nabla\phi \paren*{\frac{2^{5}|\eta_{3}|}{|\eta_{1}|}}};
\end{align}
\item
\begin{equation}
|\nabla_{\eta_{2}}^{k_{2}}\nabla_{\eta_{4}}^{k_{4}} \paren*{\frac{|\eta_{4}|}{|\eta_{2}|}}| \lesssim_{k_{2},k_{4}} \frac{|\eta_{4}|}{|\eta_{4}|^{k_{4}}|\eta_{2}|^{k_{2}+1}} \lesssim_{k_{2}, k_{4}} \frac{1}{(|\eta_{1}|+\cdots+|\eta_{4}|)^{k_{2}+k_{4}}}, \qquad \forall \ueta_{4}\in\supp(m_{1,l_{2}}) \cap \supp\paren*{\nabla\phi \paren*{\frac{2^{5}|\eta_{4}|}{|\eta_{2}|}}};
\end{equation}
\item
\begin{equation}
|\nabla_{\eta_{2}}^{k_{2}}\nabla_{\eta_{3}}^{k_{3}}\paren*{\frac{(\eta_{2}+\eta_{3})_{1}^{2}}{|\eta_{2}+\eta_{3}|^{2}}}| \lesssim_{k_{2},k_{3}} \frac{1}{|\eta_{2}+\eta_{3}|^{k_{2}+k_{3}}} \lesssim_{k_{2},k_{3}}  \frac{1}{(|\eta_{1}|+\cdots+|\eta_{4}|)^{k_{2}+k_{3}}}, \qquad \ueta_{4} \in\supp(m_{1,l_{2}});
\end{equation}
\item
\begin{equation}
|\nabla_{\eta_{1}}^{k_{1}}\nabla_{\eta_{2}}^{k_{2}}\paren*{\frac{|\eta_{1}|}{|\eta_{2}|}}| \lesssim_{k_{1},k_{2}} \frac{|\eta_{1}|}{|\eta_{1}|^{k_{1}}|\eta_{2}|^{k_{2}+1}} \lesssim_{k_{1},k_{2}} \frac{1}{(|\eta_{1}|+\cdots+|\eta_{4}|)^{k_{1}+k_{2}}}, \qquad \forall \ueta_{4}\in \supp(m_{1,l_{2}});
\end{equation}
and
\item
\begin{equation}
\begin{split}
&|\nabla_{\eta_{1}}^{k_{1}}\nabla_{\eta_{2}}^{k_{2}}\paren*{1-\phi\paren*{\eta_{1}-2^{-l_{2}}\xi(t)}\phi\paren*{\eta_{2}-2^{-l_{2}}\xi(t)}}| \\
&\phantom{=} \lesssim_{k_{1},k_{2}} |(\nabla^{k_{1}}\phi)\paren*{\eta_{1}-2^{-l_{2}}\xi(t)} \otimes (\nabla^{k_{2}}\phi)\paren*{\eta_{2}-2^{-l_{2}}\xi(t)}|, \qquad \forall \ueta_{4}\in \supp(m_{1,l_{2}}).
\end{split}
\end{equation}
Observe that $\nabla\phi(\xi_{1})\neq 0$ implies that $1\leq |\xi_{1}|\leq 2$. Since $|\xi(t)|\leq 2^{l_{2}-19}$ for all $t\in G_{\beta}^{l_{2}}$, it follows that $|\eta_{1}| \sim 1$ on the support of $\nabla\phi(\eta_{1}-2^{-l_{2}}\xi(t))$. By symmetry, $|\eta_{2}| \sim 1$ on the support of $\nabla\phi(\eta_{2}-2^{-l_{2}}\xi(t))$. Since $\vartheta \in C_{c}^{\infty}(\R)$, it follows that for any $k_{1},k_{2}\in\mathbb{N}$,
\begin{equation}
|(\nabla^{k_{1}}\vartheta)(\eta_{1}-2^{-l_{2}}\xi(t)) \otimes (\nabla^{k_{2}}\vartheta)(\eta_{2}-2^{-l_{2}}\xi(t))| \lesssim_{k_{1},k_{2}} \frac{1}{(|\eta_{1}|+\cdots+|\eta_{4}|)^{k_{1}+k_{2}}},  \qquad \forall \ueta_{4}\in\supp(m_{1,l_{2}}).
\end{equation}
\end{itemize}

Combining the above derivative estimates together with the fact that $\phi\in C_{c}^{\infty}(\R^{2})$, $\vartheta,\varphi\in C_{c}^{\infty}(\R)$, and the Leibnitz rule, we see that
\begin{equation}
\begin{split}
&\left|\nabla_{\eta_{1}}^{k_{1}^{1}}\nabla_{\eta_{2}}^{k_{2}^{1}}\nabla_{\eta_{3}}^{k_{3}^{1}}\nabla_{\eta_{4}}^{k_{4}^{1}}\left(\frac{1}{\Phi(\eta)}\right) \otimes \nabla_{\eta_{2}}^{k_{2}^{2}}\nabla_{\eta_{3}}^{k_{3}^{2}}\left(\frac{(\eta_{2}+\eta_{3})_{1}^{2}}{|\eta_{2}+\eta_{3}|^{2}}\right) \otimes \nabla_{\eta_{2}}^{k_{2}^{3}}\nabla_{\eta_{3}}^{k_{3}^{3}}\left(\vartheta\left(2^{5}\frac{|\eta_{3}|}{|\eta_{2}|}\right)\right) \otimes \nabla_{\eta_{1}}^{k_{1}^{2}}\nabla_{\eta_{4}}^{k_{4}^{2}}\left(\vartheta\left(2^{5}\frac{|\eta_{4}|}{|\eta_{1}|}\right)\right) \right.\\
&\phantom{=}\qquad \left.\otimes \nabla_{\eta_{1}}^{k_{1}^{3}}\nabla_{\eta_{2}}^{k_{2}^{4}}\left(\varphi\left(\frac{|\eta_{1}|}{|\eta_{2}|}\right)\right)\otimes \nabla_{\eta_{1}}^{k_{1}^{4}}\nabla_{\eta_{2}}^{k_{2}^{5}}\paren*{1-\phi\left(\eta_{1}-2^{-l_{2}}\xi(t)\right)\phi\left(\eta_{2}-2^{-l_{2}}\xi(t)\right)}\right| \\
&\phantom{=}\lesssim_{k_{j}^{i}} \frac{1}{(|\eta_{1}|+\cdots+|\eta_{4}|)^{(k_{1}^{1}+\cdots+k_{1}^{4})+(k_{2}^{1}+\cdots+k_{2}^{5}) + (k_{3}^{1}+\cdots+k_{3}^{3}) + (k_{4}^{1}+k_{4}^{2})}}, \qquad \forall k_{j}^{i}\in\mathbb{N}_{0}.
\end{split}
\end{equation}
Satisfaction of the C-M condition then follows from the definition of $m_{1,l_{2}}$ and the Leibnitz rule.
\end{proof}

Now performing a normal form transformation (i.e. integrate by parts in time), we obtain that
\begin{align}
&\int_{G_{\beta}^{l_{2}}}dt\int_{\R^{4}}dxdy\int_{\R^{8}}d\ueta_{4}\frac{(x-y)}{|x-y|}\cdot\Im{\bar{v}(\nabla-i\xi(t))v}(t,x) \frac{d}{dt}\left(e^{it\Phi(\ueta_{4})}\right) e^{iy\cdot \eta}m_{1}(t,\ueta_{4}) \hat{w}_{h}(t,\eta_{1})\hat{w}_{h}(t,\eta_{2})\hat{\bar{w}}_{l}(t,\eta_{3})\hat{\bar{w}}_{l}(t,\eta_{4})\nonumber\\
&=\int_{\R^{4}}dxdy\int_{\R^{8}}d\ueta_{4} \frac{(x-y)}{|x-y|}\cdot\Im{\bar{v}(\nabla-i\xi(t))v}(t,x) e^{it\Phi(\ueta_{4})}e^{iy\cdot\eta}m_{1}(t,\ueta_{4})\hat{w}_{h}(t,\eta_{1})\hat{w}_{h}(t,\eta_{2})\hat{\bar{w}}_{l}(t,\eta_{3})\hat{\bar{w}}_{l}(t,\eta_{4})|_{t=a}^{t=b}\\
&\phantom{=}-\int_{G_{\beta}^{l_{2}}}dt\int_{\R^{4}}dxdy\int_{\R^{8}}d\ueta_{4} \frac{(x-y)}{|x-y|}\cdot\frac{d}{dt}\left(\Im{\bar{v}(\nabla-i\xi(t))v}\right)(t,x) e^{it\Phi(\ueta_{4})}e^{iy\cdot \eta}m_{1}(t,\ueta_{4}) \hat{w}_{h}(t,\eta_{1})\hat{w}_{h}(t,\eta_{2})\hat{\bar{w}}_{l}(t,\eta_{3})\hat{\bar{w}}_{l}(t,\eta_{4})\\
&\phantom{=}-\int_{G_{\beta}^{l_{2}}}dt\int_{\R^{4}}dxdy\int_{\R^{8}}d\ueta_{4} \frac{(x-y)}{|x-y|}\cdot\Im{\bar{v}(\nabla-i\xi(t))v}(t,x) e^{it\Phi(\ueta_{4})}e^{iy\cdot \eta}\frac{d}{dt}(m_{1})(t,\ueta_{4}) \hat{w}_{h}(t,\eta_{1})\hat{w}_{h}(t,\eta_{2})\hat{\bar{w}}_{l}(t,\eta_{3})\hat{\bar{w}}_{l}(t,\eta_{4})\\
&\phantom{=}-\int_{G_{\beta}^{l_{2}}}dt\int_{\R^{4}}dxdy\int_{\R^{8}}d\ueta_{4}\frac{(x-y)}{|x-y|}\cdot\Im{\bar{v}(\nabla-i\xi(t))v}(t,x) e^{it\Phi(\ueta_{4})}e^{iy\cdot \eta} m_{1}(t,\ueta_{4})  \frac{d}{dt}\paren*{\hat{w}_{h}(t,\eta_{1})\hat{w}_{h}(t,\eta_{2})\hat{\bar{w}}_{l}(t,\eta_{3})\hat{\bar{w}}_{l}(t,\eta_{4})}\\
&\eqqcolon \mathrm{Term}_{\partial} + \mathrm{Term}_{v}+\mathrm{Term}_{m} + \mathrm{Term}_{w}.
\end{align}

\begin{description}[leftmargin=*]
\item[Estimate for $\mathrm{Term}_{\partial}$:]
By H\"{o}lder's inequality and Plancherel's theorem,
\begin{align}
2^{k-2i} |\mathrm{Term}_{\partial}| &\lesssim 2^{k-2i}2^{i}\|v_{0}\|_{L^{2}}^{2}\|m_{1}(t,\uD_{4})\paren*{u_{h}, u_{h}, \bar{u}_{l}, \bar{u}_{l}}\|_{L_{t}^{\infty}L_{x}^{1}} \nonumber\\
&\lesssim 2^{k-i}2^{-2l_{2}}\|v_{0}\|_{L^{2}}^{2}\|u_{h}\|_{L_{t}^{\infty}L_{x}^{2}}^{2}\|u_{l}\|_{L_{t,x}^{\infty}}^{2} \nonumber\\
&\lesssim 2^{k-i}\|v_{0}\|_{L^{2}}^{2}. 
\end{align}
where we use lemma \ref{lem:CM_1} with $(p_{1},p_{2},p_{3},p_{4})=(2,2,\infty,\infty)$ to obtain the penultimate inequality, and Bernstein's lemma and mass conservation to obtain the ultimate inequality.

\item[Estimate for $\mathrm{Term}_{v}$:]
Calculus shows that
\begin{equation}
\partial_{t}\Im{\bar{v}(\nabla-i\xi(t))_{\mu}v}(t,x) =  -\xi_{\mu}'(t)|v(t,x)|^{2}-\partial_{\nu}\Re{\partial_{\nu}\bar{v}(\partial_{\mu}-i\xi_{\mu}(t))v} + \partial_{\nu}\Re{\bar{v}(\partial_{\mu}-i\xi_{\mu}(t))\partial_{\nu}v}(t,x).
\end{equation}
Substituting this identity into $\mathrm{Term}_{v}$ and applying the triangle inequality, we have reduced to estimating three terms. First, by H\"{o}lder's inequality, Bernstein's lemma, lemma \ref{lem:CM_1}  with $(p_{1},p_{2},p_{3},p_{4})=(2,2,\infty,\infty)$, and mass conservation,
\begin{align}
&2^{k-2i}\left|\int_{G_{\beta}^{l_{2}}}\int_{\R^{4}}\frac{(x-y)}{|x-y|}\cdot \xi'(t)|v(t,x)|^{2}m_{1}(t,\uD_{4})\paren*{u_{h}, u_{h}, \bar{u}_{l}, \bar{u}_{l}}(t,y)dxdydt\right| \nonumber\\
&\phantom{=}\lesssim 2^{k-2i}2^{-2l_{2}}\int_{G_{\beta}^{l_{2}}}|\xi'(t)| \|v\|_{L_{t}^{\infty}L_{x}^{2}}^{2}\|u_{h}\|_{L_{t}^{\infty}L_{x}^{2}}^{2}\|u_{l}\|_{L_{t,x}^{\infty}}^{2} \nonumber\\
&\phantom{=}\lesssim 2^{k-2i}\|v_{0}\|_{L^{2}}^{2}\int_{G_{\beta}^{l_{2}}}\epsilon_{1}^{-1/2}N(t)^{3}dt \nonumber\\
&\phantom{=}\lesssim \epsilon_{3}^{1/2}2^{k+l_{2}-2i}\|v_{0}\|_{L^{2}}^{2}.
\end{align}
Next, integrating by parts in $x$ to move a $\partial_{\nu}$ onto the potential $\frac{(x-y)}{|x-y|}$, followed by H\"{o}lder's inequality and Hardy-Littlewood-Sobolev lemma, we see that
\begin{align}
&2^{k-2i}\left|\int_{G_{\beta}^{l_{2}}}\int_{\R^{4}}\frac{(x-y)_{\mu}}{|x-y|}\partial_{\nu}\Re{\partial_{\nu}\bar{v}(\nabla-i\xi(t))_{\mu}v}(t,x) m_{1}(t,\uD_{4})\paren*{u_{h}, u_{h}, \bar{u}_{l}, \bar{u}_{l}}(t,y) dxdydt\right| \nonumber \\
&\phantom{=} \lesssim 2^{k-2i}\| |\nabla|^{-1}(\partial_{\nu}\bar{v}(\nabla-i\xi(t))v)\|_{L_{t}^{3}L_{x}^{6}} \|m_{1}(t,\uD_{4})\paren*{u_{h}, u_{h}, \bar{u}_{l}, \bar{u}_{l}}\|_{L_{t}^{3/2}L_{x}^{6/5}} \nonumber\\
&\phantom{+} \lesssim 2^{k-2i}\|\partial_{\nu}v(\nabla-i\xi(t))v\|_{L_{t}^{3}L_{x}^{3/2}}\|m_{1}(t,\uD_{4})\paren*{u_{h}, u_{h}, \bar{u}_{l}, \bar{u}_{l}}\|_{L_{t}^{3/2}L_{x}^{6/5}}. 
\end{align}
Using H\"{o}lder's inequality, Bernstein's lemma, Strichartz estimates, and  lemma \ref{lem:CM_1} with $(p_{1},p_{2},p_{3},p_{4})=(6,6,4,4)$, we obtain that the preceding expression is
\begin{align}
\lesssim 2^{k-2l_{2}}\|v_{0}\|_{L^{2}}^{2} \left\| \|u_{h}(t)\|_{L_{x}^{6}} \|u_{h}(t)\|_{L_{x}^{6}} \|u_{l}(t)\|_{L_{x}^{4}} \|u_{l}(t)\|_{L_{x}^{4}} \right\|_{L_{t}^{3/2}} &\lesssim 2^{k-2l_{2}}\|v_{0}\|_{L^{2}}^{2}\|u_{h}\|_{L_{t}^{3}L_{x}^{6}}^{2}\|u_{l}\|_{L_{t}^{\infty}L_{x}^{4}}^{2} \nonumber\\
&\lesssim 2^{k-l_{2}}\|v_{0}\|_{L^{2}}^{2}\|u\|_{\tilde{X}_{l_{2}}(G_{\beta}^{l_{2}}\times\R^{2})}^{2},
\end{align}
where we use lemma \ref{lem:lohi_embed}, another application of Bernstein's lemma, and mass conservation to obtain the ultimate inequality. Repeating the same argument \emph{mutatis mutandis}, we also obtain the estimate
\begin{equation}
\begin{split}
&2^{k-2i}\left|\int_{G_{\beta}^{l_{2}}}\int_{\R^{4}} \frac{(x-y)}{|x-y|}\cdot\partial_{\nu}\Re{\bar{v}(\nabla-i\xi(t))\partial_{\nu}v}(t,x)m_{1}(t,\uD_{4})\paren*{u_{h}, u_{h}, \bar{u}_{l}, \bar{u}_{l}}(t,y) dxdydt\right| \nonumber\\
&\phantom{=}\lesssim 2^{k-l_{2}}\|v_{0}\|_{L^{2}}^{2}\|u\|_{\tilde{X}_{l_{2}}(G_{\beta}^{l_{2}}\times\R^{2})}^{2},
\end{split}
\end{equation}
which completes the proof of the estimate of $\mathrm{Term}_{v}$.

\item[Estimate for $\mathrm{Term}_{m}$:]
Calculus shows that
\begin{align}
(\partial_{t}m)(t,\ueta_{4}) &= \frac{1}{\Phi(\ueta_{4})}\phi\left(\frac{\eta}{2^{l_{2}-10}}\right)\left(\phi\left(\frac{\eta_{2}-\xi(t)}{2^{l_{2}}}\right)\nabla\phi\left(\frac{\eta_{1}-\xi(t)}{2^{l_{2}}}\right) + \phi\left(\frac{\eta_{1}-\xi(t)}{2^{l_{2}}}\right)\nabla\phi\left(\frac{\eta_{2}-\xi(t)}{2^{l_{2}}}\right)\right) \cdot 2^{-l_{2}}\xi'(t)\frac{(\eta_{2}+\eta_{3})_{1}^{2}}{|\eta_{2}+\eta_{3}|^{2}}.
\end{align}
Define the time-dependent $\C^{2}$-valued symbol $m_{2}$ by
\begin{equation}
\begin{split}
m_{2}(t,\ueta_{4}) &\coloneqq \frac{1}{\Phi(\ueta_{4})}\phi\left(\frac{\eta}{2^{l_{2}-10}}\right)\frac{(\eta_{2}+\eta_{3})_{1}^{2}}{|\eta_{2}+\eta_{3}|^{2}}\phi\left(\frac{\eta_{3}}{2^{l_{2}-10}}\right)\phi\left(\frac{\eta_{4}}{2^{l_{2}-10}}\right)\vartheta\left(2^{5}\frac{|\eta_{3}|}{|\eta_{1}|}\right)\vartheta\left(2^{5}\frac{|\eta_{4}|}{|\eta_{2}|}\right)\varphi\left(\frac{|\eta_{1}|}{|\eta_{2}|}\right) \times \\
&\phantom{=}\left(\phi\left(\frac{\eta_{2}-\xi(t)}{2^{l_{2}}}\right)\nabla\phi\left(\frac{\eta_{1}-\xi(t)}{2^{l_{2}}}\right) + \phi\left(\frac{\eta_{1}-\xi(t)}{2^{l_{2}}}\right)\nabla\phi\left(\frac{\eta_{2}-\xi(t)}{2^{l_{2}}}\right)\right), \qquad \forall (t,\ueta_{4})\in G_{\beta}^{l_{2}} \times \R^{8}.
\end{split}
\end{equation}
By repeating the analysis in the proof of lemma \ref{lem:CM_1}, one can show that the family of multilinear multiplier operators $\{m_{2}(t,\uD_{4})\}_{t\in G_{\beta}^{l_{2}}}$ extend to bounded operators $\prod_{l=1}^{4}L^{p_{l}}(\R^{2}) \rightarrow L^{p}(\R^{2};\C^{2})$ with operator norm $\lesssim_{p_{1},\ldots,p_{4}} 2^{-2l_{2}}$ for all $t\in G_{\beta}^{l_{2}}$. This is precisely the content of the next lemma, the proof of which we omit.

\begin{lemma}[Second C-M estimate]\label{lem:CM_2} %Second Coifman-Meyer estimate
For any real numbers $1<p_{1},\ldots,p_{4}\leq\infty$ and $0<p<\infty$ satisfying $\frac{1}{p}=\frac{1}{p_{1}}+\cdots+\frac{1}{p_{4}}$, there exists a constant $C_{p_{1},\ldots,p_{4}}>0$ such that the multilinear operator $m_{2}(t,\uD_{4}): \prod_{l=1}^{4}\mathcal{S}(\R^{2})\rightarrow\mathcal{S}'(\R^{2})$ with symbol $m_{2}(t)$ extends to a bounded operator
\begin{equation}
m_{2}(t,\uD_{4}):\prod_{l=1}^{4}L^{p_{l}}(\R^{2}) \rightarrow L^{p}(\R^{2};\C^{2}), \qquad \sup_{t\in G_{\beta}^{l_{2}}} \|m_{2}(t,\uD_{4})\|_{L^{p_{1}}\times\cdots\times L^{p_{4}}\rightarrow L^{p}} \leq C_{p_{1},\ldots,p_{4}}2^{-2l_{2}}
\end{equation}
for all $20\leq i\leq j\leq k_{0}$, $0\leq l_{2}\leq i-10$, and $G_{\beta}^{l_{2}}\subset [0,T]$.
\end{lemma}

Now using H\"{o}lder's inequality, lemma \ref{lem:CM_2} with $(p_{1},p_{2},p_{3},p_{4})=(2,2,\infty,\infty)$, Bernstein's lemma, and mass conservation, we obtain that
\begin{equation}
2^{k-2i} |\mathrm{Term}_{m}| \lesssim 2^{k-2i}2^{i-3l_{2}}\int_{G_{\beta}^{l_{2}}}|\xi'(t)| \|u_{h}\|_{L_{t}^{\infty}L_{x}^{2}}^{2} \|u_{l}\|_{L_{t,x}^{\infty}}^{2}dt \lesssim \epsilon_{3}^{1/2}2^{k-i}\|v_{0}\|_{L^{2}}^{2}.
\end{equation}

\item[Estimate for $\mathrm{Term}_{w}$:]
First observe that $w_{h}$ and $\bar{w}_{l}$ respectively solve the equations
\begin{align}
\partial_{t}w_{h} &= e^{-it\Delta}\paren*{\frac{d}{dt}P_{\xi(t),>l_{2}-5}}u -ie^{-it\Delta}P_{\xi(t),>l_{2-5}}F(u) \\
\partial_{t}\bar{w}_{l} &= e^{it\Delta}\ol{\paren*{\frac{d}{dt}P_{\xi(t),\leq l_{2}-15}}u} + ie^{it\Delta}\ol{P_{\xi(t),\leq l_{2}-15}F(u)}.
\end{align}
It therefore follows from the product rule, triangle inequality, Plancherel's theorem, and mass conservation that
\begin{align}
|\mathrm{Term}_{w}| &\lesssim 2^{i}\|v_{0}\|_{L^{2}}^{2}\|m_{1}(t,\uD_{4})\paren*{\paren*{\frac{d}{dt}P_{\xi(t),>l_{2}-5}}u, u_{h}, \bar{u}_{l},\bar{u}_{l}}\|_{L_{t,x}^{1}} \nonumber\\
&\phantom{=}+2^{i}\|v_{0}\|_{L^{2}}^{2}\|m_{1}(t,\uD_{4})\paren*{u_{h},\paren*{\frac{d}{dt}P_{\xi(t), >l_{2}-5}}u, \bar{u}_{l}, \bar{u}_{l}}\|_{L_{t,x}^{1}} \nonumber\\
&\phantom{=}+2^{i}\|v_{0}\|_{L^{2}}^{2}\|m_{1}(t,\uD_{4})\paren*{u_{h}, u_{h}, \ol{\paren*{\frac{d}{dt}P_{\xi(t), \leq l_{2}-15}}u}, \bar{u}_{l}}\|_{L_{t,x}^{1}} \nonumber\\
&\phantom{=}+2^{i}\|v_{0}\|_{L^{2}}^{2}\|m_{1}(t,\uD_{4})\paren*{u_{h}, u_{h}, \bar{u}_{l}, \ol{\paren*{\frac{d}{dt}P_{\xi(t), \leq l_{2}-15}}u}}\|_{L_{t,x}^{1}} \nonumber\\
&\phantom{=}+2^{i}\|v_{0}\|_{L^{2}}^{2}\|m_{1}(t,\uD_{4})\paren*{P_{\xi(t), >l_{2}-5}F(u), u_{h}, \bar{u}_{l}, \bar{u}_{l}}\|_{L_{t,x}^{1}} \nonumber\\
&\phantom{=}+2^{i}\|v_{0}\|_{L^{2}}^{2}\|m_{1}(t,\uD_{4})\paren*{u_{h}, P_{\xi(t), >l_{2}-5}F(u), \bar{u}_{l}, \bar{u}_{l}}\|_{L_{t,x}^{1}}\nonumber \\
&\phantom{=}+2^{i}\|v_{0}\|_{L^{2}}^{2}\|m_{1}(t,\uD_{4})\paren*{u_{h}, u_{h}, \ol{P_{\xi(t), \leq l_{2}-15}F(u)}, \bar{u}_{l}}\|_{L_{t,x}^{1}} \nonumber\\
&\phantom{=}+2^{i}\|v_{0}\|_{L^{2}}^{2}\|m_{1}(t,\uD_{4})\paren*{u_{h}, u_{h}, \bar{u}_{l}, \ol{P_{\xi(t), \leq l_{2}-15}F(u)}}\|_{L_{t,x}^{1}} \nonumber\\
&\eqqcolon \mathrm{Term}_{w,1}+\cdots+\mathrm{Term}_{w,8}.
\end{align}
By lemma \ref{lem:CM_1} with $(p_{1},p_{2},p_{3},p_{4})=(2,2,\infty,\infty)$, H\"{o}lder's inequality, Bernstein's lemma, and mass conservation,
\begin{align}
2^{k-2i}\paren*{\mathrm{Term}_{w,1}+\cdots+\mathrm{Term}_{w,4}} \lesssim 2^{k-2i}2^{i-2l_{2}}2^{-l_{2}}\|P_{\xi(t),\leq l_{2}+2}u\|_{L_{t,x}^{\infty}}^{2}\int_{G_{\beta}^{l_{2}}}|\xi'(t)|dt \lesssim 2^{k-i}\|v_{0}\|_{L^{2}}^{2}.
\end{align}

To estimate the remaining terms $\mathrm{Term}_{w,5},\ldots,\mathrm{Term}_{w,8}$, we first perform a near-far decomposition $u=u_{l}+u_{h}$, where $u_{l}\coloneqq P_{\xi(t),\leq l_{2}-5}u$, substitute this decomposition into $F(u)$, and then algebraically expand to obtain
\begin{align}
P_{\xi(t), > l_{2}-5}F(u) &= P_{\xi(t), > l_{2}-5} \brak*{\E(|u_{h}|^{2})u_{h}} +P_{\xi(t), >l_{2}-5}\brak*{2\E\paren*{\Re{u_{l}\bar{u}_{h}}}u_{h} + \E(|u_{h}|^{2})u_{l}} \nonumber\\
&\phantom{=} +P_{\xi(t), >l_{2}-5}\brak*{\E(|u_{l}|^{2})u_{h} + 2\E\paren*{\Re{u_{h}\bar{u}_{l}}}u_{l}} \nonumber\\
&\eqqcolon F_{h,3}+F_{h,2}+F_{h,1},
\end{align}
and
\begin{align}
P_{\xi(t),\leq l_{2}}F(u) &= P_{\xi(t), \leq l_{2}-15}\brak*{\E(|u_{h}|^{2})u_{h}} +P_{\xi(t), \leq l_{2}-15}\brak*{2\E\paren*{\Re{u_{h}\bar{u}_{l}}}u_{h} + \E(|u_{h}|^{2})u_{l}} \nonumber\\
&\phantom{=}+P_{\xi(t), \leq l_{2}-15}\brak*{\E(|u_{l}|^{2})u_{h} + 2\E\paren*{\Re{u_{l}\bar{u}_{h}}}u_{l}}+P_{\xi(t),\leq l_{2}-15}\brak*{\E(|u_{l}|^{2})u_{l}} \nonumber\\
&\eqqcolon F_{l,3}+F_{l,2}+F_{l,1}+F_{l,0}.
\end{align}
After substituting these decompositions into $\mathrm{Term}_{w,5},\ldots,\mathrm{Term}_{w,8}$, using the multilinearity of the multiplier $m_{1}$, then applying the triangle inequality, we proceed to consider each case separately. The arguments involved will vary based on the number of factors $u_{h}$ and factors $u_{l}$ present in the considered multilinear expressions. We include the details for estimating $\mathrm{Term}_{w,5}$ and $\mathrm{Term}_{w,7}$ and leave the remaining terms to the reader.

\begin{itemize}[leftmargin=*] %Case analysis
\item
We estimate $\|m_{1}(t,\uD_{4})\paren*{F_{h,3}, u_{h}, \bar{u}_{l}, \bar{u}_{l}}\|_{L_{t,x}^{1}}$. By lemma \ref{lem:CM_1} with $(p_{1}, p_{2}, p_{3}, p_{4}) = (4,4/3,\infty,\infty)$, H\"{o}lder's inequality, Calder\'{o}n-Zygmund theorem, then Bernstein's lemma, lemma \ref{lem:lohi_embed}, and mass conservation, we have that
\begin{align}
\|m_{1}(t,\uD_{4})\paren*{F_{h,3}, u_{h}, \bar{u}_{l}, \bar{u}_{l}}\|_{L_{t,x}^{1}} \lesssim 2^{-2l_{2}}\|F_{h,3}\|_{L_{t,x}^{4/3}} \|u_{h}\|_{L_{t,x}^{4}}\|u_{l}\|_{L_{t,x}^{\infty}}^{2}\lesssim \|u_{h}\|_{L_{t,x}^{4}}^{4} \lesssim  \|u\|_{\tilde{X}_{l_{2}}(G_{\beta}^{l_{2}}\times\R^{2})}^{4}.
\end{align}

\item
We estimate $\|m_{1}(t,\uD_{4})\paren*{u_{h}, u_{h}, \bar{F}_{l,3}, \bar{u}_{l}}\|_{L_{t,x}^{1}}$. By lemma \ref{lem:CM_1} with $(p_{1}, p_{2}, p_{3}, p_{4}) = (4,4,2,\infty)$, Bernstein's lemma, H\"{o}lder's inequality, Calder\'{o}n-Zygmund theorem, mass conservation, interpolation, and lemma \ref{lem:lohi_embed}, we have that
\begin{align}
\|m_{1}(t,\uD_{4})\paren*{u_{h}, u_{h}, \bar{F}_{l,3}, \bar{u}_{l}}\|_{L_{t,x}^{1}} &\lesssim 2^{-2l_{2}}\|u_{h}\|_{L_{t,x}^{4}}^{2}\|P_{\xi(t), \leq l_{2}}\brak*{\E(|u_{h}|^{2})u_{h}}\|_{L_{t,x}^{2}}\|u_{l}\|_{L_{t,x}^{\infty}} \nonumber\\
&\lesssim \|u_{h}\|_{L_{t,x}^{4}}^{2} \|\E(|u_{h}|^{2})u_{h}\|_{L_{t}^{2}L_{x}^{1}} \nonumber\\
&\lesssim \|u_{h}\|_{L_{t,x}^{4}}^{4} \nonumber\\
&\lesssim \|u\|_{\tilde{X}_{l_{2}}(G_{\beta}^{l_{2}}\times\R^{2})}^{4}.
\end{align}

\item
We estimate $\|m_{1}(t,\uD_{4})\paren*{F_{h,2}, u_{h}, \bar{u}_{l}, \bar{u}_{l}}\|_{L_{t,x}^{1}}$. By the triangle inequality, lemma \ref{lem:CM_1} with $(p_{1}, p_{2}, p_{3}, p_{4}) = (6/5, 6, \infty,\infty)$, H\"{o}lder's inequality, Calder\'{o}n-Zygmund theorem, lemma \ref{lem:lohi_embed}, Bernstein's lemma, and mass conservation,
\begin{align}
\|m_{1}(t,\uD_{4})\paren*{F_{h,2}, u_{h}, \bar{u}_{l}, \bar{u}_{l}}\|_{L_{t,x}^{1}} &\lesssim 2^{-2l_{2}}\|F_{h,2}\|_{L_{t}^{3/2}L_{x}^{6/5}}\|u_{h}\|_{L_{t}^{3}L_{x}^{6}}\|u_{l}\|_{L_{t,x}^{\infty}}^{2} \nonumber\\
&\lesssim 2^{-2l_{2}}\|u_{h}\|_{L_{t}^{3}L_{x}^{6}}^{3}\|u_{l}\|_{L_{t}^{\infty}L_{x}^{2}}\|u_{l}\|_{L_{t,x}^{\infty}}^{2} \nonumber\\
&\lesssim \|u\|_{\tilde{X}_{l_{2}}(G_{\beta}^{l_{2}}\times\R^{2})}^{3}.
\end{align}

\item
We estimate $\|m_{1}(t,\uD_{4})\paren*{u_{h}, u_{h}, \bar{F}_{l,2}, \bar{u}_{l}}\|_{L_{t,x}^{1}}$. By the triangle inequality, lemma \ref{lem:CM_1} with $(p_{1}, p_{2}, p_{3}, p_{4}) = (4,4, 2,\infty)$, H\"{o}lder's inequality, Calder\'{o}n-Zygmund theorem, lemma \ref{lem:lohi_embed}, Bernstein's lemma, and mass conservation,
\begin{align}
\|m_{1}(t,\uD_{4})\paren*{u_{h}, u_{h}, \bar{F}_{l,2}, \bar{u}_{l}}\|_{L_{t,x}^{1}} &\lesssim 2^{-2l_{2}}\|u_{h}\|_{L_{t,x}^{4}}^{2}\|F_{l,2}\|_{L_{t,x}^{2}}\|u_{l}\|_{L_{t,x}^{\infty}} \nonumber\\
&\lesssim 2^{-2l_{2}}\|u\|_{\tilde{X}_{l_{2}}(G_{\beta}^{l_{2}}\times\R^{2})}^{2} \|u_{h}\|_{L_{t,x}^{4}}^{2}\|u_{l}\|_{L_{t,x}^{\infty}}^{2} \nonumber\\
&\lesssim \|u\|_{\tilde{X}_{l_{2}}(G_{\beta}^{l_{2}}\times\R^{2})}^{4}.
\end{align}

\item
We estimate $\|m_{1}(t,\uD_{4})\paren*{\bar{F}_{h,1}, u_{h}, \bar{u}_{l}, \bar{u}_{l}}\|_{L_{t,x}^{1}}$. By the triangle inequality, lemma \ref{lem:CM_1} with $(p_{1}, p_{2}, p_{3}, p_{4}) = (2,3,12,12)$, H\"{o}lder's inequality, Calder\'{o}n-Zygmund theorem, we have that
\begin{align}
\|m_{1}(t,\uD_{4})\paren*{\bar{F}_{h,1}, u_{h}, \bar{u}_{l}, \bar{u}_{l}}\|_{L_{t,x}^{1}} &\lesssim 2^{-2l_{2}}\|F_{h,1}\|_{L_{t,x}^{2}}\|u_{h}\|_{L_{t}^{6}L_{x}^{3}} \|u_{l}\|_{L_{t}^{6}L_{x}^{12}}^{2} \nonumber\\
&\lesssim 2^{-2l_{2}}\|u_{h}\|_{L_{t}^{6}L_{x}^{3}}^{2}\|u_{l}\|_{L_{t}^{6}L_{x}^{12}}^{4}\nonumber\\
&\lesssim 2^{-2l_{2}} \|u\|_{\tilde{X}_{l_{2}}(G_{\beta}^{l_{2}}\times\R^{2})}^{2} \| |\nabla-i\xi(t)|^{1/2} u_{l}\|_{L_{t}^{6}L_{x}^{3}}^{4}\nonumber\\
&\lesssim 2^{-l_{2}} \|u\|_{\tilde{X}_{l_{2}}(G_{\beta}^{l_{2}}\times\R^{2})}^{2}  \| |\nabla-i\xi(t)|^{1/2} u_{l}\|_{L_{t}^{3}L_{x}^{6}}^{2}\nonumber\\
&\lesssim \|u\|_{\tilde{X}_{l_{2}}(G_{\beta}^{l_{2}}\times\R^{2})}^{4},
\end{align}
where we use Sobolev embedding and lemma \ref{lem:lohi_embed} to obtain the antepenultimate inequality; interpolation, Bernstein's lemma, and mass conservation to obtain the penultimate inequality; and lemma \ref{lem:BT_embed} to obtain the ultimate inequality.

\item
We estimate $\|m_{1}(t,\uD_{4})\paren*{u_{h}, u_{h}, \bar{F}_{l,1}, \bar{u}_{l}}\|_{L_{t,x}^{1}}$. By the triangle inequality, lemma \ref{lem:CM_1} with $(p_{1}, p_{2}, p_{3}, p_{4}) = (6,6, 3/2, \infty)$, H\"{o}lder's inequality, Calder\'{o}n-Zygmund theorem, followed by lemma \ref{lem:lohi_embed}, Bernstein's lemma, and mass conservation, we have that
\begin{align}
\|m_{1}(t,\uD_{4})\paren*{u_{h}, u_{h}, \bar{F}_{l,1}, \bar{u}_{l}}\|_{L_{t,x}^{1}} &\lesssim 2^{-2l_{2}}\|u_{h}\|_{L_{t}^{3}L_{x}^{6}}^{2}\|F_{l,1}\|_{L_{t}^{3}L_{x}^{3/2}}\|u_{l}\|_{L_{t,x}^{\infty}} \nonumber\\
&\lesssim 2^{-l_{2}}\|u_{h}\|_{L_{t}^{3}L_{x}^{6}}^{3}\|u_{l}\|_{L_{t}^{\infty}L_{x}^{2}}\|u_{l}\|_{L_{t,x}^{\infty}} \nonumber\\
&\lesssim \|u\|_{\tilde{X}_{l_{2}}(G_{\beta}^{l_{2}}\times\R^{2})}^{3}.
\end{align}

\item
We estimate $\|m_{1}(t,\uD_{4})\paren*{u_{h}, u_{h}, \bar{F}_{l,0}, \bar{u}_{l}}\|_{L_{t,x}^{1}}$. By lemma \ref{lem:CM_1} with $(p_{1}, p_{2}, p_{3}, p_{4}) = (3,3, 4, 12)$, H\"{o}lder's inequality, Calder\'{o}n-Zygmund theorem, we have that
\begin{align}
\|m_{1}(t,\uD_{4})\paren*{u_{h}, u_{h}, \bar{F}_{l,0}, \bar{u}_{l}}\|_{L_{t,x}^{1}} &\lesssim 2^{-2l_{2}}\|u_{h}\|_{L_{t}^{6}L_{x}^{3}}^{2} \|F_{l,0}\|_{L_{t}^{2}L_{x}^{4}} \|u_{l}\|_{L_{t}^{6}L_{x}^{12}} \nonumber\\
&\lesssim 2^{-2l_{2}}\|u_{h}\|_{L_{t}^{6}L_{x}^{3}}^{2}\|u_{l}\|_{L_{t}^{6}L_{x}^{12}}^{4} \nonumber\\
&\lesssim 2^{-2l_{2}}\|u_{h}\|_{L_{t}^{3}L_{x}^{6}} \| |\nabla-i\xi(t)|^{1/2}u_{l}\|_{L_{t}^{6}L_{x}^{3}}^{4} \nonumber\\
&\lesssim 2^{-l_{2}}\|u_{h}\|_{L_{t}^{3}L_{x}^{6}}\||\nabla-i\xi(t)|^{1/2}u_{l}\|_{L_{t}^{3}L_{x}^{6}}^{2} \nonumber\\
&\lesssim \|u\|_{\tilde{X}_{l_{2}}(G_{\beta}^{l_{2}}\times\R^{2})}^{3},
\end{align}
where we use interpolation, mass conservation, and Sobolev embedding to obtain the antepenultimate inequality; interpolation, mass conservation, and Bernstein's lemma to obtain the penultimate inequality; and lemmas \ref{lem:lohi_embed} and \ref{lem:BT_embed} to obtain the ultimate inequality.
\end{itemize}
Combining the above estimates, we conclude that
\begin{equation}
2^{k-2i}\paren*{\mathrm{Term}_{w,5}+\cdots+\mathrm{Term}_{w,8}} \lesssim 2^{k-i}\|v_{0}\|_{L^{2}}^{2} \paren*{1+\|u\|_{\tilde{X}_{l_{2}}(G_{\beta}^{l_{2}}\times\R^{2})}^{4}},
\end{equation}
and therefore $|\mathrm{Term}_{w}| \leq 2^{k-i}\|v_{0}\|_{L^{2}}^{2}(1+\|u\|_{\tilde{X}_{l_{2}}(G_{\beta}^{l_{2}}\times\R^{2})}^{4})$.
\end{description}

Collecting our estimates, we have shown that
\begin{equation}
|\eqref{eq:ibs2_err3}| \lesssim 2^{k-l_{2}}\|v_{0}\|_{L^{2}}^{2}\paren*{1+\|u\|_{\tilde{X}_{l_{2}}(G_{\beta}^{l_{2}}\times\R^{2})}^{4}}.
\end{equation}
\end{description}

Bookkeeping the estimates for \eqref{eq:ibs2_main}-\eqref{eq:ibs2_err3} and using that $\|u\|_{\tilde{X}_{l_{2}}(G_{\beta}^{l_{2}}\times\R^{2})}\leq \|u\|_{\tilde{X}_{i}(G_{\alpha}^{i}\times\R^{2})}$ for all $l_{2}\leq i-10, G_{\beta}^{l_{2}}\subset G_{\alpha}^{i}$ completes the proof of proposition \ref{prop:ibs_2}.
\end{proof}

\subsection{Bilinear Strichartz estimate III}\label{ssec:BSE_3} %Proof of third bilinear Strichartz estimate
For the reader's benefit, we reproduce the statement of proposition \ref{prop:ibs_3} below.

\ibsIII*

\begin{proof}
Let $v_{0}$ satisfy the conditions of the statement of the proposition, and for each integer $k\leq i-8$ and vector $a\in\mathbb{Z}^{2}$, define
\begin{equation}
v_{a,k} \coloneqq e^{it\Delta}P_{Q_{a}^{k}}v_{0}.
\end{equation}
For each integer $0\leq l_{2}\leq i-10$, define
\begin{equation}
w_{l_{2}} \coloneqq P_{\xi(t),\leq l_{2}}u.
\end{equation}
For each $(l_{2},k,a)\in \N_{\leq i-10}\times\Z_{\leq i-8}\times\Z^{2}$, define the interaction Morawetz functional
\begin{equation}
\begin{split}
M_{l_{2},k,a}(t) &\eqqcolon \int_{\mathbb{S}^{1}}\int_{\R^{4}} \frac{(x-y)_{\omega}}{|(x-y)_{\omega}|} |w_{l_{2}}(t,y)|^{2} \Im{\bar{v}_{a,k}\partial_{\omega}v_{a,k}}(t,x)dxdyd\omega\\
&\phantom{=}+\int_{\mathbb{S}^{1}}\int_{\R^{4}} \frac{(x-y)_{\omega}}{|(x-y)_{\omega}|} |v_{a,k}(t,y)|^{2} \Im{\bar{w}_{l_{2}}\partial_{\omega}w_{l_{2}}}(t,x)dxdyd\omega.
\end{split}
\end{equation}
Proceeding as in the proof of propositions \ref{prop:ibs_1} and \ref{prop:ibs_2}, we obtain the estimate
\begin{align}
&\sum_{0\leq l_{2}\leq i-10} \paren*{\sup_{k\leq  l_{2}+2}\paren*{\sum_{a\in\Z^{2}}\int_{G_{\alpha}^{i}}\int_{\R^{2}}|w_{l_{2}}(t,x)v_{a,k}(t,x)|^{2}dx dt}^{1/2}}^{2} \nonumber\\
&\phantom{=}\lesssim \sum_{0\leq l_{2}\leq i-10}\paren*{\sup_{k\leq l_{2}+2} \paren*{\sum_{a\in\Z^{2}} 2^{k-2i}\sup_{t\in G_{\alpha}^{i}} |M_{l_{2},k,a}(t)|}^{1/2}}^{2} \label{eq:ibs3_main}\\
&\phantom{=}\qquad +\sum_{0\leq l_{2}\leq i-10}\paren*{\sup_{k\leq l_{2}+2}\paren*{\sum_{a\in\Z^{2}} 2^{k-2i}\int_{G_{\alpha}^{i}}\int_{\R^{4}}\frac{1}{|x-y|} |v_{a,k}(t,y)|^{2} |\vec{R}^{2}(|w_{l_{2}}|^{2})(t,x)|^{2}dxdydt}^{1/2}}^{2}\label{eq:ibs3_nl}\\
&\phantom{=}\qquad +\sum_{0\leq l_{2}\leq i-10}\paren*{\sup_{k\leq  l_{2}+2} \paren*{\sum_{a\in\Z^{2}} 2^{k-2i}\left|\int_{G_{\alpha}^{i}}\int_{\R^{4}}\frac{(x-y)}{|x-y|}\cdot |v_{a,k}(t,y)|^{2} \Re{\bar{\mathcal{N}}_{l_{2}}(\nabla-i\xi(t))w_{l_{2}}}(t,x)dxdydt\right|}^{1/2}}^{2} \label{eq:ibs3_err1}\\
&\phantom{=}\qquad+\sum_{0\leq l_{2}\leq i-10}\paren*{\sup_{k\leq l_{2}+2} \paren*{\sum_{a\in\Z^{2}} 2^{k-2i}\left|\int_{G_{\alpha}^{i}}\int_{\R^{4}} \frac{(x-y)}{|x-y|}\cdot |v_{a,k}(t,y)|^{2}\Re{\bar{w}_{l_{2}}(\nabla-i\xi(t))\mathcal{N}_{l_{2}}}(t,x)dxdydt\right|}^{1/2}}^{2} \label{eq:ibs3_err2}\\
&\phantom{=}\qquad +\sum_{0\leq l_{2}\leq i-10}\paren*{\sup_{k\leq l_{2}+2} \paren*{\sum_{a\in\Z^{2}} 2^{k-2i}\left|\int_{G_{\alpha}^{i}}\int_{\R^{4}} \frac{(x-y)}{|x-y|}\cdot \Im{\bar{w}_{l_{2}}\mathcal{N}_{l_{2}}}(t,y)\Im{\bar{v}_{a,k}(\nabla-i\xi(t))v_{a,k}}(t,x) dxdydt\right|}^{1/2}}^{2}. \label{eq:ibs3_err3}
\end{align}
We now estimate each of the terms \eqref{eq:ibs3_main}-\eqref{eq:ibs3_err3} above separately.

\begin{description}[leftmargin=*]
\item[Estimate for \eqref{eq:ibs3_main}:]
By Cauchy-Schwarz, Plancherel's theorem, and mass conservation,
\begin{equation}
\sup_{t\in G_{\alpha}^{i}} |M_{l_{2},k,a}(t)| \lesssim 2^{i}\|P_{Q_{a}^{k}}v_{0}\|_{L^{2}}^{2}\|w_{l_{2}}\|_{L_{t}^{\infty}L_{x}^{2}(G_{\alpha}^{i}\times\R^{2})}^{2} \lesssim 2^{i}\|P_{Q_{a}^{k}}v_{0}\|_{L^{2}}^{2}.
\end{equation}
Now summing over the $a\in\mathbb{Z}^{2}$, we have by Plancherel's theorem that $\|P_{Q_{a}^{k}}v_{0}\|_{\ell_{a}^{2}L_{x}^{2}(\Z^{2}\times\R^{2})} \leq \|v_{0}\|_{L^{2}(\R^{2})}$. Therefore,
\begin{equation}
\sup_{k\leq l_{2}+2} \left(\sum_{a\in\Z^{2}} 2^{k-2i}2^{i}\|P_{Q_{a}^{k}}v_{0}\|_{L^{2}}^{2} \right)^{1/2} \lesssim 2^{(l_{2}-i)/2} \|v_{0}\|_{L^{2}},
\end{equation}
which implies that
\begin{equation}
|\eqref{eq:ibs3_main}| \lesssim \sum_{0\leq l_{2}\leq i-10} \left( 2^{(l_{2}-i)/2} \|v_{0}\|_{L^{2}}\right)^{2} \lesssim \|v_{0}\|_{L^{2}}^{2}.
\end{equation}

\item[Estimate for \eqref{eq:ibs3_nl}:]
Applying the estimate for \eqref{eq:ibs2_nl} in the proof of proposition \ref{prop:ibs_2} on each subinterval $G_{\beta}^{l_{2}}\subset G_{\alpha}^{i}$ (with $v_{0}$ replaced by $P_{Q_{a}^{k}}v_{0}$) together with the fact that there are $2^{i-l_{2}}$ subintervals $G_{\beta}^{l_{2}}\subset G_{\alpha}^{i}$, we see that
\begin{align}
\left|\int_{G_{\alpha}^{i}}\int_{\R^{4}}\frac{1}{|x-y|} |v_{a,k}(t,y)|^{2}|\vec{R}^{2}(|w_{l_{2}}|^{2})(t,x)|^{2}dxdydt\right|
&\leq \sum_{G_{\beta}^{l_{2}}\subset G_{\alpha}^{i}} \left|\int_{G_{\beta}^{l_{2}}}\int_{\R^{4}}\frac{1}{|x-y|} |v_{a,k}(t,y)|^{2}|\vec{R}^{2}(|w_{l_{2}}|^{2})(t,x)|^{2}dxdydt\right| \nonumber\\
&\lesssim 2^{i}\|P_{Q_{a}^{k}}v_{0}\|_{L^{2}}^{2}\|u\|_{\tilde{X}_{i}(G_{\alpha}^{i}\times\R^{2})}^{4}.
\end{align}
So by Plancherel's theorem,
\begin{equation}
\sup_{k\leq l_{2}+2}\left(\sum_{a\in\mathbb{Z}^{2}} 2^{k-2i}2^{i} \|P_{Q_{a}^{k}}v_{0}\|_{L^{2}}^{2}\|u\|_{\tilde{X}_{i}(G_{\alpha}^{i}\times\R^{2})}^{4}\right)^{1/2} \lesssim 2^{(l_{2}-i)/2}\|v_{0}\|_{L^{2}} \|u\|_{\tilde{X}_{i}(G_{\alpha}^{i}\times\R^{2})}^{2},
\end{equation}
which implies that
\begin{align}
|\eqref{eq:ibs3_nl}| \lesssim \sum_{0\leq l_{2}\leq i-10} \left(2^{(l_{2}-i)/2}\|v_{0}\|_{L^{2}} \|u\|_{\tilde{X}_{i}(G_{\alpha}^{i}\times\R^{2})}^{2}\right)^{2} \lesssim \|v_{0}\|_{L^{2}}^{2}\|u\|_{\tilde{X}_{i}(G_{\alpha}^{i}\times\R^{2})}^{4}.
\end{align}

\item[Estimate for \eqref{eq:ibs3_err1}:]
Using the estimate for \eqref{eq:ibs2_err1} in the proof of proposition \ref{prop:ibs_2} on each subinterval $G_{\beta}^{l_{2}}\subset G_{\alpha}^{i}$ (with $v_{0}$ replaced by $P_{Q_{a}^{k}}v_{0})$ together with the fact that there are $2^{i-l_{2}}$ subintervals $G_{\beta}^{l_{2}}\subset G_{\alpha}^{i}$, we see that
\begin{align}
&\left|\int_{G_{\alpha}^{i}}\int_{\R^{4}}\frac{(x-y)}{|x-y|}\cdot |v_{a,k}(t,y)|^{2} \Re{\bar{\mathcal{N}}_{l_{2}}(\nabla-i\xi(t))w_{l_{2}}}(t,x) dxdydt\right| \\
&\phantom{=}\leq \sum_{G_{\beta}^{l_{2}}\subset G_{\alpha}^{i}} \left|\int_{G_{\beta}^{l_{2}}}\int_{\R^{4}}\frac{(x-y)}{|x-y|}\cdot |v_{a,k}(t,y)|^{2} \Re{\bar{\mathcal{N}}_{l_{2}}(\nabla-i\xi(t))w_{l_{2}}}(t,x)dxdydt\right| \nonumber \\
&\phantom{=}\lesssim 2^{i}\|P_{Q_{a}^{k}}v_{0}\|_{L^{2}}^{2}  \paren*{1+\|u\|_{\tilde{X}_{i}(G_{\alpha}^{i}\times\R^{2})}^{3}}.
\end{align}
So by Plancherel's theorem,
\begin{align}
\sup_{k\leq l_{2}+2}\left(\sum_{a\in\Z^{2}} 2^{k-2i}2^{i}\|P_{Q_{a}^{k}}v_{0}\|_{L^{2}}^{2} \paren*{1+\|u\|_{\tilde{X}_{i}(G_{\alpha}^{i}\times\R^{2})}^{3}}\right)^{1/2 } \lesssim 2^{(l_{2}-i)/2} \|v_{0}\|_{L^{2}} \paren*{1+\|u\|_{\tilde{X}_{i}(G_{\alpha}^{i}\times\R^{2})}^{3}}^{1/2},
\end{align}
which implies that
\begin{align}
|\eqref{eq:ibs3_err1}| \lesssim \sum_{0\leq l_{2}\leq i-10} 2^{l_{2}-i} \|v_{0}\|_{L^{2}}^{2}\paren*{1+\|u\|_{\tilde{X}_{i}(G_{\alpha}^{i}\times\R^{2})}^{3}} \lesssim \|v_{0}\|_{L^{2}}^{2}\paren*{1+\|u\|_{\tilde{X}_{i}(G_{\alpha}^{i}\times\R^{2})}^{3}}.
\end{align}

\item[Estimate for \eqref{eq:ibs3_err2}:]
Using the estimate for \eqref{eq:ibs2_err2} in the proof of proposition \ref{prop:ibs_2} on each subinterval $G_{\beta}^{l_{2}}\subset G_{\alpha}^{i}$ (with $v_{0}$ replaced by $P_{Q_{a}^{k}}v_{0}$) together with the fact that there are $2^{i-l_{2}}$ subintervals $G_{\beta}^{l_{2}}\subset G_{\alpha}^{i}$, we see that
\begin{align}
&\left|\int_{G_{\alpha}^{i}}\int_{\R^{4}} \frac{(x-y)}{|x-y|}\cdot |v_{a,k}(t,y)|^{2}\Re{\bar{w}_{l_{2}}(\nabla-i\xi(t))\mathcal{N}_{l_{2}}}(t,x)dxdydt\right| \nonumber\\
&\phantom{=} \leq \sum_{G_{\beta}^{l_{2}}\subset G_{\alpha}^{i}} \left|\int_{G_{\beta}^{l_{2}}}\int_{\R^{4}} \frac{(x-y)}{|x-y|}\cdot |v_{a,k}(t,y)|^{2}\Re{\bar{w}_{l_{2}}(\nabla-i\xi(t))\mathcal{N}_{l_{2}}}(t,x)dxdydt\right| \nonumber\\
&\phantom{=} \lesssim 2^{i}\|P_{Q_{a}^{k}}v_{0}\|_{L^{2}}^{2} \paren*{1+\|u\|_{\tilde{X}_{i}(G_{\alpha}^{i}\times\R^{2})}^{3}}.
\end{align}
So by Plancherel's theorem
\begin{align}
\sup_{k\leq l_{2}+2}\left(\sum_{a\in\Z^{2}} 2^{k-2i} 2^{i}\|P_{Q_{a}^{k}}v_{0}\|_{L^{2}}^{2} \paren*{1+\|u\|_{\tilde{X}_{i}(G_{\alpha}^{i}\times\R^{2})}^{3}}\right)^{1/2 } \lesssim 2^{(l_{2}-i)/2} \|v_{0}\|_{L^{2}} \paren*{1+\|u\|_{\tilde{X}_{i}(G_{\alpha}^{i}\times\R^{2})}^{3}}^{1/2},
\end{align}
which implies that
\begin{align}
|\eqref{eq:ibs3_err2}| \lesssim \sum_{0\leq l_{2}\leq i-10} \left(2^{(l_{2}-i)/2} \|v_{0}\|_{L^{2}} \paren*{1+\|u\|_{\tilde{X}_{i}(G_{\alpha}^{i}\times\R^{2})}^{3}}^{1/2 }\right)^{2} \lesssim \|v_{0}\|_{L^{2}}^{2}\paren*{1+\|u\|_{\tilde{X}_{i}(G_{\alpha}^{i}\times\R^{2})}^{3}}.
\end{align}

\item[Estimate for \eqref{eq:ibs3_err3}:]
As in the proof of proposition \ref{prop:ibs_2}, we first split $\mathcal{N}_{l_{2}}=\mathcal{N}_{1,l_{2}}+\mathcal{N}_{2,l_{2}}$ and estimate the contribution of $\mathcal{N}_{2,l_{2}}$. Using the estimate \eqref{eq:ibs2_err3_N2'} in the proof of proposition \ref{prop:ibs_2} on each subinterval $G_{\beta}^{l_{2}}\subset G_{\alpha}^{i}$ (with $v_{0}$ replaced by $P_{Q_{a}^{k}}v_{0}$), we obtain that
\begin{align}
&2^{k-2i} \left|\int_{G_{\alpha}^{i}}\int_{\R^{4}}\frac{(x-y)}{|x-y|}\cdot \Im{\bar{w}_{l_{2}}\mathcal{N}_{2,l_{2}}}(t,y)\Im{\bar{v}_{a,k}(\nabla-i\xi(t))v_{a,k}}(t,x)dxdydt\right| \nonumber\\
&\phantom{=} \leq \sum_{G_{\beta}^{l_{2}}\subset G_{\alpha}^{i}} 2^{k-2i} \left|\int_{G_{\beta}^{l_{2}}}\int_{\R^{4}}\frac{(x-y)}{|x-y|}\cdot \Im{\bar{w}_{l_{2}}\mathcal{N}_{2,l_{2}}}(t,y)\Im{\bar{v}_{a,k}(\nabla-i\xi(t))v_{a,k}}(t,x) dxdydt\right| \nonumber\\
&\phantom{=} \lesssim \sum_{G_{\beta}^{l_{2}}\subset G_{\alpha}^{i}} 2^{k-l_{2}-i}\|P_{Q_{a}^{k}}v_{0}\|_{L^{2}}^{2}\int_{G_{\beta}^{l_{2}}}|\xi'(t)| \|P_{\xi(t),l_{2}-3\leq\cdot\leq l_{2}+3}u(t)\|_{L_{x}^{2}}\sum_{0\leq l_{3}\leq l_{2}}2^{l_{3}-l_{2}}\|P_{\xi(t),l_{3}}u(t)\|_{L_{x}^{2}(\R^{2})}dt \nonumber\\
&\phantom{=} =2^{k-l_{2}-i}\|P_{Q_{a}^{k}}v_{0}\|_{L^{2}}^{2}\int_{G_{\alpha}^{i}}|\xi'(t)| \|P_{\xi(t),l_{2}-3\leq\cdot\leq l_{2}+3}u(t)\|_{L_{x}^{2}} \sum_{0\leq l_{3}\leq l_{2}}2^{l_{3}-l_{2}}\|P_{\xi(t),l_{3}}u(t)\|_{L_{x}^{2}}dt.
\end{align}
Now summing over $a\in\mathbb{Z}^{2}$ and using Plancherel's theorem, then taking the supremum over $k\leq l_{2}+2$ of the square root, we obtain that
\begin{equation}
\begin{split}
&\sup_{k\leq l_{2}+2}\left(\sum_{a\in\Z^{2}}\sum_{0\leq l_{3}\leq l_{2}}2^{k+l_{3}-2l_{2}-i}\|P_{Q_{a}^{k}}v_{0}\|_{L^{2}}^{2}\int_{G_{\alpha}^{i}} |\xi'(t)| \|P_{\xi(t),l_{2}-3\leq\cdot\leq l_{2}+3}u(t)\|_{L_{x}^{2}} \|P_{\xi(t), l_{3}}u(t)\|_{L_{x}^{2}}dt\right)^{1/2} \\
&\lesssim \left(\sum_{0\leq l_{3}\leq l_{2}} 2^{l_{3}-l_{2}-i}\|v_{0}\|_{L^{2}}^{2}\int_{G_{\alpha}^{i}} |\xi'(t)| \|P_{\xi(t),l_{2}-3\leq\cdot\leq l_{2}+3}u(t)\|_{L_{x}^{2}} \|P_{\xi(t), l_{3}}u(t)\|_{L_{x}^{2}}dt\right)^{1/2}.
\end{split}
\end{equation}
Observe that
\begin{equation}
\sup_{0\leq l_{3}\leq i-10}\sum_{0\leq l_{2}\leq i-10} 2^{l_{3}-l_{2}}1_{l_{3}\leq l_{2}} \lesssim 1, \qquad \sup_{0\leq l_{2}\leq i-10}\sum_{0\leq l_{3}\leq i-10} 2^{l_{3}-l_{2}}1_{l_{3}\leq l_{2}} \lesssim 1.
\end{equation}
So by Schur's test, Cauchy-Schwarz, almost orthogonality, Plancherel's theorem, and mass conservation, we see that
\begin{align}
&\sum_{0\leq l_{2}\leq i-10}\left(\sum_{0\leq l_{3}\leq l_{2}} 2^{l_{3}-l_{2}-i}\|v_{0}\|_{L^{2}}^{2}\int_{G_{\alpha}^{i}} |\xi'(t)| \|P_{\xi(t),l_{2}-3\leq\cdot\leq l_{2}+3}u(t)\|_{L_{x}^{2}} \|P_{\xi(t), l_{3}}u(t)\|_{L_{x}^{2}}dt\right) \nonumber\\
&\phantom{=} \lesssim 2^{-i}\|v_{0}\|_{L^{2}}^{2}\int_{G_{\alpha}^{i}}|\xi'(t)|dt \nonumber\\
&\phantom{=} \leq 2^{-i}\|v_{0}\|_{L^{2}}^{2}\int_{G_{\alpha}^{i}} 2^{-20}\epsilon_{1}^{-1/2}N(t)^{3}dt \nonumber\\
&\phantom{=} \leq 2^{-19}\epsilon_{3}^{1/2}\|v_{0}\|_{L^{2}}^{2}.
\end{align}
Hence,
\begin{equation}
\begin{split}
&\sum_{0\leq l_{2}\leq i-10}\left(\sup_{k\leq l_{2}+2}\left(\sum_{a\in\Z^{2}}2^{k-2i}\left|\int_{G_{\alpha}^{i}}\int_{\R^{4}}\frac{(x-y)}{|x-y|}\cdot \Im{\bar{w}_{l_{2}}\mathcal{N}_{2,l_{2}}}(t,y)\Im{\bar{v}_{a,k}(\nabla-i\xi(t))v_{a,k}}(t,x)dxdydt\right|\right)^{1/2}\right)^{2}\\
&\phantom{=} \lesssim \|v_{0}\|_{L^{2}}^{2},
\end{split}
\end{equation}
which completes the estimate for the contribution of $\mathcal{N}_{2,l_{2}}$ to \eqref{eq:ibs3_err3}.

We now estimate the contribution of $\mathcal{N}_{1,l_{2}}$ to $\eqref{eq:ibs3_err3}$. We follow the steps in the proof of proposition \ref{prop:ibs_2}. For each $0\leq l_{2}\leq i-10$, we first perform a near-far frequency decomposition $u=u_{l}+u_{h}$, where $u_{l} \coloneqq P_{\xi(t),\leq l_{2}-5}u$, and then decompose
\begin{equation}
\Im{\bar{w}_{l_{2}}\mathcal{N}_{1,l_{2}}} = F_{0,l_{2}} + F_{1,l_{2}} + F_{2,l_{2}} + F_{3,l_{2}} + F_{4,l_{2}},
\end{equation}
where $F_{j,l_{2}}$ denotes the terms containing $j$ factors $u_{h}$ and $4-j$ factors $u_{l}$ for $j=0,\ldots,3$.

We first consider the contribution of $F_{0,l_{2}}$. Fourier support analysis shows that $F_{0,l_{2}}=0$, hence there is no contribution to \eqref{eq:ibs3_err3}.

We next estimate the contribution of $F_{1,l_{2}}$. Further frequency decomposing $u_{h}$ by
\begin{equation}
u_{h}=P_{\xi(t),l_{2}-5<\cdot\leq l_{2}-2}u+P_{\xi(t),>l_{2}-2}u_{h}
\end{equation}
and then substituting this decomposition into $F_{1,l_{2}}$, we see that
\begin{equation}
\begin{split}
F_{1,l_{2}} &= \Im{\bar{u}_{l}P_{\xi(t),\leq l_{2}}\brak*{\E(|u_{l}|^{2})(P_{\xi(t), >l_{2}-2}u_{h})} + 2\bar{u}_{l}P_{\xi(t), \leq l_{2}}\brak*{\E\paren*{\Re{u_{l}(\ol{P_{\xi(t), > l_{2}-2}u_{h}})}}u_{l}}}\\
&\phantom{=}-\Im{\bar{u}_{l}\E(|u_{l}|^{2})(P_{\xi(t), l_{2}-2<\cdot\leq l_{2}}u_{h})},
\end{split}
\end{equation}
which has Fourier support in the region $\{|\xi|\geq 2^{l_{2}-4}\}$. By arguing as for obtaining the estimate \eqref{eq:ibs2_F1_est} (with $v_{0}$ replaced by $P_{Q_{a}^{k}}v_{0}$), we see that
\begin{align}
&2^{k-2i}\left|\int_{G_{\alpha}^{i}}\int_{\R^{4}}\frac{(x-y)}{|x-y|}\cdot F_{1,l_{2}}(t,y)\Im{\bar{v}_{a,k}(\nabla-i\xi(t))v_{a,k}}(t,x)dxdydt\right| \lesssim 2^{k-2l_{2}-i}\|P_{Q_{a}^{k}}v_{0}\|_{L^{2}}^{2}\|F_{1,l_{2}}\|_{L_{t,x}^{2}(G_{\alpha}^{i}\times\R^{2})}.
\end{align}
Summing over $a\in\mathbb{Z}^{2}$ and using Plancherel's theorem, then taking the supremum over $k\leq l_{2}+2$ of the square root, we obtain that
\begin{equation}
\begin{split}
\sup_{k\leq l_{2}+2}\paren*{\sum_{a\in\Z^{2}}  2^{k-2l_{2}-i}\|P_{Q_{a}^{k}}v_{0}\|_{L^{2}}^{2}\|F_{1,l_{2}}\|_{L_{t,x}^{2}(G_{\alpha}^{i}\times\R^{2})}}^{1/2} &\lesssim \paren*{2^{-l_{2}-i}\|v_{0}\|_{L^{2}}^{2} \|F_{1,l_{2}}\|_{L_{t,x}^{2}(G_{\alpha}^{i}\times\R^{2})}}^{1/2}.
\end{split}
\end{equation}
Now observe from the triangle, Minkowski's, and H\"{o}lder's inequalities together with some elementary Fourier support analysis that
\begin{align}
&\|F_{1,l_{2}}\|_{L_{t,x}^{2}(G_{\alpha}^{i}\times\R^{2})} \nonumber\\
&\phantom{=} = \paren*{\sum_{G_{\beta}^{l_{2}}\subset G_{\alpha}^{i}} \|F_{1,l_{2}}\|_{L_{t,x}^{2}(G_{\beta}^{l_{2}}\times\R^{2})}^{2}}^{1/2} \nonumber\\
&\phantom{=} \lesssim \paren*{\sum_{G_{\beta}^{l_{2}}\subset G_{\alpha}^{i}} \paren*{\|\E(|u_{l}|^{2})\|_{L_{t}^{\infty}L_{x}^{2}(G_{\beta}^{l_{2}}\times\R^{2})} \|u_{l}\|_{L_{t}^{4}L_{x}^{\infty}(G_{\beta}^{l_{2}}\times\R^{2})} \|P_{\xi(t),l_{2}-2<\cdot\leq l_{2}+2}u_{h}\|_{L_{t}^{4}L_{x}^{\infty}(G_{\beta}^{l_{2}}\times\R^{2})}}^{2}}^{1/2} \nonumber\\
&\phantom{=}\qquad +\paren*{\sum_{G_{\beta}^{l_{2}}\subset G_{\alpha}^{i}} \paren*{\|u_{l}\|_{L_{t,x}^{\infty}(G_{\beta}^{l_{2}}\times\R^{2})}\|u_{l}\|_{L_{t}^{\infty}L_{x}^{2}(G_{\beta}^{l_{2}}\times\R^{2})}\|\E P_{l_{2}-4\leq\cdot\leq l_{2}+4}\paren*{\Re{u_{l}(\ol{P_{\xi(t),l_{2}-2<\cdot\leq l_{2}+2}u_{h}})}}\|_{L_{t}^{2}L_{x}^{\infty}(G_{\beta}^{l_{2}}\times\R^{2})}}^{2}}^{1/2} \nonumber\\
&\lesssim 2^{l_{2}}\paren*{\sum_{G_{\beta}^{l_{2}}\subset G_{\alpha}^{i}} \|u_{l}\|_{L_{t}^{4}L_{x}^{\infty}(G_{\beta}^{l_{2}}\times\R^{2})}^{2} \|P_{\xi(t),l_{2}-2\leq\cdot\leq l_{2}+2}u\|_{L_{t}^{4}L_{x}^{\infty}(G_{\beta}^{l_{2}}\times\R^{2})}^{2}}^{1/2},
\end{align}
where we use another application of H\"{o}lder's inequality, Bernstein's lemma, and mass conservation to obtain the ultimate inequality. By lemma \ref{lem:BT_embed}, triangle inequality, and Bernstein's lemma, we have that
\begin{align}
\|u_{l}\|_{L_{t}^{4}L_{x}^{\infty}(G_{\beta}^{l_{2}}\times\R^{2})} \|P_{\xi(t),l_{2}-2\leq\cdot\leq l_{2}+2}u\|_{L_{t}^{4}L_{x}^{\infty}(G_{\beta}^{l_{2}}\times\R^{2})} \lesssim 2^{l_{2}/2}\|u\|_{\tilde{X}_{l_{2}}(G_{\beta}^{l_{2}}\times\R^{2})} \paren*{\sum_{l_{2}-2\leq l\leq l_{2}+2} 2^{l/2}\|P_{\xi(t),l}u\|_{L_{t,x}^{4}(G_{\beta}^{l_{2}}\times\R^{2})}}.
\end{align}
Taking the $\ell_{G_{\beta}^{l_{2}}}^{2}$ norm of the RHS and using the embedding $U_{\Delta}^{2}\subset L_{t,x}^{4}$, we see that
\begin{align}
&\paren*{\sum_{G_{\beta}^{l_{2}}\subset G_{\alpha}^{i}} \paren*{2^{l_{2}/2}\|u\|_{\tilde{X}_{l_{2}}(G_{\beta}^{l_{2}}\times\R^{2})} \paren*{\sum_{l_{2}-2\leq l\leq l_{2}+2} 2^{l/2}\|P_{\xi(t),l}u\|_{L_{t,x}^{4}(G_{\beta}^{l_{2}}\times\R^{2})}}}^{2}}^{1/2} \nonumber\\
&\phantom{=}\lesssim 2^{l_{2}}\|u\|_{\tilde{X}_{i}(G_{\alpha}^{i}\times\R^{2})} \sum_{l_{2}-2\leq l\leq l_{2}+2}\paren*{\sum_{G_{\beta}^{l}\subset G_{\alpha}^{i}} \|P_{\xi(G_{\beta}^{l}),l-2\leq\cdot\leq l+2}u\|_{U_{\Delta}^{2}(G_{\beta}^{l}\times\R^{2})}^{2}}^{1/2}.
\end{align}
Hence,
\begin{equation}
\begin{split}
\|F_{1,l_{2}}\|_{L_{t,x}^{2}(G_{\beta}^{l_{2}}\times\R^{2})} \lesssim 2^{2l_{2}}\|u\|_{\tilde{X}_{i}(G_{\alpha}^{i}\times\R^{2})} \sum_{l_{2}-2\leq l\leq l_{2}+2} \paren*{\sum_{G_{\beta}^{l}\subset G_{\alpha}^{i}} \|P_{\xi(G_{\beta}^{l}), l-2\leq\cdot\leq l+2}u\|_{U_{\Delta}^{2}(G_{\beta}^{l}\times\R^{2})}^{2}}^{1/2},
\end{split}
\end{equation}
and therefore by Cauchy-Schwarz,
\begin{align}
&\sum_{0\leq l_{2}\leq i-10} \paren*{2^{-l_{2}-i}\|v_{0}\|_{L^{2}}^{2} \|F_{1,l_{2}}\|_{L_{t,x}^{2}(G_{\alpha}^{i}\times\R^{2})}} \nonumber\\
&\phantom{=} \lesssim \|v_{0}\|_{L^{2}}^{2}\|u\|_{\tilde{X}_{i}(G_{\alpha}^{i}\times\R^{2})} \sum_{0\leq l_{2}\leq i-10} 2^{l_{2}-i} \sum_{l_{2}-2\leq l\leq l_{2}+2} \paren*{\sum_{G_{\beta}^{l}\subset G_{\alpha}^{i}} \|P_{\xi(G_{\beta}^{l}), l-2\leq\cdot\leq l+2}u\|_{U_{\Delta}^{2}(G_{\beta}^{l}\times\R^{2})}^{2}}^{1/2} \nonumber\\
&\phantom{=} \lesssim \|v_{0}\|_{L^{2}}^{2}\|u\|_{\tilde{X}_{i}(G_{\alpha}^{i}\times\R^{2})} \paren*{\sum_{0\leq l_{2}\leq i-8} 2^{l_{2}-i}\sum_{G_{\beta}^{l_{2}}\subset G_{\alpha}^{i}} \|P_{\xi(G_{\beta}^{l_{2}}),l_{2}-2\leq\cdot\leq l_{2}+2}u\|_{U_{\Delta}^{2}(G_{\beta}^{l_{2}}\times\R^{2})}^{2}}^{1/2} \nonumber\\
&\leq \|v_{0}\|_{L^{2}}^{2} \|u\|_{\tilde{X}_{i}(G_{\alpha}^{i}\times\R^{2})}^{2},
\end{align} 
where the ultimate inequality follows from the definition of the $X(G_{\alpha}^{i}\times\R^{2})$ norm. Thus, we have shown that
\begin{equation}
\begin{split}
&\sum_{0\leq l_{2}\leq i-10} \left(\sup_{k\leq l_{2}+2}\left(\sum_{a\in\Z^{2}}2^{k-2i}\left|\int_{G_{\alpha}^{i}}\int_{\R^{4}}\frac{(x-y)}{|x-y|}\cdot F_{1,l_{2}}(t,y)\Im{\bar{v}_{a,k}(\nabla-i\xi(t))v_{a,k}}(t,x)dxydt\right|\right)^{1/2}\right)^{2}\\
&\phantom{=} \lesssim \|v_{0}\|_{L^{2}}^{2}\|u\|_{\tilde{X}_{i}(G_{\alpha}^{i}\times\R^{2})}^{2},
\end{split}
\end{equation}
which is an acceptable estimate for the contribution $F_{1,l_{2}}$.

We now estimate the contributions of the remaining terms $F_{2,l_{2}}, F_{3,l_{2}}, F_{4,l_{2}}$. Hereafter, the mixed norm notation $L_{t}^{p}L_{x}^{q}$ is always taken over the spacetime slab $G_{\alpha}^{i}\times\R^{2}$. Taking absolute values and using Cauchy-Schwarz and Plancherel's theorem, we see that
\begin{equation}
\begin{split}
&2^{k-2i} \left|\int_{G_{\alpha}^{i}}\int_{\R^{4}}\frac{(x-y)}{|x-y|}\cdot \paren*{F_{2,l_{2}}+F_{3,l_{2}}+F_{4,l_{2}}}(t,y)\Im{\bar{v}_{a,k}(\nabla-i\xi(t))v_{a,k}}(t,x)dxdydt\right|\\
&\phantom{=} \lesssim 2^{k-i}\|F_{2,l_{2}}+F_{3,l_{2}}+F_{4,l_{2}}\|_{L_{t,x}^{1}} \|P_{Q_{a}^{k}}v_{0}\|_{L^{2}}^{2}.
\end{split}
\end{equation}
By Plancherel's theorem,
\begin{align}
&\sum_{0\leq l_{2}\leq i-10}\paren*{\sup_{k\leq l_{2}+2}\left(\sum_{a\in\Z^{2}} 2^{k-i}\|F_{2,l_{2}}+F_{3,l_{2}}+F_{4,l_{2}}\|_{L_{t,x}^{1}} \|P_{Q_{a}^{k}}v_{0}\|_{L^{2}}^{2}\right)^{1/2}}^{2} \lesssim \sum_{0\leq l_{2}\leq i-10}\left(2^{l_{2}-i}\|F_{2,l_{2}}+F_{3,l_{2}}+F_{4,l_{2}}\|_{L_{t,x}^{1}} \|v_{0}\|_{L^{2}}^{2}\right).
\end{align}
We now use the triangle inequality to break up the preceding expression into three terms, each of which we estimate separately.

We first estimate the $F_{4,l_{2}}$ term. By H\"{o}lder's inequality and Calder\'{o}n-Zygmund theorem,
\begin{equation}
\sum_{0\leq l_{2}\leq i-10} 2^{l_{2}-i} \|F_{4,l_{2}}\|_{L_{t,x}^{1}} \lesssim \sum_{0\leq l_{2}\leq i-10}2^{l_{2}-i} \|P_{\xi(t),\l_{2}-5\leq\cdot\leq i}u\|_{L_{t,x}^{4}}^{4} + \sum_{0\leq l_{2}\leq i-10} 2^{l_{2}-i} \|P_{\xi(t),\geq i}u\|_{L_{t,x}^{4}}^{4} \eqqcolon \mathrm{Term}_{1}+\mathrm{Term}_{2}.
\end{equation}
Lemma \ref{lem:lohi_embed} implies the $\mathrm{Term}_{2}$ estimate
\begin{equation}
\mathrm{Term}_{2} \lesssim \|u\|_{\tilde{X}_{i}(G_{\alpha}^{i}\times\R^{2})}^{4}.
\end{equation}
To estimate $\mathrm{Term}_{1}$, we first write $P_{\xi(t),\leq l_{2}-5\leq\cdot\leq i}u=\sum_{l_{3}=l_{2}-5}^{i}P_{\xi(t),l_{3}}u$ and then algebraically expand the $L_{t,x}^{4}$ norm as a sum to exploit almost orthogonality between the Littlewood-Paley projections. We then apply the triangle inequality, Plancherel's theorem (which tells us that the two highest frequencies relative to $\xi(t)$ are comparable), followed by H\"{o}lder's inequality to obtain
\begin{align}
&\sum_{0\leq l_{2}\leq i-10} 2^{l_{2}-i} \|P_{\xi(t),l_{2}-5\leq\cdot\leq i}u\|_{L_{t,x}^{4}}^{4} \nonumber\\
&\phantom{=} \lesssim \sum_{0\leq l_{2}\leq i-10}2^{l_{2}-i}\sum_{{l_{2}-5\leq j_{4}\leq j_{3}\leq j_{2}\leq j_{1} \leq i}\atop {|j_{1}-j_{2}|\leq 3}} \|P_{\xi(t), j_{4}}u\|_{L_{t}^{\infty}L_{x}^{2}} \|P_{\xi(t), j_{3}}u\|_{L_{t}^{3}L_{x}^{6}} \|P_{\xi(t),j_{2}}u\|_{L_{t}^{3}L_{x}^{6}} \|P_{\xi(t), j_{1}}u\|_{L_{t}^{3}L_{x}^{6}} \nonumber\\
&\phantom{=} \lesssim \sum_{{0\leq j_{4}\leq j_{3}\leq j_{2}\leq j_{1}\leq i}\atop {|j_{1}-j_{2}|\leq 3}} 2^{j_{4}-i} \|P_{\xi(t), j_{4}}u\|_{L_{t}^{\infty}L_{x}^{2}} \|P_{\xi(t), j_{3}}u\|_{L_{t}^{3}L_{x}^{6}} \|P_{\xi(t),j_{2}}u\|_{L_{t}^{3}L_{x}^{6}} \|P_{\xi(t), j_{1}}u\|_{L_{t}^{3}L_{x}^{6}},
\end{align}
where the ultimate line follows from interchanging the order of the $l_{2}$ summation and the $j_{4}$ summation. By mass conservation, the above is
\begin{align}
&\lesssim \sum_{{0\leq j_{3}\leq j_{2}\leq j_{1}\leq i}\atop{|j_{1}-j_{2}|\leq 3}} 2^{j_{3}-i}\|P_{\xi(t), j_{3}}u\|_{L_{t}^{3}L_{x}^{6}} \|P_{\xi(t),j_{2}}u\|_{L_{t}^{3}L_{x}^{6}} \|P_{\xi(t), j_{1}}u\|_{L_{t}^{3}L_{x}^{6}} \nonumber\\
&=\sum_{-3\leq j\leq 3} \sum_{0\leq j_{3}\leq j_{2}\leq i} 2^{j_{3}-i}\|P_{\xi(t), j_{3}}u\|_{L_{t}^{3}L_{x}^{6}} \|P_{\xi(t),j_{2}}u\|_{L_{t}^{3}L_{x}^{6}} \|P_{\xi(t), j_{2}+j}u\|_{L_{t}^{3}L_{x}^{6}} \nonumber\\
&\leq \sum_{-3\leq j\leq 3}\left(\sum_{0\leq j_{2}\leq i}\left(\sum_{0\leq j_{3}\leq j_{2}\leq i} 2^{j_{3}-\frac{j_{2}}{3}-\frac{2i}{3}}\|P_{\xi(t), j_{3}}u\|_{L_{t}^{3}L_{x}^{6}} \|P_{\xi(t),j_{2}}u\|_{L_{t}^{3}L_{x}^{6}}\right)^{3/2}\right)^{2/3}\left(\sum_{0\leq j_{2}\leq i} 2^{j_{2}-i}\|P_{\xi(t),j_{2}+j}u\|_{L_{t}^{3}L_{x}^{6}}^{3}\right)^{1/3}, 
\end{align}
where we use Holder's inequality in $j_{2}$ to obtain the ultimate inequality. By Young's inequality on $\ell_{j_{2}}^{3}$ applied to the kernel $K(j_{3},j_{2}) := 2^{2(j_{3}-j_{2})/3}1_{\geq 0}(j_{2}-j_{3})$,
\begin{equation}
\left(\sum_{0\leq j_{2}\leq i}\left(\sum_{0\leq j_{3}\leq j_{2}}2^{j_{3}-\frac{2j_{2}}{3}-\frac{i}{3}}\|P_{\xi(t),j_{3}}u\|_{L_{t}^{3}L_{x}^{6}}\right)^{3}\right)^{1/3} \lesssim \left(\sum_{0\leq j_{2}\leq i} 2^{j_{2}-i}\|P_{\xi(t),j_{2}}u\|_{L_{t}^{3}L_{x}^{6}}^{3}\right)^{1/3},
\end{equation}
and so by H\"{o}lder's inequality applied to the $\ell_{j_{2}}^{3/2}$ norm,
\begin{equation}
\begin{split}
\left(\sum_{0\leq j_{2}\leq i}\left(\sum_{0\leq j_{3}\leq j_{2}} 2^{j_{3}-\frac{j_{2}}{3}-\frac{2i}{3}}\|P_{\xi(t), j_{3}}u\|_{L_{t}^{3}L_{x}^{6}} \|P_{\xi(t),j_{2}}u\|_{L_{t}^{3}L_{x}^{6}}\right)^{3/2}\right)^{2/3}\lesssim \left(\sum_{0\leq j_{2}\leq i} 2^{j_{2}-i} \|P_{\xi(t),j_{2}}u\|_{L_{t}^{3}L_{x}^{6}}^{3}\right)^{2/3}.
\end{split}
\end{equation}
Using the embedding $U_{\Delta}^{2}\subset L_{t}^{3}L_{x}^{6}$, we see that
\begin{align}
\sum_{0\leq j_{1}\leq i} 2^{j_{1}-i} \|P_{\xi(t),j_{1}}u\|_{L_{t}^{3}L_{x}^{6}(G_{\alpha}^{i}\times\R^{2})}^{3} &= \sum_{0\leq j_{1}\leq i}2^{j_{1}-i} \sum_{G_{\beta}^{j_{1}}\subset G_{\alpha}^{i}} \|P_{\xi(t),j_{1}}u\|_{L_{t}^{3}L_{x}^{6}(G_{\beta}^{j_{1}}\times\R^{2})}^{3} \nonumber\\
&\lesssim \sum_{0\leq j_{1}\leq i}\sum_{G_{\beta}^{j_{1}}\subset G_{\alpha}^{i}} \|P_{\xi(G_{\beta}^{j_{1}}), j_{1}-2\leq\cdot\leq j_{1}+2}u\|_{U_{\Delta}^{2}(G_{\beta}^{j_{1}}\times\R^{2})}^{3} \nonumber\\
&\lesssim \|u\|_{\tilde{X}_{i}(G_{\alpha}^{i}\times\R^{2})}\sum_{0\leq j_{1}\leq i}2^{j_{1}-i}\sum_{G_{\beta}^{j_{1}}\subset G_{\alpha}^{i}}  \|P_{\xi(G_{\beta}^{j_{1}}), j_{1}-2\leq\cdot\leq j_{1}+2}u\|_{U_{\Delta}^{2}(G_{\beta}^{j_{1}}\times\R^{2})}^{2} \nonumber\\
&\lesssim \|u\|_{\tilde{X}_{i}(G_{\alpha}^{i}\times\R^{2})}^{3},
\end{align}
where we use the definition of the $\tilde{X}_{i}(G_{\alpha}^{i}\times\R^{2})$ norm to obtain the last two inequalities. Hence, we have shown that $\mathrm{Term}_{1} \lesssim \|u\|_{\tilde{X}_{i}(G_{\alpha}^{i}\times\R^{2})}^{3}$, and we conclude that
\begin{equation}
\sum_{0\leq l_{2}\leq i-10}2^{l_{2}-i}\|F_{4,l_{2}}\|_{L_{t,x}^{1}}\|v_{0}\|_{L^{2}}^{2} \lesssim \paren*{1+\|u\|_{\tilde{X}_{i}(G_{\alpha}^{i}\times\R^{2})}}\|u\|_{\tilde{X}_{i}(G_{\alpha}^{i}\times\R^{2})}^{3}\|v_{0}\|_{L^{2}}^{2}.
\end{equation}

We next estimate the $F_{2,l_{2}}$ term. As all the other cases are strictly easier or follow by completely analogous arguments, we only present the details for the case where both factors $u_{l}$ fall inside the argument of the operator $\E$. Thus, to estimate $\|F_{2,l_{2}}\|_{L_{t,x}^{1}}$, it suffices for us to estimate the quantity
\begin{equation}
\|\Im{(\ol{P_{\xi(t),\leq l_{2}}u_{h}})P_{\xi(t),\leq l_{2}}\brak*{\E(|u_{l}|^{2})u_{h})}}\|_{L_{t,x}^{1}}.
\end{equation}
%In principle, we could use the double frequency decomposition used throughout this paper in order to apply bilinear Strichartz estimates. Instead, we give here a different argument which exploits the Schwartz kernel of the operator $\frac{\partial_{1}^{2}}{\Delta}$ together with the presence of the imaginary operator. Our motivation for doings so is as follows. If we re-examine the application of proposition \ref{prop:ibs_2} in the preceding section, then we see that the full strength was not used: we never exploited a factor of $2^{\frac{k}{2}}$ hidden in the estimate \eqref{eq:LTSE_ibs2_app}. Rather, to sum over the range $k\leq l_{2}+2$, we only used the factor of $2^{(k-i)\frac{1}{2}^{-}}$ coming from the application of the bilinear Strichartz estimate of proposition \ref{prop:bs_i} used to obtain the estimate \eqref{eq:LTSE_ibs3_comm}. However, in the sequel, we will apply proposition \ref{prop:ibs_2} to both high-low pairs and therefore need the full strength of proposition \ref{prop:ibs_2} (i.e. the factor of $2^{\frac{k-i}{2}}$ we previously did not use) in order to the eliminate the summation over $k\leq l_{2}+2$. We now turn to the details.

Recall from subsection \ref{ssec:prelim_ha} that the operator $\E$ admits the decomposition
\begin{equation}
\E = cId+(W_{\Omega}\ast \cdot),
\end{equation}
where $c\in\R$ and $W_{\Omega}$ is the principal value distribution
\begin{equation}
\ipp{W_{\Omega},f} = \lim_{\epsilon\rightarrow 0}\int_{|x|\geq \epsilon} \frac{\Omega(x/|x|)}{|x|^{2}}f(x)dx.
\end{equation}
By the triangle inequality,
\begin{align}
\|\Im{ (\ol{P_{\xi(t),\leq l_{2}}u_{h}})P_{\xi(t),\leq l_{2}}\paren*{\E(|u_{l}|^{2})u_{h}}}\|_{L_{t,x}^{1}} &\lesssim \|(\ol{P_{\xi(t),\leq l_{2}}u_{h}})P_{\xi(t),\leq l_{2}}\paren*{|u_{l}|^{2}u_{h}}\|_{L_{t,x}^{1}} \nonumber\\
&\phantom{=} +\|\Im{(\ol{P_{\xi(t),\leq l_{2}}u_{h}})P_{\xi(t),\leq l_{2}}\paren*{(W_{\Omega}\ast |u_{l}|^{2})u_{h}}}\|_{L_{t,x}^{1}} \nonumber\\
&\eqqcolon \mathrm{Term}_{1}+\mathrm{Term}_{2}.
\end{align}
Consider $\mathrm{Term}_{2}$. Let $0\leq\chi\leq 1$ be a $C^{\infty}$ bump function which is supported on the ball $B(0,2)$ and identically one on the ball $B(0,1)$, and write
\begin{align}
(W_{\Omega}\ast |u_{l}|^{2})(x) &= \lim_{\epsilon\rightarrow 0}\int_{|y|\geq\epsilon} \chi(2^{l_{2}}y)\mathcal{K}(y)|u_{l}(x-y)|^{2}dy + \int_{\R^{2}}\left(1-\chi(2^{l_{2}}y)\right) \mathcal{K}(y)|u_{l}(x-y)|^{2}dy \nonumber\\
&\eqqcolon (\mathcal{K}_{loc}\ast |u_{l}|^{2})(x) + (\mathcal{K}_{glob}\ast |u_{l}|^{2})(x), \qquad \forall x\in\R^{2},
\end{align}
where we have introduced the notation $\mathcal{K}(x) \coloneqq \frac{\Omega(x/|x|)}{|x|^{2}}$ above. Note that since $u_{l}$ is $C^{\infty}$, the limit defining $\K_{loc}\ast|u_{l}|^{2}$ exists pointwise everywhere. Now substituting this decomposition into $\mathrm{Term}_{2}$ and using that $\mathcal{K}$ is real-valued to add zero, we obtain the pointwise identity
\begin{equation}
\begin{split}
\Im{(\ol{P_{\xi(t),\leq l_{2}}u_{h}})P_{\xi(t),\leq l_{2}}\paren*{(W_{\Omega}\ast |u_{l}|^{2})u_{h}}} &= \Im{(\ol{(P_{\xi(t),\leq l_{2}}u_{h}}) \comm{P_{\xi(t),\leq l_{2}}}{(\mathcal{K}_{loc}\ast |u_{l}|^{2})}u_{h}} \\
&\phantom{=}+\Im{(\ol{P_{\xi(t),\leq l_{2}}u_{h}}) \comm{P_{\xi(t),\leq l_{2}}}{(\mathcal{K}_{glob}\ast |u_{l}|^{2})}u_{h}}.
\end{split}
\end{equation}
By the standard commutator argument,
\begin{align}
\paren*{\comm{P_{\xi(t),\leq l_{2}}}{(\mathcal{K}_{loc}\ast |u_{l}|^{2})}u_{h}}(x) &= \int_{0}^{1}\int_{\R^{2}}2^{2l_{2}}\phi^{\vee}(2^{l_{2}}y)(-y)\cdot e^{-i(x-y)\cdot\xi(t)} (\tau_{y}u_{h})(x) \paren*{\mathcal{K}_{loc}\ast \nabla (|u_{l}|^{2})}(x-\theta y)dy d\theta.
\end{align}
Since $\K_{loc}(z)=\K(z)$ for $|z|\leq 2^{-l_{2}-1}$, we can write
\begin{align}
\paren*{\K_{loc}\ast \nabla(|u_{l}|^{2})}(x-\theta y) &= \int_{|z|>2^{-l_{2}-1}} \chi(2^{l_{2}}z)\K(z)(\nabla |u_{l}|^{2})(x-\theta y-z)dz  \nonumber\\
&\phantom{=} + \lim_{\epsilon\rightarrow 0} \int_{\epsilon\leq |z|\leq 2^{-l_{2}-1}} \K(z) (\nabla |u_{l}|^{2})(x-\theta y-z)dz \nonumber\\
&\eqqcolon \paren*{\K_{loc,1} \ast \nabla(|u_{l}|^{2})}(x-\theta y) + \paren*{\K_{loc,2}\ast \nabla (|u_{l}|^{2})}(x-\theta y).
\end{align}
By dilation invariance, $\|1_{>2^{-l_{2}-1}}\K_{loc}\|_{L^{1}(\R^{2})} \lesssim 1$. Therefore by the reproducing identity $\nabla(|u_{l}|^{2}) = (\nabla P_{\leq l_{2}})(|u_{l}|^{2})$, two applications of Minkowski's inequality, followed by Cauchy-Schwarz, we see that
\begin{align}
\|\Im{(\ol{P_{\xi(t),\leq l_{2}}u_{h}}) \comm{P_{\xi(t),\leq l_{2}}}{(\mathcal{K}_{loc,1}\ast |u_{l}|^{2})}u_{h}}\|_{L_{t,x}^{1}} &\lesssim 2^{-l_{2}}\int_{0}^{1}d\theta \sup_{y\in\R^{2}} \|(P_{\xi(t),\leq l_{2}}u_{h}) (\tau_{y}P_{\xi(t),\leq l_{2}+1}u_{h}) \tau_{\theta y}(P_{\leq l_{2}}\nabla |u_{l}|^{2}) \|_{L_{t,x}^{1}} \nonumber\\
&\lesssim \sup_{y\in\R^{2}} \|(P_{\xi(t),\leq l_{2}+1}u_{h})(\tau_{y}u_{l})\|_{L_{t,x}^{2}}^{2}.
\end{align}
Since $\Omega$ has mean-value zero on $\S^{1}$, $\int_{r_{1}<|z|<r_{2}}\K(z)dz=0$ for any $0<r_{1}<r_{2}<\infty$. Hence, we can write
\begin{align}
\paren*{\K_{loc,2}\ast \nabla(|u_{l}|^{2})}(x-\theta y) &= \lim_{\epsilon\rightarrow 0}\int_{\epsilon\leq |z| \leq 2^{-l_{2}-1}} \K(z) \paren*{(\nabla |u_{l}|^{2})(x-\theta y-z) - (\nabla |u_{l}|^{2})(x-\theta y)}dz \nonumber\\
&= \lim_{\epsilon\rightarrow 0} \int_{\epsilon \leq |z| \leq 2^{-l_{2}-1}} \K(z)\paren*{\int_{0}^{1} (\nabla^{2} |u_{l}|^{2})(x-\theta y-\theta'z) \cdot (-z)d\theta'}dz,
\end{align}
where the ultimate equality follows from the fundamental theorem of calculus. Using that $|\K(z)| \lesssim |z|^{-2}$, we see from Minkowski's inequality and dominated convergence that
\begin{align}
&\left|\lim_{\epsilon\rightarrow 0} \int_{\epsilon \leq |z| \leq 2^{-l_{2}-1}} \K(z)\paren*{\int_{0}^{1} (\nabla^{2} |u_{l}|^{2})(x-\theta y-\theta' z)\cdot (-z)d\theta'}dz\right| \nonumber\\
&\phantom{=} \lesssim \int_{0}^{1}\int_{|z|\leq 2^{-l_{2}-1}} \frac{1}{|z|} |(P_{\leq l_{2}}\nabla^{2}|u_{l}|^{2})(x-\theta y-\theta' z)|dzd\theta
\end{align}
Since $\|1_{\leq 2^{-l_{2}-1}}|z|^{-1}\|_{L^{1}(\R^{2})} \lesssim 2^{-l_{2}}$, we see again from Minkowski's inequality together with Cauchy-Schwarz that
\begin{align}
&\|\Im{(\ol{P_{\xi(t),\leq l_{2}}u_{h}}) \comm{P_{\xi(t),\leq l_{2}}}{(\mathcal{K}_{loc,2}\ast |u_{l}|^{2})}u_{h}}\|_{L_{t,x}^{1}} \nonumber\\
&\phantom{=} \lesssim 2^{-2l_{2}}\int_{0}^{1}\int_{0}^{1}d\theta d\theta' \sup_{y,z\in\R^{2}} \|(P_{\xi(t),\leq l_{2}}u_{h})(\tau_{y}P_{\xi(t),\leq l_{2}+1}u_{h})\tau_{\theta y}\tau_{\theta' z} (P_{\leq l_{2}}\nabla^{2}|u_{l}|^{2})\|_{L_{t,x}^{1}} \nonumber\\
&\phantom{=} \lesssim \sup_{y\in\R^{2}} \|(P_{\xi(t),\leq l_{2}+1}u_{h})(\tau_{y}u_{l})\|_{L_{t,x}^{2}}^{2}
\end{align}
By the same commutator argument used above,
\begin{align}
\paren*{\comm{P_{\xi(t),\leq l_{2}}}{(\K_{glob}\ast |u_{l}|^{2})}u_{h}}(x) &= \int_{0}^{1}\int_{\R^{2}}2^{2l_{2}}\phi^{\vee}(2^{l_{2}}y)(-y)\cdot e^{-i(x-y)\cdot\xi(t)}u_{h}(x-y)(\nabla \K_{glob}\ast |u_{l}|^{2})(x-\theta y)dyd\theta.
\end{align}
Since $|(\nabla \K_{glob})(z)|\lesssim |z|^{-3}$ for $|z|\gtrsim 2^{-l_{2}}$, dilation invariance implies that $\|\nabla \K_{glob}\|_{L^{1}} \lesssim 2^{l_{2}}$. Hence by two applications of Minkowski's inequality followed by Cauchy-Schwarz, we see that
\begin{equation}
\| \Im{(\ol{P_{\xi(t),\leq l_{2}}u_{h}})\comm{P_{\xi(t),\leq l_{2}}}{(\mathcal{K}_{glob}\ast |u_{l}|^{2})}u_{h}}\|_{L_{t,x}^{1}} \lesssim \sup_{y\in\R^{2}} \|(P_{\xi(t),\leq l_{2}+1}u_{h})(\tau_{y}u_{l})\|_{L_{t,x}^{2}}^{2}.
\end{equation}
Therefore,
\begin{equation}
\mathrm{Term}_{1}+\mathrm{Term}_{2} \lesssim \sup_{y\in\R^{2}} \|(P_{\xi(t),\leq l_{2}+1}u_{h})(\tau_{y}u_{l})\|_{L_{t,x}^{2}}^{2}.
\end{equation}
In the sequel, we suppress the spatial translation $\tau_{y}$ in the preceding expression, as all of our estimates are uniform in the parameter $y$.

Next, observe that
\begin{align}
&\sum_{0\leq l_{2}\leq i-10} 2^{l_{2}-i}\|(P_{\xi(t),\leq l_{2}+1}u_{h})u_{l}\|_{L_{t,x}^{2}}^{2} \nonumber\\
&\phantom{=} \lesssim \sum_{0\leq l_{2}\leq i-10} 2^{l_{2}-i}\|(P_{\xi(t),\leq l_{2}+1}u_{h}) (P_{\xi(t),\leq l_{2}-15}u)\|_{L_{t,x}^{2}}^{2} + \sum_{0\leq l_{2}\leq i-10}2^{l_{2}-i}\|(P_{\xi(t),\leq l_{2}+1}u_{h})(P_{\xi(t), l_{2}-15<\cdot\leq l_{2}-5}u)\|_{L_{t,x}^{2}}^{2} \nonumber\\
&\phantom{=} \eqqcolon \mathrm{Term}_{1}+\mathrm{Term}_{2}.
\end{align}

To estimate $\mathrm{Term}_{2}$, we use triangle and H\"{o}lder's inequality to obtain
\begin{align}
\|(P_{\xi(t),\leq l_{2}+1}u_{h})(P_{\xi(t), l_{2}-15<\cdot\leq l_{2}-5}u)\|_{L_{t,x}^{2}(G_{\alpha}^{i}\times\R^{2})}^{2} &=\sum_{G_{\beta}^{l_{2}}\subset G_{\alpha}^{i}} \|(P_{\xi(t),\leq l_{2}+1}u_{h})(P_{\xi(t), l_{2}-15<\cdot\leq l_{2}-5}u)\|_{L_{t,x}^{2}(G_{\beta}^{l_{2}}\times\R^{2})}^{2} \nonumber\\
&\lesssim \sum_{{l_{2}-5\leq l_{3}\leq l_{2}+1}\atop {l_{2}-15\leq l_{4}\leq l_{2}-5}} \|(P_{\xi(t),l_{3}}u)(P_{\xi(t),l_{4}}u)\|_{L_{t,x}^{2}(G_{\beta}^{l_{2}}\times\R^{2})}^{2} \nonumber\\
&\leq \sum_{G_{\beta}^{l_{2}}\subset G_{\alpha}^{i}} \sum_{{l_{2}-5\leq l_{3}\leq l_{2}}\atop{l_{2}-15\leq l_{4}\leq l_{2}-5}} \|P_{\xi(t),l_{3}}u\|_{L_{t,x}^{4}(G_{\beta}^{l_{2}}\times\R^{2})}^{2} \|P_{\xi(t),l_{4}}u\|_{L_{t,x}^{4}(G_{\beta}^{l_{2}}\times\R^{2})}^{2} \nonumber\\
&\phantom{=} + \sum_{G_{\beta}^{l_{2}}\subset G_{\alpha}^{i}}\sum_{l_{2}-15\leq l_{4}\leq l_{2}-5} \|P_{\xi(t),l_{2}+1}u\|_{L_{t,x}^{4}(G_{\beta}^{l_{2}}\times\R^{2})}^{2} \|P_{\xi(t),l_{4}}u\|_{L_{t,x}^{4}(G_{\beta}^{l_{2}}\times\R^{2})}^{2}.
\end{align}
By lemma \ref{lem:lohi_embed} and that $2^{l_{4}}\sim 2^{l_{2}}$,
\begin{equation}
\begin{split}
&\sum_{G_{\beta}^{l_{2}}\subset G_{\alpha}^{i}}\sum_{l_{2}-15\leq l_{4}\leq l_{2}-5} \|P_{\xi(t),l_{2}+1}u\|_{L_{t,x}^{4}(G_{\beta}^{l_{2}}\times\R^{2})}^{2} \|P_{\xi(t),l_{4}}u\|_{L_{t,x}^{4}(G_{\beta}^{l_{2}}\times\R^{2})}^{2} \\
&\phantom{=} \lesssim \|u\|_{\tilde{X}_{i}(G_{\alpha}^{i}\times\R^{2})}^{2}\sum_{G_{\beta}^{l_{2}}\subset G_{\alpha}^{i}} \|P_{\xi(t),l_{2}+1}u\|_{L_{t,x}^{4}(G_{\beta}^{l_{2}}\times\R^{2})}^{2}.
\end{split}
\end{equation}
Now let $G_{\gamma(\beta)}^{l_{2}+1}$ be the unique parent of $G_{\beta}^{l_{2}}$ so that by Bernstein's lemma and the embedding $U_{\Delta}^{2}\subset L_{t,x}^{4}$,
\begin{equation}
\|P_{\xi(t),l_{2}+1}u\|_{L_{t,x}^{4}(G_{\beta}^{l_{2}}\times\R^{2})} \lesssim \|P_{\xi(G_{\gamma(\beta)}^{l_{2}}),l_{2}-1\leq\cdot\leq l_{2}+3}u\|_{U_{\Delta}^{2}(G_{\gamma(\beta)}^{l_{2}+1}\times\R^{2})}.
\end{equation}
Now by a change of variable, we conclude that
\begin{align}
\sum_{0\leq l_{2}\leq i-10}2^{l_{2}-i}  \|u\|_{\tilde{X}_{i}(G_{\alpha}^{i}\times\R^{2})}^{2}\sum_{G_{\beta}^{l_{2}}\subset G_{\alpha}^{i}} \|P_{\xi(t),l_{2}+1}u\|_{L_{t,x}^{4}(G_{\beta}^{l_{2}}\times\R^{2})}^{2}  \lesssim \|u\|_{\tilde{X}_{i}(G_{\alpha}^{i}\times\R^{2})}^{4}.
\end{align}
Next, again using lemma \ref{lem:lohi_embed} and that $2^{l_{4}}\sim 2^{l_{2}}$, we see that
\begin{equation}
\begin{split}
&\sum_{G_{\beta}^{l_{2}}\subset G_{\alpha}^{i}} \sum_{{l_{2}-5\leq l_{3}\leq l_{2}}\atop{l_{2}-15\leq l_{4}\leq l_{2}-5}} \|P_{\xi(t),l_{3}}u\|_{L_{t,x}^{4}(G_{\beta}^{l_{2}}\times\R^{2})}^{2} \|P_{\xi(t),l_{4}}u\|_{L_{t,x}^{4}(G_{\beta}^{l_{2}}\times\R^{2})}^{2} \\
&\phantom{=} \lesssim \|u\|_{\tilde{X}_{i}(G_{\alpha}^{i}\times\R^{2})}^{2}\sum_{G_{\beta}^{l_{2}}\subset G_{\alpha}^{i}} \sum_{l_{2}-5\leq l_{3}\leq l_{2}} \|P_{\xi(t),l_{3}}u\|_{L_{t,x}^{4}(G_{\beta}^{l_{2}}\times\R^{2})}^{2}.
\end{split}
\end{equation}
Using the embedding $\ell^{2}\subset \ell^{4}$, we see that
\begin{align}
\|P_{\xi(t),l_{3}}u\|_{L_{t,x}^{4}(G_{\beta}^{l_{2}}\times\R^{2})}^{2} = \paren*{\sum_{G_{\gamma}^{l_{3}}\subset G_{\beta}^{l_{2}}} \|P_{\xi(t),l_{3}}u\|_{L_{t,x}^{4}(G_{\gamma}^{l_{3}}\times\R^{2})}^{4}}^{1/2} &\leq \sum_{G_{\gamma}^{l_{3}}\subset G_{\beta}^{l_{2}}} \|P_{\xi(t),l_{3}}u\|_{L_{t,x}^{4}(G_{\gamma}^{l_{3}}\times\R^{2})}^{2} \nonumber\\
&\lesssim \sum_{G_{\gamma}^{l_{3}}\subset G_{\beta}^{l_{2}}} \|P_{\xi(G_{\gamma}^{l_{3}}),l_{3}-2\leq\cdot\leq l_{3}+2}u\|_{U_{\Delta}^{2}(G_{\gamma}^{l_{3}}\times\R^{2})}^{2},
\end{align}
where we use Bernstein's lemma followed by the embedding $U_{\Delta}^{2}\subset L_{t,x}^{4}$ to obtain the ultimate inequality. It now follows from $2^{l_{3}}\sim 2^{l_{2}}$ and a change of variable that
\begin{equation}
\sum_{0\leq l_{2}\leq i-10}2^{l_{2}-i}\|u\|_{\tilde{X}_{i}(G_{\alpha}^{i}\times\R^{2})}^{2}\sum_{G_{\beta}^{l_{2}}\subset G_{\alpha}^{i}} \sum_{l_{2}-5\leq l_{3}\leq l_{2}} \|P_{\xi(t),l_{3}}u\|_{L_{t,x}^{4}(G_{\beta}^{l_{2}}\times\R^{2})}^{2} \lesssim \|u\|_{\tilde{X}_{i}(G_{\alpha}^{i}\times\R^{2})}^{4}.
\end{equation}
Hence, we conclude that $\mathrm{Term}_{2}\lesssim \|u\|_{\tilde{X}_{i}(G_{\alpha}^{i}\times\R^{2})}^{4}$.

To estimate $\mathrm{Term}_{1}$, we proceed similarly to as above, obtaining that
\begin{align}
\mathrm{Term}_{1} &\lesssim \sum_{0\leq l_{2}\leq i-10}2^{l_{2}-i}\sum_{l_{2}-5<l_{3}\leq l_{2}+1}\sum_{G_{\beta}^{l_{3}}\subset G_{\alpha}^{i}} \|(P_{\xi(t), l_{3}}u)(P_{\xi(t),\leq l_{2}-18}u)\|_{L_{t,x}^{2}(G_{\beta}^{l_{3}}\times\R^{2})}^{2} \nonumber\\
&\lesssim \sum_{0\leq l_{2}< 25}2^{l_{2}-i}\sum_{l_{2}-5<l_{3}\leq l_{2}+1}\sum_{G_{\beta}^{l_{3}}\subset G_{\alpha}^{i}} \sup_{z\in\R^{2}} \|(P_{\xi(G_{\beta}^{l_{3}}), l_{3}-2\leq\cdot\leq l_{3}+2}\tau_{z}u)(P_{\xi(t), \leq l_{2}-18}u)\|_{L_{t,x}^{2}(G_{\beta}^{l_{3}}\times\R^{2})}^{2} \nonumber\\
&\phantom{=} +\sum_{25\leq l_{2}\leq i-10}2^{l_{2}-i}\sum_{l_{2}-5<l_{3}\leq l_{2}+1}\sum_{G_{\beta}^{l_{3}}\subset G_{\alpha}^{i}} \sup_{z\in\R^{2}} \|(P_{\xi(G_{\beta}^{l_{3}}), l_{3}-2\leq\cdot\leq l_{3}+2}\tau_{z}u)(P_{\xi(t), \leq l_{2}-18}u)\|_{L_{t,x}^{2}(G_{\beta}^{l_{3}}\times\R^{2})}^{2} \nonumber\\
&\eqqcolon \mathrm{Term}_{1,1} + \mathrm{Term}_{1,2}.
\end{align}
To estimate $\mathrm{Term}_{1,1}$, we use H\"{o}lder's inequality and lemma \ref{lem:lohi_embed} to obtain the estimate $\mathrm{Term}_{1,1} \lesssim \|u\|_{\tilde{X}_{i}(G_{\alpha}^{i}\times\R^{2})}^{4}$. To estimate $\mathrm{Term}_{1,2}$, we apply proposition \ref{prop:ibs_2}, with the indices $l_{2}$ and $i$ in the statement of the proposition replaced by $l_{2}-18$ and $l_{3}$, respectively, to the atoms of $P_{\xi(G_{\beta}^{l_{3}}),l_{3}-2\leq\cdot\leq l_{3}+2}\tau_{z}u$ together with the spatial translation invariance of the $U_{\Delta}^{2}$ norm, to obtain that
\begin{equation}
\sup_{z\in\R^{2}} \|(P_{\xi(G_{\beta}^{l_{3}}), l_{3}-2\leq\cdot\leq l_{3}+2}\tau_{z}u)(P_{\xi(t), \leq l_{2}-18}u)\|_{L_{t,x}^{2}(G_{\beta}^{l_{3}}\times\R^{2})}^{2} \lesssim_{u} \|P_{\xi(G_{\beta}^{l_{3}}),l_{3}-2\leq\cdot\leq l_{3}+2}u\|_{U_{\Delta}^{2}(G_{\beta}^{l_{3}}\times\R^{2})}^{2}\paren*{1+\|u\|_{\tilde{X}_{i}(G_{\alpha}^{i}\times\R^{2})}^{4}}.
\end{equation}
Hence by interchanging the order of the $l_{2}$ and $l_{3}$ summations, using that $2^{l_{2}}\sim 2^{l_{3}}$, and using the definition of the $X(G_{\alpha}^{i}\times\R^{2})$ norm, we see that
\begin{equation}
\begin{split}
&\sum_{0\leq l_{2}\leq i-10}\sum_{l_{2}-5\leq l_{3}\leq l_{2}+1} 2^{l_{3}-i} \sum_{G_{\beta}^{l_{3}}\subset G_{\alpha}^{i}} \|P_{\xi(G_{\beta}^{l_{3}}), l_{3}-2\leq\cdot\leq l_{3}+2}u\|_{U_{\Delta}^{2}(G_{\beta}^{l_{3}}\times\R^{2})}^{2} \paren*{1+\|u\|_{\tilde{X}_{i}(G_{\alpha}^{i}\times\R^{2})}^{4}}\\
&\phantom{=} \lesssim \|u\|_{\tilde{X}_{i}(G_{\alpha}^{i}\times\R^{2})}^{2} \paren*{1+\|u\|_{\tilde{X}_{i}(G_{\alpha}^{i}\times\R^{2})}^{4}}.
\end{split}
\end{equation}
Thus, we have shown that $\mathrm{Term}_{1,2}\lesssim 1+\|u\|_{\tilde{X}_{i}(G_{\alpha}^{i}\times\R^{2})}^{6}$, and therefore $\mathrm{Term}_{1} \lesssim 1+\|u\|_{\tilde{X}_{i}(G_{\alpha}^{i}\times\R^{2})}^{6}$.

Combining the estimates for $\mathrm{Term}_{1}$ and $\mathrm{Term}_{2}$, we have shown that
\begin{equation}
\sum_{0\leq l_{2}\leq i-10} 2^{l_{2}-i} \|v_{0}\|_{L^{2}}^{2} \|F_{2,l_{2}}\|_{L_{t,x}^{1}} \lesssim \|v_{0}\|_{L^{2}}^{2}\paren*{1+\|u\|_{\tilde{X}_{i}(G_{\alpha}^{i}\times\R^{2})}^{6}}.
\end{equation}

Lastly, by interpolating between the estimates for $\|F_{2,l_{2}}\|_{L_{t,x}^{1}}$ and $\|F_{4,l_{2}}\|_{L_{t,x}^{1}}$, it follows that
\begin{equation}
\sum_{0\leq l_{2}\leq i-10} 2^{l_{2}-i} \|v_{0}\|_{L^{2}}^{2} \|F_{3,l_{2}}\|_{L_{t,x}^{1}} \lesssim \|v_{0}\|_{L^{2}}^{2}\paren*{1+\|u\|_{\tilde{X}_{i}(G_{\alpha}^{i}\times\R^{2})}^{6}}.
\end{equation}
We omit the details.

Bookkeeping the estimates for the contributions of $F_{2,l_{2}},F_{3,l_{2}},F_{4,l_{2}}$, we finally conclude that
\begin{equation}
\begin{split}
&\sum_{0\leq l_{2}\leq i-10}\left(\sup_{k\leq l_{2}+2}\left(\sum_{a\in\Z^{2}}2^{k-2i} \left|\int_{G_{\alpha}^{i}}\int_{\R^{4}}\frac{(x-y)}{|x-y|}\cdot \sum_{m=2}^{4}F_{m,l_{2}}(t,y)\Im{\bar{v}_{a,k}(\nabla-i\xi(t))v_{a,k}}(t,x)dxdydt\right|\right)^{1/2}\right)^{2}\\
&\phantom{=} \lesssim \|v_{0}\|_{L^{2}}^{2}\paren*{1+\|u\|_{\tilde{X}_{i}(G_{\alpha}^{i}\times\R^{2})}^{6}}.
\end{split}
\end{equation}

With this last estimate we have shown that
\begin{equation}
|\eqref{eq:ibs3_err3}| \lesssim \|v_{0}\|_{L^{2}}^{2} \paren*{1+\|u\|_{\tilde{X}_{i}(G_{\alpha}^{i}\times\R^{2})}^{6}},
\end{equation}
which completes the proof of the proposition \ref{prop:ibs_3}.
\end{description}
\end{proof}

\section{Rigidity: Rapid frequency cascade}\label{sec:RFC}
In this section, we preclude the rapid frequency cascade scenario $\int_{0}^{\infty}N(t)^{3}dt<\infty$ in which the energy is migrating from high to low frequencies as time progresses. We follow the argument of \cite{Dodson2016} and the earlier works \cite{Tao2006}, \cite{Killip2008}, and \cite{Killip2009} by showing that the hypothesis $\int_{0}^{\infty}N(t)^{3}dt=K<\infty$ forces the solution $u$ to have additional regularity--specifically, $u\in C_{t}^{0}H_{x}^{3}([0,\infty)\times\R^{2})$. We then use this additional regularity to derive a contradiction from the conservation of energy for $H^{1}$-solutions, concluding that there the rapid frequency cascade scenario does not occur for admissible blowup solutions. We now turn to the details.

We first prove lemma \ref{lem:a_reg} stated in subsection \ref{ssec:in_out}. We reproduce the statement of the lemma below for the reader's benefit.
\areg*

\begin{proof}
Fix an admissible blowup solution $u:I\times\R^{2}\rightarrow \C$. Let $(\epsilon_{1},\epsilon_{2},\epsilon_{3})$ be an admissible tuple with $0<\epsilon_{j}<\epsilon_{j}(u)$, where $\epsilon_{j}(u)$ is the constant dictated by the long-time Strichartz estimate (theorem \ref{thm:LTSE}), for $j=1,2,3$. Let $T>0$ be such that
\begin{equation}
\int_{0}^{T}\int_{\R^{2}} |u(t,x)|^{4}dxdt = 2^{k_{0}} \enspace \text{for some} \enspace k_{0}\in\mathbb{N} \enspace \text{and} \enspace \int_{0}^{T}N(t)^{3}dt\geq \frac{K}{2}.
\end{equation}
Now define the parameter
\begin{equation}
\lambda \coloneqq \frac{\epsilon_{3}2^{k_{0}}}{\int_{0}^{T}N(t)^{3}dt}
\end{equation}
and rescale the solution $u$ by defining $u_{\lambda} \coloneqq \lambda u(\lambda^{2}\cdot,\lambda\cdot)$, so that $u_{\lambda}$ is in the form of theorem \ref{thm:LTSE}.  Applying theorem \ref{thm:LTSE}, we obtain the estimates
\begin{align}
\|u_{\lambda}\|_{\tilde{X}_{k_{0}}([0,\lambda^{-2}T]\times\R^{2})} &\leq C_{0} \label{eq:NRFC_est1}\\
\|u_{\lambda}\|_{\tilde{Y}_{k_{0}}([0,\lambda^{-2}T]\times\R^{2})} &\leq \epsilon_{2}^{1/2}. \label{eq:NRFC_est2}
\end{align}
where $C_{0}$ is the constant in the statement of theorem \ref{thm:LTSE}.

Next, we claim that $N(t)\rightarrow 0$ as $t\rightarrow\infty$. Indeed, by absolutely continuity of the Lebesgue integral, given $\varepsilon>0$, there exists $t_{0}>0$ such that $\int_{t_{0}}^{\infty}N(t)^{3}dt\leq \varepsilon\epsilon_{1}^{1/2}$. Moreover, since $N(t)$ cannot be bounded from below on $[0,\infty)$ under the assumption that $\int_{0}^{\infty}N(t)^{3}dt<\infty$ and taking $t_{0}$ larger if necessary, we can choose $t_{0}$ so that $N(t_{0})\leq \frac{\varepsilon}{2}$. Hence, for any $t_{1}\geq t_{0}$, we have by the fundamental theorem of calculus that
\begin{equation}
|N(t_{1})| \leq \frac{\varepsilon}{2}+|N(t_{1})-N(t_{0})| \leq \frac{\varepsilon}{2}+2^{-20}\epsilon_{1}^{-1/2}\int_{t_{0}}^{t_{1}}N(t)^{3}dt \leq \varepsilon,
\end{equation}
which implies the claim. Additionally, we observe that
\begin{equation}
|\xi(t)| \leq 2^{-20}\epsilon_{1}^{-1/2}\int_{0}^{t}N(\tau)^{3}d\tau\leq 2^{-20}\epsilon_{1}^{-1/2}K, \qquad \forall t\geq 0,
\end{equation}
which implies that $\sup_{t\geq 0}|\xi_{\lambda}(t)|\leq 2^{k_{0}-19}\epsilon_{3}^{1/2}$. Lastly, since $N(t)\leq 1$ by definition of admissible blowup solution, we observe that
\begin{equation}
N_{\lambda}(t) = \lambda N(\lambda^{2}t) \leq \lambda \leq \frac{\epsilon_{3}2^{k_{0}+1}}{K}.
\end{equation}

The estimate \eqref{eq:NRFC_est1} implies that
\begin{equation}
\sum_{i\geq k_{0}} \|P_{i}u_{\lambda}\|_{U_{\Delta}^{2}([0,\lambda^{-2}T]\times\R^{2})}^{2} \leq C_{1}^{2},
\end{equation}
where $C_{1}\lesssim C_{0}$. So by undoing the scaling and using the scale invariance of the $U_{\Delta}^{2}$ norm together with some elementary Littlewood-Paley theory, we obtain that
\begin{equation}
\sum_{i\geq \log_{2}(\frac{4K}{\epsilon_{3}})} \|P_{i}u\|_{U_{\Delta}^{2}([0,T]\times\R^{2})}^{2} \leq C_{2}^{2},
\end{equation}
where $C_{2}\lesssim C_{1}$. Since $T>0$ may be taken arbitrarily large, Fatou's lemma implies that
\begin{equation}\label{eq:areg_base}
\sum_{i\geq\log_{2}(\frac{4K}{\epsilon_{3}})} \|P_{i}u\|_{U_{\Delta}^{2}([0,\infty)\times\R^{2})}^{2} \leq C_{2}^{2}.
\end{equation}
By Duhamel's formula, duality, and the estimate $\|\xi_{\lambda}\|_{L_{t}^{\infty}([0,\infty)}\leq 2^{k_{0}-19}\epsilon_{3}^{1/2}$, we have that for every $j\geq k_{0}$,
\begin{align}
\sum_{i\geq j}\|P_{i}u_{\lambda}\|_{U_{\Delta}^{2}([0,\lambda^{-2}T]\times\R^{2})}^{2} &\lesssim \sum_{i\geq j} \paren*{\inf_{t\in [0,\lambda^{-2}T]} \brac*{\|P_{i}u_{\lambda}(t)\|_{L_{x}^{2}(\R^{2})}^{2}} + \|P_{i}F(u_{\lambda})\|_{U_{\Delta}^{2}([0,\lambda^{-2}T]\times\R^{2})}^{2}} \nonumber\\
&\lesssim \inf_{t\in [0,\lambda^{-2}T]} \brac*{\sum_{i\geq j} \|P_{i}u_{\lambda}(t)\|_{L_{x}^{2}(\R^{2})}^{2}} + \sum_{i\geq j}\|P_{\xi_{\lambda}([0,\lambda^{-2}T]), i-2\leq\cdot\leq i+2}F(u_{\lambda})\|_{U_{\Delta}^{2}([0,\lambda^{-2}T]\times\R^{2})}^{2}.
\end{align}
Using the estimates in the proof of lemma \ref{lem:LTSE_b_easy} together with the estimates of lemma \ref{lem:LTSE_b_hard}, it follows that for every $i\geq j$,
\begin{align}
&\|P_{\xi_{\lambda}([0,\lambda^{-2}T]),i-2\leq\cdot\leq i+2}F(u_{\lambda})\|_{U_{\Delta}^{2}([0,\lambda^{-2}T]\times\R^{2})}^{2}\nonumber\\
&\phantom{=} \lesssim_{u} \paren*{\epsilon_{2}\|u_{\lambda}\|_{\tilde{X}_{k_{0}}([0,\lambda^{-2}T]\times\R^{2})}^{3}+\epsilon_{2}^{1/3}\|u_{\lambda}\|_{\tilde{X}_{k_{0}}([0,\lambda^{-2}T]\times\R^{2})}^{5/3}} \|P_{\xi_{\lambda}([0,\lambda^{-2}T]),i-5\leq\cdot\leq i+5}u_{\lambda}\|_{U_{\Delta}^{2}([0,\lambda^{-2}T]\times\R^{2})}^{2} \nonumber\\
&\phantom{=}\quad + \paren*{\epsilon_{2}+\|u_{\lambda}\|_{\tilde{Y}_{k_{0}}([0,\lambda^{-2}T]\times\R^{2})}\paren*{1+\|u_{\lambda}\|_{\tilde{X}_{k_{0}}([0,\lambda^{-2}T]\times\R^{2})}^{4}}}^{2} \|P_{\xi_{\lambda}([0,\lambda^{-2}T]),i-5\leq\cdot\leq i+5}u_{\lambda}\|_{U_{\Delta}^{2}([0,\lambda^{-2}T]\times\R^{2})}^{2} \nonumber\\
&\phantom{=} \leq \paren*{\epsilon_{2}C_{0}^{3}+\epsilon_{2}^{1/3}C_{0}^{5/3}+\paren*{\epsilon_{2}+\epsilon_{2}^{1/2}\paren*{1+C_{0}^{4}}}^{2}}  \|P_{\xi_{\lambda}([0,\lambda^{-2}T]),i-5\leq\cdot\leq i+5}u_{\lambda}\|_{U_{\Delta}^{2}([0,\lambda^{-2}T]\times\R^{2})}^{2} \nonumber\\
&\phantom{=} \leq 2\epsilon_{2}^{1/3}C_{0}^{5/3}  \|P_{\xi_{\lambda}([0,\lambda^{-2}T]),i-5\leq\cdot\leq i+5}u_{\lambda}\|_{U_{\Delta}^{2}([0,\lambda^{-2}T]\times\R^{2})}^{2},
\end{align}
provided that $\epsilon_{2}\in (0,\epsilon_{2}(u))$ is sufficiently small depending on $C_{0}$. Using the estimate
\begin{equation}
\|P_{\xi_{\lambda}([0,\lambda^{-2}T]),i-5\leq\cdot\leq i+5}u_{\lambda}\|_{U_{\Delta}^{2}([0,\lambda^{-2}T]\times\R^{2})} \lesssim \|P_{i-7\leq\cdot\leq i+7}u_{\lambda}\|_{U_{\Delta}^{2}([0,\lambda^{-2}T]\times\R^{2})},
\end{equation}
we obtain that for every $j\geq k_{0}$,
\begin{equation}
\begin{split}
&\sum_{i\geq j} \|P_{\xi_{\lambda}([0,\lambda^{-2}T]),i-2\leq\cdot\leq i+2}F(u_{\lambda})\|_{U_{\Delta}^{2}([0,\lambda^{-2}T]\times\R^{2})}^{2} \\
&\phantom{=} \lesssim_{u} \epsilon_{2}^{1/3}C_{0}^{5/3}\sum_{i\geq j-7}\|P_{i}u_{\lambda}\|_{U_{\Delta}^{2}([0,\lambda^{-2}T]\times\R^{2})}^{2}.
\end{split}
\end{equation}
Undoing the scaling and using the scale invariance of the $L_{x}^{2}$ and $U_{\Delta}^{2}$ norms together with Plancherel's theorem, we have that
\begin{equation}
\sum_{i\geq\log_{2}(\frac{2^{j+2-k_{0}}K}{\epsilon_{3}})} \|P_{i}u\|_{U_{\Delta}^{2}([0,T]\times\R^{2})}^{2} \lesssim_{u} \inf_{t\in [0,T]} \brac*{\|P_{\geq \frac{2^{j-k_{0}-1}K}{\epsilon_{3}}}u(t)\|_{L_{x}^{2}(\R^{2})}^{2}} + \epsilon_{2}^{1/3}C_{0}^{5/3}\sum_{i\geq \log_{2}(\frac{2^{j-k_{0}-9}K}{\epsilon_{3}})} \|P_{i}u\|_{U_{\Delta}^{2}([0,T]\times\R^{2})}^{2}.
\end{equation}
Since $T>0$ can be taken arbitrarily large in the preceding inequality, the monotone convergence theorem implies that for every integer $j \geq \log_{2}(4\frac{K}{\epsilon_{3}})$,
\begin{equation}
\sum_{i\geq j} \|P_{i}u\|_{U_{\Delta}^{2}([0,\infty)\times\R^{2})}^{2} \lesssim \inf_{t\in [0,\infty)} \brac*{\|P_{\geq j-4}u(t)\|_{L_{x}^{2}(\R^{2})}^{2}} +\epsilon_{2}^{1/3}C_{0}^{5/3} \sum_{i\geq j-9} \|P_{i}u\|_{U_{\Delta}^{2}([0,\infty)\times\R^{2})}^{2}.
\end{equation}

Next, we claim that for any integer $j\geq \log_{2}(\frac{K}{\epsilon_{3}})-4$, 
\begin{equation}
\lim_{t\rightarrow\infty} \|P_{\geq j}u(t)\|_{L_{x}^{2}(\R^{2})} = 0,
\end{equation}
which implies that $\inf_{t\in[0,\infty)} \|P_{\geq j-4}u(t)\|_{L_{x}^{2}(\R^{2})}^{2}=0$ for $j\geq \log_{2}(\frac{4K}{\epsilon_{3}})$. Indeed, given any $\varepsilon>0$, we see from the bound $|\xi(t)| \leq 2^{-20}\epsilon_{1}^{-1/2}K$, $\lim_{t\rightarrow\infty}N(t)=0$, and the frequency localization property \eqref{eq:sp_frq_loc} that for fixed $j$, there exists $t_{0}(\varepsilon)>0$ such that for all $t\geq t_{0}(\varepsilon)$,
\begin{equation}
\|P_{\geq j}u(t)\|_{L_{x}^{2}(\R^{2})}^{2} \leq \int_{|\xi-\xi(t)| \geq C(\varepsilon)N(t)} |\hat{u}(t,\xi)|^{2}d\xi \leq \varepsilon,
\end{equation}
which proves the claim.

Therefore, we have shown that there exists a constant $C(u)$ such that for any integer $j\geq \log_{2}(\frac{4K}{\epsilon_{3}})$,
\begin{equation}\label{eq:areg_it_est}
\sum_{i\geq j} \|P_{i}u\|_{U_{\Delta}^{2}([0,\infty)\times\R^{2})}^{2} \leq \epsilon_{2}^{1/3}C_{0}^{5/3}C(u) \sum_{i\geq j-9} \|P_{i}u\|_{U_{\Delta}^{2}([0,\infty)\times\R^{2})}^{2}.
\end{equation}
Now choose $\epsilon_{2}\leq\epsilon_{2}(u)$ sufficiently small so that $\epsilon_{2}^{1/3}C_{0}^{5/3}C(u) \leq 2^{-100}$. Next, for every integer $j\geq \log_{2}(\frac{4K}{\epsilon_{3}})$, define the positive integer
\begin{equation}
j_{*} \coloneqq \left\lfloor \frac{1}{9}\log_{2}\left(\frac{\epsilon_{3}2^{j-2}}{K}\right)\right\rfloor.
\end{equation}
Therefore, by iterating the estimate \eqref{eq:areg_it_est} $j_{*}$ times, for each $j\geq \log_{2}(\frac{4K}{\epsilon_{3}})$, we see that
\begin{align}
\sum_{2^{j}\geq \epsilon_{3}^{-1}4K} 2^{6j} \|P_{j}u\|_{U_{\Delta}^{2}([0,\infty)\times\R^{2})}^{2} &\leq \sum_{2^{j}\geq \epsilon_{3}^{-1}4K} 2^{6j}\sum_{i\geq j} \|P_{i}u\|_{U_{\Delta}^{2}([0,\infty)\times\R^{2})}^{2} \nonumber\\
&\leq \sum_{2^{j}\geq \epsilon_{3}^{-1}4K} 2^{6j}2^{-100j_{*}} \sum_{2^{i}\geq \epsilon_{3}^{-1}4K} \|P_{i}u\|_{U_{\Delta}^{2}([0,\infty)\times\R^{2})}^{2} \nonumber\\
&\leq C_{2}^{2}\sum_{2^{j}\geq \epsilon_{3}^{-1}4K} \paren*{\frac{\epsilon_{3}}{4K}}^{-100/9}2^{6j}2^{-100j/9} \nonumber\\
&\leq \frac{C_{3}^{2}K^{6}}{\epsilon_{3}^{6}},
\end{align}
where $C_{3}\lesssim C_{2}$, where we use the estimate \eqref{eq:areg_base} to obtain the penultimate inequality. Since $U_{\Delta}^{2}\subset L_{t}^{\infty}L_{x}^{2}$, it follows from the triangle inequality, Plancherel's theorem, and mass conservation that
\begin{align}
\|u(t)\|_{H_{x}^{3}(\R^{2})} \lesssim \left(\sum_{2^{i}<\epsilon_{3}^{-1}4K}2^{6j}\|P_{i}u(t)\|_{L_{x}^{2}(\R^{2})}^{2}\right)^{1/2} + \frac{C_{3}K^{3}}{\epsilon_{3}^{3}} \lesssim \frac{K^{3}}{\epsilon_{3}^{3}},
\end{align}
for almost every $t\in [0,\infty)$. By a standard persistence of regularity argument, it follows that $u\in C_{t}^{0}H_{x}^{3}([0,\infty)\times\R^{2})$.
\end{proof}

Using lemma \ref{lem:a_reg}, we can now prove theorem \ref{thm:no_rfc}, the statement of which we recall below.
\norfc*

\begin{proof}
We first claim that $\xi_{\infty} \coloneqq \lim_{t\rightarrow\infty}\xi(t)$ exists. Indeed, for any $t_{2}\geq t_{1}\geq 0$, we have by the fundamental theorem of calculus that
\begin{equation}
|\xi(t_{1})-\xi(t_{2})| \leq 2^{-20}\epsilon_{1}^{-1/2}\int_{t_{1}}^{t_{2}}N(t)^{3}dt \xrightarrow{t_{1},t_{2}\rightarrow\infty} 0.
\end{equation}
Hence, $\{\xi(t)\}_{t\geq 0}$ is Cauchy, from which the claim follows. Moreover, $|\xi(t)| \leq 2^{-20}\epsilon_{1}^{-1/2}K$ implies that $|\xi_{\infty}|\leq 2^{-20}\epsilon_{1}^{-1/2}K$. Applying a Galilean transformation to $u$ which maps $\xi_{\infty}$ to the origin, we obtain another admissible blowup solution
\begin{equation}
v(t,x) \coloneqq e^{-ix\cdot\xi_{\infty}}e^{-it|x_{\infty}|^{2}}u(t,x+2\xi_{\infty}t)
\end{equation}
with parameters $x_{v}(t),\xi_{v}(t), N(t)$ and compactness modulus function $C$.

We now claim that $\lim_{t\rightarrow \infty} \|v(t)\|_{\dot{H}_{x}^{1}(\R^{2})} =0$. Indeed, observe that by lemma \ref{lem:a_reg} and mass conservation,
\begin{equation}
\|v(t)\|_{\dot{H}_{x}^{3}(\R^{2})} \lesssim |\xi_{\infty}|^{3}\|u(t)\|_{L_{x}^{2}(\R^{2})}+\||\nabla|^{3}u(t)\|_{L_{x}^{2}(\R^{2})} \lesssim_{u} \epsilon_{1}^{-3/2}K^{3}+K^{3}, \qquad \forall t\geq 0.
\end{equation}
We make the sub-claim that for every $\eta>0$,
\begin{equation}
\lim_{t\rightarrow\infty} \|P_{\xi_{v}(t), \leq C(\eta)N(t)} v(t)\|_{\dot{H}_{x}^{3}(\R^{2})} = 0.
\end{equation}
To see this, observe that since $\xi_{v}(t)\rightarrow 0$ and $N(t)\rightarrow 0$ as $t\rightarrow \infty$, it follows from mass conservation that
\begin{equation}
\lim_{t\rightarrow \infty} \|P_{\xi_{v}(t),\leq C(\eta)N(t)} v(t)\|_{\dot{H}_{x}^{3}(\R^{2})} \lesssim \lim_{t\rightarrow \infty} \paren*{|\xi_{v}(t)|+C(\eta)N(t)}^{3} \|v_{0}\|_{L^{2}(\R^{2})} \lesssim \lim_{t\rightarrow \infty} \paren*{|\xi_{v}(t)|+C(\eta)N(t)}^{3}=0.
\end{equation}
Interpolating with the $L^{2}$ norm and using the the frequency localization property \eqref{eq:sp_frq_loc}, we have that
\begin{align}
\|P_{\xi_{v}(t),\geq C(\eta)N(t)}v(t)\|_{\dot{H}_{x}^{1}(\R^{2})} \leq \|P_{\xi_{v}(t), \geq C(\eta)N(t)}v(t)\|_{L_{x}^{2}(\R^{2})}^{2/3} \|P_{\xi_{v}(t),\geq C(\eta)N(t)}v(t)\|_{\dot{H}_{x}^{3}(\R^{2})}^{1/3} \lesssim K\eta^{1/3}
\end{align}
for all $t\geq 0$. Hence, for any $\eta>0$, we have by the triangle inequality that
\begin{equation}
\limsup_{t\rightarrow\infty} \|v(t)\|_{\dot{H}_{x}^{1}(\R^{2})} \lesssim K\eta^{1/3}.
\end{equation}
Since $\eta>0$ may be taken arbitrarily small, we conclude the claim.

Now by the Gagliardo-Nirenberg inequality and mass conservation,
\begin{equation}
\lim_{t\rightarrow\infty}\|v(t)\|_{L_{x}^{4}(\R^{2})} \lesssim \lim_{t\rightarrow\infty}\|v(t)\|_{L_{x}^{2}(\R^{2})}^{1/2} \|v(t)\|_{\dot{H}_{x}^{1}(\R^{2})}^{1/2} = 0.
\end{equation}
Plancherel's theorem then implies that the quartic term in $E(v(t))$ tends to zero as $t\rightarrow \infty$. We therefore conclude that $E(v(t))\rightarrow 0$ as $t\rightarrow \infty$, which by energy conservation implies that $E(v(t))\equiv 0$. Since the energy is nonnegative under our assumptions on the initial data, we conclude that $v\equiv 0$, which trivially implies that $u\equiv 0$, a contradiction.
\end{proof}

\section{Rigidity: Quasi-soliton}\label{sec:QS}

In this section, we prove theorem \ref{thm:no_qs}, the statement of which we recall below, by constructing a frequency-localized interaction Morawetz type estimate, thus completing the rigidity step in the proof of theorem \ref{thm:main} and hence theorem \ref{thm:main} itself.

\noqs*

\subsection{Preliminaries}\label{ssec:QS_pre}
In this subsection, we record some preliminary lemmas which we will need to handle the various error terms arising when we attempt to preclude the quasi-soliton scenario. Throughout this section we use the notation $o_{A}(1)$ to denote a quantity satisfying
\begin{equation}
\lim_{A\rightarrow\infty} o_{A}(1)=0.
\end{equation}
Furthermore, $u$ always denotes an admissible blowup solution throughout this section. Additionally, we continue to use the notation $\lesssim_{u},\sim_{u},\gtrsim_{u}$ to denote implicit constants which depend on a given admissible blowup solution $u$ through its APMS parameters.

Lemmas \ref{lem:IM_pre_err1}, \ref{lem:IM_pre_err2}, \ref{lem:IM_pre_err3}, and \ref{lem:IM_pre_sm} are from \cite{Dodson2015}. We include proofs of the first three lemmas in order to obtain lemma \ref{lem:IM_pre_CZ}, which is a new extension for our nonlocal setting.

\begin{lemma}\label{lem:IM_pre_err1}
We have the estimate
\begin{equation}
\|P_{\xi(t), \geq RN(t)}u\|_{L_{t,x}^{4}(J\times\R^{2})}^{4} + \|\mathrm{1}_{x(t), \geq R/N(t)}u\|_{L_{t,x}^{4}(J\times\R^{2})}^{4} \leq o_{R}(1),
\end{equation}
uniformly in small intervals $J$.
\end{lemma}
\begin{proof}
Interpolating between the admissible pairs $(3,6)$ and $(\infty,2)$ to get $(4,4)$ and using Bernstein's lemma together with the fact that $J$ is small, we see that
\begin{align}
&\|P_{\xi(t), \geq RN(t)}u\|_{L_{t,x}^{4}(J\times\R^{2})}^{4} + \|\mathrm{1}_{x(t), \geq R/N(t)}u\|_{L_{t,x}^{4}(J\times\R^{2})}^{4} \nonumber\\
&\phantom{=} \lesssim \|P_{\xi(t),\geq RN(t)}u\|_{L_{t}^{3}L_{x}^{6}(J\times\R^{2})}^{3}\|P_{\xi(t),\geq RN(t)}u\|_{L_{t}^{\infty}L_{x}^{2}(J\times\R^{2})} + \|1_{x(t), \geq R/N(t)}u\|_{L_{t}^{3}L_{x}^{6}(J\times\R^{2})}^{3}\|1_{x(t), \geq R/N(t)}u\|_{L_{t}^{\infty}L_{x}^{2}(J\times\R^{2})} \nonumber\\
&\phantom{=} \lesssim \|P_{\xi(t),\geq RN(t)}u\|_{L_{t}^{\infty}L_{x}^{2}(J\times\R^{2})}+\|1_{x(t), \geq R/N(t)}u\|_{L_{t}^{\infty}L_{x}^{2}(J\times\R^{2})} \nonumber\\
&\phantom{=} =o_{R}(1),
\end{align}
where the ultimate equality follows from the frequency localization property \eqref{eq:sp_frq_loc}.
\end{proof}

\begin{lemma}\label{lem:IM_pre_err2}
Given $\eta\in (0,1)$, there exists an $R(\eta)>0$, such that for all $R\geq R(\eta)$,
\begin{equation}
\int_{0}^{T}N(t)\|P_{\leq RN(t)}u(t)\|_{L_{x}^{4}(\R^{2})}^{4}dt \gtrsim_{u} \eta K
\end{equation}
for all $T>0$ such that $\int_{0}^{T}N(t)^{3}dt=K$.
\end{lemma}
\begin{proof}
Partition $[0,T]$ into finitely many consecutive small intervals $J_{l}$. Then for each $J_{l}$, we have that
\begin{align}
\int_{J_{l}}N(t)\|P_{\xi(t),\leq RN(t)}u(t)\|_{L_{x}^{4}}^{4}dt  &\sim_{u} N(J_{l})\int_{J_{l}}\|P_{\xi(t),\leq RN(t)}u(t)\|_{L_{x}^{4}}^{4}dt \nonumber\\
&= N(J_{l})\int_{J_{l}}\|u(t)\|_{L_{x}^{4}}^{4}dt - N(J_{l})\int_{J_{l}}\|P_{\xi(t),>RN(t)}u(t)\|_{L_{x}^{4}}^{4}dt \nonumber\\
&=N(J_{l}) - N(J_{l})\int_{J_{l}}\|P_{\xi(t),>RN(t)}u(t)\|_{L_{x}^{4}}^{4}dt \nonumber\\
&\gtrsim \eta N(J_{l}),
\end{align}
provided that $R\geq R(\eta)$ by lemma \ref{lem:IM_pre_err1}. Therefore, summing over $l$, we obtain the lower bound
\begin{equation}
\int_{0}^{T}N(t)\|P_{\xi(t),\leq RN(t)}u(t)\|_{L_{x}^{4}}^{4}dt \gtrsim_{u} \sum_{l}\eta N(J_{l}) \sim_{u} \eta\int_{0}^{T}N(t)^{3}dt = \eta K,
\end{equation}
which completes the proof.
\end{proof}

\begin{lemma}\label{lem:IM_pre_err3}
We have the estimate
\begin{equation}
\int_{J}\int_{|x-x(t)|\geq\frac{R}{N(t)}}|w(t,x)|^{4}dx=o_{R}(1),
\end{equation}
where $w:=P_{\leq K}u$, uniformly in small intervals $J$ and real numbers $K\geq 1$.
\end{lemma}
\begin{proof}
Observe that by Minkowski's inequality,
\begin{align}
\left(\int_{|x-x(t)| \geq \frac{R}{N(t)})} |w(t,x)|^{4}dx\right)^{1/4} &\leq \int_{\R^{2}}K^{2}|\phi^{\vee}(Kz)| \left(\int_{\R^{2}} \mathrm{1}_{x(t), \geq \frac{R}{N(t)}}(x)|u(t,x-z)|^{4}dx\right)^{1/4 }dz \nonumber\\
&\lesssim \int_{|z| \leq \frac{R}{4N(t)}}K^{2}|\phi^{\vee}(Kz)| \left(\int_{|x-x(t)| \geq \frac{R}{2N(t)}}|u(t,x)|^{4}dx\right)^{1/4}dz \nonumber\\
&\phantom{=}+ \int_{|z| \geq \frac{R}{4N(t)}} K^{2}\jp{Kz}^{-3}\left(\int_{\R^{2}}|u(t,x)|^{4}dx\right)^{1/4}dz \nonumber\\
&\lesssim \left(\int_{|x-x(t)| \geq \frac{R}{N(t)}}|u(t,x)|^{4}dx\right)^{1/4} + \frac{N(t)}{RK}\|u(t)\|_{L_{x}^{4}(\R^{2})},
\end{align}
where we use that $\|\phi^{\vee}\|_{L^{1}(\R^{2})}\lesssim 1$ to obtain the ultimate inequality. Therefore, taking the $L_{t}^{4}(J)$ norm of both sides of the final inequality, we obtain that
\begin{equation}
\left(\int_{J}\int_{|x-x(t)|\geq \frac{R}{N(t)}}|w(t,x)|^{4}dxdt\right)^{1/4} \lesssim \left(\int_{J}\int_{|x-x(t)|\geq \frac{R}{2N(t)}}|u(t,x)|^{4}dxdt\right)^{1/4} + \frac{N(J)}{RK} = o_{R}(1).
\end{equation}
\end{proof}

We now use lemma \ref{lem:IM_pre_err3} to prove a more general $o_{R}(1)$-type estimate for expressions involving Calder\'{o}n-Zygmund operators. We shall make heavy use of this result in the sequel to handle the various error terms arising in our calculations. We expect that one could prove similar estimates for more general paraproducts, but we have no need for such generality in this work.

\begin{lemma}\label{lem:IM_pre_CZ}%Estimate for o_{R}(1) terms involving CZ operators
Let $\E$ be the identity or any Calder\'{o}n-Zygmund operator of convolution type with kernel $\mathcal{K}$. Then
\begin{equation}
\int_{0}^{T}N(t)\int_{|x-y|\geq \frac{R}{N(t)}} |\E(|w|^{2})(t,x)|^{2} |w(t,y)|^{2}dxdydt = o_{R}(1)K,
\end{equation}
where $w \coloneqq P_{\leq K}u$, uniformly in intervals $[0,T]$ such that $\int_{0}^{T}N(t)^{3}dt=K$.
\end{lemma}
\begin{proof}
First, partition $[0,T]$ into finitely many consecutive small intervals $J_{l}$. Next, observe from the Fubini-Tonelli theorem that
\begin{align}
\int_{|x-y| \geq \frac{R}{N(t)}}|\E(|w|^{2})(t,x)|^{2}|w(t,y)|^{2}dxdy &= \int_{\R^{2}}|\E(|w|^{2})(t,x)|^{2}\left(\int_{|x-y|\geq \frac{R}{N(t)}} |w(t,y)|^{2}dy\right)dx \nonumber\\
&\leq \int_{|x-x(t)| \geq \frac{R}{2N(t)}} |\E(|w|^{2})(t,x)|^{2}\left(\int_{|y-x(t)|\leq \frac{R}{2N(t)}}|w(t,y)|^{2}dy\right)dx \nonumber\\
&\phantom{=}+ \int_{\R^{2}}|\E(|w|^{2})(t,x)|^{2}\left(\int_{|y-x(t)|\geq \frac{R}{2N(t)}} |w(t,y)|^{2}dy\right)dx \nonumber\\
&\lesssim \int_{|x-x(t)| \geq \frac{R}{2N(t)}} |\E(|w|^{2})(t,x)|^{2}dx + o_{R}(1)\|u(t)\|_{L_{x}^{4}(\R^{2})}^{4},
\end{align}
where we use the reverse triangle inequality to obtain the penultimate inequality and we apply the Calder\'{o}n-Zygmund theorem to $\E$ together with mass conservation and lemma \ref{lem:IM_pre_err1} to obtain the ultimate inequality. To estimate the first term on the RHS of the ultimate inequality, first observe that
\begin{align}
\int_{\R^{2}}\mathrm{1}_{x(t), \geq \frac{R}{2N(t)}}(x) |\E(|w|^{2})(t,x)|^{2}dx &\lesssim\int_{\R^{2}}1_{x(t),\geq \frac{R}{2N(t)}}(x)|\E\paren*{1_{x(t),\leq\frac{R}{4N(t)}}|w|^{2}}(t,x)|^{2}dx \nonumber\\
&\phantom{=}+\int_{\R^{2}}1_{x(t),\geq \frac{R}{2N(t)}}(x)|\E\paren*{1_{x(t),>\frac{R}{4N(t)}} |w|^{2}}(t,x)|^{2}dx \nonumber\\
&\eqqcolon \mathrm{Term}_{1}(t)+\mathrm{Term}_{2}(t).
\end{align}

To estimate $\mathrm{Term}_{2}$, we use the Calder\'{o}n-Zygmund theorem together with lemma \ref{lem:IM_pre_err3} to obtain that 
\begin{equation}
\int_{J_{l}}N(t)\mathrm{Term}_{2}(t)dt \lesssim \int_{J_{l}}N(t)\|\mathrm{1}_{x(t), > \frac{R}{4N(t)}}w(t)\|_{L_{x}^{4}(\R^{2})}^{4}dt = o_{R}(1)N(J_{l}). %Estimate for Term2
\end{equation}

To estimate $\mathrm{Term}_{1}$, we observe that
\begin{equation}
\E\paren*{|w|^{2}1_{x(t),\leq \frac{R}{4N(t)}}}(x) = \PV \int_{\R^{2}}\mathcal{K}(z) (1_{x(t),\leq \frac{R}{4N(t)}}|w|^{2})(x-z)dz, \qquad \forall x\in\R^{2}.
\end{equation}
If $|x-z-x(t)|\leq \frac{R}{4N(t)}$ and $|x-x(t)|\geq \frac{R}{2N(t)}$, then by the reverse triangle inequality,
\begin{equation}
|z| \geq |z-(x-x(t)) + (x-x(t))| \geq -|z-(x-x(t))| +|x-x(t)| \geq \frac{R}{4N(t)}.
\end{equation}
So, we can write
\begin{equation}
\E\paren*{|w|^{2}1_{x(t),\leq\frac{R}{4N(t)}}}(x)=\int_{\R^{2}}\mathcal{K}(z)1_{\geq\frac{R}{4N(t)}}(z) (1_{x(t),\leq\frac{R}{4N(t)}}|w|^{2})(x-z)dz, \qquad \forall |x-x(t)| \geq \frac{R}{2N(t)},
\end{equation}
where the integral is absolutely convergent by the CZK size condition $|\mathcal{K}(z)|\lesssim_{\E} |z|^{-2}$. We now make a change of variable in the inner integral to write
\begin{equation}
\mathrm{Term}_{1}(t) = \int_{\R^{2}}\mathrm{1}_{x(t), \geq\frac{R}{2N(t)}}(x)\paren*{\int_{\R^{2}} (\mathrm{1}_{>\frac{R}{4N(t)}}\mathcal{K})(x-z) |(\mathrm{1}_{x(t), \leq \frac{R}{4N(t)}}w(t))(z)|^{2}dz}^{2}dx.
\end{equation}
Then using Minkowski's inequality together with the CZK size condition , we obtain that the RHS of the preceding equality is $\lesssim$
\begin{align}
\left(\int_{|z-x(t)|\leq \frac{R}{4N(t)}}\left(\int_{|x-z|\geq \frac{R}{4N(t)}}\frac{1}{|x-z|^{4}} |w(t,z)|^{4}dx\right)^{1/2}dz\right)^{2}&\lesssim \left(\int_{|z-x(t)| \leq \frac{R}{4N(t)}}|w(t,z)|^{2} \left(\int_{|x|\geq \frac{R}{4N(t)}} \frac{1}{|x|^{4}}dx\right)^{1/2}dz\right)^{2} \nonumber\\
&\lesssim \frac{N(t)^{2}\|u\|_{L_{t}^{\infty}L_{x}^{2}}^{4}}{R^{2}}, %Estimate for Term1
\end{align}
where the ultimate inequality follows from dilation invariance. It follows from mass conservation that
\begin{equation}
\int_{J_{l}}N(t)\mathrm{Term}_{1}(t)dt \lesssim o_{R}(1)\int_{J_{l}}N(t)^{3}dt.
\end{equation}

Now summing our final estimates over $l$ completes the proof.
\end{proof}

Lastly, given an admissible blowup solution $u$ with frequency scale function $N(\cdot)$, we need the existence of a sequence frequency scale functions $N_{m}:I\rightarrow [0,\infty)$ which are pointwise dominated by $N(t)$ and slowly varying compared to $N(t)$. The existence of such a sequence was proved in \cite{Dodson2015} by a ``smoothing algorithm" (see subsection 6.1 of that reference). Although that work considered the focusing mass-critical NLS in dimensions $d\geq 1$, the argument does not rely on any property of the equation which is not shared by the eeDS equation \eqref{eq:DS}.

\begin{lemma}[Smoothing algorithm]\label{lem:IM_pre_sm} %Dodson smoothing algortihm
Let $u:I\times\R^{2}\rightarrow\mathbb{C}$ be an admissible blowup solution to \eqref{eq:DS} with frequency scale function $N(\cdot)$. Then there exists a sequence of frequency scale functions $N_{m}:I\rightarrow [0,+\infty)$ with the following properties:
\begin{enumerate}[(i)]
\item
$N_{0}=N$;
\item
$N_{m}(t)\leq N(t)$ for all $t\in [0,\infty)$;
\item
If $C(u)>0$ is such that for any small interval $J$,
\begin{equation}
\frac{1}{C(u)}N(J) \leq N(t)\leq N(J), \qquad \forall t\in J,
\end{equation}
and $\int_{0}^{T}N(t)dt=K$, then
\begin{equation}
\int_{0}^{T}N_{m}(t)dt\gtrsim \frac{K}{C(u)};
\end{equation}
\item
\begin{equation}
\liminf_{T\rightarrow\infty} \dfrac{\int_{0}^{T} |N_{m}'(t)|dt}{\int_{0}^{T}N_{m}(t)\|u(t)\|_{L_{x}^{4}(\R^{2})}^{4}dt} \leq \frac{2}{m}, \qquad \forall m\in\N_{0}.
\end{equation}
\end{enumerate}
\end{lemma}

\subsection{Construction of Morawetz functional}\label{ssec:QS_con}
In this subsection, we fix an admissible blowup solution $u$ satisfying $\int_{0}^{\infty}N(t)^{3}dt=\infty$ and use the sequence $N_{m}(\cdot)$ of frequency scale functions to construct a family of frequency-localized interaction Morawetz type functionals adapted to $u$.

Define parameters $0<\eta_{3}\ll \eta_{2}\ll \eta_{1}\ll 1$ and $R\gg 1$. We allow $\eta_{j}$ to depend on $\eta_{k}$ for $k>j$ in addition to any implicit constants and the critical mass $M(u)$. Similarly, we allow $R$ to depend on any of the implicit constants and the mass of the solution $M(u)$. We will state the precise relationship between all the parameters at the end of subsection \ref{ssec:QS_dfoc} for the defocusing case and subsection \ref{ssec:QS_foc} for the focusing case.

In order to construct our family of interaction Morawetz functionals, we seek an approximate bump function which is essentially localized to the set
\begin{equation}
\{(x,y)\in\R^{4} : |x-y|N_{m}(t)\sim_{R,\eta_{2}} 1\}.
\end{equation}
More precisely, let $\chi:\R\rightarrow [0,+\infty)$ be an even bump function satisfying $\|\chi\|_{L^{1}(\R)}=1$ and $\supp(\chi)\subset (-1,1)$. Define a function $\zeta:\R\rightarrow [0,1]$ by the formula
\begin{equation}
\zeta(r) \coloneqq \frac{1}{\eta_{2}R}\int_{\R} 1_{[-(1+2\eta_{2})R,(1+2\eta_{2})R]}(r-s)\chi\left(\frac{s}{\eta_{2}R}\right)ds.
\end{equation}
Observe that $\zeta$ is nonnegative, decreasing, and satisfies
\begin{equation}
\zeta(r) = \begin{cases} 1, & {|r| \leq R} \\ 0, & {|r|\geq (1+4\eta_{2})R}\end{cases}.
\end{equation}
Moreover, by Young's inequality, $\zeta$ satisfies the derivative estimates
\begin{equation}
|\zeta^{(n)}(r)| \lesssim_{n} \frac{1}{(\eta_{2}R)^{n}}, \qquad \forall n\in\N_{0}.
\end{equation}
Next, define a compactly supported function $\varphi: \R \rightarrow [0,+\infty)$ by the formula
\begin{equation}
\varphi(r) \coloneqq \frac{1}{\|\zeta(|\cdot|)\|_{L^{1}(\R^{2})}} \int_{\mathbb{S}^{1}}\int_{\R^{2}} \zeta(|r\omega-z|)\zeta(|z|)dzd\omega,
\end{equation}
where $d\omega$ denotes the unit normalized surface measure on the circle $\mathbb{S}^{1}$. Observe that by rotation invariance of the measure, for any $x,y\in\R^{2}$, we have that
\begin{align}
\varphi(|x-y|) &=  \frac{1}{\|\zeta(|\cdot|)\|_{L^{1}}}\int_{\mathbb{S}^{1}}\int_{\R^{2}}\zeta(|x-y|\omega-z)\zeta(|z|)dzd\omega \nonumber\\
&= \frac{1}{\|\zeta(|\cdot|)\|_{L^{1}}}\int_{\mathbb{S}^{1}}\int_{\R^{2}}\zeta(|\mathcal{O}_{\omega}(x-y)-z|)\zeta(|z|)dzd\omega\nonumber\\
&=\frac{1}{\|\zeta(|\cdot|)\|_{L^{1}}}\int_{\mathbb{S}^{1}}\int_{\R^{2}}\zeta(|x-y-\mathcal{O}_{\omega}^{*}z|)\zeta(|\mathcal{O}_{\omega}^{*}z|)dzd\omega\nonumber\\
&=\frac{1}{\|\zeta(|\cdot|)\|_{L^{1}}}\int_{\R^{2}}\zeta(|x-y-z|)\zeta(|z|)dz.
\end{align}
Moreover, by translation invariance,
\begin{equation}
\varphi(|x-y|) = \|\zeta(|\cdot|)\|_{L^{1}}^{-1}\int_{\R^{2}} \zeta(|x-z|)\zeta(|y-z|)dz.
\end{equation}
Additionally, $\varphi$ is decreasing, $\supp(\varphi) \subset [-(2+8\eta_{2})R,(2+8\eta_{2})R]$, and $\varphi$ satisfies the derivative estimates
\begin{equation}
|\varphi^{(n)}(r)| \lesssim_{n} \frac{1}{(\eta_{2}R)^{n}}, \qquad \forall n\in\N_{0}.
\end{equation}

For later use in the estimation of error terms in subsections \ref{ssec:QS_dfoc} and \ref{ssec:QS_foc}, we decompose $\varphi$ into three pieces $\varphi_{1}, \varphi_{2},\varphi_{3}$ respectively defined by the formulae
\begin{align}
\varphi_{1}(r) &\coloneqq \|\zeta(|\cdot|)\|_{L^{1}}^{-1}\paren*{\int_{\S^{1}}\int_{|z|\leq (1-4\eta_{2})R} \zeta(|r\omega-z|) \zeta(|z|) dzd\omega} 1_{[0,\eta_{2}R]}(r), \\
\varphi_{2}(r) &\coloneqq  \|\zeta(|\cdot|)\|_{L^{1}}^{-1}\paren*{\int_{\S^{1}}\int_{|z|>(1-4\eta_{2})R}\zeta(|r\omega-z|) \zeta(|z|) dzd\omega} 1_{[0,\eta_{2}R]}(r), \\
\varphi_{3}(r) &\coloneqq \|\zeta(|\cdot|)\|_{L^{1}}^{-1}\paren*{\int_{\R^{2}}\zeta(|r\omega-z|) \zeta(|z|) dzd\omega} 1_{(\eta_{2}R,\infty)}(r).
\end{align}

\begin{lemma}[$\varphi_{j}$ properties]\label{lem:vph_est}
The functions $\varphi_{1},\varphi_{2},\varphi_{3}$ satisfy the following properties.
\begin{enumerate}[(i)]
\item\label{item:vph_est_1}
There exists a constant $C_{\eta_{2}}\sim 1$ such that
\begin{equation}
\varphi_{1}(r) = C_{\eta_{2}} 1_{[-\eta_{2}R,\eta_{2}R]}(r).
\end{equation}
\item\label{item:vph_est_2}
We have that
\begin{equation}
\varphi_{2}(r) = \|\zeta(|\cdot|)\|_{L^{1}}^{-1}\paren*{\int_{\S^{1}}\int_{(1-4\eta_{2})R<|z|\leq (1+5\eta_{2})R} \zeta(|r\omega-z|)\zeta(|z|)dzd\omega}1_{[-\eta_{2}R,\eta_{2}R]}(r)
\end{equation}
and
\begin{equation}
|\varphi_{2}'(r)| \lesssim \frac{1}{R}, \qquad \forall |r| <\eta_{2}R.
\end{equation}
\item\label{item:vph_est_3}
We have that
\begin{equation}
\varphi_{3}(r) = \|\zeta(|\cdot|)\|_{L^{1}}^{-1}\paren*{\int_{\S^{1}}\int_{|z|\leq (1+4\eta_{2})R} \zeta(|r\omega-z|)\zeta(|z|)dzd\omega} 1_{(\eta_{2}R,(2+8\eta_{2})R)}(|r|).
\end{equation}
\end{enumerate}
\end{lemma}
\begin{proof}
For \ref{item:vph_est_1}, it is tautological that $\varphi_{1}\equiv 0$ on $(\eta_{2}R,\infty)$. We recall that $\zeta\equiv 1$ on the interval $[-R,R]$. By the triangle inequality,
\begin{equation}
|r\omega-z|\leq R, \qquad \forall (r,z)\in\R\times\R^{2} \enspace \text{such that} \enspace |r|\leq \eta_{2}R \enspace \text{and} \enspace |z|\leq (1-4\eta_{2})R.
\end{equation}
Therefore, for $|r|\leq \eta_{2}R$,
\begin{equation}
\varphi_{1}(r) = \|\zeta(|\cdot|)\|_{L^{1}}^{-1}\int_{\S^{1}}\int_{|z|\leq (1-4\eta_{2}R)}1 dzd\omega = \frac{\pi (1-4\eta_{2})^{2}R^{2}}{\|\zeta(|\cdot|)\|_{L^{1}}} \sim 1,
\end{equation}
since $\|\zeta(|\cdot|)\|_{L^{1}}\sim R^{2}$.

For \ref{item:vph_est_2}, the first assertion is immediate from the triangle inequality and the support properties of $\zeta$. For the second assertion, we use the chain rule to obtain
\begin{equation}
\varphi_{2}'(r) = \|\zeta(|\cdot|)\|_{L^{1}}^{-1}\int_{\S^{1}}\int_{(1-4\eta_{2})R<|z| \leq (1+5\eta_{2})R} \zeta'(|r\omega-z|) \frac{(r\omega-z)\cdot\omega}{|r\omega-z|} \zeta(|z|)dzd\omega, \qquad \forall |r| <\eta_{2}R.
\end{equation}
Since $|\zeta'(r)| \lesssim \frac{1}{\eta_{2}R}$ by Young's inequality and $|\zeta|\leq 1$, it follows that for $|r|<\eta_{2}R$,
\begin{align}
|\varphi_{2}'(r)| &\lesssim \|\zeta(|\cdot|)\|_{L^{1}}^{-1}\int_{\S^{1}}\int_{(1-4\eta_{2})R<|z|\leq (1+5\eta_{2})R} \frac{1}{\eta_{2}R}dzd\omega \lesssim \frac{R}{\|\zeta(|\cdot|)\|_{L^{1}}} \sim \frac{1}{R},
\end{align}
since $\|\zeta(|\cdot|)\|_{L^{1}} \sim R^{2}$.

\ref{item:vph_est_3} is immediate from the triangle inequality and the support properties of $\zeta$.
\end{proof}

Now define the function $\psi_{R,m}(t,\cdot): [0,\infty)\rightarrow [0,\infty)$ by the formula
\begin{equation}
\psi_{R,m}(t,r) \coloneqq \frac{1}{r}\int_{0}^{r}\varphi\paren*{\frac{N_{m}(t)}{R}s}ds.
\end{equation}
Observe that $\psi_{R,m}(t)\geq 0$ and $\p_{r}(r\psi_{R,m}(t,r))=\varphi(\frac{N_{m}(t)r}{R})\geq 0$. Additionally, since $\supp(\varphi)\subset (-(2+8\eta_{2})R,(2+8\eta_{2})R)$, we have that
\begin{equation}
\psi_{R,m}(t,r) = \frac{1}{r}\int_{0}^{\frac{(2+8\eta_{2})R^{2}}{N_{m}(t)}}\varphi\paren*{\frac{N_{m}(t)}{R}s}ds, \qquad \forall t\in [0,\infty), |x|\geq \frac{(2+8\eta_{2})R^{2}}{N_{m}(t)}.
\end{equation}
Moreover, by the chain rule, mean value theorem, and induction, we see that $\psi_{R,m}$ satisfies the derivative estimates
\begin{equation}\label{eq:psi_d_est} %Crude derivative estimates for $\psi_{R,m}$
|\p_{r}^{n}\psi_{R,m}(t,r)| \lesssim_{n} \min\left\{\frac{1}{r^{n}}, \paren*{\frac{N_{m}(t)}{R}}^{n}\frac{1}{(\eta_{2}R)^{n}}\right\}, \qquad \forall n\in\N_{0}.
\end{equation}
We decompose $\psi_{R,m}$ into three pieces $\psi_{R,m,1}, \psi_{R,m,2}, \psi_{R,m,3}$ respectively defined by the formulae
\begin{align}
\psi_{R,m,1}(t,r) &\coloneqq \frac{1}{r}\int_{0}^{r}\varphi_{1}\paren*{\frac{N_{m}(t)s}{R}}ds, \\
\psi_{R,m,2}(t,r) &\coloneqq \frac{1}{r}\int_{0}^{r}\varphi_{2}\paren*{\frac{N_{m}(t)s}{R}}ds, \\
\psi_{R,m,3}(t,r) &\coloneqq \frac{1}{r}\int_{0}^{r}\varphi_{3}\paren*{\frac{N_{m}(t)s}{R}}ds.
\end{align}

\begin{lemma}[$\psi_{R,m,j}$ properties]\label{lem:psi_est}
The functions $\psi_{R,m,1},\psi_{R,m,2},\psi_{R,m,3}$ satisfy the following properties uniformly in $m$.
\begin{enumerate}[(i)]
\item\label{item:psi_est_1}
We have that
\begin{equation}
\psi_{R,m,1}(r) =
\begin{cases}
C_{\eta_{2}}, & {|r|\leq \frac{\eta_{2}R^{2}}{N_{m}(t)}} \\
\frac{1}{r}\int_{0}^{\frac{\eta_{2}R^{2}}{N_{m}(t)}} \varphi_{1}\paren*{\frac{N_{m}(t)s}{R}}ds, & {|r|>\frac{\eta_{2}R^{2}}{N_{m}(t)}}
\end{cases}
\end{equation}
and
\begin{equation}
|(\p_{r}\psi_{R,m,1})(r)| \lesssim \frac{\eta_{2}R^{2}}{r^{2}N_{m}(t)}1_{(\eta_{2}R^{2}/N_{m}(t),\infty)}(r), \qquad \forall r\neq \frac{\eta_{2}R^{2}}{N_{m}(t)},
\end{equation}
where $C_{\eta_{2}}$ is the same constant as in the statement of lemma \ref{lem:vph_est}.
\item\label{item:psi_est_2}
We have that
\begin{equation}
|\psi_{R,m,2}(r)| \lesssim \eta_{2}1_{[0,\eta_{2}R^{2}/N_{m}(t)]}(r) + \frac{\eta_{2}^{2}R^{2}}{rN_{m}(t)}1_{(\eta_{2}R^{2}/N_{m}(t),\infty)}(r)
\end{equation}
and
\begin{equation}
|(\p_{r}\psi_{R,m,2})(r)| \lesssim \frac{N_{m}(t)}{R^{2}}1_{[0,\eta_{2}R^{2}/N_{m}(t)]}(r) + \frac{\eta_{2}^{2}R^{2}}{r^{2}N_{m}(t)}1_{(\eta_{2}R^{2}/N_{m}(t),\infty)}(r), \qquad \forall r\neq \frac{\eta_{2}R^{2}}{N_{m}(t)}.
\end{equation}
\item\label{item:psi_est_3}
We have that
\begin{equation}
\psi_{R,m,3} = 
\begin{cases}
0, & {|r|\leq \frac{\eta_{2}R^{2}}{N_{m}(t)}} \\
\sim 1, & {\frac{\eta_{2}R^{2}}{N_{m}(t)}<r<\frac{(2+8\eta_{2})R^{2}}{N_{m}(t)}} \\
\frac{1}{r}\int_{0}^{\frac{(2+8\eta_{2})R^{2}}{N_{m}(t)}}\varphi_{3}\paren*{\frac{N_{m}(t)s}{R}}ds, & {|r|\geq \frac{(2+8\eta_{2})R^{2}}{N_{m}(t)}}
\end{cases}
\end{equation}
and
\begin{equation}
|(\p_{r}\psi_{R,m,3})(r)| \lesssim \frac{1}{r}1_{(\eta_{2}R^{2}/N_{m}(t),\infty)}(r), \qquad \forall r\neq \frac{\eta_{2}R^{2}}{N_{m}(t)}.
\end{equation}
\end{enumerate}
\end{lemma}
\begin{proof}
For \ref{item:psi_est_1}, the first assertion is immediate from $\varphi_{1}\equiv C_{\eta_{2}}$ on the interval $|r|\leq \eta_{2}R$ and $\varphi_{1}\equiv 0$ on the interval $|r|>\eta_{2}R$. The second assertion follows similarly.

For \ref{item:psi_est_2}, the first assertion follows from the fact $|\varphi_{2}| \lesssim \eta_{2}1_{[-\eta_{2}R,\eta_{2}R]}$. For the second assertion, we use the product rule and fundamental theorem of calculus to obtain that for $r<\frac{\eta_{2}R^{2}}{N_{m}(t)}$,
\begin{align}
(\p_{r}\psi_{R,m,2})(r) &= -\frac{1}{r^{2}}\int_{0}^{r}\varphi_{2}\paren*{\frac{N_{m}(t)s}{R}}ds + \frac{1}{r}\varphi_{2}\paren*{\frac{N_{m}(t)r}{R}} \nonumber\\
&=\frac{1}{r^{2}}\int_{0}^{r}\paren*{\varphi_{2}\paren*{\frac{N_{m}(t)r}{R}}-\varphi_{2}\paren*{\frac{N_{m}(t)s}{R}}}ds.
\end{align}
It now follows from mean value theorem that
\begin{equation}
|(\p_{r}\psi_{R,m,2})(r)| \lesssim \frac{1}{|r|}\|\varphi_{2}'(\frac{N_{m}(t)}{R}\cdot)\|_{L^{\infty}([0,\eta_{2}R^{2}/N_{m}])} \frac{N_{m}(t)|r|}{R} \lesssim \frac{N_{m}(t)}{R^{2}},
\end{equation}
where we use the bound $\|\varphi_{2}'\|_{L^{\infty}([0,\eta_{2}R])} \lesssim \frac{1}{R}$ to obtain the ultimate inequality. The bound for $|(\p_{r}\psi_{R,m,2}(r)|$ in the region $|r|>\frac{\eta_{2}R^{2}}{N_{m}(t)}$ follows from the same argument as in \ref{item:psi_est_1}.

For \ref{item:psi_est_3}, the first assertion is immediate from the support properties of $\varphi_{3}$ and the triangle inequality, and the second assertion follows by the same argument as in \ref{item:psi_est_1}.
\end{proof}

\begin{mydef}[Vector potential $a_{R,m}$]
We define the vector potential $a_{R,m}(t,\cdot)$ by the formula
\begin{equation}
a_{R,m}(t,x-y) \coloneqq \psi_{R,m}(t,|x-y|)N_{m}(t)(x-y), \qquad \forall (x,y,t)\in\R^{2}\times\R^{2}\times [0,\infty).
\end{equation}
We denote the $j^{th}$ component of $a_{R,m}$ by $a_{R,m,j}$. We note that $a_{R,m}$ is odd in the spatial variable and remark that this property is crucial for exploiting Galilean invariance, as observed in \cite{Dodson2015} and \cite{Dodson2016}.
\end{mydef}
We record some basic estimates for the potential $a_{R,m}$ in the next lemma, which will be used extensively in the sequel.

\begin{lemma}[$a_{R,m}$ properties]\label{lem:pot_prop} %Properties of the potential
The maps $a_{R,m}$ satisfy the following properties.
\begin{enumerate}[(i)]
\item\label{item:pot_prop_1}
$\sup_{m\geq 1} \|a_{R,m}\|_{L_{t,x}^{\infty}([0,\infty)\times\R^{2})} \lesssim R^{2}$
\item\label{item:pot_prop_2}
\begin{equation}
\sup_{m\geq 1}\sup_{t\in [0,\infty)}|\nabla a_{R,m}(t,x-y)| \lesssim \frac{R^{2}}{|x-y|}, \qquad \forall (x,y,t) \in \R^{2}\times\R^{2}\times [0,\infty)
\end{equation}
\item\label{item:pot_prop_3}
\begin{equation}
(\nabla\cdot a_{R,m})(t,x-y) = N_{m}(t)\paren*{\varphi\left(\frac{N_{m}(t)|x-y|}{R}\right) + \psi_{R,m}(t,|x-y|)} \geq 0, \qquad \forall (x,y,t)\in\R^{2}\times\R^{2}\times[0,\infty).
\end{equation}
\end{enumerate}
\end{lemma}
\begin{proof}
\ref{item:pot_prop_1} is immediate from the bound $|\psi_{R,m}(t,r)| \leq 1$ in the region $|r|\leq \frac{(2+8\eta_{2})R^{2}}{N_{m}(t)}$ and the bound $|\psi_{R,m}(t,r)| \leq \frac{(2+8\eta_{2})R^{2}}{rN_{m}(t)}$ in the region $|r|>\frac{(2+8\eta_{2})R^{2}}{N_{m}(t)}$.

For \ref{item:pot_prop_2}, we note from the calculus that
\begin{equation}
\partial_{j} a_{R,m,k}(t,x-y) = \psi_{R,m}'(t,|x-y|)\frac{(x-y)_{j}}{|x-y|}N_{m}(t)(x-y)_{k} + \delta_{jk}\psi_{R,m}(t,|x-y|)N_{m}(t).
\end{equation}
The desired conclusion then follows by using the bounds for $\psi_{R,m}$ and $\p_{r}\psi_{R,m}$ given by lemma \ref{lem:psi_est}.

For \ref{item:pot_prop_3}, we note from the calculus that
\begin{align}
\p_{j}a_{R,m,j}(t,x-y) &= (\p_{r}\psi_{R,m})(t,|x-y|) \frac{(x-y)_{j}^{2}}{|x-y|}N_{m}(t) + \psi_{R,m}(t,|x-y|)N_{m}(t) \nonumber\\
&=\paren*{(\p_{r}\psi_{R,m})(t,r)r+2\psi_{R,m}(t,r)}N_{m}(t),
\end{align}
with $r\coloneqq |x-y|$. The desired conclusion then follows from the identity $\frac{d}{dr}(r\psi_{R,m}(t,r)) = \varphi(\frac{N_{m}(t)r}{R}) + \psi_{R,m}(t,r)$.
\end{proof}

Now set $w\coloneqq P_{\leq K}u$, which satisfies the equation
\begin{equation}\label{eq:IM_eq_w}
(i\partial_{t}+\Delta)w = F(w) + P_{\leq K}F(u)-F(w) =: F(w)+\mathcal{N}.
\end{equation}
We define the interaction Morawetz functional $M_{R,m}$ by the formula
\begin{equation}
M_{R,m}(t) \coloneqq 2\int_{\R^{4}}a_{R,m}(t,x-y)\cdot |w(t,y)|^{2}\Im{\bar{w}\nabla w}(t,x)dxdy, \qquad \forall t\geq 0.
\end{equation}
Differentiating $M_{R,m}$ with respect to time and using the local conservation laws of the eeDS equation as in the proofs of propositions \ref{prop:ibs_1}, \ref{prop:ibs_2}, and \ref{prop:ibs_3}, we see that
\begin{align}
\frac{d}{dt} M_{R,m}(t) &= -\int_{\R^{4}}a_{R,m,j}(t,x-y)|w(t,y)|^{2}\partial_{k}L_{jk}^{w}(t,x)dxdy\\
&\phantom{=}-2\int_{\R^{4}}a_{R,m,j}(t,x-y)\partial_{k}T_{0k}^{w}(t,y)\Im{\bar{w}\partial_{j} w}(t,x)dxdy\\
&\phantom{=}-\int_{\R^{4}}a_{R,m,j}(t,x-y)|w(t,y)|^{2}\partial_{k}T_{jk}^{w}(t,x)dx dy dt\\
&\phantom{=}+2\int_{\R^{4}}(\partial_{t}a_{R,m})(t,x-y)\cdot |w(t,y)|^{2}\Im{\bar{w}\nabla w}(t,x)dxdy\\
&\phantom{=}+2\int_{\R^{4}}a_{R,m,j}(t,x-y)|w(t,y)|^{2}\Re{+\bar{\mathcal{N}}\partial_{j}w - \bar{w}\partial_{j}\mathcal{N}}(t,x)dxdy\\
&\phantom{=}+4\int_{\R^{4}}a_{R,m,j}(t,x-y) \Im{\bar{w}\mathcal{N}}(t,y) \Im{\bar{w}\partial_{j}w}(t,x)dx dy.
\end{align}
Unpackaging the definition of $L_{jk}^{w}$ and then integrating with respect to time over the interval $[0,T]$ and using the fundamental theorem of calculus, we obtain that
\begin{align}
M_{R,m}(t) &= -4\int_{0}^{T}\int_{\R^{4}}a_{R,m,j}(t,x-y)|w(t,y)|^{2}\partial_{k}\Re{\ol{\partial_{k}w}\partial_{j}w}(t,x)dxdy dt\label{eq:IM_main1}\\
&\phantom{=}-2\int_{0}^{T}\int_{\R^{4}}a_{R,m,j}(t,x-y)\partial_{k}T_{0k}^{w}(t,y)\Im{\bar{w}\partial_{j} w}(t,x)dxdy dt\label{eq:IM__main2}\\
&\phantom{=}-\int_{0}^{T}\int_{\R^{4}}a_{R,m,j}(t,x-y)|w(t,y)|^{2}\partial_{k}T_{jk}^{w}(t,x)dxdy dt\label{eq:IM_nl}\\
&\phantom{=}+\int_{0}^{T}\int_{\R^{4}}a_{R,m,j}(t,x-y)|w(t,y)|^{2}\partial_{j}\partial_{k}^{2}(|w|^{2})(t,x)dxdy dt\label{eq:IM_lerr}\\
&\phantom{=}+2\int_{0}^{T}\int_{\R^{4}}(\partial_{t}a_{R,m})(t,x-y)\cdot |w(t,y)|^{2}\Im{\bar{w}\nabla w}(t,x)dx dy dt\label{eq:IM_terr}\\
&\phantom{=}+2\int_{0}^{T}\int_{\R^{4}}a_{R,m,j}(t,x-y)|w(t,y)|^{2}\Re{+\bar{\mathcal{N}}\partial_{j}w - \bar{w}\partial_{j}\mathcal{N}}(t,x)dx dy dt\label{eq:IM_err1}\\
&\phantom{=}+4\int_{0}^{T}\int_{\R^{4}}a_{R,m,j}(t,x-y) \Im{\bar{w}\mathcal{N}}(t,y) \Im{\bar{w}\partial_{j}w}(t,x)dxdy dt.\label{eq:IM_err2}
\end{align}

\subsection{Frequency localization error estimate}\label{ssec:QS_fl_err}
In this subsection, we prove an estimate which we will use to bound the quantity $|\eqref{eq:IM_err1}+\eqref{eq:IM_err2}|$, which arises because the frequency-localized solution $w$ satisfies the approximate eeDS equation \eqref{eq:IM_eq_w}.

Given an admissible blowup solution $u$ and a  $u$-admissible tuple $(\epsilon_{1},\epsilon_{2},\epsilon_{3})$, we set $\lambda \coloneqq \frac{\epsilon_{3}2^{k_{0}}}{K}$, where $K=\int_{0}^{T}N(t)^{3}dt$ and rescale the solution $u$ by defining $u_{\lambda}\coloneqq \lambda u(\lambda^{2}\cdot, \lambda\cdot)$. We denote the frequency truncation of the rescaled solution $u_{\lambda}$ by $w\coloneqq P_{\leq k_{0}}u_{\lambda}$, where we suppress the dependence of $w$ on $\lambda$ for convenience.

\begin{prop}[Frequency truncation error estimate]\label{prop:IM_frq_err} %Error Estimates for the Interaction Morawetz Inequality
Let $a : [0, \infty)\times\R^{2}\rightarrow\R^{2}$ be a map which is odd in space (i.e. $a(t,x)=-a(t,-x)$) and for which there exists a constant $C(a)>0$ such that
\begin{equation}
\|a\|_{L_{t,x}^{\infty}([0,\infty)\times\R^{2})} \leq C(a), \quad \sup_{x\in\R^{2}} |x|\|\nabla a(x)\|_{L_{t}^{\infty}([0,\infty)}\leq C(a).
\end{equation}
Let $u$ be an admissible blowup solution. Then there exist constants $C(u)>0$ and $\epsilon_{1}(u)\geq \epsilon_{2}(u)\geq \epsilon_{3}(u)>0$ such that the following holds: for all $0<\eta\ll 1$, there exists a constant $T_{0}(\eta)>0$, such that for all admissible tuples $(\epsilon_{1},\epsilon_{2},\epsilon_{3})$, with $\epsilon_{j}\leq \epsilon_{j}(u)$ for $j=1,2,3$, and $T\geq T_{0}(\eta)$ satisfying
\begin{equation}
\int_{0}^{T}N(t)^{3}dt=K \enspace \text{and} \enspace \int_{0}^{T}\int_{\R^{2}}|u(t,x)|^{4}dxdt=2^{k_{0}} \enspace \text{for some} \enspace k_{0}\in N,
\end{equation}
we have the estimate
\begin{align}
&2\left|\int_{0}^{T/\lambda^{2}}\int_{\R^{4}}a(t,x-y)\cdot\Im{\bar{w}\mathcal{N}}(t,y)\Im{\bar{w}(\nabla-i\xi_{\lambda}(t))w}(t,x)dxdydt\right| \label{eq:IM_frq_err1}\\
&\phantom{=}+\left|\int_{0}^{T/\lambda^{2}}\int_{\R^{4}}a(t,x-y)\cdot |w(t,y)|^{2}\Re{\bar{\mathcal{N}}(\nabla-i\xi_{\lambda}(t))w}(t,x)dxdy dt\right| \label{eq:IM_frq_err2}\\
&\phantom{=}+\left|\int_{0}^{T/\lambda^{2}}\int_{\R^{4}}a(t,x-y)\cdot |w(t,y)|^{2}\Re{\bar{w}(\nabla-i\xi_{\lambda}(t))\mathcal{N}}(t,x)dxdy dt\right| \label{eq:IM_frq_err3}\\
&\leq C(a)C(u)\left(C(\eta)\frac{\epsilon_{3}}{K}+\eta\right)^{1/6}2^{k_{0}}.
\end{align}
\end{prop}

Before diving into the proof of proposition \ref{prop:IM_frq_err}, we briefly comment on the strategy. The idea is to perform a similar analysis of the error terms \eqref{eq:IM_frq_err1}, \eqref{eq:IM_frq_err2}, and \eqref{eq:IM_frq_err3} as in the proofs of propositions \ref{prop:ibs_1}, \ref{prop:ibs_2}, and \ref{prop:ibs_3}, except now the RHSs of our estimates will contain the norms $\|u_{\lambda}\|_{\tilde{X}_{k_{0}}([0,\lambda^{-2}T]\times\R^{2})}$. Fortunately, $u_{\lambda}$ is in the form of our long-time Strichartz estimate (theorem \ref{thm:LTSE}). Therefore, we can estimate $\|u_{\lambda}\|_{\tilde{X}_{k_{0}}([0,\lambda^{-2}T]\times\R^{2})}\lesssim_{u}1$ uniformly in the data $([0,T],k_{0},\epsilon_{1},\epsilon_{2},\epsilon_{3})$, provided that $T>0$ is sufficiently large depending on $\eta$ and $\epsilon_{j}\leq \epsilon_{j}(u)$, where the constants $\epsilon_{j}(u)$ are as in the statement of theorem \ref{thm:LTSE}.

\begin{proof}
We estimate the terms \eqref{eq:IM_frq_err1}, \eqref{eq:IM_frq_err2}, and \eqref{eq:IM_frq_err3} separately.

\begin{description}[leftmargin=*]
\item[Estimate for \eqref{eq:IM_frq_err1}:]
We first perform a frequency decomposition $u_{\lambda} = u_{\lambda,l} + u_{\lambda, h}$, where $u_{\lambda, l} \coloneqq P_{\leq k_{0}-5}u_{\lambda}$, and write
\begin{equation}
\Im{\bar{w}\mathcal{N}} = F_{0} + F_{1} + F_{2} + F_{3} + F_{4},
\end{equation}
where $F_{j}$ consists of $4-j$ factors $u_{\lambda, l}$ and $j$ factors $u_{\lambda, h}$, for $j=0,1,\ldots,4$ (see the proof of proposition \ref{prop:ibs_2} for the explicit formulae for the $F_{j}$).

Quick Fourier support analysis shows that $F_{0}=0$, hence there is no contribution.

Now note that by scaling invariance,
\begin{equation}
\int_{0}^{\lambda^{-2}T}\paren*{N_{\lambda}(t)^{3}+\epsilon_{3}\|u_{\lambda}(t)\|_{L_{x}^{4}(\R^{2})}^{4}}dt=2^{k_{0}+1}\epsilon_{3},
\end{equation}
and so $[0,\lambda^{-2}T]=G_{0}^{k_{0}}$. Hence, by the estimates for $F_{2},F_{3},F_{4}$ obtained in the proof of proposition \ref{prop:ibs_2}, we have that
\begin{equation}
\|F_{2}+F_{3}+F_{4}\|_{L_{t,x}^{1}([0,\lambda^{-2}T]\times\R^{2})} \lesssim \paren*{1+\|u_{\lambda}\|_{\tilde{X}_{k_{0}}([0,\lambda^{-2}T]\times\R^{2})}^{6}} \lesssim_{u} 1.
\end{equation}
Therefore by Cauchy-Schwarz, mass conservation, and the estimates for the map $a$, we obtain that
\begin{align}
&\left|\int_{0}^{T/\lambda^{2}}\int_{\R^{4}}a(t,x-y)\cdot\left(F_{1}+F_{2}+F_{3}\right)(y)\Im{\bar{w}(\nabla-i\xi_{\lambda}(t))w}(t,x)dxdydt\right| \nonumber\\
&\leq \|F_{2}+F_{3}+F_{4}\|_{L_{t,x}^{1}([0,\lambda^{-2}T]\times\R^{2})}]\|a\|_{L_{t,x}^{\infty}([0,\lambda^{-2}T]\times\R^{2})}\|\bar{w}(\nabla-i\xi_{\lambda}(t))w\|_{L_{t}^{\infty}L_{x}^{1}([0,\lambda^{-2}T]\times\R^{2})} \nonumber\\
&\lesssim_{u} C(a) \|(\nabla-i\xi_{\lambda}(t))w\|_{L_{t}^{\infty}L_{x}^{2}([0,\lambda^{-2}T]\times\R^{2})}.
\end{align}
To show that the RHS of the preceding inequality can be made small as measured by the parameter $\eta$, we first decompose $w$ by
\begin{equation}
w = P_{\xi_{\lambda}(t),\leq C(\eta)N_{\lambda}(t)}w + P_{\xi_{\lambda}(t),>C(\eta)N_{\lambda}(t)}w.
\end{equation}
Since $N(t) \leq 1$ by assumption that $u$ is an admissible blowup solution, it follows from scaling invariance that $N_{\lambda}(t)\leq \frac{\epsilon_{3}2^{k_{0}}}{K}$. Therefore by the triangle inequality, the frequency localization property \eqref{eq:sp_frq_loc}, Bernstein's lemma, and mass conservation,
\begin{align}
\|(\nabla-i\xi_{\lambda}(t))w\|_{L_{t}^{\infty}L_{x}^{2}([0,\lambda^{-2}T]\times\R^{2})}&\leq \|(\nabla-i\xi_{\lambda}(t))P_{\xi_{\lambda}(t),\leq C(\eta)N_{\lambda}(t)}w\|_{L_{t}^{\infty}L_{x}^{2}([0,\lambda^{-2}T]\times\R^{2})} \nonumber\\
&\phantom{=}+ \|(\nabla-i\xi_{\lambda}(t))P_{\xi_{\lambda}(t), >C(\eta)N_{\lambda}(t)}w\|_{L_{t}^{\infty}L_{x}^{2}([0,\lambda^{-2}T]\times\R^{2})} \nonumber\\
&\lesssim C(\eta)\|N_{\lambda}\|_{L_{t}^{\infty}} + \eta 2^{k_{0}} \nonumber\\
&\leq C(\eta)\frac{\epsilon_{3}2^{k_{0}}}{K} + \eta 2^{k_{0}} \label{eq:IM_err_d_est}
\end{align}

It remains to estimate the contribution from $F_{1}$. Hereafter, the mixed norm notation $L_{t}^{p}L_{x}^{q}$ will be taken over the spacetime slab $[0,\lambda^{-2}T]\times\R^{2}$. Recall from the proof of proposition \ref{prop:ibs_2} that $F_{1}$ has Fourier support in the region $\{|\xi| \geq 2^{k_{0}-4}\}$. Therefore the Fourier multiplier $\Delta^{-1}$ is well-defined on the Fourier support of $F_{1}$, and we may integrate by parts in the $y$-variable to obtain that
\begin{equation}
\begin{split}
\int_{0}^{T/\lambda^{2}}&\int_{\R^{4}}\frac{\Delta_{y}}{\Delta_{y}}F_{1}(t,y)a(t,x-y)\cdot\Im{\bar{w}(\nabla-i\xi_{\lambda}(t))w}(t,x)dxdydt\\
&= \int_{0}^{T/\lambda^{2}}\int_{\R^{4}}\left(\frac{\partial_{k}}{\Delta} F_{1}\right)(t,y)(\partial_{k}a)(t,x-y)\cdot\Im{\bar{w}(\nabla-i\xi_{\lambda}(t))w}(t,x)dxdydt.
\end{split}
\end{equation}
Using the estimate for $\nabla a$ together with mass conservation, H\"{o}lder's inequality, Hardy-Littlewood-Sobolev and Bernstein's lemmas, we see that
\begin{align}
&\left|\int_{0}^{T/\lambda^{2}}\int_{\R^{4}}\left(\frac{\partial_{k}}{\Delta}F_{1}\right)(t,y)(\partial_{k}a)(t,x-y)\cdot\Im{\bar{w}(\nabla-i\xi_{\lambda}(t))w}(t,x)dxdydt\right| \nonumber\\
&\lesssim C(a)\|\frac{\nabla}{\Delta}F_{1}\|_{L_{t,x}^{4/3}} \||\nabla|^{-1}[\bar{w}(\nabla-i\xi_{\lambda}(t))w]\|_{L_{t,x}^{4}} \nonumber\\
&\lesssim C(a) 2^{-k_{0}}\|F_{1}\|_{L_{t,x}^{4/3}} \|\bar{w}(\nabla-i\xi_{\lambda}(t))w\|_{L_{t}^{4}L_{x}^{4/3}} \nonumber\\
&\lesssim C(a) 2^{-k_{0}} \|F_{1}\|_{L_{t,x}^{4/3}} \|(\nabla-i\xi_{\lambda}(t))w\|_{L_{t,x}^{4}}.
\end{align}
By interpolation, the frequency localization property \eqref{eq:sp_frq_loc}, lemma \ref{lem:BT_embed}, and the estimate \eqref{eq:IM_err_d_est}, we see that
\begin{equation}
\begin{split}
\|(\nabla-i\xi_{\lambda}(t))w\|_{L_{t,x}^{4}} \leq \|(\nabla-i\xi_{\lambda}(t))w\|_{L_{t}^{3}L_{x}^{6}}^{3/4} \|(\nabla-i\xi_{\lambda}(t))w\|_{L_{t}^{\infty}L_{x}^{2}}^{1/4} &\lesssim 2^{\frac{3k_{0}}{4}} \|u_{\lambda}\|_{\tilde{X}_{k_{0}}([0,\lambda^{-2}T]\times\R^{2})}^{3/4} \left(C(\eta)\frac{\epsilon_{3}2^{k_{0}}}{K} + \eta 2^{k_{0}}\right)^{1/4}\\
&\lesssim_{u} 2^{k_{0}} \left(C(\eta)\frac{\epsilon_{3}}{K} + \eta\right)^{1/4}.
\end{split}
\end{equation}
Therefore,
\begin{equation}
C(a)2^{-k_{0}} \|F_{1}\|_{L_{t,x}^{4/3}} \|(\nabla-i\xi_{\lambda}(t))w\|_{L_{t,x}^{4}} \lesssim_{u} C(a)\|F_{1}\|_{L_{t,x}^{4/3}}\left(C(\eta)\frac{\epsilon_{3}}{K} + \eta\right)^{1/4},
\end{equation}
so it remains to estimate $\|F_{1}\|_{L_{t,x}^{4/3}}$. Observe from the triangle inequality that
\begin{align}
\|F_{1}\|_{L_{t,x}^{4/3}} &\leq 2\|u_{\lambda,l}P_{\leq k_{0}}\left[\E\left(\Re {(P_{>k_{0}-2}u_{\lambda,h})\bar{u}_{\lambda,l}}\right)\right]\|_{L_{t,x}^{4/3}}+\|\bar{u}_{\lambda,l}P_{\leq k_{0}}[\E(|u_{\lambda,l}|^{2})(P_{>k_{0}-2}u_{\lambda,h})]\|_{L_{t,x}^{4/3}} \nonumber\\
&\phantom{=}+\|\ol{(P_{k_{0}-2\leq\cdot\leq k_{0}}u_{\lambda,h})} \E(|u_{\lambda,l}|^{2})u_{\lambda,l}\|_{L_{t,x}^{4/3}}.
\end{align}
So by H\"{o}lder's inequality, Calder\'{o}n-Zygmund theorem, and lemma \ref{lem:BT_embed}, it follows that
\begin{equation}
\|F_{1}\|_{L_{t,x}^{4/3}} \lesssim \|P_{>k_{0}-2}u_{\lambda}\|_{L_{t,x}^{4}} \|u_{\lambda,l}\|_{L_{t,x}^{6}}^{3} \lesssim 2^{k_{0}} \|u\|_{\tilde{X}_{k_{0}}([0,\lambda^{-2}T]\times\R^{2})}^{4} \lesssim_{u} 2^{k_{0}}.
\end{equation}

Bookkeeping our estimates, we conclude that
\begin{equation}
\eqref{eq:IM_frq_err1} \lesssim_{u} C(a)2^{k_{0}}\left(C(\eta)\frac{\epsilon_{3}}{K}+\eta\right)^{1/4}.
\end{equation}

\item[Estimate for \eqref{eq:IM_frq_err2}:]
By H\"{o}lder's inequality, Plancherel's theorem, and mass conservation,
\begin{equation}
|\eqref{eq:IM_frq_err2}| \lesssim \|a\|_{L_{t,x}^{\infty}} \|w\|_{L_{t}^{\infty}L_{x}^{2}}^{2} \|\mathcal{N}(\nabla-i\xi_{\lambda}(t))w\|_{L_{t,x}^{1}} \lesssim C(a)\|\mathcal{N}\|_{L_{t}^{3/2}L_{x}^{6/5}} \|(\nabla-i\xi_{\lambda}(t))w\|_{L_{t}^{3}L_{x}^{6}}.
\end{equation}
Interpolating between the admissible pairs $(5/2,10)$ and $(\infty,2)$ to get $(3,6)$ and using mass conservation, we see that
\begin{equation}
\|(\nabla-i\xi_{\lambda}(t))w\|_{L_{t}^{3}L_{x}^{6}} \lesssim 2^{k_{0}}\left(C(\eta)\frac{\epsilon_{3}}{K}+\eta\right)^{1/6},
\end{equation}
which implies that
\begin{equation}
|\eqref{eq:IM_frq_err2}| \lesssim C(a)2^{k_{0}}\left(C(\eta)\frac{\epsilon_{3}}{K}+\eta\right)^{1/6}\|\mathcal{N}\|_{L_{t}^{3/2}L_{x}^{6/5}}.
\end{equation}

We next claim that $\|\mathcal{N}\|_{L_{t}^{3/2}L_{x}^{6/5}} \lesssim_{u} 1$. Indeed, recall that $\mathcal{N}=P_{\leq k_{0}}F(u_{\lambda})-F(w)$. Decomposing $u_{\lambda}=u_{\lambda,l}+u_{\lambda,h}$ as above, expanding the commutator $\mathcal{N}$, and grouping like terms, we have that
\begin{equation}
\mathcal{N}=\mathcal{N}_{0}+\mathcal{N}_{1}+\mathcal{N}_{2}+\mathcal{N}_{3},
\end{equation}
where $\mathcal{N}_{j}$ contains $j$ factors $u_{\lambda,h}$ and $4-j$ factors $u_{\lambda,l}$.

Fourier support analysis shows that $\mathcal{N}_{0}=0$.

We have that
\begin{align}
\mathcal{N}_{1} &= \comm{P_{\leq k_{0}}}{\E(|u_{\lambda,l}|^{2})}(u_{\lambda,h})+2\paren*{P_{\leq k_{0}}[\E\paren*{\Re{u_{\lambda,l}\bar{u}_{\lambda,h}}}u_{\lambda,l}] - \E\paren*{\Re{u_{\lambda,l}\ol{(P_{\leq k_{0}}u_{\lambda,h})}}}u_{\lambda,l}},
\end{align}
so by the triangle inequality,
\begin{align}
\|\mathcal{N}_{1}\|_{L_{t}^{3/2}L_{x}^{6/5}} &\lesssim \|\comm{P_{\leq k_{0}}}{\E(|u_{\lambda,l}|^{2})}(u_{\lambda,h})\|_{L_{t}^{3/2}L_{x}^{6/5}} + \|P_{\leq k_{0}}[\E(\Re{u_{\lambda,l}\bar{u}_{\lambda,h}})u_{\lambda,l}] - \E(\Re{u_{\lambda,l}\ol{(P_{\leq k_{0}}u_{\lambda,h})}})u_{\lambda,l}\|_{L_{t}^{3/2}L_{x}^{6/5}} \nonumber\\
&\eqqcolon \mathrm{Term}_{1}+\mathrm{Term}_{2}.
\end{align}
By the fundamental theorem of calculus, Minkowski's inequality, H\"{o}lder's inequality, Bernstein's lemma, Calder\'{o}n-Zygmund theorem, and mass conservation, we have that
\begin{align}
\|\mathrm{Term}_{1}\|_{L_{t}^{3/2}L_{x}^{6/5}} \lesssim 2^{-k_{0}} \|\nabla \E(|u_{\lambda,l}|^{2})\|_{L_{t}^{3}L_{x}^{3/2}}\|u_{\lambda,h}\|_{L_{t}^{3}L_{x}^{6}}\lesssim 2^{-k_{0}}2^{k_{0}}\|u_{\lambda}\|_{\tilde{X}_{k_{0}}([0,\lambda^{-2}T]\times\R^{2})}^{3}\lesssim_{u} 1,
\end{align}
where we use lemmas \ref{lem:lohi_embed} and \ref{lem:BT_embed} to obtain the penultimate inequality. By the triangle inequality,
\begin{align}
\|\mathrm{Term}_{2}\|_{L_{t}^{3/2}L_{x}^{6/5}}& \leq \|\comm{P_{\leq k_{0}}}{u_{\lambda,l}}\E\paren*{\Re{u_{\lambda,l}\bar{u}_{\lambda,h}}}\|_{L_{t}^{3/2}L_{x}^{6/5}} + \|\E\paren*{\comm{P_{\leq k_{0}}}{u_{\lambda,l}}(\bar{u}_{\lambda,h})}u_{\lambda,l}\|_{L_{t}^{3/2}L_{x}^{6/5}}\nonumber\\
&\eqqcolon \mathrm{Term}_{2,1}+\mathrm{Term}_{2,2}.&
\end{align}
The estimate $\mathrm{Term}_{2,1}\lesssim_{u} 1$ follows from the argument for the estimate for $\mathrm{Term}_{1}$. To estimate $\mathrm{Term}_{2,2}$, we argue similarly as to in section \ref{sec:BSE} to obtain that
\begin{align}
\mathrm{Term}_{1,2} &= \|u_{\lambda,l}\E\left(\int_{\R^{2}} 2^{2k_{0}}\phi^{\vee}(2^{k_{0}}y)(\tau_{y}\bar{u}_{\lambda,h})\left(\int_{0}^{1} (\tau_{\theta y}\nabla u_{\lambda,l})\cdot (-y) d\theta\right)dy\right)\|_{L_{t}^{3/2}L_{x}^{6/5}} \nonumber\\
&\lesssim 2^{-k_{0}}\int_{0}^{1}\int_{\R^{2}} 2^{2k_{0}}|\phi^{\vee}(2^{k_{0}}y)| |2^{k_{0}}y| \|(\tau_{\theta y}\nabla u_{\lambda,l})(\tau_{y}\bar{u}_{\lambda,h})\|_{L_{t}^{3/2}L_{x}^{3}}dyd\theta \nonumber\\
&\lesssim_{u} 1.
\end{align}
We have the final estimate $\|\mathcal{N}_{1}\|_{L_{t}^{3/2}L_{x}^{6/5}} \lesssim_{u} 1$.

To estimate $\|\mathcal{N}_{2}+\mathcal{N}_{3}\|_{L_{t}^{3/2}L_{x}^{6/5}}$, we use triangle and H\"{o}lder's inequalities, Calder\'{o}n-Zygmund theorem, followed by applications of lemma \ref{lem:BT_embed} on two of the factors $u_{\lambda,h}$ and mass conservation on the remaining factors to obtain the final estimate
\begin{equation}
\|\mathcal{N}_{2}+\mathcal{N}_{3}\|_{L_{t}^{3/2}L_{x}^{6/5}} \lesssim_{u} 1. 
\end{equation}

Bookkeeping our estimates, we have shown that
\begin{equation}
\eqref{eq:IM_frq_err2} \lesssim_{u} C(a)2^{k_{0}}\left(C(\eta)\frac{\epsilon_{3}}{K}+\eta\right)^{1/6}.
\end{equation}

\item[Estimate for \eqref{eq:IM_frq_err3}:]
We integrate by parts in $x$ and use with the product rule to obtain
\begin{align}
\eqref{eq:IM_frq_err3} &\leq \left|\int_{0}^{T/\lambda^{2}}\int_{\R^{4}}(\nabla \cdot a)(t,x-y) |w(t,y)|^{2} \Re{\bar{w}\mathcal{N}}(t,x)dxdydt\right| \nonumber\\
&\phantom{=}+\left|\int_{0}^{T/\lambda^{2}}\int_{\R^{4}} a(t,x-y)\cdot |w(t,y)|^{2} \Re{(\nabla-i\xi_{\lambda}(t))w\bar{\mathcal{N}}}(t,x)dxdydt\right| \nonumber\\
&= \eqref{eq:IM_frq_err2} +\left|\int_{0}^{T/\lambda^{2}}\int_{\R^{4}}(\nabla \cdot a)(t,x-y) |w(t,y)|^{2} \Re{\bar{w}\mathcal{N}}(t,x)dxdydt\right| \label{eq:IM_err3_red}
\end{align}
Now once again decompose $\mathcal{N}=\mathcal{N}_{0}+\cdots+\mathcal{N}_{3}$, and use the triangle inequality to to reduce to estimating the contributions of each of the $\mathcal{N}_{j}$	separately.

Fourier support analysis shows that $\mathcal{N}_{0}=0$, and hence there is no corresponding contribution to the second term in \eqref{eq:IM_err3_red}.

We next estimate the contribution of $\mathcal{N}_{3}$. By the estimate for $\nabla a$ together with H\"{o}lder's inequality in space, Hardy-Littlewood-Sobolev lemma, Calder\'{o}n-Zygmund theorem, followed by interpolation and mass conservation, we have that
\begin{align}
&\left|\int_{0}^{T/\lambda^{2}}\int_{\R^{4}}(\nabla \cdot a)(t,x-y) |w(t,y)|^{2} \Re{\bar{w}\mathcal{N}_{3}}(t,x)dxdydt\right|\nonumber\\
&\lesssim C(a)\int_{0}^{T/\lambda^{2}} \| |w(t)|^{2}\|_{L_{x}^{3/2}} \| |\nabla|^{-1}\left( \left(w\E(|u_{\lambda,h}|^{2}u_{\lambda,h}\right)(t)\right)\|_{L_{x}^{3}}dt \nonumber\\
&\lesssim  C(a)\int_{0}^{T/\lambda^{2}} \|w(t)\|_{L_{x}^{3}}^{2} \|(w\E(|u_{\lambda,h}|^{2})u_{\lambda,h})(t)\|_{L_{x}^{6/5}}dt \nonumber\\
&\lesssim  C(a)\int_{0}^{T/\lambda^{2}} \|w(t)\|_{L_{x}^{3}}^{2} \|w(t)\|_{L_{x}^{12}}\|u_{\lambda,h}(t)\|_{L_{x}^{4}}^{3}dt \nonumber\\
&\lesssim C(a)\int_{0}^{T/\lambda^{2}} \|w(t)\|_{L_{x}^{12}}^{9/5} \|u_{\lambda,h}(t)\|_{L_{x}^{4}}^{3}dt.
\end{align}
Now by H\"{o}lder's inequality in time together with the estimate $\|w\|_{L_{t}^{36/5}L_{x}^{12}} \lesssim  2^{5k_{0}/9}\|u_{\lambda,h}\|_{\tilde{X}_{k_{0}}([0,T/\lambda^{2}]\times\R^{2})}$ given by lemma \ref{lem:BT_embed} (here we use the admissibility of the pair $(36/5,36/13)$) and the estimate
\begin{align}
\|u_{\lambda,h}\|_{L_{t,x}^{4}} \lesssim_{u} \|u_{\lambda,h}\|_{L_{t}^{\infty}L_{x}^{2}}^{1/4} &\lesssim \left(2^{-k_{0}}\| (\nabla-i\xi_{\lambda}(t)) P_{\xi_{\lambda}(t),\leq C(\eta)N_{\lambda}(t)} u_{\lambda,h}\|_{L_{t}^{\infty}L_{x}^{2}}+ \|P_{\xi_{\lambda}(t), >C(\eta)N_{\lambda}(t)}u_{\lambda,h}\|_{L_{t}^{\infty}L_{x}^{2}}\right)^{1/4} \nonumber\\
&\lesssim \left(C(\eta)\frac{\epsilon_{3}}{K}+\eta\right)^{1/4}
\end{align}
given by the frequency localization property \ref{eq:sp_frq_loc} and $|\xi_{\lambda}(t)\| \leq \epsilon_{3}^{1/2}2^{-19+k_{0}}$ for $t\in [0,T]$, it follows that
\begin{equation}
\int_{0}^{T/\lambda^{2}} \|w(t)\|_{L_{x}^{12}}^{9/5} \|u_{\lambda,h}(t)\|_{L_{x}^{4}}^{3}dt \lesssim \|w\|_{L_{t}^{36/5}L_{x}^{12}}^{9/5}\|u_{\lambda,h}\|_{L_{t,x}^{4}}^{3} \lesssim_{u} 2^{k_{0}}\left(C(\eta)\frac{\epsilon_{3}}{K}+\eta\right)^{3/4}.
\end{equation}
Hence, we have shown that
\begin{equation}
\left|\int_{0}^{T/\lambda^{2}}\int_{\R^{4}}(\nabla \cdot a)(t,x-y) |w(t,y)|^{2} \Re{\bar{w}\mathcal{N}_{3}}(t,x)dxdydt\right| \lesssim_{u} 2^{k_{0}}C(a)\left(C(\eta)\frac{\epsilon_{3}}{K}+\eta\right)^{3/4},
\end{equation}
which completes the estimate for the contribution of $\mathcal{N}_{3}$.

We next estimate the contribution of $\mathcal{N}_{2}$. Since the contributions of the other terms in $\mathcal{N}_{2}$ may be estimated by similar arguments, without loss of generality we may assume that $\mathcal{N}_{2}=\E(\bar{u}_{\lambda,l}u_{\lambda,h})u_{\lambda,h}$. By similar arguments as in the case of $\mathcal{N}_{3}$, we have that
\begin{align}
&\left|\int_{0}^{T/\lambda^{2}}\int_{\R^{4}} |w(t,y)|^{2} (\nabla\cdot a)(t,x-y)\Re{\bar{w}\E(\bar{u}_{\lambda,l}u_{\lambda,h})u_{\lambda,h}}(t,x)dxdydt\right| \nonumber\\
&\lesssim C(a)\int_{0}^{T/\lambda^{2}} \| |\nabla|^{-1}(|w(t)|^{2})\|_{L_{x}^{6}} \|(w\E(\bar{u}_{\lambda,l}u_{\lambda,h})u_{\lambda,h})(t)\|_{L_{x}^{6/5}}dt \nonumber\\
&\lesssim C(a)\int_{0}^{T/\lambda^{2}} \|w(t)\|_{L_{x}^{3}}^{2}\|w(t)\|_{L_{x}^{4}}^{2}\|u_{\lambda,h}(t)\|_{L_{x}^{6}}^{2}dt \nonumber\\
&\lesssim  C(a)\|w\|_{L_{t}^{10}L_{x}^{4}}^{10/3} \|u_{\lambda,h}\|_{L_{t}^{3}L_{x}^{6}}^{2} \nonumber\\
&\lesssim C(a)\left(2^{\frac{3k_{0}}{10}}\|u_{\lambda}\|_{\tilde{X}_{k_{0}}([0,\lambda^{-2}T]\times\R^{2})}\right)^{10/3} \left(\|u_{\lambda}\|_{\tilde{X}_{k_{0}}([0,\lambda^{-2}T]\times\R^{2})}^{4/5} \left(C(\eta)\frac{\epsilon_{3}}{K}+\eta\right)^{1/5}\right)^{2} \nonumber\\
&\lesssim_{u} 2^{k_{0}}C(a)\left(C(\eta)\frac{\epsilon_{3}}{K}+\eta\right)^{2/5}.
\end{align}

Lastly, we estimate the contribution of $\mathcal{N}_{1}$. Since the contributions of the other terms in $\mathcal{N}_{1}$ may be estimated by similar argument, we may assume with without loss of generality that $\mathcal{N}_{1}=\E(\bar{u}_{\lambda,l}u_{\lambda,h})u_{\lambda,l}$. From the same combination of inequalities, Calder\'{o}n-Zygmund theorem, interpolation, and mass conservation as before, it follows that
\begin{align}
&\left|\int_{0}^{T/\lambda^{2}}\int_{\R^{4}}(\nabla\cdot a)(t,x-y) |w(t,y)|^{2} \Re{\bar{w}\E(\bar{u}_{\lambda,l}u_{\lambda,h})u_{\lambda,l}}(t,x)dxdydt\right|\nonumber\\
&\lesssim C(a)\int_{0}^{T/\lambda^{2}} \|w(t)\|_{L_{x}^{3}}^{2}\|(wE(\bar{u}_{\lambda,l}u_{\lambda,h})u_{\lambda,l})(t)\|_{L_{x}^{6/5}}dt \nonumber\\
&\lesssim C(a)\int_{0}^{T/\lambda^{2}} \|w(t)\|_{L_{x}^{3}}^{2}\|w(t)\|_{L_{x}^{9/2}}^{3}\|u_{\lambda,h}(t)\|_{L_{x}^{6}}dt \nonumber\\
&\lesssim C(a)\int_{0}^{T/\lambda^{2}} \|w(t)\|_{L_{x}^{9/2}}^{21/5}\|u_{\lambda,h}(t)\|_{L_{x}^{6}}dt \nonumber\\
&\lesssim C(a)\|w\|_{L_{t}^{63/10}L_{x}^{9/2}}^{21/5}\|u_{\lambda,h}\|_{L_{t}^{3}L_{x}^{6}} \nonumber\\
&\lesssim C(a)\left(2^{\frac{15k_{0}}{63}}\|u_{\lambda}\|_{\tilde{X}_{k_{0}}([0,\lambda^{-2}T]\times\R^{2})}\right)^{21/5}\left(C(\eta)\frac{\epsilon_{3}}{K}+\eta\right)^{1/5} \nonumber\\
&\lesssim_{u} 2^{k_{0}}C(a)\left(C(\eta)\frac{\epsilon_{3}}{K}+\eta\right)^{1/5}.
\end{align}

Bookkeeping our estimates, we have shown that
\begin{equation}
\eqref{eq:IM_frq_err3} \lesssim 2^{k_{0}}C(a)\left(C(\eta)\frac{\epsilon_{3}}{K}+\eta\right)^{1/6},
\end{equation}
which completes the proof of the proposition.
\end{description}
\end{proof}

\subsection{Defocusing case}\label{ssec:QS_dfoc}
In this subsection, we preclude the quasi-soliton scenario for the defocusing eeDS equation, which we remind the reader is
\begin{equation}
(i\partial_{t}+\Delta) u =|u|^{2}u-\E(|u|^{2})u = -\mathcal{L}(|u|^{2})u, \qquad \mathcal{L} \coloneqq -\frac{\partial_{2}^{2}}{\Delta}.
\end{equation}
Our goal is to prove an inequality of the form
\begin{equation}
K\lesssim |M_{R,m}| \lesssim o(K),
\end{equation}
by carefully balancing the parameters $m,R,\eta_{1},\eta_{2},\eta_{3},K$. Since $K$ may be taken arbitrarily large in the quasi-soliton scenario by taking $T$ arbitrarily large, we obtain a contradiction.

Recall that our Morawetz functional satisfies the identity
\begin{align}
M_{R,m}(T) &= -4\int_{0}^{T}\int_{\R^{4}}a_{R,m,j}(t,x-y)|w(t,y)|^{2}\partial_{k}\Re{\ol{\partial_{k}w}\partial_{j}w}(t,x)dxdydt \label{eq:IM_d_main1}\\
&\phantom{=}-2\int_{0}^{T}\int_{\R^{4}}a_{R,m,j}(t,x-y)\partial_{k}T_{0k}^{w}(t,y)\Im{\bar{w}\partial_{j}w}(t,x)dxdydt &\label{eq:IM_d_main2}\\
&\phantom{=}-\int_{0}^{T}\int_{\R^{4}}a_{R,m,j}(t,x-y)|w(t,y)|^{2}\partial_{k}T_{jk}^{w}(t,x)dxdydt\label{eq:IM_d_nl}\\
&\phantom{=}+\int_{0}^{T}\int_{\R^{4}}a_{R,m,j}(t,x-y)|w(t,y)|^{2}\partial_{j}\partial_{k}^{2}(|w|^{2})(t,x)dxdy dt\label{eq:IM_d_lerr}\\
&\phantom{=}+2\int_{0}^{T}\int_{\R^{4}}(\partial_{t}a_{R,m})(t,x-y)\cdot |w(t,y)|^{2}\Im{\bar{w}\nabla w}(t,x)dxdydt\label{eq:IM_d_terr}\\
&\phantom{=}+2\int_{0}^{T}\int_{\R^{4}}a_{R,m,j}(t,x-y)|w(t,y)|^{2}\Re{\bar{\mathcal{N}}\partial_{j}w -\bar{w}\partial_{j}\mathcal{N}}(t,x)dxdydt\label{eq:IM_d_err1}\\
&\phantom{=}+4\int_{0}^{T}\int_{\R^{4}}a_{R,m,j}(t,x-y) \Im{\bar{w}\mathcal{N}}(t,y) \Im{\bar{w}\partial_{j}w}(t,x)dxdydt.\label{eq:IM_d_err2}
\end{align}
The quantity \eqref{eq:IM_d_main1}+\eqref{eq:IM_d_main2} provides the crucial lower bound $\gtrsim_{K}$. The remaining terms on the RHS of the identity are error terms relative to \eqref{eq:IM_d_main1}+\eqref{eq:IM_d_main2} and can be made of size $\sim o_{R}(1)K$ or $\sim \eta_{2}K$ through delicate analysis involving a combination of the preliminary lemmas from subsection \ref{ssec:QS_pre} and proposition \ref{prop:IM_frq_err}.

We now proceed to estimating the terms \eqref{eq:IM_d_main1}-\eqref{eq:IM_d_err2} in several steps.

\begin{description}[leftmargin=*]
\item[Estimate for $\eqref{eq:IM_d_err1}+\eqref{eq:IM_d_err2}$:]
We first estimate the error terms arising from the fact that the frequency-localized solution $w$ satisfies the approximate eeDS equation
\begin{equation}
(i\partial_{t}+\Delta)w=F(w) + (P_{\leq K}F(u)-F(w)) = F(w) + \mathcal{N}
\end{equation}
by using proposition \ref{prop:IM_frq_err}. In order to do so, we rescale the frequency-localized solution $w$ by setting $\lambda:=\frac{\epsilon_{3}2^{k_{0}}}{K}$ and defining
\begin{equation}
w_{\lambda}(t,x) \coloneqq \lambda w(\lambda^{2}t, \lambda x) = P_{\leq \epsilon_{3}2^{k_{0}}}(u_{\lambda}), \qquad u_{\lambda}(t,x) \coloneqq \lambda u(\lambda^{2}t,\lambda x).
\end{equation}
Now observe that
\begin{align}
\eqref{eq:IM_d_err1} = 2\lambda^{-7}\int_{0}^{T}\int_{\R^{4}}a_{R,m,\lambda,j}\left(\lambda^{-2}t,\lambda^{-1}(x-y)\right)|w_{\lambda} (\lambda^{-2}	t, \lambda^{-1}	y)|^{2}\Re{\bar{\mathcal{N}}_{\lambda}\partial_{j}w_{\lambda}-\bar{w}_{\lambda}\partial_{j}\mathcal{N}_{\lambda}}(\lambda^{-2} t, \lambda^{-1}x) dxdy dt,
\end{align}
and
\begin{align}
\eqref{eq:IM_d_err2} = 4\lambda^{-7}\int_{0}^{T}\int_{\R^{4}}a_{R,m,\lambda,j}\left(\lambda^{-2}t,\lambda^{-1}(x-y)\right)\Im{\bar{w}_{\lambda}\mathcal{N}_{\lambda}}(\lambda^{-2}t,\lambda^{-1}y)\Im{\bar{w}_{\lambda}\partial_{j}w_{\lambda}}(\lambda^{-2}t,\lambda^{-1}x)dxdy dt,
\end{align}
where
\begin{equation}
a_{R,m,\lambda}(t,x-y) \coloneqq a_{R,m}(\lambda^{2}t,\lambda(x-y)).
\end{equation}
Then by dilation invariance,
\begin{equation}
\eqref{eq:IM_d_err1} = \frac{2}{\lambda}\int_{0}^{T/\lambda^{2}}\int_{\R^{4}}a_{R,m,\lambda,j}(t,x-y)|w_{\lambda}(t,y)|^{2}\Re{\bar{\mathcal{N}}_{\lambda}\partial_{j}w_{\lambda}-\bar{w}_{\lambda}\partial_{j}\mathcal{N}_{\lambda}}(t,x)dxdy dt, \label{eq:IM_d_err1'}
\end{equation}
and
\begin{equation}
\eqref{eq:IM_d_err2} = \frac{4}{\lambda}\int_{0}^{T/\lambda^{2}}\int_{\R^{4}}a_{R,m,\lambda,j}(t,x-y) \Im{\bar{w}_{\lambda}\mathcal{N}_{\lambda}}(t,y)\Im{\bar{w}_{\lambda}\partial_{j}w_{\lambda}}(t,x)dxdy dt.
\label{eq:IM_d_err2'}
\end{equation}

We claim that the potential $a_{R,m,\lambda}$ satisfies the conditions of proposition \ref{prop:IM_frq_err} with constant $C(a)\lesssim R^{2}$. Indeed, it is evident that $\|a_{R,m,\lambda}\|_{L_{t,x}^{\infty}} \lesssim R^{2}$ and by the chain rule,
\begin{equation}
|\nabla a_{R,m,\lambda}(t,x-y)| = \lambda |(\nabla a_{R,m,j})(\lambda	^{2}t,\lambda(x-y))| \lesssim \frac{R^{2}\lambda}{|\lambda(x-y)|} = \frac{R^{2}}{|x-y|}. 
\end{equation}

Using that \eqref{eq:IM_d_err1'} + \eqref{eq:IM_d_err2'} is Galilean invariant, we apply proposition \ref{prop:IM_frq_err} with the potential $a_{R,m,\lambda}$ to obtain that
\begin{equation}
\frac{\lambda}{2}|\eqref{eq:IM_d_err1}+\eqref{eq:IM_d_err2}| \lesssim_{u} 2^{k_{0}}R^{2}\left(\frac{C(\eta)\epsilon_{3}}{K}+\eta\right)^{1/6}, \qquad \forall 0<\eta<1, \enspace T\geq T_{0}(\eta).
\end{equation}
which implies that
\begin{equation}
|\eqref{eq:IM_d_err1}+\eqref{eq:IM_d_err2}| \lesssim_{u} K \frac{R^{2}}{\epsilon_{3}}\left(\frac{C(\eta)\epsilon_{3}}{K}+\eta\right)^{1/6}, \qquad \forall 0<\eta<1, \enspace K\geq K(\eta).
\end{equation}
Since given $\epsilon_{3},R$, we have the freedom to take $\eta$ arbitrarily small and $K(\eta)$ arbitrarily large, we see that the RHS is $o(K)$.

\item[Estimate for $M_{R,m}(t)$:] %Estimate for Morawetz functional
We next estimate the magnitude of the Morawetz functional $M_{R,m}(t)$ over the time interval $[0,T]$. Since $a_{R,m}$ is odd, it follows that $M_{R,m}(t)$ is invariant under the Galilean transformation $w\mapsto e^{-ix\cdot\xi(t)}w$. Therefore
\begin{align}
|M_{R,m}(t)| &\lesssim \|a_{R,m}\|_{L_{t,x}^{\infty}} \int_{\R^{4}} |\Im{\bar{w}(\nabla-i\xi(t))w}(t,x)| |w(t,y)|^{2}dxdy\nonumber\\
&\lesssim R^{2}\int_{\R^{4}}  |\Im{\bar{w}(\nabla-i\xi(t))w}(t,x)| |w(t,y)|^{2}dxdy.
\end{align}
Now decompose $w$ by
\begin{equation}
w = P_{\xi(t), \leq C(\eta_{3}R^{-2})N_{m}(t)}w + P_{\xi(t), >C(\eta_{3}R^{-2})N_{m}(t)}w,
\end{equation}
where we remind the reader that $C:I\rightarrow [0,\infty)$ is the compactness modulus function for $u$. Using the estimate for $|\xi'(t)|$ and the fundamental theorem of calculus, we have that $|\xi(t)| \leq 2^{-19}\epsilon_{3}^{1/2}K$ for $t\in [0,T]$, so that
\begin{equation}
P_{\xi(t),\leq C(\eta_{3}R^{-2})N_{m}(t)}w=P_{\xi(t),\leq C(\eta_{3}R^{-2})N_{m}(t)}u,
\end{equation}
provided that $K=K(\eta_{3},R)>0$ is sufficiently large. Therefore by Cauchy-Schwarz, Plancherel's theorem, the frequency-localization property \ref{eq:sp_frq_loc}, and mass conservation, we have that
\begin{align}
\sup_{m\in\N}\sup_{t\in [0,T]} M_{R,m}(t) &\lesssim \sup_{m\in\N}\sup_{t\in [0,T]} R^{2}\|u_{0}\|_{L^{2}}^{2} \|w(\nabla-i\xi(t))P_{\xi(t),\leq C(\eta_{3}R^{-2})N_{m}(t)}w\|_{L_{t}^{\infty}L_{x}^{1}}&\nonumber\\
&\phantom{=}+\sup_{t\in [0,T]}\|w(\nabla-i\xi(t))P_{\xi(t),>C(\eta_{3}R^{-2})N_{m}(t)}w\|_{L_{t}^{\infty}L_{x}^{1}}&\nonumber\\
&\lesssim (1+M(u))\paren*{C(\eta_{3}R^{-2})R^{2}+\eta_{3}K}.
\end{align}
Since given $R$, the parameter $\eta_{3}$ may be taken arbitrarily small and $K(\eta_{3})$ may be taken arbitrarily large, we see that the RHS is $o(K)$.

\item[Estimate for \eqref{eq:IM_d_lerr}:]
Integrating by parts three times in $x$, we see that
\begin{align}
\eqref{eq:IM_d_lerr} = -\int_{0}^{T}\int_{\R^{4}}(\Delta\nabla\cdot a_{R,m})(t,x-y) |w(t,y)|^{2} |w(t,x)|^{2}dxdydt.
\end{align}
Using the identity for $\nabla\cdot a_{R,m}$ in lemma \ref{lem:pot_prop}, we see that
\begin{align}
(\Delta \nabla\cdot a_{R,m})(t,x-y) = N_{m}(t)\Delta_{x}\left(\varphi\left(\frac{N_{m}(t)|x-y|}{R}\right) + \psi_{R,m}(t,|x-y|)\right).
\end{align}
It follows from the chain rule and mean value theorem together with the derivative estimates for $\psi$ and $\psi_{R,m}$ that
\begin{equation}
\left|\Delta_{x}\left(\varphi\left(\frac{N_{m}(t)|x-y|}{R}\right) + \psi_{R,m}(t,|x-y|)\right)\right| \lesssim \frac{N_{m}(t)^{2}}{(\eta_{2}R)^{4}}, \qquad \forall (x,y)\in\R^{2}\times\R^{2}.
\end{equation}
Hence by mass conservation,
\begin{align}
\left|\int_{0}^{T}\int_{\R^{4}}(\Delta\nabla\cdot a_{R,m})(t,x-y) |w(t,y)|^{2}|w(t,x)|^{2}dxdy dt\right| &\lesssim \frac{1}{(\eta_{2}R)^{4}}\int_{0}^{T}N_{m}(t)^{3}dt= \frac{K}{(\eta_{2}R)^{4}}.
\end{align}

\item[Estimate for \eqref{eq:IM_d_nl}:]
We next estimate the contribution of the nonlinear and nonlocal part of the eeDS equation to the $M_{R,m}(T)$, which is given by the term \eqref{eq:IM_d_nl}. Integrating by parts in $x$ to move the $\partial_{k}$ onto $a_{R,m,j}$, we see that
\begin{equation}
\eqref{eq:IM_d_nl} = \int_{0}^{T}\int_{\R^{4}}\partial_{k}a_{R,m,j}(t,x-y) |w(t,y)|^{2}T_{jk}^{w}(t,x)dxdy dt.
\end{equation}

\begin{remark} %Comparison of $T_{jk}^{w}$ to NLS setting
In the setting of the cubic NLS, $T_{jk}^{w}=\p_{k}a_{R,m,j}\delta_{jk}|w|^{4}$. Therefore
\begin{equation}
\int_{0}^{T}\int_{\R^{4}}\p_{k}a_{R,m,j}(t,x-y)|w(t,y)|^{2}T_{jk}^{w}(t,x)dxdydt = \int_{0}^{T}\int_{\R^{4}}(\nabla\cdot a_{R,m})(t,x-y) |w(t,y)|^{2} |w(t,x)|^{4}dxdydt \geq 0,
\end{equation}
and so this term may be simply discarded. This, however, is \emph{not} the case in the setting of the eeDS as the reader will see below.
\end{remark}

Unpackaging the definition the definition of $T_{jk}^{w}$, using the operator identity $Id=\frac{\p_{1}^{2}}{\Delta}+\frac{\p_{2}^{2}}{\Delta}$, and proceeding by direct algebraic manipulation, we see that
\begin{align}
\partial_{k}a_{R,m,j}T_{jk}^{w} &= \partial_{j}a_{R,m,j}\left(\frac{\partial_{2}^{2}}{\Delta}(|w|^{2})\right)^{2}\nonumber\\
&\phantom{=}+2\partial_{2}a_{R,m,2}\frac{\partial_{1}^{2}}{\Delta}(|w|^{2})\frac{\partial_{2}^{2}}{\Delta}(|w|^{2})\nonumber\\
&\phantom{=}-2\partial_{1}a_{R,m,2}\frac{\partial_{2}^{2}}{\Delta}(|w|^{2})\frac{\partial_{1}\partial_{2}}{\Delta}(|w|^{2})\nonumber\\
&\phantom{=}+2\partial_{2}a_{R,m,1}\frac{\partial_{1}\partial_{2}}{\Delta}(|w|^{2})\frac{\partial_{1}^{2}}{\Delta}(|w|^{2})\nonumber\\
&\phantom{=}+\left(\partial_{2}a_{R,m,2}-\partial_{1}a_{R,m,1}\right)\left(\frac{\partial_{1}\partial_{2}}{\Delta}(|w|^{2})\right)^{2}\nonumber\\
&\eqqcolon \mathrm{Term}_{1}+\cdots+\mathrm{Term}_{5}.\label{eq:IM_d_ten_dc}
\end{align}
We now consider the contributions of the $\mathrm{Term}_{j}$ in groupings.

\begin{itemize}[leftmargin=*]
\item %Term_1
We first consider $\mathrm{Term}_{1}$. First, recall from lemma \ref{lem:pot_prop} that $\nabla\cdot a_{R}\geq 0$. Therefore
\begin{align}
\int_{0}^{T}\int_{\R^{4}}(\nabla\cdot a_{R})(t,x-y)|w(t,y)|^{2}\left(\frac{\partial_{2}^{2}}{\Delta}(|w|^{2})(t,x)\right)^{2}dxdy dt\geq 0,
\end{align}
so we may harmlessly discard this term.

\item %Term_3+Term_4
We next consider $\mathrm{Term}_{3}+\mathrm{Term}_{4}$. Observe that
\begin{align}
\partial_{1}a_{R,m,2}(t,x-y) =\partial_{2}a_{R,m,1}(t,x-y) = (\p_{r}\psi_{R,m})(t,|x-y|)N_{m}(t)\frac{(x-y)_{1}(x-y)_{2}}{|x-y|}.
\end{align}
Decomposing $\psi_{R,m}=\psi_{R,m,1}+\psi_{R,m,2}+\psi_{R,m,3}$, substituting into above, it suffices to estimate the modulus of the quantity
\begin{equation}
\sum_{j=1}^{3}\int_{0}^{T}\int_{\R^{4}}(\p_{r}\psi_{R,m,j})(t,|x-y|) N_{m}(t)\frac{(x-y)_{1}(x-y)_{2}}{|x-y|} |w(t,y)|^{2} \frac{\p_{2}^{2}}{\Delta}(|w|^{2})(t,x)\frac{\p_{1}\p_{2}}{\Delta}(|w|^{2})(t,x)dxdydt.
\end{equation}
We consider the contributions of the $\p_{r}\psi_{R,m,j}$ separately.

For the contribution of $\p_{r}\psi_{R,m,1}$, observe that since $\p_{r}\psi_{R,m,1} \equiv 0$ on $|r| <\frac{\eta_{2}R^{2}}{2N_{m}(t)}$ and $|\p_{r}\psi_{R,m,1}| \lesssim \frac{1}{r}$ on $|r|\geq \frac{\eta_{2}R^{2}}{2N_{m}(t)}$, we have that
\begin{align}
&\left|\int_{0}^{T}\int_{\R^{4}}(\p_{r}\psi_{R,m,1})(t,|x-y|)N_{m}(t)\frac{(x-y)_{1}(x-y)_{2}}{|x-y|} |w(t,y)|^{2}\frac{\p_{2}^{2}}{\Delta}(|w|^{2})(t,x)\frac{\p_{1}\p_{2}}{\Delta}(|w|^{2})(t,x)dxdydt\right| \nonumber\\
&\lesssim \int_{0}^{T}N_{m}(t)\int_{|x-y|>\frac{\eta_{2}R^{2}}{2N_{m}(t)}} |w(t,y)|^{2}\left|\frac{\p_{2}^{2}}{\Delta}(|w|^{2})(t,x)|^{2}\frac{\p_{1}\p_{2}}{\Delta}(|w|^{2})(t,x)\right| dxdydt \nonumber\\
&=o_{R}(1)K
\end{align}
by Cauchy-Schwarz and lemma \ref{lem:IM_pre_CZ}.

For the contribution of $\p_{r}\psi_{R,m,2}$, observe that since $|\p_{r}\psi_{R,m,2}| \lesssim \frac{N_{m}(t)}{R^{2}}$ on $|r|\leq \frac{\eta_{2}R^{2}}{2N_{m}(t)}$ and $|(\p_{r}\psi_{R,m,2})(r)| \lesssim \frac{1}{r}$ on $|r|>\frac{\eta_{2}R^{2}}{2N_{m}(t)}$, we have that
\begin{align}
&\left|\int_{0}^{T}\int_{\R^{4}}(\p_{r}\psi_{R,m,2})(t,|x-y|)N_{m}(t)\frac{(x-y)_{1}(x-y)_{2}}{|x-y|} |w(t,y)|^{2}\frac{\p_{2}^{2}}{\Delta}(|w|^{2})(t,x) \frac{\p_{1}\p_{2}}{\Delta}(|w|^{2})(t,x)dxdydt\right| \nonumber\\
&\lesssim \int_{0}^{T}\int_{|x-y|\leq\frac{\eta_{2}R^{2}}{2N_{m}(t)}} \paren*{\frac{N_{m}(t)}{R^{2}}} N_{m}(t)|x-y| |w(t,y)|^{2} \left|\frac{\p_{2}^{2}}{\Delta}(|w|^{2})(t,x) \frac{\p_{1}\p_{2}}{\Delta}(|w|^{2})(t,x)\right|dxdydt \nonumber\\
&\phantom{=}+\int_{0}^{T}\int_{|x-y|>\frac{\eta_{2}R^{2}}{2N_{m}(t)}} N_{m}(t) |w(t,y)|^{2}\left|\frac{\p_{2}^{2}}{\Delta}(|w|^{2})(t,x) \frac{\p_{1}\p_{2}}{\Delta}(|w|^{2})(t,x)\right|dxdydt \nonumber\\
&\lesssim \eta_{2} \int_{0}^{T}N_{m}(t)\int_{\R^{4}} |w(t,y)|^{2}\left|\frac{\p_{2}^{2}}{\Delta}(|w|^{2})(t,x) \frac{\p_{1}\p_{2}}{\Delta}(|w|^{2})(t,x)\right|dxdydt \nonumber\\
&\phantom{=}\int_{0}^{T}\int_{|x-y|>\frac{\eta_{2}R^{2}}{2N_{m}(t)}} N_{m}(t) |w(t,y)|^{2}\left|\frac{\p_{2}^{2}}{\Delta}(|w|^{2})(t,x) \frac{\p_{1}\p_{2}}{\Delta}(|w|^{2})(t,x)\right|dxdydt.
\end{align}
The first term in the RHS of the ultimate inequality is $\lesssim \eta_{2}M(u)K$ by Cauchy-Schwarz and Plancherel's theorem, and the second term is $o_{R}(1)K$ by lemma \ref{lem:IM_pre_CZ}.

For the contribution of $\p_{r}\psi_{R,m,3}$, observe that since $\p_{r}\psi_{R,m,3}$ is supported in the region $\{|r|>\frac{\eta_{2}R^{2}}{2N_{m}(t)}\}$ and $|(\p_{r}\psi_{R,m,3})(r)| \lesssim \frac{1}{r}$, we can repeat the argument for the contribution of $\p_{r}\psi_{R,m,1}$ to obtain the estimate
\begin{equation}
\left|\int_{0}^{T}\int_{\R^{4}}(\p_{r}\psi_{R,m,3})(t,|x-y|)N_{m}(t)\frac{(x-y)_{1}(x-y)_{2}}{|x-y|}|w(t,y)|^{2}\frac{\p_{2}^{2}}{\Delta}(|w|^{2})(t,x) \frac{\p_{1}\p_{2}}{\Delta}(|w|^{2})(t,x)dxdydt\right| = o_{R}(1)K.
\end{equation}

Therefore, we have shown that
\begin{equation}
|\mathrm{Term}_{3}+\mathrm{Term}_{4}| \lesssim \paren*{\eta_{2}+o_{R}(1)}K.
\end{equation}

\item %Term_{2}+Term_5
Lastly, we consider $\mathrm{Term}_{2}+\mathrm{Term}_{5}$. We claim that
\begin{equation}
\int_{0}^{T}\int_{\R^{4}}|w(t,y)|^{2}\paren*{\mathrm{Term}_{2}+\mathrm{Term}_{5}}(t,x,y)dxdydt + o_{R}(1)K+O(\eta_{2}K) \geq 0.
\end{equation}
Indeed, observe that
\begin{equation}
\partial_{j}a_{R,m,j}(t,x-y) = N_{m}(t)\paren*{\p_{r}\psi_{R,m})(t,|x-y|)\frac{(x-y)_{j}^{2}}{|x-y|} + \psi_{R,m}(t,|x-y|)}, \qquad j=1,2.
\end{equation}
We can estimate the contribution of the potential $(\p_{r}\psi_{R,m})(t,|x-y|)N_{m}(t)\frac{(x-y)_{j}^{2}}{|x-y|}$ as in the case of $\mathrm{Term}_{3}+\mathrm{Term}_{4}$ to obtain that
\begin{equation}
\begin{split}
&\int_{0}^{T}\int_{\R^{4}}(\p_{r}\psi_{R,m})(t,|x-y|)N_{m}(t)\frac{(x-y)_{2}^{2}}{|x-y|}|w(t,y)^{2}\frac{\partial_{1}^{2}}{\Delta}(|w|^{2})(t,x)\frac{\partial_{2}^{2}}{\Delta}(|w|^{2})(t,x)dxdy dt\\
&\phantom{=}+\int_{0}^{T}\int_{\R^{4}}(\p_{r}\psi_{R,m})(t,|x-y|)N_{m}(t)\frac{(x-y)_{2}^{2}-(x-y)_{1}^{2}}{|x-y|} |w(t,y)|^{2}\left(\frac{\partial_{1}\partial_{2}}{\Delta}(|w|^{2})(t,x)\right)^{2}dxdy dt\\
&\lesssim o_{R}(1)K+\eta_{2}K.
\end{split}
\end{equation}
Therefore it suffices to show that the quantity
\begin{equation}
\int_{0}^{T}N_{m}(t)\int_{\R^{4}}\psi_{R,m}(t,|x-y|)|w(t,y)|^{2}\frac{\partial_{1}^{2}}{\Delta}(|w|^{2})(t,x)\frac{\partial_{2}^{2}}{\Delta}(|w|^{2})(t,x)dxdydt
\end{equation}
is nonnegative up to an error of size $o_{R}(1)K+\eta_{2}K$, which we do with the next lemma.

\begin{lemma}[Cheap lower bound]\label{lem:cheap_LB}
Let $m_{1}(D),m_{2}(D)$ be two Fourier multipliers with nonnegative symbols $m_{1},m_{2}$, respectively, which are of the form $m_{j}(D)=(\mathcal{K}_{j}\ast\cdot) + c_{j}Id$, where $\mathcal{K}_{j}$ is a real-valued CZK which defines a CZO of convolution type and $c_{j}$ is a real constant. Then there exists a constant $C>0$ such that
\begin{equation}
\int_{0}^{T}N_{m}(t)\int_{\R^{4}}\psi_{R,m}(t,x-y)|w(t,y)|^{2} m_{1}(D)(|w|^{2})(t,x)m_{2}(D)(|w|^{2})(t,x)dxdy dt + o_{R}(1)K+C\eta_{2}K\geq 0.
\end{equation}
\end{lemma}
\begin{proof}
We decompose $\psi_{R,m}=\psi_{R,m,1}+\psi_{R,m,2}+\psi_{R,m,3}$ and consider the contribution of each term separetely.

For the contribution of $\psi_{R,m,1}$, decompose
\begin{equation}
\psi_{R,m,1}(t,|x-y|) = C_{\eta_{2}} - \paren*{C_{\eta_{2}}-\psi_{R,m,1}(t,|x-y|)},
\end{equation}
where $C_{\eta_{2}}$ is the constant in lemma \ref{lem:psi_est}, so that
\begin{align}
&\int_{\R^{4}}\psi_{R,m,1}(t,|x-y|)|w(t,y)|^{2}m_{1}(D)(|w|^{2})(t,x)m_{2}(D)(|w|^{2})(t,x)dxdy \nonumber\\
&=C_{\eta_{2}}\int_{\R^{4}}|w(t,y)|^{2}m_{1}(D)(|w|^{2})(t,x)m_{2}(D)(|w|^{2})(t,x)dxdy \nonumber\\
&\phantom{=}+\int_{\R^{4}}\paren*{\psi_{R,m,1}(t,|x-y|)-C_{\eta_{2}}}|w(t,y)|^{2}m_{1}(D)(|w|^{2})(t,x)m_{2}(D)(|w|^{2})(t,x)dxdy \nonumber\\
&\eqqcolon \mathrm{Main}(t)+\mathrm{Error}(t).
\end{align}
By Plancherel's theorem in $x$, we have that
\begin{equation}
\mathrm{Main}(t) = M(w(t)) \int_{\R^{2}}m_{1}(\xi)m_{2}(\xi)\mathcal{F}(|w|^{2})(t,\xi)\ol{\mathcal{F}(|w|^{2})(t,\xi)}d\xi \geq 0.
\end{equation}
Hence,
\begin{equation}
\int_{0}^{T}N_{m}(t)\mathrm{Main}(t)dt \geq 0.
\end{equation}
We claim that $|\mathrm{Error}(t)|\lesssim o_{R}(1)$. Indeed, since $\psi_{R,m,1}\equiv C_{\eta_{2}}$ on the interval $\{|r|\leq \frac{\eta_{2}R^{2}}{N_{m}(t)}\}$, we have the pointwise bound
\begin{equation}
|\psi_{R,m,1}(t,|x-y|)-C_{\eta_{2}}| \lesssim \cdot 1_{\geq \frac{\eta_{2}R^{2}}{N_{m}(t)}}(|x-y|).
\end{equation}
Cauchy-Schwarz and lemma \ref{lem:IM_pre_CZ} then imply the claim.

For the contribution of $\psi_{R,m,2}$, recall that $|\psi_{R,m,2}| \lesssim \eta_{2}$. Hence,
\begin{align}
&\left|\int_{\R^{4}}\psi_{R,m,2}(t,|x-y|) |w(t,y)|^{2} m_{1}(D)(|w|^{2})(t,x) m_{2}(D)(|w|^{2})(t,x)dxdy\right| \nonumber\\
&\lesssim \eta_{2}\int_{\R^{4}} |w(t,y)|^{2}\left| m_{1}(D)(|w|^{2})(t,x) m_{2}(D)(|w|^{2})(t,x)\right|dxdy \nonumber\\
&\lesssim \eta_{2}M(w(t))\|w(t)\|_{L_{x}^{4}(\R^{2})}^{4}
\end{align}
by Cauchy-Schwarz and Plancherel's theorem. Multiplying the RHS of the ultimate inequality by $N_{m}(t)$ and then integrating with respect to time over $[0,T]$, we conclude that the contribution of $\psi_{R,m,2}$ is $\lesssim \eta_{2}K$.

For the contribution of $\psi_{R,m,3}$, we use that $\psi_{R,m,3}\equiv 0$ on the interval $\{|r|\leq \frac{\eta_{2}R^{2}}{N_{m}(t)}\}$ together with $\psi_{R,m,3}\leq 1$ to obtain that
\begin{equation}
\left|\int_{0}^{T}\int_{\R^{4}}N_{m}(t)\int_{\R^{4}}\psi_{R,m,3}(t,|x-y|) |w(t,y)|^{2}m_{1}(D)(|w|^{2})(t,x) m_{2}(D)(|w|^{2})(t,x)dxdydt\right| = o_{R}(1)K
\end{equation}
by Cauchy-Schwarz and lemma \ref{lem:IM_pre_CZ}.

Bookkeeping our analysis for the contributions of the $\psi_{R,m,j}$ completes the proof of the lemma.
\end{proof}

\end{itemize}

\item[Lower bound for $\eqref{eq:IM_d_main1}+\eqref{eq:IM_d_main2}$:]
We next prove a lower bound of size $K$ for the sum $\eqref{eq:IM_d_main1}+\eqref{eq:IM_d_main2}$, which is the crucial ingredient in obtaining a contradiction in the quasi-soliton scenario. Integrating by parts in the $x$ variable in \eqref{eq:IM_d_main1} and in the $y$ variable in \eqref{eq:IM_d_main2} together with unpackaging the definition of $T_{0k}^{w}$, we see that
\begin{align}
\eqref{eq:IM_d_main1} + \eqref{eq:IM_d_main2} &= 4\int_{0}^{T}\int_{\R^{4}}(\p_{k}a_{R,m,j})(t,x-y) |w(t,y)|^{2}\Re{\ol{\p_{k}w}\partial_{j}w}(t,x)dxdy dt\nonumber\\
&\phantom{=}-4\int_{0}^{T}\int_{\R^{4}}(\p_{k}a_{R,m,j}(t,x-y)) \Im{\bar{w}\p_{k}w}(t,y) \Im{\bar{w}\p_{j}w}(t,x)dxdy dt\nonumber\\
&=:\mathrm{Term}_{1}+\mathrm{Term}_{2}.
\end{align}
Since
\begin{align}
(\partial_{k}a_{R,m,j})(t,x-y) &= N_{m}(t)\left((\p_{r}\psi_{R,m})(t,|x-y|)\frac{(x-y)_{j}(x-y)_{k}}{|x-y|}+\delta_{jk}\psi_{R,m}(t,|x-y|)\right),\nonumber\\
\end{align}
it follows by direct computation that
\begin{align}
\mathrm{Term}_{1} &= 4\int_{0}^{T}N_{m}(t)\int_{\R^{4}}\varphi\left(\frac{N_{m}(t)|x-y|}{R}\right)N_{m}(t) |(\nabla w)(t,x)|^{2} |w(t,y)|^{2}dxdy dt\nonumber\\
&\phantom{=}+4\int_{0}^{T}N_{m}(t)\int_{\R^{4}}\left(\psi_{R,m}(t,|x-y|)-\varphi\left(\frac{N_{m}(t)|x-y|}{R}\right)\right) |(\nabla_{rad,y}^{\perp}w)(t,x)|^{2} |w(t,y)|^{2}dxdy dt\nonumber\\
&=:\mathrm{Term}_{1,1}+\mathrm{Term}_{1,2}.
\end{align}
Similarly,
\begin{align}
\mathrm{Term}_{2} &=-4\int_{0}^{T}N_{m}(t)\int_{\R^{4}}\varphi\left(\frac{N_{m}(t)|x-y|}{R}\right)\Im{\bar{w}\nabla w}(t,y)\cdot\Im{\bar{w}\nabla w}(t,x)dxdy dt\nonumber\\
&\phantom{=}-4\int_{0}^{T}N_{m}(t)\int_{\R^{4}}\left(\psi_{R,m}(t,|x-y|)-\varphi\left(\frac{N_{m}(t)|x-y|}{R}\right)\right)\Im{\bar{w}\nabla_{rad,y}^{\perp}w}(t,y)\cdot\Im{\bar{w}\nabla_{rad,y}^{\perp}w}(t,x)dxdy dt\nonumber\\
&=:\mathrm{Term}_{2,1}+\mathrm{Term}_{2,2}.
\end{align}
Above, we have introduced the notation $\nabla_{rad,y}f(x) \coloneqq (\nabla f(x)\cdot \frac{x-y}{|x-y|})\frac{x-y}{|x-y|}$ and $\nabla_{rad,y}^{\perp}f \coloneqq \nabla f - \nabla_{rad,y}f$.

We first claim that $\mathrm{Term}_{1,2}+\mathrm{Term}_{2,2}\geq 0$. Indeed, since $\varphi$ is decreasing, $\psi_{R,m}(t,|x-y|)-\varphi(\frac{N_{m}(t)|x-y|}{R})\geq 0$. So it follows from Cauchy-Schwarz and symmetry in $x,y$ that
\begin{equation}
|\mathrm{Term}_{2,2}| \leq 4\int_{0}^{T}N_{m}(t)\int_{\R^{4}}\left(\psi_{R,m}(t,|x-y|)-\varphi\left(\frac{N_{m}(t)|x-y|}{R}\right)\right) |w(t,y)|^{2} |(\nabla_{rad,y}^{\perp}w)(t,x)|^{2} dxdy dt,
\end{equation}
which implies the claim.

It remains to analyze the quantity $\mathrm{Term}_{1,1}+\mathrm{Term}_{2,1}$. To do so, we use the observation of \cite{Dodson2015} that $\mathrm{Term}_{1,1}+\mathrm{Term}_{2,1}$ is Galilean invariant. Therefore we can make $(t,z)$-dependent Galilean transformation $w\mapsto v_{t,z}$ such that $v_{t,z}$ has zero momentum. More precisely, for $\xi(t,z)\in\R^{2}$, we see from the definition of $\varphi$ and the Fubini-Tonelli theorem that under the Galilean transformation,
\begin{equation}
\begin{split}
\mathrm{Term}_{2,1} &\mapsto \frac{4N_{m}(t)}{\|\zeta(|\cdot|)\|_{L^{1}}}\int_{0}^{T}dt\int_{\R^{2}}dz\int_{\R^{4}}dxdy\zeta\left(\left|\frac{N_{m}(t)}{R}x-z\right|\right)\zeta\left(\left|\frac{N_{m}(t)}{R}y-z\right|\right)\times\\
&\phantom{=} \Im{\ol{e^{-iy\cdot\xi(t,z)}w}\nabla_{y}(e^{-iy\cdot\xi(t,z)}y)} \cdot \Im{\ol{e^{-ix\cdot\xi(t,z)}w}\nabla_{x}(e^{-ix\cdot\xi(t,z)}w)}(t,x).
\end{split}
\end{equation}
We choose $\xi(t,z)$ by the formula
\begin{equation}
\xi(t,z) \coloneqq \dfrac{\int_{\R^{2}} \zeta\left(|\frac{N_{m}(t)}{R}x-z|\right)\Im{\bar{w}\nabla w}(t,x)dx}{\int_{\R^{2}}\zeta\left(|\frac{N_{m}(t)}{R}x-z|\right) |w(t,x)|^{2}dx}.
\end{equation}
The integrands in the RHS above are continuous in $(t,z)$, so in particular, $\xi$ is a measurable function of $(t,z)$. With this choice of $\xi(t,z)$, we now have that
\begin{equation}
\mathrm{Term}_{1,1}+\mathrm{Term}_{2,1} = 4\int_{0}^{T}\int_{\R^{4}}\varphi\left(\frac{N_{m}(t)|x-y|}{R}\right)N_{m}(t) |\nabla_{x}(e^{-ix\cdot\xi(t,z)}w)(t,x)|^{2} |w(t,y)|^{2}dxdy dt.
\end{equation}

\begin{remark} %Remark explaining Galilean trans step unnecessary for certain eeDS equations
If our eeDS equation were instead of the form
\begin{equation}
(i\p_{t}+\Delta)u= \mu |u|^{2}u-\E(|u|^{2})u,
\end{equation}
for $\mu>1$, then this last step of making a Galilean transformation would be unnecessary. Indeed, $\mathrm{Term}_{1,1}+\mathrm{Term}_{2,1}\geq 0$ by Cauchy-Schwarz. One can then obtain the crucial $\gtrsim K$ lower bound from $\mathrm{Term}_{1}$ in \eqref{eq:IM_d_ten_dc}.
\end{remark}

Next, we define a bump function $\vartheta:\R\rightarrow [0,+\infty)$ by the formula
\begin{equation}
\vartheta(r) \coloneqq \frac{2}{\eta_{2}R}\int_{\R}1_{[-(1-4\eta_{2})R,(1-4\eta_{2})R]}(r-s)\chi\left(\frac{2s}{\eta_{2}R}\right)ds.
\end{equation}
Observe that $\vartheta$ is $C^{\infty}$, satisfies $0\leq \vartheta\leq 1$ and
\begin{equation}
\vartheta(r) = \begin{cases} 1, & {|r|\leq (1-6\eta_{2})R}\\ 0, & {|r| \geq (1-2\eta_{2})R} \end{cases},
\end{equation}
and by Young's inequality $\|\vartheta^{(n)}\|_{L^{\infty}} \lesssim_{n} (\eta_{2}R)^{-n}$ for all $n\in\N$.

Since $\zeta$ is identically one on the support of $\vartheta$, we have the lower bound
\begin{equation}
\int_{\R^{2}}\zeta\left(\left|\frac{N_{m}(t)}{R}x-z\right|\right)|\nabla_{x}(e^{-ix\cdot\xi(t,z)}w)(t,x)|^{2}dx \geq \int_{\R^{2}}\vartheta^{2}\left(\left|\frac{N_{m}(t)}{R}x-z\right|\right) |\nabla_{x}(e^{-ix\cdot\xi(t,z)}w)(t,x)|^{2}dx.\label{eq:IM_d_lb}
\end{equation}
We now massage the RHS of \eqref{eq:IM_d_lb} into a quantity consisting of a main term and error terms, which upon integrating with respect to $t\in [0,T]$, will have size $\sim K$ and size $\ll K$, respectively. For notational convenience, define
\begin{align}
v(t,x,z) &\coloneqq e^{-ix\cdot\xi(t,z)}w(t,x)\\
\vartheta_{R,m,z}(t,x) &\coloneqq \vartheta\left(\left|\frac{N_{m}(t)}{R}x-z\right|\right).
\end{align}
Using the elementary identity $\vartheta\nabla v = \nabla(\vartheta v)-v\nabla \vartheta$, we may write
\begin{equation}
\vartheta_{R,m,z}^{2}|\nabla_{x} v|^{2} = |\nabla_{x}(\vartheta_{R,m,z} v)|^{2} - 2\Re{\nabla_{x}(\vartheta_{R,m,z}v)\cdot\ol{(\nabla_{x}\vartheta_{R,m,z})v}} + |\nabla_{x} \vartheta_{R,m,z}|^{2}|v|^{2}.
\end{equation}
Substituting this identity into \eqref{eq:IM_d_lb}, we obtain
\begin{equation*}
\eqref{eq:IM_d_lb} = \int_{\R^{2}} |\nabla_{x}(\vartheta_{R,m,z}v)(x)|^{2}dx - \int_{\R^{2}}2\Re{\nabla_{x}(\vartheta_{R,m,z}v)\cdot \ol{(\nabla_{x} \vartheta_{R,m,z})v}}(x)dx + \int_{\R^{2}}|(\nabla_{x} \vartheta_{R,m,z})(x)v(x)|^{2}dx.
\end{equation*}
Now integrating by parts and using the property that $\vartheta$ is a real-valued, we see that
\begin{align}
\eqref{eq:IM_d_lb} &= \int_{\R^{2}} |\nabla_{x}(\vartheta_{R,m,z}v)(x)|^{2}dx  +2\int_{\R^{2}}\nabla_{x}\cdot(\vartheta_{R,m,z}(\nabla_{x}\vartheta_{R,m,z}))(x) |v(x)|^{2}dx-\int_{\R^{2}}|(\nabla_{x}\vartheta_{R,m,z})(x)|^{2}|v(x)|^{2}dx\nonumber\\
&\eqqcolon \mathrm{Main}(t,z)+\mathrm{Error}_{1}(t,z) + \mathrm{Error}_{2}(t,z).
\end{align}

We first consider the contribution of $\mathrm{Error}_{1}+\mathrm{Error}_{2}$. By the chain rule and the derivative estimates for $\vartheta$, we see that
\begin{equation}
\sup_{x\in\R^{2}}\max\{|\Delta_{x}\vartheta_{R,m,z}(x)|, |\nabla_{x}\vartheta_{R,m,z}(x)|^{2}\} \lesssim \frac{N_{m}(t)^{2}}{(\eta_{2}R)^{2}R^{2}}.
\end{equation}
Therefore by the Fubini-Tonelli theorem and mass conservation,
\begin{align}
\int_{0}^{T}\frac{N_{m}(t)}{\|\zeta(|\cdot|)\|_{L^{1}}}\int_{\R^{4}}|\mathrm{Error}_{1}(t,z)+\mathrm{Error}_{2}(t,z)| \zeta\left(\left|\frac{N_{m}(t)}{R}y-z\right|\right) |w(t,y)|^{2} dydz dt &\lesssim \frac{M(u)^{2}}{(\eta_{2}R)^{2}R^{2}}\int_{0}^{T}N_{m}(t)^{3}dt\nonumber\\
&\lesssim \frac{M(u)^{2}K}{\eta_{2}^{2}R^{4}}.
\end{align}

We now consider the contribution of $\mathrm{Main}(t,z)$. Using the ordinary Gagliardo-Nirenberg inequality (i.e. where $\mathcal{L}$ in theorem \ref{thm:GN} is replaced by the identity), we obtain the lower bound
\begin{align}
\mathrm{Main}(t,z) &\geq \frac{1}{C_{opt}\|(\vartheta_{R,m,z}v)(t)\|_{L_{x}^{2}}^{2}}\int_{\R^{2}} \vartheta_{R,m,z}(t,x)^{4} |v(t,x,z)|^{4}dx\nonumber\\
&\geq \eta_{1}\int_{\R^{2}} \vartheta_{R,m,z}(x)^{4} |w(t,x)|^{4}dx,
\end{align}
provided that $\eta_{1}\leq \frac{2}{C_{opt}M(u)}$ and $R,K>0$ are sufficiently large so that $M(\vartheta_{R,m,z}v)\geq M(u)/2$. By construction of $\vartheta$ and $\zeta$, we have that
\begin{equation}
\begin{split}
&\frac{\eta_{1}}{\|\zeta(|\cdot|)\|_{L^{1}}}\int_{\R^{2}}\int_{\R^{4}}\vartheta^{4}\left(\left|\frac{N_{m}(t)}{R}x-z\right|\right)\zeta\left(\left|\frac{N_{m}(t)}{R}y-z\right|\right) |w(t,x)|^{4} |w(t,y)|^{2}dxdy dz\\
&\geq \eta_{1}\int_{|x-y|\leq \frac{R^{2}}{10N_{m}(t)}} |w(t,x)|^{4} |w(t,y)|^{2} dxdy,
\end{split}
\end{equation}
where we use the Fubini-Tonelli theorem to integrate with respect to $z$ first. Therefore,
\begin{align}
\int_{0}^{T}\frac{N_{m}(t)}{\|\zeta(|\cdot|)\|_{L^{1}}}\int_{\R^{4}}\mathrm{Main}(t,z)\zeta\left(\left|\frac{N_{m}(t)}{R}y-z\right|\right) |w(t,y)|^{2}dydzdt &\geq \eta_{1}\int_{0}^{T}N_{m}(t)\int_{|x-y|\leq \frac{R^{2}}{10N_{m}(t)}}|w(t,y)|^{2}|w(t,x)|^{4}dxdy dt\nonumber\\
&\geq\frac{\eta_{1}M(u)}{2}\int_{0}^{T}N_{m}(t)|u(t,x)|^{4}dx dt - o_{R}(1)K,
\end{align}
provided that $R=R(M(u)),T=T(M(u))>0$ are sufficiently large by lemma \ref{lem:IM_pre_err3}. Bookkeeping our estimates, we have shown that
\begin{equation}
\mathrm{Term}_{1,1}+\mathrm{Term}_{2,1} \geq \frac{\eta_{1}M(u)}{2}\int_{0}^{T}N_{m}(t) \|u(t)\|_{L_{x}^{4}}^{4}dt -\frac{C(u)K}{\eta_{2}^{2}R^{4}}-C(u)o_{R}(1)K,
\end{equation}
which implies that
\begin{equation}
\eqref{eq:IM_d_main1}+\eqref{eq:IM_d_main2} \geq \frac{\eta_{1}M(u)}{2}\int_{0}^{T}N_{m}(t)\|u(t)\|_{L_{x}^{4}}^{4} dt -\frac{C(u)K}{\eta_{2}^{2}R^{4}}-C(u)o_{R}(1)K.
\end{equation}

\item[Estimate for \eqref{eq:IM_d_terr}:]
We next prove an estimate for the term \eqref{eq:IM_d_terr}, where the time derivative hits the potential $a_{R}$. This step is precisely the motivation for introducing the $m$-regular frequency scale functions $N_{m}(t)$. We observe from dilation invariance, the chain rule, and the fundamental theorem of calculus that
\begin{align}
\partial_{t}a_{R,m}(t,x-y) &=\partial_{t}\left(\frac{R(x-y)}{|x-y|}\int_{0}^{\frac{N_{m}(t)|x-y|}{R}} \varphi(s)ds\right) \nonumber\\
&=\varphi\left(\frac{N_{m}(t)|x-y|}{R}\right)N_{m}'(t)(x-y) \nonumber\\
&=\frac{N_{m}'(t)}{\|\zeta(|\cdot|)\|_{L^{1}}} \int_{\R^{2}}\zeta\left(\left|\frac{N_{m}(t)}{R}x-z\right|\right)\zeta\left(\left|\frac{N_{m}(t)}{R}y-z\right|\right)(x-y)dz,
\end{align}
where the ultimate equality follows from the definition of $\varphi$. So by the Fubini-Tonelli theorem,
\begin{equation}
\eqref{eq:IM_d_terr} = \int_{0}^{T}\int_{\R^{2}}\paren*{\int_{\R^{4}}\frac{N_{m}'(t)}{\|\zeta(|\cdot|)\|_{L^{1}}} \zeta\left(\left|\frac{N_{m}(t)}{R}x-z\right|\right)\zeta\left(\left|\frac{N_{m}(t)}{R}y-z\right|\right) (x-y)\cdot\Im{\bar{w}\nabla w}(t,x) |w(t,y)|^{2} dxdy}dz dt.
\end{equation}
Observe that the inner integral is invariant under the Galilean transformation $w\mapsto e^{-i\xi(t,z)\cdot x}w$, where $\xi(t,z)$ is chosen as above. Then using the elementary inequality
\begin{equation*}
ab \leq \frac{\delta a^{2}}{2}+\frac{b^{2}}{2\delta}, \qquad a,b\geq 0, \delta>0,
\end{equation*}
together with the bound $|x-y|\lesssim \frac{R^{2}}{N_{m}(t)}$ on the support of $\varphi$, we see that
\begin{align}
|\eqref{eq:IM_d_terr}| &\leq \frac{\delta}{\|\zeta(|\cdot|)\|_{L^{1}}} \int_{0}^{T}N_{m}(t)\int_{\R^{2}}\int_{\R^{4}} \zeta\left(\left|\frac{N_{m}(t)}{R}x-z\right|\right)\zeta\left(\left|\frac{N_{m}(t)}{R}y-z\right|\right) |\nabla_{x}(e^{-i\xi(t,z)\cdot x}w)(t,x)|^{2} |w(t,y)|^{2}dxdy dz dt\nonumber\\
&\phantom{=}+\frac{C_{0}R^{4}}{\delta\|\zeta(|\cdot|)\|_{L^{1}}} \int_{0}^{T}\frac{N_{m}'(t)^{2}}{N_{m}(t)^{3}}\int_{\R^{2}}\int_{\R^{4}}  \zeta\left(\left|\frac{N_{m}(t)}{R}x-z\right|\right)\zeta\left(\left|\frac{N_{m}(t)}{R}y-z\right|\right) |w(t,x)|^{2} |w(t,y)|^{2} dxdy dz dt\nonumber\\
&\eqqcolon \mathrm{Term}_{1}+\mathrm{Term}_{2},
\end{align}
for any $\delta>0$,	where $C_{0}>0$ is some absolute constant. Noting that
\begin{equation}
\mathrm{Term}_{1} = \delta\int_{0}^{T}\int_{\R^{4}}N_{m}(t)\varphi\left(\frac{N_{m}(t)|x-y|}{R}\right) |\nabla(e^{-ix\cdot\xi(t,z)}w)(t,x)|^{2}|w(t,y)|^{2}dxdy dt,
\end{equation}
we see that for $\delta>0$ sufficiently small, $\mathrm{Term}_{1}$ may be absorbed into $\mathrm{Term}_{1,1}+\mathrm{Term}_{1,2}$ in the computation of a lower bound for \eqref{eq:IM_d_main1}+\eqref{eq:IM_d_main2}.

To deal with $\mathrm{Term}_{2}$, we simply estimate it from above. Observe that by mass conservation and the estimate $|N_{m}'(t)| \lesssim N_{m}(t)^{3}$, uniformly in $t,m$, we have that
\begin{equation}
\mathrm{Term}_{2} \leq C(u)\frac{R^{4}M(u)^{2}}{\delta}\int_{0}^{T}N_{m}'(t)dt \leq 2C(u)\frac{M(u)^{2}R^{4}}{\delta m}\int_{0}^{T(m)}N_{m}(t)\|u(t)\|_{L_{x}^{4}}^{4}dt, \qquad \forall m\in N
\end{equation}
by the smoothing algorithm (lemma \ref{lem:IM_pre_sm}).

\item[Bookkeeping:]
We now bookkeep our estimates to obtain a contradiction for $T>0$ sufficiently large. Observe that we have shown that there exists a constant $C(u)>0$ such that
\begin{equation}
\begin{split}
C(u)\left(C(\eta_{3}R^{-2})R^{2}+\eta_{3}K\right)\geq M_{R,m}(T) &\geq \left(\frac{\eta_{1}M(u)}{2}-\frac{2C(u)R^{4}}{\delta m}\right)\int_{0}^{T}N_{m}(t)\|u(t)\|_{L_{x}^{4}}^{4}dt\\
&\phantom{=}- C(u)K\frac{R^{2}}{\epsilon_{3}}\left(\frac{C(\eta_{3})\epsilon_{3}}{K}+\eta_{3}\right)^{1/6}\\
&\phantom{=}-\frac{K}{(\eta_{2}R)^{4}}-C(u)o_{R}(1)K - \frac{C(u)}{\eta_{2}^{2}R^{4}}K - C(u)\eta_{2}K
\end{split}
\end{equation}
for all $T,R>0$ sufficiently large. Since $\int_{0}^{T}N_{m}(t)\|u(t)\|_{L_{x}^{4}}^{4}dt \geq \frac{1}{C(u)}K$, taking $C(u)$ larger if necessary, dividing both sides by $K$, we have shown that
\begin{equation}
\begin{split}
C(u)\left(\frac{C(\eta_{3}R^{-2})R^{2}}{K}+\eta_{3}\right) &\geq \frac{1}{C(u)}\left(\frac{\eta_{1}M(u)}{2}-\frac{2C(u)R^{4}}{\delta m}\right)\\
&\phantom{=}- C(u)\frac{R^{2}}{\epsilon_{3}}\left(\left(\frac{C(\eta_{3})\epsilon_{3}}{K}\right)^{1/6}+\eta_{3}^{1/6}\right)\\
&\phantom{=}-\frac{1}{(\eta_{2}R)^{4}}-C(u)o_{R}(1) - \frac{C(u)}{\eta_{2}^{2}R^{4}}-C(u)\eta_{2}
\end{split}
\end{equation}
for all $T,R>0$ sufficiently large. First, choose $\eta_{1}>0$ so that $\eta_{1}\leq \frac{2}{C_{opt}M(u)}$. Define the quantity $\epsilon$ implicitly by
\begin{equation}
\frac{\eta_{1}M(u)}{2C(u)} = 100\epsilon.
\end{equation}
Next, choose $\eta_{2}>0$ sufficiently small so that
\begin{equation}
C(u)\eta_{2}\leq \epsilon.
\end{equation}
Next, choose $R=R(\eta_{2})>0$ sufficiently large so that
\begin{equation}
\frac{1}{(\eta_{2}R)^{4}}+C(u)o_{R}(1) + \frac{C(u)}{\eta_{2}^{2}R^{4}} \leq \epsilon.
\end{equation}
Next, choose $\eta_{3}=\eta_{3}(R,\epsilon_{3})>0$ sufficiently small so that
\begin{equation}
C(u)\eta_{3}+\frac{C(u)R^{2}}{\epsilon_{3}}\eta_{3}^{1/6}\leq \epsilon.
\end{equation}
Next, choose $m=m(\delta,R)\in\mathbb{N}$ sufficiently large so that
\begin{equation}
\frac{2R^{4}}{\delta m}\leq \epsilon.
\end{equation}
Finally, choose $T=T(R,\eta_{3},\epsilon_{3})>0$ sufficiently large so that
\begin{equation}
\frac{C(u)C(\eta_{3}R^{-2})R^{2}}{K}+\frac{C(u)R^{2}}{\epsilon_{3}}\left(\frac{C(\eta_{3})\epsilon_{3}}{K}\right)^{1/6}\leq \epsilon.
\end{equation}
Therefore, we conclude that
\begin{equation}
2\epsilon\geq 100\epsilon-10\epsilon=90\epsilon,
\end{equation}
which is a contradiction.
\end{description}

\subsection{Focusing case}\label{ssec:QS_foc}
We now preclude the quasi-soliton scenario $\int_{0}^{\infty}N(t)^{3}dt=\infty$ in the case of the focusing eeDS equation, which we remind the reader is
\begin{equation}
(i\partial_{t}+\Delta)u = -|u|^{2}u-\E(|u|^{2})u=-\paren*{2\frac{\p_{1}^{2}}{\Delta}+\frac{\p_{2}^{2}}{\Delta}}(|u|^{2})u=-\mathcal{L}(|u|^{2})u.
\end{equation}
Let $C_{opt,\mathcal{L}}$ denote the optimal constant in the inequality
\begin{equation}
\|\mathcal{L}(|f|^{2})|f|^{2}\|_{L^{1}(\R^{2})} \leq C_{opt,\mathcal{L}}\|f\|_{L^{2}(\R^{2})}^{2} \|\nabla f\|_{L^{2}(\R^{2})}^{2}.
\end{equation}
We recall the characterization of $C_{opt}$ due Papanicolaou, Sulem, Sulem, and Wang (see theorem \ref{thm:GN}), which says that
\begin{equation}
C_{opt,\mathcal{L}}=\frac{2}{M(Q)}, \qquad \Delta Q-Q+\mathcal{L}(Q^{2})Q=0, \qquad Q>0.
\end{equation}
Note that in the focusing case, the critical mass $M_{c}$ satisfies the additional condition $M_{c}<M(Q)$ and hence the initial data of the solution also satisfies $M(u_{0})<M(Q)$ by definition of admissible blowup solution.

We repeat the integration by parts and fundamental theorem of calculus computations from subsection \ref{ssec:QS_con} to obtain that
\begin{align}
M_{R,m}(T) &= -4\int_{0}^{T}\int_{\R^{4}}a_{R,m,j}(t,x-y)|w(t,y)|^{2}\p_{k}\Re{\ol{\p_{k}w}\p_{j}w}(t,x)dxdy dt\label{eq:IM_f_main1}\\
&\phantom{=}-2\int_{0}^{T}\int_{\R^{4}}a_{R,m,j}(t,x-y)\partial_{k}T_{0k}^{w}(t,y)\Im{\bar{w}\p_{j} w}(t,x)dxdy dt\label{eq:IM_f_main2}\\
&\phantom{=}-\int_{0}^{T}\int_{\R^{4}}a_{R,m,j}(t,x-y)|w(t,y)|^{2}\p_{k}T_{jk}^{w}(t,x)dxdy dt\label{eq:IM_f_nl}\\
&\phantom{=}+\int_{0}^{T}\int_{\R^{4}}a_{R,m,j}(t,x-y)|w(t,y)|^{2}\p_{j}\p_{k}^{2}(|w|^{2})(t,x)dxdy dt\label{eq:IM_f_lerr}\\
&\phantom{=}+2\int_{0}^{T}\int_{\R^{4}}(\p_{t}a_{R,m})(t,x-y)\cdot |w(t,y)|^{2}\Im{\bar{w}\nabla w}(t,x)dxdy dt\label{eq:IM_f_terr}\\
&\phantom{=}+2\int_{0}^{T}\int_{\R^{4}}a_{R,m,j}(t,x-y)|w(t,y)|^{2}\Re{\bar{\mathcal{N}}\p_{j}w - \bar{w}\p_{j}\mathcal{N}}(t,x)dxdy dt\label{eq:IM_f_err1}\\
&\phantom{=}+4\int_{0}^{T}\int_{\R^{4}}a_{R,m,j}(t,x-y) \Im{\bar{w}\mathcal{N}}(t,y) \Im{\bar{w}\p_{j}w}(t,x)dxdy dt.\label{eq:IM_f_err2}
\end{align}
The only steps from the defocusing case which we need to reconsider now in the focusing case are the lower bound for \eqref{eq:IM_f_main1}+\eqref{eq:IM_f_main2} and the asymptotic for \eqref{eq:IM_f_nl}.

\begin{description}[leftmargin=*]
\item[Asymptotic for \eqref{eq:IM_f_nl}:]
We first rewrite \eqref{eq:IM_f_nl} in a convenient form. As before, we integrate by parts in $x$ to move the $\partial_{k}$ onto $a_{R,m,j}$, obtaining
\begin{equation*}
\eqref{eq:IM_f_nl} = \int_{0}^{T}\int_{\R^{4}}\partial_{k}a_{R,m,j}(t,x-y) |w(t,y)|^{2} T_{jk}^{w}(t,x)dxdy dt.
\end{equation*}
Unpackaging the definition of $T_{jk}^{w}$, using the operator identity $Id=\frac{\p_{1}^{2}}{\Delta}+\frac{\p_{2}^{2}}{\Delta}$, and proceeding by direct algebraic manipulation, we see that
\begin{align}
\p_{k}a_{R,m,j} T_{jk}^{w} &= -\partial_{j}a_{R,m,j}\left(2\frac{\p_{1}^{2}}{\Delta}(|w|^{2})+\frac{\p_{2}^{2}}{\Delta}(|w|^{2})\right)|w|^{2} \nonumber\\
&\phantom{=}+\left(\p_{2}a_{R,m,2}-\p_{1}a_{R,m,1}\right)\left(\frac{\p_{1}^{2}}{\Delta}(|w|^{2})\frac{\p_{2}^{2}}{\Delta}(|w|^{2})+\left(\frac{\p_{1}\partial_{2}}{\Delta}(|w|^{2})\right)^{2}\right) \nonumber\\
&\phantom{=}+2\partial_{2}a_{R,m,1}\frac{\p_{1}\p_{2}}{\Delta}(|w|^{2})\frac{\p_{1}^{2}}{\Delta}(|w|^{2})-2\p_{1}a_{R,m,2}\frac{\p_{2}^{2}}{\Delta}(|w|^{2})\frac{\p_{1}\p_{2}}{\Delta}(|w|^{2}) \nonumber\\
&\eqqcolon -\partial_{j}a_{R,m,j}\mathcal{L}(|w|^{2})|w|^{2} + \mathrm{Term}_{1}+\mathrm{Term}_{2}.
\end{align}

We claim that $\mathrm{Term}_{1}+\mathrm{Term}_{2}$ give a contribution to \eqref{eq:IM_f_nl} which is of size $o_{R}(1)K + \eta_{2}K$. Indeed, the estimate
\begin{equation}
\left|\int_{0}^{T}\int_{\R^{4}}\paren*{\mathrm{Term}_{1}+\mathrm{Term}_{2}}(t,x,y)|w(t,y)|^{2}dxdy dt\right| \lesssim o_{R}(1)K+\eta_{2}K,
\end{equation}
follows readily from the same analysis of the contributions of the error terms $\mathrm{Term}_{j}$ in \eqref{eq:IM_d_ten_dc}. Therefore, we have shown that
\begin{equation}
\left|\eqref{eq:IM_f_nl} + \int_{0}^{T}\int_{\R^{4}}(\partial_{j}a_{R,m,j})(t,x-y) |w(t,y)|^{2}\mathcal{L}(|w|^{2})(t,x) |w(t,x)|^{2}dxdy dt\right| \lesssim o_{R}(1)K +\eta_{2}K.
\end{equation}

\item[Lower bound for $\eqref{eq:IM_f_main1}+\eqref{eq:IM_f_main2}$:]
We now proceed to obtaining a lower bound for \eqref{eq:IM_f_main1}+\eqref{eq:IM_f_main2}. As in the defocusing case,
\begin{equation}
\begin{split}
\eqref{eq:IM_f_main1} + \eqref{eq:IM_f_main2} &= \int_{0}^{T}\frac{4N_{m}(t)}{\|\zeta(|\cdot|)\|_{L^{1}}}\int_{\R^{4}}\int_{\R^{2}} |\nabla(\vartheta_{R,m,z}v)(t,x,z)|^{2}\zeta\left(\left|\frac{N_{m}(t)}{R}y-z\right|\right)|w(t,y)|^{2}dx dydz dt\\
&\phantom{=}+O\left(\frac{K}{\eta_{2}^{2}R^{4}}\right) .
\end{split}
\end{equation}
Now decompose
\begin{equation}
\begin{split}
\int_{\R^{2}} |\nabla(\vartheta_{R,m,z}v_{z})(t,x)|^{2}dx &= \eta_{1}\int_{\R^{2}} |\nabla(\vartheta_{R,m,z}v_{z})(t,x)|^{2}dx +(1-\eta_{1})\int_{\R^{2}} |\nabla(\vartheta_{R,m,z}v_{z})(t,x)|^{2}dx\\
&\eqqcolon \mathrm{Term}_{1}+\mathrm{Term}_{2}.
\end{split}
\end{equation}

To obtain a lower bound for $\mathrm{Term}_{1}$, we use the ordinary Gagliardo-Nirenberg inequality (i.e. where $\mathcal{L}$ is replaced by $Id$) to obtain that
\begin{equation}
\mathrm{Term}_{1}(t,z) \geq \frac{\eta_{1}}{C_{opt}M(\vartheta_{R,m,z}w)} \int_{\R^{2}} |(\vartheta_{R,m,z}w)(t,x)|^{4}dx.
\end{equation}
It follows from our work in the defocusing case that
\begin{equation}
\int_{0}^{T}\frac{4N_{m}(t)}{\|\zeta(|\cdot|)\|_{L^{1}}}\int_{\R^{4}} \mathrm{Term}_{1}(t,z)\zeta\left(\left|\frac{N_{m}(t)}{R}y-z\right|\right)|w(t,y)|^{2}dydz dt \geq \frac{2\eta_{1}}{C_{opt}M(u)}K,
\end{equation}
provided that $R,T>0$ are sufficiently sufficiently large so that $M(\vartheta_{R,m,z}w)\geq \frac{M(u)}{2}$.

To obtain a lower bound for $\mathrm{Term}_{2}$, we use the sharp Gagliardo-Nirenberg ineqality for $\mathcal{L}$ (theorem \ref{thm:GN}) to obtain that
\begin{equation}
\mathrm{Term}_{2}(t,z) \geq (1-\eta_{1})\frac{M(Q)}{2M(\vartheta_{R,m,z}w)} \left|\int_{\R^{2}}\mathcal{L}(|\vartheta_{R,m,z}w|^{2})(t,x) |(\vartheta_{R,m,z}w)(t,x)|^{2} dx\right|.
\end{equation}
Therefore
\begin{equation}
\begin{split}
&\int_{0}^{T}\frac{4N_{m}(t)}{\|\zeta(|\cdot|)\|_{L^{1}}}\int_{\R^{4}}\mathrm{Term}_{2}(t,z)\zeta\left(\left|\frac{N_{m}(t)}{R}y-z\right|\right)|w(t,y)|^{2}dydz dt\\
&\geq \frac{2(1-\eta_{1})M(Q)}{M(\vartheta_{R,m,z}w)} \int_{0}^{T}dt\frac{N_{m}(t)}{\|\zeta(|\cdot|)\|_{L^{1}}}\int_{\R^{4}}dydz\zeta\left(\left|\frac{N_{m}(t)}{R}y-z\right|\right)|w(t,y)|^{2}\times\\
&\phantom{=}\left|\int_{\R^{2}}\mathcal{L}(|\vartheta_{R,m,z}w|^{2})(t,x) |(\vartheta_{R,m,z}w)(t,x)|^{2}dx\right|.
\end{split}
\end{equation}

\item[Lower bound for $\eqref{eq:IM_f_main1}+\eqref{eq:IM_f_main2}+\eqref{eq:IM_f_nl}$:]
We now claim that
\begin{equation}
\begin{split}
&-\int_{0}^{T}\int_{\R^{4}}(\nabla\cdot a_{R,m})(t,x-y) |w(t,y)|^{2}\mathcal{L}(|w|^{2})(t,x) |w(t,x)|^{2}dxdy dt\\
&=-2\int_{0}^{T}\frac{N_{m}(t)}{\|\zeta(|\cdot|)\|_{L^{1}}}\int_{\R^{4}}\zeta\left(\left|\frac{N_{m}(t)}{R}y-z\right|\right)|w(t,y)|^{2}\left|\int_{\R^{2}}\mathcal{L}(|\vartheta_{R,m,z}w|^{2})(t,x) |(\vartheta_{R,m,z}w)(t,x)|^{2}dx\right|dydzdt\\
&\phantom{=}+o_{R}(1)K+\eta_{2}O(K).
\end{split}
\end{equation}

To prove the claim, we first observe that since $\vartheta, \zeta\equiv 1$ on the interval $\{|r|\leq (1-6\eta_{2})R\}$, it follows that
\begin{align}
&\frac{1}{\|\zeta(|\cdot|)\|_{L^{1}}} \int_{\R^{2}} \left|\zeta\left(\left|\frac{N_{m}(t)}{R}x-z\right|\right) -\vartheta^{2}\left(\left|\frac{N_{m}(t)}{R}x-z\right|\right)\right|\zeta\left(\left|\frac{N_{m}(t)}{R}y-z\right|\right)dz \nonumber\\
&\phantom{=}\lesssim \frac{1}{\|\zeta(|\cdot|)\|_{L^{1}}}\int_{(1-16\eta_{2})R\leq |z|\leq (1+8\eta_{2})R} dz \nonumber\\
&\phantom{=}\lesssim \eta_{2}.\label{eq:IM_f_diff_e}
\end{align}
As a consequence of the identity $\p_{r}(r\psi_{R,m}(r))=\varphi(\frac{N_{m}(t)r}{R})$, we have that
\begin{equation}
(\nabla\cdot a_{R,m})(t,x-y) = 2N_{m}(t)\varphi\left(\frac{N_{m}(t)|x-y|}{R}\right) - (\p_{r}\psi_{R,m})(t,|x-y|)N_{m}(t)|x-y|.
\end{equation}
Hence, it follows from our analysis of the error terms in subsection \ref{ssec:QS_dfoc} that
\begin{equation}
\begin{split}
&\int_{0}^{T}\int_{\R^{4}}(\nabla\cdot a_{R,m})(t,x-y) |w(t,y)|^{2}\mathcal{L}(|w|^{2})(t,x)|w(t,x)|^{2}dxdy dt\\
&=o_{R}(1)K +\eta_{2}O(K)+ 2\int_{0}^{T}N_{m}(t)\int_{\R^{4}}\varphi\left(\frac{N_{m}(t)|x-y|}{R}\right) |w(t,y)|^{2}\mathcal{L}(|w|^{2})(t,x)|w(t,x)|^{2}dxdydt.
\end{split}
\end{equation}
Hence by the Fubini-Tonelli theorem and the triangle inequality,
\begin{align}
&\left|\int_{0}^{T}dtN_{m}(t)\int_{\R^{4}}dxdy\paren*{\varphi\left(\frac{N_{m}(t)|x-y|}{R}\right)-\|\zeta(|\cdot|)\|_{L^{1}}^{-1}\int_{\R^{2}}\vartheta^{2}\left(\left|\frac{N_{m}(t)}{R}x-z\right|\right) \zeta\left(\left|\frac{N_{m}(t)}{R}y-z\right|\right)dz} \right.\nonumber\\
&\phantom{=}\qquad \left.|w(t,y)|^{2}\mathcal{L}(|w|^{2})(t,x)|w(t,x)|^{2}\right|\nonumber\\
&\lesssim \eta_{2}\int_{0}^{T}N_{m}(t)\int_{\R^{4}} |w(t,y)|^{2} |\mathcal{L}(|w|^{2})(t,x)| |w(t,x)|^{2}dxdy dt\nonumber\\
&\lesssim \eta_{2}K.
\end{align}
where we use the estimate \eqref{eq:IM_f_diff_e}, H\"{o}lder's inequality, and Plancherel's theorem to obtain the ultimate inequality.

It suffices now for us to estimate the quantity
\begin{equation}
\|\zeta(|\cdot|)\|_{L^{1}}^{-1}\int_{0}^{T}N_{m}(t)\int_{\R^{6}}\vartheta^{2}\left(\left|\frac{N_{m}(t)}{R}x-z\right|\right)\zeta\left(\left|\frac{N_{m}(t)}{R}y-z\right|\right) |w(t,y)|^{2} |\mathcal{L}\paren*{|w|^{2}(1-\vartheta_{R,m,z}^{2})}(t,x)|^{2} |w(t,x)|^{2}dxdydzdt,\label{eq:f_claim}
\end{equation}
which we claim is of size $\sim o_{R}(1)K+\eta_{2}K$. Assuming the claim, let us conclude the lower bound for \eqref{eq:IM_f_main1}+\eqref{eq:IM_f_main2}+\eqref{eq:IM_f_nl}. Bookkeeping our estimates, we have shown that there exists a constant $C(u)>0$ such that
\begin{equation}\label{eq:IM_f_lb_bk}
\begin{split}
&\eqref{eq:IM_f_main1}+\eqref{eq:IM_f_main2}+\eqref{eq:IM_f_nl}\\
&\geq \frac{2\eta_{1}}{C_{opt}M(u)}K \\
&\phantom{=}+\frac{2(1-\eta_{1})M(Q)}{M(u)}\int_{0}^{T}\frac{N_{m}(t)}{\|\zeta(|\cdot|)\|_{L^{1}}}\int_{\R^{4}}\zeta\paren*{\left|\frac{N_{m}(t)}{R}y-z\right|}|w(t,y)|^{2}|\ipp{\mathcal{L}(|\vartheta_{R,m,z}w(t)|^{2}),|\vartheta_{R,m,z}w(t)|^{2}}|dydzdt \\
&\phantom{=}-2\int_{0}^{T}\frac{N_{m}(t)}{\|\zeta(|\cdot|)\|_{L^{1}}}\int_{\R^{4}}\zeta\paren*{\left|\frac{N_{m}(t)}{R}y-z\right|}|w(t,y)|^{2}|\ipp{\mathcal{L}(|\vartheta_{R,m,z}w(t)|^{2}),|\vartheta_{R,m,z}w(t)|^{2}}|dydzdt \\
&\phantom{=}-o_{R}(1)K-C(u)\eta_{2}K,
\end{split}
\end{equation}
provided that $\eta_{1}=\eta_{1}(u)>0$ is sufficiently small and $T,R>0$ are sufficiently large. Since $M(u)<M(Q)$, we can choose $\eta_{1}>0$ sufficiently small so that
\begin{equation}
(1-\eta_{1})\frac{M(Q)}{M(u)}>1,
\end{equation}
which implies that the sum of the second and third lines in the RHS of \eqref{eq:IM_f_lb_bk} is nonnegative. Hence, we conclude the lower bound
\begin{equation}
\eqref{eq:IM_f_main1}+\eqref{eq:IM_f_main2}+\eqref{eq:IM_f_nl} \geq \frac{2\eta_{1}}{C_{opt}M(u)}K-o_{R}(1)K-C(u)\eta_{2}K.
\end{equation}

We now estimate \eqref{eq:f_claim} by dividing the spatial integration into cases as follows. Let $\mathcal{K}$ denote the Schwartz kernel of $\mathcal{L}$, and let $\chi$ be the bump function defined in subsection \ref{ssec:QS_con}. We decompose $\mathcal{L}$ into a ``local" piece in a ball of radius $\sim \frac{\eta_{2}R}{N_{m}(t)}$ around the origin and a ``global" piece outside this ball:
\begin{align}
\mathcal{L}_{loc} &\coloneqq \left(\mathcal{K}\chi\left(\frac{4N_{m}(t)}{\eta_{2}R^{2}}\cdot\right)\right) \ast  \eqqcolon \mathcal{K}_{loc}\ast\\
\mathcal{L}_{glob} &\coloneqq \left(\mathcal{K}\left(1-\chi\left(\frac{4N_{m}(t)}{\eta_{2}R^{2}}\cdot\right)\right)\right)\ast \eqqcolon \mathcal{K}_{glob}\ast.
\end{align}
Now decompose \eqref{eq:f_claim} by
\begin{equation}
\eqref{eq:f_claim} = \mathrm{Term}_{1}+\mathrm{Term}_{2}+\mathrm{Term}_{3},
\end{equation}
where
\begin{equation}
\begin{split}
\mathrm{Term}_{1} &\coloneqq \|\zeta(|\cdot|)\|_{L^{1}}^{-1}\int_{0}^{T}dtN_{m}(t)\int_{|x-y|\leq \frac{\eta_{2}R^{2}}{4N_{m}(t)}}dxdydz \vartheta^{2}\left(\left|\frac{N_{m}(t)}{R}x-z\right|\right)\zeta\left(\left|\frac{N_{m}(t)}{R}y-z\right|\right)\\
&\phantom{=}\qquad |w(t,y)|^{2}\mathcal{L}_{loc}\left(|w|^{2}(1-\vartheta_{R,m,z}^{2})\right)(t,x) |w(t,x)|^{2},
\end{split}
\end{equation}
\begin{equation}
\begin{split}
\mathrm{Term}_{2} &\coloneqq \|\zeta(|\cdot|)\|_{L^{1}}^{-1}\int_{0}^{T}dtN_{m}(t)\int_{|x-y|\leq \frac{\eta_{2}R^{2}}{4N_{m}(t)}}dxdydz \vartheta^{2}\left(\left|\frac{N_{m}(t)}{R}x-z\right|\right)\zeta\left(\left|\frac{N_{m}(t)}{R}y-z\right|\right) \\
&\phantom{=}\qquad |w(t,y)|^{2}\mathcal{L}_{glob}\left(|w|^{2}(1-\vartheta_{R,m,z}^{2})\right)(t,x) |w(t,x)|^{2},
\end{split}
\end{equation}
and
\begin{equation}
\begin{split}
\mathrm{Term}_{3} &\coloneqq \|\zeta(|\cdot|)\|_{L^{1}}^{-1}\int_{0}^{T}dtN_{m}(t)\int_{|x-y|\geq \frac{\eta_{2}R^{2}}{4N_{m}(t)}}dxdydz \vartheta^{2}\left(\left|\frac{N_{m}(t)}{R}x-z\right|\right)\zeta\left(\left|\frac{N_{m}(t)}{R}y-z\right|\right)\\
&\phantom{=}\qquad |w(t,y)|^{2}\mathcal{L}\left(|w|^{2}(1-\vartheta_{R,m,z}^{2})\right)(t,x) |w(t,x)|^{2}.
\end{split}
\end{equation}
We now estimate each of the $\mathrm{Term}_{j}$ separately.

\begin{itemize}[leftmargin=*]
\item
We first estimate $\mathrm{Term}_{1}$. By using the decomposition of the Schwartz kernel $\mathcal{K}$ given by proposition \ref{prop:dist_class} and considering each piece separately, we may pretend that $\mathcal{K}$ is actually a measurable function.

We claim that 
\begin{equation}
1_{\leq \frac{\eta_{2}R^{2}}{4N_{m}(t)}}(|x-y|)\mathcal{L}_{loc}(|w|^{2}(1-\vartheta_{R,m,z}^{2}))(t,x) \neq 0 \Longrightarrow \left|\frac{N_{m}(t)}{R}y-z\right| > (1-7\eta_{2})R.
\end{equation}
Indeed, observe that
\begin{align}
&\mathcal{L}_{loc}(|w|^{2}(1-\vartheta_{R,m,z}^{2}))(t,x) \nonumber\\
&\phantom{=} =\lim_{\varepsilon\rightarrow 0^{+}}\int_{\varepsilon\leq |z'|\leq \frac{\eta_{2}R^{2}}{4N_{m}(t)}} \K_{loc}(z') |w(t,x-z')|^{2}|w(t,x-z')|^{2}\paren*{1-\vartheta^{2}\paren*{\left|\frac{N_{m}(t)}{R}(x-z')-z\right|}}dz'.
\end{align}
Since $\supp(1-\vartheta^{2})\subset \{|r| \geq (1-6\eta_{2})R\}$ and by the triangle inequality,
\begin{equation}
|z'|\leq \frac{\eta_{2}R^{2}}{4N_{m}(t)} \enspace \text{and} \enspace \left|\frac{N_{m}(t)}{R}x-z\right|\leq (1-6.5\eta_{2})R \Longrightarrow \left|\frac{N_{m}(t)}{R}(x-z')-z\right| \leq (1-6.25\eta_{2})R,
\end{equation}
it follows that $\L_{loc}(|w|^{2}(1-\vartheta_{R,m,z}^{2}))(t,x)=0$. By the reverse triangle inequality together with the condition $|x-y|\leq \frac{\eta_{2}R^{2}}{4N_{m}(t)}$, we then obtain the claim.

Therefore by Cauchy-Schwarz in $x$, followed by mass conservation, we have that
\begin{equation}
\begin{split}
|\mathrm{Term}_{1}| &\lesssim \frac{1}{\|\zeta(|\cdot|)\|_{L^{1}}}\int_{0}^{T}dtN_{m}(t)\int_{\R^{4}}dydz 1_{\geq (1-7\eta_{2})R}\paren*{\left|\frac{N_{m}(t)}{R}y-z\right|} \zeta\left(\left|\frac{N_{m}(t)}{R}y-z\right|\right) |w(t,y)|^{2} \\
&\phantom{=}\qquad \left(\int_{\R^{2}} |w(t,x)|^{4}dx\right)^{1/2}\left(\int_{\R^{2}}|\mathcal{L}_{loc}\left(|w|^{2}(1-\vartheta_{R,m,z}^{2})\right)(t,x)|^{2}dx\right)^{1/2}.
\end{split}
\end{equation}
Since $\|\widehat{\mathcal{K}_{loc}}\|_{L^{\infty}(\R^{2})} \lesssim \|\hat{\mathcal{K}}\|_{L^{\infty}}\lesssim 1$ by Young's inequality, we can use Plancherel's theorem to obtain the estimate
\begin{equation}
\sup_{y,z\in\R^{2}}\left(\int_{\R^{2}}|\mathcal{L}_{loc}\left(|w|^{2}(1-\vartheta_{R,m,z}^{2})\right)(t,x)|^{2}dx\right)^{1/2} \lesssim \|w(t)\|_{L_{x}^{4}(\R^{2})}^{2}.
\end{equation}
Hence by the Fubini-Tonelli theorem, the support of $\zeta$, and mass conservation,
\begin{align}
|\mathrm{Term}_{1}| &\lesssim \|\zeta(|\cdot|)\|_{L^{1}}^{-1}\int_{0}^{T}dtN_{m}(t) \|w(t)\|_{L_{x}^{4}(\R^{2})}^{4}\int_{\R^{2}}dy|w(t,y)|^{2}\int_{\R^{2}}dz 1_{(1-7\eta_{2})R\leq \cdot \leq (1+4\eta_{2})R}\paren*{\left|\frac{N_{m}(t)}{R}y-z\right|} \nonumber\\
&\lesssim \eta_{2}\int_{0}^{T}N_{m}(t)\|w(t)\|_{L_{x}^{4}(\R^{2})}^{4}dt\nonumber\\
&\lesssim_{u} \eta_{2}K.
\end{align}

\item
We next estimate $\mathrm{Term}_{2}$.	By Cauchy-Schwarz in $z'$ together with the pointwise estimate
\begin{equation}
|\mathcal{K}_{glob}(z')| \lesssim |z'|^{-2}1_{|z'| \geq \frac{\eta_{2}R^{2}}{4N_{m}(t)}},
\end{equation}
we have that
\begin{align}
\left|\int_{\R^{2}} \mathcal{K}_{glob}(z') |w(t,x-z')|^{2}\left(1-\vartheta^{2}\left(\left|\frac{N_{m}(t)}{R}(x-z')-z\right|\right)\right)dz\right| &\lesssim \left(\int_{\R^{2}} |w(t,x-z')|^{4}dz'\right)^{1/2}\left(\int_{|z'|\geq\frac{\eta_{2}R^{2}}{4N_{m}(t)}} |z'|^{-4}dz'\right)^{1/2} \nonumber\\
&\phantom{=}\lesssim \frac{N_{m}(t)}{\eta_{2}R^{2}} \|w(t)\|_{L_{x}^{4}(\R^{2})}^{2},
\end{align}
by translation and dilation invariance. Hence,
\begin{align}
|\mathrm{Term}_{2}| &\lesssim \int_{0}^{T}\frac{N_{m}(t)^{2}\|w(t)\|_{L_{x}^{4}(\R^{2})}^{2}}{\eta_{2}R^{2}\|\zeta(|\cdot|)\|_{L^{1}}}\int_{\R^{6}} \zeta\left(\left|\frac{N_{m}(t)}{R}y-z\right|\right)\vartheta^{2}\left(\left|\frac{N_{m}(t)}{R}x-z\right|\right) |w(t,x)|^{2} |w(t,y)|^{2}dxdydz dt\nonumber\\
&\lesssim \int_{0}^{T}\frac{N_{m}(t)^{2}\|w(t)\|_{L_{x}^{4}(\R^{2})}^{2}}{\eta_{2}R^{2}}dt,
\end{align}
where we use mass conservation to obtain the ultimate line. Now by Cauchy-Schwarz in time, we obtain that
\begin{align}
|\mathrm{Term}_{2}| &\lesssim \frac{1}{\eta_{2}R^{2}}\left(\int_{0}^{T}N_{m}(t)^{3}dt\right)^{1/2}\left(\int_{0}^{T}N_{m}(t)\|w(t)\|_{L_{x}^{4}(\R^{2})}^{4}dt\right)^{1/2}\nonumber\\
&\lesssim_{u} \frac{K}{\eta_{2}R^{2}}.
\end{align}

\item
Lastly, we estimate $\mathrm{Term}_{3}$. Since $(1-\vartheta_{R,m,z}^{2}) \lesssim 1$ pointwise, we can apply lemma \ref{lem:IM_pre_CZ} to obtain that
\begin{equation}
|\mathrm{Term}_{3}|=\frac{o_{R}(1)}{\eta_{2}}K.
\end{equation}
\end{itemize}

\item[Bookkeeping:]
We now have all the necessary estimates to obtain a contradiction. After a bit of bookkeeping, we have shown that there exists a constant $C(u)>0$ such that
\begin{equation}
\begin{split}
C(u)\paren*{C(\eta_{3}R^{-2})R^{2}+\eta_{3}K} \geq M_{R,m}(T)&\geq \frac{2\eta_{1}}{C_{opt}M(u)}K -\frac{2C(u)R^{4}}{\delta m}K-\frac{C(u)R^{2}K}{\epsilon_{3}}\paren*{\frac{C(\eta_{3})\epsilon_{3}}{K}+\eta_{3}}^{1/6} \\
&\phantom{=}-\frac{K}{(\eta_{2}R)^{4}} - C(u)o_{R}(1)K-\frac{C(u)K}{\eta_{2}^{2}R^{2}}-C(u)\eta_{2}K.
\end{split}
\end{equation}
for $\eta_{1}=\eta_{1}(u,M(u))>0$ sufficiently small and $R,T>0$ sufficiently large. Dividing both sides by $K$, our goal is to obtain a contradiction from the inequality
\begin{equation}
\begin{split}
C(u)\paren*{\frac{C(\eta_{3}R^{-2})R^{2}}{K}+\eta_{3}} &\geq \frac{2\eta_{1}}{C_{opt}M(u)}-\frac{2C(u)R^{4}}{\delta m} - \frac{C(u)R^{2}}{\epsilon_{3}}\paren*{\frac{C(\eta_{3})\epsilon_{3}}{K}+\eta_{3}}^{1/6} \\
&\phantom{=}-\frac{1}{(\eta_{2}R)^{4}}-C(u)o_{R}(1)-\frac{C(u)}{\eta_{2}^{2}R^{4}}-C(u)\eta_{2}
\end{split}
\end{equation}
by now judiciously choosing the parameters $\eta_{2},\eta_{3},R,m$.

Define the quantity $\epsilon>0$ implicitly by
\begin{equation}
\frac{2\eta_{1}}{C_{opt}M(u)} = 100\epsilon.
\end{equation}
Next, choose $\eta_{2}>0$ sufficiently small so that
\begin{equation}
C(u)\eta_{2}\leq \epsilon.
\end{equation}
Next, choose $R=R(\eta_{2})>0$ sufficiently large so that
\begin{equation}
\frac{1}{(\eta_{2}R)^{4}} + \frac{C(u)}{\eta_{2}^{2}R^{4}} + C(u)o_{R}(1) \leq \epsilon.
\end{equation}
Next, choose $m=m(R)\in\N$ sufficiently large so that
\begin{equation}
\frac{2C(u)R^{4}}{\delta m} \leq \epsilon.
\end{equation}
Next, choose $\eta_{3}=\eta_{3}(R,\epsilon_{3})>0$ sufficiently small so that
\begin{equation}
C(u)\eta_{3}+\frac{C(u)R^{2}}{\epsilon_{3}}\eta_{3}^{1/6} \leq \epsilon.
\end{equation}
Lastly, by choose $T=T(\epsilon_{3},m,\eta_{3},R)>0$ be sufficiently large so that
\begin{equation}
\frac{C(u)C(\eta_{3}R^{-2})R^{2}}{K}+\frac{C(u)C(\eta_{3})R^{2}}{\epsilon_{3}^{5/6}K} \leq \epsilon.
\end{equation}
With these choices of parameters, we obtain the inequality $2\epsilon\geq 90\epsilon$, which is a contradiction. This last step completes the proof.
\end{description}

\bibliographystyle{siam}
\bibliography{DaveyStewartson}

\end{document}